\documentclass[11pt,reqno]{amsart}
\usepackage{amsthm,amsfonts,amssymb,euscript,mathrsfs,graphics,color,amsmath,amssymb,latexsym,marginnote}

\usepackage[dvips]{graphicx}

\usepackage{hyperref}

\usepackage[margin=.95in]{geometry}

\newtheorem{thm}{Theorem}[section]
 \newtheorem{cor}[thm]{Corollary}
 \newtheorem{lem}[thm]{Lemma}
 \newtheorem{claim}[thm]{Claim}
 \newtheorem{prop}[thm]{Proposition}

 \newtheorem{rem}[thm]{Remark}
 \newtheorem{ex}{Example}
 
\numberwithin{equation}{section}

\def\be#1 {\begin{equation} \label{#1}}
\newcommand{\ee}{\end{equation}}

\def\C{{\mathbb C}}
\def\Z{{\mathbb Z}}

\def\R{{\mathbb R}}

\def\e{{\varepsilon}}
\def\g{{\gamma}}
\def\G{{\Gamma}}
\def\s{{\sigma}}

\def\a{{\alpha}}
\def\b{{\beta}}
\def\d{{\delta}}
\def\l{{\lambda}}
\def\k{{\kappa}}

\def\sign{ \mbox{sign} }
\def\max{\mathrm{max}}
\def\med{\mathrm{med}}
\def\min{\mathrm{min}}

\def\F{{\widetilde{\mathcal{F}}}}

\def\wtF{\widetilde{{\mathcal{F}}}}
\def\whF{\widehat{{\mathcal{F}}}}

\def\W{\mathcal{W}^*}

\def\supp{\,\mbox{supp}\,}

\def\jxi{\langle \xi \rangle}
\def\jeta{\langle \eta \rangle}
\def\jsig{\langle \sigma \rangle}

\def\jrho{\langle \rho \rangle}

\def\jt{\langle t \rangle}
\def\js{\langle s \rangle}

\def\jnab{\langle \partial_x \rangle}

\def\jx{\langle x \rangle}
\def\jy{\langle y \rangle}

\def\what{\widehat}
\def\wt{\widetilde}
\def\bar{\overline}
\def\eps{\epsilon}

\def\pv{\mathrm{p.v.}}

\DeclareMathAlphabet{\mathpzc}{OT1}{pzc}{m}{it}

\numberwithin{equation}{section}

\makeatletter
\newcommand{\pushright}[1]{\ifmeasuring@#1\else\omit\hfill$\displaystyle#1$\fi\ignorespaces}
\newcommand{\pushleft}[1]{\ifmeasuring@#1\else\omit$\displaystyle#1$\hfill\fi\ignorespaces}
\makeatother

\definecolor{bluegreen}{rgb}{0.0, 0.3, 0.9}

\begin{document}

\author{Pierre Germain}
\address{Pierre Germain, Courant Institute of Mathematical Sciences, 251 Mercer Street, New York 10012-1185 NY, USA}
\email{pgermain@cims.nyu.edu}

\author{Fabio Pusateri}
\address{Fabio Pusateri, University of Toronto, Department of Mathematics, 40 St George Street, 
  Toronto, ON, M5S 2E4, Canada.}
\email{fabiop@math.toronto.edu}

\title{Quadratic Klein-Gordon equations with \\ a potential in one dimension}




\begin{abstract}
\normalsize

This paper proposes a fairly general new point of view on the question of asymptotic stability of (topological) solitons. 
Our approach is based on the use of the distorted Fourier transform at the nonlinear level; 
it does not only rely on Strichartz or virial estimates and is therefore able to treat low power nonlinearities 
(hence also non-localized solitons) and capture the global (in space and time) behavior of solutions.

More specifically, we consider quadratic nonlinear Klein-Gordon equations with a regular and 
decaying potential in one space dimension. 
Additional assumptions are made so that the distorted Fourier transform of the solution vanishes at zero frequency.
Assuming also that the associated Schr\"odinger operator has no negative eigenvalues, 
we obtain global-in-time bounds, including sharp pointwise decay and modified asymptotics, for small solutions.

These results have some direct applications to the asymptotic stability of (topological) solitons, 
as well as several other potential applications to a variety of related problems.
For instance, we obtain full asymptotic stability of kinks with respect to odd perturbations 
for the Double Sine-Gordon problem
(in an appropriate range of the deformation parameter).
For the $\phi^4$ problem, we obtain asymptotic stability 
for small odd solutions provided the nonlinearity is projected on the continuous spectrum. 
Our results also go beyond these examples since our framework allows for the presence of 
a fully coherent phenomenon (a space-time resonance) at the level of quadratic interactions,
which creates a degeneracy in distorted Fourier space.
We devise a suitable framework that incorporate this, 
and use multilinear harmonic analysis 
in the distorted setting to control all nonlinear interactions.
\end{abstract}

\maketitle

\markboth{P. GERMAIN AND F. PUSATERI}{QUADRATIC KLEIN-GORDON WITH A POTENTIAL IN $1$D}

\setcounter{tocdepth}{2}

\begin{quote}
\large
\tableofcontents
\end{quote}

\section{Introduction}

This work concerns the global-in-time behavior of small solutions 
of one dimensional quadratic Klein-Gordon equations with an external potential.
The class of equations that we treat in this paper appears when studying the {\it asymptotic stability} 
of special solutions of
nonlinear dispersive and hyperbolic equations, such as solitons, traveling waves, kinks.

\smallskip
\subsection{The model and motivation} 

\subsubsection{The equation}
We consider the equation
\begin{equation}
\tag{KG} \label{KG}
\partial_t^2 u + (- \partial_x^2 + V(x) + m^2) u = a(x) u^2 
\end{equation}
where the unknown $u = u(t,x) \in \R$, the space and time variables $(t,x) \in \R\times\R$, $m>0$ 
is the mass parameter, 
$V$ is a real-valued, decaying and smooth external potential,
and $a$ is a sufficiently smooth function with $a(x) - \ell_{\pm\infty}$ decaying quickly
as $x\rightarrow \pm \infty$, $\ell_{\pm\infty} \in \R$.
The addition of cubic and higher order terms (with constant or non-constant coefficients)
does not bring any further complication, so we omit them for the sake of explanation.\footnote{In fact,
cubic terms such as $u^3$, and more complicated ones,
naturally appear in the analysis of \eqref{KG} performed in this paper.}

The equation \eqref{KG} derives from the Hamiltonian
\begin{align}
\label{KGVHam}
\mathscr{H}(u) = \frac{1}{2}\int_{\R} \left[ (\partial_t u)^2 + (\partial_x u)^2 + m^2 u^2 + Vu^2 \right] \, dx 
  + \frac{1}{3}\int_{\R} a(x) u^3 \, dx.
\end{align}
By rescaling we can set $m=1$ without loss of generality; we will do so in the rest of the paper.
We will be interested in the Cauchy problem with small initial data 
$(u(0,x),u_t(0,x)) = (u_0(x),u_1(x))$ in suitable weighted Sobolev spaces.
In short, under some spectral assumptions on $V$, 
our main result, Theorem \ref{maintheo}, 
gives the existence of global small solutions with sharp pointwise time-decay, and long-range asymptotics.

We will consider a broad class of external potentials in \eqref{KG} both generic and exceptional, %
with some additional assumptions in the latter case.
In all cases we assume that there is no discrete spectrum.
The class of non-generic potentials that we consider arises in applications such as, for example, 
pure power nonlinear Klein-Gordon and the $\phi^4$ model; see Subsection \ref{ssecappl}.



\smallskip
\subsubsection{Motivation}
Nonlinear equations with external potentials arise from the perturbation 
of full nonlinear problems around special solutions, such as solitons.
The quadratic problem \eqref{KG} is inspired 
by the long-standing open question of the full asymptotic stability of 
the kink solution $K=\mathrm{tanh}(x/\sqrt{2})$ for the $\phi^4$ model $\phi_{tt}-\phi_{xx} = \phi - \phi^3$
(see \S\ref{sssecappl} and \cite{KowMarMun}). 
It is also closely related to similar questions about solitons of nonlinear Klein-Gordon,
kinks of other relativistic Ginzburg-Landau theories,
and generalized sine-Gordon theories in $1+1$ dimensions.

One dimensional kinks are the simplest example of topological solitons, 
that is, non-spatially-localized special solutions,
as opposed to the more standard solitons that are localized in space.
While the mathematical theory on the stability (or instability) of solitons is very well-developed in many models,
this is not the case for topological solitons. 
There are in fact major difficulties in dealing with these objects even in the most basic one-dimensional case. 
As we will explain below, our paper aims to address some of these difficulties
by treating the deceptively simple-looking quadratic model \eqref{KG} under fairly general assumptions.
Note that models with quadratic nonlinearities, such as \eqref{KG}, also
arise in the linearization of quadratic equations
(e.g. water waves, Euler-Poisson, Zakharov etc.) around (localized) soliton solutions.

Furthermore, the study of asymptotic stability (or instability) of solitons
- as opposed to orbital or local asymptotic stability -
is motivated by problems in the theory of quasilinear equations, 
where this is often the only relevant type of stability that one can hope to achieve,
since the equations are usually not even locally well-posed in the energy space.

Before describing our result in more details, let us briefly mention some important aspects of our paper:

\setlength{\leftmargini}{1.5em}
\begin{itemize}
 
\smallskip
\item We can treat a large class of equations provided that 
the property $\wt{u}(0,t)=0$ holds; here $\wt{u}$ denotes the distorted Fourier transform of $u$.
Under this sole assumption 
we need to allow for a {\it loss of regularity in Fourier space} of our solutions.  
This loss of regularity was previously observed in some $2$d (unperturbed, i.e. with no potential) 
models \cite{DIP,DIPP};
in the $1$d case under consideration it is caused
by a coherent phenomenon, i.e., a full (space-time) nonlinear resonance,
which appears because of the potential.
See Subsection \ref{introQR} for more on this.

\medskip
\item Loss of regularity in Fourier space is expected to be a crucial phenomenon in dimension one. 
First, it should occur generically due to resonant nonlinear interactions within the continuous spectrum.
Also, singularities can arise through the coupling of internal modes of oscillations (discrete spectrum) 
and the continuous spectrum
through the `Fermi golden rule' \cite{Sigal,SofWein2};
furthermore, they can appear due to zero energy resonances of the linear(ized) operator.\footnote{See 
\S\ref{sssecappl} for more on internal modes, and the discussion after \eqref{introQR2}
for more on the zero energy resonance.}


\medskip
\item Our global stability and decay result for \eqref{KG} has direct applications 
to the stability of stationary states of nonlinear evolution problems, under additional symmetry assumptions,
when restricting the nonlinear interactions to the continuous spectrum; see Subsection \ref{ssecappl}.
We also obtain full asymptotic stability for 
certain families of kinks of the double Sine-Gordon equation (a generalized sine-Gordon theory); see \S\ref{ssecDSG}.
 
 \medskip
 \item We believe that our treatment of \eqref{KG} helps clarifying 
 the interconnected roles of the zero energy resonances, symmetries of the equation, 
 and low frequency behavior (or improved local decay) in the study of global space-time asymptotics; 
 see for example the discussion in \S\ref{ssecSG}. 
%

\medskip 
\item More generally, we believe that the approach laid out in this paper 
enables a precise analysis of the nonlinear interactions of perturbed waves that are localized, 
yielding optimal results, as far as decay is concerned, for instance. 
In this respect, it goes beyond classical methods which rely on dispersive or Strichartz estimates, 
or virial-type identities.
\end{itemize}

\smallskip
\subsection{Previous results} 
\label{introres}

\subsubsection{Methods for solitons and topological solitons}
The literature on soliton stability is extensive and a complete overview is beyond the scope of this paper, 
and our abilities.
We refer to the excellent surveys \cite{Tao09,Sof06,Sch1} and the book \cite{PDbook} and references therein.

One immediately noticeable difference between solitons, 
that are spatially localized, and topological solitons, that are not,
is in the linearized equations.
In fact, since topological solitons do not decay to zero, lower order nonlinear terms are typically
powers of the small perturbation times a non-decaying coefficient; see \eqref{linearizationK} as an example.
This lack of localization prevents the efficient use of improved local decay type estimates, which are often a key tool
when dealing with (standard) solitons.

In general, the treatment of low power nonlinearities (in low dimensions) 
for equations with potentials is a well-known problem.
Linear dispersive tools (e.g. $L^p-L^q$ estimates for the linear group, 
Strichartz estimates, improved local decay etc\dots) 
and energy estimates are typically not enough to treat these equations.
Similar issues arise when $V=0$, but in this case one can resort to well-established methods, such as normal forms, 
vectorfields, the space-time resonance method, and multilinear harmonic analysis tools.

In the perturbed case $V\neq 0$ all these methods are not directly applicable:
the (large) potential de-correlates linear frequencies, ruling out standard normal form analysis 
and multilinear Fourier analysis,
and at the same time destroys the invariance properties of the equation, ruling out vectorfields.
To address these fundamental issues we initiated
a systematic approach based on the distorted Fourier transform,
in our work with F. Rousset \cite{GPR2} on the basic\footnote{In a perturbative and dispersive setting, 
a cubic model is substantially easier to handle than a quadratic one.
The proof of \cite{GPR2} can be adapted to a cubic KG equation with some additional observations, 
but a quadratic KG model presents substantial additional difficulties.
}
$1$d cubic NLS model with a generic potential.
In this paper we advance our theory by treating the much more complex case of \eqref{KG}.

Let us now review some of the existing literature, starting with results on flat/unperturbed $1$d Klein-Gordon equations, 
and then turning to recent advances in the treatment of perturbed equations.

\subsubsection{Klein-Gordon in the flat ($V=0$) case in dimension one}
In this case, Delort \cite{DelortKG1d} obtained small data (modified) scattering for 
quaslinear quadratic nonlinearities. Similar results were obtained in the semilinear
cubic and quadratic case, respectively in \cite{LinSofKG} and \cite{HNKG}.
In the last few years, 
some works have been dedicated to inhomogeneous models of the form 
\begin{align}\label{LSmodel}
u_{tt} - u_{xx} + u = a(x) u^2 + b(x) u^3.
\end{align}
Lindblad-Soffer \cite{LinSof} and Sterbenz \cite{Sterb} treated the case of constant $a$; see also \cite{LLS}
for a recent proof when $a=0$.
Lindblad-Soffer-Luhrman \cite{LLS20} also recently treated \eqref{LSmodel} 
under the assumption that $a$ decays to zero at infinity and either $\what{a}(\pm\sqrt{3}) = 0$,
or $\what{a}(\pm\sqrt{3}) \neq 0$ but $b=0$.
In Subsection \ref{introQR} we will discuss the key role of the frequencies $\xi= \pm\sqrt{3}$
for the evolution of solutions of \eqref{KG}.
%
%
As one of the byproducts of our main result
we also obtain globally decaying solutions with modified asymptotics for \eqref{LSmodel} 
in the case of odd initial data and a general odd $a$ and even $b$; 
see Remark \eqref{RemFlat} after Theorem \ref{maintheo}.

\subsubsection{Equations with potentials in dimension one}
In the analysis of nonlinear equations with potentials, 
the first step is to understand the dispersive properties of the 
perturbed linear operator. 
There is a vast literature on dispersive properties, such as decay estimates and Strichartz estimates;
for brevity we just refer to the classical works \cite{Journe,GolSch} and \cite{Sch1} and references therein.
The literature on linear scattering theory for Schr\"odinger operators is also substantial;
limiting ourselves to the $1$d case, we refer to
Deift-Trubowitz \cite{DeiTru}, Weder \cite{Weder1}, and the books \cite{Yafaev,Wilcox,LLbook1}. 

As discussed above, linear tools are generally not sufficient
to deal with low power nonlinearities, 
which are the ones of interest for the stability of topological solitons.
Recently, a few works have been dedicated to this situation in the one dimensional case,
see the works on cubic NLS 
\cite{DelortNLSV,IPNaumkin,GPR2,ChPu,MMS}, 
and \cite{DoKri,Catenoid}
on wave equations.

Concerning kink solutions, 
Kowalczyk, Martel and Mu\~{n}oz \cite{KowMarMun}
proved asymptotic stability locally in the energy space for odd perturbations 
of the kink of the $\phi^4$ equation \eqref{phi4};
the more classical orbital stability was proven in \cite{HPWkink,GriShaStr}.
See also the related result on KG/wave models 
\cite{KowMarMun2,KowetalNLKG}, 
the proof of local asymptotic stability for a large class of 
$1$d scalar field equations by Kowalczyk, Martel, Mu\~{n}oz and Van Den Bosch \cite{KowMarMunVDB},
and the paper of Jendrej-Kowalczyk-Lawrie \cite{Kowetal2} on kink-antikink interactions.
Full asymptotic stability for kinks of relativistic GL equations \eqref{GLU}
was proven by Komech-Kopylova \cite{KomKop,KomKop2} when $p \geq 13$. 
In a very recent paper, Delort and Masmoudi \cite{DMKink} proved long time stability 
for the kink of the $\phi^4$ model, 
reaching times of order $\epsilon^{-4}$ for data of size $\epsilon$; 
their analysis is based on a semi-classical approach using conjugation by the wave operators.
Concerning this last problem, as a consequence of our general results on \eqref{KG} we 
can obtain a global stability result (in the odd class) provided
the nonlinearity is projected onto the continuous spectrum. This latter is, of course,
an important restriction, and we do not claim any new results 
in the case of a full coupling to the internal mode.
However, we are hopeful that our techniques will be relevant in this case too;
see \S\ref{sssecappl} for more on the $\phi^4$ problem.

Finally, for results on the related problem of asymptotic stability of solitary waves for NLS,
we refer to the classical works \cite{Buslaev-Perelman,Buslaev-Sulem} 
and Krieger-Schlag \cite{KriSchNLS} and references therein. 
For supercritical NLKG see Krieger-Nakanishi-Schlag \cite{KriNakSch}.


\subsubsection{Higher dimensions}
Equations with potentials and questions about the stability of (non-topological) solitons 
in higher dimensions have also been extensively studied.
Without going too much into details, we refer the reader to
the classical results \cite{Sigal,Wein86,SofWein2,TsaiYau,Gustafson}, and the surveys 
\cite{Sof06,Sch1,SchSurvey2} and references therein.
Finally, let us mention some $3$d works that are close in spirit to ours:
\cite{GHW} laid out some basic multilinear harmonic analysis tools 
and treated the nonlinear Schr\"odinger equation in the case of a non-resonant $\bar{u}^2$ nonlinearity,
while \cite{LegerVNLS,Leger2}, respectively, \cite{VNLS3d}, 
considered the case of a small, respectively, large, 
potentials and a $u^2$ nonlinearity.

\smallskip
\subsection{Main result}\label{Ssecmt}

Let us now state our main result. 
In short, for sufficiently small and localized data (as in \eqref{initcond}),
and assuming that the distorted Fourier transform of the solution vanishes at the zero frequency,
we can construct global solutions for quadratic Klein-Gordon equations, 
that decay at the optimal (i.e. linear) rate (see \eqref{mtLinfty});
moreover, we obtain full asymptotics with modified scattering via a logarithmic phase corrections 
(see \eqref{LinfSasyf0} below).

The statement of our main theorem requires some technical definitions, 
for which we give precise references to later parts of the paper.

\begin{thm}\label{maintheo}
Let 
$$
H := -\partial_x^2 + V
$$
denote the Schr\"odinger operator, and assume it has no bound states.
Let $V=V(x)$ and $a=a(x)$ be smooth and such that $V(x)$ and $a(x) - \ell_{\pm \infty}$
and their derivatives decay super-polynomially\footnote{The smoothness and decay assumption can be relaxed.
A more careful inspection of the proof shows that only a finite (possibly large) amount of smoothness 
and polynomial decay would be sufficient.}
as $x \rightarrow \pm\infty$.

\bigskip
Consider either one of the following two equations
\begin{itemize} 
\item Either 
\begin{equation}\tag{KG} \label{mtKG}
 \partial_t^2 u + (H + 1) u = a(x) u^2
\end{equation}
under one of the following three assumptions (see Section \ref{secspth} and \S\ref{SecPot} for definitions)
\begin{itemize}
  
  \smallskip
  \item[(A)]\label{A} $V$ is generic, or
  
  \smallskip
  \item[(B)]\label{B} $V$ is exceptional and even, the zero energy resonance is even, and $a(x)$ is odd, or
  
  \smallskip
  \item[(C)]\label{C} $V$ is exceptional and even, the zero energy resonance is odd, and $a(x)$ is even.
\end{itemize}

\bigskip
\item Or
\begin{equation}\tag{KG2} \label{KG2}
 \partial_t^2 u + (H + 1) u = \sqrt{H} (a(x) u^2)
\end{equation}
under one of the following two assumptions
\begin{itemize}
	\smallskip
	\item[(D)]\label{D} $V$ is generic, or
	
	\smallskip
	\item[(E)]\label{E} $V$ is exceptional, and the distorted Fourier transform associated to $H$ 
	(defined in Subsection \ref{secFT}) of 
	the data $(u,\partial_tu)(0,x)$ is vanishing at frequency zero\footnote{It is implied here
	that the distorted transform should be continuous at zero.}. 
\end{itemize}
\end{itemize}

\medskip
Consider data at the initial time
$$
(u,\partial_t u)(t=0) = (u_0,u_1)
$$
with
\begin{align}
\label{initcond}
{\big\| (\sqrt{H+1}u_0, u_1) \big\|}_{H^4} + {\big\| \jx (\sqrt{H+1}u_0, u_1) \big\|}_{H^1} = \e_0.
\end{align}

\medskip

Then, the following holds
\setlength{\leftmargini}{1.5em}
\begin{itemize}
\medskip
\item ({\it Global existence})
There exists $\bar{\e}>0$ such that for all $\e_0\leq\bar{\e}$
the equation \eqref{KG} with initial data $(u,\partial_tu)(t=0) = (u_0,u_1)$
admits a unique global solution $u\in C(\R,H^5(\R))$.

\medskip

\item ({\it Pointwise decay}) For all $t\in \R$ 
\begin{align}\label{mtLinfty}
{\big\| \big(\sqrt{H+1}u, \partial_tu \big)(t) \big\|}_{L^\infty_x} \lesssim \e_0 (1+|t|)^{-1/2}.
\end{align}

\medskip
\item ({\it Global bounds in $L^2$ spaces})
The solution satisfies the global-in-time bounds
\begin{align}\label{mtbounds}
{\big\| u(t) \big\|}_{H^5} + {\big\| \partial_t u(t) \big\|}_{H^4}  \lesssim \e_0 \jt^{p_0},
\end{align}
for some small $p_0>0$. 
Moreover, if we define the profile
\begin{align}\label{mtprofile}
g = e^{it\sqrt{H+1}}\big(\partial_t - i\sqrt{H+1}) u
\end{align}
we have
\begin{align}\label{mtwbound}
{\big\| \jxi \partial_\xi \wt{g}(t) \big\|}_{L^2_\xi} \lesssim \e_0 \jt^{1/2+\delta},
\end{align}
for some small $\delta>0$, where $\wt{g}$ denotes the distorted Fourier transform of $g$
(as defined in \eqref{distF}; see also Proposition \ref{propFT}). 

\medskip
\item ({\it Asymptotic behavior}) 
There exists a quadratic transformation $B$ (satisfying bilinear H\"older type bounds)
such that, as $|t| \rightarrow \infty$, the ``renormalized'' profile $f:= g-B(g,g)$ scatters
to a time-independent profile up to a logarithmic phase correction.
See Remark \ref{mtremscatt} for more details.


\end{itemize}

\end{thm}

\medskip
Here are a few remarks about the statement and our main assumptions.

\begin{rem}[Vanishing at the zero frequency]\label{mtzerofr0}
Hypotheses (A), (B), (C) for the equation \eqref{mtKG}, and hypothesis (D) and (E) for
\ref{KG2} are ways of ensuring that $\widetilde{f}(0)=0$, where $\widetilde{f}$ is the 
distorted Fourier transform of $f$ associated to the operator $H$; see Section \ref{secspth}
for the definitions and \eqref{RemZerofreq} below for the vanishing property. 
The zero frequency for the distorted Fourier transform is linked to a resonant phenomenon, 
hence the necessity for the cancellation $\widetilde{f}(0)=0$ for our proof to apply;
see the discussion in Subsection \ref{introQR}.
\end{rem}

\begin{rem}\label{mtzerofr}
In the course of our proof we will work (most of the time) just with the assumption that $\wt{f}(0)=0$,
so to be able to treat all cases in a unified way. 
In particular, we will carry out all our main estimates for \eqref{mtKG}, but everything can be 
easily adapted to\eqref{KG2}.
In some instances we will need to distinguish between the different cases, e.g. (A) vs. (B),
and will specify when this is so (see, for example, the proof of Lemma \ref{lemwtf0}).
\end{rem}

\begin{rem}
Theorem \ref{maintheo} remains true if the operator $H$ is allowed to have bound states, but the data and 
the nonlinearities 
in \eqref{mtKG} and \eqref{KG2} are projected on the continuous spectrum of the operator.
\end{rem}

\begin{rem}\label{remparity}
Note that the parity assumptions in (B), respectively (C), imply that the solutions are odd, respectively even.
However, in the case of \eqref{KG2} no parity assumptions are needed.

Moreover, Theorem \ref{maintheo} remains valid if one includes cubic and higher order terms in \eqref{mtKG}
provided this is done by keeping the proper parity. For example, in both cases (B) and (C) 
one can add a term $b(x)u^3$ to \eqref{KG} with an even and sufficiently regular (but not necessarily decaying) $b$.
Similarly, one can add any cubic or higher order terms to \eqref{KG2} inside the parentheses on the right-hand side.
\end{rem}



\medskip
Let us now make some remarks about our results and some of its implications.
More specific applications are discussed in \S\ref{sssecappl}.

\setlength{\leftmargini}{2em}
\begin{enumerate}

\medskip
\item {\it Assumptions on the potential: generic and exceptional}.\label{mtremV} 

\smallskip
\noindent
The assumption that $V$ is generic is the following:
\begin{align}\label{mtr1}
\int_\R V(x) \, m(x) \, dx \neq 0
\end{align}
where $m$ is the unique solution of $(-\partial_{x}^2 + V)m = 0$ 
with $\lim_{x\rightarrow \infty}m(x) = 1$. 
One can see that \eqref{mtr1} is equivalent to the condition that the
transmission coefficient $T$ (see \eqref{TRformula} for the definition) satisfies $T(0) = 0$. 
This is also equivalent to the fact that the $0$ energy level is not a resonance,
that is, there does not exist a bounded solution in the kernel of $H$; See Lemma \ref{lemVgen}.
A non-generic potential is called `exceptional'.



\medskip
\item {\it The zero frequency and symmetries}.\label{RemZerofreq}

\smallskip
\noindent
For generic $V$ one has that $\wt{f}$ is continuous everywhere for $f \in L^1$,
and $\wt{f}(0) = 0$. See the remarks after Proposition \ref{propFT}.
In the case of exceptional potentials one does not have continuity of $\wt{f}$ at $0$ in general.
Continuity holds if $T(0)=1$ or, equivalently, $a:=m(-\infty)=1$, 
since
\begin{align}\label{wtf0+-}
\wt{f}(0+) = \frac{2a}{1+a^2} \frac{1}{\sqrt{2\pi}} \int m(x) f(x) \,dx 
\qquad \mbox{and} \qquad \wt{f}(0-) = \frac{1}{a}\wt{f}(0+),
\end{align}
where $m$ is the zero energy resonance; see 
\eqref{wtf0exc}.
In the context of our nonlinear problem \eqref{KG} we are interested in the low frequency behavior of the 
solution and, in particular, in the vanishing of $\wt{u}(t,\xi)$ 
at $\xi=0$.
While for generic potentials we are guaranteed that indeed $\wt{u}(t,0) = 0$ for all times $t$,
in the case of exceptional $V$ we need to impose some additional (symmetry) conditions for this to hold,
as in (B), (C) or (E) of Theorem \ref{maintheo}.

Since in case (B), resp. (C) we have odd, resp. even, solutions (see Remark \ref{remparity})
then \eqref{wtf0+-} shows that, when the zero energy resonance $m(x)$ is even, resp. odd,
we indeed have $\wt{u}(t,0)=0$.
%
%

The structure of the equation might also guarantee the desired vanishing condition,
which is what we exploit for \eqref{KG2}.
Indeed, in case (E) the initial data is assumed to be such that $(\wt{u},\wt{u_t})(t=0,\xi=0) = 0$,
and this condition is preserved by the flow of \eqref{KG2}, since,
applying the distorted Fourier transform and evaluating at $\xi=0$, 
gives $\wt{u_{tt}}(t,0) + u(t,0) = 0$.

\medskip
\item {\it Improved local decay.}\label{remlocdec}  


\smallskip
\noindent
An important aspect in the study of nonlinear problems with potentials is local decay.
Roughly speaking, the potential, which is localized around the origin, 
typically reflects low energy particles away from it, leading to an improved local decay estimate of the form
\begin{align}\label{locdec}
{\big\| \jx^{-\sigma_1}  P_c e^{it\sqrt{H+1}}f  \big\|}_{L^\infty} \lesssim |t|^{-a} {\| \jx^{\sigma_2} f \|}_{L^1}
\end{align}
for some $\sigma_1,\sigma_2 > 0$, and a rate of decay $a$ larger than $1/2$,
which is the optimal one for general linear waves. 
$P_c$ in \eqref{locdec} denotes the projection to the continuous spectrum of $H$.
While we do not directly make use of estimates like \eqref{locdec} 
we do rely on the dual improved behavior for small frequencies.

For generic potentials it can be shown that \eqref{locdec} holds with 
$a=3/2$ and $\sigma_2 = 1$ \cite{KriSchNLS,Sch1} (the value of $\sigma_1$ is unimportant for this discussion);
such an estimate is essentially equivalent to (and scales like)
\begin{align}\label{locdecgen}
{\big\| \jx^{-\sigma_1}  P_c e^{it\sqrt{H+1}}f  \big\|}_{L^\infty} \lesssim |t|^{-1} {\| \jx f \|}_{L^2}.
\end{align}

To see the difference with the exceptional case, it suffices to consider the flat case $V=0$.
From a stationary phase expansion one sees that linear solutions satisfy, as $t \to \infty$, 
\begin{align}\label{locdec0}
e^{it \langle \partial_x \rangle} f \approx \frac{e^{i \frac{\pi}{4}}}{\sqrt{2t}} 
  \langle \xi_0 \rangle^{3/2} e^{i t \langle \xi_0 \rangle + i x \xi_0} \what{f} \left( \xi_0 \right), 
  \qquad \frac{\xi_0}{\langle \xi_0 \rangle} := - \frac{x}{t},
\end{align}
where $\what{f}$ is the regular Fourier transform.
Thus, there is no improvement to the local decay rate unless  $\what{f}(0)=0$.
However, in general, 
the next term in the expansion is only of the order of 
$|t|^{-3/4}{\| \jx f\|}_{L^2}$.
The difference between this and the faster $|t|^{-1}$ decay in \eqref{locdecgen} 
turns out to be a major issue when dealing with \eqref{KG}, under our very general assumptions.

Local decay is stronger also for exceptional potentials if, 
in addition to $\wt{f}(0) = 0$, further cancellations occur due to symmetries.
%
%
This suggests the possibility of simplifications 
to parts of our arguments if one of the assumptions (A), (B), or (C) in Theorem \ref{maintheo}
holds.
In particular, one may be able to adopt a less refined functional framework than the one we use here 
(see Subsection \ref{ssecff}).

\medskip
\item {\it The functional framework and degenerate norms.} \label{RemDeg}

\smallskip
\noindent
To deal with an example such as \eqref{KG2} 
where only $\widetilde{f}(0) = 0$ 
can be assumed, we need to pay particular attention to a phenomenon of loss of regularity in frequency space. 
As we explain in Subsection \ref{introQR}, when the distorted frequency $\xi$ approaches $\pm\sqrt{3}$ 
the $L^2$ weighted norm of the (renormalized) profile $f$ 
becomes singular. 
We then need to use a norm which captures this degenerate behavior; see \eqref{wnorm}.


It is important to point out that, while some of the complications may be avoided by making less general assumptions,
we expect that degenerate norms like the one we use in this paper
will play a key role when internal modes (positive eigenvalues of $H+1$) are present,
as well as when considering general (non-symmetric) solutions.

\medskip
\item {\it Violating the zero frequency condition}
\label{RemZerofreq'}

\smallskip
\noindent
The above discussion emphasized the technical reasons leading to the requirement 
that the solution of \eqref{KG} vanishes at zero frequency.
The works \cite{LLS20,LLSS} address a set up where the coefficient $a(x)$ is localized, 
but the solution does not have to vanish at zero in (distorted) Fourier space. 
In these papers, it is showed that the decay in time slows down by a logarithmic factor 
compared to the linear case; see also the discussion at the end of Subsection \ref{introQR}. 
Since the linear decay rate was already critical at the level 
of the cubic interaction, this additional logarithm is expected to make the nonlinear 
analysis of the full problem (including cubic terms, or a non-decaying $a$) extremely delicate.
%


\medskip
\item \label{mtremscatt} {\it Modified asymptotics}. 

\smallskip
\noindent
In the last point of Theorem \ref{maintheo} we state that a renormalized profile $f=g-B(g,g)$ 
undergoes modified scattering.
Let us postpone for the moment the exact definition of $f$, and just think of $B(g,g) \approx g^2$. 
For the profile $f$ we prove the following asymptotic formula:
there exists an asymptotic profile $W^{\infty} = (W^{\infty}_+, W^{\infty}_-) 
\in \big(\jxi^{-3/2}L^\infty_\xi\big)^2$ such that,
for $\xi>0$,
\begin{align}\label{LinfSasyf0}
\begin{split}
&  \big(\wt{f}(t,\xi),\wt{f}(t,-\xi)\big)
	\\ & = S^{-1}(\xi) \exp\Big( -\frac{5i}{12}  \mathrm{diag}\big( \ell_{+\infty}^2\big| W^{\infty}_+(\xi) \big|^2, 
	\ell_{-\infty}^2\big| W^{\infty}_-(\xi) \big|^2 \big) \log t  \Big) W^{\infty}(\xi)
	+ O\big(\e_0^2 \jt^{-\delta_0}\big)
\end{split}
\end{align}
as $t \rightarrow \infty$, for some $\delta_0>0$;
here $S(\xi)$ is the scattering matrix associated to the potential $V$ defined in \eqref{scatmat}.
As $t \rightarrow -\infty$, using the time reversal symmetry, 
one obtains a similar (in fact, simpler) formula that resembles the flat case.
While this correction to scattering is most naturally viewed in distorted Fourier space, 
it translates to physical space by standard arguments.
Note that, because of the potential $V$, this logarithmic phase correction
depends on the scattering matrix $S$ (at least in one time direction),
and `mixes' positive and negative frequencies.
We refer the reader to Proposition \ref{propLinfS} and the comments after it 
for more details.

The phenomenon of modified scattering by a logarithmic phase correction is
one of the fundamental types of nonlinear phenomena that one may observe for scattering critical (long-range) equations.
We refer the reader to the papers on NLS \cite{HN,LinSof,KP,IoPu1}
and on KG 
\cite{DelortKG1d,HNKG,LLS20} 
where this type of modified scattering is proved using various approaches for equations without potentials. 
For equation with potentials, see the already cited \cite{DelortNLSV,IPNaumkin,GPR2,ChPu}.




\medskip
\item {\it Assumptions on the data}. \label{mtremdata}
 
\smallskip
\noindent
The assumptions in \eqref{initcond} are quite standard for these type of problems.
Finiteness of the weighted norm guarantees $|t|^{-1/2}$ pointwise decay for linear solutions. 
Propagating a suitable weighted bound for all times will be one of the main goals of our proof.
For the profile $g$, we can only propagate the weak bound \eqref{mtwbound}, 
while we will be able to control a stronger weighted norm of $f$.

A certain amount of Sobolev regularity is helpful in many parts of the proof when we deal with high frequencies.
However, although \eqref{KG} is a semilinear problem, it seems to us that it is not straightforward to propagate
any desired amount of Sobolev regularity, unlike in many other similar problems.
This is essentially due to the fact that the nonlinearity contains quadratic terms
which cannot be eliminated by normal forms and that (localized) decay is at best $|t|^{-3/4}$ in the absence of symmetries.

\medskip
\item {\it Global bounds and bootstrap spaces.} \label{rem4} 

\smallskip
\noindent
Most of our analysis is performed in the distorted Fourier space.
The main task is to prove a priori estimates in suitably constructed spaces 
for a {\it renormalized profile} obtained after a partial normal form transformation.
This is the profile\footnote{In the course of the proof we will denote the 
bilinear transformation $B$ by the letter $T$ (see the definition of $f$ in \eqref{Renof}-\eqref{RenoT} with 
$g$ defined in \eqref{vKG}-\eqref{vprof}. We use the different notation $B$ in the main theorem and this intro
to avoid any confusion with the transmission coefficient $T$ (see \eqref{TRformula}) here.
In later parts of the paper the distinction should be clear from the context.}
$f:=g-B(g,g)$ alluded to in the main Theorem.
We refer the reader to Section \ref{secdec}, and  in particular to Subsection \ref{SsecReno}, for the definition of $f$. 

The profile $f$ is measured in three norms: a Sobolev norm (like $g$), a weighted-type norm which incorporates 
the degeneration close to the bad frequencies $\pm\sqrt{3}$, 
and the sup-norm of its distorted Fourier transform. 
We refer to Subsection \ref{ssecff} for details about the functional framework,
and to the beginning of Section \ref{secBoot}
for the main bootstrap propositions on $f$ and $g$. 





\medskip
\item {\it The flat case}. \label{RemFlat}

\smallskip
\noindent
For the sake of explanation, it is interesting to consider \eqref{KG} in the simplified case $V=0$
\begin{equation}\label{RemFlatKG}
\partial_t^2 u + (- \partial_x^2 + 1) u = a(x) u^2,
\end{equation}
where $a(x)$ is odd and fast approaching $\pm \ell$ as $x\rightarrow \pm\infty$. 
Cubic terms of the form $u^3$ and $b(x)u^3$ (with $b$ even) can be included in the model.
For \eqref{RemFlatKG}
our result gives globally-decaying solutions for odd initial data. 
However, as discussed in Remark \eqref{remlocdec} above, 
this specific case of odd symmetry 
is simpler due to faster local decay.
A related, and more difficult, toy model that we can include in our treatment is (see \eqref{KG2}) 
\begin{equation}\label{RemFlatKG'}
\partial_t^2 u + (- \partial_x^2 + 1) u = \partial_x( a(x) u^2 )
\end{equation}
with zero average initial data.
Note that symmetries are not needed here, and
other variants are possible provided the zero average condition is preserved.

As mentioned after \eqref{LSmodel}, the flat case \eqref{RemFlatKG} with non-symmetric localized data,
and decaying coefficient $a(x)$, was treated 
in \cite{LLS} where a logarithmic slowdown of the decay rate was also shown to occur.
Cubic terms are also included in the results of \cite{LLS} provided $\what{a}(\pm\sqrt{3})=0$.
The general case of \eqref{RemFlatKG} without symmetries and with non-decaying $a(x)$ is still open.



\end{enumerate}




\smallskip
\subsection{Applications} 
\label{ssecappl}

In this subsection we discuss the relevance of our results to questions on the asymptotic stability 
of stationary solutions for several important physical problems.
We will be considering one dimensional scalar field theories
$$
\partial_t^2 \phi - \partial_x^2 \phi + U'(\phi) = 0
$$
deriving from the Hamiltonian
\begin{align}\label{phi4H}
\mathscr{H} = \frac{1}{2} \int \left( \phi_t^2 + \phi_x^2 \right) \, dx + \int U(\phi)  \,dx.
\end{align}
Choosing the potential $U$ with a double-well (Ginzburg-Landau) structure, 
special solutions connecting stable states at $\pm \infty$, known as kinks, emerge. 
The question of their stability, or asymptotic behavior, depends very delicately on the potential $U$, 
and leads to a wealth of interesting mathematical problems. 
Our analysis sheds light on this question for various models, some of which we review below.

\smallskip
\subsubsection{The $\phi^4$ model}\label{sssecappl}
This fundamental model corresponds to the choice 
$$
U(\phi) = U_0(\phi) = \frac{1}{4}(1-\phi^2)^2, 
$$
leading to the equation
\begin{align}
\label{phi4}
\partial_{t}^2 \phi - \partial_x^2 \phi = \phi - \phi^3,
\end{align}
which admits the kink solution $K_0(x) = \tanh(x/\sqrt{2})$. 
Setting $\phi = K_0+v$, where $v$ is a small (localized) perturbation, we see that
\begin{equation}
\label{linearizationK}
(\partial_t^2 + H_0 + 2)v = - 3K_0 v^2 - v^3, \qquad H_0 := -\partial_x^2 + V_0, 
  \qquad V_0(x) := - 3 \mathrm{sech}^2(x/\sqrt 2).
\end{equation}
It is known, see \cite{Cuccagna,PDbook,LLS}, 
that the spectrum of the Schr\"odinger operator $H_0$ has the following structure:
the $-2$ eigenvalue corresponding to the translation symmetry, 
an even zero energy resonance (a bounded solution of $H\psi = 0$),
and the eigenvalue $\lambda_1=-1/2$ corresponding to an odd exponentially decaying eigenfunction $\psi_{-1/2}$.
The latter is the so-called internal mode.
For the sake of explanation, 
let us restrict our attention to the subspace of odd functions.\footnote{On the one hand this has the practical advantage 
to avoid modulating the kink in order to track the motion of its center.
On the other hand, at a deeper level, oddness suppresses the even resonance which otherwise would have to be dealt with.}
By projecting onto the discrete and continuous modes one can decompose
$v = c_0(t) \psi_{-1/2} + P_c \, u(t,x)$,
where $P_c$ is the projection onto the continuous spectrum of $H_0$,
and obtain the equation $(\partial_t^2 + H_0 + 2)u = P_c(-3K_0 v^2 - v^3)$ for the radiation component.
One is then naturally led to analyzing the ``continuous subsystem"
\begin{align}\label{phi4linmodel}
\begin{split}
(\partial_t^2 + H_0 + 2)u & = P_c \left( -3K_0 u^2 - u^3 \right).
\end{split}
\end{align}
Since $V_0$ and its zero energy resonance are even, 
our results apply to show global bounds and decay for \eqref{phi4linmodel} with odd data.

Thus, we are able to settle at least part of the kink stability problem; 
the remaining difficulty, in the odd case, is to prove that the coupling of the internal 
mode to the continuous spectrum causes the energy of the internal mode to be dispersed 
through the phenomenon of `radiation damping' \cite{SofWein2,DMKink}. 
This is a serious obstacle, since the 
presence 
of the internal mode leads to the formation of a singularity in distorted Fourier space, 
at the frequency given by the Fermi golden rule. 
However, notice that a very 
similar phenomenon is dealt with in the present paper, 
with the formation of a singularity at the distorted frequencies $\pm\sqrt{3}$.

For general data, 
one has to deal with the resonance at zero frequency,
which should at least lead a logarithmic slowdown of the decay,
as observed in \cite{LLS20,LLSS}.
Since the decay rate is already critical for the cubic nonlinearity, 
this makes this question extremely delicate.

\smallskip
\subsubsection{The Sine-Gordon equation
}\label{ssecSG}

Choosing $U(\phi) = U_{SG}(\phi) = 1 - \cos \phi$
in \eqref{phi4H} gives the Sine-Gordon equation
\begin{align}\label{SG}
\partial_{t}^2 \phi - \partial_x^2 \phi + \sin\phi = 0
\end{align}
which is integrable and admits the kink solution $K_\mathrm{SG}(x) = 4 \arctan(e^x)$, \cite[Chapter 2]{PDbook}. 
Setting $\phi = K_{\mathrm{SG}}+v$, the perturbation $v$ solves
\begin{equation}\label{SGlin}
\partial_t^2 v + (H_\mathrm{SG} + 1)v = (\sin K_\mathrm{SG}) v^2 + O(v^3), 
  \qquad H_\mathrm{SG} = -\partial_x^2 - 2 \operatorname{sech}^2(x).
\end{equation}
$H_\mathrm{SG}$ has no internal mode (only the eigenvalue $\lambda = -1$ associated with the translation invariance),
and it is exceptional, but with an {\it odd zero energy resonance};
thus, the distorted Fourier transform of an odd function does not vanish at zero energy.
Therefore, despite its similarities with the $\phi^4$ model, 
\eqref{SG} equation does not a priori fall into the class of equation that we can treat with our approach.

The asymptotic stability of the kink could however be proved by means of inverse scattering by Chen, Liu and Lu \cite{CLL}, 
since the sine-Gordon equation is completely integrable. After the first version of the present paper appeared online, another proof of the asymptotic stability of the kink was published by L\"uhrmann and Schlag~\cite{LSSineG}; their beautiful and (relatively) short paper avoids the use of inverse scattering, or of the distorted Fourier transform. They rely on two key observations: on the one hand, the linearized operator around the kink can be factorized in a very convenient way; and on the other hand, the nonlinear coupling of the resonance to the continuous spectrum is canceled by the specific form of the equation. In hindsight, we believe that the latter observation would allow us to treat the sine-Gordon problem within the framework developed in the present paper.

\subsubsection{The double sine-Gordon equation}\label{ssecDSG} 
More interestingly, our results apply to the perturbation of \eqref{SG}
given by the {\it double sine-Gordon} model
\begin{align}\label{DSG}
\partial_t^2 \phi - \partial_x^2 \phi + U'_{DSG}(\phi) = 0, \qquad 
U_{DSG}(\phi) = \frac{1}{1+|4\eta|} \big[ \eta(1-\cos \phi) + 1 + \cos\big(\tfrac{\phi}{2}\big) \big],
\end{align}
where $\eta \in \R$. This model is not integrable for $\eta \neq 0$;
see \cite{PDbook} and Campbell-Peyrard-Sodano \cite{DSG86}, and references therein, 
also for a description of the various physical contexts where \eqref{DSG} has been classically used.
For $\eta < 0$ we obtain asymptotic stability results for kinks of \eqref{DSG}. 
More precisely, there are two ranges of the parameter $\eta$ with corresponding family of kinks
that we can consider:

\setlength{\leftmargini}{1.5em}
\begin{itemize}
\smallskip
\item[1.] For $-1/4 < \eta < 0$, \eqref{DSG} has (up to symmetries) 
a single odd kink connecting the minima of the potential $\pm2\pi$;
let us call this kink $K_1$.

\smallskip
\item[2.] For $\eta <-1/4$, \eqref{DSG} has an odd kink connecting the minima of the potential 
$\pm\phi_0$ with $\cos(\phi_0/2) = 1/4\eta$; let us denote this kink by $K_2$. 
There is also another kink in this range of $\eta$ that we do not consider since we cannot apply our results to it.
\end{itemize}


We have the following asymptotic stability of the $K_1$ and $K_2$ kink solutions for odd perturbations:

\begin{cor}\label{corDSG}
Consider \eqref{DSG} with $\eta \in (-1/4,0)$, respectively $\eta < -1/4$,
with an initial condition of the form 
$(\phi,\phi_t)(0,x)= (K_1(x),0) + (u_{1,0}(x),u_{1,1}(x))$,
respectively 
$(\phi,\phi_t)(0,x)= (K_2(x),0) + (u_{2,0}(x),u_{2,1}(x))$.
Assume that $(u_{i,0},u_{i,1})$, $i=1,2$ are odd and satisfy the same smallness condition in \eqref{initcond}.
Then, the associated global solution $\phi$ can be written as 
\begin{align*}
\phi(t,x) = K_i(x) + u_i(t,x)
\end{align*}
where $u_i$ decays globally on $\R$ as in \eqref{mtLinfty}, satisfies the bounds \eqref{mtbounds},
and has the 
asymptotic behavior described in \eqref{LinfSasyf0}\footnote{
In this case $\ell_{\pm \infty}$ can be explicitly calculated from the values of 
$\partial_\phi^\ell U(K_i(\pm\infty))$ for $\ell=3,4$.}
\end{cor}

\begin{proof}

We let $\phi = K_i + v$, with $i=1,2$ and denote $U_1 := U_{DSG}$ when $-1/4 < \eta < 0$,
and $U_2 := U_{DSG}$ when $\eta < -1/4$.
Then, from \eqref{DSG} we get 
\begin{align}\label{linearizationKi}
\begin{split}
& (\partial_t^2 + H_i + m_i^2)v = - U_i'(K_i+v) + U_i'(K_i)
  + U_i''(K_i) v 
\\
& H_i = -\partial_x^2 + V_i, \qquad V_i(x) = U_i''(K_i)-m_i^2,
\qquad m_i^2 := \lim_{x\rightarrow \pm \infty} U_i''(K_i) > 0.
\end{split}
\end{align}
More precisely, $m_1^2 = (1-4\eta)^{-1}(\eta+1/4)$ and $m_2^2 = 1/(16\eta) - \eta$. 

It can be shown that {\it $H_i$ is generic and has no eigenvalues}, except the translation mode;
see the Appendix \ref{Appendix} for a short proof relying on the arguments of \cite[Section 5.6]{KowMarMunVDB}.
In particular, the assumptions of Theorem \ref{maintheo} 
hold for odd solutions of \eqref{linearizationKi}.
The conclusions of Theorem \ref{maintheo} applied to $v$ 
then imply the statement of this corollary.
\end{proof}

For the double sine-Gordon model \eqref{DSG} in the same range of $\eta$ above
(and also for several other scalar field models with the same spectral properties)
Kowalczyk-Martel-Mu\~{n}oz-Van Den Bosch \cite{KowMarMunVDB} proved local asymptotic stability in the
energy space. Compared to this latter result, Corollary \ref{corDSG} gives asymptotic
stability on the full real line, and modified scattering, provided the data is (mildly) localized and odd.

\smallskip
\subsubsection{General relativistic Ginzburg-Landau theories}\label{ssecGLintro}
Our approach and results apply similarly to general relativistic Ginzburg-Landau theories, 
where the potential in \eqref{phi4H} is taken to be of double-well type, with the following expansion at the minima $\pm a$
\begin{align}\label{GLU}
U(\phi) = U_\mathrm{GL}(\phi) = \frac{1}{2}(|\phi|-a)^2 + O\big((|\phi|-a)^{p+1}\big), \qquad p \geq 2.
\end{align}
The corresponding equations $\phi_{tt} - \phi_{xx} + U_{GL}'(\phi) = 0$
admit kink solutions $K_\mathrm{GL}$ exponentially converging to $\pm a$ at $\pm \infty$; 
see \cite{KomKop,KomKop2,Kowetal2}. 
The dynamics for the perturbation $v$ (up to a standard modulation if necessary) becomes
\begin{align}\label{linearizationK'}
\begin{split}
& (\partial_t^2 + H_\mathrm{GL} + 1)v = -U_{GL}'''(K_\mathrm{GL}) v^2 + \frac{1}{2} U_{GL}^{(4)}(K_\mathrm{GL}) v^3 
  + O(v^4), 
\\
& H_\mathrm{GL} = -\partial_x^2 + V_\mathrm{GL}, \qquad V_\mathrm{GL}(x) = U_\mathrm{GL}''(K)-1.
\end{split}
\end{align}
In analogy with the discussion on the $\phi^4$ model, 
our analysis can be applied directly to the ``continuous subsystem"  
(the analogue of \eqref{phi4linmodel}) which takes the form
\begin{align}\label{linearizationK'c}
(\partial_t^2 + H_{GL} + 1)u = P_c \big( -U_{GL}'''(K_{GL}) u^2 
  + \frac{1}{2} U_{GL}^{(4)}(K_{GL}) u^3 + O(u^4) \big).
\end{align}

If one assumes that the minima of the well are sufficiently flat, or in other words that $p$ is sufficiently big, 
the coefficients $U_\mathrm{GL}^{(k)}(K_{GL})$, $3\leq k\leq p+1$, become exponentially decaying, 
and this simplifies considerably the nonlinear analysis. 
Komech-Kopylova fully analyzed the radiation-damping phenomenon associated to the internal mode,
and obtained asymptotic stability in \cite{KomKop} for $p \geq 14$.
While Komech-Kopylova required a large $p$, 
the methods introduced in the present paper certainly allow the treatment 
of smaller values of $p$ (e.g., one should be able to comfortably reach $p=5$, that is, a non-localized quintic nonlinearity).

\smallskip
\subsubsection{The Nonlinear Klein-Gordon equation}\label{ssecNLKG}
This final example 
involves localized solitons.  The potential
$$
U(\phi) = U_{p}(\phi) = \frac{1}{2} \phi^2 - \frac{1}{p+1} \phi^{p+1}
$$
gives the $1+1$ focusing nonlinear Klein-Gordon equations
\begin{align}\label{NLKG}
\partial_{t}^2 \phi - \partial_x^2 \phi + \phi = \phi^p,
\end{align}
for $p=2,3,4,\dots$.
These admit the soliton solution
\begin{align}\label{NLKGSol}
Q(x) = Q_p(x):= (\alpha+1)^\frac{1}{2\alpha} \mathrm{sech}^{1/\alpha}(\alpha x), \qquad \alpha := \tfrac{1}{2}(p -1).
\end{align}
By assuming even symmetry we may neglect the soliton manifold obtained under Lorentz transformations.
The equation for the perturbation $v$ ($\phi=Q+v$) is
\begin{align}\label{linearizationQ}
\begin{split}
& (\partial_t^2 + H_{p} + 1)v 
  = \frac{1}{2}p(p-1) Q^{p-2} v^2 + \cdots + v^p. 
\\
& H_{p} := -\partial_x^2 + V_{p}, \qquad V_{p}(x) := -pQ^{p-1}.
\end{split}
\end{align}
It is known that $H_{p}$ has a negative eigenvalue at $-\alpha(\alpha+2)-1$, which makes the soliton unstable.
However, besides this and the $-1$ eigenvalue associated to the translation invariance, 
$H_{p}$ has no other negative eigenvalues 
when $p>3$ \cite{CGNT,KowetalNLKG}.
Note that when $p=3$, $H_3$ coincides (up to a rescaling) with $H_0$, see \eqref{linearizationK}; 
since the resonance is even our results do not apply to the corresponding continuous subsystem.

When $p=2$ instead, the linearized operator $H_2$ has an odd resonance.
Therefore, 
asymptotic stability holds for 
small even solutions of the continuous subsystem 
\begin{align}\label{linearizationQ2}
(\partial_t^2 + H_2 + 1)u = P_c\, u^2. 
\end{align}

A natural question for \eqref{NLKG} is the construction of stable 
manifolds for solutions suitably close to the soliton, 
and the asymptotic stability of the subclass of global solutions.
For $p>5$ 
this was done by Krieger-Nakanishi-Schlag \cite{KriNakSch}.
More recently, \cite{KowetalNLKG} proved a conditional 
asymptotic stability result locally in the energy space for global solutions.
For $p\leq 5$ the problem of full asymptotic stability appears to be still open.
A serious obstacle to the construction of a stable manifold is to 
prove a robust small data scattering theory for low power nonlinearities. 
While this cannot be done using Strichartz-type estimates, 
which only exploit the decay of the solution, it becomes amenable to our techniques, 
which take advantage of the full resonant structure.
In particular, the cases $p=2,4$ and $5$ 
can be directly approached with our methods.
Note that even for $p=4$ (or $5$), despite the quadratic and cubic terms in the nonlinearity are localized,
one would still need to exploit oscillations in frequency space 
to deal with the weak decaying quartic (or quintic) nonlinearity.

%
%



\medskip
\subsection*{Acknowledgements}
The authors would like to thank F. Rousset for many useful discussions.

\noindent
We would also like to thank warmly the anonymous referees who gave many helpful suggestions
which substantially improved the manuscript.

\bigskip
\section{Ideas of the proof}

The starting ingredient in our approach is the Fourier transform adapted to the Schr\"odinger operator 
$-\partial_{xx} + V$, the so called distorted Fourier transform (or Weyl-Kodaira-Titchmarsh theory).
The basic idea is to try to extend Fourier analytical techniques 
used to study small solutions of nonlinear equations without potentials,
and develop new tools in the perturbed setting.

In the setting of the distorted Fourier transform, we begin by
filtering the solution by the linear (perturbed) group, and view the (nonlinear) 
Duhamel's formula as an oscillatory integral in frequency and time.
In the unperturbed case $V=0$, this point of view was
proposed in the works \cite{GMS1,G,GMS2} 
with the so-called `space-time resonance' method; see also \cite{GNTGP}.
In the past ten years this proved to be a very useful 
approach to study the long-time behavior of weakly nonlinear dispersive equations 
in the Euclidean/unperturbed setting.
As already mentioned in Subsection \ref{introres}, the presence of a potential 
introduces some fundamental differences which lead to a number of new phenomena and difficulties. 

\medskip
\subsection{Setup: dFT and the quadratic spectral distribution}
We refer to Section \ref{secspth} for a more detailed presentation of the distorted Fourier transform (dFT),
and admit for the moment the existence of generalized eigenfunctions $\psi = \psi(x,\xi)$ such that
\begin{equation}
\label{efcts}
\forall \quad \xi \in \R, \qquad (- \partial_x^2 + V ) \psi(x,\xi) = \xi^2 \psi(x,\xi),
\end{equation}
and that the familiar formulas relating the Fourier transform and its inverse in dimension $d=1$ hold if one replaces
(up to a constant)
$e^{i\xi x}$ by $\psi(x,\xi)$:
\begin{equation}
\label{DistTrans}
\widetilde{f}(\xi) = \int_\R \overline{\psi(x,\xi)} f(x)\,dx 
  \qquad \mbox{and} \qquad f(x) = \int_\R {\psi(x,\xi)} \widetilde{f}(\xi) \,d\xi.
\end{equation}

Let us consider a solution of the equation
\begin{equation*}
\partial_t^2 u + (- \partial_x^2 + V(x) + 1) u = a(x)u^2, \qquad (u,u_t)(t=0) = (u_0,u_1).
\end{equation*}
Defining the profile $g$ by
\begin{align}\label{profile}
\begin{split}
& g(t,x):=e^{it\sqrt{H+ 1}}\big(\partial_t -i \sqrt{ H + 1 } \big)u,
\qquad \wt{g}(t,\xi) = e^{it\jxi}\big(\partial_t - i\jxi \big) \wt{u},
\end{split}
\end{align}
and denoting $\wt{g}_+=\wt{g}$, $\wt{g}_-= \bar{\wt{g}}$,
one sees that $\wt{g}$ satisfies an equation of the form
\begin{align}\label{dtgintro}
\partial_t \wt{g}(t,\xi) & = - \sum_{\iota_1,\iota_2\in\{+,-\}} 
\iota_1\iota_2  \iint e^{it \Phi_{\iota_1\iota_2}(\xi,\eta,\sigma)} 
\wt{g}_{\iota_1}(t,\eta) \wt{g}_{\iota_2}(t,\sigma) 
  \, \frac{\mu_{\iota_1\iota_2}(\xi,\eta,\sigma)}{4\jeta \jsig} 
  \, d\eta \, d\sigma,
\end{align}
where the oscillatory phase is given by
\begin{equation}\label{introPhi0}
\Phi_{\iota_1 \iota_2}(\xi,\eta,\sigma) = \jxi - \iota_1 \jeta - \iota_2 \jsig,
\end{equation}
and
\begin{align}
\label{intromu0}
\mu_{\iota_1\iota_2}(\xi,\eta,\sigma) := \int a(x) 
  \overline{\psi(x,\xi)} \psi_{\iota_2}(x,\eta) \psi_{\iota_1}(x,\sigma) \, dx
\end{align}
is what we refer to as the (quadratic) ``{\it nonlinear spectral distribution}'' (NSD).

For the sake of exposition we will drop the signs $(\iota_1,\iota_2)$ from $\wt{g}$ and $\mu$ 
since they do not play any major role. We will instead keep the relevant signs in \eqref{introPhi0}
and the analogous expressions for cubic interactions. 
We also drop the factor $\jeta\jsig$ in \eqref{dtgintro}.
With this simplifications, integrating \eqref{dtgintro} over time gives
\begin{align}\label{Duhamelintro}
\wt{g}(t,\xi) = \wt{g_0}(\xi)
-i \sum_{\iota_1,\iota_2 \in \{+,-\}} \int_0^t \iint e^{is \Phi_{\iota_1 \iota_2}(\xi,\eta,\sigma)} 
  \wt{g}(s,\eta) \wt{g}(s,\sigma) \mu(\xi,\eta,\sigma) \,d\eta\, d\sigma\,ds.
\end{align}


The first task is to analyze $\mu$ in \eqref{intromu0}, and we immediately see an essential difference with the flat case $V=0$:
in the absence of a potential, the generalized eigenfunctions $\psi(x,\xi)$ should be replaced by $e^{i\xi x}$,
in which case $\mu(\xi,\eta,\s) = \delta(\xi-\eta-\s)$; 
in particular, the sum of the frequencies of the two inputs, that is, $\eta$ and $\s$, gives the output frequency $\xi$.
This can be thought of as a `conservation of momentum' or `correlation' between the frequencies.
But if $V\neq 0$, the structure of $\mu$ becomes more involved, and there is no a priori relation between the frequencies.
This can be seen as a `de-correlation' or `uncertainty' due to the presence of the potential. 

For the sake of this presentation, we can essentially think that 
\begin{align}\label{muintro1}
\begin{split}
\mu(\xi,\eta,\sigma) = \sum_{\mu,\nu \in\{+,-\}} \Big[
& A_{\mu,\nu}(\xi,\eta,\sigma) \delta(\xi + \mu \eta + \nu \s)
\\+ & B_{\mu,\nu}(\xi,\eta,\sigma) \, \pv \frac{1}{\xi + \mu \eta + \nu \s} \Big] 
+ C(\xi,\eta,\sigma),
\end{split}
\end{align}
where $A_{\mu,\nu}$, $B_{\mu,\nu}$ and $C$ are 
smooth functions 
and ``$\pv$'' stands for principal value.

The $\delta$ component of $\mu$ gives a contribution to \eqref{Duhamelintro} which is essentially the same as in the flat case, 
only algebraically more complicated due to the different signs combinations and the coefficients 
(which are related to the transmission and reflection coefficients of the potential).
One could expect to treat these terms 
as in the classical flat case,
that is, using a normal form transformation to eliminate the quadratic term in favor of cubic ones \cite{shatahKGE,DelortKG1d,HNKG}.

The $\pv$ term in \eqref{muintro1} 
seriously 
impacts the nature of the problem at hand.
When the variable $\xi + \mu \eta + \nu \s$ that determines the singularity is very small, one could think that the 
corresponding interactions are not so different from those allowed by the $\delta$ distribution, 
possibly only logarithmically worse.
When instead $\xi + \mu \eta + \nu \s$ is not too small 
we have in essence a smooth kernel.
While this might seem like a favorable situation, it is in fact a major complication.
The de-correlation between the input and output frequencies prevents the application of a normal form transformation
(quadratic terms cannot be eliminated); even more, it creates a genuinely nonlinear phenomenon of loss of regularity (in Fourier space)
at specific bad frequencies. We explain this in more details in the following paragraphs.

\smallskip
\subsection{Oscillations and Resonances: Singular vs. Regular terms}\label{ssecosc}

Let us consider the quadratic interactions in \eqref{Duhamelintro} and according to \eqref{muintro1}
write these as
\begin{align}\label{osc1}
\int_0^t \iint e^{is \Phi_{\iota_1 \iota_2}(\xi,\eta,\sigma)} \wt{g}(s,\eta) \wt{g}(s,\sigma) 
  \, \mathfrak{m}(\xi,\eta,\sigma) \,d\eta\, d\sigma \, ds,
\end{align}
where $\mathfrak{m}(\xi,\eta,\sigma)$ can be a distribution (i.e., a $\delta$ or a $\pv$)
or a smooth function.
The properties of \eqref{osc1} are dictated by the oscillations of the 
exponential factor and the structure of the singularities of $\mathfrak{m}$.
More precisely,

\begin{itemize}

\smallskip
\item If $\mathfrak{m}=\delta(\xi - \mu \eta - \nu \sigma)$ or $\mathfrak{m} = \pv \frac{1}{\xi - \mu \eta - \nu \sigma}$, 
resonant oscillations can be characterized as the stationary points of the phase 
$s \Phi_{\iota_1 \iota_2}$, \textit{restricted} to the singular hypersurface $\{ \xi - \mu \eta - \nu \sigma = 0 \}$. 
Up to changing coordinates, we can reduce to the phase
\begin{align}\label{oscphiS0}
\Phi^S_{\iota_1 \iota_2}(\xi,\eta) 
  = \langle \xi \rangle - \iota_1 \langle \eta \rangle - \iota_2 \langle \xi-\eta \rangle,
\end{align}
(where we added the superscript $S$ to emphasize that we consider a singular $\mathfrak{m}$),
for which stationary points satisfy
\begin{align}\label{oscphiS}
\Phi^S_{\iota_1 \iota_2}(\xi,\eta) = \partial_\eta \Phi^S_{\iota_1 \iota_2}(\xi,\eta) = 0. 
\end{align}
These are the classical resonances.

\smallskip
\item If $\mathfrak{m}$ is smooth we need to look at the 
{\it unrestricted} stationary points of the phase $s \Phi^R_{\iota_1 \iota_2} = s(\jxi-\iota_1\jeta-\iota_2\jsig)$ 
(where we added the superscript $R$ to emphasize that we consider a regular $\mathfrak{m}$),
that is
\begin{align}\label{oscphiR}
\Phi^R_{\iota_1 \iota_2}(\xi,\eta,\sigma) = \partial_\eta \Phi^R_{\iota_1 \iota_2}(\xi,\eta,\sigma) 
  = \partial_\sigma \Phi_{\iota_1 \iota_2}^R(\xi,\eta,\sigma) = 0. 
\end{align}

\end{itemize}

This simple and natural distinction has important implications on the behavior of \eqref{osc1}, hence on the solution of 
the nonlinear equation, which we now discuss. 

\medskip
\subsection{Regular quadratic terms and the bad frequencies}\label{introQR}
Let us first look at the case when $\mathfrak{m}$ is smooth. 
The regular quadratic phase $\Phi^R_{\iota_1 \iota_2}(\xi,\eta,\sigma) = \jxi - \iota_1 \jeta - \iota_2 \jsig$ leads to 
rather harmless interactions if $(\iota_1\iota_2) \neq (++)$ since in this case there are no solutions to \eqref{oscphiR}.
For the $(\iota_1\iota_2)=(++)$ interaction we have that
\begin{align}\label{oscRes}
\Phi^R_{++}= \partial_\eta \Phi^R_{++} = \partial_\sigma \Phi^R_{++} = 0
\quad \Longleftrightarrow \quad (\xi,\eta,\s)=(\pm\sqrt{3},0,0).
\end{align}
This is a full resonance or coherent interaction 
and it is the source of many of the difficulties.
Notice that this sort of interaction is generic in dimension $1$ in the presence of a potential,
since in \eqref{oscphiR} there are $3$ variables 
and as many equations to solve.
Obviously, a similar phenomenon would occur already in the case $V=0$ and a nonlinear term of the form $a(x)u^2$.

Recall that the classical theory of quadratic/cubic one-dimensional dispersive problems 
revolves around trying to control weighted-type norms of the form $\|x g\|_{L^2}$. 
The natural candidate in our context is then $\| \partial_\xi \wt{g}\|_{L^2_\xi}$. 
In some cases, such as \eqref{KG}, or the more standard examples of flat cubic NLS and cubic KG equations,
one knows that a uniform-in-time bound cannot be achieved due to long-range effects already present in the 
corresponding flat problem.
As the next best thing one can try to establish
\begin{align}\label{introQRas}
\| \partial_\xi \widetilde{g}\|_{L^2_\xi} \lesssim \jt^\alpha
\end{align}
for some small $\alpha >0$. 

Let us now explain how \eqref{introQRas} is incompatible with the nonlinear resonance \eqref{oscRes}.
Since our assumptions will always guarantee $\wt{g}(0) = 0$, \eqref{introQRas} implies
\begin{equation}\label{sauterelle}
|\widetilde{g}(\xi)| \lesssim \jt^\alpha |\xi|^{1/2}.
\end{equation}
Consider then the main $(++)$ contribution to the right-hand side of \eqref{osc1}, namely
\begin{align}\label{QRintro++}
\mathcal{Q}^R_{++}(t,\xi) := \int_0^t \iint e^{is \Phi_{++}(\xi,\eta,\sigma)} 
  \wt{g}(s,\eta) \wt{g}(s,\sigma) \mathfrak{q}(\xi,\eta,\sigma) \,d\eta\, d\sigma\,ds,
\end{align}
where $\mathfrak{q}$ is a smooth symbol. 
Up to lower order terms,
\begin{equation}
\label{ecureuil}
\partial_\xi \mathcal{Q}^R_{++}(t,\xi) \approx \int_0^t \iint s \frac{\xi}{\langle \xi \rangle} 
	e^{is \Phi_{++}(\xi,\eta,\sigma)} 
	\mathfrak{q}(\xi,\eta,\sigma) \wt{g}(s,\eta) \wt{g}(s,\sigma) \, d\eta\,d\sigma. 
\end{equation}
Observe that $|s \Phi_{++} | \ll 1$ if $|\xi - \sqrt{3}| + |\eta|^2 + |\sigma|^2 \ll \js^{-1}$,
and that in this region
there are no oscillations that can help.  
Thus, when $\mathfrak{q}(\pm\sqrt{3},0,0) \neq 0$, 
we are led to the following heuristic lower bound:
for $||\xi| - \sqrt{3}| \approx r$
\begin{align}\label{introQR1}
\big| \partial_\xi \mathcal{Q}^R_{++}(t,\xi) \big|
  \gtrsim \int_1^{\min(\frac{1}{r},t)} s \cdot \js^{2\alpha} \int_{|\eta|^2 + |\sigma|^2 
  \leq s^{-1}} |\eta|^{1/2} |\sigma|^{1/2} \,d\eta\,d\sigma \,ds 
  \approx \min\Big(\frac{1}{r},t\Big)^{\frac{1}{2} + 2\alpha}.
\end{align}
This implies that, if $\jt^{-1} \leq r \ll \jt^{-1/2}$,
\begin{align}\label{introQR2}
{\big\| \partial_\xi \mathcal{Q}^R_{++} \big\|}_{L^2(|\xi-\sqrt{3}| \approx r)} \gtrsim r^{-2\alpha} \gg \jt^{\alpha},
\end{align}
which is inconsistent with the bootstrap hypothesis \eqref{introQRas}. 
We then need to modify the bootstrap norm to a version of $\| \partial_\xi \wt{f} \|_{L^2}$ 
which is localized dyadically around $\pm \sqrt{3}$ and degenerates as $|\xi| \rightarrow \sqrt{3}$.
The analysis needed to propagate such a degenerate norm
turns out to be quite delicate.
A phenomenon similar to the one described above was previously observed in \cite{DIP,DIPP} 
in the two dimensional (unperturbed) setting.

Note that in the heuristics \eqref{introQR2} one would get better bounds, 
consistent with \eqref{sauterelle}, when $\mathfrak{q}(\pm\sqrt{3},0,0) = 0$.
For the model \eqref{KG2} one has
$\mathfrak{q}(\pm\sqrt{3},0,0) \neq 0$
when $V$ is non-generic (case (E)) and $\wt{a}(\pm\sqrt{3}) \neq 0$.
A true degeneracy in frequency space will then occur for these models.
For \eqref{mtKG}, under the assumptions (A) or (B) or (C), it is instead 
possible to show that $\mathfrak{q}(\xi,0,0) = 0$;
this is connected to the discussion at the end of Remark \eqref{remlocdec},
and the possibility of simplifying the functional framework in this case.

\begin{rem}\label{RemRes}
The argument above also shows that, if $\wt{g}(0) \neq 0$, then
\begin{align*}
\mathcal{Q}^R_{++}(t,\xi) \approx  \int_0^t e^{is (\jxi-2)} \mathfrak{q}(\xi,0,0) 
  \big( \wt{g}(s,0) \big)^2 \frac{ds}{s+1} + \cdots
\end{align*}
so that $\mathcal{Q}^R_{++}(t,\pm\sqrt{3})$ is logarithmically diverging 
if $\mathfrak{q}(\pm\sqrt{3},0,0) \neq 0$. 
This suggests that $\wt{g}$ is not uniformly bounded, 
which in turn implies that the solution cannot decay pointwise at the linear rate; see \eqref{locdec0}.
In the case of \eqref{RemFlatKG} with localized $a(x)$ such that $\what{a}(\pm\sqrt{3}) \neq 0$
(and no cubic terms) 
this has been rigorously proved 
in \cite{LLS20}, where the authors construct global solutions that decay
in $L^\infty_x$ at the optimal rate of $\log t /t$.
This result was then extended in \cite{LLSS} to the case of any non-generic potential
with the corresponding condition $\wt{a}(\pm\sqrt{3}) = 0 $.


Also note that $\wt{g}(0)\neq 0$ will give an asymptotic of the form 
$\partial_\xi \mathcal{Q}^R_{++}(t,\xi) \approx |\jxi-2|^{-1}$.
When localized at the scale $||\xi|-\sqrt{3}| \approx 2^\ell$, this gives an $L^2$ norm of size $2^{-\ell/2}$.
The functional framework that we will adopt does not quite allow for such a singularity,
as this would correspond to choosing the parameter $\beta=1/2$ in the definition of the norm
in \eqref{wnorm} (this is the norm in which we will measure the derivative of our (renormalized) profile in frequency space).
However, we can allow essentially any slightly less singular behavior;
this seems to suggest that a zero energy resonance may be treated by our methods at least for long-times.
\end{rem}


\medskip
\subsection{Singular quadratic and cubic terms} 
\label{introQS}

Let us now consider the quadratic interactions in \eqref{Duhamelintro} which correspond to the first two terms in 
\eqref{muintro1}.
Disregarding the irrelevant signs $\mu,\nu$ and the coefficients $A,B$, let us denote them by
\begin{align}\label{QSintro}
\mathcal{Q}^M_{\iota_1\iota_2}(t,\xi) := \int_0^t \iint e^{is \Phi^S_{\iota_1\iota_2}(\xi,\eta,\sigma)} 
  \wt{g}(s,\eta) \wt{g}(s,\sigma) M(\xi-\eta-\sigma) \,d\eta\, d\sigma\,ds, 
  \quad M \in \{ \delta,\pv\} .
\end{align}

\
\subsubsection*{The $\delta$ case}
The case of the $\delta$ distribution corresponds to the Euclidean ($V=0$) quadratic Klein-Gordon
which is not resonant (in any dimension), 
in the sense that for any $\xi,\eta \in \R$, and $\iota_1,\iota_2\in\{+,-\}$, 
\eqref{oscphiS0} never vanishes,
and more precisely
\begin{align}\label{introlb}
\big| \jxi -\iota_1 \jeta -\iota_2 \langle \xi -\eta \rangle \big| \gtrsim \min(\jxi, \jeta,\langle\xi-\eta\rangle)^{-1}.
\end{align}
This implies that the quadratic interactions $\mathcal{Q}^{\delta}_{\iota_1\iota_2}(t,\xi)$ can be eliminated by a normal form transformation.
This was first shown in the seminal work of Shatah \cite{shatahKGE} in $3$d,
and crucially used in the $1$d case in \cite{DelortKG1d} and \cite{HNKG}.
 
Applying a normal form transformation to \eqref{QSintro} 
gives quadratic boundary terms that we disregard for simplicity, and cubic terms when 
$\partial_t$ hits the profile $\wt{g}$. From \eqref{Duhamelintro}-\eqref{muintro1} we see that
these cubic terms can be of several types depending on the various combinations of convolutions between $\delta,\pv$ 
and smooth functions.
Without going into the details of these (we refer the reader to Section \ref{secdec}), 
we concentrate on the simplest interaction,
that is, the `flat' one
\begin{align}\label{CSintro}
\begin{split}
\mathcal{C}^{S}_{\iota_1 \iota_2 \iota_3}
(t,\xi)
& = \iint e^{it \Phi^S_{\iota_1 \iota_2 \iota_3}(\xi,\eta,\zeta)} \mathfrak{c}^S_{\iota_1 \iota_2 \iota_3}(\xi,\eta,\zeta) 
	\,  \wt{g}_{\iota_1}(t,\xi-\eta) \wt{g}_{\iota_2}(t,\xi-\eta-\zeta)  \wt{g}_{\iota_3}(t,\xi-\zeta) \, d\eta \, d\zeta,
\end{split}
\end{align}
with a smooth symbol $\mathfrak{c}^S_{\iota_1 \iota_2 \iota_3}$, and phase functions
$$
\Phi^S_{\iota_1 \iota_2 \iota_3} (\xi,\eta,\zeta) 
= \langle \xi \rangle - \iota_1 \langle \xi - \eta \rangle - 
  \iota_2 \langle \xi - \eta -\zeta \rangle -\iota_3 \langle \xi - \zeta \rangle. 
$$
We observe that if $\{\iota_1,\iota_2,\iota_3\} \neq \{+,+,-\}$, the equations
$\partial_\eta \Phi^S_{\iota_1 \iota_2 \iota_3} 
= \partial_\zeta \Phi^S_{\iota_1 \iota_2 \iota_3} = \Phi^S_{\iota_1 \iota_2 \iota_3} = 0$
have no solutions,
and therefore the case $\{\iota_1,\iota_2,\iota_3\} = \{+,+,-\}$ is the main one.
If we look at the $(+-+)$ phase for simplicity,
we see that, for every fixed $\xi$,
$$
\Phi^S_{+-+} = \partial_\eta \Phi^S_{+-+} = \partial_\zeta \Phi^S_{+-+} = 0 \quad \Longleftrightarrow \quad \eta = \zeta = 0.
$$
This resonance is responsible for the logarithmic phase correction appearing in \eqref{LinfSasyf0}.
We refer the reader to \cite{KP,IoPu1,GPR2} 
where a similar phenomenon has been dealt with.
We should point out however that, in our case, 
the asymptotic behavior \eqref{LinfSasyf0} is slightly harder to capture
because of the degenerate weighted norm, and of the algebraic complications 
due to the treatment of potentials with general transmission and reflection coefficients.

\subsubsection*{The $\pv$ case}
The main observation that allows us to treat the terms 
$\mathcal{Q}^\pv_{\iota_1\iota_2}$ 
is the following: when $|\xi-\eta-\s|$ 
is much smaller than the right-hand side of \eqref{introlb} these terms 
are similar to $\mathcal{Q}^S_{\iota_1\iota_2}$.
When instead $|\xi-\eta-\s|$ is away from zero, the symbol in \eqref{QSintro} is actually smooth,
which gives a term like the regular $\mathcal{Q}^R_{\iota_1\iota_2}$ discussed before.

\medskip
\subsection{The functional framework}\label{ssecff}
To measure the evolution of our solutions we need to take into account various aspects including
pointwise decay, spatial localization (which we measure through regularity on the distorted Fourier side),
the coherent space-time resonance phenomenon \eqref{oscRes} (which dictates the choice of our $L^2$-based norm), 
and long-range asymptotics.
We describe our functional setting below after introducing the necessary notation.

%
%
%
%
%

\smallskip
\subsubsection{Notation}\label{secNotation}
To introduce our functional framework, we first define the Littlewood-Paley frequency decomposition. 

\medskip
\noindent
{\it Frequency decomposition}.
We fix a smooth even cutoff function  $\varphi: \R \to [0,1]$ 
supported in $[-8/5,8/5]$ and equal to $1$ on $[-5/4,5/4]$.
Note that the choice of the number $8/5$ for the support of $\varphi$ is fairly arbitrary, 
and other choices are possible; however, this number is chosen to be less than $\sqrt{3}$ 
so that when we define the cutoffs $\chi_\ell$ centered around $\pm\sqrt{3}$ in \eqref{chil} 
we can start the indexing at $0$.

For $k \in \Z$ we define $\varphi_k(x) := \varphi(2^{-k}x) - \varphi(2^{-k+1}x)$, 
so that the family $(\varphi_k)_{k \in\Z}$ forms a partition of unity,
\begin{equation*}
 \sum_{k\in\Z}\varphi_k(\xi)=1, \quad \xi \neq 0.
\end{equation*}
We  let
\begin{align}\label{cut0}
\varphi_{I}(x) := \sum_{k \in I \cap \Z}\varphi_k, \quad \text{for any} \quad I \subset \R, \quad
\varphi_{\leq a}(x) := \varphi_{(-\infty,a]}(x), \quad \varphi_{> a}(x) = \varphi_{(a,\infty)}(x),
\end{align}
with similar definitions for $\varphi_{< a},\varphi_{\geq a}$.
We will also denote $\varphi_{\sim k}$ a generic smooth cutoff function 
that is supported around $|\xi| \approx 2^k$, for example, $\varphi_{[k-2,k+2]}$ or $\varphi'_k$.

We denote by $P_k$, $k\in \Z$, the Littlewood-Paley projections adapted to the regular Fourier transform:
\begin{equation*}
\what{P_k f}(\xi) = \varphi_k(\xi) \what{f}(\xi), \quad \what{P_{\leq k} f}(\xi) = \varphi_{\leq k}(\xi) \what{f}(\xi), \quad \textrm{ etc.}
\end{equation*}
We will avoid using, as a recurrent notation, the distorted analogue of these projections.

We also define the cutoff functions
\begin{equation}\label{cut1}
\varphi_k^{(k_0)}(\xi) = 
\left\{
\begin{array}{ll}
\varphi_k(\xi) \quad & \mbox{if} \quad k> \lfloor k_0 \rfloor,
\\        
\varphi_{\leq \lfloor k_0 \rfloor }(\xi) \quad & \mbox{if} \quad k= \lfloor k_0 \rfloor,
\end{array}\right. 
\end{equation}
and
\begin{equation}\label{cut2}
\varphi_k^{[k_0,k_1]}(\xi) = 
\left\{
\begin{array}{ll}
\varphi_k(\xi) \quad & \mbox{if} \quad k \in (\lfloor k_0 \rfloor, \lfloor k_1 \rfloor) \cap \Z,
\\        
\varphi_{\leq \lfloor k_0 \rfloor}(\xi) \quad & \mbox{if} \quad k= \lfloor k_0 \rfloor,
\\
\varphi_{\geq \lfloor k_1 \rfloor}(\xi) \quad & \mbox{if} \quad k= \lfloor k_1 \rfloor.
\end{array}\right. 
\end{equation}
We are adopting the standard notation $\lfloor x \rfloor$ to denote the largest integer smaller than $x$.
Note that the indexes $k_0$ and $k_1$ in \eqref{cut1}-\eqref{cut2} do not need to be integers.
We also adopt the convention that if $k_0=k_1$ then $\varphi_k^{[k_0,k_1]}=1$.

We will denote by $T$ a positive time, and always work on an interval $[0,T]$
for our bootstrap estimates; see for example Proposition \ref{propbootg}.
To decompose the time integrals such as \eqref{Duhamelintro} for any $t \in [0,T]$
(this is first done in \eqref{wproof5.1} and then systematically throughout Sections \ref{secwR}-\ref{secw'})
we will use a suitable decomposition of the indicator function $\mathbf{1}_{[0,t]}$
by fixing functions $\tau_0,\tau_1,\cdots, \tau_{L+1}: \R \to [0,1]$, 
for an integer $L$ with $|L-\log_2 (t+2)| < 2$,
with the properties that 
\begin{align}\label{timedecomp}
\begin{split}
& \sum_{n=0}^{L+1}\tau_n(s) = \mathbf{1}_{[0,t]}(s),  
\qquad \supp (\tau_0) \subset [0,2], \quad  \supp (\tau_{L+1}) \subset \big[{\tfrac{1}{4}}t,t\big],
\\
& \mbox{and} \quad \supp(\tau_n) \subseteq [2^{n-1},2^{n+1}], 
  \quad |\tau_n'({s})|\lesssim 2^{-n}, \quad \mbox{for} \quad n= 1,\dots, L.
\end{split}
\end{align}
In all our arguments we also will often restrict to $n\geq 1$, as the contribution for $n=0$ 
is always trivial to handle.

In light of the coherent phenomenon explained in Subsection \ref{introQR} 
we also need cutoff functions
\begin{equation}
\label{chil}
\chi_{\ell,\sqrt 3}(z) = \varphi_\ell(|z|-\sqrt{3}), \quad \ell \in \Z \cap (-\infty,0],
\end{equation}
which localize around $\pm \sqrt{3}$ at a scale $\approx 2^\ell$.
In analogy with \eqref{cut0} and \eqref{cut1} 
we also define
\begin{equation}
\label{chil'}
\chi_{\ast,\sqrt 3}(z) = \varphi_\ast(|z|-\sqrt{3}),
\qquad \chi_{\ell,\sqrt 3}^{\ast}(z) = \varphi_\ell^{\ast}(|z|-\sqrt{3}). 
\end{equation}

\medskip
\noindent
{\it More notation}.

\smallskip
\noindent
For any $k \in \Z$, let $k^+ := \max(k,0)$ and $k^- := \min(k,0)$.

\smallskip
\noindent
We denote $\mathbf{1}_A$ the characteristic function of a set $A\subset \R$,
and let $\mathbf{1}_\pm$ be the characteristic function of $\{\pm x > 0\}$.

\smallskip
\noindent
We use $a\lesssim b$ when $a \leq Cb$ for some absolute constant $C>0$ independent on $a$ and $b$.
$a \approx b$ means that $a\lesssim b$ and $b\lesssim a$.
When $a$ and $b$ are expressions depending on variables or parameters, the inequalities
are assumed to hold uniformly over these.

\smallskip
\noindent
Given $c\in \R$, we will use the notation $c+$ to denote a number $d$ 
larger than $c$ but that can be chosen arbitrarily close to it.
Similarly, we will use $c-$ for a number smaller than $c$ that can be chosen arbitrarily close to it;
see for example \eqref{lemTboundmain}. 
We will sometimes use this convention also with $c=\infty$ to denote an arbitrarily large number
(see for example \eqref{dxiQR2+infty}).

\smallskip
\noindent
We will denote by $\min(x_1,x_2,\dots)$, resp. $\max(x_1,x_2,\dots)$, the minimum,
resp. maximum, over the set $\{x_1,x_2,\dots\}$.
We will also denote by $\min_2(x_1,x_2,\dots)$, resp. $\max_2(x_1,x_2,\dots)$, the second smallest,
resp. second largest, element in the set $\{x_1,x_2,\dots\}$.
We are also using $\med(x_1,x_2,x_3)$ for $\max_2(x_1,x_2,x_3)$; see, for example, \eqref{cregular1}
or \eqref{l<porder}.

\smallskip
\noindent
We denote by 
\begin{align}
\what{f}=\whF(f):= \frac{1}{\sqrt{2\pi}} \int_\R e^{-ix\xi} f(x) \, dx
\end{align}
the standard Fourier transform of $f$.

\smallskip
\noindent
We use the standard notation for Lebesgue $L^p$ spaces, and for Sobolev spaces $W^{k,p}$ and $H^k=W^{k,2}$.


\medskip
\subsubsection{Norms}
For $T>0$, we let $W_T$ be the space given by the norm
\begin{align}\label{wnorm}
\begin{split}
{\| h \|}_{W_T} := \sup_{n \geq 0} \sup_{\ell \in \mathbb{Z} \cap [\lfloor -\gamma n\rfloor,0]} 
  {\big\| \chi_{\ell,\sqrt{3}}^{[-\gamma n,0]}(\,\cdot\,) \, \tau_n(t) \, h(t,\cdot) 
  \mathbf{1}_{0 \leq t \leq T} \big\|}_{L^\infty_t L^2_\xi} 2^{\beta \ell} 2^{-\alpha n}
 \end{split}
\end{align}
where\footnote{We are using the same notation from \eqref{timedecomp}
for time-cutoffs to avoid introducing an additional notation,
but in the definition \eqref{wnorm} we do not need regularity assumptions on the $\tau_n$,
but just that they are a partition of unity.}
$\tau_n$ here denotes a partition of unity as in \eqref{timedecomp} with $T$ in place of $t$,
and where the parameters $0 < \alpha, \beta, \gamma < \frac{1}{2}$ satisfy
\begin{align}
\label{wnormparam0}
\gamma \beta' < \alpha < \frac{\beta'}{2}, \quad \beta' \ll 1,
\qquad \beta' := \frac{1}{2} - \beta, \quad \gamma' := \frac 1 2 - \gamma.
\end{align}
$\beta'$ is a fixed constant that needs to be chosen small enough
to satisfy various inequalities that we will impose in the course of the proof.
Note that we automatically have $\gamma<1/2$, 
and that one possible way to impose all of the conditions \eqref{wnormparam0} is to choose $\alpha$ sufficiently small and
$$\beta' = 2\alpha + 2\alpha^2, \quad  \gamma' = 2\alpha + \alpha^2.$$

Let us briefly explain the choice of the norm and parameters: 

\begin{itemize}

\smallskip
\item The norm \eqref{wnorm} will be used to measure our solution on the Fourier side.
More precisely, we will show that ${\|\jxi \partial_\xi \wt{f}\|}_{W_T} \lesssim \e_0$, 
where $\wt{f}$ is a renormalized version of the profile $\wt{g}$ in \eqref{mtprofile}.
As already pointed out, measuring $\partial_\xi$ on the Fourier side is akin to measuring a weighted norm in real space.



\smallskip
\item The quantity $2^\ell$ measures the distance from $\pm \sqrt{3}$ starting at smallest scale 
$2^{-\g n}$, where $2^n \approx |t|$, 
and the norm is penalized by the factor $2^{\beta \ell}$. 
The additional penalization of $2^{-\alpha n}$ is added globally
to take into account 
long-range effects which are present at every frequency.




\smallskip
\item To make sure that localization and derivation in the $W_T$ norm commute
(under the hypothesis that $\wt{f}$ is uniformly bounded) 
one needs $\beta' \gamma \leq \alpha$.

\smallskip
\item In order to deduce from a bound on the $W_T$ 
norm (together with a bound on the $\wtF^{-1}\jxi^{-3/2}L^\infty$) 
the necessary linear decay estimate at the optimal rate of $\jt^{-1/2}$, 
we need $\alpha + \beta \gamma < 1/4$; see Proposition \ref{propdisp}.
Since 
$$\alpha + \beta \gamma = \alpha + 1/4 -  \beta\g' - \beta'/2,$$ 
it suffices to impose $\beta' \geq 2\alpha$.

\end{itemize}

\medskip
\subsection{The main bootstrap and proof of Theorem \ref{maintheo}}\label{Ssecmtpr0}
For $T >0$, consider a local solution $u\in C([0,T],H^5(\R)) \cap C^1([0,T],H^4(\R))$ 
of \eqref{KG} constructed by standard methods. 
Our proof is based on showing an a priori estimate for the following norm: 
\begin{align}\label{apriori0}
{\| u \|}_{X_T} = \sup_{t\in[0,T]} \Big[ 
  \langle t \rangle^{-p_0} {\big\| (\sqrt{H+1}\,u,u_t)(t) \big\|}_{H^4} 
  + \jt^{1/2} {\| (\partial_t,\partial_x)u(t) \|}_{L^\infty} \Big], \qquad 0<p_0<\alpha.
\end{align}
Under the initial smallness condition \eqref{initcond}, we will assume the a priori bound
\begin{equation}
\label{bootstrap}
{\|u \|}_{X_T} \leq \e_1,
\end{equation}
and show that this implies
\begin{align}
\label{bootstrapest}
{\| u \|}_{X_T} \leq C\e_{0} +  C\e_{1}^2,
\end{align}
for some absolute constant $C>0$.
Picking $\e_0$ sufficiently small and using a standard bootstrap argument with $\e_1 = 2 C \e_0$,
\eqref{bootstrapest} gives global existence of solutions that are small in the space $X_\infty$.
Using also time reversibility we obtain solutions for all times.

The structure of the paper and of the proof of Theorem \ref{maintheo}, with details
on how the main bootstrap \eqref{bootstrapest} will be proved, are described below.

%

\medskip
\subsection{Structure of the paper and the proof of Theorem \ref{maintheo}}\label{Ssecmtpr}
In this subsection we discuss the organization of the paper,
describe the overall structure of the 
proof, and give more details about the various estimates needed to show \eqref{bootstrapest},
under the a priori assumption \eqref{bootstrap}.






\setlength{\leftmargini}{1.5em}
\begin{itemize}

\bigskip
\item Section \ref{secspth} contains an exposition of the elements 
of the scattering theory for Schr\"odinger operators $H:=-\partial_x^2 + V$ on $\R$ which we will need.

\smallskip
\noindent
After introducing the Jost functions $f_\pm$, see \eqref{f+-},
and the transmission and reflection coefficients $T$ and $R_\pm$, see \eqref{f+f-} and \eqref{TRformula},
we define the distorted Fourier transform (dFT) as in  \eqref{DistTrans} (see \eqref{distF}), 
with the `distorted' (or generalized) exponentials (or eigenfunctions) $\psi(x,\xi)$ given by \eqref{psixk}.  

\smallskip
\noindent
Some basic properties of the dFT are discussed in \S\ref{ssecdFT}.
Then, the $\psi(x,\xi)$ are analyzed in details in Subsection \ref{ssecpsi}
and decomposed into a singular and a regular part.
The singular part behaves at spatial infinity like linear combinations of (standard) complex exponentials,
while the regular part is fast decaying.
This decomposition is also at the heart of the decomposition of the nonlinear spectral distribution $\mu$
defined in \eqref{intromu0}.

\smallskip
\noindent
In Subsection \ref{secLin} we prove the first non-trivial result involving the dFT,
that is, the estimate for the linear flow $e^{it\sqrt{H+1}} = e^{it\langle \wt{D}\rangle}$ 
(see the notation for Fourier multipliers in \S\ref{ssecmultipliers})
given in \eqref{disp2}, which involves the degenerate norm $W_T$.
This estimate shows that sharp $L^\infty_x$ decay (i.e., at the rate of $t^{-1/2}$)
for the evolution $e^{it\sqrt{H+1}} h(t)$ of a (time-dependent) profile $h$ on an interval $[0,T]$, 
follows from controlling the norms
\begin{align}\label{mtpr2}
\sup_{t \in [0,T]} {\big\| \jxi^{3/2} \wt{h}(t) \big\|}_{L^\infty_\xi} \leq C,
\qquad  {\| \jxi \partial_\xi \wt{h} \|}_{W_T} \leq C,
\qquad 
\sup_{t \in[0,T]} \big( \jt^{-p_0} {\| \jxi^4 \wt{h} \|}_{L^2} \big) \leq C,
\end{align}
where $W_T$ is the norm defined in \eqref{wnorm}, with the restriction on the parameters \eqref{wnormparam0}, 
$p_0$ is sufficiently small, and $C$ is an absolute constant independent of $T$.
As it turns out, we cannot control the norms in \eqref{mtpr2} for the profile associated to the solution $u$,
and infer the bound for the $L^\infty_x$ norm in \eqref{apriori0} from this.
Instead, we need to take a longer route and estimate norms as in \eqref{mtpr2} for a renormalized profile,
which is defined in Section \ref{secdec}, see \eqref{Renof}.


\bigskip
\item Before moving on to the analysis of the nonlinear time evolution, we study more precisely
the nonlinear spectral measure in Section \ref{secmu}. 

\smallskip
\noindent
The main Proposition \ref{muprop} describes the precise structure 
of the NSD $\mu$, see \eqref{intromu0}; 
for lighter notation we omit the indexes $\iota_1,\iota_2$ here.
The main goal is to decompose $\mu$ into a `singular' and a `regular' part.

\smallskip
\noindent
The singular part, denoted $\mu^S$, is a linear combination of $\delta$ and $\pv$ distributions,
as anticipated in \eqref{muintro1}; 
the precise definition is given by \eqref{muSdec}-\eqref{muL+-},
with formulas for the coefficients given in \eqref{mucoeff}-\eqref{mucoeffexp}.
Notice that these coefficients may not be smooth at $\xi=0$ (e.g., in the generic case).
As it is apparent, handling formulas involving these coefficients requires a good amount of 
somewhat tedious bookkeeping;
however, this is necessary for two main reasons: first, we need the exact expressions
to calculate the final asymptotics for the solution of \eqref{KG} and, second,
we will need to check some smoothness properties for the multipliers of the trilinear terms that 
will appear after a normal form transformation, and which involve these coefficients.

\smallskip
\noindent
The regular part of the NSD, denoted $\mu^R$ is defined in \eqref{muR} with \eqref{muR'},
and it is essentially a smooth function of the three frequencies $(\xi,\eta,\s)$
up to possible jump singularities on the axes.
The mapping properties of the associated bilinear operator
are established in Subsection \ref{sseclemmuR}, with \eqref{lemmuRbound}
showing that it essentially behaves like multiplication by a localized function.

\bigskip
\item 
In Section \ref{secdec} we begin the analysis of the time evolution by defining 
the profile associated to $u$ as 
\begin{align}\label{mtprug}
g := e^{it\sqrt{H+1}} (\partial_t - i\sqrt{H+1})u; 
\end{align}
see \eqref{vprof} and \eqref{vKG}. 
From the main equation \eqref{KGu} we write the nonlinear evolution for $\wt{g}$ 
as in \eqref{dtwtf}-\eqref{mu12} (which is the same as the formula \eqref{dtgintro}).

\smallskip
\noindent
Using the decomposition of $\mu = \mu^S + \mu^R$ we would like to decompose accordingly
the quadratic terms in the formula for $\partial_t\wt{g}$ into singular terms 
and regular terms. 
However, as briefly mentioned in \eqref{introQS},
because of the presence of the $\pv$ term coming from $\mu^S$ we cannot do this decomposition directly.
We instead need a further distinction within the terms containing the $\pv$
into `truly' singular terms, where the $\pv$ is restricted close to its singularity, 
and more regular ones that are supported away from the singularity.
This is the role of the cutoff $\varphi^\ast$ defined in \eqref{QZphistar} and appearing in \eqref{QZ}.
The singular terms are then defined according to \eqref{Q}-\eqref{QZ}.
The precise choice of $\varphi^\ast$ is made so that, on its support, we can derive lower bounds
for the oscillating phases $\Phi$ in \eqref{mu12}.

\smallskip
\noindent
The main motivation for the splitting $$\partial_t\wt{g} = \mathcal{Q}^S + \mathcal{Q}^R,$$
as done in Subsection \ref{SecDecN1},
is that the singular quadratic terms resemble the quadratic terms that one would get 
for a flat ($V=0$) quadratic KG equation. 
In particular, the oscillating phases are lower bounded on the support of $\mathcal{Q}^S$, 
as established in Lemma \ref{lemmaetazeta};
then, we can apply a normal form transformation to recast these terms into cubic ones.

\smallskip
\noindent
The algebra for the normal form step is carried out in Subsection \ref{SecDecN2}.
Starting from the simple identity \eqref{SecDecN20} 
we naturally define the {\it bilinear normal form transformation}\footnote{This is the
bilinear operator which we denoted by $B$ in Theorem \ref{maintheo} 
to avoid confusion with the reflection coefficient there.} 
$T$ in \eqref{Tgg},
which arises from the boundary terms in the time-integration by parts.
More precisely (but still omitting the various sums over the signs such as $\iota_1$ and $\iota_2$) 
we have
\begin{align}
\int_0^t\mathcal{Q}^S \,ds = \wtF{T}(g,g)(t) -  \wtF{T}(g,g)(0) + \int_0^t B_1(s) + B_2(s) \, ds
\end{align}
where the bulk terms $B_1$ and $B_2$ are the expressions defined in \eqref{bulk1'} and \eqref{bulk2}
and are cubic in $g$.
The only quadratic terms left are then the $\mathcal{Q}^R$ terms which include the contribution
from $\mu^R$, and the $\pv$ part restricted outside the support of $\varphi^\ast$.

\smallskip
\noindent
In Subsection \ref{SecDecS1}, respectively, Subsection \ref{SecDecS2}, 
we analyze the leading order symbol $\mathfrak{b}^1$, respectively, the lower order symbol $\mathfrak{b}^2$,
of the bulk term $B_1$, respectively, $B_2$ above.
These are somewhat complicated expressions since they involve the symbol \eqref{QZ}, and
its variant without the cutoff $\varphi^\ast$, and therefore combinations of the (non-smooth)
coefficients \eqref{mucoeff}.
The leading order symbol is made by the convolution of $\delta$ and $\pv$-type distributions, 
see \eqref{coeff1}-\eqref{coeff10}.
For the later nonlinear analysis we need to make sure that this symbol is nice enough
so that the associated trilinear operators satisfies H\"older type bounds.
An important technical point then is the verification of the smoothness with respect to
variable in which the convolution is performed; this is done \S\ref{SecDecsig}.
In \S\ref{SecDectop} we then calculate precisely the top order (singular $\delta$ and $\pv$-type) contribution from $\mathfrak{b}^1$,
that is the symbol $\mathfrak{c}^S$ in \eqref{formulacubiccoeff};
the associated trilinear operator will be denoted by $\mathcal{C}^S$.
Other contributions from $\mathfrak{b}^1$ and the symbol $\mathfrak{b}^2$ are analyzed in Subsection \ref{SecDecS2};
the associated trilinear operator will be denoted by $\mathcal{C}^R$.
The mapping properties of these trilinear operators are analyzed in Section \ref{secmulti}.

\smallskip
\noindent
In the last Subsection \ref{SsecReno} we finally arrive at the definition \eqref{Renof}
of the {\it renormalized profile} 
\begin{align}\label{mtpr3}
f = g - T(g,g). 
\end{align}
We see that $f$ satisfies an equation where the only quadratic terms are regular ones,
and the cubic terms are those analyzed in the previous subsections.
The equation \eqref{Renodtf} for $\wt{f}$ is the starting point for the nonlinear analysis,
and we record it here for ease of reference in a slightly simplified form 
(omitting the easier regular cubic terms, see \eqref{CubicR}-\eqref{cregular1})
\begin{align}\label{mtpr4}
\partial_t \wt{f} = \mathcal{Q}^R(g,g) + \mathcal{C}^{S}(g,g,g);
\end{align}
see the definitions \eqref{QR}-\eqref{CubicS}.

\smallskip
\noindent
The heart of the proof of the bootstrap \eqref{bootstrapest} is 
another bootstrap argument for the renormalized profile $f$ involving the norms in \eqref{mtpr2}
and is based on (a renormalization of) the equation \eqref{mtpr4}. 
See the description of the contents of Section \ref{secBoot} below.



\bigskip
\item Section \ref{secmulti} contains bilinear and trilinear estimates for the various
operators appearing in our problem. 
Here we need to analyze different types of pseudo-product operators,
from the standard bilinear ones \eqref{defCC}, to 
trilinear ones involving a $\pv$ as \eqref{defUUVV}. 
Bounds for general bilinear and trilinear operators of the types that appear in our proof
are established in Lemmas \ref{lemmamultilin1} and \ref{lemmamultilin2},
and basic criteria to check the assumptions in these lemmas area also given.


\smallskip
\noindent
In Subsection \ref{ssecT} we analyze in detail the normal form operator $T$,
and establish, in Lemma \ref{lemT}, that it satisfies H\"older-type bounds with a gain of regularity on the inputs.

\smallskip
\noindent
The other main results in this Section are: 
Lemma \ref{lemQR} which gives improved H\"older-type inequalities
for the smooth bilinear operator $\mathcal{Q}^R$,
and Lemma \ref{lemCS} which gives sharp H\"older-type bounds with some gain of regularity for the
singular cubic terms $\mathcal{C}^S$.


\bigskip
\item In Section \ref{secBoot} we setup the proof
of the main bootstrap bound \eqref{bootstrapest}.
As mentioned above, these estimates will mostly involve the renormalized profile $f$,
but we first need to relate the desired bounds for $g$ (and $u$, as stated in \eqref{mtLinfty}-\eqref{mtbounds}) 
to the necessary bounds on $f$. Here is how we proceed.

\smallskip
\noindent
With $\e_0$ as in \eqref{initcond}, we let $\e_2 \gg \e_1 \gg \e_0$.
Proposition \ref{propbootg} gives an a priori bootstrap on $g$ for the norms
\begin{align}\label{mtpr11}
\sup_{t\in[0,T]} \Big[ \jt^{-p_0} {\big\| \jxi^{4} \wt{g}(t) \big\|}_{L^{2}} 
  + \jt^{1/2} {\| e^{-it\jnab} \mathbf{1}_{\pm}(D)\W g(t) \|}_{L^\infty} \Big], 
\end{align}
where $\mathcal{W}^\ast$ is the adjoint of the wave operator defined in \eqref{propWdef};
we assume that \eqref{mtpr11} is bounded by $2\e_2$ and claim that it can bounded by $\e_2$.
Proposition \ref{propbootf} instead gives the main bootstrap for the following norms of $f$:
\begin{align}\label{mtpr12}
\sup_{t\in[0,T]} \Big[ \jt^{-p_0} {\big\|  \jxi^{4} \wt{f}(t) \big\|}_{L^{2}} + {\| \jxi \partial_\xi \wt{f} \|}_{W_t} 
  + {\| \jxi^{3/2} \wt{f}(t) \|}_{L^\infty} \Big]; 
\end{align}
assuming that \eqref{mtpr12} is bounded by $2\e_1$ we claim that it can be bounded by $\e_1$.
Note that we are assuming much stronger information on $f$ than on $g$.

\smallskip
\noindent
Subsection \ref{ssecBoot} is 
dedicated to showing how the a priori bound on \eqref{mtpr12} by $2\e_1$, 
can be used to close the claimed bootstrap for the norms in \eqref{mtpr11}.
This is not too hard to do using the relation $g = f + T(g,g)$, see \eqref{mtpr3},
the bilinear bounds on the $T$ operator established in Section \ref{ssecT}, and the linear estimate \eqref{disp2}.

\smallskip
\noindent
Note that once we have proven a bound for \eqref{mtpr11} by $\e_2 = C\e_0$,
we can immediately deduce the Sobolev bound in \eqref{mtbounds} from
\eqref{mtprug} and the boundedness of wave operators (Theorem \ref{thmweder}).
The decay bound \eqref{mtLinfty} does not follow directly from the $L^\infty$ bound in \eqref{mtpr11}
(because wave operators may be unbounded on $L^\infty$), and it is proved separately in \eqref{propbootginfty}.
The weighted bound in \eqref{mtwbound} 
is proved in Lemma \ref{lemdxig}.
Since \eqref{mtLinfty}-\eqref{mtwbound} follow from the bound on \eqref{mtpr11},
then the proof of the main theorem has been reduced to proving the bootstrap
Proposition \ref{propbootf}.
As part of the arguments needed to prove these bootstrap estimates on $f$
we will also establish its asymptotic behavior, see \eqref{LinfSasyf0} (and Section \ref{secLinfS}).

\smallskip
\noindent
The rest of Section \ref{secBoot} prepares for later analysis and the proof of Proposition \ref{propbootf}.
Subsection \ref{ssecpre} contains some preliminary bounds on $f$ that follow from the a priori bound on 
the norms in \eqref{mtpr12} (Lemma \ref{lembb}).
Then, in Subsection \ref{Ssecexp}, using \eqref{mtpr3}, we rewrite the equation for $f$, 
see \eqref{mtpr4}, as
\begin{align}\label{mtpr4'}
\wt{f}(t) - \wt{f}(0) = \int_0^t \mathcal{Q}^R(f,f) \,ds + \int_0^t \mathcal{C}^{S}(f,f,f) \,ds + \cdots + \mathcal{R}(t),
\end{align}
where the ``$\dots$'' denote other cubic and quartic terms in $f$,
and $\mathcal{R}$ denote terms that have a higher degree of homogeneity in $f$ and $g$ and can be treated 
as remainders.
We actually use expansions at different orders depending on 
which norm we are trying to estimate.

\smallskip
\noindent
To close the bootstrap for $f$ we then need to estimate
the terms on the right-hand side of \eqref{mtpr4'}.
Lemmas \ref{lemQRexp} and \ref{lemCSexp} give, among other things, suitable bounds on the remainders $\mathcal{R}$,
in all the norms in \eqref{mtpr12}.

\smallskip
\noindent
In Subsection \ref{Ssecmidsum}, for the convenience of the reader, 
we 
summarize the bounds obtained thus far and list all the bounds that are left to prove.




\medskip
\item Section \ref{secwR} and \ref{secwL}, constitute the heart of the 
paper and the more technical part of the analysis.
The goal of these two sections is to carry out the main parts of the 
estimates for the weighted $L^2$ norm \eqref{wnorm} 
of the regular quadratic terms, $\mathcal{Q}^R(f,f)$, and of the singular cubic terms, $\mathcal{C}^S(f,f,f)$.

\smallskip
\noindent
The desired weighted bound for $\mathcal{Q}^R(f,f)$ 
is \eqref{prowRest} in Proposition \ref{prowR}.
Section \ref{secwR} is then entirely dedicated to proving this key bound when the interactions are restricted
to 
the main resonant ones, 
that is, $(\eta,\sigma) = (0,0) \longrightarrow  \xi = \pm\sqrt{3}$ (see the notation used in \eqref{wproofQR2} and \eqref{wRmain}).

\smallskip
\noindent
Subsections \ref{ssecwR1}-\ref{secdecred} give some preliminary bounds and reductions.
We first take care of frequencies $\xi$ that are 
very close to $\pm\sqrt{3}$ 
and reduce the desired bound 
to showing \eqref{wproofdecest} with \eqref{wproofdecpar},
for the localized operator\footnote{We continue to adopt our convention
of omitting the various indexes in this discussion.} $I(t,\xi)$ defined in \eqref{wproofnot1++};
these reductions are summarized in Lemma \ref{lemred}.

\smallskip
\noindent
Note that the estimate \eqref{wproofnot1++} 
involves localization in the size of the input variables $|\eta|$ and $|\s|$,
in the distance $||\xi|-\sqrt{3}|$, in the size of the oscillating phase $|\Phi|$, and in the integrated time $s$,
at dyadic scales with respective parameters $k_1,k_2,\ell,p$ and $m$.
These localizations allows us to distinguish various cases, 
and to exploit efficiently the oscillations in either frequency space or in time 
depending on the relative size of the quantities involved.
The estimates are split into four main regions, as described in \eqref{regions}.
and treating each of these regions occupies one of the Subsections \ref{Ssecpsmall}-\ref{Ssecl<2k_1}.

\smallskip
\noindent
We remark that a useful quantity is the one defined in \eqref{wproofnot2}, which incorporates some improved decay
properties of the solution (for small frequencies). 

\medskip
\item In Section \ref{secwL} we estimate the weighted $L^2$ norm for the singular cubic terms $\mathcal{C}^S(f,f,f)$
of the form \eqref{CSintro}  (see \eqref{CubicS} with \eqref{formulacubiccoeff} for the precise definition),
focussing on the case of the main resonant interactions $\pm (\sqrt{3},\sqrt{3},\sqrt{3}) \rightarrow \pm\sqrt{3}$.
In particular, we achieve the main step in the proof of Proposition \ref{propCSbound}, that deals with
the interactions of the type $(+-+)$, where the signs 
correspond to the signs of the oscillating factors in the cubic phases in \eqref{CubicS}.
Subsection \ref{subsecCS1} treats the terms that involve a $\delta$ factor, 
while Subsection \ref{subsecCS2} those with a $\pv$
(recall the form of the cubic symbols \eqref{formulacubiccoeff}).
Once again we need to distinguish various cases depending on the distance of the input and output variables from
the bad frequency $\sqrt{3}$, relative to time and to the size of their differences 
(see, for example, the dyadic localizations in \eqref{defljm2}).

\medskip
\item 
With Sections \ref{secwR} and \ref{secwL} we have taken care of estimating the weighted norm for the 
leading order terms on the right-hand side of \eqref{mtpr4'}, in the case of the main resonant interactions.
All of the other non-resonant interactions are estimated in Section \ref{secw'}.

\smallskip
\noindent
Section \ref{secLinfS} contains the main part of the proof for the control of the Fourier-$L^\infty$ norm in \eqref{mtpr12},
that is, the proof of Proposition \ref{propLinfS}, which gives asymptotics for the singular cubic terms $\mathcal{C}^S$ 
(see \eqref{LinfSasy} where the Hamiltonian function is given in \eqref{asyHam0}).
From this we can then derive an asymptotic ODE for $\wt{f}$ 
and thus the asymptotic behavior of the solution as in \eqref{LinfSasyf} (see \eqref{LinfSasyf0}).


\smallskip
\noindent
Subsection \ref{secha} provides first a formal computation for the asymptotics, based on the stationary phase lemma.
Subsection \ref{secModScatt} utilizes these computations to give the exact structure of the long-range asymptotics,
and the form of the Hamiltonian $H$ appearing in the statement of Proposition \ref{propLinfS}). 
Rigorous bounds are then proved in \eqref{secra}. 



\medskip
\item Section \ref{secw'} contains the estimates needed to control all 
the contributions from the nonlinear terms on the right-hand side of \eqref{mtpr4'}
that have not been dealt with in Sections \ref{secBoot}-\ref{secLinfS},
since they are lower order compared to the main ones.
We refer the reader to the first paragraph of Section \ref{secw'} 
for a list of the estimates that are carried out there,
and the details on how they complete the proofs of the main propositions stated in the previous sections.

\smallskip
\noindent
The estimates of Section \ref{secw'} complete the bootstrap on the norm \eqref{mtpr12}.

\bigskip
\item Finally, Appendix \ref{Appendix} contains a verification of the spectral assumptions needed to apply our results to the
Double Sine-Gordon model \eqref{DSG} and obtain Corollary \ref{corDSG}.

\end{itemize}



\bigskip
\section{Spectral theory and distorted Fourier transform in $1$d}\label{secspth}
We develop in this section the spectral and scattering theory of
$$
H = - \partial_x^2 + V,
$$
assuming that $V \in \mathcal{S}$, and that $H$ only has continuous spectrum. 
We state the results that are needed for the nonlinear problem that interests us here, 
and sketch the important proofs.

This theory is due to Weyl, Kodaira and Titchmarsh (who also considered more general Sturm-Liouville problems).
Complete expositions can be found in \cite{Dunford} and \cite{Wilcox};
we mention in particular Yafaev \cite[Chapter 5]{Yafaev},
where the operator $H$ is considered, and direct proofs are given.

\medskip
\subsection{Linear scattering theory}\label{Ssecspth}

\subsubsection{Jost solutions}\label{spth1}

Define $f_{+}(x,\xi)$ and $f_-(x,\xi)$ by the requirements that
\begin{align}
\label{f+-}
(- \partial_x^2 + V) f_{\pm}  = \xi^2 f_{\pm}, \quad \mbox{for all $x,\xi\in\R$, \quad and} \quad
\left\{
\begin{array}{ll}
\lim_{x\rightarrow \infty} |f_{+}(x,\xi) - e^{ix\xi}| = 0,
\\
\lim_{x\rightarrow -\infty} |f_{-}(x,\xi) - e^{- ix\xi}| = 0.
\end{array}
\right.
\end{align}
Define further
\begin{align}
\label{m+-}
m_{+}(x,\xi) = e^{-i\xi x} f_{+}(x,\xi) \quad \mbox{and} \quad m_{-}(x,\xi) = e^{i\xi x} f_{-}(x,\xi),
\end{align}
so that $m_{\pm}$ is a solution of 
\begin{equation}
\label{equationm}
\partial_x^2 m_{\pm} \pm 2i\xi \partial_x m_{\pm} = Vm_{\pm}, \qquad m_\pm(x,\xi) \to 1 \;\mbox{as $x \to \pm \infty$}.
\end{equation}

The functions $m_\pm$ satisfy symbol type bounds for $\pm x>0$, as stated in the following lemma.

\begin{lem}
\label{lemm+-}
For all non-negative integers $\alpha, \beta, N$,
\begin{align}
\label{mgood}
& \big| \partial_x^\alpha \partial_{\xi}^\beta (m_{\pm}(x,\xi) - 1) \big| 
  \lesssim \langle x \rangle^{-N} \langle \xi \rangle^{-1-\beta}, \quad \pm x \geq -1,
  \\
\label{mbad}
& \big| \partial_x^\alpha \partial_{\xi}^\beta (m_{\pm}(x,\xi) - 1) \big| 
  \lesssim  \langle x \rangle^{1+\beta}\langle \xi \rangle^{-1-\beta}, \quad \pm x \leq 1.
  \end{align}
\end{lem}

The estimates \eqref{mgood}-\eqref{mbad} 
can be obtained from 
the integral form of \eqref{equationm},
\begin{equation}
\label{integralmpm}
\begin{split}
& m_+(x,\xi) = 1 + \int_x^\infty D_\xi(y-x) V(y) m_+(y,\xi) \,dy,
 \\
& m_-(x,\xi) = 1 + \int_{-\infty}^x D_\xi(x-y) V(y) m_-(y,\xi) \,dy,
\end{split}
\end{equation}
where
$$
D_\xi(z) = \frac{e^{2i\xi z} - 1}{2i\xi}.
$$
Since the proof is fairly standard we skip the details, 
and refer the reader to \cite[Appendix A]{DelortNLSV}.

\smallskip
\subsubsection{Transmission and Reflection coefficients}\label{spth2}
A classical reference for the formulas that we recall here is \cite{DeiTru}
(see also \cite{Weder2,Yafaev}, for example).
Denote $T(\xi)$ and $R_{\pm}(\xi)$, respectively, the {\it transmission} 
and {\it reflection} coefficients associated to the potential $V$.
These coefficients are such that
\begin{align}
\label{f+f-}
\begin{split}
&f_+ (x,\xi) = \frac{1}{T_+(\xi)} f_-(x,-\xi) + \frac{R_-(\xi)}{T_+(\xi)} f_-(x,\xi),
\\
&f_- (x,\xi) = \frac{1}{T_-(\xi)} f_+(x,-\xi) + \frac{R_+(\xi)}{T_-(\xi)} f_+(x,\xi),
\end{split}
\end{align}
or, equivalently,
\begin{align*}
&f_+ (x,\xi) \sim \frac{1}{T_+(\xi)} e^{i\xi x} + \frac{R_-(\xi)}{T_+(\xi)} e^{-i\xi x} \quad \mbox{as $x \to - \infty$},
\\
&f_- (x,\xi) \sim \frac{1}{T_-(\xi)} e^{-i\xi x} + \frac{R_+(\xi)}{T_-(\xi)} e^{i\xi x} \quad \mbox{as $x \to \infty$}.
\end{align*}
In the equalities above $T_+$ and $T_-$ do a priori differ;
however, since the Wronskian
\begin{align}\label{Wronskian}
W(\xi) := W(f_+(\xi),f_-(\xi)), \qquad  W(f,g) = f'g - fg'
\end{align}
is independent of the point $x$ where it is computed for solutions of~\eqref{f+-}, 
one sees (taking $x\rightarrow \pm\infty)$ that $T_+ = T_- = T$ and 
\begin{align}\label{W=xi/T}
W(\xi) = \frac{2i\xi}{T(\xi)}.
\end{align}
Since $\overline{f_{\pm}(x,\xi)} = f_{\pm}(x,-\xi)$, we obtain furthermore that
\begin{align}\label{TRconj}
\overline{T(\xi)} = T(-\xi) \qquad \mbox{and} \qquad \overline{R_{\pm}(\xi)} = R_{\pm}(-\xi).
\end{align}
Finally, computing $W(f_{+}(\xi),f_-(\xi))$, $W(f_{+}(\xi),f_+(-\xi))$, $W(f_{-}(\xi),f_-(-\xi))$ at $x = \pm \infty$ gives
\begin{align} \label{TR}
& |R_{\pm}(\xi)|^2 + |T(\xi)|^2 = 1, \qquad \mbox{and} \qquad T(\xi)\overline{R_-(\xi)} + R_+(\xi)\overline{T(\xi)} = 0.
\end{align}
As a consequence, the scattering matrix associated to the potential $V$ is unitary:
\begin{align}
\label{scatmat}
S(\xi) :=
\left( \begin{array}{cc}
T(\xi)  & R_+(\xi) \\ R_-(\xi) & T(\xi)
\end{array}
 \right), \qquad
S^{-1}(\xi) :=
\left( \begin{array}{cc}
\overline{T(\xi)} & \overline{R_-(\xi)} \\  \overline{R_+(\xi)} & \overline{T(\xi)}
\end{array}
 \right).
\end{align}

Starting from the integral formula \eqref{integralmpm} giving $m_{\pm}$, 
letting $x \to \mp \infty$, and relating it to the definition of $T$ and $R_{\pm}$ gives
\begin{align}
\label{TRformula}
\begin{split}
&T(\xi) = \frac{2i\xi}{2i\xi - \int V(x) m_{\pm} (x,\xi)\,dx},\\
& R_{\pm}(\xi) = \frac{\int e^{\mp 2i\xi x} V(x) m_{\mp}(x,\xi)\,dx}{2i\xi - \int V(x) m_{\pm} (x,\xi)\,dx}.
\end{split}
\end{align}
These formulas are only valid for $\xi \neq 0$ a priori. 
But a moment of reflection shows that $T$ and $R_{\pm}$ can be extended to be smooth functions on the whole real line. 
Combining these formulas with Lemma \ref{lemm+-} gives the following lemma.

\begin{lem}\label{lemTR} 
Let $T$ and $R_\pm$ be defined as in \eqref{TRformula}.
Then, under our assumptions on $V$, for any $\beta$ and $N$, we have
\begin{align}\label{TRk}
|\partial_\xi^\beta [T(\xi) - 1]| \lesssim \langle \xi \rangle^{-1-\beta},
\qquad |\partial_\xi^\beta R_{\pm}(\xi) | \lesssim \langle \xi \rangle^{-N}.
\end{align}
\end{lem}

\smallskip
\subsubsection{Generic and exceptional potentials}\label{SecPot} 
We call the potential $V$
\begin{itemize}
\item \textit{generic} if $\displaystyle \int V(x) m_{\pm} (x,0)\,dx \neq 0$
\item \textit{exceptional} if $\displaystyle \int V(x) m_{\pm} (x,0)\,dx = 0$
\item \textit{very exceptional} if $\displaystyle \int V(x) m_{\pm} (x,0)\,dx = \int x V(x) m_{\pm} (x,0)\,dx = 0$
\end{itemize}

\begin{lem}\label{lemVgen}
The following four assertions are equivalent

\begin{itemize}
\item[(i)] 
$V$ is generic.

\smallskip
\item[(ii)] $\displaystyle T(0) = 0, R_{\pm}(0) = -1$.

\smallskip
\item[(iii)] $W(0) \neq 0$.

\smallskip
\item[(iv)] The potential $V$ does not have a resonance at $\xi=0$, 
in other words there does not exist a bounded non trivial solution in the kernel of $-\partial_x^2 + V$.

\end{itemize}

\end{lem}

Checking the equivalence of these assertions is easy based on the formulas \eqref{TRformula}.

\begin{prop}[Low energy scattering]\label{LES}
If $V$ is generic, there exists $\alpha \in i \mathbb{R}$ such that
\begin{align}
\label{LEST}
T(\xi) = \alpha \xi + O(\xi^2).
\end{align}

If $V$ is exceptional, let
$$
a := f_+(-\infty,0) \in \mathbb{R} \setminus \{ 0 \}.
$$
Then, 
\begin{align}\label{LESTR0}
T(0) = \frac{2a}{1+a^2}, \qquad R_+(0) = \frac{1 - a^2}{1+ a^2}, \qquad \mbox{and} \qquad R_-(0) = \frac{a^2-1}{1+ a^2}.
\end{align}
\end{prop}

\begin{proof} In the generic case, observe that
$$
T(\xi) = \frac{2i}{-\int V(x) m_{\pm}(x,0)\,dx} \xi + O(\xi^2),
$$
hence the desired result since $m_{\pm}(\cdot,0)$ is real-valued.

We now turn to the exceptional case. 
Denoting
$$
b = \int V(x) \partial_\xi m_{\pm}(x,0) \,dx, \quad \mbox{and} \quad c_{\pm} = \int V(x) x m_{\pm}(x,0)\,dx,
$$
$T(0)$ and $R_{\pm}(0)$ can, thanks to~\eqref{TRformula}, be expressed as
\begin{equation}
\label{newformulaRT}
T(0) = \frac{2i}{2i-b}, \qquad R_{\pm}(0) = \frac{b \mp 2ic_{\mp}}{2i-b}.
\end{equation}
There remains to determine the values of $b$ and $c_{\pm}$.
In order to determine  $c_{+}$, recall the integral equation \eqref{integralmpm} satisfied by $m_+$, 
and let $\xi \to 0$ and $x \to -\infty$ in that formula. 
Taking advantage of the condition $\int V(y) m_+(y,0)\,dy=0$, 
we observe that
\begin{align}\label{LESa}
a = m_+(-\infty,0) = 1 + \int_{-\infty}^\infty y V(y) m_+(y,0)\,dy = 1 + c_+.
\end{align}
Similarly, we find $\frac{1}{a} = 1 - c_-$.

Turning to $b$,  we first claim that it is purely imaginary. 
Indeed, differentiating the equation \eqref{equationm}, setting $\xi = 0$, and taking the real part, we obtain that
$$
\mathfrak{Re} \left[ (\partial_x^2 - V) \partial_\xi m_+(x,0) \right] = 0.
$$ 
Since $\partial_\xi m_+(x,0) \to 0$ as $x \to \infty$, we deduce that $\mathfrak{Re} \partial_\xi m_+ = 0$.
Using this fact, and plugging the formulas \eqref{newformulaRT} into the identity $|T(0)|^2 + |R_{-}(0)|^2 = 1$, we find
$$
b = \left(2 - a - \frac{1}{a}\right)i.
$$
The formulas giving $b$ and $c$ in terms of $a$ now lead to the desired formulas for $T(0)$ and $R(0)$.
\end{proof}

\begin{rem}\label{RemVE}
From \eqref{LESa} we see that in the very exceptional case (see the definition before Lemma \ref{lemVgen})
we have $a=1$ and therefore $T(0) = 1$ and $R_\pm(0) = 0$.
Also notice that $a=1$ in the exceptional case when the zero energy resonance is even.
When instead it is odd, we have $a=-1$, and therefore $T(0)=-1$ and $R_\pm(0)=0$.
\end{rem}

\smallskip
\subsubsection{Resolvent and spectral projection} 
If $\mathfrak{Im} \lambda \neq 0$, the resolvent of $H$ is defined by $R_V(\lambda) = (H - \lambda^{-1})^{-1}$.

Assuming first that $\mathfrak{Re}(\lambda) >0$ and $\mathfrak{Im} \lambda > 0$, 
we let $\xi + i \eta = \sqrt{\lambda}$, with $\xi,\eta>0$. 
Then $f_{\pm}(\xi + i \eta)$ can be defined through natural extensions of the above definition, 
and the resolvent $R_V(\lambda)$ is given by the kernel
$$
R_V(\lambda)(x,y) = -\frac{1}{W(\xi + i \eta)} [f_+(\max(x,y),\xi+i\eta) f_-(\min(x,y),\xi+i\eta)].
$$
Letting $\mathfrak{Im} \lambda \to 0$, (and still with the convention that $\xi>0$),
$$
R_V(\xi^2 + i0) = - \frac{1}{W(\xi)} f_+(\max(x,y),\xi) f_-(\min(x,y),\xi).
$$
Similarly,
$$
R_V(\xi^2 - i0) = - \frac{1}{W(-\xi)} f_+(\max(x,y),-\xi) f_-(\min(x,y),-\xi).
$$
By Stone's formula, the spectral measure associated to $H$ is, for $\lambda >0$
$$
E(d\lambda) = \frac{1}{2\pi i} [R_V(\lambda + i 0) - R_V(\lambda - i0)] d\lambda.
$$
The formulas above for $R_V(\lambda \pm i 0)$ lead to
$$
E(d\lambda)(x,y) = \frac{1}{4\pi} \frac{|T(\sqrt \lambda)|^2}{\sqrt \lambda} [f_-(x,\sqrt \lambda) \overline{f_-(y,\sqrt \lambda)} + f_+(x,\sqrt \lambda) \overline{f_+(y,\sqrt \lambda)}] d\lambda.
$$

\medskip
\subsection{Distorted Fourier transform}\label{secFT}

\subsubsection{Definition and first properties}\label{ssecdFT}
We adopt the following normalization for the (flat) Fourier transform on the line:
$$
\widehat{\mathcal{F}} \phi (\xi) = \widehat{\phi}(\xi) = \frac{1}{\sqrt{2\pi}} \int e^{-i\xi x} \phi(x) \, dx.
$$
As is well-known,
$$
\widehat{\mathcal{F}}^{-1} \phi = \frac{1}{\sqrt{2\pi}} \int e^{i\xi x} \phi(\xi) \, d\xi= \what{\mathcal{F}}^* \phi,
$$
and $\mathcal{F}$ is an isometry on $L^2(\mathbb{R})$.

We now define the wave functions associated to $H$:
\begin{align}
 \label{psixk}
\psi(x,\xi) := \frac{1}{\sqrt{2\pi}}
\left\{
\begin{array}{ll}
T(\xi) f_+(x,\xi) & \mbox{for $\xi \geq 0$} \\ \\
T(-\xi) f_-(x,-\xi) & \mbox{for $\xi < 0$}.
\end{array}
\right.
\end{align}
Once again, this definition a priori only makes sense for $\xi \neq 0$, 
but it can be extended by continuity to $\xi=0$.
It follows from the estimates on $T$ and $f_{\pm}$ that for any $\alpha$, $\beta$, and for any $\xi \neq 0$,
\begin{equation}
\label{derpsi}
|\partial_x^\alpha \partial_\xi^\beta \psi (x,\xi)| \lesssim \langle x \rangle^\beta \langle \xi \rangle^\alpha.
\end{equation}
The distorted Fourier transform is then defined by
\begin{align}
\label{distF}
\widetilde{\mathcal{F}} \phi(\xi) = \widetilde {\phi}(\xi) = \int_\mathbb{R} \overline{\psi(x,\xi)} \phi(x)\,dx.
\end{align}

\begin{prop}[Mapping properties of the distorted Fourier transform]\label{propFT}
With $\wtF$ defined in \eqref{distF},

\begin{itemize}
\smallskip
\item[(i)] $\widetilde{\mathcal{F}}$ is a unitary operator from $L^2$ onto $L^2$. In particular, its inverse is
$$
\widetilde{\mathcal{F}}^{-1} \phi(x) = {\widetilde{\mathcal{F}}}^\ast \phi(x) = \int_\R \psi(x,\xi) \phi(\xi)\,d\xi.
$$

\smallskip
\item[(ii)] $\widetilde{\mathcal{F}}$ maps $L^1(\mathbb{R})$ to functions in $L^\infty(\mathbb{R})$ 
that are continuous at every point except $0$, and converge to $0$ at $\pm \infty$. 

\smallskip
\item[(iii)] $\widetilde{\mathcal{F}}$ maps the Sobolev space $H^s(\mathbb{R})$ 
onto the weighted space $L^2(\langle \xi \rangle^{2s} \,d\xi)$.

\smallskip
\item[(iv)] If $\widetilde{f}$ is continuous at zero, then, for any integer $s \geq 0$
\begin{align*}
\|\langle \xi \rangle^s \partial_\xi \widetilde{f} \|_{L^2} \lesssim \| f \|_{H^s} + \| \langle x \rangle f \|_{H^s}.
\end{align*}

\end{itemize}
\end{prop}

\begin{proof}
{As in other parts of this section we follow Yafaev \cite[Chap. 5]{Yafaev}.}

\medskip
$(i)$ To see that $\mathcal{F}$ is an isometry, we use the Stone formula 
derived in the previous subsection to write, for any functions $g,h \in L^2$
(recall that $E$ is the 
spectral measure associated to $H$)
\begin{align*}
\langle g, h \rangle & = \int E
(d\lambda) g \, \overline{h} 
\\
& = \frac{1}{4\pi} \iiint \frac{|T(\sqrt{\lambda})|^2}{\sqrt \lambda} 
  \left[ f_-(x, \sqrt{\lambda}) \overline{f_-(y,\sqrt \lambda)} 
  + f_+(x,\sqrt \lambda) \overline{f_+(y,\sqrt \lambda)} \right] g(y) \overline{h(x)} \,dy \,dx \,d\lambda
  .
\end{align*}
Changing the integration variable to $\xi = \sqrt{\lambda}$, this is 
\begin{align*}
\frac{1}{2\pi} \int_{\xi \in \R_+} \int_{\R_x} \int_{\R_y} |T(\xi)|^2 \left[ f_-(x,\xi) \overline{f_-(y,\xi)} 
  + f_+(x,\xi) \overline{f_+(y,\xi)} \right] g(y) \overline{h(x)} \,dy \,dx \, d\xi 
  = \langle \widetilde{g} \,,\, \widetilde{h} \rangle.
\end{align*}


To see that the range of $\widetilde{\mathcal{F}}$ is $L^2$, we argue by contradiction. 
If this was not the case, there would exist $g \in L^2$ not zero, such that, 
%
%
for any $f \in \mathcal{C}^\infty_0$ 
{and any $0<R_0<R_1$
$$\langle g \,,\,\widetilde{\mathcal{F}} E([R_0^2,R_1^2]) f \rangle = 0.$$
Using the spectral theorem representation, and the intertwining identity $\wtF H = k^2 \wtF$, 
we deduce that
$$
\widetilde{\mathcal{F}} \left[ E([R_0^2,R_1^2]) f \right] (\xi) = 
  \big( \mathbf{1}_{[-R_1,-R_0]}(\xi) + \mathbf{1}_{[R_0,R_1]}(\xi) \big) \wt{f}(\xi).
$$
Therefore, for any $f \in \mathcal{C}^\infty_0$ 
\begin{equation*}
0 = \langle g \,,\,\widetilde{\mathcal{F}} E([R_0^2,R_1^2]) f \rangle 
  = \int_{R_0}^{R_1} \int \big[ \psi(x,\xi) g(\xi)  + \psi(x,-\xi) g(-\xi) \big] \overline{f(x)} \,dx \,d\xi.
\end{equation*}
This implies that 
$$
\int_{R_0}^{R_1} \big[ \psi(x,\xi)g(\xi) + \psi(x,-\xi) g(-\xi) \big] \,d\xi = 0.
$$
Since $R_0,R_1$ are arbitrary we deduce
$
\psi(x,\xi)g(\xi) + \psi(x,-\xi) g(-\xi) = 0,
$
a.e. $\xi$.
Since $x \mapsto \psi(x,\xi)$ and $x \mapsto \psi(x,-\xi)$ 
are independent functions (non-vanishing Wronskian), this implies $g = 0$.
}


\medskip
$(ii)$ is a consequence of \eqref{derpsi} and the Riemann-Lebesgue lemma.

\medskip 
$(iii)$ is a consequence of Theorem \ref{thmweder} below.

\medskip
$(iv)$ Focusing on $x>0$ (through a smooth cutoff function $\chi_+$) and $\xi>0$, the distorted Fourier transform can be written as a pseudodifferential operator
$$
\widetilde{\mathcal{F}} \left[ \chi_+ f \right] (\xi) = \int a(x,\xi) e^{-ix\xi} f(x)\,dx,
$$
with symbol
$$
a(x,\xi) = \frac{1}{\sqrt{2\pi}} \overline{T(\xi) m_+(x,\xi)}.
$$
Taking a derivative in $\xi$,
$$
\partial_\xi \widetilde{\mathcal{F}} \left[ \chi_+ f \right] (\xi) = \int \partial_\xi a(x,\xi) e^{-ix\xi} f(x)\,dx + \int a(x,\xi) e^{-ix\xi} (-ix) f(x)\,dx.
$$
From the bounds \eqref{mgood} and~\eqref{TRk}, 
along with a classical theorem on the boundedness of pseudo-differential operators, 
the statement $(iv)$ follows for $s = 0$. If $s\in \mathbb{N}$, it suffices to multiply the above by 
$\langle \xi \rangle^s$, and integrate by parts in $x$ in the integrals. 
We only discussed the case of positive frequencies, but the case of negative frequencies is 
identical. It remains to check that no singularity arises at $\xi=0$ 
when applying $\xi>0$, which is ensured by the assumption that $\wt{f}$ is continuous.
\end{proof}

\medskip
\begin{lem}\label{lemevenodd}
If the potential $V$ is even, then the distorted Fourier transform preserves evenness and oddness.
\end{lem}

\begin{proof}
Observe that when $V$ is even we have the relation $f_+(x,\xi) = f_-(-x,\xi)$ between the
generalized eigenfunctions \eqref{f+-}, by uniqueness of solutions for the ODE.
From this and the definition \eqref{psixk} we see that $\psi(x,\xi)=\psi(-x,-\xi)$ .
The preservation of parity for the distorted Fourier transform then follows
directly from the definition \eqref{distF}.
\end{proof}

\medskip
As appears in Proposition \ref{propFT}, one of the main differences between the mapping properties of 
$\widehat{\mathcal{F}}$ and $\widetilde{\mathcal{F}}$ has to do with zero frequency. 
Since the zero frequency furthermore plays a key role in the nonlinear analysis developed in the present paper,
we investigate this question a bit more.

\begin{itemize}

\item If $V$ is generic, then $\psi(x,0) = 0$ and $\widetilde{f}(0) = 0$ if $f \in L^1$. 
Furthermore, assuming better integrability properties at $\infty$,
\begin{align}\label{wtf0gen}
\begin{split}
& \mbox{if $\xi > 0$}, \qquad \widetilde{f}(\xi) = - \frac{\alpha\xi}{\sqrt{2\pi}} \int f(x) f_+(x,0)\,dx +O(\xi^2)
\\
& \mbox{if $\xi < 0$}, \qquad \widetilde{f}(\xi) = \frac{\alpha\xi}{\sqrt{2\pi}} \int f(x) f_-(x,0)\,dx+O(\xi^2),
\end{split}
\end{align}
where $\alpha$ was defined in \eqref{LEST}.
Thus, $\widetilde{f}$ is typically continuous, but not continuously differentiable at zero.

\smallskip
\item If $V$ is exceptional, then
\[
\sqrt{2\pi} \, \psi(x,0+) = \frac{2a}{1+a^2} f_+(x,0), 
  \qquad \mbox{and} \qquad \sqrt{2\pi} \, \psi(x,0-) = \frac{1}{a} \psi(x,0+),
\]
where $a$ was defined in Proposition \ref{LES}. Therefore, if $f \in L^1$,
\begin{align}\label{wtf0exc}
\begin{split}
\widetilde{f}(0+) = \frac{2a}{1+a^2} \frac{1}{\sqrt{2\pi}} \int f(x) f_+(x,0) \,dx 
\qquad \mbox{and} \qquad \widetilde{f}(0-) = \frac{1}{a} \widetilde{f}(0+).
\end{split}
\end{align}
As a consequence, $\widetilde{f}$ is continuous if $a=1$, but might not be otherwise.
\end{itemize}

\smallskip
\subsubsection{Fourier multipliers}\label{ssecmultipliers}
Given $m$ a function on the real line, the flat and distorted Fourier multipliers are defined by
\begin{align*}
& m(D) = \widehat{\mathcal F}^{-1} m(\xi)  \widehat{\mathcal F} 
\\
& m(\widetilde D) =  \widetilde{\mathcal F}^{-1} m(\xi)  \widetilde{\mathcal F}.
\end{align*}
Denoting
$H_0$ and $H$ for the flat and perturbed Schr\"odinger operators
$$
H_0 = - \partial_x^2, \qquad H = - \partial_x^2 + V,
$$
these operators are diagonalized by $\widehat{\mathcal{F}}$ and $\widetilde{\mathcal{F}}$, giving the functional calculus
\begin{align*}
& f(H_0) = \widehat{\mathcal{F}}^{-1} f(\xi^2)\widehat{\mathcal{F}}\\
& f(H) = \widetilde{\mathcal{F}}^{-1} f(\xi^2) \widetilde{\mathcal{F}}.
\end{align*}
In particular,
$$
e^{it \sqrt{1 + H_0}} = e^{it \langle D \rangle} \qquad \mbox{and} \qquad e^{it \sqrt{1 + H}} = e^{it \langle \widetilde D \rangle}.
$$

\begin{lem} \label{reality}
Assume that $f$ is real-valued, and that $m$ is even and real-valued. Then $m(\widetilde{D}) f$ is real-valued.
\end{lem}

\begin{proof} This follows from the simple observation that $f$ is real valued if and only if
$$
\left\{
\begin{array}{l}
T(\xi) \widetilde{f}(\xi) = - T(-\xi) R_+(\xi) \overline{\widetilde{f}(\xi)} + \overline{\widetilde{f}(-\xi)} \qquad \mbox{for $\xi>0$} \\
T(-\xi) \widetilde{f}(\xi) = - T(\xi) R_-(-\xi) \overline{\widetilde{f}(\xi)} + \overline{\widetilde{f}(-\xi)} \qquad \mbox{for $\xi<0$}.
\end{array}
\right.
$$
\end{proof}

\smallskip
\subsubsection{The wave operator}
The wave operator $\mathcal{W}$ is given by
$$
\mathcal{W} = \operatorname{s-lim}_{t\to \infty} e^{itH} e^{-itH_0}.
$$

\begin{prop}\label{propW}
The wave operator is unitary on $L^2$ and given by
$$
\mathcal{W} =  \widetilde{\mathcal{F}}^{-1}  \widehat{\mathcal{F}}.
$$
As a consequence,
\begin{align}\label{propWdef}
\mathcal{W}^{-1} = \mathcal{W}^* = \widehat{\mathcal{F}}^{-1} \widetilde{\mathcal{F}},
\end{align}
and the wave operator intertwines $H$ and $H_0$:
$$
f(H) = \mathcal{W} f(H_0) \mathcal{W}^*.
$$
\end{prop}

\begin{proof}
In order to prove the desired formula for the wave operator, it suffices to check that, for any $f \in L^2$,
$$
\left\| e^{itH} e^{-itH_0} f - \widetilde{\mathcal{F}}^{-1}  \widehat{\mathcal{F}} f \right\|_2 \rightarrow 0.
$$
By the functional calculus, this is equivalent to
$$
\left\| \widehat{\mathcal{F}}^{-1} e^{it \xi^2} f - \widetilde{\mathcal{F}}^{-1} e^{it \xi^2} f \right\|_2 \rightarrow 0.
$$
By unitarity, it suffices to check the above for a dense subset of $f$, 
thus we might assume $f \in \mathcal{C}^\infty_0$. 
By symmetry between positive and negative frequencies, 
we can furthermore assume that $\operatorname{Supp} f \subset (0,\infty)$.
Therefore, matters reduce to proving that
$$
\left\| \int_0^\infty e^{i(x\xi + t\xi^2)} (1- T(\xi) m_+(x,\xi)) f(\xi)\,d\xi \right\|_{L^2_x} \rightarrow 0.
$$
To see that the above is true, we split the function whose $L^2$ norm we want to estimate into
\begin{align*}
& \int_0^\infty e^{i(x\xi + t\xi^2)} \mathbf{1}_+(x) (1-T(\xi) m_+(x,\xi)) f(\xi)\,d\xi 
  - \int_0^\infty e^{i(-x\xi + t\xi^2)} \mathbf{1}_-(x) R_-(\xi) m_-(x,\xi) f(\xi)\,d\xi
\\
& \qquad +  \int_0^\infty e^{i(x\xi + t\xi^2)} \mathbf{1}_-(x) (1- m_-(x,\xi)) f(\xi)\,d\xi
\\
& = I + II + III.
\end{align*}
The terms $I$ and $II$ have non-stationary phases, 
from which it follows that they converge to zero as $t \to \infty$. 
As for $III$, it goes to zero pointwise by the stationary phase lemma, 
and is uniformly (in $t$) bounded by a decaying function of $x$, as follows from the estimates on $m_-$;
therefore, it goes to zero in $L^2$ by the dominated convergence theorem.
\end{proof}

Finally, the following theorem gives boundedness of the wave operators on Sobolev spaces.

\begin{thm}[Weder \cite{Weder1}]\label{thmweder}
$\mathcal{W}$ and $\mathcal{W}^*$ extend to bounded operators on $W^{k,p}(\R)$ for any $k$ and $1 < p < \infty$. 
Furthermore, in the exceptional case, if $f_+(-\infty,0) = 1$, this remains true if $p=1$ or $\infty$.
\end{thm}

\smallskip
\subsubsection{What if discrete spectrum is present?}
The above discussion relied on the assumption that
$$
L^2_{ac} = P_{ac} L^2 = L^2,
$$
where we denoted $P_{ac}$ the projector on the absolutely continuous spectrum of $H$. 
Since we are assuming $V \in \mathcal{S}$, we can exclude singularly continuous spectrum, 
as well as embedded discrete spectrum, 
but there might be a finite number of negative eigenvalues $\lambda_N < \dots < \lambda_1 < 0$,
with corresponding eigenfunctions $\phi_1,\dots,\phi_N$, see \cite{DeiTru}.
Then all the statements made above require small adaptations. 
Indeed, $\widetilde{\mathcal{F}}$ is zero on $\phi_j$ for all $j$, and unitary from $L^2_{ac}$ to $L^2$. Thus,
$$
\wtF \wtF^{-1} = \operatorname{Id}_{L^2} \qquad \mbox{and} \qquad \wtF^{-1} \wtF = P_{ac}.
$$

\medskip
\subsection{Decomposition of $\psi(x,\xi)$}\label{ssecpsi}
Let $\rho$ be an even, smooth, non-negative function, equal to $0$ outside of $B(0,2)$ and such that $\int \rho = 1$. Define $\chi_{\pm}$ by
\begin{equation}
\label{chi+-}
\chi_+(x) = H * \rho = \int_{-\infty}^{x} \rho(y)\,dy , \quad \mbox{and} \quad \chi_+(x) + \chi_-(x) = 1,
\end{equation}
where $H$ is the Heaviside function, $H=\mathbf{1}_{+}$. Notice that
$$
\chi_+(x) = \chi_-(-x).
$$

With $\chi_\pm$ as above, and using the definition of $\psi$ in \eqref{psixk}
and $f_\pm$ and $m_\pm$ in \eqref{f+-}-\eqref{m+-}, as well as the identity \eqref{f+f-} we can write
\begin{align}
\label{psi>}
\begin{split}
\mbox{for} \quad \xi>0 \qquad \sqrt{2\pi}\psi(x,\xi) & = \chi_+(x)T(\xi)m_+(x,\xi)e^{ix\xi}
  \\ & + \chi_-(x) \big[ m_-(x,-\xi)e^{i\xi x} + R_-(\xi) m_-(x,\xi)e^{-i\xi x} \big],
\end{split}
\end{align}
and
\begin{align}
\label{psi<}
\begin{split}
\mbox{for} \quad \xi<0 \qquad \sqrt{2\pi}\psi(x,\xi) & = \chi_-(x)T(-\xi)m_-(x,-\xi)e^{ix\xi}
  \\ & + \chi_+(x) \big[m_+(x,\xi) e^{i\xi x} + R_+(-\xi)m_+(x,-\xi) e^{-i\xi x} \big].
\end{split}
\end{align}

We then decompose
\begin{align}
\label{psidec}
\sqrt{2\pi}\psi(x,\xi) = \psi^S(x,\xi) + \psi^R(x,\xi) ,
\end{align}
where, on the one hand, the singular part (non-decaying in $x$) is
\begin{align}
\label{psiSdec}
\begin{split}
\mbox{for} \quad \xi>0 \qquad \psi^S(x,\xi) & 
  := \chi_+(x) T(\xi) e^{i\xi x } + \chi_-(x) (e^{i\xi x} + R_-(\xi)  e^{-i\xi x}),
\\
\mbox{for} \quad \xi<0 \qquad \psi^S(x,\xi) & 
  := \chi_-(x) T(-\xi) e^{i\xi x} + \chi_+(x) (e^{i\xi x} + R_+(-\xi)  e^{-i\xi x}),
\end{split}
\end{align}
and the regular part is
\begin{align}
\label{psiR+-}
\begin{split}
\mbox{for} \quad \xi>0 \qquad \psi^R(x,\xi) & := \chi_+(x) T(\xi) (m_+(x,\xi)-1) e^{i\xi x}
  \\ & + \chi_-(x) \big[ (m_-(x,-\xi) - 1)e^{i\xi x} + R_-(\xi)(m_-(x,\xi) - 1) e^{-ix\xi} \big],
\\
\mbox{for} \quad \xi<0 \qquad \psi^R(x,\xi) & := \chi_-(x) T(-\xi) (m_-(x,-\xi)-1) e^{i\xi x}
  \\ & + \chi_+(x) \big[ (m_+(x,\xi) - 1)e^{i\xi x} + R_+(-\xi)(m_+(x,-\xi) - 1) e^{-ix\xi} \big].
\end{split}
\end{align}

\medskip
\subsection{Linear estimates}\label{secLin}
Recall that  $\langle D \rangle = \sqrt{ -\partial_x^2 + 1 }$ and
$\langle \wt{D} \rangle= \sqrt{ -\partial_x^2  + V + 1} = \wtF^{-1} \jxi \wtF$.

\begin{prop}[Dispersive estimates]\label{propdisp}
Recall the definition \eqref{wnorm} with \eqref{wnormparam0}. 
The following statements hold true:

\smallskip
\begin{itemize}

\item[(i)] For any $0\leq |t| \leq T$, and for $I =  [0,\infty)$ or $(-\infty,0]$,
\begin{align}\label{disp1}
\begin{split}
{\| e^{\pm it\langle D \rangle}  \mathbf{1}_I (D)  f \|}_{L^\infty_x} 
  \lesssim  \frac{1}{\jt^{1/2}} \, {\big\| \jxi^{3/2} \what{f} \big\|}_{L^\infty_\xi}
  +   \frac{1}{\jt^{3/4 - \alpha - \beta \gamma}} {\big\| 
  \jxi 
  \partial_\xi \what{f} \big\|}_{W_T} 
  + \frac{1}{\jt^{7/12}}  {\| \jxi^4 \what{f} \|}_{L^2}.
\end{split}
\end{align}

\smallskip
\item[(ii)] If $V$ satisfies the a priori assumptions of Theorem \ref{maintheo}, then for any $0\leq |t| \leq T$,
\begin{align}\label{disp2}
\begin{split}
{\big\| e^{\pm it \langle \wt{D} \rangle}  f \big\|}_{L^\infty_x} 
 \lesssim  \frac{1}{\jt^{1/2}} \, {\big\| \jxi^{3/2} \wt{f} \big\|}_{L^\infty_\xi}
  +   \frac{1}{\jt^{3/4 - \alpha - \beta\g}} {\big\| 
  \jxi 
  \partial_\xi \wt{f} \big\|}_{W_T} 
  +  \frac{1}{\jt^{7/12}}  {\| \jxi^4 \wt{f} \|}_{L^2}.
\end{split}
\end{align}



\end{itemize}

\end{prop}

A more precise asymptotic formula with an explicit leading order term
can be read off the proof of Proposition \ref{propdisp}; 
in particular, up to a faster decaying remainder of the same form of those appearing in \eqref{disp2},
we have
\begin{align}\label{remlinear}
e^{it\langle \widetilde D \rangle} f \approx \frac{e^{i \frac{\pi}{4}}}{\sqrt{2t}} 
  \langle \xi_0 \rangle^{3/2} e^{i t \langle \xi_0 \rangle + i x \xi_0} \widetilde{f} \left( \xi_0 \right) 
  \qquad \mbox{as $t \to \infty$}, \quad  \frac{\xi_0}{\langle \xi_0 \rangle} = - \frac{x}{t}.
\end{align}

\begin{rem}\label{disprem}
Note that, in view of \eqref{wnormparam0}, we have $\alpha + \beta\gamma < 1/4$.
Therefore, uniform-in-time control of the profile in $\wtF^{-1}\jxi^{-3/2} L^\infty$ and $W_t$, 
and in $H^4$  with small time growth, gives the sharp $|t|^{-1/2}$ decay for linear solutions through \eqref{disp2}.

Furthermore, let $\mathbf{a}$ be any of the coefficients defined in \eqref{mucoeffexp}.
In view of \eqref{disp1} and the regularity of $T(\xi)$ and $R(\xi)$ in \eqref{TRk}, we have
\begin{align}\label{dispinfty}
\begin{split}
{\big\| e^{\pm it \langle D \rangle} \whF^{-1} \big(\mathbf{a}(\xi) \wt{f} \, \big) \big\|}_{L^\infty_x}
  \lesssim \frac{1}{\jt^{1/2}} {\| \jxi^{3/2} \wt{f} \|}_{L^\infty_\xi} +  \frac{1}{\jt^{3/4-\alpha -\beta\gamma}} \big( {\| \jxi \partial_\xi \wt{f} \|}_{W_t}  
  	+ {\| \wt{f} \|}_{L^2} \big)
	\\ +   \frac{1}{\jt^{7/12}}  {\| \jxi^4 \wt{f} \|}_{L^2}.
\end{split}
\end{align}
\end{rem}

\smallskip
\begin{rem}\label{remdispint}
Besides the pointwise decay estimates of Proposition \ref{propdisp} above, we will also use the following variant:
for $k \geq 5$
\begin{align}\label{dispint}
{\| e^{\pm it\langle D \rangle} \varphi_k(D) f \|}_{L^\infty} 
  \lesssim \frac{1}{t^{1/2}} 2^{3k/2} {\big\| \varphi_k\widehat{f} \big\|}_{L^2}^{1/2} 
   \big( {\big\| \varphi_k\partial_{\xi}\widehat{f} \big\|}_{L^2} + {\big\| \varphi_k\widehat{f} \big\|}_{L^2}  \big)^{1/2},
\end{align}
which follows from the standard $L^1\rightarrow L^\infty$ decay and the 
interpolation inequality 
\begin{align*}
{\big\| \varphi_k(D) f \big\|}_{L^1}
	\lesssim {\big\| \varphi_k(D) f \big\|}_{L^2}^{1/2}  {\big\| x \, \varphi_k(D) f \big\|}_{L^2}^{1/2}.
\end{align*}
Notice that \eqref{dispint} also implies, see \eqref{propWdef},
\begin{align}\label{dispintW}
{\| e^{\pm it\langle D \rangle} \W \varphi_k(\wt{D}) f \|}_{L^\infty} 
  \lesssim \frac{1}{|t|^{1/2}} 2^{3k/2} {\big\| \varphi_k\wt{f} \big\|}_{L^2}^{1/2} 
   \big( {\big\| \varphi_k\partial_{\xi}\wt{f} \big\|}_{L^2} + {\big\| \varphi_k\wt{f} \big\|}_{L^2} \big)^{1/2}.
\end{align}

\end{rem}


To prove Proposition \ref{propdisp}, we use the following stationary phase lemma: 

\begin{lem}
\label{lemoscil}
 Consider for $X \in \mathbb{R}$, $t  \geq 0$, $x\in \R$ the integrals 
\begin{equation*}
I_{\mu, \nu}(t,X,x) = \int_{\R_{\mu}}  e^{it ( \nu \langle \xi \rangle -  \xi X)} a(x,\xi) g(\xi)\, d\xi, \qquad \mu,\nu \in \{+,-\},
 \end{equation*}
 and assume that 
\begin{equation}
\label{hypoa}
  \sup_{x \in \R, \, \xi \in \R_\mu}\left(|a(x,\xi)| + \langle \xi \rangle |\partial_{\xi}a(x,\xi)| \right)\lesssim 1.
\end{equation}
Then we have the estimate
\begin{align}
\begin{split}\label{dispdes}
\big| I_{\mu, \nu}(t,X,x) \big| \lesssim \frac{1}{\jt^{1/2}} {\| \jxi^{3/2} g(\xi) \|}_{L^\infty}
	+ \frac{1}{\jt^{3/4 - \alpha - \beta\gamma}} {\| 
	\jxi \partial_\xi g \|}_{W_T} 
	+ \frac{1}{\jt^{7/12}}  {\| \jxi^4 g \|}_{L^2}.
\end{split}
\end{align}
\end{lem}

We postpone the proof of the lemma and give first the proof of Proposition \ref{propdisp}.

\smallskip
\begin{proof}[Proof of Proposition \ref{propdisp}]
In order to prove \eqref{disp1}, we write
\begin{align*}
e^{\pm i t \langle D \rangle} \mathbf{1}_I (D) f = \frac{1}{\sqrt{2\pi}} \int_{I} e^{i t(\pm \langle \xi \rangle - \xi X)} \what{f}(\xi)\, d\xi, \quad X := -  x/t,
\end{align*}
and use Lemma \ref{lemoscil} on $I=\R_{+}$ or $\R_{-}$ and $a\equiv1$.
 
 To prove \eqref{disp2}, we use the distorted Fourier inversion, see \eqref{propFT}, to write
\begin{align*} 
e^{\pm i t \langle \wt{D} \rangle} f 
 = \int_{\R_+}  e^{\pm i t \langle \xi \rangle} \psi(x,\xi) \wt{f}(\xi)\, d\xi 
 +  \int_{\R_-}  e^{\pm i t \langle \xi \rangle}  \psi (x,\xi) \wt{f}(\xi)\, d\xi.
\end{align*}
Let us estimate the first integral, the other one being similar.
Using \eqref{psi>}, we can write
\begin{align*}
\nonumber
\sqrt{2\pi} \int_{\R_{+}}  e^{\pm i t \langle \xi \rangle} \psi(x,\xi) f(\xi)\, d\xi
	& =  \chi_+(x) \int_{\R_+}  e^{ i t(\pm \langle \xi \rangle - \xi X)} T(\xi) m_{+}(x,\xi) f(\xi)\, d\xi
	\\
	& + \chi_-(x) \int_{\R_+}  e^{i t(\pm \jxi - \xi X)}  m_{-}(x,-\xi) f(\xi)\, d\xi
	\\
	&+ \chi_-(x) \int_{\R_+}  e^{ i t(\pm \jxi  + \xi X)} R_{-}(\xi) m_{-}(x,\xi) f(\xi)\, d\xi.
\end{align*}
Then, the desired estimate follows by using Lemma \ref{lemoscil} with $a(x,\xi)= T(\xi) m_{+}(x,\xi)$,
$m_{-}(x,-\xi)$, and $R_-(\xi)m_{-}(x,\xi)$, where the assumption \eqref{hypoa} holds thanks to Lemmas 
\ref{lemm+-} and  \ref{lemTR}.
\end{proof}

\medskip
\begin{proof}[Proof of Lemma \ref{lemoscil}]
It suffices to consider only the case $\mu=+$, $\nu=+$ and $t \geq 1$, $X \geq 0$; all other cases are similar or easier.
We let 
\begin{align}\label{disppr1}
I_{++} = \sum_{k\in\Z} I_k , \qquad I_k(t,X) := \int_{\R_+}  e^{it ( \jxi - \xi X)}  a(x,\xi) g(\xi) \varphi_k(\xi) \, d\xi.
\end{align}
First notice that since
\begin{align}\label{disppr2}
\big| I_k(t,x) \big| \lesssim \int_{\R_+}  | g(\xi) | \varphi_k(\xi) \, d\xi \lesssim \min \big( 2^k {\| g \|}_{L^\infty}, 2^{-7k/2} {\| \jxi^4 g \|}_{L^2} \big) 
\end{align}
we see that $|I_{++}|$ enjoys the desired bound if  $2^k \gtrsim t^{1/6}$, or $2^k \lesssim t^{-1/2}$.
From now on we assume 
\begin{align}\label{dispprk}
C t^{-1/2} \leq 2^k \leq (1/C) t^{1/6}
\end{align}
for a suitably large absolute constant $C>0$.

Let us denote
\begin{align}\label{dispprphi}
\begin{split}
& \phi_{X}(\xi) := \jxi -   \xi X, \quad
\\
& \phi_X'(\xi) = \frac{\xi}{\jxi} - X, 
  \quad \phi_{X}''(\xi) = \frac{1}{\jxi^3}, \qquad \xi_{0} := X/\sqrt{1-X^2},
\end{split}
\end{align}
and note that the phase $\phi_X$ has no stationary points if $X \geq 1$, and a unique, non-degenerate, 
stationary point at $\xi_0$ for any $X \in [0, 1)$.
Consider $n \in \Z_+$ such that $t \in [2^{n-1},2^n]$ and let $q_0 \in \Z$ be the smallest integer such that 
$2^{q_0} \geq 2^{(3/2)k^+} 2^{-n/2} \approx \jxi^{3/2} t^{-1/2}$.
Note that $2^{q_0} \ll \min(2^k,1)$, if $C$ in \eqref{dispprk} is large enough.

In what follows we may assume that $|\xi_0|\approx 2^k$, for otherwise there is no stationary point
on the support of \eqref{disppr1}, $|\phi_X'(\xi)| \gtrsim 2^k 2^{-3k^+}$, and the proof of the statement is easier.
In other words, for fixed $\xi_0$, we may assume that there are only a finite number of indexes $k$
for which $I_k$ does not vanish.

Using the notation \eqref{cut2} 
we decompose 
\begin{align}\label{dispprkq}
I_k =  \sum_{q \in [q_0,\infty) \cap \Z} I_{k,q},
	\qquad I_{k,q}(t,X) := \int_{\R_+}  e^{it ( \jxi - \xi X)}  a(x,\xi) \, \varphi_q^{(q_0)}(\xi-\xi_0) \, \varphi_k(\xi) \, g(\xi)\, d\xi.
\end{align}
Bounding the contribution to the sum over $k$ of the term with $q = q_0$ is immediate.
Let us then consider $q > q_0$ and note that on the support of the integral in \eqref{dispprkq} we have
\begin{align}\label{dispprphi'}
|\phi_X'(\xi)| \approx |\xi-\xi_0| |\phi_X''| \approx 2^{q} 2^{-3k^+} \gtrsim 2^{-n/2} 2^{-(3/2)k^+}.
\end{align}
Integrating by parts using $(it \phi_{X}')^{-1} \partial_\xi e^{it \phi_X} =  e^{it \phi_X}$, we obtain:
\begin{align}\label{dispprIBP}
\begin{split}
& I_{k,q} = \frac{1}{it} \big[ J_{k,q}^{(1)} + J_{k,q}^{(2)} + J_{k,q}^{(3)} + J_{k,q}^{(4)} \big] ,
\\
& J_{k,q}^{(1)}(t,X) = \int_{\R_+}  e^{it ( \jxi - \xi X)} \frac{\phi_X''}{(\phi_X')^2} a(x,\xi) \, \varphi_q(\xi-\xi_0) \, \varphi_k(\xi) \, g(\xi) \, d\xi,
\\
& J_{k,q}^{(2)}(t,X) = -\int_{\R_+}  e^{it ( \jxi - \xi X)}  \frac{1}{\phi_X'} \, \partial_\xi a(x,\xi) \, \varphi_q(\xi-\xi_0) \, \varphi_k(\xi) \, g(\xi) \, d\xi,
\\
& J_{k,q}^{(3)}(t,X) = - \int_{\R_+}  e^{it ( \jxi - \xi X)}  \frac{1}{\phi_X'} a(x,\xi) \, \partial_\xi \big[ \varphi_q(\xi-\xi_0) \, \varphi_k(\xi) \big] \, g(\xi) \, d\xi,
\\
& J_{k,q}^{(4)}(t,X) = - \int_{\R_+}  e^{it ( \jxi - \xi X)}  \frac{1}{\phi_X'} a(x,\xi) \, \varphi_q(\xi-\xi_0) \, \varphi_k(\xi) \, \partial_\xi g(\xi) \, d\xi.
\end{split}
\end{align}

Using \eqref{dispprphi} and \eqref{dispprphi'} and changing the index of summation $q \mapsto p+(3/2)k^+$ we have
\begin{align}\label{dispprb1}
\begin{split}
\Big| \sum_{q\geq q_0, k} J_{k,q}^{(1)}(t,X) \Big|& \lesssim 
  \sum_{p \geq -n/2, k} \int_{\R_+} \frac{\jxi^3}{(\xi-\xi_0)^2} \, \varphi_{\sim p}\big( (\xi-\xi_0)2^{-(3/2)k^+} \big) \, \varphi_k(\xi) \, |g(\xi)| \, d\xi 
  \\
& \lesssim  t^{1/2} {\| \jxi^{3/2} g \|}_{L^\infty}.
\end{split}
\end{align}
For the second term, using $|\partial_\xi a|\leq 1$,  we can estimate
\begin{align}\label{dispprb2}
\begin{split}
\Big| \sum_{q\geq q_0, k} J_{k,q}^{(2)}(t,X) \Big| &\lesssim 
	\sum_{p \geq -n/2, k} \int_{\R_+} \frac{\jxi^3}{|\xi-\xi_0|} \, \varphi_{\sim p}\big( (\xi-\xi_0) 2^{-(3/2)k^+} \big) \, \varphi_k(\xi) \, |g(\xi)| \, d\xi 
	\\
&	\lesssim  t^{1/2} {\| g \|}_{L^\infty}.
\end{split}
\end{align}
The third term in \eqref{dispprIBP} is similar to the first one:
\begin{align}\label{dispprb3}
\begin{split}
\Big| \sum_{q\geq q_0, k} J_{k,q}^{(3)}(t,X) \Big| 
	\lesssim  t^{1/2} {\| \jxi^{3/2} g \|}_{L^\infty}.
\end{split}
\end{align}
The upper bounds \eqref{dispprb1}-\eqref{dispprb3}, after being multiplied by $t^{-1}$ 
are bounded by the first term on the right-hand side of \eqref{dispdes}.

For the last integral in \eqref{dispprIBP} we want to distinguish cases depending on the location of $\xi$ relative to the frequency $\sqrt{3}$.
We insert cutoffs $\varphi_\ell^{(\ell_0)}(\xi-\sqrt{3})$, for $\ell_0 := -\gamma n$, and bound
\begin{align}\label{disppr4}
\begin{split}
& t^{-1} \big| J_{k,q}^{(4)} \big| \lesssim t^{-1} \big[ K_{\leq \ell_0} + \sum_{\ell_0 < \ell \leq 0} K_{\ell} + K_{>0} \big] 
\\
& K_{\ast}(t,X) = \int_{\R_+} \frac{1}{|\phi_X'|} |a(x,\xi)| \, \varphi_q(\xi-\xi_0) \, \varphi_k(\xi) \, \varphi_{\ast}\big(\xi-\sqrt{3}\big) \, |\partial_\xi g(\xi)| \, d\xi.
\end{split}
\end{align}
The first term can be estimated as follows:
\begin{align*}
t^{-1} \sum_{q > q_0} \big| K_{\leq \ell_0}(t,X) \big| 
	& \lesssim t^{-1} \sum_{q>q_0} 2^{-q} \cdot 2^{\min(q,-\g n)/2} \cdot {\big\|  \varphi_{\leq -\g n} \big(\xi-\sqrt{3}\big)\partial_\xi g \big\|}_{L^2}
	\\
	& \lesssim t^{-1} \sum_{q>q_0} 2^{-q/2} 2^{(\b \g +\alpha) n} {\| \chi_{\leq 0,\sqrt{3}} \, \partial_\xi g \|}_{W_T}
	\\
	& \lesssim t^{-1} 2^{-q_0/2} t^{\b\g +\alpha} {\| \chi_{\leq 0,\sqrt{3}} \, \partial_\xi g \|}_{W_T},
\end{align*}
consistently with \eqref{dispdes}, since we must have $|k|\leq 5$ and $2^{q_0} \approx t^{-1/2}$.
The last term in \eqref{disppr4} can be estimated similarly:
\begin{align}\label{disppr10}
\begin{split}
t^{-1} \sum_{q>q_0} \big| K_{>0}(t,X) \big| 
	& \lesssim t^{-1} \sum_{q>q_0} 2^{3k^+-q} \cdot 2^{\min(q,k)/2} \cdot 
	{\big\| \varphi_k \varphi_{>0} \big(\xi-\sqrt{3}\big)  \partial_\xi g \big\|}_{L^2}
	\\
	& \lesssim t^{-3/4} 2^{(5/4)k^+} \cdot t^{\alpha} {\| \chi_{> 0,\sqrt{3}} \jxi \partial_\xi g \|}_{W_T};
\end{split}
\end{align}
upon summing over $2^k \leq t^{1/6}$ the right-hand side of \eqref{disppr10},
we obtain a contribution bounded by the second term on the right-hand side of \eqref{dispdes},
since $-3/4 + (5/4) (1/6) + \alpha < 
-3/4 + \alpha + \beta \gamma$ with our choice of parameters 
in \eqref{wnormparam0} (provided, for example, that $\beta',\gamma' < 1/24$).

Finally, we estimate
\begin{align*}
t^{-1} \big| K_{\ell}(t,X) \big| 
	& \lesssim t^{-1} 2^{-q} \cdot 2^{\min(q,\ell)/2} \cdot {\big\| \varphi_{\ell}\big(\xi-\sqrt{3}\big)\partial_\xi g \big\|}_{L^2}
	\\
	& \lesssim t^{-1} 2^{-q/2} \cdot 2^{-\b \ell} 2^{\alpha n} {\| \chi_{\leq 0,\sqrt{3}} \, \partial_\xi g \|}_{W_T},
\end{align*}
and, using again that $|k|\leq 5$ for $\ell \leq 0$,
we see that this contributions can be summed over $q>q_0$ with $2^{q_0} \gtrsim t^{-1/2}$,
and $\ell > \ell_0$ with $2^{\ell_0} \gtrsim t^{-\g}$, and be bounded by the 
second term on the right-hand side of \eqref{dispdes}.
\end{proof}


\bigskip
\section{The quadratic spectral distribution}\label{secmu}

In this section we study the distribution \eqref{intromu0}.

\subsection{The structure of the quadratic spectral distribution}

Recall that we denote, for a function $f$,
$$
f_+(x) = f(x) \quad \mbox{and} \quad f_-(x) = \overline{f(x)}.
$$

\begin{prop}\label{muprop}
Under the assumptions on $V=V(x)$ and $a=a(x)$ in Theorem \ref{maintheo}, 
there exists a tempered distribution $\mu_{\iota_1 \iota_2} \in \mathcal{S}'(\R^3)$,
for $\iota_1,\iota_2 \in \{+,-\}$, such that, if $f,g \in \mathcal{S}$,
$$
\widetilde{\mathcal{F}} \left[ a(x) f_{\iota_1} g_{\iota_2} \right] (\xi)
= \iint \widetilde{f}_{\iota_1} (\eta) \widetilde{g}_{\iota_2}(\sigma) 
\mu_{\iota_1 \iota_2} (\xi,\eta,\sigma) \,d\eta \,d\sigma.
$$
The distribution $\mu_{\iota_1 \iota_2}$ can be decomposed into
\begin{align}
\label{mudec}
\begin{split}
2\pi \mu_{\iota_1 \iota_2} (\xi,\eta,\sigma) = \mu_{\iota_1 \iota_2}^S(\xi,\eta,\sigma)  + \mu_{\iota_1 \iota_2}^R(\xi,\eta,\sigma)
\end{split}
\end{align}
where the following hold:

\setlength{\leftmargini}{1.5em}
\begin{itemize}

\medskip
\item  The `$\mathrm{singular}$' part of the distribution can be written as
\begin{align}
\label{muSdec}
\mu_{\iota_1 \iota_2}^S(\xi,\eta,\sigma) = \mu_{\iota_1 \iota_2} ^{S,-}(\xi,\eta,\sigma) + \mu_{\iota_1 \iota_2} ^{S,+}(\xi,\eta,\sigma),
\end{align}
with $\epsilon\in\{+,-\}$,
\begin{align}\label{muL+-}
\begin{split}
\mu_{\iota_1 \iota_2}^{S,\epsilon} (\xi,\eta,\sigma) :=  \ell_{\epsilon\infty}
  \sum_{\lambda,\mu,\nu \in \{+,-\}} a_{\substack{- \iota_1 \iota_2 \\ \lambda \mu \nu}}^\epsilon (\xi,\eta, \sigma) \, \Big[\sqrt{\frac \pi 2} \delta(p) + \epsilon  \pv \frac{\widehat \phi(p)}{ip} \Big],
  \\ p := \lambda \xi - \iota_1 \mu \eta - \iota_2 \nu \sigma,
\end{split}
\end{align}
where $\phi$ is a smooth, even, real-valued, compactly supported function with integral one;
the coefficients are given by
\begin{align}\label{mucoeff}
a_{\substack{\iota_0 \iota_1 \iota_2 \\ \lambda \mu \nu}}^\epsilon(\xi,\eta,\sigma) 
  = \mathbf{a}_{\lambda,\iota_0}^\epsilon(\xi) \mathbf{a}_{\mu,\iota_1}^\epsilon(\eta) 
  \mathbf{a}_{\nu,\iota_2}^\epsilon(\sigma) 
  \qquad \mbox{with} \qquad \mathbf{a}_{\mu,\iota}^\epsilon = \big( \mathbf{a}_{\mu}^\epsilon \big)_{\iota}
\end{align}
and
\begin{align}\label{mucoeffexp}
\left\{ \begin{array}{l}
\mathbf{a}^-_+(\xi) = \mathbf{1}_+(\xi) + \mathbf{1}_-(\xi) T(-\xi) \\
\mathbf{a}^-_-(\xi) = \mathbf{1}_+(\xi) R_-(\xi)
\end{array} \right. 
\qquad
\left\{ \begin{array}{l}
\mathbf{a}^+_+(\xi) = T(\xi) \mathbf{1}_+(\xi) + \mathbf{1}_-(\xi) \\
\mathbf{a}^+_-(\xi) = \mathbf{1}_-(\xi) R_+(-\xi).
\end{array} \right. 
\end{align}

\medskip
\item The `$\mathrm{regular}$' part of the distribution $\mu^R_{\iota_1\iota_2}$
can be written as a linear combination of the form
\begin{align}\label{muR}
\mu_{\iota_1 \iota_2}^R(\xi,\eta,\sigma) = \sum_{\epsilon_1,\epsilon_2,\epsilon_3 \in \{+,-\}} 
  \mathbf{1}_{\epsilon_1}(\xi)\mathbf{1}_{\epsilon_2}(\eta)\mathbf{1}_{\epsilon_3}(\s)
    \mathfrak{r}_{\epsilon_1\epsilon_2\epsilon_3}(\xi,\eta,\sigma),
\end{align}
where the symbols $\mathfrak{r}_{\epsilon_1\epsilon_2\epsilon_3}: \R^3 \rightarrow \C$ are smooth, 
and satisfy, for any non-negative integer $N$, and $a,b,c$,
\begin{align}\label{muR'}
\begin{split}
\big|\partial_\xi^a \partial_\eta^b \partial_\s^c \mathfrak{r}_{\epsilon_1\epsilon_2\epsilon_3}(\xi,\eta,\s) \big| 
  \lesssim \langle \inf_{\mu,\nu} |\xi - \mu \eta - \nu \sigma| \rangle^{-N}.
\end{split}
\end{align}
\end{itemize}
\end{prop}

\smallskip 
\begin{proof}
We proceed in a few steps.

\medskip 
\noindent 
{\it The Fourier transform of $(\chi_{\pm})^3$}.
By the choice \eqref{chi+-} of $\chi_{-}$, $ \partial_{x} (\chi_-)^3$ is a $\mathcal{C}^\infty_{c}$ function,
which we can write as $ \partial_{x} (\chi_-)^3= \phi^o - \phi^e$, where $\phi^o$ and $\phi^e$
are respectively odd and even and $\mathcal{C}^\infty_{c}$.
Furthermore, since $\phi^o$ is odd,  we can write $\phi^o =  \partial_{x} \psi$ where $\psi \in \mathcal{C}^\infty_{c}$ and $\psi$ is even.
We have thus obtained that
\begin{align*}
(\chi_-)^3= \psi + \int_{x}^{+ \infty} \phi^e (y)\, dy = \psi + \phi^e \ast \mathbf{1}_-, \qquad \int_\R\phi^e (y)\,dy =1,
\end{align*}
where we denoted $\mathbf{1}_\pm$ the characteristic function of $\{\pm x > 0\}$.
Taking the Fourier transform, and using the classical formulas
\begin{align}\label{Fsign}
\what{f \ast g} = \sqrt{2\pi} \what{f} \cdot \what{g}, \qquad \what{1} = \sqrt{2\pi} \delta_0, \qquad
\what{\sign x} = \pv \sqrt{\frac{2}{\pi}}\frac{1}{i\xi},
\end{align}
we see that $\widehat{\mathbf{1}_-} = \sqrt{\dfrac{\pi}{2}} \delta 
  - \dfrac{1}{\sqrt{2\pi}}  \pv\dfrac{1}{i\xi}$, and therefore, since $\widehat{\phi^e}(0) = \dfrac{1}{\sqrt{2\pi}}$,
\begin{equation*}
\what{(\chi_-)^3} - \what{\psi} = \widehat{\mathcal{F}} \big(\phi^e \ast \mathbf{1}_- \big)
  = \sqrt{2\pi}  \what{\mathbf{1}_-}(\xi) \what{\phi^e}(\xi)
  = \sqrt{\frac{\pi}{2}} \delta(\xi) -  \pv\frac{\what{\phi^e}(\xi)}{i\xi}.
\end{equation*}
Since $\chi_+(-x) = \chi_-(x)$, this implies a corresponding formula for $\chi_+$. 
To summarize, setting $\phi = \phi^e$,
\begin{equation}
\label{decphi-}
\what{(\chi_{\pm})^3}(\xi) = \sqrt{\frac{\pi}{2}} \delta(\xi) \pm \pv {\what{\phi}(\xi)\over i\xi} 
+ \what{\psi}(\xi).
\end{equation}

\medskip 
\noindent 
{\it The regularization step}. 
Considering for simplicity the case $\iota_1 = \iota_2 = +$. 
If $f,g,h \in \mathcal{S}$, denoting $w$ a cutoff function,
\begin{align}
\nonumber
\langle \widetilde{a(x)fg}(\xi) ,\widetilde{h}(\xi) \rangle 
	& = \iint \overline{\psi(x,\xi)} a(x) 
	\int \wt{f}(\eta) \psi(x,\eta)\,d\eta  \int \wt{g}(\sigma) \psi(x,\sigma)\,d\sigma \,dx \, \overline{\wt{h}(\xi)} \,d\xi
	\\
\nonumber
& = \lim_{R \to \infty} \iint a(x) w(x/R) \overline{\psi(x,\xi)} \int \widetilde{f}(\eta) \psi(x,\eta)\,d\eta  \int \widetilde{g}(\sigma) \psi(x,\sigma)\,d\sigma \,dx \, \overline{\widetilde{h}(\xi)} \,d\xi 
\\
\nonumber
& = \lim_{R \to \infty} \iiint \widetilde{f}(\eta) \widetilde{g}(\sigma)  \overline{\widetilde{h}(\xi)} 
  \Big( \int a(x) w(x/R) \overline{\psi(x,\xi)} \psi(x,\eta) \psi(x,\sigma) \,dx \Big) \,d\eta \,d\sigma \,d\xi
\\
& =  \iiint \widetilde{f}(\eta) \widetilde{g}(\sigma)
  \overline{\widetilde{h}(\xi)} \mu_{++} (\xi,\eta,\sigma)  \,d\eta \,d\sigma \,d\xi 
  \label{pairing}
\end{align}
where $\mu_{++}$ is defined as the limit in the sense of (tempered) distributions
\begin{align}\label{limit}
\mu_{++} (\xi,\eta,\sigma) = \lim_{R \to \infty} 
  \int a(x) w(x/R) \overline{\psi(x,\xi)} \psi(x,\eta) \psi(x,\sigma) \,dx.
\end{align}
Note that while the limit \eqref{limit} can be easily seen to exist in the sense of (tempered) distribution,
the limit leading to \eqref{pairing} needs to be understood in a different topology.
In fact, although $\mu_{++}$ is a tempered distribution, $\widetilde{f}$ is not a Schwartz function, 
even if $f$ is a Schwartz function: 
$\widetilde{f}$ might not be smooth, or may even be discontinuous, at zero; see \eqref{wtf0gen} and \eqref{wtf0exc}. 
Nevertheless, one can still make sense rigorously of the limit and the pairing in \eqref{pairing} 
for $f,g,h \in \mathcal{S}$. 
It actually suffices to consider $\wt{f},\wt{g},\wt{h} \in L^1\cap L^\infty$, for example.


First, let us see that $\mu_{++}$  can be integrated against 
$\widetilde{f}(\eta) \widetilde{g}(\sigma)  \overline{\widetilde{h}(\xi)}$ provided that
$\widetilde{f}$, $\widetilde{g}$ and $\widetilde{h}$ are in $L^1 \cap L^2$ (which is the case if $f,g,h \in \mathcal{S}$).
Indeed, we will see that, up to more regular terms,
$\mu_{++}$ is a linear combination of $\delta(p)$ and $\pv \frac{1}{p}$ distributions,
where $p$ is as in \eqref{muL+-}, with piece-wise smooth coefficients in the variables $\xi,\eta$ and $\s$. 
The coefficients do not matter, so it suffices to look at the cases 
$\mu_{++} = \delta(\xi - \eta - \sigma)$ or $\pv \frac{1}{\xi - \eta - \sigma}$
(since also the signs $\lambda, \mu, \nu$ in the definition of $p$ are not relevant). 
Then, in the case of the $\delta$ distribution we have
$$
\iiint \delta(\xi - \eta - \sigma) \widetilde{f}(\eta) \widetilde{g}(\sigma)
  \overline{\widetilde{h}(\xi)} d\eta \, d\sigma \,d\xi = \iint \widetilde{f}(\eta) \widetilde{g}(\sigma)
  \overline{\widetilde{h}(\eta + \sigma)} d\eta \, d\sigma,
$$
which is well defined by the Cauchy-Schwarz inequality if $\widetilde{f}, \widetilde{g}, \widetilde{h} \in L^1 \cap L^2$ . 
In the case of the $\pv \frac{1}{p}$ distribution, denoting by $H$ the (standard) Hilbert transform, we have
$$
\iiint \pv \frac{1}{\xi - \eta - \sigma} \widetilde{f}(\eta) \widetilde{g}(\sigma)  \overline{\widetilde{h}(\xi)} d\eta \, d\sigma \,d\xi = \iint \widetilde{f}(\eta) \widetilde{g}(\sigma)  \overline{[H \widetilde{h}](\eta + \sigma)} d\eta \, d\sigma,
$$
which is well defined by the boundedness of the Hilbert transform on $L^2$, and the Cauchy-Schwarz inequality.
%

Second, to justify the limit \eqref{pairing}
let us split
$$
\wt{f} =  \wt{f} (1 - \chi(\cdot / \epsilon) )  + \wt{f} \chi(\cdot / \epsilon) 
  = \wt{f_{1,\epsilon}} + \widetilde{f_{2,\epsilon}},
$$
where $\chi$ is a smooth cutoff function equal to $1$ in a neighborhood of $0$, 
and similarly for $g$ and $h$.
We can then write
\begin{align}\label{eps1}
\langle \wtF\big( a(x)w(x/R) fg\big)(\xi), \wt{h}(\xi) \rangle 
  & = \langle \wtF\big( a(x)w(x/R) f_{1,\epsilon} g_{1,\epsilon}\big)(\xi),\wt{h}_{1,\epsilon}(\xi) \rangle
  \\
  \label{eps2}
  & + \langle \wtF\big(a(x)w(x/R)f_{2,\epsilon} g_{1,\epsilon}\big)(\xi),\wt{h}_{1,\epsilon}(\xi) \rangle 
  + \{ \mbox{similar terms}\}.
\end{align}
Here, the ``similar terms'' contain at least one factor with an index $2$, 
namely $g_{2,\epsilon}$ or $h_{2,\epsilon}$. 
We claim that the terms in \eqref{eps2} are $O(\epsilon)$ remainder terms, uniformly in $R$.
If this is the case, then
$$
\langle \wtF\big( a(x)w(x/R)fg\big)(\xi), \wt{h}(\xi) \rangle 
  = \langle \wtF \big(a(x) w(x/R) f_{1,\epsilon} g_{1,\epsilon} \big)(\xi),
  \wt{h_{1,\epsilon}}(\xi) \rangle + O(\epsilon).
$$
For the main term on the right-hand side above, the
limit as in \eqref{pairing} is justified, since the functions involved are 
Schwartz. 
Therefore, one can let $R \to \infty$ first, and then let $\epsilon \to 0$, and obtain the desired formula.

To show that the remainder terms in \eqref{eps2} are $O(\epsilon)$ we use 
the properties of the distorted Fourier transform:
\begin{align*}
\Big| \langle \wtF \big( a(x)w(x/R)f_{2,\epsilon} g_{1,\epsilon} \big) (\xi),
  \wt{h_{1,\epsilon}}(\xi) \rangle \Big|
  = \Big| \langle a(x) w(x/R) f_{2,\epsilon} g_{1,\epsilon}, h_{1,\epsilon} \rangle \Big| 
\\
\lesssim {\| f_{2,\epsilon} \|}_{L^\infty} {\| g_{1,\epsilon} \|}_{L^2} 
  \| h_{1,\epsilon} \|_{L^2} 
  \lesssim {\| \wt{f_{2,\epsilon}} \|}_{L^1} {\| g_{1,\epsilon} \|}_{L^2} 
  {\| h_{1,\epsilon} \|}_{L^2} 
  \\
  \lesssim \epsilon {\| \wt{f} \|}_{L^\infty} {\| g \|}_{L^2} {\| h \|}_{L^2} \lesssim \epsilon.
\end{align*}

\smallskip
In the following, we simply denote
$$
\mu_{\iota_1 \iota_2} (\xi,\eta,\sigma) = \int a(x) \overline{\psi(x,\xi)} \psi_{\iota_1}(x,\eta) \psi_{\iota_2}(x,\sigma) \, dx.
$$

\medskip 
\noindent 
{\it Decomposition of the quadratic spectral distribution}. 
We can write $\mu_{\iota_1\iota_2}$ as a sum of terms of the form
\begin{align}
\label{muterms}
\frac{1}{(2\pi)^{3/2}} \int a(x) \overline{\psi^A(x,\xi)}\psi_{\iota_1}^B(x,\eta) \psi_{\iota_2}^C(x,\sigma) \, dx,
\qquad A,B,C  \in \{S,R\},
\end{align}
where we are using our main decomposition of $\psi$ in \eqref{psidec}.

\medskip 
\noindent 
{\it The singular part $\mu_S$}.
The main singular component comes from part of the contribution to \eqref{muterms} with $A,B,C = S$.
The decomposition \eqref{psiSdec} can be written under the form
\begin{align}\label{psiSpm}
\begin{split}
\psi^S(x,\xi)& = \chi_+(x) \sum_{\lambda \in \{+,-\} } \mathbf{a}_{\lambda}^+(\xi) e^{i \lambda \xi x}
  + \chi_-(x) \sum_{\lambda \in \{+,-\} } \mathbf{a}_{\lambda}^-(\xi) e^{i \lambda \xi x} 
  \\
 & =: \chi_+(x) \psi^{S,+}(x,\xi) + \chi_-(x) \psi^{S,-}(x,\xi)
\end{split}
\end{align}
where
\begin{equation}
\label{a-esym}
\mathbf{a}_{\lambda}^-(\xi) =
\left\{
\begin{array}{llll}
1 & \mbox{if} \qquad \lambda = + & & \mbox{and} \qquad \xi > 0, \\
R_-(\xi) & \mbox{if} \qquad \lambda = - & & \mbox{and} \qquad \xi > 0, \\
T(-\xi) & \mbox{if} \qquad \lambda= + & & \mbox{and} \qquad \xi < 0, \\
0 & \mbox{if} \qquad \lambda = - & & \mbox{and} \qquad \xi < 0.
\end{array}
\right.
\end{equation}
and
\begin{equation}
\label{a+esym}
\mathbf{a}_{\lambda}^+(\xi) =
\left\{
\begin{array}{llll}
T(\xi) & \mbox{if} \qquad \lambda = + & & \mbox{and} \qquad \xi > 0, \\
0 & \mbox{if} \qquad \lambda = - & & \mbox{and} \qquad \xi > 0, \\
1 & \mbox{if} \qquad \lambda= + & & \mbox{and} \qquad \xi < 0, \\
R_+(-\xi) & \mbox{if} \qquad \lambda = - & & \mbox{and} \qquad \xi < 0,
\end{array}
\right.
\end{equation}
or, equivalently, 
\begin{align}
\label{a+-sym}
\begin{split}
& \mathbf{a}^-_+(\xi) = \mathbf{1}_+(\xi) + \mathbf{1}_-(\xi) T(-\xi) \\
& \mathbf{a}^-_-(\xi) = \mathbf{1}_+(\xi) R_-(\xi)
\\
& \mathbf{a}^+_+(\xi) = T(\xi) \mathbf{1}_+(\xi) + \mathbf{1}_-(\xi) \\
& \mathbf{a}^+_-(\xi) = \mathbf{1}_-(\xi) R_+(-\xi).
\end{split}
\end{align}

Consider the terms in \eqref{muterms} with $A,B,C = S$ and such that in each decomposition of $\psi^S$
there are only contributions containing $\chi_+$ (that is, $\psi^{S,+}$) or $\chi_-$ (that is $\psi^{S,-}$). 
We can write this as
\begin{align}
\label{mu0S2}
\begin{split}
&  \frac{1}{\sqrt{2\pi}} \int_\R a(x) \chi_\pm^3(x) \overline{\psi^{S,\pm}(x,\xi)} 
  \psi^{S,\pm}_{\iota_1}(x,\eta) \psi^{S,\pm}_{\iota_2}(x,\s) \,dx 
  \\
& = \frac{1}{\sqrt{2\pi}} \sum_{\lambda,\mu,\nu \in \{\pm \}} \int_\R a(x) \chi_\pm^3(x)  
  \, {a}^\pm_{\substack{- \iota_1 \iota_2 \\ \lambda \mu \nu}}(\xi,\eta,\sigma) \,
  \overline{e^{\lambda i\xi x}} e^{\iota_1 \mu i\eta x}e^{\iota_2 \nu i \sigma x}  \, dx, 
  \\
& =  \sum_{\lambda,\mu,\nu \in \{\pm\}}  \whF\big(a(x)(\chi_{\pm})^3\big)(\lambda \xi - \iota_1 \mu \eta - \iota_2 \nu \sigma) 
  \, a^\pm_{\substack{- \iota_1 \iota_2 \\ \lambda \mu \nu}}(\xi,\eta,\sigma).
\end{split}
\end{align}
We then write $a(x)(\chi_{\pm})^3 = \ell_{\pm\infty}(\chi_{\pm})^3 + (a(x)-\ell_{\pm\infty})(\chi_{\pm})^3$,
where this last function is Schwartz.
Using the formula \eqref{decphi-} for $\widehat{(\chi_{\pm})^3}$,
we see that the first terms in the right-hand side of \eqref{decphi-}, 
namely $\frac{\sqrt{\pi}}{2} \delta \pm \pv \frac{\widehat{\phi}}{i\xi}$, 
make up the singular part of the distribution, $\mu^{S,\pm}$ in \eqref{muL+-}. 
The contribution corresponding to the last term, $\widehat{\psi}(\xi)$, together with 
the one from $\whF \big(a(x)-\ell_{\pm\infty}\big)$, 
can be absorbed into regular part of the distribution $\mu^R$, see \eqref{mu0R22}.

\medskip 
\noindent
{\it The regular part $\mu_R$}.
The regular part $\mu_R$ contains all other contributions.
These are of two main types: terms of the form \eqref{muterms} when one of the indexes $A,B,C$ is $R$,
or contributions where both $\chi_+$ and $\chi_-$ appear, see \eqref{psiSpm}.
More precisely, we can write
\begin{align}\label{muR0}
\mu_{\iota_1 \iota_2}^R(\xi,\eta,\sigma) 
= \mu_{\iota_1 \iota_2}^{R1}(\xi,\eta,\sigma) + \mu_{\iota_1 \iota_2}^{R2}(\xi,\eta,\sigma)
\end{align}
where, if we let $X_R = \{ (A_1,A_2,A_3) \, : \, \exists \, j=1,2,3 \,\, \mbox{ s.t.} \, A_j = R\}$,
\begin{align}
\label{muR01}
\mu_{\iota_1 \iota_2}^{R1}(\xi,\eta,\sigma) & := \sum_{(A,B,C) \in X_R} 
  \int a(x) \, \overline{\psi^A(x,\xi)}\psi^B_{\iota_1}(x,\eta) \psi^C_{\iota_2}(x,\sigma) \, dx
\end{align}
and
\begin{align}
\label{muR02}
\begin{split}
\mu_{\iota_1 \iota_2}^{R2}(\xi,\eta,\sigma) := \sum_{A,B,C = S} 
  \int a(x) \, \overline{\psi^A(x,\xi)} \psi^B_{\iota_1}(x,\eta) \psi^C_{\iota_2}(x,\sigma)  \, dx - \mu^S_{\iota_1\iota_2} (\xi,\eta,\sigma).
\end{split}
\end{align}
In the remaining of the proof we verify the properties \eqref{muR}-\eqref{muR'} for \eqref{muR01}-\eqref{muR02}.



To understand \eqref{muR01} we start by looking at the case $A=R$ and $B,C=S$.
We restrict our analysis to $\xi>0$, see \eqref{psiR+-}; $\xi<0$ can be treated in the same way.
According to \eqref{psiR+-} this gives the terms
\begin{align}\label{muR09}
\begin{split}
& \int a(x) \overline{\chi_+(x) T(\xi) (m_+(x,\xi)-1) e^{i\xi x} } \psi^S(x,\eta) \psi^S(x,\sigma) \, dx
\\
& + \int a(x) \overline{\chi_-(x) \big[ (m_-(x,-\xi) - 1)e^{i\xi x} + R_-(\xi)(m_-(x,\xi) - 1) e^{-ix\xi} \big]}
  \psi^S(x,\eta) \psi^S(x,\sigma) \, dx.
\end{split}
\end{align}
Let us look at the first term above and only at the contributions
to $\psi^S$ coming from $\psi^{S,+}$, see \eqref{psiSpm}, that is 
\begin{align}\label{muR10}
R_{\mu\nu}(\xi,\eta,\s) & := \bar{T(\xi)} \mathbf{a}_\mu^+(\eta) \mathbf{a}_\nu^+(\sigma) 
  \int a(x) \chi_+^3(x) (\overline{m_+(x,\xi)}-1) \, e^{-i\xi x} e^{i\mu \eta x} e^{i\nu\s x} \, dx.
\end{align}
Notice that the coefficients in front of the integral are products of indicator functions and smooth functions, 
consistently with \eqref{muR}-\eqref{muR'}.
Dropping the irrelevant signs $\mu,\nu$, it then suffices to treat
\begin{align}\label{muR11}
R(\xi,\eta,\s) & := \int a(x) \chi^3_+(x) (\overline{m_+(x,\xi)}-1) \, e^{-i x\xi} \, e^{ix(\eta+\s)} \, dx.
\end{align}
We use the fast decay and smoothness of $m_+-1$ from Lemma \ref{lemm+-} to integrate by parts.
More precisely, for any $M$, we write
\begin{align*}
\big| R(\xi,\eta,\s) \big| & = \Big| \frac{1}{[i(\xi-\eta-\s)]^M} 
  \int e^{-i x\xi} \, e^{ix(\eta+\s)} \partial_x^M \big[ a(x) \chi_+^3(x) \overline{(m_+(x,\xi)-1)} \big] \, dx \Big|
\\
& \lesssim \frac{1}{|\xi-\eta-\s|^M} 
  \sum_{0 \leq \alpha \leq M} \int_{x\geq -1} \big| \partial_x^\alpha (m_+(x,\xi)-1) \big|\, dx
  \lesssim \frac{1}{|\xi-\eta-\s|^M},
\end{align*}
having used \eqref{mgood} for the last inequality.

We have therefore bounded the expression \eqref{muR10} by the right-hand side of \eqref{muR'}, for $a=b=c=0$.

To estimate the derivatives, notice that applying multiple $\eta$- and $\s$-derivatives 
is harmless since these result in additional powers of $x$, but $m_+-1$ decays as fast as desired.
Similarly, again from \eqref{mgood} we see that $\partial_\xi$ derivatives can also be handled easily
since $\partial_\xi^\alpha m$ decays fast as well.
Notice that the second line in \eqref{muR09} 
can be treated exactly like the first one, using the properties of $m_-$ from \eqref{mgood}.
All the other terms in \eqref{muR01} 
can also be treated in the same way.

Let us look at the remaining piece \eqref{muR02}. 
We can write, according to the notation \eqref{psiSpm} and the definition \eqref{mu0S2},
\begin{align}
\label{mu0R21}
\begin{split}
\mu_{\iota_1 \iota_2}^{R2}(\xi,\eta,\sigma) = 
  \sum
  \int a(x) \, \chi_{\epsilon_1}(x)\chi_{\epsilon_2}(x)\chi_{\epsilon_3}(x) 
  \overline{\psi^{S,\epsilon_1}(x,\xi)} \psi^{S,\epsilon_2}(x,\eta) \psi^{S,\epsilon_3}(x,\sigma) \, dx
  \qquad 
\end{split}
\end{align}
where the sum is over $(\epsilon_1,\epsilon_2,\epsilon_3) \neq (+,+,+), (-,-,-)$.
In particular, this means that $a \, \chi_{\epsilon_1}\chi_{\epsilon_2}\chi_{\epsilon_3}$ is a smooth
compactly supported function, which we denote by $\chi$ 
(omitting the dependence on the signs which is not relevant here),
and \eqref{mu0R21} is a linear combination of terms of the form
\begin{align}\label{mu0R22}
\begin{split}
\int \chi(x) 
  \overline{\mathbf{a}_\lambda^{\epsilon_1}(\xi) e^{i\lambda \xi x}} 
  \mathbf{a}_\mu^{\epsilon_2}(\eta) e^{i\mu \eta x} \mathbf{a}_\nu^{\epsilon_3}(\s) e^{i\nu \s x} \, dx
  = 
  \what{\chi}(\lambda\xi -\mu\eta-\nu\s) 
  \overline{\mathbf{a}_\lambda^{\epsilon_1}(\xi)}\mathbf{a}_\mu^{\epsilon_2}(\eta)\mathbf{a}_\nu^{\epsilon_3}(\s)
\end{split}
\end{align}
The desired conclusion \eqref{muR}-\eqref{muR'} follows from the properties of the coefficients 
${\bf a}_{\epsilon}^\lambda$ and the fact that $\what\chi$ is Schwartz.
\end{proof}

\medskip
\subsection{Mapping properties for the regular part of the quadratic spectral distribution}
\label{sseclemmuR}
The product operation $(f,g) \mapsto fg$ obviously satisfies H\"older's inequality; but it is natural to ask about the mapping properties of the bilinear operators associated to the distributions $\mu^S$ and $\mu^R$
$$
(f,g) \mapsto \widetilde{\mathcal{F}}^{-1} \int \mu^{S,R}(\xi,\eta,\sigma) \widetilde{f}(\eta) \widetilde{f}(\sigma) \,d\eta\,d\sigma.
$$
 The singular part $\mu^S$ can be thought of as the leading order term; and indeed, it does satisfy H\"older's inequality, and this is optimal. The regular part is lower order, in that it gains integrability "at $\infty$", but it does not gain regularity. Thus, it can essentially be thought of as an operator of the type $(f,g) \mapsto F f g$, where $F$ is bounded and rapidly decaying. The following lemma gives a rigorous statement along these lines.
 
\begin{lem}[Bilinear estimate for $\mu^R$]\label{lemmuR}
Under the same assumptions and with the same notations as in Proposition \ref{muprop},
consider the measure $\mu^R = \mu^R_{\iota_1\iota_2}$ and the corresponding bilinear operator
\begin{align}\label{lemmuRop}
M_R[a,b] := \widetilde{\mathcal{F}}^{-1} \iint \mu^R(\xi,\eta,\s) \wt{a}(\eta) \wt{b}(\s) d\eta\,d\s.
\end{align}
Then, for all 
\begin{align*}
p_1,p_2\in[2,\infty), \qquad \frac{1}{p_1}+\frac{1}{p_2} \leq \frac{1}{2}
\end{align*}
it holds
\begin{align}\label{lemmuRbound} 
\begin{split}
{\big\| M_R[a,b] \big\|}_{L^2} \lesssim {\| a \|}_{L^{p_1}} {\| b \|}_{L^{p_2}}.
\end{split}
\end{align}
Moreover, for $p_1,p_2$ as above and $p_3,p_4$ another pair satisfying the same assumptions, we have,
for any integer $l\geq0$,
\begin{align}\label{lemmuRSob} 
\begin{split}
{\big\| \jnab^l M_R[a,b] \big\|}_{L^2} \lesssim {\| \jnab^l a \|}_{L^{p_1}} {\| b \|}_{L^{p_2}}
	+ {\| a \|}_{L^{p_3}} {\| \jnab^l b \|}_{L^{p_4}}.
\end{split}
\end{align}
\end{lem}

\medskip
\begin{proof}[Proof of Lemma \ref{lemmuR}]
The starting point is the splitting of $\mu^R$ in \eqref{muR0}-\eqref{muR02}.
We will omit the irrelevant signs $\iota_1\iota_2$ in what follows and just denote 
$\mu^{R1,2}=\mu_{\iota_1 \iota_2}^{R1,2}$.
Also notice that in the definition of $M_R$ we can replace 
all the distorted Fourier transforms by flat Fourier transforms, 
in view of the boundedness of the (adjoint) wave operator $\W := \whF^{-1} \wtF$ on $L^p$, $p\in[2,\infty)$; 
see Proposition \ref{propW} and Theorem \ref{thmweder}.

\medskip
\noindent
{\it Proof of \eqref{lemmuRbound}}.
From \eqref{muR01} we see that $\mu^{R1}$ is a linear combination of terms of the form
\begin{align}\label{muRpr1}
\int \overline{\psi^A(x,\xi)}\psi^B(x,\eta) \psi^C(x,\sigma) \, dx
\end{align}
where at least one of the apexes $A,B$ or $C$ is equal to $R$; 
recall the definition of  $\psi^S$ and $\psi^R$ in \eqref{psiSdec} and \eqref{psiR+-}.
It suffices to look at the two cases $A=R$ or $B=R$.

Let us first look at the case $A=R$ and further restrict our attention to $\xi>0$ and the contribution from $\chi_+$;
all the other contributions can be handled in the same way. 
We are then looking at the distribution
\begin{align}\label{muRpr2}
\mu_1(\xi,\eta,\s) := \int \overline{\chi_+(x) T(\xi) (m_+(x,\xi)-1) e^{i\xi x}} \psi^B(x,\eta) \psi^C(x,\sigma) \, dx,
  \qquad B,C = S \mbox{ or } R.
\end{align}
The bilinear operator associated to it is
\begin{align}\label{muRpr3}
\begin{split}
M_1&[a,b]  = \whF^{-1}_{\xi\mapsto x} \iint \mu_1(\xi,\eta,\s) \what{a}(\eta) \what{b}(\s) d\eta\,d\s
  \\ 
  & = \whF^{-1}_{\xi\mapsto x}
  \int \chi_+(y) \overline{T(\xi) (m_+(y,\xi)-1)}  e^{-i\xi y}  \Big( \int_\R \what{a}(\eta) \psi^B(y,\eta) d\eta \Big) 
  \Big( \int_\R \what{b}(\s) \psi^C(y,\sigma) \, d\s \Big) \, dy.
\end{split}
\end{align}
If we define
\begin{align}\label{muRpr4}
u^A(x) := \int_\R \what{u}(\xi) \psi^A(x,\xi)\,d\xi, \qquad A=S,R,
\end{align}
and the symbol $m(y,\xi):= \jy \chi_+(y) \overline{T(\xi) (m_+(y,\xi)-1)}$,
we see that
\begin{align*}
{\| M_1[a,b] \|}_{L^2} & \lesssim
  {\Big\|\int_\R 
  m(y,\xi) e^{-i\xi y}  \cdot \jy^{-1} \cdot a^B(y) \cdot b^C(y) \, dy \Big\|}_{L^2_\xi}
\end{align*}
In view of Lemmas \ref{lemm+-} and \ref{lemTR} we see that
$m=m(y,\xi)$ satisfies standard pseudo-differential symbol estimates,
and deduce that the associated operator is bounded $L^2\mapsto L^2$.
It follows that 
\begin{align}\label{muRpr5}
\begin{split}
{\| M_1[a,b] \|}_{L^2} & \lesssim
  {\| \jy^{-1} \cdot a^B \cdot b^C \|}_{L^2_y}
  \lesssim {\| a^B \|}_{L^{p_1}} {\| b^C \|}_{L^{p_2}}.
\end{split}
\end{align}

The estimate \eqref{muRpr5} gives us the right-hand side of \eqref{lemmuRbound}
provided we show that $u \mapsto u^S , u^R$ 
as defined in \eqref{muRpr4} are bounded on $L^p$, $p\in[2,\infty)$. 
Since $u^S + u^R = u$ it suffices to show
${\| u^S \|}_{L^p} \lesssim {\| u \|}_{L^p}$.
From the definition of $\psi^S$ in \eqref{psiSdec} and \eqref{psiSpm}--\eqref{a+-sym}, 
we see that this reduces to proving
\begin{align}\label{muRpr6}
{\Big\| \int_\R e^{i\lambda x\xi} \mathbf{a}_{\lambda}^\pm(\xi) 
	\what{u}(\xi) \,d\xi \Big\|}_{L^p} \lesssim {\|u\|}_{L^p}. 
\end{align}
In view of the boundedness of the Hilbert transform it is enough to obtain the same bound 
where the coefficients $\mathbf{a}_{\lambda}^\pm$ are replaced just by $T(\pm\xi)$ or $R_{\pm}(\mp\xi)$.
The desired bound then follows 
since $T(\pm\xi)-1$ and $R_{\pm}(\mp\xi)$ are $H^1$ functions, see \eqref{TRk},
so that their Fourier transforms are in $L^1$.

Consider next the case $B=R$. Again, without loss of generality we may 
restrict our attention to $\eta>0$ and the contribution from $\chi_+$, that is, we look at the measure
\begin{align}\label{muRpr2'}
\mu_1'(\xi,\eta,\s) := \int \chi_+(x) \psi^S(x,\xi) \overline{T(\eta) (m_+(x,\eta)-1) e^{i\eta x}} \psi^C(x,\sigma) \, dx,
  \qquad C = S \mbox{ or } R.
\end{align}
Letting the associated operator be 
\begin{align}\label{muRpr3'}
\begin{split}
M_1'&[a,b] := \whF^{-1}_{\xi\mapsto x} \iint \mu_1'(\xi,\eta,\s) \what{a}(\eta) \what{b}(\s) \,d\eta\,d\s
  \\ 
  & = \whF^{-1}_{\xi\mapsto x}
  \int \chi_+(y) \psi^S(y,\xi)  \Big( \int_\R \what{a}(\eta) \overline{T(\eta) (m_+(y,\eta)-1) e^{i\eta y}} \, d\eta \Big) 
  \Big( \int_\R \what{b}(\s) \psi^C(y,\sigma) \, d\s \Big) \, dy,
\end{split}
\end{align}
we see that
\begin{align*}
{\| M_1'[a,b] \|}_{L^2} & \lesssim
  {\Big\|  \jy \int_\R \overline{T(\eta) (m_+(y,\eta)-1) e^{i\eta y}} \, \what{a}(\eta) \, d\eta \Big\|}_{L^{p_1}}
  {\| b^C \|}_{L^{p_2}}
\end{align*}
having used that $\psi^S$ defines a bounded PDO on $L^2$, as we showed above.
The desired conclusion \eqref{lemmuRbound} then follows 
since $\jx T(\eta) (m_+(x,\eta)-1)$ is the symbol of a bounded PDO on $L^p$, for $p\in(2,\infty)$
in view of Lemmas \ref{lemm+-}, \ref{lemTR} and standard results on PDOs; see for example \cite{CVPDO}.


We now analyze the $\mu^{R,2}$ component from \eqref{muR02}
by looking at the more explicit expression \eqref{mu0R22} for it. 
From this we see that it suffices to look at bilinear operators of the form
\begin{align}\label{muRpr10}
M_2[a,b] := \whF^{-1}_{\xi\mapsto x} \iint_{\R\times\R}  \what{\chi}(\lambda\xi -\mu\eta-\nu\s) 
  \overline{\mathbf{a}_\lambda^{\epsilon_1}(\xi)}\mathbf{a}_\mu^{\epsilon_2}(\eta)\mathbf{a}_\nu^{\epsilon_3}(\s)
  \what{a}(\eta) \what{b}(\s) \,d\eta\,d\s,
\end{align}
where $\chi$ is Schwartz.
By boundedness of the Fourier multipliers $\mathbf{a}^{\epsilon}_{\lambda}$,
\begin{align*}
{\| M_2[a,b] \|}_{L^2} & \lesssim 
	{\big\| \chi \cdot \whF^{-1}\big( \mathbf{a}_\mu^{\epsilon_2} \what{a}\big) \whF^{-1}\big( \mathbf{a}_\nu^{\epsilon_3}\what{b} \big)
	\big\|}_{L^2}
	\\ & \lesssim 
	{\big\| \whF^{-1}\big( \mathbf{a}_\mu^{\epsilon_2} \what{a}\big) \big\|}_{L^{p_1}}
	{\big\| \whF^{-1}\big( \mathbf{a}_\nu^{\epsilon_3}\what{b} \big) \big\|}_{L^{p_2}}
	\lesssim  {\| a \|}_{L^{p_1}} {\| b \|}_{L^{p_2}}.
\end{align*}

\medskip
\noindent
{\it Proof of \eqref{lemmuRSob}}.
We proceed similarly to the proof of \eqref{lemmuRbound}, and reduce to estimating
derivatives of the bilinear operators $M_1, M_1^\prime$ and $M_2$, 
respectively defined in \eqref{muRpr3}, \eqref{muRpr3'} and \eqref{muRpr10};

Applying derivatives to $M_1$ gives 
\begin{align}\label{muRpr15}
\begin{split}
\partial_x^l M_1[a,b] & = \whF^{-1}_{\xi\mapsto x}
  \int_\R \chi_+(y) \overline{T(\xi) (m_+(y,\xi)-1)} \, (i\xi)^l e^{-i\xi y}
  a^A(y) \, b^B(y) \, dy. 
\end{split}
\end{align}
Integrating by parts in $y$ and distributing derivatives on $a^A$, $b^B$ and $m_+-1$  
%
gives a linear combination of terms of the form
\begin{align}\label{muRpr16}
\begin{split}
M_{l_1,l_2}[a,b] := \whF^{-1}_{\xi\mapsto x}
  \int_\R \overline{T(\xi)} \, \partial_y^{l_1} \big( \chi_+(y) (\overline{m_+(y,\xi)-1} ) \big) \, e^{-i\xi y}
  \cdot \partial_y^{l_2} \big(a^A(y) \, b^B(y) \big) \, dy
\end{split}
\end{align}
with $l_1+l_2=l$. 
From Lemmas \ref{lemm+-}, \ref{lemTR} we see that 
$m_{l_1}(x,\xi) := \overline{T(\xi)} \, \partial_x^{l_1} \big( \chi_+(x) (\overline{m_+(x,\xi)-1} ) \big)$
gives rise to a standard PDO bounded on $L^2$. 
Therefore, to bound \eqref{muRpr16} by the right-hand side of \eqref{lemmuRSob},
it suffices to use product Sobolev inequalities and 
${\| \partial_x^l u^S \|}_{L^p} \lesssim {\| \jnab^l u \|}_{L^p}$, $p\in[2,\infty)$,
which follows from the inequality \eqref{muRpr6} with $\partial_x^lu$ instead of $u$.

A similar argument can be used for $M_1^\prime$:
from \eqref{muRpr3'} we see that $x$-derivatives become powers of $\xi$,
which in turn can be transformed to $y$-derivatives since $\psi^S(y,\xi)$ is a linear combination of exponentials
$e^{\pm i y\xi}$ by harmless $\xi$-dependent coefficients;
integrating by parts in $y$ and using the boundedness on $L^p$ of the PDO
with symbol $\jx m_{l_1}(x,\xi) $ 
gives the desired bound.

The argument for $M_2$ is straightforward, using that $T-1$ is a bounded multiplier and $\chi$ is Schwartz:
\begin{align*}
{\| \jnab^l M_2[a,b] \|}_{L^2} \lesssim 
	{\Big\| \jnab^l \Big[ \chi \cdot \whF^{-1}\big( a_\mu^{\epsilon_2} \what{a}\big) 
	\whF^{-1}\big( a_\nu^{\epsilon_3}\what{b} \big) \Big] \Big\|}_{L^2}
	\lesssim
	{\Big\| \jnab^l \Big[ \whF^{-1}\big( a_\mu^{\epsilon_2} \what{a}\big) 
	\whF^{-1}\big( a_\nu^{\epsilon_3}\what{b} \big) \Big] \Big\|}_{L^q}
\end{align*}
with $1/q=1/p_1+1/p_2$, and we can use standard Gagliardo-Nirenberg-Sobolev inequalities
and \eqref{muRpr6} to obtain \eqref{lemmuRSob}.
\end{proof}



\bigskip
\section{The main nonlinear decomposition}\label{secdec}

In this section, we first write Duhamel's formula in distorted Fourier space
and decompose the nonlinear terms according to the results in Section \ref{secmu}
and their nonlinear resonance properties.
In particular, in Subsection \ref{SecDecN1},
we give our main splitting of the quadratic terms into `singular' and `regular'.
In Subsection \ref{SecDecPhi} we prove lower bounds for the oscillating phases that
appear in the singular quadratic terms, and use this in Subsection \ref{SecDecN2}
to apply normal form transformations. 
We then analyze the 
various resulting cubic terms in Subsections \ref{SecDecS1} and \ref{SecDecS2}.
Here, there is a substantial algebraic component because we are treating general transmission and reflection coefficients,
and we need to keep track of exact expressions to calculate asymptotics later on;
moreover, the coefficients \eqref{mucoeffexp} may have jump discontinuities which we need to take care of
after the normal form transformations; 
finally we also need to study convolutions of $\delta$
distributions and (cutoff) $\pv$-type distributions and prove various symbol type estimates on the expressions
obtained after the normal forms.
In the final Subsection \ref{SsecReno} we introduce the renormalized profile $f$
on which we will perform all main estimates moving forward;
we then recapitulate all the formulas and properties obtained so far
and prove regularity in $\xi$ for the symbols of the relevant operators.

\medskip
\subsection{Duhamel's formula}\label{secdec1}
Let $u = u(t,x)$ be a solution of the quadratic Klein-Gordon equation
\begin{align}\tag{KG}
\label{KGu}
\begin{split}
& \partial_{t}^2 u +( - \partial_x^2 +V + 1) u = a(x)u^2,  \qquad (u,u_t)(t=0) = (u_0,u_1),
\end{split}
\end{align}
with the assumptions of Theorem \ref{maintheo}.
In the distorted Fourier space \eqref{KGu} is
\begin{align}
\label{KGuF}
& \partial_{t}^2 \wt{u} + (\xi^2 + 1) \wt{u} = \wt{\mathcal{F}}(a(x)u^2),
  \qquad (\wt{u},\wt{u_t})(t=0) = (\wt{u_0},\wt{u_1}).
\end{align}
To write Duhamel's formula in the distorted Fourier space we define (recall $H = - \partial_x^2 + V$)
\begin{align}
\label{vKG}
& v(t,x) := \big(\partial_{t} -i \sqrt{ H + 1 } \big)u, \qquad \wt{v}(t,\xi) := \big(\partial_{t} - i \langle \xi \rangle \big) \wt{u}.
\end{align}
Notice that, by Lemma \ref{reality}, $\sqrt{H + 1} u$ is real-valued since $u$ is; 
therefore,
\begin{align}
\label{vu}
& u = \frac{v-\bar{v}}{-2i\sqrt{H + 1}}
\end{align}
and
\begin{align}
\label{veq}
& \big(\partial_{t} +i \sqrt{H + 1}\big)v = a(x)u^2, 
  \qquad \big(\partial_{t} + i \langle \xi \rangle \big)\wt{v} = \wt{\mathcal{F}}(a(x)u^2).
\end{align}

By defining the profile
\begin{align}
\label{vprof}
& g(t,x) := \Big(e^{it\sqrt{H+ 1}}v(t,\cdot) \Big)(x), \qquad \wt{g}(t,\xi) = e^{it\langle \xi \rangle} \wt{v}(t,\xi)
\end{align}
we have
\begin{align}
\label{profeq}
& \partial_t \wt{g}(t,\xi) = e^{it\langle \xi \rangle} \wt{\mathcal{F}}(a(x)u^2).
\end{align}
Using the definition of the distorted Fourier transform \eqref{distF},
in view of \eqref{vu} and \eqref{vprof}, this becomes
\begin{equation}
\label{dtwtf}
\begin{split}
&\partial_t \wt{g}(t,\xi) 
\\
& = \sum_{\iota_1,\iota_2 \in \{+,-\}} \iota_1 \iota_2 e^{it \jxi} 
  \iint \Big( \int a(x) \bar{\psi(x,\xi)} \psi_{\iota_1}(x,\eta) 
  \psi_{\iota_2}(x,\sigma) \, dx \Big) 
  \frac{e^{-\iota_1 i t \jeta}}{2i\jeta}\wt{g}_{\iota_1}(t,\eta)
  \frac{e^{-\iota_2 i t \jsig}}{2i\jsig}\wt{g}_{\iota_2}(t,\sigma) \, d\eta\, d\sigma 
  \\
& =  - \sum_{\iota_1,\iota_2 \{+,-\}}  \iota_1\iota_2
\iint e^{it \Phi_{\iota_1\iota_2}(\xi,\eta,\sigma)} \wt{g}_{\iota_1}(t,\eta) \wt{g}_{\iota_2}(t,\sigma) 
\, \frac{1}{4\langle \eta \rangle \langle \sigma \rangle} \mu_{\iota_1\iota_2}(\xi,\eta,\sigma) \, d\eta \, d\sigma,
\end{split}
\end{equation}
where the quadratic spectral distribution $\mu_{\iota_1 \iota_2}$ is defined in Proposition \ref{muprop},
\begin{align}
\label{mu12}
& \Phi_{\iota_1\iota_2}(\xi,\eta,\sigma) := \jxi - \iota_1 \jeta - \iota_2 \jsig,
\end{align}
and we have denoted
\begin{align*}
\wt{g}_{+} = \wt{g}, \qquad \wt{g}_{-} = \overline{\wt{g}}.
\end{align*}

\medskip
\subsection{Decomposition of the quadratic nonlinearity}\label{SecDecN1}
Starting from \eqref{dtwtf}-\eqref{mu12} and using the decomposition of the distribution 
$\mu$ in Proposition \ref{muprop},
we can decompose accordingly the nonlinearity.
More precisely, we write
\begin{align}
\label{DecN1}
\partial_t \wt{g} & = \mathcal{Q}^S + \mathcal{Q}^R =  \sum_{\iota_1,\iota_2 \in \{+,-\}}
  \mathcal{Q}^{S}_{\iota_1\iota_2} + \mathcal{Q}_{\iota_1\iota_2}^R,
\end{align}
where $ \mathcal{Q}^{S}_{\iota_1\iota_2}$ and  $\mathcal{Q}_{\iota_1\iota_2}^R$ are defined below.

\medskip
\noindent
{\it Notation convention}.
When summing over different combinations of signs, such as in the formula \eqref{DecN1}, 
we will often just indicate the indexes or apexes with the understanding that they can be either $+$ or $-$.
Also, we will have expressions which depend on several signs, such as the ones
appearing in \eqref{QZ}. 
In such cases we will only separate the various indexes or apexes by commas when there is a risk of confusion;
see for example \eqref{bulk1} versus \eqref{Q}.

\medskip
\noindent
{\it The singular quadratic interaction $\mathcal{Q}^S_{\iota_1\iota_2}$}. 
We define $\mathcal{Q}^S_{\iota_1\iota_2}$ to be the contribution coming from the singular part of $\mu$, 
see \eqref{mudec}-\eqref{muSdec},
with an additional cutoff in frequency which localizes the principal value part to a suitable  
neighborhood of the singularity:
\begin{align}
\label{Q}
\begin{split}
& \mathcal{Q}_{\iota_1\iota_2}^{S} (t,\xi) := 
  -\iota_1\iota_2 \sum_{\substack{\lambda,\mu,\nu \in \{+,-\}\\ \epsilon \in \{+,-\}}}
  \iint e^{it \Phi_{\iota_1\iota_2}(\xi,\eta,\sigma)} 
  \widetilde{g}_{\iota_1}(t,\eta) \widetilde{g}_{\iota_2}(t, \sigma) Z^{\epsilon} _{\substack{- \iota_1 \iota_2 \\ \lambda \mu \nu}}(\xi,\eta,\sigma)\, d\eta \,d\sigma,
\end{split}
\end{align}
with
\begin{align}
\label{QZ}
\begin{split}
Z^{\epsilon} _{\substack{\iota_0 \iota_1 \iota_2 \\ \lambda \mu \nu}}(\xi,\eta,\sigma) 
  & := \ell_{\epsilon \infty} \frac{a^\epsilon_{\substack{\iota_0 \iota_1 \iota_2 \\ \lambda\mu\nu}} (\xi, \eta,\sigma)}{8\pi\langle \eta \rangle \langle \sigma \rangle} 
  \left[ \sqrt{\frac{\pi}{2}} \delta(p) + \epsilon \varphi^\ast(p,\eta,\sigma) \, \pv \frac{\widehat{\phi}(p)}{ip} \right], 
\end{split}
\end{align}
where
\begin{align}\label{QZphistar}
\varphi^\ast(p,\eta,\sigma) := \varphi_{\leq -D_0}\big(p R(\eta,\sigma)\big),
\qquad p =  - \iota_0 \lambda \xi - \iota_1 \mu \eta - \iota_2 \nu \sigma,
\end{align}
for $D_0$ a suitably large absolute constant, and 
\begin{align}\label{Rfactor}
R(\eta,\sigma) = \frac{\langle \eta \rangle \langle \sigma \rangle}{\langle \eta \rangle + \langle \sigma \rangle}.
\end{align}
The last expression may be thought of as a regularization of $\min(\langle \eta \rangle ,\langle \sigma \rangle)$,
and satisfies
\begin{equation}
\label{estimatesR}
| \partial_\eta^a \partial_\sigma^b R(\eta,\sigma)| \lesssim 
\min(\langle \eta \rangle ,\langle \sigma \rangle) \langle \eta \rangle^{-a} \langle \sigma \rangle^{-b}.
\end{equation}

\bigskip
\noindent
{\it The regular quadratic interaction $\mathcal{Q}_{\iota_1\iota_2}^R$}.
The term $\mathcal{Q}_{\iota_1\iota_2}^R$ 
gathers the contributions coming from the smooth distribution $\mu^R$, see \eqref{mudec} and \eqref{muR}-\eqref{muR'},
and the smooth part from the $\pv$ that is not included in  \eqref{QZ}. 
We can write it as 
\begin{align}
\label{QR1}
\begin{split}
& \mathcal{Q}_{\iota_1\iota_2}^R(t,\xi) := -\iota_1\iota_2 
  \iint e^{it \Phi_{\iota_1\iota_2}(\xi,\eta,\sigma)} \, \mathfrak{q}_{\iota_1\iota_2}(\xi,\eta,\sigma) \,
  \wt{g}_{\iota_1}(t,\eta) \wt{g}_{\iota_2}(t,\sigma) \, d\eta \, d\sigma,
\end{split}
\end{align}
where $\Phi_{\iota_1\iota_2}$ is the phase in \eqref{mu12}, and the symbol is
\begin{align}\label{QR2}
\begin{split}
& \mathfrak{q}_{\iota_1\iota_2}(\xi,\eta,\sigma) 
  = \mathfrak{q}_{\iota_1\iota_2}^+(\xi,\eta,\sigma) + \mathfrak{q}_{\iota_1\iota_2}^-(\xi,\eta,\sigma) 
 + \frac{1}{8\pi \jeta\jsig}\mu^R_{\iota_1\iota_2}(\xi,\eta,\sigma),
\\
& \mathfrak{q}^\epsilon_{\iota_1\iota_2}(\xi,\eta,\sigma) = \frac{\epsilon}{8\pi\jeta\jsig}
  \sum_{\lambda, \mu, \nu} a^\epsilon_{\substack{- \iota_1 \iota_2 \\ \lambda\mu\nu}} (\xi,\eta,\sigma)
  \big(1 - \varphi^\ast(p,\eta,\sigma) \big)
  \frac{\what{\phi}(p)}{i p}, 
\end{split}
\end{align}
with $\mu^R_{\iota_1\iota_2}$ satisfying the properties \eqref{muR}-\eqref{muR'}
(also recall that $p = \lambda \xi - \iota_1 \mu \eta - \iota_2 \nu \sigma$).
Here is a remark that will help us simplify the notation:

\begin{rem}[A more convenient rewriting of $\mathcal{Q}_{\iota_1\iota_2}^R$]\label{Remkappas} 
For $\iota,\kappa = \pm 1$, let
\begin{align}
\label{gki}
\widetilde{g}^{\kappa}_{\iota}(\xi) := \widetilde{g}_\iota (\xi) \mathbf{1}_{\kappa}(\xi),
\end{align}
and notice that for all $\iota,\kappa$, $\widetilde{g}_\iota^\kappa$
enjoys the same bootstrap assumptions as $\wt{g}$; see \eqref{propbootgas}.
Then, inspecting the definition of \eqref{QR1} and of its symbol \eqref{QR2}, 
and recalling the definitions of the coefficients \eqref{mucoeff}-\eqref{mucoeffexp},
and the property of $\mu^R_{\iota_1\iota_2}$ in \eqref{muR}-\eqref{muR'}
we see that we can peel off all indicator functions and write
\begin{align}\label{remkappasQ}
\begin{split}
\mathcal{Q}^R_{\iota_1 \iota_2} & = \displaystyle \sum_{\kappa_0, \kappa_1, \kappa_2} 
  \mathcal{Q}^R_{\substack{\iota_1 \iota_2 \\ \kappa_0 \kappa_1 \kappa_2}},
\\
\mathcal{Q}^R_{\substack{\iota_1 \iota_2 \\ \kappa_0 \kappa_1 \kappa_2}} & 
  :=  -\iota_1\iota_2 \mathbf{1}_{\kappa_0}(\xi)
  \iint e^{it \Phi_{\iota_1\iota_2}(\xi,\eta,\sigma)} \, \mathfrak{q}_{\substack{\iota_1 \iota_2 \\ \kappa_0\kappa_1 \kappa_2}}
  (\xi,\eta,\sigma) \,
  \wt{g}_{\iota_1}^{\kappa_1} (t,\eta) \wt{g}_{\iota_2}^{\kappa_2}(t,\sigma) \, d\eta \, d\s,
\end{split}
\end{align}
where the symbols $\mathfrak{q}_{\substack{\iota_1 \iota_2 \\ \kappa_0\kappa_1 \kappa_2}}$ are smooth.
In what follows we will often omit the signs $\kappa_0,\kappa_1,\kappa_2$ in our notation (for the operators and the symbols), 
as these play no essential role.
We will instead keep the $\iota_1,\iota_2$ signs since they do play a role: the case $\iota_1,\iota_2=+$ is
the main resonant one, while the other cases are relatively easier to treat.
Also notice that the indicator function in the output variable $\xi$ will not be a problem upon differentiation
(which will happen when estimating weighted $L^2$-norms, see \eqref{mtpr2}) as shown in Lemma \ref{lemdxiQR}.
\end{rem}

\medskip
\subsection{Estimates on the phases}\label{SecDecPhi}
As a preparation for the normal form transformation to come, we need very precise estimates on the phase. 
The complication here arises since the quadratic modulus of resonance, 
although positive for all interactions, degenerates at $\infty$ in certain directions.

\begin{lem}[Lower bound for the phases]
\label{lemmaetazeta}
For any $\eta,\sigma \in \R$,
\begin{equation}
\label{equivalent}
\langle \eta + \sigma \rangle - \langle \eta \rangle - \langle \sigma \rangle \approx
\left\{
\begin{array}{ll}
- \min( \langle \eta \rangle , \langle \sigma \rangle) & \mbox{if $\eta \sigma < 0$} \\
\displaystyle - \frac{1}{\min(  \langle \eta \rangle , \langle \sigma \rangle)} & \mbox{if $\eta \sigma > 0$} \, .
\end{array}
\right.
\end{equation}
As a consequence, for any choice of $\iota_1,\iota_2 \in \{+,-\}$, and any $\eta,\sigma \in \R$,
\begin{equation}
\label{hummingbirdz}
\left| \frac{1}{\Phi_{\iota_1 \iota_2}(\eta + \sigma,\eta,\sigma)} \right| 
  \lesssim \min(\langle \eta + \sigma \rangle, \jeta, \jsig).
\end{equation}
Furthermore, if $p:= \xi - \iota_1\eta - \iota_2\s$ is such that
\begin{align*}
|p| \leq \frac{2^{-D_0+2}}{ R(\eta, \sigma)}, 
\end{align*}
with $D_0$ sufficiently large, then
\begin{equation}
\label{Dsufficientlylarge}
\left| \frac{1}{\Phi_{\iota_1 \iota_2}(\xi,\eta,\sigma)} \right| 
  \lesssim \min(
  \jeta, \jsig).
\end{equation}
\end{lem}

\begin{proof} 
In order to prove \eqref{equivalent}, 
we focus on the case where $\eta$ and $\sigma$ have equal signs, since the other case is trivial. 
The expression under study can be written
$$
\langle \eta + \sigma \rangle - \jeta - \jsig 
  = \frac{-1 + 2\eta \sigma - 2 \jeta \jsig }{\langle \eta + \sigma \rangle + \jeta + \jsig}.
$$
If $\eta$ and $\sigma$ are $O(1)$, 
the result is obvious, so we focus on the case where $\eta + \sigma \gg 1$. 
On the one hand, the denominator above is $\sim \max(\jeta, \jsig)$.
On the other hand, if $\eta \approx\sigma$, the numerator above can be expanded as
\begin{align*}
-1 + 2\eta \sigma - 2 \jeta \jsig  
& =  -1 + 2 \eta \sigma \left(- \frac{1}{2\eta^2} - \frac{1}{2 \sigma^2}
  + O\left(\frac{1}{ \eta^4} \right) \right) \approx -1.
\end{align*}
If $\eta \gg \sigma$, the numerator can be written
\begin{align*}
-1 + 2\eta \sigma - 2 \jeta \jsig 
& =  -1 + 2 \eta \left(\sigma - \jsig - \frac{\jsig}{2 \eta^2} 
  + O \left( \frac{\jsig}{\eta^4} \right)\right) \approx - \frac{\eta}{\jsig},
\end{align*}
where the above line follows from $\sigma - \jsig \approx - \frac{1}{\jsig}$ 
and $\frac{\jsig}{\eta^2} \ll \frac{1}{\jsig}$. 
\eqref{equivalent} follows from the above relations.

In order to prove \eqref{hummingbirdz}, we observe that the case $\iota_1,\iota_2 = +$ 
was just treated, while the case $--$ is trivial. There remains the case $+-$, which easily reduces to $++$.

Finally, in order to prove~\eqref{Dsufficientlylarge}, only the cases $++$ and $+-$ require attention. 
We focus on the former, the argument for the latter being an immediate adaptation. 
It follows from the estimate~\eqref{equivalent} that only the case $\eta,\sigma > 0$ requires attention.
Then
\begin{align}\label{phaselb1}
\begin{split}
\big| \Phi_{++}(\xi,\eta,\s) \big| & = 
| \jxi - \jeta -  \langle  \xi -  \eta \rangle 
 + (\langle \xi -  \eta \rangle - \jsig)|
 \\
 & \geq | \jxi -  \jeta - \langle \xi -  \eta \rangle |
 - \frac{| \xi -  \eta|^2 - |\s|^2}{\langle \xi -   \eta \rangle + \jsig} 
 \\
 & \geq  \frac{C}{R(\eta,\sigma)}
 - |p| \cdot \frac{| \xi - \eta + \s|}{\langle \xi -  \eta \rangle + \jsig}.
\end{split}
\end{align}
By choosing the absolute constant $D_0$ large enough,
it follows that
\begin{align}\label{phaselb}
 \big| \Phi_{++}(\xi,\eta,\sigma) \big| \gtrsim \frac{1}{R(\eta,\sigma)} 
 \approx\frac{1}{\min(
 \jeta, \jsig)}.
\end{align}
\end{proof}

\begin{lem}[Derivatives of the phases] \label{californiacondor} 
Assume that $|p| \leq 2^{-D_0+2} R(\eta,\sigma)^{-1}$ 
(note that here $p$ is regarded as an independent variable) 
and let $a,b,c$ be arbitrary non-negative integers. Then:

\begin{itemize}
\item[(i)] For any $\eta, \sigma>0$,
\begin{equation}\label{californiacondor0}
\left| \partial_\eta^a \partial_\sigma^b \partial_p^c 
  \frac{1}{\langle p+\eta+\s \rangle - \langle \eta \rangle - \langle \sigma \rangle} \right| 
  \lesssim \frac{R(\eta,\sigma)^{1+c}}{\langle \eta \rangle^a \langle \sigma \rangle^b}.
\end{equation}

\item[(ii)] For any $\eta, \sigma>0$,
\begin{equation}\label{californiacondor01}
\left| \partial_\eta^a \partial_\sigma^b \partial_p^c 
  \frac{1}{\langle p+\eta+\s \rangle + \langle \eta \rangle - \langle \sigma \rangle} \right| 
  \lesssim \frac{1}{\langle \eta \rangle^a \langle \sigma \rangle^b}.
\end{equation}

\item[(iii)] For any $\iota_1,\iota_2 \in \{+,-\}$
\begin{equation}\label{californiacondor1}
\left| \partial_\eta^a \partial_\sigma^b \partial_p^c 
  \frac{1}{\Phi_{\iota_1 \iota_2}(p+\eta+\s,\eta,\sigma)} \right|
  \lesssim \min(\langle p+\eta+\s \rangle, \langle \eta \rangle, \langle \sigma \rangle)^{1+c}.
\end{equation}
\end{itemize}
\end{lem}

\begin{proof} 
Let us denote $\xi=p+\eta+\s$.
The proof of the first assertion relies
on the lower bound $|\Phi_{++}(\xi,\eta,\sigma)| \gtrsim R(\eta,\sigma)^{-1}$, and 
on the bounds on derivatives
\begin{align*}
& \left| \partial_\eta^a \Phi_{++}(\xi,\eta,\sigma) \right| \lesssim  \langle \eta \rangle^{-a-1} \\
& \left| \partial_\sigma^a \Phi_{++}(\xi,\eta,\sigma) \right| \lesssim  \langle \sigma \rangle^{-a-1} \\
& \left| \partial_p \Phi_{++}(\xi,\eta,\sigma)  \right| \lesssim 1 \\
& \left| \partial_p^a \partial_\eta^b \partial_\sigma^c \Phi_{++}(\xi,\eta,\sigma) \right| \lesssim \langle \eta + \sigma \rangle^{-a-b-c-1} \qquad \mbox{if at most one of $a,b,c$ vanishes, or $a \geq 2$}.
\end{align*}
Similarly, the proof of the second assertion relies on the lower bound $|\Phi_{-+}(\xi,\eta,\sigma)| \gtrsim \langle \eta \rangle$ and on the bounds on derivatives
\begin{align*}
& \left| \partial_\eta^a \Phi_{-+}(\xi,\eta,\sigma) \right| \lesssim 1 
\qquad \mbox{if $a=1, \quad$ and} \quad \langle \eta \rangle^{-a-1} \quad  \mbox{if $a\geq 2$}
\\
& \left| \partial_\sigma^a \Phi_{-+}(\xi,\eta,\sigma) \right| \lesssim  \langle \sigma \rangle^{-a-1}
\\
& \left| \partial_p \Phi_{-+}(\xi,\eta,\sigma)  \right| \lesssim 1 
\\
& \left| \partial_p^a \partial_\eta^b \partial_\sigma^c \Phi_{-+}(\xi,\eta,\sigma) \right| \lesssim \langle \eta + \sigma \rangle^{-a-b-c-1} \qquad \mbox{if at most one of $a,b,c$ vanishes, or $a \geq 2$}.
\end{align*}
In order to prove the third assertion, we must distinguish several cases. 
First, the case $(\iota_1,\iota_2) = (-,-)$ is trivial. 
Second, if $(\iota_1,\iota_2) = (+,+)$ and $\eta, \sigma$ have the same sign, 
then it suffices to use~\eqref{californiacondor0}, 
while if they have opposite signs, the inequality is trivial. 
Finally, if $(\iota_1,\iota_2) = (-,+)$, the only difficult case is that for which $\eta \sigma <0$ and $|\sigma| > |\eta|$. In that case, $\Phi_{-+}$ enjoys the lower bound
$$
|\Phi_{-+} (\xi,\eta,\sigma) | \sim 
  \min( \jxi,\jeta)^{-1} \sim 
  \min( \jxi,\jeta,\jsig)^{-1},
$$
while its derivatives can be bounded as follows
\begin{align*}
& \left| \partial_\eta^a \Phi_{-+}(\xi,\eta,\sigma) \right| \lesssim \min(\langle \eta \rangle, \langle \eta + \sigma \rangle )^{-a-1} \\
& \left| \partial_\sigma^a \Phi_{-+}(\xi,\eta,\sigma) \right| \lesssim  \langle \eta + \sigma \rangle^{-a-1} \\
& \left| \partial_p \Phi_{-+}(\xi,\eta,\sigma)  \right| \lesssim 1 \\
& \left| \partial_p^a \partial_\eta^b \partial_\sigma^c \Phi_{-+}(\xi,\eta,\sigma) \right| \lesssim \langle \eta + \sigma \rangle^{-a-b-c-1} \qquad \mbox{if at most one of $a,b,c$ vanishes, or $a \geq 2$}.
\end{align*}
Combining these estimates gives the desired bound~\eqref{californiacondor1}.
\end{proof}

\medskip
\subsection{Performing the normal form transformation}\label{SecDecN2}

We will now perform a normal form transformation on $\mathcal{Q}^{S}$. 
It is not possible to do so globally on $\mathcal{Q}^R$, 
which is ultimately one of the main difficulties in the nonlinear analysis.
The lower bounds in Lemma \ref{lemmaetazeta} 
allow us to integrate by parts using the identity
\begin{align}\label{SecDecN20}
\frac{1}{i \Phi_{\iota_1\iota_2}} \partial_s e^{is \Phi_{\iota_1\iota_2}} = e^{is \Phi_{\iota_1 \iota_2}}.
\end{align}
By symmetry, it will suffice to consider the case when the time derivative hits the second function. 
This gives
\begin{align}\label{SecDecN21}
\sum_{\substack{\iota_1,\iota_2 }} \int_0^t \mathcal{Q}_{\iota_1\iota_2}^{S} (s,\xi) \, ds = 
\{ \mbox{boundary terms} \} + \{ \mbox{integrated terms} \}.
\end{align}
The boundary terms are given by the following expression:
\begin{align}\label{Tgg}
\begin{split}
& \{ \mbox{boundary terms} \} = \sum_{\substack{\iota_1,\iota_2 }}  \F T_{\iota_1\iota_2}(g,g)(t) 
  - \F T_{\iota_1\iota_2}(g,g)(0)
\\
& \F \big(T_{\iota_1\iota_2}(g,g)\big)(t,\xi) := - \iota_1\iota_2  \sum_{\substack{\lambda, \mu, \nu \\ \epsilon} } 
  \iint e^{it \Phi_{\iota_1 \iota_2}(\xi,\eta,\sigma)} \wt{g}_{\iota_1}(t,\eta) \wt{g}_{\iota_2}(t, \sigma) 
  \frac{Z^{\epsilon} _{\substack{- \iota_1 \iota_2 \\ \lambda \mu \nu}}(\xi,\eta,\sigma)}{i \Phi_{\iota_1 \iota_2}(\xi,\eta,\sigma)}\, d\eta\, d\sigma 
  \\
\end{split}
\end{align}
The integrated terms read
\begin{align}\label{SecDecN22}
\begin{split}
& \{ \mbox{integrated terms} \} =
\\
& \qquad \qquad \sum_{\substack{\iota_1, \iota_2 \\ \epsilon}} 2 \iota_1\iota_2 \int_0^t \sum_{\lambda, \mu, \nu}  
\iint e^{it \Phi_{\iota_1 \iota_2}(\xi,\eta,\sigma)} 
\wt{g}_{\iota_1}(s,\eta) \partial_s \wt{g}_{\iota_2}(s,\s) 
\frac{Z^{\epsilon}_{\substack{- \iota_1 \iota_2 \\ \lambda \mu \nu}}(\xi,\eta,\sigma)}{i 
\Phi_{\iota_1 \iota_2}(\xi,\eta,\sigma)}\, d\eta\, d\sigma\,ds.
\end{split}
\end{align}
We now plug in $\partial_s \widetilde{g} = \sum_{\iota_1,\iota_2} \mathcal{Q}^{S \sharp}_{\iota_1\iota_2} 
+ \mathcal{Q}^{R \sharp}_{\iota_1\iota_2}$, 
where $\mathcal{Q}^{R\sharp,S\sharp}_{\iota_1\iota_2}$
are defined exactly as $\mathcal{Q}^{R,S}_{\iota_1\iota_2}$, see \eqref{Q} and \eqref{QR1}, 
with the exception that $\varphi^*$ is replaced by $1$; 
similarly for $Z^{\sharp \,\epsilon}_{\substack{ \iota_0,\iota_1 , \iota_2  \\ \lambda \mu \nu}}$ 
versus $Z^{\epsilon}_{\substack{ \iota_0,\iota_1 , \iota_2  \\ \lambda \mu \nu}}$, 
and $\mathfrak{q}^\sharp$ versus $\mathfrak{q}$ below. In particular, see \eqref{QR2},
$\mathfrak{q}^\sharp(\xi,\eta,\sigma) = (8\pi \jeta\jsig)^{-1}\mu^R_{\iota_1\iota_2}(\xi,\eta,\sigma)$.

An important observation is that, since $\widehat{\phi}$ is even and real-valued,
$$
\left( Z^{\epsilon'}_{\substack{\iota_0',\iota_1',\iota_2' \\ \\\lambda',\mu',\nu'}}(\sigma,\eta',\sigma') 
\right)_{\iota_2} = Z^{\epsilon'}_{\substack{\iota_2 \iota_0',\iota_2 \iota_1',\iota_2 \iota_2' \\ \\\lambda',\mu',\nu'}}(\sigma,\eta',\sigma').
$$
This gives
\begin{align}\label{SecDecN22'}
& \{ \mbox{integrated terms} \} = \int_0^t \big( B_1(s) + B_2(s) \big) \, ds
\end{align}
where
\begin{align}\label{bulk1}
\begin{split}
B_1(s) & = -2 \sum_{\substack{\epsilon, \epsilon' \\ \lambda, \mu, \nu \\ \lambda', \mu', \nu' \\ \iota_1, \iota_2, 
\iota_1', \iota_2'}} \iota_1\iota_2 \iota_1' \iota_2'
  \iiiint e^{is \Phi_{\iota_1, \iota_2 \iota_1',\iota_2 \iota_2'}(\xi,\eta,\eta',\sigma')} 
Z^{\sharp \,\epsilon'} _{\substack{- \iota_2,\iota_2 \iota_1', \iota_2 \iota_2' \\ \lambda' \mu' \nu'}}(\s,
\eta',\sigma')
\frac{Z^{\epsilon} _{\substack{- \iota_1 \iota_2 \\ \lambda \mu \nu}}(\xi,\eta,\sigma)}{i \Phi_{\iota_1 \iota_2}(\xi,\eta,\sigma)} 
\\
& \qquad \qquad \qquad \qquad \qquad \qquad \times \wt{g}_{\iota_1}(s,\eta) 
\wt{g}_{\iota_2 \iota_1'}(s,\eta') \wt{g}_{\iota_2 \iota_2'}(s,\sigma') \, d\eta\, d\eta' \, d\s\, d\sigma', \\
\end{split}
\end{align}
 and
\begin{align}\label{bulk1pahse}
\Phi_{\kappa_1 \kappa_2 \kappa_3}(\xi,\eta,\eta',\s') := 
  \jxi - \kappa_1 \jeta - \kappa_2  \langle \eta' \rangle - \kappa_3 \langle \s' \rangle.
\end{align}
Upon setting $\kappa_1 := \iota_1$, $\kappa_2 := \iota_2 \iota_1'$, $\kappa_3 := \iota_2 \iota_2'$, this becomes
\begin{align}\label{bulk1'}
B_1(s) = \sum_{\kappa_1 \kappa_2 \kappa_3} \iiint e^{is \Phi_{\kappa_1 \kappa_2 \kappa_3  }(\xi,\eta,\eta',\sigma')}
\mathfrak{b}^1_{\kappa_1 \kappa_2 \kappa_3}(\xi,\eta,\eta',\sigma') \wt{g}_{\kappa_1}(s,\eta) 
\wt{g}_{\kappa_2} (s,\eta') \wt{g}_{\kappa_3}(s,\sigma') \, d\eta\, d\eta' \, d\sigma',
\end{align}
with the natural definition of the symbol $\mathfrak{b}^1_{\kappa_1 \kappa_2 \kappa_3}$
obtained by carrying out the $d\s$ integration in \eqref{bulk1}.

Similarly,
\begin{align}
\begin{split}
\label{bulk2}
B_2(s) & = - 2 \sum_{\substack{\epsilon \\ \lambda \mu \nu 
\\ \iota_1,\iota_2, \iota_1', \iota_2'}} \iota_1\iota_2 \iota'_1\iota'_2
  \iiiint e^{is \Phi_{\iota_1, \iota_2 \iota_1',\iota_2 \iota_2'}(\xi,\eta,\eta',\sigma')} 
\mathfrak{q}^\sharp_{\iota_1'\iota_2'}(\s,\eta',\s')
\frac{Z^{\epsilon} _{\substack{-,\iota_1 ,\iota_2 \\ \lambda, \mu ,\nu}}(\xi,\eta,\sigma)}{i \Phi_{\iota_1 \iota_2}(\xi,\eta,\sigma)} 
\\
& \qquad \qquad \qquad \qquad \qquad \qquad \times \wt{g}_{\iota_1}(s,\eta)
\wt{g}_{\iota_2 \iota_1'}(s,\eta') \wt{g}_{\iota_2 \iota_2'}(s,\sigma') \, d\eta\, d\eta' \, d\s\, d\sigma' 
\\
& = \sum_{\kappa_1 \kappa_2 \kappa_3} \iiint e^{is \Phi_{\kappa_1 \kappa_2 \kappa_3  }(\xi,\eta,\eta',\sigma')} \mathfrak{b}^2_{\kappa_1 \kappa_2 \kappa_3}(\xi,\eta,\eta',\sigma') \wt{g}_{\kappa_1}(s,\eta) 
\wt{g}_{\kappa_2} (s,\eta') \wt{g}_{\kappa_3}(s,\sigma') \, d\eta\, d\eta' \, d\sigma',
\end{split}
\end{align}
with the natural definition of the symbol $\mathfrak{b}^2_{\kappa_1 \kappa_2 \kappa_3}$.

It remains to obtain a good description of the symbols $ \mathfrak{b}^1_{\kappa_1 \kappa_2 \kappa_3}$
and $ \mathfrak{b}^2_{\kappa_1 \kappa_2 \kappa_3}$ obtained when 
carrying out the integration over $\sigma$ in the expressions above.
We do this in the following two subsections.
Subsection \ref{SecDecS1} deals with the top order symbol $\mathfrak{b}^1$ 
whose description requires us to study, and obtain precise formulas for, 
the convolutions of $\delta$ and $\pv 1/\xi$ distributions that are cutoff as in \eqref{QZ}.
Subsection \ref{SecDecS2} deals with the symbol $\mathfrak{b}^2$ 
which is lower order since $\mathfrak{q}^\sharp$ is smooth.

\medskip
\subsection{Top order symbols}\label{SecDecS1}  

\subsubsection{Regularity in $\sigma$}\label{SecDecsig}
The first question that we need to address is that of the possible lack of regularity
of the coefficients in the top order symbols,
that could arise from the lack of regularity of the coefficient ${\bf a}^{\epsilon}_{\l}$
defined in \eqref{mucoeffexp} (for instance, in the case of a generic potential, these are discontinuous at the origin).

First, we observe that the coefficients of the type ${\bf a}^{\eps}_{\l}(x)$, with $x=\xi,\eta,\eta'$ or $\sigma'$,
that appear in \eqref{bulk1}-\eqref{bulk2}, are not harmful.
For the input variables $\eta,\eta'$ and $\s'$ this follows from 
the fact that the corresponding input functions, $\wt{g}$, vanish at zero;
in particular the non-smooth coefficients can be handled as in Remark \ref{Remkappas},
by pairing the indicator functions with the input profiles;
see also Lemma \ref{lemwtf0} which guarantees that also the renormalized profile $\wt{f}$ (see \eqref{Renof})
which will be put in place of $\wt{g}$, vanishes at 0).
For the output variable $\xi$, we can also disregard the jump singularities of ${\bf a}^{\eps}_{\l}(\xi)$
thanks to the following:
first, Lemma \ref{lemdxiT} and Remark \ref{remdxi}(iii) allow us to differentiate once in $\xi$
as needed to estimate the weighted $L^2$-norms; second, we will always estimate $L^p$ norm of the operators
with $1<p<\infty$, where $\mathbf{1}_\pm(\xi)$ is a bounded symbol.

Therefore, we do not need to worry about the coefficients ${\bf a}^{\eps}_{\l}(x)$, with $x=\xi,\eta,\eta'$ or $\sigma'$,
and can just assume they are smooth disregarding the $\mathbf{1}_\pm(x)$ factors that they contain.
However, coefficients ${\bf a}^{\epsilon}_{\l}(\sigma)$, 
which enter the definition of $\mathfrak{b}^1$ through integration over $\sigma$, see \eqref{bulk1}-\eqref{bulk1'}, 
might be harmful. 
We now check that a cancellation occurs upon a proper symmetrization of the symbol.

The symbol $\mathfrak{b}^1_{\kappa_1 \kappa_2 \kappa_3}$ can be written
\begin{align}\label{coeff1}
\mathfrak{b}^1_{\kappa_1 \kappa_2 \kappa_3}(\xi,\eta,\eta',\sigma') & 
= -2 \kappa_1 \kappa_2 \kappa_3  \sum_{\iota_2} \iota_2 \sum_{\substack{\lambda, \mu, \mu', \nu'\\ \epsilon,  \epsilon'}} \int M(\xi,\eta,\s,\eta',\s')\,d\sigma,
\\
\label{coeff2}
M (\xi,\eta,\s,\eta',\s') & :=  \frac{1}{i \Phi_{\kappa_1 \iota_2}(\xi,\eta,\sigma)}
  \sum_{\nu, \lambda'} 
  Z^{\sharp \epsilon'} _{\substack{- \iota_2,\kappa_2, \kappa_3 \\ \lambda' ,\mu' ,\nu'}}(\s,\eta',\sigma') \,
  Z^{\epsilon}_{\substack{-, \kappa_1, \iota_2 \\ \lambda, \mu, \nu}}(\xi,\eta,\sigma),
\end{align}
where we have omitted the dependence on the signs for easier notation.
We can write this out as
\begin{align}\label{coeff10}
\begin{split}
M (\xi,\eta,\s,\eta',\s') :=
  \sum_{\nu, \lambda'}  
  \frac{1}{64 \pi^2 \langle \eta \rangle \langle \sigma \rangle\langle \eta' \rangle \langle \sigma' \rangle}  
  \frac{1}{i \Phi_{\kappa_1 \iota_2}(\xi,\eta,\sigma)}
  \, a^\epsilon_{\substack{- ,\kappa_1, \iota_2 \\ \lambda,\mu,\nu}} (\xi, \eta,\sigma)
  \, a^{\epsilon'}_{\substack{- \iota_2,\kappa_2, \kappa_3 \\ \lambda', \mu' ,\nu'}} (\s,\eta',\sigma')
  \\
  \times \ell_{\epsilon \infty} \ell_{\epsilon' \infty} \left[ \sqrt{\frac{\pi}{2}} \delta(p) 
  + \epsilon \, \varphi^\ast(p,\eta,\sigma) \, \frac{\widehat{\phi}(p)}{ip} \right]
  \left[ \sqrt{\frac{\pi}{2}} \delta(p') + \epsilon'  \pv \frac{\widehat{\phi}(p')}{ip'} \right]
\end{split}
\end{align}
where $p := \lambda \xi - \kappa_1 \mu \eta - \iota_2 \nu \sigma$ as in \eqref{QZphistar}, 
we denoted $$p' := \iota_2\lambda' \s - \kappa_2 \mu' \eta' - \kappa_3 \nu' \sigma',$$ 
and dropped the $\pv$ symbols for brevity.

The main observation is that exchanging simultaneously $\s \mapsto -\s$ and $(\nu,\lambda')\mapsto(-\nu,-\lambda')$ 
leaves $\Phi_{\iota_1 \iota_2}$, $p$ and $p'$ invariant, 
and therefore in particular does not change the distributions in square brackets in \eqref{coeff10}; 
therefore, the coefficients appearing in the first line of \eqref{coeff10} can be symmetrized and 
we may write instead
\begin{align}\label{coeff11}
\begin{split}
& \frac{1}{2} 
\left[ a^{\epsilon'} _{\substack{-\iota_2, \kappa_2, \kappa_3 \\ \lambda' \mu' \nu'}}(\s,\eta',\sigma')
\cdot 
a^{\epsilon}_{\substack{-, \kappa_1, \iota_2 \\ \lambda, \mu, \nu}}(\xi,\eta,\sigma)
+
a^{\epsilon'} _{\substack{- \iota_2,\kappa_2 , \kappa_3 \\ -\lambda', \mu', \nu'}}(-\s,\eta',\sigma')
\cdot 
a^{\epsilon}_{\substack{-, \kappa_1, \iota_2 \\ \lambda, \mu, -\nu}}(\xi,\eta,-\sigma) \right].
\end{split}
\end{align}
Recalling \eqref{mucoeff}, the terms in the sum above can be written more explicitly as
\begin{align}\label{coeff12}
\begin{split}
& = \frac{1}{2} 
  \left[ \mathbf{a}_{\lambda',-\iota_2}^{\epsilon'}(\s) 
  \mathbf{a}_{\mu',\kappa_2}^{\epsilon'}(\eta') \mathbf{a}_{\nu', \kappa_3}^{\epsilon'}(\s')
  \cdot \mathbf{a}_{\lambda,-}^\epsilon(\xi) 
  \mathbf{a}_{\mu,\kappa_1}^\epsilon(\eta) \mathbf{a}_{\nu,\iota_2}^\epsilon(\sigma) \right.
  \\
  & \qquad \qquad \left. + \mathbf{a}_{-\lambda',-\iota_2}^{\epsilon'}(-\s) 
  \mathbf{a}_{\mu',\kappa_2}^{\epsilon'}(\eta') \mathbf{a}_{\nu',\kappa_3}^{\epsilon'}(\s')
  \cdot \mathbf{a}_{\lambda,-}^\epsilon(\xi) 
  \mathbf{a}_{\mu,\iota_1}^\epsilon(\eta) \mathbf{a}_{-\nu,\iota_2}^\epsilon(-\sigma) \right]
  \\
  & = \underbrace{\frac{1}{2} 
  \Big[ \mathbf{a}_{\lambda',-\iota_2}^{\epsilon'}(\s)  \mathbf{a}_{\nu,\iota_2}^\epsilon(\sigma)
  + \mathbf{a}_{-\lambda',-\iota_2}^{\epsilon'}(-\s) 
  \mathbf{a}_{-\nu,\iota_2}^\epsilon(-\sigma) \Big]}_{
  := \Big(\displaystyle A^{\epsilon, \epsilon'}_{\nu,\lambda'}(\sigma) \Big)_{\iota_2}} 
  \, \cdot \, \underbrace{\mathbf{a}_{\lambda,-}^\epsilon(\xi) \mathbf{a}_{\mu,\kappa_1}^\epsilon(\eta) 
  \mathbf{a}_{\mu',\kappa_2}^{\epsilon'}(\eta') 
  \mathbf{a}_{\nu',\kappa_3}^{\epsilon'}(\s')}_{\displaystyle  
  := a^{\epsilon,\epsilon'}_{\substack{\lambda,\mu,\mu',\nu' \\ -,\kappa_1,\kappa_2,\kappa_3}}}
\end{split}
\end{align}

Using the formulas for the coefficients $\bf{a}^\epsilon_\lambda$ in \eqref{mucoeffexp} and the relations \eqref{TRconj}-\eqref{TR} 
for the transmission and reflection coefficients, we have
\begin{align*}
2
A^{+,+}_{+,+}(\sigma)
& = \bar{\mathbf{a}_+^+(\s)}  \mathbf{a}_{+}^+(\sigma)
  + \bar{\mathbf{a}_{-}^{+}(-\s)} \mathbf{a}_{-}^+(-\sigma) 
  \\ & = \big( |T(\sigma)|^2 \mathbf{1}_+(\sigma) + \mathbf{1}_-(\sigma) \big) + \mathbf{1}_+(\sigma) |R_+(\sigma)|^2
  \equiv 1,
\end{align*}
and
\begin{align*}
2
A^{+,+}_{+,-}(\sigma)
& = \bar{\mathbf{a}_-^+(\s)}  \mathbf{a}_{+}^+(\sigma)
  + \bar{\mathbf{a}_{+}^{+}(-\s)} \mathbf{a}_{-}^+(-\sigma) 
  \\ & = \bar{R_+(-\sigma)} \mathbf{1}_-(\sigma) + R_+(\sigma) \mathbf{1}_+(\sigma) 
  \equiv R_+(\sigma).
\end{align*}
We can similarly calculate the other expression and arrive at the following formulas:
\begin{align}\label{Acoeff}
\begin{split}
& A^{+,+}_{+,+}(\sigma) = A^{+,+}_{-,-}(-\sigma) = \frac{1}{2},
  \qquad 
  A^{+,+}_{+,-}(\sigma) = A^{+,+}_{-,+}(-\sigma) = \frac{1}{2}R_+(\sigma),
\\
& A^{-,-}_{+,+}(\sigma) = A^{-,-}_{-,-}(-\sigma) = \frac{1}{2},
  \qquad 
  A^{-,-}_{+,-}(\sigma) = A^{-,-}_{-,+}(-\sigma) = \frac{1}{2}R_-(-\sigma),
\\
& A^{+,-}_{+,+}(\sigma) = A^{+,-}_{-,-}(-\sigma) = \frac{1}{2} T(\sigma),
  \qquad 
  A^{+,-}_{+,-}(\sigma) = A^{+,-}_{-,+}(-\sigma) = 0,
\\
& A^{-,+}_{+,+}(\sigma) = A^{-,+}_{-,-}(-\sigma) = \frac{1}{2} T(-\sigma),
  \qquad 
  A^{-,+}_{+,-}(\sigma) = A^{-,+}_{-,+}(-\sigma) = 0.
\end{split}
\end{align}
In particular, we see that this coefficient is smooth. 
The exact values above will be relevant when computing the nonlinear scattering correction in Subsection \ref{secModScatt}.


\smallskip
\subsubsection{Integrating over $\sigma$}\label{SecDectop}
There remains to integrate~\eqref{coeff10} over $\sigma$. 
Observe that the integrand is singular when the variables $p$ or $p'$ hit zero. 
They can be written
\begin{equation}
\begin{split}
\label{variousdef}
\left\{ 
\begin{array}{l}
p = \iota_2 \nu( \Sigma_0 - \sigma) \\
p' = \iota_2 \lambda'(\sigma -\Sigma_1)
\end{array}
\right.
\qquad \mbox{with} \qquad
\left\{
\begin{array}{l}
\Sigma_0 = \iota_2 \nu (\lambda \xi - \iota_1 \mu \eta) \\
\Sigma_1 = \lambda' \iota_2(\kappa_2\mu' \eta' + \kappa_3 \nu' \sigma').
\end{array}
\right.
\end{split}
\end{equation}
Let furthermore
\begin{equation}
\label{defpstar}
p_* := 
\iota_2 \nu p + \iota_2 \lambda' p' = \Sigma_0 - \Sigma_1.
\end{equation}
Depending on whether $Z$ and $Z^\sharp$ contribute $\delta$ or $\frac{1}{x}$, $M$ can be split into
$$
M = \sum_{\nu, \lambda'} 
\big(M^{\delta,\delta} + M^{\delta,\frac{1}{x}} + M^{\frac{1}{x},\delta} + M^{\frac{1}{x},\frac{1}{x}}\big)
$$
with
\begin{align}\label{xyz}
\begin{split}
& M^{\delta,\delta} (\xi,\eta,\s,\eta',\s') :=\mathcal{M}(\xi,\eta,\s,\eta',\s')
 \frac{\pi}{2} \delta(p) \delta(p'),  \\
& M^{\delta,\frac{1}{x}}(\xi,\eta,\s,\eta',\s') := \mathcal{M}(\xi,\eta,\s,\eta',\s') \sqrt{\frac{\pi}{2}} \delta(p) \epsilon'  \frac{\widehat{\phi}(p')}{ip'},  \\
& M^{\frac{1}{x},\delta}(\xi,\eta,\s,\eta',\s') := \mathcal{M}(\xi,\eta,\s,\eta',\s') \epsilon \, \varphi^\ast(p,\eta,\sigma) \, \frac{\widehat{\phi}(p)}{ip} \sqrt{\frac{\pi}{2}} \delta(p'), \\
& M^{\frac{1}{x},\frac{1}{x}}(\xi,\eta,\s,\eta',\s') := \mathcal{M}(\xi,\eta,\s,\eta',\s') \epsilon \epsilon' \, \varphi^\ast(p,\eta,\sigma) \, \frac{\widehat{\phi}(p)}{ip}  \frac{\widehat{\phi}(p')}{ip'},
\end{split}
\end{align}
where
\begin{align}\label{xyz'}
\mathcal{M}(\xi,\eta,\s,\eta',\s') := 
\frac{\ell_{\eps\infty}\ell_{\eps'\infty}}{64 \pi^2 \jeta \jsig \langle \eta' \rangle \langle \sigma' \rangle}
\frac{1}{i \Phi_{\kappa_1 \iota_2}(\xi,\eta,\sigma)} \Big(A^{\epsilon,\epsilon'}_{\nu,\l'}(\sigma)\Big)_{\iota_2}
a^{\epsilon,\epsilon'}_{\substack{\lambda,\mu,\mu',\nu' \\ -,\kappa_1,\kappa_2,\kappa_3}}(\xi,\eta,\eta',\sigma').
\end{align}
When integrating over $\sigma$, we rely on the identities
$$
\delta * \delta = \delta , \qquad  \delta * \frac{1}{x} = \frac{1}{x}, 
  \qquad \frac{1}{x} * \frac{1}{x} = - \pi^2 \delta,
$$
see \eqref{Fsign},
which imply that, for a smooth function $F$,
\begin{align*}
& \int \delta(p) \delta(p') F(\sigma) \,d\sigma = F(\Sigma_0) \delta(p_*) \\
& \int \delta(p)  \frac{\widehat{\phi}(p')}{p'}F(\sigma)\,d\sigma = \iota_2 \lambda' F(\Sigma_0) \frac{\widehat{\phi}(p_*)}{p_*} \\
& \int \varphi^*(p,\eta,\sigma) \frac{\widehat{\phi}(p)}{p} \delta(p')F(\sigma) \,d\sigma = \iota_2 \nu F(\Sigma_0) \frac{\widehat{\phi}(p_*)}{p_*}  + \{\mbox{error}\} \\
& \int \varphi^*(p,\eta,\sigma) \frac{\widehat{\phi}(p)}{p}\frac{\widehat{\phi}(p')}{p'}F(\sigma)\,d\sigma = -\frac{\pi}{2} \nu \lambda' F(\Sigma_0) \delta(p_*) + \{\mbox{error}\}.
\end{align*}
The error terms will be dealt with in the following subsection, in 
Lemmas \ref{birdofparadise1} and \ref{birdofparadise2}. 
For the moment, we record the top order contribution to $\mathfrak{b}^1$, namely
\begin{equation}
\label{formulacubiccoeff}
\begin{split}
& \mathfrak{c}^S_{\kappa_1,\kappa_2,\kappa_3}(\xi,\eta,\eta',\sigma') = 
\\
& \qquad = -2 \kappa_1 \kappa_2 \kappa_3 \sum_{\substack{\epsilon,\epsilon',\iota_2 \\ \lambda,\mu,\nu,\lambda',\mu',\nu'}}
\frac{1}{64 \pi^2 \langle \eta \rangle \langle \Sigma_0 \rangle \langle \eta' \rangle \langle \sigma' \rangle } 
\frac{\big( A^{\epsilon,\epsilon'}_{\nu,\lambda'}(\Sigma_0) \big)_{\iota_2}}{i \Phi_{\kappa_1 \iota_2} (\xi,\eta,\Sigma_0)}
a^{\epsilon,\epsilon'}_{\substack{\lambda,\mu,\mu',\nu' \\ -,\kappa_1,\kappa_2,\kappa_3}}(\xi,\eta,\eta',\sigma')
\\
& \qquad \qquad \qquad \qquad \qquad \times \ell_{\epsilon \infty} \ell_{\epsilon' \infty} 
  \left[\iota_2 \frac{\pi}{2}({1+ \epsilon \epsilon' \nu \lambda'} ) \delta(p_*) 
  + \sqrt{ \frac{\pi}{2}} (\epsilon' \lambda'+ \epsilon \nu) \frac{\widehat{\phi}(p_*)}{i p_*} \right] 
\\
& \qquad = \mathfrak{c}^{S,1}_{\kappa_1,\kappa_2,\kappa_3}(\xi,\eta,\eta',\sigma') 
  + \mathfrak{c}^{S,2}_{\kappa_1,\kappa_2,\kappa_3}(\xi,\eta,\eta',\sigma'),
\end{split}
\end{equation}
where $\mathfrak{c}^{S,1}$ gathers all terms containing $\delta$ functions, 
while $\mathfrak{c}^{S,2}$ gathers all terms containing terms of the type $\widehat{\phi}(p_*)/p_*$. 

The multilinear operator with symbol $\mathfrak{c}^S_{\kappa_1,\kappa_2,\kappa_3}$ will be denoted $\mathcal{C}^S_{\kappa_1,\kappa_2,\kappa_3}$. The decomposition of  $\mathfrak{c}^S_{\kappa_1,\kappa_2,\kappa_3}$ into $\mathfrak{c}^{S,1}_{\kappa_1,\kappa_2,\kappa_3} + \mathfrak{c}^{S,2}_{\kappa_1,\kappa_2,\kappa_3}$ gives a further decomposition of $\mathcal{C}^S_{\kappa_1,\kappa_2,\kappa_3}$:
\begin{equation}
\label{defCS12}
\mathcal{C}^S_{\kappa_1,\kappa_2,\kappa_3} = \mathcal{C}^{S,1}_{\kappa_1,\kappa_2,\kappa_3} + \mathcal{C}^{S,2}_{\kappa_1,\kappa_2,\kappa_3}.
\end{equation}

\medskip
\subsection{Lower order symbols}\label{SecDecS2} 

\subsubsection{The symbol $ \mathfrak{b}^2$} 
Dropping unnecessary subscripts and superscripts, the symbol $\mathfrak{b}^2$ in \eqref{bulk2}
can be written as a sum of terms of the type
$$
\mathbf{a}(\xi) \mathbf{a}(\eta) C(\xi,\eta,\eta',\sigma')
$$
with
$$
C(\xi,\eta,\eta',\sigma') = \frac{1}{\jeta \langle \sigma' 
  \rangle \langle \eta' \rangle} \int \frac{1}{\jsig} 
  \mu^R(\sigma,\eta',\sigma') \mathbf{a}(\sigma) \left[ \sqrt{\frac{\pi}{2}} \delta(p) 
  \pm \varphi^*(p,\eta,\sigma) \frac{\what{\phi}(p)}{p} \right] \frac{1}{\Phi(\xi,\eta,\sigma)} \,d\sigma.
$$

In what follows we adopt the convention that the measure $\mu_R$ appearing above is smooth
in the variables $\eta'$ and $\s'$; in other words, we are disregarding indicator functions in
these two variables which, as explained at the beginning of Subsection \ref{SecDecS1}, can be done without
loss of generality. 

\begin{lem}\label{birdofparadise0}
The symbol $C$ can be split into 
\begin{align*}
C(\xi,\eta,\eta',\sigma') = \mathbf{a}(\Sigma_0) C_1(\xi,\eta,\eta',\sigma') + C_2(\xi,\eta,\eta',\sigma'), 
\end{align*}
where $\Sigma_0$ is defined as in \eqref{variousdef}, and
with
\begin{align}
\label{housefinch1} & |\partial_\xi^a \partial_{\eta}^b \partial_{\eta'}^c \partial_{\sigma'}^d 
  C_1(\xi,\eta,\eta',\sigma')| \lesssim  
  \frac{1}{\langle \eta \rangle \langle \sigma' \rangle \langle \eta' \rangle}  \langle  
  \inf_{\mu,\nu}| \Sigma_0 + \mu \eta' + \nu \sigma' | \rangle^{-N},
\\
\label{housefinch2} & |\partial_\xi^a \partial_{\Sigma_0}^b \partial_{\eta'}^c 
  \partial_{\sigma'}^d C_2(\xi,\eta,\eta',\sigma') | 
  \lesssim \frac{1}{\langle \eta \rangle \langle \sigma' \rangle \langle \eta' \rangle} \langle  
  \inf_{\mu,\nu} | \Sigma_0 + \mu \eta' + \nu \sigma'| \rangle^{-N} 
\left\{ \begin{array}{ll}
| \log |\Sigma_0| | & \mbox{if $a+b=0$} \\ |\Sigma_0|^{-a-b} & \mbox{if $a+b\geq 1$.}
\end{array}\right.  
\end{align}
Note that in \eqref{housefinch1} we regard $\Sigma_0$ as a dependent variable
(since the main singular dependence on $\Sigma_0$ can be factorized) 
while in \eqref{housefinch2} we regard it as an independent one. 
\end{lem}

\begin{proof}
The term $C_1$ is given by the contribution of the $\delta$ term to the symbol $C$:
$$	
C_1(\xi,\eta,\eta',\sigma') = \frac{1}{\langle \eta \rangle \langle \sigma' \rangle \langle \eta' \rangle \langle \Sigma_0 \rangle} \mu^R(\Sigma_0,\eta',\sigma')\frac{1}{\Phi(\xi,\eta,\Sigma_0)}.
$$
It satisfies the desired estimates by~\eqref{californiacondor1} and \eqref{muR}-\eqref{muR'}.
As for the contribution of the principal value term, it can be written as the sum of $C_2'$ and $C_2''$ defined as follows
\begin{align*}
& C_2'(\xi,\eta,\eta',\sigma') = \frac{1}{\langle \eta \rangle \langle \sigma' \rangle \langle \eta' \rangle} \Lambda(\xi,\eta,0) \int  \mathbf{a}(\sigma)\mu^R(\sigma,\eta',\sigma')  \varphi^*(p,\eta,\sigma)  \frac{\widehat{\phi}(p)}{p}   \,dp \\
& C_2''(\xi,\eta,\eta',\sigma') = \frac{1}{\langle \eta \rangle \langle \sigma' \rangle \langle \eta' \rangle}  \int  \mathbf{a}(\sigma)\mu^R(\sigma,\eta',\sigma') [\Lambda(\xi,\eta,-\iota_2 \nu p) - \Lambda(\xi,\eta,0)] \varphi^*(p,\eta,\sigma)  \frac{\widehat{\phi}(p)}{p}   \,dp;
\end{align*}
here we changed the integration variable to $p$, 
so that $\sigma$ is now considered a function of $p$: $\sigma = \Sigma_0 - \iota_2 \nu p$,
and denoted 
\begin{equation}
\label{honeycreeper}
\Lambda(\xi,\eta,q) := \frac{1}{\Phi(\xi,\eta,\Sigma_0 + q) \langle \Sigma_0 + q \rangle},
\end{equation}
which, by \eqref{californiacondor1}, and provided $|q| \ll \frac{1}{R(\eta,\Sigma_0)}$, satisfies
\begin{equation}
\label{derivativelambda}
\left| \partial_\xi^a \partial_\eta^b \partial_{q}^c \Lambda(\xi,\eta,q) \right| \lesssim R(\eta,\Sigma_0)^c.
\end{equation}
This bound, together with the estimates on $\mu^R$ in \eqref{muR}-\eqref{muR'}, 
and the possible singularity of $\mathbf{a}$ at the origin, 
lead to the estimate \eqref{housefinch2} on $C_2'$. 
In order to bound $C_2''$, observe that
\begin{equation}
\label{purplefinch}
f(\xi,\eta,p) = \varphi^*(p,\eta,\sigma) \left[\Lambda(\xi,\eta,-\iota_2 \nu p) - \Lambda(\xi,\eta,0) \right] \frac{\widehat{\phi}(p)}{p}
\end{equation}
satisfies, by~\eqref{californiacondor1},
$$
\operatorname{Supp} f \subset \{ |p| \lesssim R(\eta,\Sigma_0)^{-1} \} \qquad \mbox{and} \qquad \left|\partial_\xi^a \partial_\eta^b \partial_p^c f(\xi,\eta,p) \right| \lesssim R(\eta,\Sigma_0)^{1+c}.
$$
In other words, we can think of $f(p,\eta,\sigma)$ as a normalized cutoff function (in $p$) at scale $R(\eta,\Sigma_0)^{-1}$, such as $R(\eta,\Sigma_0) \chi(R(\eta,\Sigma_0) p)$. Coming back to $C_2''$, it can be written
$$
 C_2''(\xi,\eta,\eta',\sigma') = \frac{1}{\langle \eta \rangle \langle \sigma' \rangle \langle \eta' \rangle} \int  \mathbf{a}(\sigma)\mu^R(\sigma,\eta',\sigma') f(\xi,\eta,p)\,dp,
$$
from which the desired estimate~\eqref{housefinch2} follows.
\end{proof}

\smallskip
\subsubsection{The remainder from integrating $M^{\frac{1}{x},\delta}$} 
Dropping irrelevant indexes and constants, the integral in $\s$ of $M^{\frac{1}{x},\delta}$ can be written as
$$
\frac{\mathbf{a}(\xi) \mathbf{a}(\eta) \mathbf{a}(\eta') 
\mathbf{a}(\sigma')}{\langle \eta \rangle \langle \sigma' \rangle \langle \eta' \rangle}
\int \frac{A(\sigma)}{\Phi(\xi,\eta,\sigma) \langle \sigma \rangle} \varphi^*(p,\eta,\sigma) 
\frac{\widehat{\phi}(p)}{p} \delta(p') \,d\sigma.
$$
The following lemma extracts the leading order contribution, and bounds the remainder term.

\begin{lem} \label{birdofparadise1} 
Recalling the definitions of $p$, $p'$, $\Sigma_0$, $\Sigma_1$ in \eqref{variousdef}
as well as $p_*$ in \eqref{defpstar},
we have the following decomposition
\begin{align*} 
& \int \frac{A(\sigma)}{\Phi(\xi,\eta,\sigma) \langle \sigma \rangle} \varphi^*(p,\eta,\sigma) \frac{\widehat{\phi}(p)}{p} \delta(p') \,d\sigma =  - \iota_2 \nu \frac{A( \Sigma_0)}{\Phi(\xi,\eta,\Sigma_0) \langle \Sigma_0 \rangle}\frac{\widehat{\phi}(p_*)}{p_*} +  C(\xi,\eta,\eta',p_*)
\end{align*}
where, for any $a,b,c,d \in \mathbb{N}_0$,
\begin{align*}
|\partial_\xi^a \partial_\eta^b \partial_{\eta'}^c \partial_{p_*}^d C(\xi,\eta,\eta',p_*)| 
  \lesssim \frac{1}{(|p_\ast| + \frac{1}{R(\eta,\Sigma_0)})^{1+d}}.
\end{align*}
\end{lem}

\begin{rem}
In the above lemma, we chose to parameterize $C$ as a function of $p_*,\xi,\eta$, and $\eta'$. Of course, other choices are also possible; the main point is that derivatives across level sets of $p_*$ are more singular (larger) than along them.
\end{rem}

\begin{proof}[Proof of Lemma \ref{birdofparadise1}] 
First note that, due to the fast decay of $\widehat{\phi}$, 
we can assume that $|p_*| = |\Sigma_0 - \Sigma_1| \lesssim 1$, hence $R(\eta,\Sigma_0) \approx R(\eta,\Sigma_1)$. 
By definition of $\Sigma_0$ and $\Sigma_1$ in~\eqref{variousdef}, and recalling the formula for $\Lambda$ in~\eqref{honeycreeper},
$$
 \int \frac{A(\sigma)}{\Phi(\xi,\eta,\sigma) \langle \sigma \rangle} \varphi^*(p,\eta,\sigma) \frac{\widehat{\phi}(p)}{p} \delta(p') \,d\sigma = \iota_2 \nu \varphi^*(p_*,\eta,\Sigma_1) A(\Sigma_1) \Lambda(\xi,\eta,\Sigma_1 - \Sigma_0)  \frac{\widehat{\phi}(p_*)}{p_*}.
$$
We can now decompose
\begin{align*}
\iota_2 \nu \varphi^*(p_*,\eta,\Sigma_1)A(\Sigma_1) & \Lambda(\xi,\eta,\Sigma_1 - \Sigma_0)  \frac{\widehat{\phi}(p_*)}{p_*} \\
&  = \iota_2 \nu A(\Sigma_0) \Lambda(\xi,\eta,0)  \frac{\widehat{\phi}(p_*)}{p_*}\\
&  \qquad +  \iota_2 \nu A(\Sigma_0) \Lambda(\xi,\eta,0)  [\varphi^*(p_*,\eta,\Sigma_1) - 1]  \frac{\widehat{\phi}(p_*)}{p_*} \\
&   \qquad +  \iota_2 \nu   \varphi^*(p_*,\eta,\Sigma_1) [A(\Sigma_1) \Lambda(\xi,\eta,\Sigma_1 - \Sigma_0) - A(\Sigma_0) \Lambda(\xi,\eta,0)] \frac{\widehat{\phi}(p_*)}{p_*} \\
& = I + II + III.
\end{align*}
The term $I$ is the desired leading order term. As for $II$ and $III$, they make up the error term $C(\xi,\eta,\eta',\sigma')$, and it follows from~\eqref{derivativelambda} that they satisfy the desired estimates.
\end{proof}

\smallskip
\subsubsection{The remainder from integrating $M^{\frac{1}{x},\frac{1}{x}}$}
Dropping irrelevant indexes, $\int M^{\frac{1}{x},\frac{1}{x}} \,d\sigma$ can be written as
$$
\frac{\mathbf{a}(\xi) \mathbf{a}(\eta) \mathbf{a}(\eta') \mathbf{a}(\sigma')}{
\langle \eta \rangle \langle \sigma' \rangle \langle \eta' \rangle}\int \frac{A(\sigma)}{\Phi(\xi,\eta,\sigma) 
\langle \sigma \rangle} \varphi^*(p,\eta,\sigma) \frac{\widehat{\phi}(p)}{p} \frac{\widehat{\phi}(p')}{p'} \,d\sigma.
$$
The following lemma extracts the leading order contribution and bounds the remainder term.

\begin{lem} \label{birdofparadise2} Recalling the definitions of $p$, $p'$, $\Sigma_0$,
$\Sigma_1$ in \eqref{variousdef} as well as $p_*$ in \eqref{defpstar},
we have the following decomposition
\begin{align*} 
& \int \frac{A(\sigma)}{\Phi(\xi,\eta,\sigma) \langle \sigma \rangle} \varphi^*(p,\eta,\sigma) 
  \frac{\widehat{\phi}(p)}{p} \frac{\widehat{\phi}(p')}{p'} \,d\sigma =  
  - \nu \lambda' \frac{\pi}{2}\frac{A( \Sigma_0)}{\Phi(\xi,\eta,\Sigma_0) \langle \Sigma_0 \rangle}\delta(p_\ast) 
  +  C(\xi,\eta,\eta',p_*)
\end{align*}
where, for any $a,b,c,d \in \mathbb{N}_0$,
\begin{align*}
|\partial_\xi^a \partial_\eta^b \partial_{\eta'}^c \partial_{p_*}^d C(\xi,\eta,\eta',p_*)| 
  \lesssim \frac{1}{(|p_\ast| + \frac{1}{R(\eta,\Sigma_0)})^{1+d}}.
\end{align*}
\end{lem}

\begin{proof} 
It will be convenient to adopt lighter notations, by setting
\begin{align*}
\alpha = - \iota_2 \nu, \qquad \alpha' = \iota_2 \lambda',
\end{align*}
so that
\begin{align*}
p = \alpha (\sigma - \Sigma_0) \qquad \mbox{and} \qquad  p' = \alpha' (\sigma - \Sigma_1).
\end{align*}
The integral can be decomposed as follows
\begin{align*}
& \int A(\sigma) \Lambda(\xi,\eta,\sigma-\Sigma_0) \varphi^*(p,\eta,\sigma)  \frac{\widehat{\phi}(p)}{p}
\frac{\widehat{\phi}(p')}{p'} \,d\sigma 
\\
&  \quad =A(\Sigma_0)\Lambda(\xi,\eta,0) \int \frac{\widehat{\phi}(p)}{p} \frac{\widehat{\phi}(p')}{p'} \,d\sigma
  + A(\Sigma_0) \Lambda(\xi,\eta,0) \int [ \varphi^*(p,\eta,\sigma) - 1]  \frac{\widehat{\phi}(p)}{p}   \frac{\widehat{\phi}(p')}{p'} \,d\sigma \\
& \qquad \qquad+ \int \varphi^*(p,\eta,\sigma) \left[A(\sigma) \Lambda(\xi,\eta,p) - A(\Sigma_0) \Lambda(\xi,\eta,0) \right] \frac{\widehat{\phi}(p)}{p}   \frac{\widehat{\phi}(p')}{p'} \,d\sigma \\
& \quad = I + II + III.
\end{align*}
Using that $\frac{\widehat{\phi}(\sigma)}{\sigma} = \frac{i}{2} \widehat{\mathcal{F}} [\phi * \operatorname{sign}]$, $\widehat{f} * \widehat{g} = \sqrt{2\pi} \widehat{fg}$ and $\widehat{1} = \sqrt{2\pi} \delta$, we get that
\begin{align*}
 \int  \frac{\widehat{\phi}(p)}{p} \frac{\widehat{\phi}(p')}{p'} \,d\sigma & = \int  \frac{\widehat{\phi}(\alpha( \sigma - \Sigma_0 + \Sigma_1)}{\alpha( \sigma - \Sigma_0 + \Sigma_1)}  \frac{\widehat{\phi}(\alpha' \sigma)}{\alpha' \sigma} \,d\sigma = - \alpha \alpha' \left[ \frac{\widehat{\phi}(\sigma)}{\sigma} * \frac{\widehat{\phi}(\sigma)}{\sigma} \right](\Sigma_0 - \Sigma_1) \\
& = \alpha \alpha' \frac{\sqrt{2\pi}}{4} \widehat{\mathcal{F}}[(\phi * \operatorname{sign})^2](\Sigma_0 - \Sigma_1) =\frac{\sqrt{2\pi}}{4} \alpha \alpha' \widehat{\mathcal{F}}[1 + \widehat{\mathcal{F}}^{-1} G_0]( \Sigma_0 - \Sigma_1) \\
& = \frac{\pi}{2}\alpha \alpha' \delta(\Sigma_0 - \Sigma_1) + \frac{\sqrt{2\pi}}{4} \alpha \alpha' G_0 (\Sigma_0 - \Sigma_1)
\end{align*}
where $G_0$ is a  Schwartz function. Therefore, modifying the definition of $G_0$ to take the constant factor into account,
$$
I = \frac{\pi}{2}  \alpha \alpha'A(\Sigma_0) \Lambda(\xi,\eta,0) \delta(\Sigma_0 - \Sigma_1) 
  + A(\Sigma_0) \Lambda(\xi,\eta,0) G_0( \Sigma_0 -  \Sigma_1).
$$
Turning to $II$, it can be written
\begin{align*}
II & = A(\Sigma_0) \Lambda(\xi,\eta,0) \int [ \varphi^*(p,\eta,\sigma) - 1]  \frac{\widehat{\phi}(p)}{p}   \frac{\widehat{\phi}(p')}{p'} \,d\sigma 
\\
& = - A(\Sigma_0)\Lambda(\xi,\eta,0) \int \varphi_{>-D_0}(R(\eta,\Sigma_0) p) \frac{\widehat{\phi}(p)}{p}   \frac{\widehat{\phi}(p')}{p'} \,d\sigma 
\\
& \qquad \qquad  \qquad - \Lambda(\xi,\eta,0) \int [ \varphi_{>-D_0}(R(\eta,\sigma) p) -\varphi_{>-D_0}(R(\eta,\Sigma_0) p)] \frac{\widehat{\phi}(p)}{p}   \frac{\widehat{\phi}(p')}{p'} \,d\sigma 
\\
& = II' + II''.
\end{align*}
The term $II''$ is an error term which enjoys better bounds than $II'$, 
so we only focus on the latter, which can be written as
$$
II' = -A(\Sigma_0)\Lambda(\xi,\eta,0) \frac{\sqrt{2\pi}i}{2} \alpha \alpha' 
  \widehat{\mathcal{F}} \big[ (\widehat{F_{R(\eta,\Sigma_0)}}*\operatorname{sign})(\phi * \operatorname{sign})\big](p_*)
$$
where $F_R = \varphi_{>-D_0}(R \cdot) \widehat{\phi}$. 
Essentially, $F_R$ can be written as $\sum_{\frac{2^{-D_0}}{R(\eta,\Sigma_0)} < 2^j < 1} \varphi_j$, and therefore we need to bound
$$
\Big| \sum_{\frac{2^{-D_0}}{R(\eta,\Sigma_0)} < 2^j < 1}  \widehat{\mathcal{F}} [ (\widehat{\varphi_j} * 
  \sign)  (\phi * \sign) (p_*)] \Big|.
$$
Since the average of $\widehat{\varphi_j}$ is zero, 
the convolution $\widehat{\varphi_j} *  \operatorname{sign}$ can be written $\chi(2^j \cdot)$, 
for a Schwartz function $\chi$. 
Then, $(\widehat{\varphi_j} *  \operatorname{sign})  (\phi * \operatorname{sign})$ 
enjoys the same bounds as $\chi(2^j \cdot)$, and therefore, the above can be bounded by
$$
\sum_{\frac{2^{-D_0}}{R(\eta,\Sigma_0)} < 2^j < 1} 2^{-j} \widehat{\chi}(2^{-j} p_*) \lesssim \frac{1}{\frac{1}{R(\eta,\Sigma_0)} + |p_*|},
$$
with natural bounds on the derivatives.

We are left with $III$, which can be written (up to the factor $A$, which does not affect the estimates) as
$$
\int f(p,\eta,\sigma)   \frac{\widehat{\phi}(p')}{p'} \,dp,
$$
where $f(p,\eta,\sigma)$ was introduced in \eqref{purplefinch} 
(notice that we changed the integration variable to $p$, 
so that $p'$ and $\sigma$ are now thought of as functions of $p$). 
As we saw earlier, the function $f(\xi,\eta,p)$ can be thought of as normalized smooth function in $p$
on a scale $R(\eta,\Sigma_0)^{-1}$, such as $R(\eta,\Sigma_0) \chi(R(\eta,\Sigma_0) p)$, with $\chi \in \mathcal{C}_0^\infty$. The desired result follows.
\end{proof}

\medskip
\subsection{Final decomposition and renormalized profile}\label{SsecReno} 
Let us summarize here our findings from the previous subsections regarding the decomposition of the nonlinearity. 

We define the {\it renormalized profile} $f$ by
\begin{align}\label{Renof}
f := g - T(g,g), \qquad T(g,g) := \sum_{\iota_1,\iota_2} T_{\iota_1 \iota_2}^+(g,g) + T_{\iota_1 \iota_2}^-(g,g),
\end{align}
where, according to \eqref{Tgg}, we have
\begin{align}\label{RenoT}
\begin{split}
\wt{\mathcal{F}} T_{\iota_1 \iota_2}^\pm(g,g)(t) & = \iint e^{it \Phi_{\iota_1 \iota_2}(\xi,\eta,\sigma)}
  \widetilde{g}(t,\eta) \widetilde{g}(t,\sigma) \mathfrak{m}_{\iota_1 \iota_2}^\pm(\xi,\eta,\sigma)\,d\eta \,d\sigma 
\\
\mathfrak{m}_{\iota_1 \iota_2}^\pm(\xi,\eta,\s)
  & := -\iota_1\iota_2 \sum_{\l,\mu,\nu}
  \frac{Z^{\pm}_{\substack{- \iota_1 \iota_2 \\ \lambda \mu \nu}}(\xi,\eta,\sigma)}{i \Phi_{\iota_1 \iota_2}(\xi,\eta,\sigma)},
\end{split}
\end{align}
where the symbol $Z$ is defined in \eqref{QZ}.
We then see that $f$ satisfies
\begin{align}\label{Renodtf}
\partial_t \wt{f} = \mathcal{Q}^R(g,g) + \mathcal{C}^{S}(g,g,g) + \mathcal{C}^R(g,g,g)
\end{align}
where:

\begin{itemize}

\medskip
\item The regular quadratic term is given by
\begin{align}\label{QR}
\begin{split}
& \mathcal{Q}^R(a,b) = \sum_{\iota_1, \iota_2} \mathcal{Q}_{\iota_1\iota_2}^R(a,b)
\\
& \mathcal{Q}_{\iota_1\iota_2}^R[a,b](t,\xi)
  = \iint e^{it \Phi_{\iota_1\iota_2}(\xi,\eta,\sigma)} \, \mathfrak{q}(\xi,\eta,\sigma) 
  \, \wt{a}_{\iota_1}(t,\eta) \wt{b}_{\iota_2}(t,\sigma) \, d\eta \, d\sigma
\\
& \Phi_{\iota_1 \iota_2}(\xi,\eta,\s) :=  \jxi - \iota_1 \jeta - \iota_2  \jsig,
\end{split}
\end{align}
with \eqref{QR1}-\eqref{remkappasQ}. 

\medskip
\noindent
{\it Notation convention for the parentheses}. 
Note that in \eqref{QR} above we have used both square and round parentheses for the arguments of $\mathcal{Q}^R$.
When only a pair of arguments appears we will mostly use round brackets, 
both when the arguments are time and frequency $(t,\xi)$ 
or a pair of functions (such as $(a,b)$ above, or $(g,g)$ in \eqref{Renodtf}).
In cases where we write both the input functions and the independent variables we will 
often highlight the distinction between them by using square parentheses for the input functions,
as done in the second line of \eqref{QR} above, in \eqref{lemQRpq}, in \eqref{wproofQR2}, and so on.
We will adopt a similar notation for other similar multilinear expressions (see, for example, 
\eqref{wproofnot1} and \eqref{wproofnot1++}).

Also when the arguments of the bilinear form $\mathcal{Q}^R$ are given by other multilinear expressions
we will use square parentheses (throughout the given formula) to provide a clearer distinction; 
see, for example, \eqref{QRexp}.
We will adopt a similar notation for the trilinear terms \eqref{CubicS} and \eqref{CubicR}.

\def\k{\kappa}

\medskip
\item The singular cubic term is given by
\begin{align}
\label{CubicS}
\begin{split}
& \mathcal{C}^{S}(a,b,c) = \sum_{\kappa_1 \kappa_2 \kappa_3} \mathcal{C}^{S}_{\kappa_1 \kappa_2 \kappa_3}(a,b,c) 
\\
& \mathcal{C}^{S}_{\kappa_1 \kappa_2 \kappa_3}[a,b,c](t,\xi)
= \iiint e^{it \Phi_{\kappa_1 \kappa_2 \kappa_3}(\xi,\eta,\sigma,\theta)} \mathfrak{c}^S_{\kappa_1 \kappa_2 \kappa_3}(\xi,\eta,\sigma,\theta) 
\,  \wt{a}_{\k_1}(t,\eta) \wt{b}_{\k_2}(t,\sigma) \wt{c}_{\k_3}(t,\theta) \, d\eta \, d\sigma \,d\theta,
\\
& \Phi_{\kappa_1 \kappa_2 \kappa_3}(\xi,\eta,\eta',\s') := 
  \jxi - \kappa_1 \jeta - \kappa_2  \langle\sigma \rangle - \kappa_3 \langle \theta \rangle,
\end{split}
\end{align}
with the exact formula for the symbol $\mathfrak{c}^S$ appearing in \eqref{formulacubiccoeff}. 
The operator  $\mathcal{C}^{S}_{\kappa_1 \kappa_2 \kappa_3}$ can be further decomposed into
\begin{align}\label{CubicS12}
\mathcal{C}^{S}_{\kappa_1 \kappa_2 \kappa_3}
  = \mathcal{C}^{S,1}_{\kappa_1 \kappa_2 \kappa_3} + \mathcal{C}^{S,2}_{\kappa_1 \kappa_2 \kappa_3},
\end{align}
see 
\eqref{defCS12}.

\medskip
\item The regular cubic term is given by
\begin{align}
\label{CubicR}
\begin{split}
& \mathcal{C}^{R}(a,b,c) = \sum_{\kappa_1 \kappa_2 \kappa_3} \mathcal{C}^{R}_{\kappa_1 \kappa_2 \kappa_3}(a,b,c)
\\
& \mathcal{C}^{R}_{\kappa_1 \kappa_2 \kappa_3}[a,b,c](t,\xi)
= \iiint e^{it \Phi_{\kappa_1 \kappa_2 \kappa_3}(\xi,\eta,\sigma,\theta)}
  \mathfrak{c}^R_{\kappa_1 \kappa_2 \kappa_3}(\xi,\eta,\sigma,\theta) 
  \, \wt{a}_{\iota_1}(t,\eta) \wt{b}_{\iota_2}(t,\sigma)  
  \wt{c}_{\iota_3}(t,\theta) \, d\eta \, d\sigma \,d\theta,
\end{split}
\end{align}
where, in view of the estimates for the symbols 
appearing in Lemmas \ref{birdofparadise0}, \ref{birdofparadise1} and \ref{birdofparadise2}, 
we have that $\mathfrak{c}^R$ enjoys bounds of the form
\begin{align}\label{cregular1}
\big| \partial_\xi^a \partial_\eta^b \partial_\s^c \partial_\theta^d
  \mathfrak{c}^R_{\kappa_1 \kappa_2 \kappa_3}(\xi,\eta,\sigma,\theta) \big| \lesssim
  \frac{1}{\jeta \jsig \langle\theta\rangle} 
  \frac{\med(|\eta|,|\s|,|\theta|)^{1+a+b+c+d}}{\langle\xi-\eta-\s-\theta\rangle^{N}},
\end{align}
up to possible logarithmic losses like those appearing in \eqref{housefinch2};
recall also the notation for $\med$ from the end of \S\ref{secNotation}.
We are again adopting the convention explained at the beginning of Subsection \ref{SecDecS1}
of disregarding singularities at $0$ in the variables of the inputs of \eqref{CubicR}.

We then note that the terms \eqref{CubicR} are essentially a cubic version of the regular quadratic terms 
$\mathcal{Q}^R$ \eqref{QR}. 
A good way to think of them is that they are essentially of the form $T[\wtF^{-1}\mathcal{Q}^{R\sharp}(g,g),g]$.
Therefore, estimating \eqref{CubicR} is much easier than estimating \eqref{QR},
or other cubic terms that appear in our arguments, such as those in Section \ref{secwL};
see also Propositions \ref{proQRLinftyxi} and \ref{proQRwot}
where terms similar to \eqref{CubicR} are treated.
Therefore, in all that follows, we will skip the estimate for the $\mathcal{C}^R$ 
terms from \eqref{Renodtf}. 

\end{itemize}

\medskip
The next Lemma shows that the renormalized profile satisfies the key assumption
about vanishing at the zero frequency like the original profile $g$.

\begin{lem}\label{lemwtf0}
The renormalized profile \eqref{Renof} satisfies $\wt{f}(0)=0$.
Moreover, when $V$ is exceptional and even, 
$f$ has the same parity of $g$ (even/odd in the case of odd/even resonance).
\end{lem}

\begin{proof}
In the generic case $\wt{f}(0)=0$ is automatically satisfied, see Proposition \ref{propFT}.\footnote{Technically, 
one should check $f\in L^1$ (for fixed $t$), 
but this is not hard to do and, in fact, we will prove 
this type of control later on; see for example Proposition \ref{propbootf}.}
Moreover, in the case where $\wt{u}(0)=0$ because of the structure of the equations
as for (KG2), the claimed property for $\wt{f}$ 
is easy to verify because the quadratic symbols under consideration will vanish at $\xi=0$.

We now verify the statement in the exceptional case by distinguishing between the case of odd 
versus even solutions. 
Note that an odd, respectively even, solution $u$ corresponds to an odd, respectively even, 
profile $\wt{g}$ in distorted Fourier space;
see \eqref{vKG}-\eqref{vprof} and Lemma \ref{lemevenodd}.

In the case of odd solutions our assumptions dictate that the zero energy resonance is even, that is, $T(0)=1$,
and the the coefficient $a(x)$ is odd, hence $\ell_{+\infty} = -\ell_{-\infty}$.
In the case of even solutions instead, we have that the zero energy resonance is odd, that is, $T(0)=-1$,
and the the coefficient $a(x)$ is even, hence $\ell_{+\infty} = \ell_{-\infty}$.

In both exceptional cases, 
since $V$ is even, we have $m_+(-x,\xi)=m_-(x,\xi)$ and $R_+(\xi) = R_-(\xi)$; see \eqref{TRformula}.
In particular, the coefficients defined in \eqref{mucoeffexp} satisfy the symmetry
\begin{align}\label{lemwtf0a}
\mathbf{a}^\epsilon_\l(\xi) = \mathbf{a}^{-\epsilon}_\l(-\xi), \quad \l,\epsilon \in \{+,-\},
\end{align}
and 
$R_\pm(0)=0$.

Next, we inspect the formulas \eqref{RenoT} with \eqref{QZ}.
Since $\wt{g}(0)=0$, 
it suffices to prove that, for fixed $\iota_1,\iota_2$ we have that
$\wtF \big(T^+_{\iota_1\iota_2}(g,g) + T^-_{\iota_1\iota_2}(g,g) \big)$ vanishes at $\xi=0$.
The contribution to $T^\epsilon_{\iota_1\iota_2}(g,g)$ at fixed $\lambda,\mu,\nu$ when $\xi\rightarrow 0$ is
\begin{align*}
\lim_{\xi\rightarrow 0} \big( I_\epsilon(\xi) + I_{-\epsilon}(\xi) \big),
\end{align*}
where
\begin{align}\label{Iepsilon}
\begin{split}
I_\epsilon(\xi) := -\iota_1\iota_2 
  & \iint e^{it \Phi_{\iota_1\iota_2}(0,\eta,\sigma)} 
  \widetilde{g}_{\iota_1}(t,\eta) \widetilde{g}_{\iota_2}(t, \sigma) \, 
  \frac{a^\epsilon_{\substack{\iota_0 \iota_1 \iota_2 \\ \lambda\mu\nu}}(\xi,\eta,\sigma)
  }{8\pi i \jeta \jsig \Phi_{\iota_1 \iota_2}(0,\eta,\sigma)} \, 
  \\
  & \times \ell_{\epsilon \infty} \, \left[ \sqrt{\frac{\pi}{2}} \delta(p_0) 
    + \epsilon \varphi^\ast(p_0,\eta,\sigma) \, \pv \frac{\what{\phi}(p_0)}{ip_0} \right]
  \, d\eta \,d\sigma, \qquad p_0=\iota_1\mu\eta-\iota_2\nu\sigma.
\end{split}
\end{align}
Note that the coefficient ${\bf a}^\epsilon_\l(\xi)$ may be discontinuous at $0$, 
and this is why we kept the dependence on $\xi$ for the coefficient `$a$' in \eqref{Iepsilon} and the limit in $\xi$.

Next, we change variables $(\eta,\sigma) \mapsto (-\eta,-\sigma)$ in the expression \eqref{Iepsilon};
note that $p_0\mapsto-p_0$, and recall that $\phi$ is even.
In the case of odd $\wt{g}$, using \eqref{lemwtf0a} and $\lim_{\xi\rightarrow 0}{\bf a}^\epsilon_+(\xi) = 1$,
$\lim_{\xi\rightarrow 0} {\bf a}^\epsilon_-(\xi) = 0$ (here the coefficients are continuous, see \eqref{mucoeffexp}),
we see that
\begin{align*}
a^{\epsilon}_{\substack{\iota_0 \iota_1 \iota_2 \\ \lambda\mu\nu}}(0,-\eta,-\sigma)
= a^{-\epsilon}_{\substack{\iota_0 \iota_1 \iota_2 \\ \lambda\mu\nu}}(0,\eta,\sigma).
\end{align*}
Since $\ell_{+\infty}=-\ell_{-\infty}$, it follows that $I_\epsilon(0) = - I_{-\epsilon}(0)$, 
hence the desired conclusion.

In the case of even $\wt{g}$, we have instead 
$\lim_{\xi\rightarrow 0}({\bf a}^\epsilon_+(\xi) + {\bf a}^{-\epsilon}_+(\xi))=0$ 
and $\lim_{\xi\rightarrow 0} {\bf a}^\epsilon_-(\xi)=0$, which give,
\begin{align*}
\lim_{\xi\rightarrow 0} \big(a^{\epsilon}_{\substack{\iota_0 \iota_1 \iota_2 \\ \lambda\mu\nu}}(\xi,-\eta,-\sigma)
 + a^{-\epsilon}_{\substack{\iota_0 \iota_1 \iota_2 \\ \lambda\mu\nu}}(\xi,\eta,\sigma)\big)=0. 
\end{align*}
Then, changing $(\eta,\sigma) \mapsto (-\eta,-\sigma)$, 
and taking the $\xi\rightarrow 0$ in \eqref{Iepsilon}, using that $\ell_{+\infty}=\ell_{-\infty}$ here,
we see that $\lim_{\xi\rightarrow 0}(I_\epsilon(\xi) + I_{-\epsilon}(\xi)) = 0$.

To show that $f$ has the same parity of $g$ we can use similar arguments.
Let us just look at the case when $g$ is even (which corresponds to an odd resonance), 
as the odd case is analogous.
It suffices to show that for even $\wt{g}$ we have that $\wt{T}(g,g)$ is even.
Looking again at the definition of $T$ in \eqref{RenoT} and of $Z$ in \eqref{QZ} and \eqref{mucoeff},
we see that
\begin{align*}
Z^{\eps} _{\substack{\iota_0 \iota_1 \iota_2 \\ \lambda \mu \nu}}(-\xi,-\eta,-\sigma) 
  & = \ell_{\eps \infty} \frac{a^{-\eps}_{\substack{\iota_0 \iota_1 \iota_2 \\ \lambda\mu\nu}}
  (\xi,\eta,\sigma)}{8\pi\jeta \jsig} 
  \left[ \sqrt{\frac{\pi}{2}} \delta(p) + \eps \varphi^\ast(p,\eta,\sigma) \, \pv \frac{\what{\phi}(p)}{-ip} \right]
  = Z^{-\eps} _{\substack{\iota_0 \iota_1 \iota_2 \\ \lambda \mu \nu}}(\xi,\eta,\sigma) 
\end{align*}
having used \eqref{lemwtf0a} and $\ell_{+\infty} = \ell_{-\infty}$. 
It then follows $\wtF T_{\iota_1 \iota_2}^+(g,g)(-\xi) =  \wtF T_{\iota_1 \iota_2}^-(g,g)(\xi)$
and therefore, see \eqref{Renof}, $\wt{T}(g,g)(\xi)$ is even.
\end{proof}

\medskip
The next lemmas give regularity properties for the symbols of the bilinear operators 
$\mathcal{Q}^R$ (Lemma \ref{lemdxiQR}), and $T$ and $\mathcal{C}^S$ (Lemma \ref{lemdxiT}).
These are based on the results from Section \ref{secspth} and \ref{secmu},
but we chose to place them here (although they were refereed to, and used, before) 
since parts of the proofs are similar to the proof of Lemma \ref{lemwtf0} above.

\begin{lem}\label{lemdxiQR}
Let $\mathcal{Q}^R = \mathcal{Q}_{\iota_1\iota_2}^R$ be the bilinear operator defined in \eqref{QR1},
with symbol $\mathfrak{q}=\mathfrak{q}_{\iota_1,\iota_2}$ as in \eqref{QR2},
where $\mu^R = \mu^R_{\iota_1\iota_2}$ is given by \eqref{muR}-\eqref{muR'}.
Then, under the assumptions of Theorem \ref{maintheo} we have that 
\begin{align}\label{lemdxiQRa}
\begin{split}
\partial_\xi \mathcal{Q}_{\iota_1\iota_2}^R
  = it \frac{\xi}{\jxi} \mathcal{Q}_{\iota_1\iota_2}^R
  -\iota_1\iota_2
  \iint e^{it \Phi_{\iota_1\iota_2}(\xi,\eta,\sigma)} 
  \mathfrak{q}'_{\iota_1\iota_2}(\xi,\eta,\sigma) \,
  \wt{g}_{\iota_1}(t,\eta) \wt{g}_{\iota_2}(t,\sigma) \, d\eta \, d\sigma
\end{split}
\end{align}
where
\begin{align}\nonumber
& \mathfrak{q}'_{\iota_1\iota_2} := \mathfrak{q}_1 + \mathfrak{q}_2 + \mathfrak{q}_3,
\\
\label{lemdximu1}
& \mathfrak{q}_1(\xi,\eta,\sigma) :=  \frac{1}{8\pi\jeta\jsig}\sum_{\epsilon_1,\epsilon_2,\epsilon_3 \in \{+,-\}} 
  \mathbf{1}_{\epsilon_1}(\xi)\mathbf{1}_{\epsilon_2}(\eta)\mathbf{1}_{\epsilon_3}(\s)
  \partial_\xi  \mathfrak{r}_{\epsilon_1\epsilon_2\epsilon_3}(\xi,\eta,\sigma),
  \\
\label{lemdximu2}
& \mathfrak{q}_2(\xi,\eta,\sigma) := \frac{1}{8\pi\jeta\jsig}
  \sum_{\epsilon,\lambda,\mu,\nu} \epsilon \, \overline{\mathbf{b}_\lambda^\epsilon(\xi)} 
  \mathbf{a}_{\mu,\iota_1}^\epsilon( \eta) \mathbf{a}_{\nu,\iota_2}^\epsilon(\sigma)
  \Big[ \big(1 - \varphi^\ast(p,\eta,\sigma) \big) \frac{\what{\phi}(p)}{i p} \Big],
  \\  
\label{lemdximu3}
& \mathfrak{q}_3(\xi,\eta,\sigma) := \frac{1}{8\pi\jeta\jsig}
  \sum_{\lambda, \mu, \nu, \epsilon} \epsilon a^\epsilon_{\substack{- \iota_1 \iota_2 \\ \lambda\mu\nu}} (\xi,\eta,\sigma)
  \, \partial_\xi \Big[ \big(1 - \varphi^\ast(p,\eta,\sigma) \big) \frac{\what{\phi}(p)}{i p} \Big], 
\end{align}
where $\mathbf{b}_\lambda^\epsilon(\xi)$ is the function defined for $\xi \neq 0$ 
by $\mathbf{b}_\lambda^\epsilon(\xi) = \partial_\xi \mathbf{a}_\lambda^\epsilon(\xi)$ (see \eqref{mucoeffexp}).
\end{lem}

\begin{proof}
For notational convenience, let us define the operator
(we will often drop the time variable which is a fixed parameter here,
and omit the $\iota_1\iota_2$ signs since they do not play any role)
\begin{align}\label{lemdxiQRpr1}
\begin{split}
T_{\mathfrak{m}}[F](\xi) & 
:= \iint \mathfrak{m}_{\iota_1\iota_2}(\xi,\eta,\sigma) \, F(\xi,\eta,\sigma) \, d\eta \, d\sigma,
\end{split}
\end{align}
so that 
\begin{align}\label{lemdxiQRpr2}
\begin{split}
\partial_\xi \mathcal{Q}_{\iota_1\iota_2}^R(t,\xi)
 = 
 it \frac{\xi}{\jxi} \mathcal{Q}_{\iota_1\iota_2}^R(t,\xi)[g,g] 
 + T_{\partial_\xi \mathfrak{q}_{\iota_1\iota_2}}[G](\xi).
\end{split}
\end{align}
For the second term on the right-hand side of \eqref{lemdxiQRpr2} we have, recalling the definition of $\mathfrak{q}$
from \eqref{QR2},
\begin{align*}
  T_{\partial_\xi \mathfrak{q}} [G]
  = T_{\mathfrak{m}_1}[G] + T_{\mathfrak{m}_2}[G] + T_{\mathfrak{q}_3}[G],
\end{align*}
where
\begin{align}
\label{lemdxiQR'}
\mathfrak{m}_1(\xi,\eta,\sigma) & :=  \frac{1}{8\pi\jeta\jsig} \partial_\xi \mu^R(\xi,\eta,\sigma),
\\
\label{lemdxiQR2'}
\mathfrak{m}_2(\xi,\eta,\sigma) & :=  \frac{1}{8\pi\jeta\jsig} 
  \sum_{\epsilon,\lambda,\mu,\nu} \epsilon \, \partial_\xi \overline{\mathbf{a}_\lambda^\epsilon(\xi)} 
  \mathbf{a}_{\mu,\iota_1}^\epsilon( \eta) \mathbf{a}_{\nu,\iota_2}^\epsilon(\sigma)
  \Big[ \big(1 - \varphi^\ast(p,\eta,\sigma) \big) \frac{\what{\phi}(p)}{i p} \Big],
\end{align}
with $p = \lambda \xi - \iota_1 \mu \eta - \iota_2 \nu \sigma$, and $\mathfrak{q}_3$ is the symbol in \eqref{lemdximu3}.
To prove the lemma it then suffices to show that
\begin{align}
\label{lemdxiQRpr5}
T_{\mathfrak{m}_1}[G] = T_{\mathfrak{q}_1}[G],
\\
\label{lemdxiQRpr6}
T_{\mathfrak{m}_2}[G] = T_{\mathfrak{q}_2}[G].
\end{align}

\smallskip
\noindent
{\it Proof of \eqref{lemdxiQRpr5}}.
We look back at the definition of $\mu^R$ 
in \eqref{muR}-\eqref{muR'} and see that \eqref{lemdxiQRpr5} amounts to showing that the $\delta$ contribution
which arises from of 
$\partial_\xi \mu^R$ vanishes.
Recall the description of $\mu^R = \mu^{R1} + \mu^{R2}$ in \eqref{muR0}-\eqref{muR02}.
Note that all the integrals under consideration are absolutely convergent because of the fast decay of $\psi^R$.
We first look at $\mu^{R1}$ and distinguish between the generic and exceptional cases.

In the generic case, since $T(0)=0$ and $R_{\pm}(0) = -1$ (see Lemma \ref{lemVgen}) 
we have $\partial_\xi \psi^A(x,\xi=0) = 0$, with $A=S$ or $R$, which suffices.

In the exceptional cases, let us write $\psi(x,\xi) = {\bf 1}_{+}(\xi)\psi^{>}(x,\xi) + {\bf 1}_{-}(\xi)\psi^{<}(x,\xi)$
where $\psi^{>}$ is given in \eqref{psi>} and $\psi^{<}$ in \eqref{psi<},
and similarly
let $\psi^A(x,\xi) = {\bf 1}_{+}(\xi)\psi^{A,>}(x,\xi) + {\bf 1}_{-}(\xi)\psi^{A,<}(x,\xi)$, with $A=S,R$,
according to the formulas in \eqref{psiSdec} and \eqref{psiR+-}.
Differentiating \eqref{muR01} we get the singular contribution
\begin{align}\label{dximuR01}
\mathfrak{m}_{R1,\delta}(\xi,\eta,\sigma) := \delta(\xi) \sum_{(A,B,C) \in X_R} 
  \int a(x) \, \big[ \overline{\psi^{A,>}(x,\xi)} - \overline{\psi^{A,<}(x,\xi)} 
  \big]\psi^B_{\iota_1}(x,\eta) \psi^C_{\iota_2}(x,\sigma) \, dx.
\end{align}
In the case $a:=f_+(-\infty,0)=1$, 
we have $T(0)=1$ and $R_\pm(0)=0$ (see Lemma \ref{LESTR0}) which shows,
looking at the formulas \eqref{psi>}-\eqref{psiR+-}, that \eqref{dximuR01} vanishes.


When instead $a=-1$, we have $T(0)=-1$ and $R_\pm(0)=0$, 
and we need to look at the bilinear operator associated to \eqref{dximuR01},
that is $T_{\mathfrak{m}_{R1,\delta}}[G]$ with $G$ as in \eqref{lemdxiQRpr2}.
Changing variables $(\eta,\sigma) \rightarrow (-\eta,-\sigma)$ 
leaves $G$ unchanged in view of Lemma \ref{lemevenodd}.
At the same time, using that $a(x)$ is even, and that $\psi^{A,>}(x,\xi) = \psi^{A,<}(-x,-\xi)$
(since $R_+=R_-$, and $m_+(x,\xi) = m_-(-x,\xi)$),
changing $x\rightarrow -x$ in \eqref{dximuR01} shows that
\begin{align*}
\mathfrak{m}_{R1,\delta}(\xi,-\eta,-\sigma) = \delta(\xi)\sum_{(A,B,C) \in X_R} 
  \int a(x) \, \big[ \overline{\psi^{A,>}(-x,0)} - \overline{\psi^{A,<}(-x,0)} \big]
  \psi^B_{\iota_1}(x,\eta) \psi^C_{\iota_2}(x,\sigma) \, dx
\\  = - \mathfrak{m}_{R1,\delta}(\xi,\eta,\sigma).
\end{align*}
This gives $T_{\mathfrak{m}_{1,\delta}}[G]=0$, as desired.

For $\mu^{R2}$ we start from formula \eqref{mu0R21} and, upon applying $\partial_\xi$, 
we need to look at the symbol containing a $\delta(\xi)$ contribution, that is,
\begin{align}\label{dximuR02}
\mathfrak{m}_{R2,\delta}(\xi,\eta,\sigma) =
  \sum_{\eps_1,\eps_2,\eps_3, \lambda,\mu,\nu} \what{\chi_{\eps_1\eps_2\eps_3}}(\lambda\xi -\mu\eta-\nu\s) 
  \big[ \partial_\xi \overline{\mathbf{a}_\lambda^{\epsilon_1}(\xi)} - \mathbf{b}_\lambda^{\epsilon_1}(\xi) 
  \big]\mathbf{a}_\mu^{\epsilon_2}(\eta)\mathbf{a}_\nu^{\epsilon_3}(\s),
\end{align}
where, recall, the sum is over triples of signs $(\eps_1,\eps_2,\eps_3) \neq (+,+,+), (-,-,-)$,
and we have $\chi_{\eps_1\eps_2\eps_3}(x) = a(x)\chi_{\eps_1}\chi_{\eps_2}\chi_{\eps_3}(x)$.
To show that this symbol gives a vanishing contribution, we first
recall the definition of the $\mathbf{a}_\lambda^{\epsilon}$ coefficients from \eqref{mucoeffexp},
and see that in the generic case we have
$\partial_\xi \mathbf{a}_\lambda^\epsilon(\xi) = \mathbf{b}_\lambda^\epsilon(\xi) - \epsilon \lambda \delta(\xi).$
Then, the right-hand side of \eqref{dximuR02} is 
\begin{align*}
-\delta(\xi)  \sum_{\eps_1,\eps_2,\eps_3, \lambda,\mu,\nu} \what{\chi_{\eps_1\eps_2\eps_3}}(-\mu\eta-\nu\s) 
  \epsilon_1\lambda\, \mathbf{a}_\mu^{\epsilon_2}(\eta)\mathbf{a}_\nu^{\epsilon_3}(\s)
\end{align*}
which vanishes upon summing over $\lambda = +$ and $-$

In the exceptional cases, using Lemma \ref{LES},
we have
\begin{align}\label{dxia}
\partial_\xi \mathbf{a}_\lambda^\epsilon(\xi) = \mathbf{b}_\lambda^\epsilon(\xi)
- \epsilon \delta(\xi) \cdot \left\{ \begin{array}{ll}
 1 - \dfrac{2a}{1+a^2} & \mbox{if $\lambda = +$\,,}
 \\
 \dfrac{a^2-1}{1+a^2} & \mbox{if $\lambda = -$\,.}
\end{array} \right.
\end{align}
Note that we can then use the formula \eqref{dxia} in all cases, with the convention that $a=0$ in the generic case.
When $a=1$ the vanishing of \eqref{dximuR02} is obvious since the coefficients of the $\delta$ in \eqref{dxia} vanish.
In the case $a=-1$, instead, only the $\lambda=+$ term remains in \eqref{dxia},
and we look at the bilinear operator 
associated to \eqref{dximuR02}, that is, $T_{\mathfrak{m}_{R2,\delta}}[G]$ with $G$ as in \eqref{lemdxiQRpr2}.
More precisely, we can see that changing signs to $(\eta,\sigma)$ 
in $T_{\mathfrak{m}_{R2,\delta}}[G]$ leaves $G$ invariant, while
\begin{align*}
\mathfrak{m}_{R2,\delta}(\xi,-\eta,-\sigma) & = 
	\delta(\xi) \sum_{\eps_1,\eps_2,\eps_3,\mu,\nu} \what{\chi_{\eps_1\eps_2\eps_3}}(\mu\eta+\nu\s) 
 	(-2\epsilon_1) \mathbf{a}_\mu^{\epsilon_2}(-\eta)\mathbf{a}_\nu^{\epsilon_3}(-\s)
	\\
	& = \delta(\xi) \sum_{\epsilon_1,\epsilon_2,\epsilon_3,\mu,\nu} \what{\chi_{\eps_1\eps_2\eps_3}}(-\mu\eta-\nu\s) 
 	(2\epsilon_1) \mathbf{a}_\mu^{\epsilon_2}(\eta)\mathbf{a}_\nu^{\epsilon_3}(\s)
	= - \mathfrak{m}_{R2,\delta}(\xi,\eta,\sigma),
\end{align*}
having used \eqref{lemwtf0a} and changed the signs of the $\epsilon$'s to get the second identity,
using also that $ \chi_{\eps_1\eps_2\eps_3}(-x) = \chi_{-(\eps_1\eps_2\eps_3)}(x)$,
since $\chi_+(-x)=\chi_-(x)$ and $a(x)$ is even.


\medskip
\noindent
{\it Proof of \eqref{lemdxiQRpr6}}.
We can use arguments similar to those used above for \eqref{lemdxiQRpr5}.
As before, it suffices to show that the contribution to \eqref{lemdxiQR2'} that contains the $\delta$ factor from \eqref{dxia} vanishes.
In the generic case (using \eqref{dxia} with $a=0$) this contribution is
\begin{align}\label{lemdxiQRpr7}
\mathfrak{m}_{2,\delta}(\xi,\eta,\sigma) & :=  \frac{1}{8\pi\jeta\jsig} 
  \sum_{\epsilon,\lambda,\mu,\nu}  \eps \, \big[- \epsilon \lambda \delta(\xi) \big]
  \mathbf{a}_{\mu,\iota_1}^\epsilon(\eta) \mathbf{a}_{\nu,\iota_2}^\epsilon(\sigma)
  \Big[ \big(1 - \varphi^\ast(p_0,\eta,\sigma) \big) \frac{\what{\phi}(p_0)}{i p_0} \Big],
\end{align}
with $p_0 := - \iota_1 \mu \eta - \iota_2 \nu \sigma$,
and vanishes upon summing over $\lambda = +,-$. 
For $a=1$ the vanishing is obvious since the coefficient of the $\delta$ in \eqref{dxia} vanish.
To see the cancellation in the case $a=-1$, similarly to what was done in the previous paragraph, 
we look at the bilinear operator with symbol
\begin{align*}
\mathfrak{m}_{2,\delta}(\xi,\eta,\sigma) & :=  \frac{1}{8\pi\jeta\jsig} 
  \sum_{\epsilon,\mu,\nu} \big[ -2 \delta(\xi) \big]
  \mathbf{a}_{\mu,\iota_1}^\epsilon( \eta) \mathbf{a}_{\nu,\iota_2}^\epsilon(\sigma)
  \Big[ \big(1 - \varphi^\ast(p_0,\eta,\sigma) \big) \frac{\what{\phi}(p_0)}{i p_0} \Big]. 
\end{align*}
Notice that the $\eps$ factor from \eqref{dxia} and the one present initially in \eqref{lemdxiQR2'} canceled out.
Using \eqref{lemwtf0a}, the fact that $\what{\phi}$ is even, 
recalling the definition of $\varphi^\ast$ (see \eqref{QZphistar}), 
and then changing the sign of $\epsilon$ in the sum, we have 
\begin{align*}
\mathfrak{m}_{2,\delta} (\xi,-\eta,-\sigma) & = \frac{1}{8\pi\jeta\jsig} 
  \sum_{\epsilon,\mu,\nu} \big[ -2 \delta(\xi) \big]
  \mathbf{a}_{\mu,\iota_1}^{-\epsilon}(\eta) \mathbf{a}_{\nu,\iota_2}^{-\epsilon}(\sigma)
  \Big[ \big(1 - \varphi^\ast(p_0,\eta,\sigma) \big) \frac{\what{\phi}(p_0)}{-i p_0} \Big]
\\
& = - \mathfrak{m}_{2,\delta} (\xi,\eta,\sigma) 
\end{align*}
which completes the proof.
\end{proof}

\medskip
\begin{lem}\label{lemdxiT}
Let $T=T_{\iota_1\iota_2}$ be the bilinear operator defined in \eqref{Tgg} 
with the definitions \eqref{QZ} and \eqref{mu12}.
Then, under the assumptions of Theorem \ref{maintheo} we have that
\begin{align}\label{lemdxiT0}
\begin{split}
\partial_\xi \wtF T_{\iota_1\iota_2}
  = it \frac{\xi}{\jxi} T_{\iota_1\iota_2}
  -\iota_1\iota_2 \iint e^{it \Phi_{\iota_1\iota_2}} 
  \mathfrak{t}'_{\iota_1\iota_2}(\xi,\eta,\sigma) \,
  \wt{g}_{\iota_1}(t,\eta) \wt{g}_{\iota_2}(t,\sigma) \, d\eta \, d\sigma
\end{split}
\end{align}
with
\begin{align}\nonumber
& \mathfrak{t}'_{\iota_1\iota_2} := \mathfrak{t}_1 + \mathfrak{t}_2,
\\
\label{lemdxiT1}
& \mathfrak{t}_1(\xi,\eta,\sigma):=  
  \sum_{\epsilon,\lambda,\mu,\nu} \frac{ \ell_{\epsilon \infty} }{8i\pi\jeta \jsig}
  \frac{\overline{\mathbf{b}_\lambda^\epsilon(\xi)} 
  \mathbf{a}_{\mu,\iota_1}^\epsilon(\eta) \mathbf{a}_{\nu,\iota_2}^\epsilon(\sigma)}{
  \Phi_{\iota_1\iota_2}(\xi,\eta,\sigma)} \left( \sqrt{\frac{\pi}{2}} \delta(p) 
  + \epsilon \varphi^\ast(p,\eta,\sigma) \, \pv \frac{\widehat{\phi}(p)}{ip} \right), 
  \\
\label{lemdxiT2}
& \mathfrak{t}_2(\xi,\eta,\sigma):=  
  \sum_{\epsilon,\lambda,\mu,\nu} \ell_{\epsilon \infty} 
  \frac{a^\epsilon_{\substack{- \iota_1 \iota_2 \\ \lambda\mu\nu}}(\xi,\eta,\sigma)}{8
  i\pi\jeta \jsig} 
  \partial_\xi \left[ \frac{1}{\Phi_{\iota_1\iota_2}(\xi,\eta,\sigma)}
  \left( \sqrt{\frac{\pi}{2}} \delta(p) + \epsilon \varphi^\ast(p,\eta,\sigma) \, \pv \frac{\widehat{\phi}(p)}{ip} \right)
  \right], 
\end{align}
where $\mathbf{b}_\lambda^\epsilon(\xi)$ is the function defined for $\xi \neq 0$ 
by $\mathbf{b}_\lambda^\epsilon(\xi) = \partial_\xi \mathbf{a}_\lambda^\epsilon(\xi)$ (see \eqref{mucoeffexp}).
\end{lem}

\begin{proof}
The proof follows along the same lines of the proof of Lemma \ref{lemdxiQR} above;
compare \eqref{lemdximu2}-\eqref{lemdximu3} with \eqref{lemdxiT1}-\eqref{lemdxiT2}.
Starting from the definition of the coefficient \eqref{QZ} we see that the only thing to prove
is that the $\delta(\xi)$ contribution that arises when differentiating the 
$\mathbf{a}_\lambda^\epsilon(\xi)$ factor in the numerator vanishes.
This can be shown exactly as in the proof of \eqref{lemdxiQRpr6};
see the formulas for the symbols \eqref{lemdxiQR2'} and \eqref{lemdximu2}.
In particular, if we let $\mathfrak{t}_{1,\delta}$ denote the $\delta(\xi)$ contribution,
that is, a symbol as in \eqref{lemdxiT1} with ${\bf b}_\lambda^\epsilon$ replaced by 
$\partial_\xi \mathbf{a}_\lambda^{\epsilon} - \mathbf{b}_\lambda^{\epsilon}$,
and look at the exceptional case with odd resonance, 
we can use $\ell_{+\infty}=-\ell_{-\infty}$ and \eqref{lemwtf0a} to see that
$\mathfrak{t}_{1,\delta}(\xi,-\eta,-\sigma) = -\mathfrak{t}_{1,\delta}(\xi,\eta,\sigma)$.
We can then conclude as before.
\end{proof}

\smallskip
\begin{rem}\label{remdxi}

Here are some remarks that we will often use in what follows:

\medskip
\noindent
(i) Lemmas \ref{lemdxiQR} and \ref{lemdxiT} show that the derivatives of the symbols
of the bilinear operators $\mathcal{Q}^R$ and $T$ are smooth up to
up to (possible) singularities along the axis $\xi,\eta$ or $\sigma =0$.
These latter can then be handled as in Remark \ref{Remkappas}.

\medskip
\noindent
(ii) Note that the statements of Lemmas \ref{lemdxiQR} and \ref{lemdxiT} remain 
valid when the operators are applied to any other two inputs in the generic case. 
In the exceptional cases they remain valid for inputs
with the same parity of $g$ (even/odd in the case of odd/even resonance), such as $(g,f)$ or $(g, T(g,g))$;
see Lemma \ref{lemwtf0}.
In the rest of our analysis it will always be the case that the parity of the inputs
is the proper one.

\medskip
\noindent
(iii)
Also notice that similar formulas to \eqref{lemdxiT0}-\eqref{lemdxiT2} hold 
for the cubic bulks \eqref{bulk1'} and \eqref{bulk2},
as it can be seen from \eqref{SecDecN22} and the fact that $\partial_s g$ has the same parity of $g$.

\end{rem}



\bigskip
\section{Multilinear estimates}\label{secmulti}
In this section we first examine general multilinear estimates
which will be useful in particular in Section \ref{secwL}, 
and then establish multilinear estimates for all the operators appearing in Subsection \ref{SsecReno}. 

\medskip
\subsection{Bilinear operators}
General bilinear operators can be written as
$$
B_{\mathfrak{a}}(f,g)(x) = \widehat{\mathcal{F}}^{-1}_{\xi \to x} \iint  {\mathfrak{a}}(\xi,\eta,\zeta) \widehat{f}(\eta) \widehat{g}(\zeta) \,d\eta\,d\zeta.
$$
As far as the present article is concerned, 
we are mostly interested in two classes of bilinear operators: 
those whose symbol $\mathfrak{a}$ contains a singular factor $\delta(\xi-\eta-\zeta)$, 
and those whose symbol contains a singular factor $\pv \frac{\widehat{\phi}(\xi-\eta-\zeta)}{\xi-\eta-\zeta}$;
for simplicity we will drop the $\pv$ sign in what follows.
We parameterize these operators as
\begin{equation}
\label{defCC}
\begin{split}
C_{\mathfrak{a}}(f,g)(x) & := \widehat{\mathcal{F}}^{-1}_{\xi \to x} \int \mathfrak{a}(\eta,\xi-\eta) \widehat{f}(\eta) \widehat{g}(\xi-\eta) \,d\eta \\
& = \frac{1}{\sqrt{2\pi}} \iint \mathfrak{a}(\eta,\zeta) \widehat{f}(\eta) \widehat{g}(\zeta)  e^{ix(\eta + \zeta)} \,d\eta \,d\zeta
\end{split}
\end{equation}
and
\begin{equation}
\label{defDD}
\begin{split}
D_{\mathfrak{b}}(f,g)(x) & := \widehat{\mathcal{F}}^{-1}_{\xi \to x} \iint \mathfrak{b}(\eta,\zeta,\xi-\eta-\zeta) \widehat{f}(\eta) \widehat{g}(\zeta)   \frac{\widehat{\phi}(\xi - \eta - \zeta)}{\xi - \eta - \zeta} \,d\eta \,d\zeta \\
& =  \frac{1}{\sqrt{2\pi}} \iiint \mathfrak{b}(\eta,\zeta,\theta) \widehat{f}(\eta) \widehat{g}(\zeta)  e^{ix(\eta + \zeta+\theta)} \frac{\widehat{\phi}(\theta)}{\theta} \,d\eta \,d\zeta \, d\theta.
\end{split}
\end{equation}
Notice that $C_{\mathfrak{a}}$ operators fall into the category of pseudo-products. 
As for $D_{\mathfrak{b}}$ operators, they are translation invariant to leading order, 
since their symbol is smooth outside of the set $\{ \theta = 0\}$. 

A short computation shows that one can express these in physical space as 
\begin{equation}
\label{quadraticphysical}
\begin{split}
& C_{\mathfrak{a}}(f,g)(x) = \frac{1}{\sqrt{2\pi}} \iint \widehat{\mathfrak{a}}(y-x,z-x) f(y) g(z)\,dy\,dz,
\\
& D_{\mathfrak{b}}(f,g)(x) = \frac{1}{\sqrt{2\pi}} \iint K(x,y,z) f(y) g(z)\,dy\,dz,
\\
& \qquad \mbox{with} \quad K(x,y,z) := \int \widehat{\mathfrak{b}}(y-x,z-x,w-x) Z(w) \,dw, 
  \quad Z := \widehat{\mathcal{F}}^{-1} \frac{\widehat{\phi}(\theta)}{\theta}.
\end{split}
\end{equation}

For $D_{\mathfrak{b}}$, this can be seen as follows 
(we omit the computation for $C_{\mathfrak{a}}$, which is more elementary):
\begin{align*}
D_{\mathfrak{b}}(f,g)(x) & = \frac{1}{\sqrt{2\pi}} \iiint \mathfrak{b}(\eta,\zeta,\theta) \widehat{f}(\eta) \widehat{g}(\zeta)  e^{ix(\eta + \zeta+\theta)} \frac{\widehat{\phi}(\theta)}{\theta} \,d\eta \,d\zeta \, d\theta \\
 & = \frac{1}{(2\pi)^{2}} \int \dots \int \mathfrak{b}(\eta,\zeta,\theta) f(y) g(z) Z(w) e^{i\eta(x-y)} e^{i\zeta(x-z)} e^{i\theta(x-w)} \,d\theta \, d\eta \, d\zeta \,dy\,dz\,dw \\
 & = \frac{1}{\sqrt{2\pi}} \iiint  \widehat{\mathfrak{b}}(y-x,z-x,w-x) Z(w) f(y) g(z)\,dw\,dy\,dz.
\end{align*}

\begin{lem}[boundedness for the $C_{\mathfrak{a}}$ operators]\label{lemmamultilinquadratic1}
If $1\leq p,q,r \leq \infty$ satisfy $\frac{1}{p} + \frac{1}{q}  = \frac{1}{r} $,
$$
\left\| C_{\mathfrak{a}} \right\|_{L^p \times L^q \to L^{r}} \lesssim \big\| \, \widehat{{\mathfrak{a}}} \, \big\|_{L^1}.
$$ 
\end{lem}

This is a standard result, see for example \cite[Lemma 5.2]{IoPu2}.


\begin{rem}[Bounds on symbols]\label{marsupilami} 
Given a symbol $\mathfrak{a}$ we will often bound its Fourier transform in $L^1$
using the following
criterion: 
if $\mathfrak{a}$ is supported on $(\eta,\zeta) \in [t_1-r_1,t_1+r_1] \times [t_2-r_2,t_2 + r_2 ]$ with
$$
\left| \partial_\eta^{k_1} \partial_\zeta^{k_2}  \mathfrak{a} \right| \lesssim r_1^{-k_1} r_2^{-k_2},
$$
then
$$
\left| \widehat {\mathfrak{a}}(x,y) \right| \lesssim \frac{r_1}{(1 + r_1 x)^N} \frac{r_2}{(1 + r_2 y)^N} , 
$$
so that in particular $\| \widehat {\mathfrak{a}} \|_{L^1} \lesssim 1$.
Indeed, the assumption on $\mathfrak{a}$ implies that 
$$
\left| \partial_\eta^{k_1} \partial_\zeta^{k_2} \mathfrak{a} (\eta,\zeta)\right| \lesssim r_1^{-k_1} r_2^{-k_2}  \chi\left( \frac{\eta - t_1}{r_1} \right)  \chi\left( \frac{\zeta - t_2}{r_2} \right),
$$
where $\chi$ is a cutoff function.  Taking the Fourier transform, and using that it maps $L^1$ to $L^\infty$, gives
$$
\left| x^{k_1} y^{k_2} \widehat {\mathfrak{a}}(x,y) \right| \lesssim r_1^{-k_1} r_2^{-k_2} r_1 r_2,
$$
which is the desired result.

The criterion mentioned above can be combined with a change of coordinates since, if $L$ is a 
non-degenerate linear transformation, then
$$
\| \widehat{ \mathfrak{a} \circ L } \|_{L^1} = \| \widehat{ \mathfrak{a} } \|_{L^1}.
$$
\end{rem}

\begin{lem}[boundedness for the $D_{\mathfrak{b}}$ operators]\label{lemmamultilinquadratic2}
Assume that there exists $F \in L^1$ such that
$$
\left| \int \widehat {\mathfrak{b}}(x,y,z) \,dz \right| \leq F(x,y).
$$
Then, if $1\leq p,q,r\leq \infty$ satisfy $\frac{1}{p} + \frac{1}{q} = \frac{1}{r} $, 
$$
\| D_{\mathfrak{b}} \|_{L^p \times L^q  \to L^r} \lesssim \| F \|_{L^1}.
$$
\end{lem}

\begin{proof} 
Using the physical space representation \eqref{quadraticphysical}, the proof reduces to that of Lemma
\ref{lemmamultilinquadratic1} after noticing that $Z \in L^\infty$.
\end{proof}

\begin{rem}\label{marsupilami2}
In order for the condition of Lemma \ref{lemmamultilinquadratic2} to be satisfied, 
it suffices that $\mathfrak{b}$ be supported on $(\eta,\zeta,\theta) \in [t_1-r_1,t_1+r_1] 
\times [t_2-r_2,t_2 + r_2 ] \times [t_3-r_3,t_3 + r_3 ]$ with
$$
\left| \partial_\eta^{k_1} \partial_\zeta^{k_2}  \partial_\theta^{k_3} \mathfrak{b} \right| 
  \lesssim r_1^{-k_1} r_2^{-k_2} r_3^{-k_3}.
$$
Indeed, this implies that
$$
\left| \widehat {\mathfrak{b}}(x,y,z) \right| \lesssim \frac{r_1}{(1 + r_1 x)^N} \frac{r_2}{(1 + r_2 y)^N} \frac{r_1}{(1 + r_3 z)^N}.
$$
This observation can be combined with a change of coordinates: it actually suffices that, for a non-degenerate linear transformation $L$,
$$
\left| \partial_\eta^{k_1} \partial_\zeta^{k_2}  \partial_\theta^{k_3} \mathfrak{b}(L(\eta,\zeta),\theta) \right| \lesssim r_1^{-k_1} r_2^{-k_2} r_3^{-k_3}.
$$
\end{rem}

\medskip
\subsection{Trilinear operators}
General trilinear operators can be written as
\begin{equation}
\label{defgeneraltrilin}
T_{\mathfrak{m}}(f,g,h)(x) = \widehat{\mathcal{F}}^{-1}_{\xi \to x} \iint  {\mathfrak{m}}(\xi,\eta,\zeta,\theta) \widehat{f}(\eta) \widehat{g}(\zeta) \widehat{h}(\theta) \,d\eta\,d\zeta \,d\theta.
\end{equation}
Two classes of trilinear operators of particular relevance in the present article are given by
\begin{equation}
\begin{split}
\label{defUUVV}
& U_{\mathfrak{m}}(f,g,h)(x) = \widehat{\mathcal{F}}^{-1}_{\xi \to x} 
  \iint {\mathfrak{m}}(\xi,\eta,\zeta) \widehat{f}(\xi-\eta) \widehat{g}(\xi-\eta-\zeta) 
  \widehat{h}(\xi-\zeta) \,d\eta\,d\zeta, 
  \\
& V_{\mathfrak{n}}(f,g,h)(x) = \widehat{\mathcal{F}}^{-1}_{\xi \to x} 
  \iiint {\mathfrak{n}}(\xi,\eta,\zeta,\theta) \widehat{f}(\xi-\eta) \widehat{g}(\xi-\eta-\zeta-\theta) 
  \widehat{h}(\xi-\zeta) \frac{\widehat{\phi}(\theta)}{\theta} \, d\eta \,d\zeta \,d\theta.
\end{split}
\end{equation}
Of course, other parameterizations of $U_{\mathfrak{m}}$ and $V_{\mathfrak{n}}$ would be possible; 
but the parameterization above will
be particularly relevant since it is the one adopted in Section \ref{secwL}. 

In physical space, these are given by
\begin{equation}
\label{cubicphysical}
\begin{split}
& U_{\mathfrak{m}}(f,g,h)(w) = \frac{1}{\sqrt{2\pi}} 
\int \widehat {\mathfrak{m}}(-w+x+y+z,-x-y,-y-z) f(x) g(y) h(z)\,dx\,dy\,dz,
\\
& V_{\mathfrak{n}}(f,g,h)(w) = \iiint K(w,x,y,z) f(x) g(y) h(z) \,dx \,dy \, dz,
\\
& \quad \mbox{with} \quad
K(w,x,y,z) := \frac{1}{\sqrt{2\pi}} \int \widehat{\mathfrak{n}}(-w+x+y+z,-x-y,-y-z,y'-y) Z(y')\,dy', 
\end{split}
\end{equation}
with $Z$ as in \eqref{quadraticphysical}.

We have the following standard trilinear analogue of Lemma \ref{lemmamultilinquadratic1}.

\begin{lem}[boundedness for the $U_{\mathfrak{m}}$ operators]\label{lemmamultilin1}
If $1\leq p,q,r,s \leq \infty$ satisfy $\frac{1}{p} + \frac{1}{q} + \frac{1}{r} = \frac{1}{s} $,
$$
{\left\| U_{\mathfrak{m}} \right\|}_{L^p \times L^q  \times L^r \to L^{s}} \lesssim \| \widehat{{\mathfrak{m}}} \|_{L^1}.
$$ 
\end{lem}


\begin{rem}\label{multilinrem}
Given a symbol $\mathfrak{m}$, to check in practice that its Fourier transform is in $L^1$
we will use the following principles:

\begin{itemize}
\item If $\mathfrak{m}$ is supported on $(\xi,\eta,\zeta) \in [t_1-r_1,t_1+r_1] \times [t_2-r_2,t_2 + r_2 ] \times [t_3 - r_3 , t_3 + r_3]$ with
$$
\left| \partial_\xi^{k_1} \partial_\eta^{k_2} \partial_\zeta^{k_3} \mathfrak{m} \right| \lesssim r_1^{-k_1} r_2^{-k_2} r_3^{-k_3},
$$
then
$$
\left| \widehat {\mathfrak{m}}(x,y,z) \right| \lesssim \frac{r_1}{(1 + r_1 x)^N} \frac{r_2}{(1 + r_2 x)^N} \frac{r_3}{(1 + r_3 x)^N}, 
$$
so that in particular $\| \widehat {\mathfrak{m}} \|_{L^1} \lesssim 1$.

\item By the algebra property of the space $\mathcal{F} L^1$ (Wiener algebra), there holds
$$
\left\| \mathcal{F} (\mathfrak{m} \mathfrak{n}) \right\|_{L^1} 
  \lesssim \| \widehat {\mathfrak{m}} \|_{L^1} \| \widehat {\mathfrak{n}} \|_{L^1}.
$$

\item The previous point can be generalized to the case where $\mathfrak{n}$ is $L^1$ in a single direction, 
  and constant in the others. For instance, for any $a,b,c$ such that $|a| + |b| + |c| \sim 1$,
$$
\left\| \widehat{\mathcal{F}} [ \mathfrak{m}(\xi,\eta,\zeta) \varphi_{j}(a\xi + b\eta + c\zeta) ] \right\|_{L^1} 
  \lesssim \| \widehat {\mathfrak{m}} \|_{L^1}.
$$
This remains true if $\varphi_j$ is replaced by $\varphi_{<j}$ or $\varphi_{>j}$.
Indeed, for any linear transformation $L$ of $\mathbb{R}^3$ of determinant one, 
$\| \widehat{\mathfrak{m}} \|_{L^1} = \| \widehat{\mathfrak{m} \circ L} \|_{L^1}$. 
Therefore, it suffices to examine the case $a=1$ and $b=c=0$, which
immediately reduces to the fact that $L^1$ is an algebra for convolution.
\end{itemize}

\end{rem}

\begin{lem}[boundedness for the $V_{\mathfrak{n}}$ operators]\label{lemmamultilin2}
Assume that there exists $F \in L^1$ such that
$$
\left| \int \what{\mathfrak{n}}(x,y,z,t) \,dt \right| \leq F(x,y,z).
$$
Then if $1\leq p,q,r,s \leq \infty$ satisfy $\frac{1}{p} + \frac{1}{q} + \frac{1}{r} = \frac{1}{s} $, 
$$
\| V_{\mathfrak{n}} \|_{L^p \times L^q  \times L^r \to L^{s}} \lesssim \| F \|_{L^1}.
$$
\end{lem}

\begin{proof} 
 Since $Z \in L^\infty$, the proof reduces to that of Lemma \ref{lemmamultilin1}.
\end{proof}

\begin{rem}\label{multilinrem2}
Given a symbol $\mathfrak{n}$, to check in practice that it satisfies the condition of Lemma
\ref{lemmamultilin2} we will mostly rely on the following principles:

\begin{itemize}

\item It suffices that
\begin{equation} \label{condition*}
| \widehat{\mathfrak{n}}(x,y,z,t) | \lesssim F(x,y,z) G(t-L(x,y,z)),
\end{equation}
where $L$ is a linear function, and $F,G$ are rapidly decaying functions, with $L^1$ norm equal to $1$.

\item If the condition \eqref{condition*} holds for $\mathfrak{n}(\xi,\eta,\zeta,\theta)$,
it also does for $\mathfrak{n}(\xi,\eta,\zeta,\theta) \varphi_{j}(a\xi + b\eta + c\zeta + d \theta)$ 
(for a non-degenerate choice of $a,b,c,d$). The same holds if  $\varphi_j$ is replaced by $\varphi_{<j}$ or $\varphi_{>j}$.

\item Finally, if $\mathfrak{n}$ is supported on 
$(\xi,\eta,\zeta,\theta) \in [-r_1,r_1] \times [-r_2,r_2 ] \times [r_3 \times r_3] \times [-r_4,r_4]$ 
with
$$
\left| \partial_\xi^{k_1} \partial_\eta^{k_2} \partial_\zeta^{k_3} \partial_\theta^{k_4} \mathfrak{n} \right| \lesssim r_1^{-k_1} r_2^{-k_2} r_3^{-k_3} r_4^{-k_4},
$$
then the condition~\eqref{condition*} is satisfied.
\end{itemize}
\end{rem}

\medskip
\subsection{The normal form operator $T$}\label{ssecT} 
Recall the definition of $T^{\pm}_{\iota_1 \iota_2}$ in~\eqref{RenoT}.
Before bounding the full operator, we focus on an operator ($B_{\mathfrak{m}_{\iota_1 \iota_2}}^{\pm}$ below), 
which shares the same symbol as $T^{\pm}_{\iota_1 \iota_2}$, 
but where the phase $e^{it\Phi_{\iota_1 \iota_2}}$ is replaced by $1$, 
and the distorted Fourier transform by the flat Fourier transform.

\begin{lem}
\label{lemTbound}
Let $\mathfrak{m}_{\iota_1 \iota_2}^\pm(\xi,\eta,\zeta)$ be the symbol
defined in \eqref{RenoT}. Then, for any $\iota_1,\iota_2 \in \{+,-\}$, the bilinear operator
\begin{align}\label{lembountTop}
B_{\mathfrak{m}_{\iota_1 \iota_2}^\pm}: 
\quad (f,g) \mapsto \widehat{\mathcal{F}}^{-1} \iint \widehat{f}(\eta) \widehat{g}(\zeta) 
  \mathfrak{m}_{\iota_1 \iota_2}^\pm(\xi,\eta,\zeta) \,d\eta \,d\zeta
\end{align}
is bounded from $L^p \times L^q$ to $L^r$, 
where $\frac{1}{p} + \frac{1}{q} = \frac{1}{r}$, and $1 < p,q,r < \infty$,
and almost gains a derivative:
\begin{align}\label{lemTboundmain}
\begin{split}
{\| B_{\mathfrak{m}_{\iota_1 \iota_2}^\pm}(f,g) \|}_{L^r} 
   \lesssim \min\big( {\| \jnab^{-1+}f \|}_{L^p} {\| g \big\|}_{L^q},
   {\| f \|}_{L^p} {\| \jnab^{-1+} g \|}_{L^q} \big) .
\end{split}
\end{align}
Here we are using the notation `$-1+$' from the end of \S\ref{secNotation} 
to denote any number that is strictly larger than $-1$.
\end{lem}

\begin{proof} 
First observe that the Fourier multipliers $\mathbf{a}^{\epsilon}_{\l}(D)$, $\epsilon,\l \in \{+,-\}$ 
are bounded on $L^p$, $1< p < \infty$, by \eqref{TRk} and Mikhlin's multiplier theorem. 
Three different phase functions have to be considered. 
The case $(\iota_1,\iota_2)=(-,-)$ is clearly the simplest, and will not be examined any further. 
This leaves us with the cases $(+,+)$ and $(+,-)$, in other words it suffices to treat the operators 
$C_{\mathfrak{p}_1}$, $C_{\mathfrak{p}_2}$, $D_{\mathfrak{q}_1}$ and $D_{\mathfrak{q}_2}$ 
(these notations being defined in \eqref{defCC} and \eqref{defDD}) with
\begin{equation}
\begin{split}
\label{p1}
& \mathfrak{p}^1(\eta,\zeta) = \frac{1}{\langle \eta \rangle \langle \zeta \rangle}
  \frac{1}{\langle \eta + \zeta \rangle - \langle \eta \rangle - \langle \zeta \rangle},
\\
& \mathfrak{p}^2(\eta,\zeta) = \frac{1}{\langle \eta \rangle \langle \zeta \rangle} 
  \frac{1}{\langle \eta + \zeta \rangle + \langle \eta \rangle - \langle \zeta \rangle}, 
\end{split}
\end{equation}
and
\begin{equation}
\begin{split}
\label{p2}
& \mathfrak{q}^1(\eta,\zeta,\theta) 
  = \frac{1}{\langle \eta \rangle \langle \zeta \rangle} 
  \frac{1}{\langle \eta + \zeta + \theta \rangle - \langle \eta \rangle - \langle \zeta \rangle} 
  \varphi_{<-D_0} (R(\eta,\zeta) \theta),
  \\
& \mathfrak{q}^2(\eta,\zeta,\theta) 
  = \frac{1}{\langle \eta \rangle \langle \zeta \rangle}
  \frac{1}{\langle \eta + \zeta + \theta \rangle + \langle \eta \rangle - \langle \zeta \rangle} 
  \varphi_{<-D_0} (R(\eta,\zeta) \theta).
\end{split}
\end{equation}
We observe that bounds for the symbols 
$\frac{1}{\langle \eta + \zeta \rangle - \langle \eta \rangle - \langle \zeta \rangle}$ 
and $\frac{1}{\langle \eta + \zeta \rangle + \langle \eta \rangle - \langle \zeta \rangle}$, 
on the one hand, and $ \frac{1}{\langle \eta + \zeta + \theta \rangle - \langle \eta \rangle - \langle \zeta \rangle}$ 
and $ \frac{1}{\langle \eta + \zeta + \theta \rangle + \langle \eta \rangle - \langle \zeta \rangle}$, 
on the other hand, can be deduced one from the other by duality. 
They are not quite equivalent due to the factors $\frac{1}{\langle \eta \rangle \langle \zeta \rangle}$ 
and $ \varphi_{<-D_0} (R(\eta,\zeta) \theta)$, but the required changes in the proofs are superficial, 
and we shall only focus on $\mathfrak{p}_1$ and $\mathfrak{q}_1$.

With the definition of $\chi_\epsilon$ in \eqref{chi+-}, and the definition \eqref{cut1},
we localize the symbols by setting
\begin{align*}
& \mathfrak{p}^1_{\substack{j,k \\ \epsilon_1,\epsilon_2}}(\eta,\zeta) = \mathfrak{p}^1(\eta,\zeta) \chi_{\epsilon_1} (\eta) \varphi_j^{(0)}(\eta) \chi_{\epsilon_2} (\zeta) \varphi_k^{(0)}(\zeta),
\end{align*}
with a similar definition for $\mathfrak{q}^1_{\substack{j,k \\ \epsilon_1,\epsilon_2}}$.

\medskip
\noindent
{\it Case 1: $\epsilon_1=\epsilon_2$}. 
It follows from \eqref{californiacondor0} that
\begin{align}
\label{p1est}
& | \partial_\eta^a \partial_\zeta^b \mathfrak{p}^1_{\substack{j,k \\ \epsilon_1,\epsilon_2}} (\eta,\zeta) | 
\lesssim 2^{-\max(j,k)} 2^{-aj} 2^{-bk} 
\\
\label{q1est}
& | \partial_\eta^a \partial_\zeta^b \partial_\theta^c \mathfrak{q}^1_{\substack{j,k \\ \epsilon_1,\epsilon_2}} 
  (\eta,\zeta,\theta) | \lesssim 2^{-\max(j,k)} 2^{c \, \min(j,k)} 2^{-aj} 2^{-bk} .
\end{align}
By remarks \ref{marsupilami} and \ref{marsupilami2} 
and Lemmas \ref{lemmamultilinquadratic1} and \ref{lemmamultilinquadratic2},
$$
\Big\| C_{\mathfrak{p}^1_{\substack{j,k \\ \epsilon_1,\epsilon_2}}} \Big\|_{L^p \times L^q \to L^r} 
  + \Big\| D_{\mathfrak{q}^1_{\substack{j,k \\ \epsilon_1,\epsilon_2}}} \Big\|_{L^p \times L^q \to L^r} 
  \lesssim 2^{-\max(j,k)}
$$
and therefore, for $\delta > 0$,
$$
\Big\| C_{\langle \eta \rangle^{1-\delta} \mathfrak{p}^1_{\substack{j,k \\ \epsilon_1,\epsilon_2}}} 
  \Big\|_{L^p \times L^q \to L^r} + \left\| D_{\langle \eta \rangle^{1-\delta} \mathfrak{q}^1_{ \substack{j,k \\ \epsilon_1,\epsilon_2}}} \right\|_{L^p \times L^q \to L^r} \lesssim 2^{- \delta \max(j,k)}.
$$
Summing over $k,j \geq 0$ gives the desired result.

\medskip
\noindent
{\it Case 2: $\epsilon_1\neq\epsilon_2$}. 
Adding a localization in $\eta + \zeta$, let 
\begin{align*}
& \mathfrak{p}^1_{\substack{j,k,\ell \\ \epsilon_1,\epsilon_2,\epsilon_3}}(\eta,\zeta) 
  = \mathfrak{p}^1(\eta,\zeta) \chi_{\epsilon_1} (\eta) \varphi_j^{(0)}(\eta) 
  \chi_{\epsilon_2} (\zeta) \varphi_k^{(0)}(\zeta) \chi_{\epsilon_3} (\eta + \zeta) \varphi_\ell^{(0)}(\eta + \zeta),
\\
& \mathfrak{q}^1_{\substack{j,k,\ell \\ \epsilon_1,\epsilon_2,\epsilon_3}}(\eta,\zeta,\theta)
  = \mathfrak{q}^1(\eta,\zeta,\theta) \chi_{\epsilon_1} (\eta) \varphi_j^{(0)}(\eta) \chi_{\epsilon_2} (\zeta) 
  \varphi_k^{(0)}(\zeta) \chi_{\epsilon_3} (\eta + \zeta) \varphi_\ell^{(0)}(\eta + \zeta).
\end{align*}
Without loss of generality, we can assume that $\eta>0$, $\zeta<0$, and $|\eta| > |\zeta|$. 
Changing variables to $\alpha = \eta + \zeta$ and $\beta = - \zeta$, the above symbols become
\begin{align*}
& \mathfrak{P}^1_{\substack{j,k,\ell \\ \epsilon_1,\epsilon_2,\epsilon_3}}(\alpha,\beta) = 
  \mathfrak{p}^1(\alpha+\beta,-\beta) \chi_{+} (\alpha + \beta) \varphi_j^{(0)}(\alpha + \beta) 
  \chi_{+} (\beta) \varphi_k^{(0)}(\beta) \chi_{+} (\alpha) \varphi_\ell^{(0)}(\alpha),
\\
& \mathfrak{Q}^1_{\substack{j,k,\ell \\ \epsilon_1,\epsilon_2,\epsilon_3}}(\alpha,\beta,\theta)
  = \mathfrak{q}^1(\alpha+\beta,-\beta,\theta) \chi_{+} (\alpha + \beta) \varphi_j^{(0)}(\alpha + \beta)
  \chi_{+} (\beta) \varphi_k^{(0)}(\beta) \chi_{+} (\alpha) \varphi_\ell^{(0)}(\alpha),
\end{align*}
where $j \geq \max(k,\ell) +C$. 
By~\eqref{californiacondor0} and~\eqref{californiacondor01},
\begin{align*}
& \Big| \partial_\alpha^a \partial_\beta^b 
  \mathfrak{P}^1_{\substack{j,k,\ell \\ \epsilon_1,\epsilon_2,\epsilon_3}}(\alpha,\beta) \Big|
  \lesssim 2^{-j-k} 2^{-a \ell} 2^{-b k},
  \\
& \Big| \partial_\alpha^a \partial_\beta^b \partial_\theta^c \, \mathfrak{Q}^1_{\substack{j,k,\ell}}(\alpha,\beta,\theta) 
  \Big| \lesssim 2^{-j-k} 2^{-a \ell} 2^{-b k} 2^{ck}.
\end{align*}
The desired estimate follows through Remarks \ref{marsupilami} and~\ref{marsupilami2} 
(in particular the paragraphs on change of coordinates) and 
Lemmas \ref{lemmamultilinquadratic1} and~\ref{lemmamultilinquadratic2}.
\end{proof}

In the previous lemma we derived bounds for the bilinear operator $B_{\mathfrak{m}_{\iota_1 \iota_2}^{\pm}}$.
In order to deduce bounds for $T_{\iota_1 \iota_2}^{\pm}$ itself, 
we need to substitute the distorted Fourier transform to the flat Fourier transform 
(this is achieved through the wave operator $\mathcal{W}$, see \eqref{propWdef}), 
and take into account the phase $e^{i\Phi_{\iota_1 \iota_2}}$. 

\begin{lem}[Estimates for $T$]\label{lemT}
Consider the operators $T^\pm_{\iota_1\iota_2}$ defined in \eqref{Renof}-\eqref{RenoT}.
For all $p,p_1,p_2 \in (1,\infty)$ with $\frac{1}{p_1}+\frac{1}{p_2} = \frac{1}{p}$, we have
\begin{align}\label{bilboundT}
\begin{split}
{\big\| e^{-it\jnab} \mathcal{W}^* T^\pm_{\iota_1\iota_2}(f_1,f_2)(t) \big\|}_{L^p} 
   \lesssim 
   \min\Big( & {\big\| \jnab^{-1+}e^{-\iota_1it\jnab} \mathcal{W}^* f_1 \big\|}_{L^{p_1}}
   {\big\| e^{-\iota_2it\jnab} \mathcal{W}^*f_2 \big\|}_{L^{p_2}},
   \\
   & {\big\| e^{-\iota_1it\jnab} \mathcal{W}^* f_1 \big\|}_{L^{p_1}}
   {\big\| \jnab^{-1+}e^{-\iota_2it\jnab} \mathcal{W}^*f_2 \big\|}_{L^{p_2}}\Big) .
\end{split}
\end{align}
Furthermore, for any $k \geq 0$
\begin{align}\label{bilboundT'rev}
\begin{split}
{\big\|  e^{-it\jnab} \mathcal{W}^* T^{\pm}_{\iota_1 \iota_2} (f_1,f_2)(t) \big\|}_{W^{k,p}} & \lesssim 
   {\| \jnab^{k-1+} e^{-\iota_1it\jnab} \mathcal{W}^* f_1 \|}_{L^{p_1}} 
   {\big\| e^{-\iota_2it\jnab} \mathcal{W}^*f_2 \big\|}_{L^{p_2}}
   \\
   & \qquad + {\big\| e^{-\iota_1it\jnab}\mathcal{W}^*f_1 \big\|}_{L^{p_3}} 
   {\| \jnab^{k-1+} e^{-\iota_2it\jnab} \mathcal{W}^* f_2 \|}_{L^{p_4}},
\end{split}
\end{align}
with $(p_3,p_4)$ satisfying the same constraints as $(p_1,p_2)$ above.

Finally, if $p\in(1,\infty)$, and $f$ is a function that satisfies the
(second and third) assumptions in \eqref{propbootfas},
then, for all $t\in[0,T]$
\begin{align}\label{bilboundTep}
\begin{split}
{\big\| e^{-it\jnab} \mathcal{W}^* T^\pm_{\iota_1\iota_2}
  \big( 
  f, f_2 \big)(t) \big\|}_{L^p} 
   & \lesssim \frac{\e_1}{\sqrt{t}} \cdot 
   {\big\| \jnab^{-1+}e^{-\iota_2it\jnab} \mathcal{W}^*f_2 \big\|}_{L^p}.
\end{split}
\end{align}

\end{lem}

\begin{proof}
We can write
\begin{align}\label{lemTpr1}
\begin{split}
e^{-it\jnab} \mathcal{W}^* T_{\iota_1 \iota_2}^\pm(f_1,f_2)(t) & = 
\what{\mathcal{F}}^{-1} e^{-it\jxi} \wt{\mathcal{F}} T^\pm_{\iota_1 \iota_2}(f_1,f_2)(t)
\\
& = \what{\mathcal{F}}^{-1}\iint e^{-it\iota_1\jeta}
  \wt{f_1}(t,\eta) \, e^{-it\iota_2\jsig} \wt{f_2}(t,\sigma) 
  \mathfrak{m}_{\iota_1\iota_2}^\pm(\xi,\eta,\sigma)\,d\eta \,d\sigma 
\\
& = B_{\mathfrak{m}_{\iota_1\iota_2}^\pm} 
  \big( e^{-it\iota_1\jnab} \mathcal{W}^*{f_1}(t), e^{-it\iota_2\jnab} \mathcal{W}^*{f_2}(t) \big),
\end{split}
\end{align}
see the notation of Lemma \ref{lemTbound}. 
Applying the conclusion of Lemma \ref{lemTbound} immediately gives \eqref{bilboundT}.

To prove \eqref{bilboundT'rev} we first write 
\begin{align*}
{\| e^{-it\jnab} \mathcal{W}^* T^{\pm}_{\iota_1 \iota_2}(f_1,f_2)(t) \|}_{W^{k,p}} 
  & \lesssim 
  {\big\| B_{\jxi^k \mathfrak{m}_{\iota_1\iota_2}} \big( e^{-it\iota_1\jnab} \langle D \rangle^k \mathcal{W}^*{f_1}(t), 
  e^{-it\iota_2\jnab} \mathcal{W}^*{f_2}(t) \big) \big\|}_{L^p}
\end{align*}
(we are dropping the irrelevant $\pm$ apex).
Without loss of generality, we may assume that $|\eta| \geq |\s|$ and $|\xi| \geq 1$ on the support of \eqref{lemTpr1}.
We then want to estimate the $L^p$ norm of 
\begin{align}\label{lemTpr3}
\begin{split}
\what{\mathcal{F}}^{-1}\iint e^{-it\iota_1\jeta}
  \jeta^k \wt{f_1}(t,\eta) \, e^{-it\iota_1\jsig} \wt{f_2}(t,\sigma) 
  \big[ \jxi^k \jeta^{-k} \mathfrak{m}_{\iota_1\iota_2}(\xi,\eta,\sigma) \big]\,d\eta \,d\sigma 
\\
= B_{\mathfrak{m}'_{\iota_1\iota_2}} \big( e^{-it\iota_1\jnab} \mathcal{W}^*{f_1}(t), 
  e^{-it\iota_2\jnab} \mathcal{W}^*{f_2}(t) \big)
\end{split}
\end{align}
with the obvious definition of $\mathfrak{m}_{\iota_1\iota_2}'$.
Note that, from the definition \eqref{RenoT} with \eqref{QZ}, 
$\mathfrak{m}_{\iota_1\iota_2}'$ can be written, up to irrelevant constants, as
$\mathfrak{m}_{\iota_1\iota_2}' = \mathfrak{a}_{\iota_1\iota_2} + \mathfrak{b}_{\iota_1\iota_2}$, with
\begin{align}\label{lemTpr4}
\begin{split}
\mathfrak{a}_{\iota_1\iota_2} & := \frac{\jxi^k}{\jeta^k} \sum_{\l,\mu,\nu}
  a^\pm_{\substack{-\iota_1 \iota_2 \\ \lambda\mu\nu}}(\xi, \eta,\sigma) 
  \frac{1}{\Phi_{\iota_1\iota_2}(\xi,\eta,\s)\jeta \jsig} \, \delta(p),
\\
\mathfrak{b}_{\iota_1\iota_2} & := \frac{\jxi^k}{\jeta^k} \sum_{\l,\mu,\nu}
  a^\pm_{\substack{-\iota_1 \iota_2 \\ \lambda\mu\nu}}(\xi,\eta, \sigma)
  \frac{\varphi^\ast(p,\eta,\sigma)}{\Phi_{\iota_1\iota_2}(\xi,\eta,\s)\jeta\jsig}
  \,\pv \frac{\widehat{\phi}(p)}{ip}, 
  \qquad p := \lambda \xi - \mu \eta - \nu \sigma. 
\end{split}
\end{align}

On the support of $\mathfrak{a}_{\iota_1\iota_2}$ 
we automatically must have $\jxi \lesssim \max(\jeta,\jsig) = \jeta$,
so that $\mathfrak{a}_{\iota_1\iota_2}$ is a regular bounded symbol with the same properties as 
$\mathfrak{m}_{\iota_1\iota_2}^{\pm}$;  from the result of Lemma \ref{lemTbound} we deduce
\begin{align*}
\begin{split}
& {\| B_{\mathfrak{a}_{\iota_1\iota_2}} \big( e^{-it\iota_1\jnab} \jnab^k \mathcal{W}^*{f_1}(t),
  e^{-it\iota_2\jnab} \mathcal{W}^*{f_2}(t) \big) \|}_{L^p}
  \\ 
  & \lesssim {\| \jnab^{k-1+} e^{-it\iota_1\jnab} \mathcal{W}^*{f_1}(t) \|}_{L^{p_1}} 
  {\| e^{-it\iota_2\jnab} \mathcal{W}^*{f_2}(t) \|}_{L^{p_2}},
\end{split}
\end{align*}
consistently with the right-hand side of \eqref{bilboundT'rev}.

On the support of the $\pv$ component $\mathfrak{b}_{\iota_1\iota_2}$ we might not have that $\jxi \lesssim \jeta$.
However if $\jxi \gg \jeta$, then $|p| \gtrsim |\xi|$ (in particular the $\pv$ is not singular)
and one can absorb the factor of $\jxi^k$.
More precisely, we can write (dispensing of the $\iota_1\iota_2$ indexes)
\begin{align*}
\mathfrak{b} 
  = \mathfrak{b}_{1} + \mathfrak{b}_{2}, \qquad \mathfrak{b}_{1}:=\varphi_{\leq 10}(|\xi|/|\eta|) \mathfrak{b}
\end{align*}
and observe that $\mathfrak{b}_{1}$ has the same properties as (the $\pv$ part of) $\mathfrak{m}$ so that Lemma 
\ref{lemTbound} applies and 
\begin{align*}
{\big\| B_{\mathfrak{b}_1} \big( e^{-it\iota_1\jnab} \jnab^k \mathcal{W}^*{f_1}(t), 
  e^{-it\iota_2\jnab} \mathcal{W}^*{f_2}(t) \big) \big\|}_{L^p} 
  \\ \lesssim
  {\| \jnab^{k-1+} e^{-it\iota_1\jnab} \mathcal{W}^* f_1 \|}_{L^{p_1}} {\big\| e^{-it\jnab} \mathcal{W}^*f_2 \big\|}_{L^{p_2}}.
\end{align*}
The contribution from the remaining piece $\mathfrak{b}_2$ can be written as
\begin{align*}
B_{\mathfrak{b}_{2}} \big( e^{-it\iota_1\jnab} \jnab^k \mathcal{W}^*{f_1}(t), 
  e^{-it\iota_2\jnab} \mathcal{W}^*{f_2}(t) \big)
  = B_{\mathfrak{b}'} \big( e^{-it\iota_1\jnab} \mathcal{W}^*{f_1}(t), 
  e^{-it\iota_2\jnab} \mathcal{W}^*{f_2}(t) \big)
\end{align*}
where
\begin{align*}
\mathfrak{b}' & := \sum_{\l,\mu,\nu}
  a^\pm_{\substack{-\iota_1 \iota_2 \\ \lambda\mu\nu}}(\xi,\eta, \sigma)
  \frac{\varphi^\ast(\xi,\eta,p)}{\Phi_{\iota_1\iota_2}(\xi,\eta,\s)\jeta\jsig}
  \cdot \frac{\widehat{\phi}(p)}{ip} \cdot \jxi^k \varphi_{>10}(|\xi|/|\eta|) \varphi_{\geq 0}(|p|/|\xi|). 
\end{align*}
Since $|p| \gtrsim |\xi|$ and $\what{\phi}$ is a Schwartz function, the symbol $\mathfrak{b}'$
has the same properties as $\mathfrak{m}$;
using an $L^{p_1}\times L^{p_2}$ estimate from Lemma \ref{lemTbound} gives 
\begin{align*}
{\big\| B_{\mathfrak{b}'} \big( e^{-it\iota_1\jnab} \mathcal{W}^*{f_1}(t), 
  e^{-it\iota_2\jnab} \mathcal{W}^*{f_2}(t) \big) \big\|}_{L^p} 
  \lesssim
  {\| e^{-it\iota_1\jnab} \mathcal{W}^* f_1 \|}_{L^{p_1}} {\big\| e^{-it\iota_2\jnab} \mathcal{W}^*f_2 \big\|}_{L^{p_2}}.
\end{align*}
which is better than the desired conclusion.

Finally, to prove \eqref{bilboundTep}, we use the linear dispersive estimate \eqref{dispinfty}
to take care of the $\mathbf{a}^\epsilon_\l$ multipliers,
instead of the Mikhlin multiplier theorem.
\end{proof}

\medskip
\subsection{The smooth bilinear operator $\mathcal{Q}^R$}
\begin{lem}[Estimates for $\mathcal{Q}^R$]\label{lemQR}
Let $\mathcal{Q}^R$ be the bilinear term defined in \eqref{QR} and \eqref{QR1}-\eqref{QR2}.
Then, for any 
\begin{align*}
p_1,p_2\in[2,\infty), \qquad \frac{1}{p_1}+\frac{1}{p_2} < \frac{1}{2},
\end{align*}
one has the improved H\"older-type inequality
\begin{align}\label{lemQRpq}
\begin{split}
{\big\| 
  \mathcal{Q}^R_{\iota_1\iota_2}[f_1,f_2](t,\xi) \big\|}_{L^2} 
  \lesssim \min\big( & {\| \jnab^{-1+} 
  e^{-i\iota_1t\jnab} \mathcal{W}^*f_1 \|}_{L^{p_1}} 
  {\| 
  e^{-i\iota_2t\jnab} \mathcal{W}^*f_2 \|}_{L^{p_2}},
  \\ 
  & {\| 
  e^{-i\iota_1t\jnab} \mathcal{W}^*f_1 \|}_{L^{p_1}} 
  {\| \jnab^{-1+} 
  e^{-i\iota_2t\jnab} \mathcal{W}^*f_2 \|}_{L^{p_2}} \big).
\end{split}
\end{align}
Moreover, for $k \geq 0$,
\begin{align}\label{lemQRsob}
\begin{split}
{\big\| 
  \jxi^k \mathcal{Q}^R_{\iota_1\iota_2}[f_1,f_2](t,\xi) \big\|}_{L^2} 
  \lesssim \, & \, {\|  \jnab^{k-1+} e^{-i\iota_1t\jnab} \mathcal{W}^*f_1 \|}_{L^{p_1}} 
  {\| e^{-i\iota_2t\jnab} \mathcal{W}^*f_2 \|}_{L^{p_2}}
  \\ 
  + \, & \, {\| e^{-i\iota_1t\jnab} \mathcal{W}^*f_1 \|}_{L^{p_3}}
  {\| \jnab^{k-1+} e^{-i\iota_2t\jnab} \mathcal{W}^*f_2 \|}_{L^{p_4}}
\end{split}
\end{align}
for $(p_3,p_4)$ satisfying the same constraints as $(p_1,p_2)$ above.



\end{lem}

\begin{proof}
Recall the structure of the symbol of $\mathcal{Q}^R$
from \eqref{QR1}-\eqref{QR2} and \eqref{muR}-\eqref{muR'}.
For the piece coming from $\mu^R_{\iota_1\iota_2}$, estimates stronger than the desired
\eqref{lemQRpq}-\eqref{lemQRsob} follow directly from Lemma \ref{lemmuR}.
We then only need to look at operators of the form
\begin{align}\label{lemQRpr1} 
\begin{split}
& \mathcal{Q}[a,b](\xi)
  = \iint \mathfrak{q}(\xi,\eta,\sigma) 
  \, \widehat{a}(\eta) \widehat{b}(\sigma) \, d\eta \, d\sigma
\\
& \mathfrak{q}(\xi,\eta,\sigma) = 
  \frac{a^\pm_{\substack{\iota_0 \iota_1 \iota_2 \\ \lambda\mu\nu}}(\xi,\eta,\sigma)}{\langle \eta \rangle \langle \sigma \rangle } 
  \big(1 - \varphi^\ast(p,\eta,\s) \big)
  \frac{\what{\phi}(p)}{p}, \quad p=\lambda \xi - \iota_1 \mu \eta - \iota_2 \nu \sigma,
\end{split}
\end{align}
and prove that
\begin{align}\label{prlemQRpq}
\begin{split}
{\big\| \whF^{-1} \big( \mathcal{Q}[a,b](\xi) \big) \big\|}_{L^p} 
  \lesssim {\| \jnab^{-1+} a \|}_{L^{p_1}} {\|  b \|}_{L^{p_2}},
\end{split}
\end{align}
since by symmetry between the arguments $a$ and $b$
it follows that the right-hand side above can be replaced by 
$\min\big( {\| \jnab^{-1+} a \|}_{L^{p_1}} {\|  b \|}_{L^{p_2}}, 
{\| a \|}_{L^{p_1}} {\| \jnab^{-1+} b \|}_{L^{p_2}} \big)$, 
and that
\begin{align}\label{prlemQRsob}
\begin{split}
{\big\| 
  \jxi^l \mathcal{Q} [a,b](\xi) \big\|}_{L^2} 
  \lesssim \, & \, {\|  \jnab^{l-1+} a \|}_{L^{p_1}} {\| b \|}_{L^{p_2}} 
  + {\| a \|}_{L^{p_3}} {\|  \jnab^{l-1+} b \|}_{L^{p_4}}.
\end{split}
\end{align}

\medskip
\noindent
{\it Proof of \eqref{prlemQRpq}}. 
As usual, the first step is to observe that the multipliers 
$a^\pm_{\substack{\iota_0 \iota_1 \iota_2 \\ \lambda\mu\nu}}(\xi,\eta,\sigma)$ can be discarded. 
Also, we may assume without loss of generality that $p=\xi-\eta-\sigma$.
Next, we insert Littlewood-Paley cutoffs in each of the variables $\eta,\s$ and $p$
and consider the localized operator $\mathcal{Q}_{\underline{k}}[a,b](\xi)$,
with the same form as $\mathcal{Q}$ in \eqref{lemQRpr1} but a localized symbol
\begin{align}\label{lemQRqk}
\begin{split}
& \mathfrak{q}_{\underline{k}}(\xi,\eta,\s) =  
  \frac{\mathfrak{m}_{\underline{k}} (\xi,\eta,\s)}{\jeta \jsig}
  \frac{\what{\phi}(p)}{p} 
\\
& \mathfrak{m}_{\underline{k}} (\xi,\eta,\s) = \varphi_{k_1}(\eta)\varphi_{k_2}(\s)\varphi_{k_3}(p)
  ( 1 - \varphi^*(p,\eta,\s)).
\end{split}
\end{align}
We then make the following restrictions on the indexes
$$
k_1 \geq k_2 \geq 0 , \qquad \mbox{and} \qquad  0 \geq k_3 \geq -k_2^+ - D.
$$
These can be explained as follows: 
$k_1 \geq k_2$ is the harder case, since it implies that the derivative gain in \eqref{prlemQRpq} 
will be on the larger input frequency; 
$k_2 \geq 0$ amounts to restricting to the case of frequencies $\gtrsim 1$, 
which is also the hardest case;
$k_3 \leq 0$ is a consequence of $\phi$ being Schwartz: 
values of $|p| \gg 1$ are exponentially damped, 
and we will not worry about them here; and finally, $k_3 \geq -k_2^+ - D$ follows from the definition of $\varphi^\ast$.

The idea then is to regard $\mathcal{Q}_{\underline{k}}(a,b)$ as a trilinear operator acting on $a$, $b$
and $\psi = \mathcal{F}^{-1} (\widehat{\phi}/p)$, and note that the symbol
$\mathfrak{n}_{\underline{k}} (\eta,\s,p) := \mathfrak{m}_{\underline{k}}(\xi,\eta,\s)$ satisfies
$$
\left| \partial_\eta^a \partial_\sigma^b \partial_p^c  \mathfrak{m}_{\underline{k}} (\eta,\sigma,p) \right| 
  \lesssim 2^{-ak_1 - b k_2 - ck_3}.
$$
Up to a change of coordinates, Lemma \ref{lemmamultilin1} and Remark \ref{multilinrem} apply, 
leading to the estimate, for any $1 < q,p_1,p_2 < \infty$ 
such that $\frac{1}{q} + \frac{1}{p_1} + \frac{1}{p_2} = \frac{1}{p}$,
\begin{align*}
{\big\| \mathcal{Q}_{\underline{k}} (a,b) \big\|}_{L^p} 
  & \lesssim {\| P_{k_3} \psi \|}_{L^q} {\| P_{k_1} \jnab^{-1} a \|}_{L^{p_1}} {\| P_{k_2} \jnab^{-1}  b \|}_{L^{p_2}} 
\\
& \lesssim 2^{-\frac{k_3}{q}- k_2} 2^{-(0+)k_1} {\| \jnab^{-1+} P_{k_1}  a \|}_{L^{p_1}}  {\| P_{k_2} b \|}_{L^{p_2}}.
\end{align*}
It remains to observe that, provided $1 < q < \infty$,
$$
\sum_{\substack {0 \geq k_3 \geq -k_2^+ \\ k_1 \geq k_2 \geq 0}} 2^{-\frac{k_3}{q} - k_2} 2^{-(0+)k_1} < \infty.
$$

\medskip
\noindent
{\it Proof of \eqref{prlemQRsob}}.
One can proceed as above, modifying the definition of $\mathfrak{m}_{\underline{k}}$ to
$$
\mathfrak{m}_{\underline{k}} (\xi,\eta,\sigma) = \varphi_{k_1}(\eta)\varphi_{k_2}(\s)\varphi_{k_3}(p)( 1 - \varphi^*(\xi,\eta,\sigma)) \frac{\jxi^k}{\jeta^k + \jsig^k}.
$$
and observing that if $|\xi| \geq 3\max(|\eta|,|\sigma|)$, then $|p| \gtrsim |\xi|$,
and any power of $\jxi$ can absorbed by $\what{\phi}(p)$.
\end{proof}

From the proof of Lemma \ref{lemQR} above we can also deduce the following property,
which will be useful in Subsection \ref{Ssecexp}. 

\begin{claim}\label{remdxiQR}
We have the following schematic identity for the operator $\mathcal{Q}^R$ in \eqref{QR}:
\begin{align}\label{dxiQRid}
\jxi \partial_\xi \mathcal{Q}^R[f_1,f_2] \approx t \cdot \jxi \mathcal{Q}^R[f_1,f_2]
  + \jxi \mathcal{Q}^R\big[\wtF^{-1}(\partial_\xi \wt{f_1}),f_2\big].
\end{align}

In particular, \eqref{dxiQRid} and \eqref{lemQRpq} imply 
the following H\"older-type estimate for $\jxi \partial_\xi \mathcal{Q}^R$,
up to lower order terms that can be discarded:
\begin{align}\label{dxiQR2+infty}
\begin{split}
{\| \jxi \partial_\xi \mathcal{Q}^R(f_1,f_2) \|}_{L^2} \lesssim
   \jt {\big\| e^{-it\jnab} \jnab^{0+} \mathcal{W}^* f_1 \big\|}_{L^{\infty-}}
   {\big\| e^{-it\jnab} \jnab^{0+} \mathcal{W}^*f_2 \big\|}_{L^{\infty-}}
   \\
   + {\big\| \jxi^{0+} \partial_\xi \wt{f_1} \big\|}_{L^2}
   {\big\| \jnab^{0+} e^{-it\jnab} \mathcal{W}^*f_2 \big\|}_{L^{\infty-}}.
\end{split}
\end{align}
Here, we are using $\infty-$ to denote any arbitrarily large number
(see the notation at the end of \S\ref{secNotation}).
Note that this last bound is technically a little worse than what one could get,
that is, a bound with only one term at a time carrying a $\jxi^{0+}$ factor in the last product.
\end{claim}

An analogous claim holds for the operator $T$, see Remark \ref{remdxiT}.
For the case of $T$ the proof is contained in the proof of Lemma \ref{lemdxig};
we refer the reader to that for more details on the type of argument 
that leads to \eqref{dxiQRid}, and provide a more succinct argument below.

\begin{proof}[Proof of Claim \ref{remdxiQR}]
To see the validity of \eqref{dxiQRid} we look at the expression \eqref{QR1}-\eqref{QR2}.
Applying $\jxi \partial_\xi$ gives two contributions: one where $\jxi\partial_\xi$ hits the exponential phase
and one where it hits the symbol $\mathfrak{q}$. 
The first contribution is $t \xi \cdot \mathcal{Q}^R[f_1,f_2]$, which appears on the right-hand side of \eqref{dxiQRid}. 

When $\partial_\xi$ hits the symbol we get a few more contributions.
First, we observe that $\partial_\xi \mu_0^R$ behaves exactly like $\mu_0^R$ so this is a lower order term that we 
can disregard; see Proposition \ref{muprop} and Lemma \ref{lemmuR}. 
When $\partial_\xi$ hits $\mathfrak{q}$ we get similar lower order terms, with the exception
of the contributions coming from $\partial_\xi$ hitting $\pv 1/p$ or $\varphi^\ast$.
Under the assumption that $|\eta|\gtrsim |\s|$, in view of the definition of $p$,
we convert $\partial_\xi$ to $\partial_\eta$ and integrate by parts in $\eta$.
When $\partial_\eta$ hits the profile $\wt{f_1}(\eta)$ we get the second term on the right-hand side of \eqref{dxiQRid}.
When $\partial_\eta$ hits the oscillating phase, we get a term like the first one in \eqref{dxiQRid}.
The other terms where $\partial_\eta$ hits the remaining part of the symbol only contribute lower order
terms which satisfy stronger estimates than the terms in \eqref{dxiQRid}.
\end{proof}

\medskip
\subsection{The singular cubic terms $\mathcal{C}^S$}
The next Lemma is a H\"older-type estimate for the singular cubic terms.

\begin{lem}[Estimates for ``cubic singular'' symbols]\label{lemCS}
With the definition in \eqref{CubicS}-\eqref{CubicS12},
consider $\mathcal{C}^{S} = \mathcal{C}^{Sr}_{\iota_1 \iota_2 \iota_3}$, 
for $r=1$ or $2$, and any combination of signs $\iota$.
Then, for all $p,p_1,p_2,p_3 \in (1,\infty)$ with $\frac{1}{p_1} + \frac{1}{p_2} + \frac{1}{p_3} = \frac{1}{p}$, 
\begin{align}\label{lemCS1}
\begin{split}
& {\|e^{-it\jnab}\whF^{-1}\mathcal{C}^{S}(a,b,c)\|}_{L^p}
 \\
 & \qquad \lesssim
 {\| \jnab^{-1+} e^{-it\jnab} \mathcal{W}^* a \|}_{L^{p_1}}
 {\| \jnab^{-1+} e^{-it\jnab} \mathcal{W}^* b \|}_{L^{p_2}}
 {\| \jnab^{-1+} e^{-it\jnab} \mathcal{W}^* c \|}_{L^{p_3}}.
\end{split}
\end{align}

Furthermore, if $k \geq0$, and with $(p_4,p_5,p_6)$ and $(p_7,p_8,p_9)$ 
satisfying the same conditions as $(p_1,p_2,p_3)$,
\begin{align}\label{lemCS2}
\begin{split}
& {\|e^{-it\jnab}\whF^{-1}\mathcal{C}^{S}(a,b,c)\|}_{W^{k,p}}
 \\
 & \qquad \lesssim 
 {\| \jnab^{k-1+ } e^{-it\jnab} \mathcal{W}^* a \|}_{L^{p_1}}
 {\| \jnab^{-1+} e^{-it\jnab} \mathcal{W}^* b \|}_{L^{p_2}}
 {\| \jnab^{-1+} e^{-it\jnab} \mathcal{W}^* c \|}_{L^{p_3}}  \\
&  \qquad \qquad + 
 {\| \jnab^{-1+ } e^{-it\jnab} \mathcal{W}^* a \|}_{L^{p_4}}
 {\| \jnab^{k-1+} e^{-it\jnab} \mathcal{W}^* b \|}_{L^{p_5}}
 {\| \jnab^{-1+} e^{-it\jnab} \mathcal{W}^* c \|}_{L^{p_6}} \\
& \qquad \qquad +
 {\| \jnab^{-1+ } e^{-it\jnab} \mathcal{W}^* a \|}_{L^{p_7}}
 {\| \jnab^{-1+} e^{-it\jnab} \mathcal{W}^* b \|}_{L^{p_8}}
 {\| \jnab^{k-1+} e^{-it\jnab} \mathcal{W}^* c \|}_{L^{p_9}}.
\end{split}
\end{align}

Finally, if $p_1 = \infty$, and $f$ is a function that satisfies the
(second and third) assumptions in \eqref{propbootfas},
then, for all $t\in[0,T]$ and $ \frac{1}{p_2} + \frac{1}{p_3} = \frac{1}{p}$, we have
\begin{align}\label{lemCSend}
& {\|e^{-it\jnab}\whF^{-1}\mathcal{C}^{S}(f 
  ,b,c)\|}_{L^p} 
  \lesssim \frac{\e_1}{\sqrt{t}} 
  {\| \jnab^{-1+} e^{-it\jnab} \mathcal{W}^* b \|}_{L^{p_2}}  
  {\| \jnab^{-1+} e^{-it\jnab} \mathcal{W}^* c \|}_{L^{p_3}},
\end{align}
with a similar statement if $p_1 = p_2 = \infty$.
\end{lem}

\begin{proof}
Starting from the formulas \eqref{formulacubiccoeff} 
giving  $\mathfrak{c}^{S1}$ and $\mathfrak{c}^{S2}$, we first discard the factors 
$$a^{\epsilon,\epsilon'}_{\substack{\lambda,\mu,\mu',\nu' \\ -,\kappa_1,\kappa_2,\kappa_3}}(\xi,\eta,\eta',\sigma'),$$
which is possible thanks to the Mikhlin multiplier theorem. 
Omitting these factors and irrelevant constants and indexes, 
it suffices to deal with $T_{\mathfrak{e}^1}$ and $T_{\mathfrak{e}^2}$
(recall the definition in~\eqref{defgeneraltrilin}), where
\begin{align*}
& \mathfrak{e}^1(\xi,\eta,\zeta,\theta) = \frac{A(\xi \pm \eta)}{\Phi_{\iota_1\iota_2}(\xi,\eta,\xi\pm \eta)} 
  \frac{1}{\langle \eta \rangle \langle \xi \pm \eta\rangle \langle \zeta \rangle \langle \theta \rangle} 
  \delta(\xi \pm \eta \pm \zeta \pm \theta), 
\\
& \mathfrak{e}^2(\xi,\eta,\zeta,\theta) = \frac{A(\xi \pm \eta)}{\Phi_{\iota_1\iota_2}(\xi,\eta,\xi\pm \eta)} 
  \frac{1}{\langle \eta \rangle \langle \xi \pm \eta\rangle \langle \zeta \rangle \langle \theta \rangle}
\frac{\widehat{\phi}(\xi \pm \eta \pm \zeta \pm \theta)}{\xi \pm \eta \pm \zeta \pm \theta}.
\end{align*}
For the sake of concreteness, we make a choice of signs (which one it is exactly does not matter):
\begin{align*}
& \mathfrak{e}^1(\xi,\eta,\zeta,\theta) = \frac{A(\xi + \eta)}{\Phi_{\iota_1\iota_2}(\xi,\eta,\xi + \eta)} 
  \frac{1}{\langle \eta \rangle \langle \xi + \eta\rangle \langle \zeta \rangle \langle \theta \rangle}
  \delta(\xi - \eta + \zeta - \theta), 
  \\
& \mathfrak{e}^2(\xi,\eta,\zeta,\theta) = \frac{A(\xi + \eta)}{\Phi_{\iota_1\iota_2}(\xi,\eta,\xi + \eta)} 
  \frac{1}{\langle \eta \rangle \langle \xi + \eta\rangle \langle \zeta \rangle \langle \theta \rangle}
  \frac{\widehat{\phi}(\xi - \eta + \zeta - \theta)}{\xi - \eta + \zeta - \theta}.
\end{align*}
With the convention for $U$ and $V$ operators (see~\eqref{defUUVV}), this corresponds respectively to the symbols
\begin{align*}
& \mathfrak{f}^1(\xi,\eta,\zeta)= \frac{A(2\xi-\eta)}{\Phi_{\iota_1\iota_2}(\xi,\xi-\eta,2\xi-\eta)} \frac{1}{\langle \xi-\eta \rangle \langle 2\xi-\eta\rangle \langle \xi-\eta -\zeta \rangle \langle \xi-\zeta \rangle} \\
& \mathfrak{f}^2(\xi,\eta,\zeta,\theta)= \frac{A(2\xi-\eta)}{\Phi_{\iota_1\iota_2}(\xi,\xi-\eta,2\xi-\eta)} \frac{1}{\langle \xi - \eta \rangle \langle 2 \xi - \eta\rangle \langle \xi-\eta-\zeta-\theta \rangle \langle \xi-\zeta \rangle}.
\end{align*}
We will now only focus on $\mathfrak{f}^1$, since $\mathfrak{f}^2$ can be treated nearly identically. Different signs $\iota_1,\iota_2$ cannot be treated identically; for the sake of brevity, we will only treat the most delicate case, namely $(\iota_1,\iota_2) = (+,+)$. Changing coordinates to $\alpha = -\xi +\eta$, $\beta =  2\xi - \eta$, $\gamma = \xi -\zeta$, and localizing dyadically, this becomes
$$
\mathfrak{g}^1(\alpha,\beta,\gamma)_{\underline{k}} = 
  \underbrace{\frac{A(\beta)}{\Phi_{\iota_1\iota_2}(\alpha + \beta,\alpha,\beta)} 
  \frac{1}{\langle \alpha \rangle \langle \beta \rangle \langle \gamma \rangle}
  \varphi_{k_1}^{(0)}(\alpha)\varphi_{k_2}^{(0)}(\beta)\varphi_{k_3}^{(0)}(\gamma)}_{
  \displaystyle := \mathfrak{h}^{1,1}_{\underline{k}}(\alpha,\beta,\gamma)} 
  \underbrace{ \frac{1}{\langle \gamma -2\alpha - \beta\rangle}\varphi_{k_4}^{(0)}(\gamma-2\alpha-\beta) }_{
  \displaystyle := \mathfrak{h}^{1,2}_{k_4}(\gamma -2\alpha - \beta)}.
$$
Finally, we need to distinguish cases depending on the signs of $\alpha$ and $\beta$; 
once again, we only consider the worst case, namely $\alpha, \beta>0$. 
By \eqref{p1est}, there holds, for all $a,b,c$,
$$
\left| \partial_\alpha^a \partial_\beta^b \partial_\gamma^c \mathfrak{h}^{1,1}(\alpha,\beta,\gamma) \right| 
  \lesssim 2^{-(1+a) k_1 - b k_2 - (1+c) k_3},
$$
therefore $\| \widehat{\mathfrak{h}^{1,1}_{\underline{k}}} \|_{L^1} \lesssim 2^{- k_1 - k_3}$. 
Since $\| \widehat{\mathfrak{h}^{1,2}_{\underline{k_4}}} \|_{L^1} \lesssim 2^{-k_4^+}$, 
we obtain 
that $\| \widehat{\mathfrak{h}^{1}_{\underline{k}}} \|_{L^1} \lesssim 2^{- k_1 - k_3-k_4}$. 
Applying Lemma \ref{lemmamultilin1} and summing over dyadic blocks gives the desired result \eqref{lemCS1}.
\eqref{lemCS2} follows in the same way.
Finally, using the linear dispersive estimate \eqref{dispinfty} instead of Mikhlin's
multiplier theorem, we obtain the endpoint estimate \eqref{lemCSend}.
\end{proof}

\smallskip
\begin{rem}[Derivatives of the cubic symbols]\label{lemCSrem}
In the estimates of Sections \ref{secLinfS} and \ref{secw'} we will
perform various integration by parts arguments in frequency space 
and will therefore end up differentiating the cubic symbols appearing in Lemma \ref{lemCS} above.
The estimates satisfied by the trilinear operators associated with these differentiated
symbols might vary from case to case, depending on the variables that are differentiated;
the localizations imposed in each specific case will determine how these estimates 
need to be modified by additional factors.
In any case, in all our arguments, the terms obtained when differentiating the symbols 
$\mathfrak{c}^{S1}$ and $\mathfrak{c}^{S2}$ will always give lower order contributions.
\end{rem}

%
%


\bigskip
\section{Bootstrap and basic a priori bounds}\label{secBoot}

In this section, we first give the details of our bootstrap strategy
as presented in Subsections \ref{Ssecmtpr0} and \ref{Ssecmtpr}, see \eqref{mtpr11}-\eqref{mtpr12}.
In particular, we close the bootstrap for the profile $g$, assuming the bootstrap for the 
renormalized profile $f$.
In Subsection \ref{ssecpre} we give some preliminary bounds on $f$ that will be useful in later sections.
In Subsection \ref{Ssecexp} we expand the nonlinear expressions in terms of $f$ and establish 
several bounds that do not require the analysis of oscillations.
Subsection \ref{Ssecmidsum} recalls the main equation for $\wt{f}$,
and lists all the estimates that are left to be proven in the remainder of the paper.

\subsection{Bootstrap strategy}\label{ssecBoot}
Recall from \eqref{initcond} that we are considering an initial data such that
\begin{align}\label{Bootinitcond0}
{\| (\jnab u_0,u_1) \|}_{H^4} + {\| \jx (\jnab u_0,u_1) \|}_{H^1} \leq \e_0.
\end{align}
From the definition of $v$ and $g$ in \eqref{vKG} and \eqref{vprof}, we see that $g_0 = u_1 - i \sqrt{H+1}u_0$.
Therefore, Proposition \ref{propFT} and Theorem \ref{thmweder} 
imply that
\begin{align}\label{Bootinitcond}
{\| \jxi^{4} \wt{g_0} \|}_{L^2} + {\| \jxi\partial_\xi \wt{g_0} \|}_{L^2} \lesssim \e_0.
\end{align}
From this and the interpolation inequality 
$|\varphi_k(\xi) \wt{h}(\xi)|^2 \lesssim {\| \varphi_k \wt{h} \|}_{L^2} {\| \partial_\xi \varphi_k \wt{h} \|}_{L^2}$,
we see that
\begin{align}\label{Bootinitcondinfty}
{\| \jxi^{3/2} \wt{g_0} \|}_{L^\infty} \lesssim \e_0.
\end{align}
According to the definition \eqref{Renof}-\eqref{RenoT} for the renormalized profile, we have
\begin{align*}
f(t=0) =: f_0 = g_0 - T(g_0,g_0)(t=0), 
\end{align*}
so that using \eqref{bilboundT'rev}, and estimating as in the proof of Lemma \ref{lemdxig} below,
see in particular \eqref{dxig2}, we have
\begin{align}\label{Bootinitcondf}
{\| \jxi^{4} \wt{f_0} \|}_{L^2} + {\| \jxi\partial_\xi \wt{f_0} \|}_{L^2} \lesssim \e_0.
\end{align}
Again by interpolation we obtain
\begin{align}\label{Bootinitcondfinfty}
{\| \jxi^{3/2} \wt{f_0} \|}_{L^\infty} \lesssim \e_0.
\end{align}

In what follows we consider $\e_1,\e_2$ 
satisfying
\begin{align}\label{epsilonboot}
\e_0 \ll \e_1 \ll \e_2, 
  \qquad \e_2 \leq \bar{\e_0} 
\end{align}
with $\bar{\e_0}$ sufficiently small. 
The main bootstrap estimate for $g$ is given by the following:
\begin{prop}\label{propbootg}
Assume that, for all $t\in[0,T]$,
\begin{align}\label{propbootgas}
\jt^{-p_0} {\big\| \jxi^{4} \wt{g}(t) \big\|}_{L^{2}} 
  + \jt^{1/2} {\| e^{-it\jnab} \mathbf{1}_{\pm}(D)\W g(t) \|}_{L^\infty} \leq 2\e_2.
\end{align}
Then, for all $t\in[0,T]$,
\begin{align}\label{propbootgconc}
\jt^{-p_0} {\big\| \jxi^{4} \wt{g}(t)\big\|}_{L^{2}}
  + \jt^{1/2} {\| e^{- it\jnab} \mathbf{1}_{\pm}(D)\W g(t) \|}_{L^\infty} \leq \e_2.
\end{align}
Moreover, we also have
\begin{align}\label{propbootginfty} 
{\| e^{-it\langle \wt D \rangle}g(t) \|}_{L^\infty} \lesssim \e_2 \jt^{-1/2} .
\end{align}
\end{prop}

Proposition \ref{propbootg} above implies global-in-time bounds on $g$ and 
$v= e^{it\langle \wt{D} \rangle}g(t)$, 
hence on the solution $u$ of \eqref{KG}, see \eqref{vu}; in particular,
together with \eqref{dxig}, it gives the 
the bounds \eqref{mtLinfty}, \eqref{mtbounds} and \eqref{mtwbound} stated in Theorem \ref{maintheo}.
However, since we cannot bootstrap directly bounds on norms of $g$ 
we reduce the proof of Proposition \ref{propbootg} 
to bootstrap estimates on the renormalized profile $f := g -T(g,g)$, 
see \eqref{Renof}-\eqref{RenoT}.
This is our main bootstrap proposition for $f$:
\begin{prop}\label{propbootf}
Assume that for all $t\in[0,T]$ we have
\begin{align}\label{propbootfas}
\jt^{-p_0} {\big\|  \jxi^{4} \wt{f}(t) \big\|}_{L^{2}} + {\| \jxi\partial_\xi \wt{f} \|}_{W_t} 
  + {\| \jxi^{3/2} \wt{f}(t) \|}_{L^\infty} \leq 2\e_1,
\end{align}
and that the bounds \eqref{propbootgas} on $g$ hold with $\e_2=\e_1^{2/3}$.
Then, for all $t \in [0,T]$,
\begin{align}\label{propbootfconc}
\jt^{-p_0} {\big\|  \jxi^{4} \wt{f}(t) \big\|}_{L^{2}} + {\| \jxi \partial_\xi \wt{f} \|}_{W_t} 
  + {\| \jxi^{3/2} \wt{f}(t) \|}_{L^\infty} \leq \e_1.
\end{align}
\end{prop}

The proof of Proposition \ref{propbootf} will occupy the rest of the paper, Sections \ref{secwR}-\ref{secw'}.
For now, we show how Proposition \ref{propbootf} implies Proposition \ref{propbootg}
by using the estimates on the operator $T$ 
from Lemma \ref{lemT}. First let us make the following remarks:

\begin{rem}
Note that the a priori assumptions \eqref{propbootfas}
and the linear dispersive estimates \eqref{disp2} and \eqref{disp1} imply
\begin{align}\label{propbootfdecay}
{\big\| e^{-it\langle \wt{D} \rangle}f(t) \big\|}_{L^{\infty}}
  + {\big\| e^{-it\jnab} \mathbf{1}_{\pm}(D)\mathcal{W}^*f(t) \big\|}_{L^{\infty}} \lesssim \e_1\jt^{-1/2}.
\end{align}

Also note that, in view of the conservation of the energy \eqref{KGVHam}, we have
that, for all times,
\begin{align}\label{L2cons}
{\| g(t) \|}_{L^2} + {\| f(t) \|}_{L^2} \leq \e_1.
\end{align}
The bound for $g$ follows from its definition, and the bound for $f$ can be deduced from $f=g-T(g,g)$,
the bilinear bound for $T$ in \eqref{bilboundT}, and the a priori assumptions \eqref{propbootgas}. 
\end{rem}

\begin{rem}\label{remfki}
For $\iota,\kappa \in \{+,-\}$,
\begin{align}\label{fki}
\widetilde{f}^{\kappa}_{\iota}(\xi) := \widetilde{f}_\iota (\xi) \mathbf{1}_{\kappa}(\xi),
\end{align}
enjoys the same bootstrap assumptions as $\wt{f}$, since $\wt{f}(0) = 0$; see Lemma \ref{lemwtf0}.
\end{rem}

\medskip
\begin{proof}[Proof of Proposition \ref{propbootg} assuming Proposition \ref{propbootf}]
Recall from \eqref{Renof} that
$g = f + T(g,g)$.
From this, using the bounds on the Sobolev-type norms
\begin{align*}
\jt^{-p_0}  {\big\| \jxi^{4} \wt{f}(t) \big\|}_{L^{2}} \leq \e_1, \qquad 
  \jt^{-p_0} {\big\|  \jxi^{4} \wt{g}(t) \big\|}_{L^{2}} \leq 2 \e_2,
\end{align*}
the bilinear bound \eqref{bilboundT'rev}, 
and the decay estimate from \eqref{propbootgas}, we get
\begin{align*}
{\big\|  \jxi^{4} \wt{g}(t) \big\|}_{L^{2}} & \leq  {\big\|  \jxi^{4} \wt{f}(t) \big\|}_{L^{2}} 
  + {\big\| \jxi^{4} \wt{\mathcal{F}}T(g,g)(t) \big\|}_{L^{2}}
  \leq \e_1\jt^{p_0} + {\big\| \mathcal{W}^*T(g,g)(t) \big\|}_{H^{4}}
\\
& \leq \e_1\jt^{p_0} + C {\big\| \mathcal{W}^*g(t) \big\|}_{H^{4}} {\| e^{-it\jnab} \mathbf{1}_{\pm}(D) \mathcal{W}^*g(t) \|}_{L^\infty}
\\
& \leq \e_1\jt^{p_0} + C\e_2 \jt^{p_0} \cdot \e_2 \jt^{-1/2}
\\ & \leq \e_1\jt^{p_0} + C\e_2^2.
\end{align*}
This gives the first bound in \eqref{propbootgconc}.

To estimate the $L^\infty_x$-norm in \eqref{propbootgconc}
we use successively the estimate \eqref{propbootfdecay}, Sobolev's embedding, 
and \eqref{bilboundT'rev} 
to get
\begin{align*}
\begin{split}
{\| e^{-it\jnab} \mathbf{1}_{\pm}(D) \mathcal{W}^*g \|}_{L^\infty}
&  \leq C\e_1\jt^{-1/2} + C {\| e^{-it\jnab} \mathbf{1}_{\pm}(D) \mathcal{W}^* T(g,g) \|}_{L^\infty}
\\
& \leq C\e_1\jt^{-1/2} + C {\| e^{-it\jnab} \mathcal{W}^* T(g,g)\|}_{W^{0+,\infty-}}
\\
& \leq C\e_1\jt^{-1/2} + C {\| e^{-it\jnab} \mathcal{W}^* g \|}_{L^{\infty-}}^2
\\ 
& \leq C\jt^{-1/2} (\e_1 + \e_2^2) \\ & \leq \jt^{-1/2} \e_2
\end{split}
\end{align*}
as desired.
We have used here the notation $\infty-$ to denote an arbitrarily large (but finite) number
(which may be different from line to line)
consistently with the notation introduced in \S\ref{secNotation}.


Finally, we show \eqref{propbootginfty}. Note that this does not follow at once from \eqref{propbootgconc}
since $\W$ is not necessarily bounded on $L^\infty$.
Observe that, 
by interpolation of \eqref{propbootfdecay} and \eqref{L2cons},
we have
\begin{align}\label{bootgpr4q}
{\| e^{-it\langle \wt D \rangle}f \|}_{L^q} 
  + {\| e^{-it\jnab}  \mathbf{1}_{\pm}(\partial_x)\mathcal{W}^*f \|}_{L^q} \leq C\e_1 \jt^{-1/2(1-2/q)}.
\end{align}
Therefore, for finite $q$, we have
\begin{align*}
\begin{split}
{\| e^{-it\jnab} \mathcal{W}^*g \|}_{L^q}
&  \leq C\e_1\jt^{-1/2(1-2/q)} + {\| e^{-it\jnab} \mathcal{W}^* T(g,g) \|}_{L^q}
\\
&  \leq C\e_1\jt^{-1/2(1-2/q)} + C {\| e^{-it\jnab} \mathcal{W}^* g \|}_{L^{2q}}^2
\\ 
& \leq C\jt^{-1/2(1-2/q)} (\e_1 + \e_2^2) \\ & \leq \jt^{-1/2(1-2/q)} \e_2.
\end{split}
\end{align*}
Using Gagliardo-Nirenberg interpolation, with the Sobolev-type norm bound in \eqref{propbootgas}, we obtain,
provided $q$ is large enough,
\begin{align}\label{bootgpr3}
{\| e^{-it\jnab} \mathcal{W}^*g \|}_{W^{1,q}} \lesssim  \jt^{-1/3} \e_2. 
\end{align}
Then we can estimate, using \eqref{propbootfdecay} and Sobolev's embedding,
\begin{align}\label{bootgpr5}
\begin{split}
\jt^{1/2}{\| e^{-it\langle \wt D \rangle} g \|}_{L^\infty} 
  & \leq C\e_1 + \jt^{1/2} \cdot C {\| e^{-it\langle \wt{D} \rangle} T(g,g) \|}_{L^\infty}
\\
& \leq C\e_1 + \jt^{1/2} \cdot C {\| e^{-it\langle \wt{D} \rangle} T(g,g) \|}_{W^{1,\infty-}}.
\end{split}
\end{align}
Using \eqref{bilboundT'rev} we have
\begin{align*}
{\| e^{-it\langle \wt{D} \rangle} T(g,g) \|}_{W^{1,\infty-}} 
& \lesssim {\| e^{-it\jnab} \mathcal{W}^*T(g,g) \|}_{L^{\infty-}}
  + {\| \langle \partial_x \rangle e^{-it\jnab} \mathcal{W}^*T(g,g) \|}_{L^{\infty-}}
  \\
 & \lesssim {\| e^{-it\jnab} \mathcal{W}^* g \|}_{W^{1,\infty-}}^2
  \lesssim \e_2^2 \jt^{-2/3}. 
\end{align*}
Plugging this into \eqref{bootgpr5} 
gives \eqref{propbootginfty} provided $\e_2$ is sufficiently small.
\end{proof}

\medskip
\subsection{Preliminary bounds}\label{ssecpre}
Recall that our main aim from now on is to prove Proposition \ref{propbootf}.
Therefore, we will work under the a priori assumptions \eqref{propbootfas} on $f$,
as well as the a priori assumptions \eqref{propbootgas} on $g$.
We collect below several bounds on $f$ that are immediate consequences of the a priori assumptions.

\begin{lem}\label{lembb}
Under the a priori assumptions \eqref{propbootfas}, for all $t\in[0,T]$ the following hold true:

\begin{itemize}

\medskip
\item[(i)] {\normalfont ({\it Basic bounds for $f$})} We have
\begin{align}
\label{lembb2}
& {\| \wt{f}(t) \|}_{L^2} + {\| \langle \xi \rangle^{3/2} \wt{f}(t) \|}_{L^\infty} \lesssim \e_1,
\\
\label{lembb1}
& {\| \jxi \partial_\xi \wt{f}(t) \|}_{L^2} \lesssim \e_1 \jt^{\alpha + \beta \gamma},
\\
\label{lembb3}
& {\| \chi_{\ell,\sqrt{3}} \partial_\xi \wt{f}(t) \|}_{L^1} \lesssim \e_1 2^{\beta' \ell} \jt^\alpha,
  \qquad \jt^{-\gamma} \leq 2^\ell \leq 1.
\end{align}

\medskip
\item[(ii)] {\normalfont ({\it Improved low frequency bounds})}
For all $k\leq-5$
 \begin{align}
& \label{apriori11bis}
{\| \varphi_{\leq k + 2}\partial_\xi \widetilde{f} \|}_{L^2} \leq \e_1 \jt^\alpha,
\\
\label{apriori11}
&{\| \varphi_k \wt{f}(t) \|}_{L^\infty_\xi} \lesssim \e_1 2^{k/2} \jt^\alpha,
\\ 
& \label{apriori12}
{\| \varphi_k \wt{f}(t) \|}_{L^1_\xi}   \lesssim \e_1 2^{3k/2} \jt^\alpha,
\end{align}
and for all $k\in\Z$
\begin{align}
\label{apriori13}
{\big\| \partial_\xi (\varphi_k \wt{f})(t) \big\|}_{L^1_\xi}
  & \lesssim \e_1 \min(2^{k/2}, 1) \jt^{\alpha}.
\end{align}

\medskip
 \item[(iii)] {\normalfont ({\it Linear dispersive estimates})}
 For all $t\in \R$ we have
\begin{align}
\label{aprioridecay}
{\big\| e^{-it \langle \wt{D} \rangle} f(t) \big\|}_{L^\infty}
  + {\big\| e^{-it \jnab}  \mathbf{1}_{\pm}(\partial_x) \mathcal{W}^*f(t) \big\|}_{L^\infty} 
  & \lesssim \e_1 \langle t \rangle^{-1/2}.
\end{align}
\end{itemize}
\end{lem}

\begin{proof}
\noindent
Proof of (i):
The first norm in \eqref{lembb2} is bounded in view of the conservation of the Hamiltonian, see \eqref{L2cons},
while the second is part of the a priori assumptions \eqref{propbootfas}.
\eqref{lembb1} follows from \eqref{propbootfas} and the definition of $W_t$ in \eqref{wnorm}
by summation over $\ell$ with $c\jt^{-\gamma} \leq 2^\ell \leq 1$.
For \eqref{lembb3} we apply the Cauchy-Schwarz inequality and the a priori bound on the $W_t$ norm
to estimate
\[ \| \chi_{\ell,\sqrt{3}} \partial_\xi \widetilde{f} \|_{L^1} 
  \lesssim 2^{\ell/2} \| \chi_{\ell,\sqrt{3}} \partial_\xi \widetilde{f} \|_{L^2}
  \lesssim \e_1 2^{\ell/2} 2^{-\beta \ell} \jt^\alpha 
  = \e_1 2^{\beta'\ell} \jt^\alpha.
\]

\medskip
\noindent
Proof of (ii): \eqref{apriori11bis} follows from the definition of the norm \eqref{wnorm}.
Since $\widetilde{f}(0)=0$, we have, for $k\leq -5$,
\begin{align}
| \varphi_{k}(\xi)\wt{f}(\xi) | & = \varphi_{k}(\xi) \Big| \int_0^\xi \partial_y \wt{f}(y)\,dy \Big|
  \leq \varphi_{k}(\xi) |\xi|^{1/2} {\| \varphi_{\leq k+2} \partial_y \wt{f} \|}_{L^2_y} 
  \lesssim 2^{k/2} \e_1 \jt^\alpha,
\end{align}
and
\begin{align}
{\| \varphi_k\wt{f} \|}_{L^1} & \lesssim 2^k {\| \varphi_k \wt{f} \|}_{L^\infty}  
\lesssim \e_1 2^{3k/2} \jt^\alpha.
\end{align}
The estimate \eqref{apriori13} follows from the a priori assumption on the weighted norm in \eqref{propbootfas},
and from \eqref{apriori11}-\eqref{apriori12} above
as long as $||\xi|-\sqrt{3}| \geq 1$, and it follows from \eqref{lembb3} when 
$||\xi|-\sqrt{3}| \leq 1$ (which implies $|k|\leq 5$).

\medskip
\noindent
Proof of (iii). These estimates follow directly from the linear dispersive estimate \eqref{disp1} and \eqref{disp2}
and the a priori bounds \eqref{propbootfas}.
\end{proof}

\medskip
We now prove a weak bound on the basic weighted norm of $g$.
This 
and the a priori bounds \eqref{propbootgas} will help us to estimate various remainders that 
come from expanding the nonlinear expressions in $g$, see the right-hand side of \eqref{Renodtf},
in terms of the renormalized profile $f$; see Subsection \ref{Ssecexp}.

\begin{lem}\label{lemdxig}
Under the a priori assumptions \eqref{propbootfas} and \eqref{propbootgas}, for all $t\in[0,T]$
\begin{align}
\label{dxig}
{\| \jxi \partial_\xi \wt{g} \|}_{L^2_\xi} \leq C\e_1 \jt^{1/2+p_0/2}.
\end{align}
\end{lem}

\begin{proof}
We obtain \eqref{dxig} through a bootstrap argument.
More precisely, assuming that for some $C$ large enough \eqref{dxig} holds,
it suffices to show the same inequality with $C/2$ instead of $C$.
In view of the formula $g=f+T(g,g)$ in \eqref{Renof}-\eqref{RenoT},
the bootstrap assumptions on $f$ (in particular the bound \eqref{lembb1},
with $C$ above chosen much larger than the implicit constant there)  it is enough to prove that
\begin{align}\label{dxig2}
{\| \jxi \partial_\xi \wt{T}(g,g) \|}_{L^2_\xi} & \lesssim \e_2^2 \jt^{1/2+p_0/2}
\end{align}
under the assumptions \eqref{dxig} and \eqref{propbootgas} (recall $\e_2 = \e_1^{2/3}$).

From the explicit formula \eqref{RenoT}, 
we see that
\begin{align}\label{dxig3}
\begin{split}
\jxi \partial_\xi \wt{\mathcal{F}} T^{\pm}_{\iota_1 \iota_2}(g,g) & = T_1(g,g) + T_2(g,g), 
\\
T_1(f_1,f_2) & := it \xi \, \wt{\mathcal{F}} T^{\pm}_{\iota_1 \iota_2}(f_1,f_2),
\\
T_2(f_1,f_2) & := \iint e^{it \Phi_{\iota_1 \iota_2}(\xi,\eta,\sigma)} \wt{f_1}(t,\eta) \wt{f_2}(t,\sigma) 
  \,\jxi \partial_\xi \mathfrak{m}_{\iota_1 \iota_2}^\pm(\xi,\eta,\sigma)\,d\eta \,d\sigma.
\end{split}
\end{align}
We need to analyze the formula for $\mathfrak{m}^\pm_{\iota_1\iota_2}$ from \eqref{RenoT}, \eqref{QZ}, \eqref{mucoeff}.
We can restrict our attention to the more complicated contribution involving the $\pv$,
since the $\delta$ part is easier to estimate.
This main contribution is
(we are dropping all the irrelevant signs, such as $\lambda,\mu,\nu$, and numerical constants from our notation)
\begin{align}\label{dxig5}
\begin{split}
& \mathfrak{m}_{\pv} := \frac{1}{\Phi_{\iota_1 \iota_2}(\xi,\eta,\sigma)} 	
  \cdot \frac{a(\xi, \eta,\sigma)}{\langle \eta \rangle \langle \sigma \rangle}  
  \varphi^\ast(p,\eta,\sigma) \, \pv \frac{\widehat{\phi}(p)}{ip}, \qquad 
  \\
& \varphi^\ast(\xi,\eta,p) = \varphi_{\leq -D_0}\big(p R(\eta,\sigma)) 
\qquad p := \lambda \xi - \iota_1 \mu \eta - \iota_2  \nu \sigma.
\end{split}
\end{align}

For $T_1$ 
we can use \eqref{bilboundT'rev}, the $L^\infty$ decay in \eqref{propbootgas}, 
and the interpolation of \eqref{L2cons} and the Sobolev bound in \eqref{propbootgas}
to obtain
\begin{align*}
{\| T_1[g,g](t) \|}_{L^2} 
&   \lesssim \jt {\| \jnab \mathcal{W}^* T[g,g](t) \|}_{L^2} 
  \lesssim \jt {\big\| e^{-it\jnab} \mathcal{W}^* g \big\|}_{L^{\infty}} {\| \jnab^{0+}g\|}_{L^2}
  \\ & \lesssim \jt \cdot \e_2 \jt^{-1/2} \cdot \e_2 \jt^{0+} \lesssim  \e_2^2 \jt^{1/2+}.
\end{align*}

To handle $T_2$ we need to look more closely at the formulas \eqref{RenoT} and \eqref{QZ} 
for $\mathfrak{m}^\pm_{\iota_1\iota_2}$.We apply $\partial_\xi$ and write the result as
\begin{align}\label{dxig6}
\begin{split}
& \jxi \partial_\xi \mathfrak{m}_{\pv} := \mathfrak{a} + \mathfrak{b},
\\
& \mathfrak{a} := \jxi \partial_\xi \Big[ \frac{1}{\Phi_{\iota_1 \iota_2}(\xi,\eta,\sigma)} 	
  \cdot \frac{a(\xi, \eta,\sigma)}{\langle \eta \rangle \langle \sigma \rangle}  
  \varphi^\ast(p,\eta,\sigma)\Big] \, \pv \frac{\widehat{\phi}(p)}{ip}
\\
& \mathfrak{b} := \jxi \frac{1}{\Phi_{\iota_1 \iota_2}(\xi,\eta,\sigma)}
  \cdot \frac{a(\xi,\eta,\sigma)}{\langle \eta \rangle \langle \sigma \rangle}  
  \varphi^\ast(p,\eta,\sigma)\, \partial_\xi \, \pv \frac{\widehat{\phi}(p)}{ip},
\end{split}
\end{align}
and, according to this, we define $T_\mathfrak{a}$ and $T_\mathfrak{b}$ similarly to $T_2$ in \eqref{dxig3}.

By the estimate \eqref{californiacondor1}, we deduce that $\mathfrak{a}$
is a symbol that behaves like (the $\pv$ contribution to) $\mathfrak{m}^\pm_{\iota_1\iota_2}$ times 
an extra factor of $\jxi \cdot R(\eta,\sigma)$. 
In practice, the factor of $R$ loses one derivative on the input with smaller frequency.
Using the H\"older bound from Lemma \ref{lemT},
estimating in $L^\infty$ the input with higher frequency and in $L^2$ the one with lower frequency,
we obtain
\begin{align*}
{\| T_\mathfrak{a}[g,g](t) \|}_{L^2} & 
  \lesssim {\big\| \jnab^{0+} e^{-it\jnab} \mathcal{W}^* g \big\|}_{L^{\infty}} 
  {\| \jnab \mathcal{W}^* g \|}_{L^2} \lesssim \e_2^2 \jt^{-1/2+p_0/2},
\end{align*}
having used interpolation of the Sobolev a priori bound \eqref{propbootgas} and \eqref{L2cons}
on both norms in the last inequality.

We now estimate the contribution involving $\mathfrak{b}$, assuming without loss of generality that $|\eta| \geq |\s|$. The idea is to use that $p = \lambda \xi - \iota_1 \mu \eta - \iota_2 \nu \sigma$
to convert $\partial_\xi$ into $\partial_\eta$ and integrate by parts in $\eta$;
this gives three types of terms: 
(1) a term where $\partial_\eta$ hits the profile $\wt{f}(\eta)$, 
(2) a term where $\partial_\eta$ hits $e^{it\Phi_{\iota_1\iota_2}}$
and (3) a term where its hits the rest of the symbol.
This last term is essentially the same as $\mathfrak{a}$ in \eqref{dxig6} 
(with $\partial_\eta$ replacing $\partial_\xi$ there) and can be handled identically,
so we skip it.
The contribution from (2) is of the same form as that of $T_1$ in \eqref{dxig3},
with $(\eta/\jeta)\wt{f_1}$ instead of $\wt{f_1}$, and therefore satisfies the same bound.
The remaining term is
\begin{align}\label{dxig10}
\iint e^{it \Phi_{\iota_1 \iota_2}(\xi,\eta,\sigma)}
  \partial_\eta \wt{g}(t,\eta) \wt{g}(t,\sigma) \, \jxi \mathfrak{m}_\pv(\xi,\eta,\sigma)\,d\eta \,d\sigma
\end{align}

This term is of the form $\jxi \wt{T}\big(\wtF^{-1} \partial_\eta \wt{g},g\big)$,
where the symbol is given by the $\pv$ part of the full symbol $\mathfrak{m}^{\pm}_{\iota_1\iota_2}$.
An application of Lemma \ref{lemT}, with the bounds \eqref{dxig} and \eqref{propbootgas}
gives the following upper bound: 
\begin{align*}
{\| \eqref{dxig10} \|}_{L^2} \lesssim {\big\| \whF^{-1} \jxi^{0+} \partial_\xi \wt{g} \big\|}_{L^2} 
   {\big\| \jnab^{0+} e^{-it\jnab} \mathcal{W}^* g \big\|}_{L^{\infty}}
  \lesssim \e_2 \jt^{1/2+p_0/2} \cdot \e_2 \jt^{-1/2+} \lesssim \e_2^2 \jt^{p_0/2+}.
\end{align*}
This concludes the estimate \eqref{dxig2}. 
\end{proof}

\medskip
\begin{rem}\label{remdxiT}
The argument in the proof of Lemma \ref{lemdxig} shows that 
we have the following schematic identity for the operator $\wt{T}$ in \eqref{RenoT}:
\begin{align}\label{dxiTid}
\jxi \partial_\xi \wt{T}(f_1,f_2) \approx t \cdot \jxi \wt{T}(f_1,f_2) 
  + \jxi \wt{T}\big(\wtF^{-1}\partial_\xi \wt{f_1},f_2\big).
\end{align}
This is the analogue of \eqref{dxiQRid} for $\mathcal{Q}^R$.
%
In particular, \eqref{dxiTid} implies, via Lemma \ref{lemT}, that
\begin{align}\label{dxiT2+infty}
\begin{split}
& {\| e^{-it\jnab} \whF^{-1} \jxi \partial_\xi \wt{T}(f_1,f_2) \|}_{L^2+L^{\infty-}}
  \\ 
  &  \lesssim \jt {\| e^{-it\jnab} \whF^{-1} \jxi \wt{T}(f_1,f_2)  \|}_{L^{\infty-}}
  + {\| e^{-it\jnab} \whF^{-1} \jxi \wt{T}\big(\wtF^{-1}\partial_\xi \wt{f_1},f_2\big) \|}_{L^2}
  \\
  &\lesssim \jt {\| e^{-it\jnab} \W T(f_1,f_2)  \|}_{W^{1,\infty-}}
  + {\| e^{-it\jnab} \W T\big(\wtF^{-1}\partial_\xi \wt{f_1},f_2\big) \|}_{H^1}
  \\
  & \lesssim \jt {\big\| e^{-it\jnab} \jnab^{0+}
  \mathcal{W}^* f_1 \big\|}_{L^{\infty-}}  {\big\| e^{-it\jnab} \jnab^{0+}\mathcal{W}^*f_2 \big\|}_{L^{\infty-}}
  \\ & \quad + {\big\| \jxi^{0+} \partial_\xi \wt{f_1} \big\|}_{L^2}
  {\big\| \jnab^{0+} e^{-it\jnab} 
  \mathcal{W}^*f_2 \big\|}_{L^{\infty-}}.
\end{split}
\end{align}


\end{rem}

\medskip
\subsection{Expansions of the nonlinear terms}\label{Ssecexp}
Our starting point to prove Proposition \ref{propbootf} is the equation \eqref{Renodtf}.
To obtain the desired bounds we first need to convert the nonlinear terms on the right-hand side of \eqref{Renodtf} 
into multilinear expressions which depend only on $f$, plus remainders which depend
on both $f$ and $g$ but have a higher degree of homogeneity (they are at least quartic terms)
and, therefore, are easier to bound.
This is done by expanding $g=f + T(g,g)$, see \eqref{Renof}-\eqref{RenoT}.
Thanks to the expansions below we will obtain leading order quadratic and cubic (and some quartic) 
terms that only depend on the renormalized $f$. 
For these leading orders we can use the stronger bootstrap assumptions \eqref{propbootfas}, 
but the analysis is still quite involved, and will occupy Sections \ref{secwR}-\ref{secw'}.
The higher order remainder terms involving both $f$ and $g$ are taken care of in Lemmas 
\ref{lemQRexp} and \ref{lemCSexp} below.

Recall the bracket notation introduced after \eqref{QR}.
The following Lemma gives an expansion for the regular quadratic terms.

\begin{lem}[Expansion of $\mathcal{Q}^R$]\label{lemQRexp}
Consider $\mathcal{Q}^R$ as defined in \eqref{QR}, and $T$ as in \eqref{Renof}-\eqref{RenoT}.
Under the a priori assumptions \eqref{propbootgas} and \eqref{propbootfas} we can write
\begin{align}\label{QRexp}
\begin{split}
\mathcal{Q}^R[g,g] & = \mathcal{Q}^R[f,f] + \mathcal{R}_1(f,g)
\\ & = \mathcal{Q}^R[f,f] + \mathcal{Q}^R[f,T(f,f)] + \mathcal{Q}^R[T(f,f),f] + \mathcal{R}_2(f,g)
\end{split}
\end{align}
with
\begin{align}\label{QRexprem}
\begin{split}
\jt^{-p_0}{\| \jxi^4 \mathcal{R}_1(f,g)(t) \|}_{L^2}
  + \jt^{-\alpha}{\| \jxi \partial_\xi \mathcal{R}_2(f,g)(t) \|}_{L^2}& \lesssim \e_2^2 \jt^{-1}.
\end{split}
\end{align}
\end{lem}

\smallskip
\begin{proof}
For any bilinear form $A$, using $g = f + T(g,g)$ we have $A(g,g)- A(f,f) =  A(f,T(g,g)) + A(T(g,g),g)$.
Thus, we see that the remainders in \eqref{QRexp} are given by
\begin{align}\label{QRexp0H}
\begin{split}
\mathcal{R}_1(f,g) & = \mathcal{Q}^R[f,T(g,g)]  + \mathcal{Q}^R[T(g,g),g],
\end{split}
\end{align}
and
\begin{align}\label{QRexp0}
\begin{split}
\mathcal{R}_2(f,g) & = \mathcal{Q}^R[f,T(g,g)-T(f,f)]  + \mathcal{Q}^R[T(g,g)-T(f,f),f] 
+ \mathcal{Q}^R[T(g,g),g-f]
\\
& = \mathcal{Q}^R[f,T(f,T(g,g))] + \mathcal{Q}^R[f,T(T(g,g),g)] 
\\ & \qquad + \mathcal{Q}^R[T(f,T(g,g)),f] + \mathcal{Q}^R[T(T(g,g),g),f] + \mathcal{Q}^R[T(g,g),T(g,g)].
\end{split}
\end{align}

Let us first show how to obtain the Sobolev type bound in \eqref{QRexprem}.
Since $f$ enjoys better estimates than $g$ it suffices to bound
\begin{align}\label{QRexpaim1}
{\| \jxi^4 \mathcal{Q}^R[T(g,g),g] \|}_{L^2} \lesssim \e_2^2 \jt^{p_0-1}.
\end{align}
From \eqref{lemQRsob}, we get 
\begin{align}\label{lemQRestapp2}
{\big\| \jxi^4 \mathcal{Q}^R[a,b](t) \big\|}_{L^2} 
  \lesssim {\| \jxi^{3+} \wt{a} \|}_{L^2} {\| e^{-i t\jnab} \mathcal{W}^*b \|}_{L^{\infty-}}
  + {\| e^{-i t\jnab} \mathcal{W}^*a \|}_{L^{\infty-}} {\| \jxi^{3+} \wt{b} \|}_{L^2},
\end{align}
Interpolating between \eqref{propbootgas} and \eqref{L2cons}, 
and using the bilinear bounds \eqref{bilboundT} and \eqref{bilboundT'rev}, we have
\begin{align*}
& {\| \jxi^4 \mathcal{Q}^R[T(g,g),g] \|}_{L^2} 
\\
&\qquad \lesssim {\| \jxi^{3+} \wt{g} \|}_{L^2} {\| e^{-it\jnab} \mathcal{W}^*T(g,g) \|}_{L^{\infty-}}
  + {\| e^{-it\jnab} \mathcal{W}^* g \|}_{L^{\infty-}} {\| \jxi^{3+} \wt{T}(g,g) \|}_{L^2}
\\
& \qquad\lesssim \e_2 \jt^{(3/4+)p_0} {\| e^{-it\jnab} \mathcal{W}^*g \|}_{L^{\infty-}}^2 
  + \e_2 \jt^{-1/2+} \cdot {\| e^{-it\jnab} \mathcal{W}^*g \|}_{L^{\infty-}} {\| \jxi^3 \wt{g} \|}_{L^2}
\\
&\qquad \lesssim \e_2^3 \jt^{(3/4+)p_0} \cdot \jt^{-1+}
\end{align*}
which is bounded by $\e_2^3 \jt^{p_0-1}$. 

We now show how to obtain the weighted bound in \eqref{QRexprem} 
for each of the terms on the right-hand side of \eqref{QRexp0}.
We are going to use the identity \eqref{dxiQRid}, which we restate here for ease of reference,
\begin{align}\label{dxiQRid'}
\jxi \partial_\xi \mathcal{Q}^R[f_1,f_2] \approx t \cdot \jxi \mathcal{Q}^R[f_1,f_2] 
  + \jxi \mathcal{Q}^R[\wtF^{-1}\partial_\xi \wt{f_1},f_2],
\end{align}
and the bilinear estimate \eqref{lemQRsob}.
The idea is that applying $\jxi \partial_\xi$ to the quartic 
expressions in \eqref{QRexp0} will cost at most a factor of $t$ as we see from \eqref{dxiQRid'};
then, estimating all the inputs 
in $L^{\infty-}$ will give a decaying factor of $\e_2\jt^{-1/2+}$ for each of them,
for a total gain of $\e_2^4\jt^{-2+}$, and this will suffice to obtain \eqref{QRexprem}.

Let us look more in detail at the term $\mathcal{Q}^R[T(g,g),T(g,g)]$; 
the other terms being similar or better since they contain at least one $f$.
According to \eqref{dxiQRid'}, 
we need to estimate 
\begin{align}\label{QRexp5}
t {\| \mathcal{Q}^R[ \wtF^{-1}\jxi \wt{T}(g,g),T(g,g)] \|}_{L^2}, 
  \qquad \mbox{and} \qquad {\| \jxi \mathcal{Q}^R[ \wtF^{-1}\partial_\xi \wt{T}(g,g),T(g,g)] \|}_{L^2}.
\end{align}
For the first term, we use \eqref{lemQRsob} followed by \eqref{bilboundT}:
\begin{align*}
& t {\| \jxi \mathcal{Q}^R[ T(g,g),T(g,g)](t) \|}_{L^2} 
\\
& \qquad\lesssim 
  \jt {\| \jnab^{0+} e^{-it\jnab} \mathcal{W}^*T(g,g) \|}_{L^{\infty-}} 
  {\| e^{-it\jnab} \mathcal{W}^*T(g,g) \|}_{L^{\infty-}}
\\
& \qquad \lesssim \jt 
  {\| e^{-it\jnab} \mathcal{W}^*g \|}_{L^{\infty-}}^4
  \lesssim \e_2^4 \jt^{-1+}.
\end{align*}
For the second term in \eqref{QRexp5}, we first estimate the first input: 
using \eqref{dxiT2+infty}, the a priori bounds, and \eqref{dxig}, give us
\begin{align}\label{dxiTinfty} 
\begin{split}
& {\big\| \jnab e^{-it\jnab} \mathcal{W}^* (\wtF^{-1} \partial_\xi \wt{T}(g,g))\big\|}_{L^2 + L^{\infty-}} 
  \\
  & \lesssim
  \jt {\big\| \jnab^{0+} e^{-it\jnab}\mathcal{W}^* g \big\|}_{L^{\infty-}}
  {\big\|  \jnab^{0+}e^{-it\jnab} \mathcal{W}^* g \big\|}_{L^{\infty-}}
  \\
  & + {\big\| \jxi^{0+} \partial_\xi \wt{g} \big\|}_{L^2}
  {\big\| \jnab^{0+} e^{-it\jnab} \mathbf{1}_{\pm}(D) \mathcal{W}^* g \big\|}_{L^\infty}
  \\
  & \lesssim \jt \cdot \e_2 \jt^{-1/2+} \cdot \e_2 \jt^{-1/2+} + \e_2 \jt^{1/2+p_0/2} \cdot \e_2 \jt^{-1/2}
  \lesssim \e_2^2 \jt^{p_0/2}.
\end{split}
\end{align}

Then, using \eqref{lemQRsob} with $p_2=p_4 = \infty-$ and $p_1=p_3=2+$ or $\infty-$, 
and \eqref{bilboundT'rev}, we have
\begin{align*}
& {\| \jxi \mathcal{Q}^R[ \wtF^{-1}\partial_\xi \wt{T}(g,g),T(g,g)](t) \|}_{L^2} 
  \\
  & \lesssim {\big\| \jnab e^{-it\jnab} \mathcal{W}^* (\wtF^{-1} \partial_\xi \wt{T}(g,g))\big\|}_{L^2 + L^{\infty-}} 
  \cdot {\| \jnab^{0+} e^{-it\jnab} \mathcal{W}^*T(g,g) \|}_{L^{\infty-}}
  \\
  & \lesssim \e_2^2 \jt^{p_0/2} \cdot \e_2^2 \jt^{-1+},
\end{align*}
which is sufficient since $p_0<\alpha$.

The remaining terms in \eqref{QRexp0} can be treated similarly, using the
estimates of Lemmas \ref{lemT} and \ref{lemQR},
see also the expressions for $\jxi\partial_\xi\mathcal{Q}^R$ and $\jxi\partial_\xi\wt{T}$,
in \eqref{dxiQRid}-\eqref{dxiQR2+infty} and \eqref{dxiTid}-\eqref{dxiT2+infty}, 
and the weighted bound \eqref{dxig} for $\jxi \partial_\xi \wt{g}$.
\end{proof}

\medskip
Here is a similar expansion for the cubic terms.

\begin{lem}[Expansion of $\mathcal{C}^{S}$]\label{lemCSexp}
Consider $\mathcal{C}^S$ defined in~\eqref{CubicS}.
Under the a priori assumptions \eqref{propbootgas} and \eqref{propbootfas} we have
\begin{align}\label{CSexp}
\begin{split}
\mathcal{C}^{S}[g,g,g] & = \mathcal{C}^{S}[f,f,f] \\ & + \mathcal{C}^{S}[T(f,f),f,f] + \mathcal{C}^{S}[f,T(f,f),f] 
	+ \mathcal{C}^{S}[f,f,T(f,f)] + \mathcal{R}_3(f,g)
\end{split}
\end{align}
with
\begin{align}\label{CSexprem}
{\| \jxi \partial_\xi \mathcal{R}_3(f,g)(t) \|}_{L^2} & \lesssim \e_2^3 \jt^{-1+\alpha}.
\end{align}
Moreover
\begin{align}\label{CSSob}
{\| \jxi^4 \mathcal{C}^{S}(g,g,g) \|}_{L^2} \lesssim \e_2^3 \jt^{-1+p_0}.
\end{align}
\end{lem}

\begin{proof}
We have
\begin{align}\label{CSexp0}
\begin{split}
\mathcal{R}_3(f,g) = \mathcal{C}^{S}[T(g,g),g,g] - \mathcal{C}^{S}[T(f,f),f,f] + \mathcal{C}^{S}[f,T(g,g),g] 
	- \mathcal{C}^{S}[f,T(f,f),f] \\ + \mathcal{C}^{S}[f,f,T(g,g)] -  \mathcal{C}^{S}[f,f,T(f,f)].
\end{split}
\end{align}
Let us analyze the first two terms, the others being similar, and write the difference as a sum of $5$-linear terms:
\begin{align}\label{CSexp1}
\begin{split}
\mathcal{C}^{S}[T(g,g),g,g] - \mathcal{C}^{S}[T(f,f),f,f] & = 
	\mathcal{C}^{S}[T(T(g,g),g),g,g] + \mathcal{C}^{S}[T(f,T(g,g)),g,g] 
	\\ & + \mathcal{C}^{S}[T(f,f),T(g,g),g] +  \mathcal{C}^{S}[T(f,f),f,T(g,g)].
\end{split}
\end{align}
The terms on the right-hand side of \eqref{CSexp1} are all $5$-linear convolution terms, with bounded and 
sufficiently regular symbols, 
where each entry, $f$ or $g$, satisfies a linear decay estimate at the rate of $\jt^{-1/2}$, see \eqref{propbootgas} 
and  \eqref{propbootfdecay}, and an $L^2$-weighted bound, see \eqref{dxig} and \eqref{lembb1}.
It suffices to look at the first term on the right-hand side of \eqref{CSexp1} -
the other terms are better since they contain a factor of $f$ which satisfies stronger assumptions -
and show that
\begin{align}\label{CSexpbound1}
{\| \jxi\partial_\xi \mathcal{C}^{S}[T(T(g,g),g),g,g] \|}_{L^2} \lesssim \e_1^3 \jt^{-1}.
\end{align}

Inspecting the formula for $\mathcal{C}^{S}$, we see that applying $\partial_\xi$ gives three types of terms: 
(1) a term where $\partial_\xi$ hits the exponential, which will cost a factor of $t$;
(2) terms where $\partial_\xi$ hits the symbol; 
(3) terms where $\partial_\xi$ hits $\delta$ or $\pv$.
In the terms (3) we can convert $\partial_\xi$ into $\partial_\eta$,
integrate by parts in $\eta$ and obtain terms like (1) and (2) above, 
plus terms where the derivatives hits one of the three inputs;
see the similar argument detailed in the proof of Lemma \ref{lemdxig}.

The terms (2) are lower order so we skip them. 
The main contribution comes from the terms of the type (1).
In the case of \eqref{CSexpbound1}, this gives a term whose $L^2$ norm can be bounded
using the trilinear estimate of Lemma \ref{lemCS} and the bilinear bounds for $T$ in Lemma \ref{lemT} as follows:
\begin{align*}
& \jt {\| \jxi \mathcal{C}^{S}[T(T(g,g),g),g,g] \|}_{L^2} 
\\
& \qquad \lesssim  \jt {\| \jnab^{0+} \mathcal{W}^*T(T(g,g),g) \|}_{L^2} 
  {\| \jnab^{0+} e^{-it\jnab} \mathcal{W}^*g \|}_{L^\infty}^2
\\
& \qquad \lesssim \jt 
  {\| \mathcal{W}^*T(g,g) \|}_{L^2} {\| e^{-it\jnab} \mathcal{W}^*g \|}_{L^\infty}  
  \cdot \e_2^2 \jt^{-1+}
\\
& \qquad \lesssim \jt {\| e^{-it\jnab} \mathcal{W}^*g \|}_{L^\infty}^2 {\| g \|}_{L^2} \cdot \e_2^2 \jt^{-1+} \lesssim 
  \e_2^5 \jt^{-1+}
\end{align*}
having  also used the a priori assumptions on $g$ \eqref{propbootgas} and \eqref{L2cons}.

Terms of the type (3) above are of the form
\begin{align}\label{CSexp5}
{\| \jxi \mathcal{C}^{S}[\wtF^{-1} \partial_\xi \wt{T}(T(g,g),g),g,g] \|}_{L^2} \qquad \mbox{and} \qquad
  {\| \jxi \mathcal{C}^{S}[T(T(g,g),g),\wtF^{-1}(\partial_\xi g),g] \|}_{L^2}
\end{align}
The second one is estimated directly using the weak weighted bound \eqref{dxig} for ${\|\partial_\xi \wt{g}\|}_{L^2}$,
and estimating the other $4$ terms in $L^\infty$ via Lemma \ref{lemCS} followed by Lemma \ref{lemT}: 
this gives a bound of $\e_2^5 \jt^{-3/2+}$.
The first term in \eqref{CSexp5} can handled similarly to the proof of Lemma \ref{lemQRexp} above.
In particular, iterating the identity \eqref{dxiTid} gives
\begin{align}\label{CSestdxiT}
{\| \jxi \partial_\xi \wt{T}(T(g,g),g) \|}_{L^2} \lesssim \e_2^3 \jt^{p_0/2}.
\end{align}
Then, up to faster decaying terms, we can 
use Lemma \ref{lemCS} to bound the $L^2$-norm of the first term in \eqref{CSexp5} by 
\begin{align*}
& C {\| \jxi^{0+} \partial_\xi \wt{T}(T(g,g),g) \|}_{L^2} {\| \jnab^{0+} e^{-it\jnab} \mathcal{W}^*g \|}_{L^\infty}^2
\lesssim  \e_2^3 \cdot \e_2^2 \jt^{-1+p_0/2+}
\end{align*}
which suffices for \eqref{CSexprem}.

The last estimate \eqref{CSSob} follows from a direct application of Lemma \ref{lemCS}
and the a priori bounds \eqref{propbootgas}.
\end{proof}

\medskip
Since we will need to look at iterations of Duhamel's formula,
it is also useful to establish some bounds for $\partial_t f$.

\begin{lem}[Estimates for $\partial_t f$]\label{lemdtf}
Let $f$ be the renormalized profile defined in \eqref{RenoT}-\eqref{Renof}.
Following the notation \eqref{Renodtf}-\eqref{CubicS}, we can write, under the a priori assumptions \eqref{propbootgas} and \eqref{propbootfas},
\begin{align}
\label{lemdtfdec}
\partial_t \widetilde{f} = \mathcal{C}^{S}(f,f,f) + \mathcal{R}(f,g),
\end{align}
where
\begin{align}\label{lemdtfR}
{\| \mathcal{R}(t) \|}_{L^2_\xi} \lesssim \e_2^3 \jt^{-3/2 + 2\alpha}.
\end{align}
In particular, we have
\begin{align}\label{dtfsimple} 
{\| e^{-it\jnab} \W \partial_t f \|}_{L^2+L^{\infty-}} \lesssim \e_2^3 \jt^{-3/2 + 2\alpha}
\end{align}
and
\begin{align}\label{dtfL2}
{\| \partial_t f \|}_{L^2} \lesssim \e_1^2 \jt^{-1}
\end{align}
\end{lem}

\begin{proof}
From \eqref{Renodtf}, we can write
\begin{align}\label{dtfpr1}
\begin{split}
\partial_t \widetilde{f} & = \mathcal{Q}^R(g,g) + \mathcal{C}^{S}(g,g,g) + \mathcal{C}^R(g,g,g)
  = \mathcal{C}^{S}(f,f,f) + \mathcal{R}(f,g)
\end{split}
\end{align}
with (recall the notation introduced after \eqref{QR})
\begin{align}\label{dtfpr2}
\begin{split}
\mathcal{R}(f,g) & := \mathcal{Q}^R[f,f]  + \mathcal{Q}^R[f,T(g,g)] +   \mathcal{Q}^R[T(g,g),g]+ \mathcal{C}^R(g,g,g)
  \\ & + \mathcal{C}^{S}[T(g,g),g,g] + \mathcal{C}^{S}[f,T(g,g),g] + \mathcal{C}^{S}[f,f,T(g,g)].
\end{split}
\end{align}

We estimate each of the terms above, with the exception of $\mathcal{Q}^R[f,f]$.
The treatment of this term is postponed to Subsection \ref{SsecLinfR},
where the desired bound is given in \eqref{QRLinftyxi}
(and proven using an argument from Section \ref{secwR}).

Recall the multilinear estimates of Lemmas \ref{lemT}, \ref{lemQR} and \ref{lemCS}.
For the second term on the right-hand side of \eqref{dtfpr2} we use \eqref{lemQRpq} 
followed by \eqref{bilboundT} 
and the a priori decay estimate \eqref{aprioridecay}, to obtain
\begin{align*}
{\| \mathcal{Q}^R[f,T(g,g)] \|}_{L^2} & \lesssim 
  {\big\| e^{-it\jnab} \mathcal{W}^*f \big\|}_{L^{\infty-}}
  {\big\| e^{-it\jnab} \mathcal{W}^*T(g,g) \big\|}_{L^{\infty-}}
  \\ & \lesssim {\big\| e^{-it\jnab} \mathcal{W}^*f \big\|}_{L^{\infty-}} 
  \cdot {\big\| e^{-it\jnab}\mathcal{W}^*g \big\|}_{L^{\infty-}}^2
  \\ & \lesssim \e_1 \jt^{-1/2} (\e_2 \jt^{-1/2+})^2, 
\end{align*}
which suffices for \eqref{lemdtfR}.
The third term on the right-hand side of \eqref{dtfpr2} can be estimated identically. 
For the fifth term we have
\begin{align*}
{\| \mathcal{C}^S[T(g,g),g,g] \|}_{L^2} & \lesssim 
  {\big\| e^{-it\jnab} \mathcal{W}^* T(g,g) \big\|}_{L^2} 
  {\big\| e^{-it\jnab} \mathbf{1}_{\pm}(D) \mathcal{W}^*g \big\|}_{L^\infty}^2
  \\ & \lesssim {\big\| e^{-it\jnab} \mathcal{W}^*g \big\|}_{L^2}
  {\big\| e^{-it\jnab} \mathbf{1}_{\pm}(D)\mathcal{W}^*g \big\|}_{L^\infty}^3 \lesssim \e_2^4 \jt^{-3/2}.
\end{align*}
The remaining two terms involving $\mathcal{C}^S$ can be estimated in the same way.
\end{proof}

\medskip
\subsection{Summary and remaining estimates}\label{Ssecmidsum}

Recall the equation \eqref{Renodtf} for the evolution of $\wt{f}$.
According to Lemmas \ref{lemQRexp} and \ref{lemCSexp} the right-hand side of \eqref{Renodtf}
can be expressed in terms of $\wt{f}$ itself, up to remainders of sufficiently high homogeneity 
(in $f$ and $g$), depending on the norms that one wants to bound.

For further reference we recall here, see \eqref{QRexp} and \eqref{CSexp}, that we can write
\begin{align}\label{eqdtf}
\begin{split}
\partial_t \wt{f} & = \mathcal{Q}^R[f,f] + \mathcal{C}^{S}[g,g,g] + \mathcal{R}_1(f,g)
\\ 
& = \mathcal{Q}^R[f,f] + \mathcal{Q}^R[f,T(f,f)] + \mathcal{Q}^R[T(f,f),f] + \mathcal{C}^{S}[f,f,f] 
\\ & + \mathcal{C}^{S}[T(f,f),f,f] + \mathcal{C}^{S}[f,T(f,f),f] 
  + \mathcal{C}^{S}[f,f,T(f,f)] + \mathcal{R}_2(f,g) + \mathcal{R}_3(f,g).
\end{split} 
\end{align}
Notice that, compared to \eqref{Renodtf}, here we are discarding the $\mathcal{C}^R$ terms, 
according to the discussion in Subsection \ref{SsecReno}.
The remainder terms $\mathcal{R}_1$, respectively $\mathcal{R}_2$ and $\mathcal{R}_3$,
decay sufficiently fast in the Sobolev-type, respectively weighted norm, 
so that they can be bounded by simply integrating in time the estimates 
\eqref{QRexprem}, respectively \eqref{QRexprem} and \eqref{CSexprem}.

We now list the terms that we still need to handle in order to conclude the proof of 
the main bootstrap Proposition \ref{propbootf}. 

\bigskip
\noindent
{\it Sobolev estimate}. 
In view of \eqref{eqdtf} and 
\eqref{CSSob},
only $\mathcal{Q}^R(f,f)$ remains to be bounded in the Sobolev type norm;
we do this in Subsection \ref{SsecSob}.

\bigskip
\noindent
{\it Weighted estimate}. So far, we have only taken care of higher order remainder terms,
which did not require any refined multilinear analysis. 
The estimate for the main terms, which are much more delicate, are distributed as follows:

\begin{itemize}
\item $\mathcal{Q}^R(f,f)$ is treated in Section \ref{secwR} for the main
  interacting frequencies, and in Subsection \ref{secQRother} for the rest of the interactions.

\smallskip
\item The terms $\mathcal{Q}^R[f,T(f,f)]$ and $\mathcal{Q}^R[T(f,f),f]$ are estimated in Subsection \ref{SsecLinfR}.

\smallskip
\item For $\mathcal{C}^{S}(f,f,f)$, see Section \ref{secwL} for the main
  interactions, and Subsection \ref{Cubother} for the other interactions.

\smallskip

\item The terms $\mathcal{C}^{S}[T(f,f),f,f]$, $\mathcal{C}^{S}[f,T(f,f),f]$ 
and $\mathcal{C}^{S}[f,f,T(f,f)]$, are estimated in Subsection \ref{SsecLinfR}.
\end{itemize}

\bigskip
\noindent 
{\it Distorted Fourier $L^\infty$-norm}.
We deal with the last piece of the bootstrap norm \eqref{propbootfas} as follows:

\begin{itemize}

\item Section \ref{secLinfS} contains the main part of the argument: we analyze the cubic terms 
of the form $\mathcal{C}^S(f,f,f)$, and derive an asymptotic expression for them as $t\rightarrow \infty$.
We first do this with formal stationary phase arguments in Subsection \ref{secha}.
The expressions obtained will lead to an ODE for $\partial_t\wt{f}$, which we show 
is Hamiltonian at leading order, and preserves $|\wt{f}(k)|^2 + |\wt{f}(-k)|^2$; 
see Subsection \ref{secModScatt}.
From this we derive a long-range scattering correction,
and estimates for the leading order terms in the $L^\infty_\xi$-type norm.
Then, in Subsection \ref{secra} we show how to rigorously justify the above asymptotics
and complete the control over the $L^\infty_\xi$ norm of the ``singular'' cubic terms.

\smallskip
\item The results in Subsection \ref{SsecLinfR} give us integrable-in-time decay for the $L^\infty_\xi$-norm of 
$\mathcal{Q}^R(f,f)$ and of all the other cubic and quartic order terms
on the right-hand side of \eqref{eqdtf}.
\end{itemize}


\bigskip
\section{Weighted estimates part I: the main ``regular'' interaction}\label{secwR}

The weighted estimates for the ``regular'' interactions 
are one of the most technical parts of the paper 
due to the presence of a fully coherent interaction at output frequencies $\pm \sqrt{3}$.
Our main goal is to show the following:

\begin{prop}\label{prowR}
Consider $u$ solution of \eqref{KG} such that 
the a priori assumptions \eqref{propbootfas} on the renormalized profile $f$ hold. 
The ``regular'' quadratic term $\mathcal{Q}^R=\mathcal{Q}^R(f,f)$, see \eqref{QR1}, satisfies
\begin{align}
\label{prowRest}
{\left\| \jxi \partial_\xi \int_{0}^t \mathcal{Q}^R[f,f](s,\xi) \,ds \right\|}_{W_T} \lesssim \e_1^2.
\end{align}
\end{prop}

\smallskip
After setting up the framework for the proof of \eqref{prowRest},
in the rest of this section we will focus on the main interactions within $\mathcal{Q}^R$,
which, using the notation from \eqref{QR1}, are those involving frequencies
\begin{align}\label{wRmain}
|\eta| + |\sigma| + |\jxi - 2| \ll 1.
\end{align}
We will leave the rest of the interactions, for example those with $|\eta|\approx 1$ or 
$|\xi| \not\approx \sqrt{3}$ for later; see Subsection \ref{secQRother}.

For ease of reference we recall the definition of the norm we are estimating, see \eqref{timedecomp}-\eqref{wnorm}:
\begin{align}
\label{wnormref}
{\| g \|}_{W_T} := \sup_{n \geq 0 
  } \sup_{\ell \in \mathbb{Z} \cap [\lfloor -\gamma n \rfloor,0]} 
  {\| \chi_{\ell,\sqrt{3}}^{[-\gamma n,0]}(\,\cdot\,) 
  \, \tau_n(t) \mathbf{1}_{[0,T]}(t) \, g(t,\cdot) \|}_{L^\infty_t L^2_\xi} 
  2^{\beta \ell} 2^{-\alpha n},
\end{align}
where our parameters satisfy $0 < \alpha, \beta,\gamma < 1/2 $ with
\begin{equation}
\label{wnormparam} 
\gamma \beta' < \alpha < \frac{1}{2}\beta', 
  \qquad \gamma' := \frac{1}{2} -\gamma < \beta' := \frac{1}{2} - \beta \ll 1.
\end{equation}
We also recall the a priori assumptions \eqref{propbootfas} that we will use throughout the proof:
\begin{align}\label{aprioriw}
\sup_{t \in [0,T]} \left[ {\| \jxi^{3/2} \widetilde{f}(t) \|}_{L^\infty}
  + \langle t \rangle^{-p_0} {\| \jxi^4 \wt{f} \|}_{L^2} 
  \right] 
  + {\| \langle \xi \rangle \partial_\xi \widetilde{f} \|}_{W_T} \leq 2 \e_1.
\end{align}

\medskip
\subsection{Setup and reductions}\label{ssecwR1}
In view of the definitions, 
we aim to show that for any integer $n = 0,1,\dots,
[\log_2(T+2)]+1$ and $\ell \in \Z$ we have, for $t \approx 2^n$,
\begin{align}
\label{wnormQ}
2^{-\alpha n} 2^{\beta \ell} {\Big\| \chi_{\ell,\sqrt{3}}^{[-\gamma n,0]}(\xi) \,
  \partial_\xi \int_{0}^t \mathcal{Q}^R(f,f)(s) \,ds \Big\|}_{L^2} \lesssim  \e_{1}^2.
\end{align}
Recall from \eqref{QR1} and Remarks \ref{Remkappas} and \ref{remfki} 
that we can effectively work with
\begin{align}\label{wproofQR1}
\mathcal{Q}^R(t,\xi) = 
\sum_{\substack{\iota_1,\iota_2 \in \{+,-\} \\ \kappa_0,\kappa_1, \kappa_2 \in \{+,-\}}}
  \mathcal{Q}^{R}_{\substack{\iota_1 \iota_2 \\ \kappa_0 \kappa_1 \kappa_2}}(t,\xi)
\end{align}
where
\begin{align}\label{wproofQR2}
\begin{split}
& \mathcal{Q}^{R}_{\substack{\iota_1 \iota_2 \\ \kappa_0 \kappa_1 \kappa_2}}[f,f](t,\xi)
  = -\iota_1\iota_2 \mathbf{1}_{\kappa_0}(\xi)
  \iint e^{it \Phi_{\iota_1\iota_2}(\xi,\eta,\s)} 
  \, \mathfrak{q}_{\substack{\iota_1 \iota_2 \\ \kappa_0 \kappa_1 \kappa_2}}(\xi,\eta,\s) \,
  \wt{f}_{\iota_1}^{\kappa_1}(t,\eta) \wt{f}_{\iota_2}^{\kappa_2}(t,\s) \, d\eta \, d\s,
\\
& \Phi_{\iota_1\iota_2}(\xi,\eta,\s) := \jxi - \iota_1 \jeta - \iota_2 \jsig,
\end{split}
\end{align}
and the symbols satisfy for any $a,b,c$
\begin{align}\label{symest0}
| \varphi_{k_1}(\eta) \varphi_{k_2}(\s)
  \partial_\xi^a \partial_\eta^b \partial_\s^c \, 
  \mathfrak{q}_{\substack{\iota_1 \iota_2 \\ \kappa_0 \kappa_1 \kappa_2}}(\xi,\eta,\s)| \lesssim 
  2^{-\max(k_1,k_2)} 2^{(a+b+c)\min(k_1,k_2)}.
\end{align}

\medskip
\noindent
{\it Notation convention for the indexes}. 
For notational simplicity, we will drop the superscripts $\kappa_j$ which play no role.
We will also drop the subscripts $\iota_1,\iota_2$ from the profiles $\widetilde{f}$,
since $\widetilde{f}_\iota^\kappa$ enjoys the same bootstrap bounds as $\widetilde{f}$.
We do keep the signs $\iota_1,\iota_2$ for the phases $\Phi_{\iota_1,\iota_2}$ as these do play a role in the estimates.
Also, recall that we are adopting the notation introduced after \eqref{QR}.

\medskip
When applying $\partial_\xi$ to $\mathcal{Q}_{\iota_1\iota_2}^R$, we can,  
by Lemma \ref{lemdxiQR}, omit the prefactor $\mathbf{1}_{\kappa_0}(\xi)$; 
furthermore, we only need to estimate the terms where $\partial_\xi$ hits the phase
as the terms where $\partial_\xi$ hits the symbols $\mathfrak{q}$ are much easier to treat.
In other words, we can consider that 
$$\jxi \partial_\xi \mathcal{Q}_{\iota_1\iota_2}^R(t,\xi) \approx \mathcal{I}_{\iota_1\iota_2}^R(t,\xi)$$
where
\begin{align}
\label{wproof5}
\begin{split}
& \mathcal{I}_{\iota_1\iota_2}^R(t,\xi) := t \xi
  \iint e^{it \Phi_{\iota_1\iota_2}(\xi,\eta,\s)} \, \mathfrak{q}(\xi,\eta,\s) \,
  \wt{f}(t,\eta) \wt{f}(t,\s) \, d\eta \, d\s,
\end{split}
\end{align}
and restrict all our attention to these terms.

In view of \eqref{wnormQ} and the definitions \eqref{timedecomp}-\eqref{chil'} it will suffice to show that
for $n=0,1,\dots, [\log_2(T+2)]+1$ and $t\approx 2^n$ ($t \in [0,T]$), we have
\begin{align}
\label{wproof5.0}
{\Big\| \varphi_{\leq -\gamma n}(|\xi|-\sqrt{3}) 
  \, \int_{0}^t \mathcal{I}_{\iota_1\iota_2}^R(s,\cdot)\,ds \Big\|}_{L^2_\xi} 
  \lesssim  \e_{1}^2 \, 2^{\alpha n}  2^{\beta \gamma n},
\end{align}
and that for all $n=1,\dots, [\log_2(T+2)]+1$, any $\ell \in \Z \cap (-\gamma n,0)$, and 
for any $m =0,1,\dots$, we have, for all $t\approx 2^n$,
\begin{align}\label{wproof5.1}
{\Big\|  
  \chi_{\ell ,\sqrt{3}}(\xi) \, \int_{0}^t \mathcal{I}_{\iota_1\iota_2}^R(s,\cdot)\, \tau_m(s) \,ds \Big\|}_{L^2_\xi} 
  \lesssim  \e_{1}^2 \, 2^{\alpha m}  2^{-\beta \ell},
\end{align}
where the functions $\tau_0,\tau_1,\dots$ in \eqref{wproof5.1}
are a partition of the interval $[0,t]$, with properties as in \eqref{timedecomp}.


We begin with a reduction of the main bounds \eqref{wproof5.0}-\eqref{wproof5.1}
to estimates for each fixed $m$.

\begin{lem}\label{lemred0}
To prove \eqref{wproof5}-\eqref{wproof5.1} it suffices to show the following three inequalities:

\begin{subequations}\label{lemredeqs}

\begin{itemize}
 \item[(1)] For all $m=0,1,\dots$
\begin{align}\label{lemred1}
\begin{split} 
{\Big\| \varphi_{\leq \ell_0}(|\xi|-\sqrt{3}) 
  \, \int_{0}^t \mathcal{I}_{\iota_1\iota_2}^R(s) \, \tau_m(s) \,ds \Big\|}_{L^2_\xi} 
  \lesssim  \e_1^2 \, 2^{\alpha m}  2^{\beta \gamma m}, \qquad \ell_0 := -\frac{1}{2}m - 3\beta'm;
\end{split}
\end{align}

\item[(2)] For all $m=1,2,\dots$, and $\ell \in (\ell_0, -\gamma m] \cap \Z$
\begin{align}\label{lemred2}
\begin{split}
{\Big\| \varphi_{\ell}(|\xi|-\sqrt{3})  
  \, \int_{0}^t \mathcal{I}_{\iota_1\iota_2}^R(s)\, \tau_m(s) \,ds \Big\|}_{L^2_\xi} 
  \lesssim  \e_1^2 \, 2^{\alpha m}  2^{-\beta\ell} \cdot 2^{-2\beta'm};
\end{split}
\end{align}

\item[(3)] For all $m=1,2,\dots$, and $\ell \in (-\gamma m,0] \cap \Z$

\begin{align}\label{lemred3}
{\Big\| \varphi_{\ell}(|\xi|-\sqrt{3}) 
  \, \int_{0}^t \mathcal{I}_{\iota_1\iota_2}^R(s)\, \tau_m(s) \,ds \Big\|}_{L^2_\xi} 
  \lesssim  \e_1^2 \, 2^{\alpha m}  2^{-\beta\ell}.
\end{align}

\end{itemize}

\end{subequations}
 
\end{lem}

\begin{proof}
Let us first show how \eqref{lemred1}-\eqref{lemred2} imply \eqref{wproof5.0}.
For all $n=0,1,\dots$ we estimate 
\begin{align}
\nonumber
& {\Big\|  \varphi_{\leq -\gamma n}(|\xi|-\sqrt{3}) 
  \, \int_{0}^t \mathcal{I}_{\iota_1\iota_2}^R(s)\,ds \Big\|}_{L^2_\xi} 
  \leq  \sum_{0\leq m \leq n} {\Big\|  \varphi_{\leq -\gamma m}(|\xi|-\sqrt{3}) 
  \, \int_{0}^t \mathcal{I}_{\iota_1\iota_2}^R(s)\, \tau_m(s) \,ds \Big\|}_{L^2_\xi} 
  \\ \label{prlemred1}
  & \qquad \leq \sum_{0\leq m \leq n}
  {\Big\|  \varphi_{\leq \ell_0}(|\xi|-\sqrt{3}) 
  \, \int_{0}^t \mathcal{I}_{\iota_1\iota_2}^R(s)\, \tau_m(s) \,ds \Big\|}_{L^2_\xi} 
  \\ \label{prlemred2}
  & \qquad  + \sum_{0\leq m \leq n} \sum_{\ell_0 < \ell \leq -\gamma m}
  {\Big\|  \varphi_{\ell}(|\xi|-\sqrt{3}) 
  \, \int_{0}^t \mathcal{I}_{\iota_1\iota_2}^R(s)\, \tau_m(s) \,ds \Big\|}_{L^2_\xi}.
\end{align}
The inequality \eqref{lemred1} takes care directly of \eqref{prlemred1}
giving a bound of $\e_{1}^2 \, 2^{\alpha n}  2^{\beta \gamma n}$ as desired.
For \eqref{prlemred2} we use \eqref{lemred2} to obtain
\begin{align*}
  \eqref{prlemred2} \lesssim \e_1^2 \sum_{0\leq m \leq n} 
  \, 2^{\alpha m}  2^{-\beta\ell_0} \cdot 2^{-2\beta'm}  
  = \e_1^2 \, 2^{\alpha n}  2^{\beta(1/2 + 3\beta')n} \cdot 2^{-2\beta'n}  
  \lesssim \e_{1}^2 \, 2^{\alpha n}  2^{\beta \gamma n},
\end{align*}
where the last inequality follows from 
$\beta(1/2 + 3\beta')-2\beta' = \beta/2 + \beta'(3\beta-2)
\leq \beta/2 - \beta\gamma' = \beta \gamma$, 
see \eqref{wnormparam}.

Next, observe that the inequalities \eqref{lemred2} and \eqref{lemred3} 
directly imply \eqref{wproof5.1} when $\ell \in (\ell_0,0] \cap \Z$.
When $-\gamma n < \ell \leq \ell_0$, \eqref{lemred1} gives
\begin{align*}
  {\Big\|  \varphi_{\ell}(|\xi|-\sqrt{3}) 
  \, \int_{0}^t \mathcal{I}_{\iota_1\iota_2}^R(s)\, \tau_m(s) \,ds \Big\|}_{L^2_\xi} 
  \lesssim \e_{1}^2 \, 2^{\alpha m}  2^{\beta \gamma m}
  \lesssim \e_{1}^2 \, 2^{\alpha m}  2^{-\beta \ell},
\end{align*}
since $\gamma < 1/2 < \frac{1}{2} + 3 \beta'$.
\end{proof}

\medskip
\subsection{Proof of \eqref{lemred1}}\label{ssecwR2}
For any function $c$, $m=0,1,\dots$ and $k\leq0$ we define
\begin{align}
\label{wproofnot2.0}
X_{k,m}(c) := \min\Big( {\| \varphi_k \wt{c} \|}_{L^1}, 
  2^{-m-k} \big({\|  \partial_\xi(\varphi_k\wt{c}) \|}_{L^1} 
  + 2^{-k}{\| \varphi_{[k-5,k+5]} \wt{c} \|}_{L^1}\big) \Big). 
\end{align}
A more general variant of this quantity will appear in \eqref{Xvar} when we will 
also include the treatment of input frequencies $\gtrsim 1$.
Note that, in view of the a priori assumptions \eqref{aprioriw},
and the consequent bounds \eqref{apriori12}-\eqref{apriori13}, for the profile $f$ we have, for $k<0$,
\begin{align}\label{wproofnot2}
\begin{split}
& X_{k,m} := X_{k,m}\big( f(t)\tau_m(t) \big) \lesssim \e_1 \min\big(2^{3k/2}, 2^{-m - k/2}\big) 2^{\alpha m}, 
\\
& \sum_{k<0} X_{k,m} \lesssim \e_1 2^{-3m/4} 2^{\alpha m}.
\end{split}
\end{align}

We have the following lemma.

\begin{lem}\label{lemwproof6}
Let $\mathcal{I}^R_{\iota_1\iota_2}$ be the term defined in \eqref{wproof5}.
Then, under the a priori assumptions \eqref{aprioriw}, we have
\begin{align}
\label{wproof6}
{\big\| \mathcal{I}^R_{\iota_1\iota_2}(s,\cdot) \big\|}_{L^\infty_\xi} 
  \lesssim \e_1^2 2^{-m/2 + 2\alpha m}, \qquad s \approx 2^m.
\end{align}
\end{lem}

\begin{proof}
The signs $(\iota_1\iota_2)$ are not relevant for this bound, 
so we drop them from our notation, and denote
$\mathcal{I}_{\iota_1\iota_2}^R$ simply as $\mathcal{I}$.
We look at the expression \eqref{wproof5} and decompose  dyadically
the frequencies $\eta$ and $\sigma$, estimating
\begin{align}
\label{wproof7}
\begin{split}
& \sup_{s\approx 2^m}| \mathcal{I}(s,\xi) | \lesssim  2^m \sum_{k_1,k_2} \sup_{s\approx 2^m} |I^{k_1,k_2}(s,\xi)|,
\\
& I^{k_1,k_2}(s,\xi) := I^{k_1,k_2}[f,f](s,\xi),
\\
& I^{k_1,k_2}[a,b](s,\xi) := \iint e^{is \Phi_{\iota_1\iota_2}(\xi,\eta,\s)} \, \mathfrak{q}(\xi,\eta,\s) \, 
  \varphi_{k_1}(\eta) \varphi_{k_2}(\sigma)\,\wt{a}(t,\eta) \wt{b}(t,\s) \, d\eta \, d\s.
\end{split}
\end{align}
Note that we are adopting the same notation used for $\mathcal{Q}^R$ (see below \eqref{QR})
for the above bilinear terms.
We claim that for any two functions $a,b$ 
we have
\begin{align}
\label{wproof6.0}
|I^{k_1,k_2}[a,b](s,\xi)| \lesssim X_{k_1,m}(a) \cdot X_{k_2,m}(b), \qquad s\approx 2^{m}.
\end{align}
Then, \eqref{wproof7}, \eqref{wproof6.0} and \eqref{wproofnot2} give the desired conclusion \eqref{wproof6}.

Let us prove \eqref{wproof6.0}. 
A first estimate is obtained by using $|\mathfrak{q}|\lesssim 1$:
\begin{align}\label{wproof6.0a}
\begin{split}
| I^{k_1,k_2}[a,b](s,\xi)| 
  & \lesssim {\| \varphi_{k_1}\wt{a} \|}_{L^1} {\| \varphi_{k_2}\wt{b} \|}_{L^1}. 
\end{split}
\end{align}
For our second estimate, we integrate by parts in $\eta$ to obtain
\begin{align}
\nonumber
& |I^{k_1,k_2}[a,b](s,\xi)| 
\\
\label{wproof6.0ibp} 
& \qquad \lesssim \frac{1}{s} \Big| \iint e^{is \Phi_{\iota_1\iota_2}(\xi,\eta,\s)} \,
  \partial_\eta \Big[ \frac{\jeta}{\eta} \mathfrak{q}(\xi,\eta,\s) \, 
  \varphi_{k_1}(\eta) \,\wt{a}(s,\eta) \Big] \varphi_{k_2}(\sigma) \wt{b}(s,\s) \, d\eta \, d\s \Big|.
\end{align}
We then have a few different contributions depending on the term upon which $\partial_\eta$ falls.
The term when $ \partial_\eta$ hits $\varphi_{k_1} \wt{a}$ 
is bounded by 
\begin{align}\label{wproof6.02}
\begin{split}
C 2^{-m} 2^{-k_1} {\| \partial_\eta( \varphi_{k_1} \wt{a}(s)) \|}_{L^1}
  {\| \varphi_{k_2}\, \wt{b}(s) \|}_{L^1},
\end{split}
\end{align}
and the term when $\partial_\eta$ hits the factor $1/\eta$ is upper bounded by 
\begin{align}\label{wproof6.03}
\begin{split}
C 2^{-m} 2^{-2k_1} {\|\varphi_{k_1}\wt{a}(s) \|}_{L^1}
  {\| \varphi_{k_2}\, \wt{b}(s) \|}_{L^1} .
  \end{split}
\end{align}
The term where $\partial_\eta$ hits $\mathfrak{q}$ is a lower order term
since $|\partial_\eta \mathfrak{q}|\lesssim 1$, so we can disregard it.

Finally, for our last estimate, we can integrate by parts in \eqref{wproof6.0ibp} also in the $\s$ variable. 
Arguing as above we obtain
\begin{align}\label{wproof6.0c}
\begin{split}
| I^{k_1,k_2}[a,b](s,\xi)| & \lesssim X_{k_1,m}(a)
  \\ & \times
  \big( 2^{-m} 2^{-k_2} {\| \partial_\eta( \varphi_{k_2} \wt{b}(s)) \|}_{L^1} 
  + 2^{-m} 2^{-2k_2} {\| \varphi_{k_2} \wt{b}(s)) \|}_{L^1} \big).
\end{split}
\end{align}
Putting together \eqref{wproof6.0a}, \eqref{wproof6.03} and \eqref{wproof6.0c} gives \eqref{wproof6.0}
and completes the proof.
\end{proof}

\medskip
As an immediate application of Lemma \ref{lemwproof6} we
complete the proof of \eqref{lemred1}.
Using  H\"older's and \eqref{wproof6}, recalling that $\ell_0 := -m/2 - 3\beta'm$, we have
\begin{align*}
2^{-\alpha m} 2^{-\beta\gamma m}
  {\Big\| \varphi_{\leq \ell_0}(|\xi|-\sqrt{3}) \,\int_{0}^t \mathcal{I}_{\iota_1\iota_2}^R\,\tau_m(s) ds \Big\|}_{L^2} 
  & \lesssim 2^{-\alpha m}  2^{-\beta\gamma m} 2^{\ell_0/2} \cdot 2^m \sup_{s\approx2^m} 
  {\big\| \mathcal{I}_{\iota_1\iota_2}^R(s)\big\|}_{L^\infty}
  \\
  & \lesssim 2^{-\beta\gamma m} \cdot 2^{-m/4 - (3/2)\beta'm} \cdot \e_1^2 2^{m/2} 2^{\alpha m} 
  \lesssim \e_1^2,
\end{align*}
since, by \eqref{wnormparam},
\[ -\beta \gamma + 1/4 -(3/2)\beta' + \alpha = \beta'/2+\gamma'/2 - \beta'\gamma' -(3/2)\beta' + \alpha 
\leq -\beta' /2+ \alpha \leq 0.\]
In view of the above estimate, and Lemma \ref{lemred0},
to show the desired bounds \eqref{wproof5.0}-\eqref{wproof5.1}, it remains to prove \eqref{lemred2}-\eqref{lemred3}.

\medskip
\subsection{Proof of \eqref{lemred2}-\eqref{lemred3}: Preliminary Decompositions}\label{secdecred}
We proceed with the proofs of \eqref{lemred2} and \eqref{lemred3} 
by looking at various sub-cases depending on the sizes of the modulation 
and frequencies relative to time. 
For the remaining of the section we assume furthermore that
\begin{align}\label{red1}
\ell \leq -7\beta' m.
\end{align}
We will deal with $\ell > -7\beta' m$ in Subsection \ref{secQRother}.

For notational convenience, we slightly redefine the time-cutoff function 
appearing in the expression \eqref{wproof5} for $\mathcal{I}_{\iota_1\iota_2}^R(t)$
to be $2^{-m} t \tau_m(t)$ (but still denote it with the same letter $\tau_m$)
so that we can estimate
\begin{align}
\label{wproofnot1.0}
& \Big| \chi_{\ell,\sqrt{3}} \int_{0}^t \mathcal{I}_{\iota_1\iota_2}^R(s)\,\tau_m(s) \, ds \Big| 
  \lesssim 2^m \sum_{p\geq p_0, \, k_1,k_2} \Big| \chi_{\ell,\sqrt{3}} \int_{0}^t I^{p,k_1,k_2}[f,f](s,\xi)\,\tau_m(s) \, ds \Big|
\end{align}
where, for any two functions $a,b$ we denote
\begin{align}
\label{wproofnot1}
\begin{split}
I^{p,k_1,k_2}_{\iota_1\iota_2}[a,b](t,\xi) & := \iint e^{it \Phi_{\iota_1\iota_2}(\xi,\eta,\s)} 
  \, \varphi_p^{(p_0)}\big(\Phi_{\iota_1\iota_2}(\xi,\eta,\s)\big) 
  \, \mathfrak{q}_{\iota_1\iota_2}(\xi,\eta,\s) \, 
  \\ & \qquad \times \varphi_{k_1}(\eta) \wt{a}_{\iota_1}(\eta) \, 
  \varphi_{k_2}(\sigma) \wt{b}_{\iota_2}(\s) \, d\eta \, d\s,
  \qquad p_0 := -m + \d m,
\\
\Phi_{\iota_1\iota_2}(\xi,\eta,\s) & := \jxi - \iota_1 \jeta - \iota_2 \jsig,
\end{split}
\end{align}
for some fixed $\delta \in (0,10^{-3})$. 
Note that we have inserted a localization $\varphi_p^{(p_0)}(\Phi_{\iota_1\iota_2})$ 
in the size of the phase; see the notation \eqref{cut0}-\eqref{cut1}.
Also note that the parameter $p_0$ here is not the same as the one appearing in the 
a priori estimates, e.g. in \eqref{propbootfas}; however this should not cause any confusion 
here, since the one in \eqref{wproofnot1} is the only $p_0$ that will appear in this section.

To better focus on the main interactions, for the remaining of this section we will assume in addition that 
\begin{align}\label{red2}
k_1,k_2 \leq -10,
\end{align}
see \eqref{wRmain}, and we will deal with the complementary case in Subsection \ref{secQRother}.
Note that \eqref{red2} and \eqref{red1} imply that $p\leq 10$.
Without loss of generality, we can also assume that
$$k_1 \geq k_2.$$
The a priori bound \eqref{apriori12} gives
\begin{align}\label{wproofdecest00}
\begin{split}
\Big| \chi_{\ell,\sqrt{3}} \int_{0}^t \mathcal{I}_{\iota_1\iota_2}^R(s)\,\tau_m(s) \, ds \Big| 
  & \lesssim 2^{2m}  \sum_{p,k_1,k_2} \sup_{s\approx 2^m} {\| \varphi_{k_1}\wt{f}(s) \|}_{L^1}  
  {\| \varphi_{k_2}\wt{f}(s) \|}_{L^1}
\\
  & \lesssim 2^{2m} \sum_{p,k_1,k_2} 2^{3k_1/2} 2^{\alpha m} \e_1 \cdot 2^{3k_2/2} 2^{\alpha m} \e_1.
\end{split}
\end{align}
Since there are at most $O(m)$ indexes $p$ (because $p_0\leq p \leq 10$),
if we take the sum in \eqref{wproofdecest00} over $k_2 < -2m$ or $k_1 < -2m/3$,
we obtain an upperbound of $C \e_1^2 2^{2\alpha m} m$,
which, also in view of \eqref{red1} and $\alpha<\beta'/2$, gives \eqref{lemred2}-\eqref{lemred3}.
We can then assume $k_2 \geq -2m$ and $k_1 \geq -2m/3$.

At this point we also restrict our estimates to the case 
\begin{align}\label{red3}
(\iota_1\iota_2)=(++) 
\end{align}
in \eqref{wproofnot1} and will deal with the other relatively simpler cases in Subsection \ref{secQRother}.
We drop the signs from the expression in \eqref{wproofnot1} by denoting
\begin{align}\label{wproofnot1++}
\begin{split}
I^{p,k_1,k_2}(t,\xi) := I^{p,k_1,k_2}_{++}[f,f](t,\xi) 
  & := \iint e^{it \Phi(\xi,\eta,\s)} \, \varphi_p^{(p_0)}\big(\Phi(\xi,\eta,\s)\big) 
  \, \mathfrak{q}(\xi,\eta,\s) \, 
  \\ & \qquad \times \varphi_{k_1}(\eta)\wt{f}(t,\eta) \, \varphi_{k_2}(\sigma) \wt{f}(t,\s) \, d\eta \, d\s,
\\
\Phi(\xi,\eta,\s) := \jxi - \jeta - \jsig. &
\end{split}
\end{align}
Note that, since $|\eta|,|\sigma| \leq 1/100$, we have
\begin{align}
 \label{wproof10}
& \Phi(\xi,\eta,\s) = \frac{\sqrt 3}{2} (|\xi| - \sqrt 3) 
  - \frac{1}{2} \eta^2 - \frac{1}{2} \sigma^2 + O( (|\xi|-\sqrt 3)^2 + \eta^4 + \sigma^4).
\end{align}
Moreover, on the support of the integrals \eqref{wproofnot1++} we have, when $p > p_0$,
\begin{align*}
|\Phi(\xi,\eta,\s)| \approx 2^p, \qquad |\eta| \approx 2^{k_1}, \qquad |\xi^2-3| \approx 2^\ell.
\end{align*}
Then, in particular,
\begin{align}
\label{wproofRel}
\left\{
\begin{array}{ll}
2^p \approx 2^\ell & \mbox{if $2^\ell \gg 2^{2k_1}$},
\\
2^p \approx 2^{2k_1} & \mbox{if $2^\ell \ll 2^{2k_1}$},
\\
2^p \lesssim 2^\ell & \mbox{if $2^\ell \approx 2^{2k_1}$}.
\end{array}
\right.
\end{align}
In the case $p=p_0$, we have $|\Phi| \lesssim 2^{-m+\delta m} \ll 2^\ell$ since 
$\ell > \ell_0 = -m/2 - 3\beta' m$ (see \eqref{lemred1}).

Summarizing the reductions above we have the following lemma:

\begin{lem}\label{lemred}
Let $I^{p,k_1,k_2}$ be as in \eqref{wproofnot1++}. To prove \eqref{lemred2}-\eqref{lemred3} 
for $(\iota_1\iota_2) = (++)$ 
it will suffice to show that for all $m=1,2,\dots$ 
\begin{align}\label{wproofdecest}
\begin{split}
2^m \left\| \chi_{\ell,\sqrt{3}}(\cdot)  \int_{0}^t I^{p,k_1,k_2}(\cdot,s) \, \tau_m(s)\, ds \right\|_{L^2_\xi} \lesssim 
  \e_1^2 \, 2^{-\beta\ell} 2^{-2\beta'm} 
\end{split}
\end{align}
for all
\begin{align}
\label{wproofdecpar}
\begin{split}
& -(1/2+3\beta') m =: \ell_0 < \ell \leq -7\beta' m,
\\ & -m+\delta m =: p_0\leq p \leq 0, 
\\ & -2m \leq k_2 \leq k_1 \leq -10, \qquad k_1\geq -2m/3. 
\end{split}
\end{align}
\end{lem}

Note that the quantity on the right-hand side of \eqref{wproofdecest}, with no $2^{\alpha m}$ factor,
also takes into consideration the summation over $k_1,k_2$ and $p$,
which is made of at most $O(m^3)$ terms. 
In several cases we will not need to use cancellations coming from the time integration,
and will prove the following stronger version of the bound \eqref{wproofdecest}-\eqref{wproofdecpar}:
\begin{align}\label{wproofdecest'}
2^m {\big\| \chi_{\ell,\sqrt{3}}(\xi) I^{p,k_1,k_2}(s,\xi) \big\|}_{L^2_\xi} 
  \lesssim \e_1^2 \, 2^{-m}  2^{-\beta \ell} 
  \cdot 2^{-2\beta'm}, \qquad \forall \, s \approx 2^m.
\end{align}

\smallskip
Let us now prove a general lemma which improves on Lemma \ref{lemwproof6}
and will help deal with several basic cases.

\begin{lem}\label{Lemma1}
With the definition \eqref{wproofnot1} (but omitting the signs $\iota_1,\iota_2$ for lighter notation)
and \eqref{wproofnot2.0}, 
we have, for all $s \approx 2^m$,
\begin{align}
\label{wproofclaim0}
\big| \chi_{\ell,\sqrt{3}} \, I^{p,k_1,k_2}[f,f](s,\xi) \big| \lesssim X_{k_1,m}(f) \cdot X_{k_2,m}(f),
\end{align}
and, in particular,
\begin{align}
\label{wproofclaim1}
{\big\| \chi_{\ell,\sqrt{3}} \, I^{p,k_1,k_2}[f,f](s) \big\|}_{L^2} 
  \lesssim 2^{\ell/2} \cdot X_{k_1,m}(f) \cdot X_{k_2,m}(f).
\end{align}
Furthermore,
\begin{align}
\label{wproofclaim2}
{\big\| \chi_{\ell,\sqrt{3}} \, I^{p,k_1,k_2}[f,f](s) \big\|}_{L^2} 
  \lesssim 2^{p-k_1/2} \cdot 2^{-m-k_1} \left[ {\| \partial_\xi [\varphi_{k_1} \wt{f} ] \|}_{L^2} 
  + 2^{-k_1} \| \varphi_{[k_1 - 5,k_1 + 5]} \widetilde{f} \|_{L^2} \right] X_{k_2,m}(f).
\end{align}
\end{lem}

As the proof below will show,
the estimates of Lemma \ref{Lemma1} hold for general expressions as in \eqref{wproofnot1}
with any combination of signs $(\iota_1,\iota_2)$ and not only for the expression in \eqref{wproofnot1++}.
In particular, we can use this result in Subsection \ref{secQRother} for the proof of \eqref{othestsigns}.

\smallskip
\begin{proof}[Proof of Lemma \ref{Lemma1}]

The bound \eqref{wproofclaim0} follows similarly to the bound \eqref{wproof6.0}, 
the only difference being the presence of the cutoff $\varphi_p^{(p_0)}(\Phi)$ 
in the definition of $I^{p,k_1,k_2}$, see \eqref{wproofnot1}, versus that of $ I^{k_1,k_2}$, see \eqref{wproof7}.
However, this is easily dealt with by observing that, for $|\eta|\approx 2^{k_1}$,
\begin{equation*}
\Big| \frac{1}{s \partial_{\eta}\Phi} \partial_{\eta} \varphi_p^{(p_0)}(\Phi) \Big| 
= \Big| \frac{1}{s \partial_{\eta}\Phi}  { \varphi_p^{(p_0)}}' (\Phi) \partial_\eta \Phi \Big| 
  \lesssim 2^{-m-p} \leq 2^{-m-p_0} = 2^{-\delta m},
\end{equation*}
so that hitting this additional cutoff gives lower order contributions,
and one can iterate the integration by parts in $\eta$ again.

\begin{rem}\label{rem1/Phi}
We will apply the above argument several times in what follows, and
treat as lower order remainders all those terms where derivatives in $\eta$ and $\s$ 
fall on an expression of the form $\chi(2^{-p}\Phi(\xi,\eta,\s))$ for some smooth $\chi$.
\end{rem}

\eqref{wproofclaim1} follows directly from Cauchy-Schwarz in $\xi$. 
Let us now prove \eqref{wproofclaim2}. Notice that we may assume $p \leq 2k_1-10$
for otherwise \eqref{wproofclaim1} already gives the desired inequality.
Indeed, if $p > 2k_1-10$ then we must have $2^\ell \lesssim 2^p$
and $2^{\ell/2} \lesssim 2^{p - k_1}$ so that using
\begin{align*}
X_{k_1,m} \lesssim 2^{-m-k_1/2} \left[ {\| \partial_\xi [\varphi_{k_1} \wt{f} ] \|}_{L^2}
  + 2^{-k_1} \| \varphi_{[k_1 - 5,k_1 + 5]} \widetilde{f} \|_{L^2} \right],
\end{align*} 
recall \eqref{wproofnot2.0}-\eqref{wproofnot2},
we get \eqref{wproofclaim2} from \eqref{wproofclaim1}.

We look at the integral \eqref{wproofnot1}
and begin with an integration by parts in $\eta$ obtaining a main contribution of 
\begin{align}\label{S-1}
\frac{1}{s} \iint e^{is \Phi(\xi,\eta,\s)} 
  \, \varphi_p^{(p_0)}\big(\Phi(\xi,\eta,\s)\big) \frac{\jeta}{\eta}
  \, \mathfrak{q}(\xi,\eta,\s) \, \partial_\eta\big[ \varphi_{k_1}(\eta) \wt{f}(s,\eta) \big] \, 
  \varphi_{k_2}(\sigma) \wt{f}(s,\s) \, d\eta \, d\s.
\end{align}
A lower order contribution comes from $\partial_\eta $ hitting the symbol $\mathfrak{q}$.
We can bound \eqref{S-1} by
\begin{align}
\label{S0}
\begin{split}
& C 2^{-m} 2^{-k_1} \int \,K(\xi,\eta) \, \big| \partial_\eta [\varphi_{k_1}(\eta) \wt{f}(s,\eta)] \big| \, d\eta,
\\
& K(\xi,\eta) := \varphi_{[k_1-2,k_1+2]}(\eta) \int \varphi_p^{(p_0)} 
  \big(\Phi(\xi,\eta,\s)\big) \varphi_{k_2}(\sigma) |\wt{f}(s,\s)| \, d\s.
\end{split}
\end{align}
We have
\begin{align*}
\begin{split}
\int K(\xi,\eta) \, d\eta 
\lesssim \int \Big( \int_{E_{k_1,p}} 
     \,d\eta \Big) 
    \varphi_{k_2}(\sigma) |\wt{f}(s,\s)| \, d\s
\end{split}
\end{align*}
where
\begin{align}\label{S1.0}
E_{k_1,p} := \{ \eta \in \R \,:\, |\eta| \approx 2^{k_1}, \, |-\jxi + \jeta +  \jsig| \approx 2^p\}.
\end{align}
Notice that for fixed $\xi$ and $\s$, the set $E_{k_1,p}$ is contained in at most two intervals of length $\approx 2^{p - k_1}$.
We can then estimate 
\begin{align}\label{S1}
& \sup_\xi \int K(\xi,\eta) \, d\eta \lesssim 2^{p-k_1} \| \varphi_{k_2} \widetilde{f} \|_{L^1}.
\end{align}
Similarly, we also have
\begin{align}
\label{S2}
& \sup_\eta \int K(\xi,\eta) \, d\xi \lesssim 2^p \| \varphi_{k_2} \widetilde{f} \|_{L^1}.
\end{align}
The first bound needed for \eqref{wproofclaim2} then follows from the definition \eqref{wproofnot2.0},
\eqref{S0}, \eqref{S1}-\eqref{S2} and Schur's test:
\begin{align*}
{\big\| \chi_{\ell,\sqrt{3}} \, I^{p,k_1,k_2}(s) \big\|}_{L^2} & \lesssim 
  2^{-m-k_1} {\Big\| 
  \int \,K(\xi,\eta) \, \big| \partial_\eta [\varphi_{k_1}(\eta) \wt{f}(s,\eta)] \big| \, d\eta \Big\|}_{L^2}
\\
& \lesssim 2^{-m-k_1}  \cdot 2^{p-k_1/2} 
\cdot {\| \partial_\eta [\varphi_{k_1} \wt{f}] \|}_{L^2} {\| \varphi_{k_2} \widetilde{f} \|}_{L^1}.
\end{align*}
To complete the proof of \eqref{wproofclaim2}, we integrate by parts also in $\sigma$ in \eqref{S-1} 
and then use Schur's test as above. 
\end{proof}

%
%
%
%

\medskip
Before proceeding, let us note that, as a corollary of Lemma \ref{Lemma1}, 
we may assume the two following inequalities on our parameters:
\begin{align}\label{l<ppar'}
\begin{split}
(\frac{1}{2}+\beta) \ell + \min(-m - k_1/2, 3k_1/2) + \min(-m - k_2/2, 3k_2/2) \geq -2m -(2\alpha+2\beta')m
\end{split}
\end{align}
and
\begin{align}\label{l<ppar''}
\begin{split}
\beta \ell + p-3k_1/2 + \min(-m - k_2/2, 3k_2/2) \geq -m -(2\alpha+2\beta')m.
\end{split}
\end{align}
Indeed, if \eqref{l<ppar'} does not hold, 
the bound \eqref{wproofdecest'} follows using \eqref{wproofclaim1}.
Similarly, if \eqref{l<ppar''} does not hold,
then we can use \eqref{wproofclaim2} to obtain \eqref{wproofdecest'}.

\medskip
We now proceed with the proof of \eqref{wproofdecest}-\eqref{wproofdecpar}, 
or the stronger \eqref{wproofdecest'} when possible.
We will analyze separately the following regions:
\begin{align}\label{regions}
\begin{split}
\\
& \mbox{Region 1 (Subsection \ref{Ssecpsmall}):} \qquad p \leq  -m/2-3\beta'm -10,
\\
\\
& \mbox{Region 2 (Subsection \ref{Ssecl>p}):} \qquad -m/2-3\beta'm - 10 \leq p \leq -m/3-10\beta'm, 
  \,\quad \ell \geq p+10,
\\
\\
& \mbox{Region 3 (Subsection \ref{Ssecl<p}):} \qquad -m/2-3\beta'm -10\leq p, 
  \qquad \ell \leq p+10,
\\
\\
& \mbox{Region 4 (Subsection \ref{Ssecl<2k_1}):} \qquad p \geq -m/3-10\beta'm, 
  \qquad \ell \geq p+10.
\\
\end{split}
\end{align}

\medskip
\subsection{Case $p \leq -m/2 - 3\beta' m-10$}\label{Ssecpsmall} 
In this region there is almost no oscillation in time $s$ and we prove \eqref{wproofdecest'}.
Since we are working under the assumptions $-m/2 - 3\beta' m \leq \ell \leq -10$ 
we have
\begin{align}\label{wproof12}
|\ell - 2k_1 | \leq 5. 
\end{align}
Applying \eqref{wproofclaim1} and \eqref{wproofclaim2}, we see that to obtain a bound consistent 
with \eqref{wproofdecest'} it suffices to show that
\begin{align*}
\min \big( 2^{\ell/2} 2^{-k_1/2}, \, 2^{p-3k_1/2} \big) \cdot 2^{(-3/4 + 2\alpha)m}
  \lesssim 2^{-m} 2^{-\beta\ell} 2^{-2\beta'm}.
\end{align*}
In view of \eqref{wproof12} it then suffices that 
\begin{align}
\label{wproof14}
\mbox{either} \qquad 2^{3k_1/2} \lesssim 2^{-m/4} 2^{-2\alpha m} 2^{-3\beta'm}
	\qquad \mbox{or} \qquad 2^{p -k_1/2} \lesssim 2^{-m/4} 2^{-2\alpha m} 2^{-3\beta'm}.
\end{align}
The verification of \eqref{wproof14} follows from $p \leq -m/2$.

\medskip
\subsection{Case $-m/2-3\beta' m -10 \leq p \leq -m/3-10\beta'm $, and $\ell \geq p+10$}\label{Ssecl>p} 
In this case, we also have $|\ell-2k_1|\leq 10$. 
Relying again on \eqref{wproofclaim2}, for \eqref{wproofdecest}
%
it suffices to prove that
\begin{equation}
\label{abeille}
2^{p-3k_1/2} 2^{2\alpha m} 2^{-3m/4} \lesssim 2^{-m} 2^{-\beta\ell} 2^{-2\beta' m}.
\end{equation}
We then consider two possibilities:
\begin{itemize}
\item[-] If we use that $p \leq 2k_1 +20$ and $|\ell-2k_1|\leq 10$,~\eqref{abeille} is implied by
\begin{equation}
\label{abeille1}
2^{3k_1/2} \lesssim 2^{-(1/4+2\alpha+3\beta')m}
  .
\end{equation}
\item[-] If we use that $p \leq - \frac{m}{3} - 10\beta' m$ and $|\ell-2k_1|\leq 10$,~\eqref{abeille} is implied by
\begin{equation}
\label{abeille2}
2^{-k_1/2} \lesssim 2^{m(1/12- 2\alpha + 7\beta')}  
  .
\end{equation}
\end{itemize}
Then, we observe that \eqref{abeille1} is satisfied if 
$k_1 \leq - \frac{m}{6} - \frac{2}{3}[2\alpha + 3\beta']m$,
while \eqref{abeille2} is satisfied if $k_1 \geq - \frac{m}{6} + (4\alpha - 14\beta')m$;
finally, we notice that these latter two inequalities cover all possible values of $k_1$ since $\alpha < \beta'/2$.

%

%
%

\medskip
\subsection{Case $p \geq -m/2 - 3\beta'm-10$, and $\ell \leq p+10$}\label{Ssecl<p}  
This case is more delicate than the previous ones. Moreover, many of the arguments 
that we will perform here will also be relevant in the last case in Subsection \ref{Ssecl<2k_1}.
In order to obtain a bound consistent with \eqref{wproofdecest} for this case,
it suffices to show
\begin{align}
\label{l<p1}
2^{\ell/2} \, {\Big\| \chi_{\ell,\sqrt{3}}(\xi) \int_{0}^t I^{p,k_1,k_2}(s,\xi) \, \tau_m(s)\, ds \Big\|}_{L^2} 
  \lesssim \e_1^2 \, 2^{-m} 2^{-3\beta'm},
\end{align}
for all
\begin{align}\label{l<ppar}
\begin{split}
& -m/2 - 2\beta'm -10\leq p, \quad -(1/2+3\beta')m \leq \ell \leq p + 10, 
\quad 
k_2 \leq k_1 \leq -10,  \quad k_1 \geq -2m/3.
\end{split}
\end{align}


\smallskip
\subsubsection*{{\bf Step 1}: Integration by parts in time}
The first step is to resort to integration by parts in $s$, using that $|\Phi| \approx 2^p \gtrsim 2^\ell$.
Let us denote
\begin{align}
\label{l<pulI}
\begin{split}
\underline{I}^{p,k_1,k_2}[g,h](s,\xi) := \iint e^{is \Phi} \, \frac{\varphi_p\big(\Phi\big)}{\Phi}
  \, \mathfrak{q}(\xi,\eta,\s) \, \varphi_{k_1}(\eta)\wt{g}(\eta) \, \varphi_{k_2}(\sigma) \wt{h}(\s) \, d\eta \, d\s,
\end{split} 
\end{align}
where we have dropped some of the dependence on the time $s$ and on the frequencies for ease of notation.
Note that we are writing $\underline{I}^{p,k_1,k_2}$ for a bilinear term similar to $I^{p,k_1,k_2}$
but there the symbol has an additional division by $\Phi$.

Integrating by parts in $s$,
\begin{align}
\label{l<pibp}
\Big| &  \chi_{\ell,\sqrt{3}}(\xi)\int_{0}^t I^{p,k_1,k_2}(s,\xi) \, \tau_m(s)\, ds \Big| 
	\lesssim  | J(t,\xi) | + | K(t,\xi) | + | L(t,\xi) |,
\end{align}
where
\begin{align}
\label{l<pibp1}
\begin{split}
J(t,\xi) := \chi_{\ell,\sqrt{3}}(\xi) \underline{I}^{p,k_1,k_2}[f,f](t,\xi)
	-  \chi_{\ell,\sqrt{3}}(\xi)  \underline{I}^{p,k_1,k_2}[f,f](0,\xi)
	\\ - \, \chi_{\ell,\sqrt{3}}(\xi) \int_{0}^t \underline{I}^{p,k_1,k_2}[f,f](s,\xi) \, \tfrac{d}{ds} \tau_m(s) \, ds,
\end{split}
\\
\label{l<pibp2}
& K(t,\xi) := \chi_{\ell,\sqrt{3}}(\xi) \int_{0}^t \underline{I}^{p,k_1,k_2}[\partial_sf,f](s,\xi) \, \tau_m(s) \, ds,
\\
\label{l<pibp3}
& L(t,\xi) := \chi_{\ell,\sqrt{3}}(\xi) \int_{0}^t \underline{I}^{p,k_1,k_2}[f,\partial_sf](s,\xi) \, \tau_m(s) \, ds. 
\end{align}
For \eqref{l<p1} it then suffices to prove
\begin{align}
\label{l<p2L2}
2^{\ell/2} \, {\big\| A(t,\cdot) \big\|}_{L^2} \lesssim \e_1^2 2^{-m} 2^{-3\beta'm}, \qquad A = J,K,L,
\end{align}
or the stronger
\begin{align}
\label{l<p2}
2^\ell \, \big| A(t,\xi) \big| \lesssim \e_1^2 2^{-m} 2^{-3\beta'm}, \qquad A = J,K,L.
\end{align}
In the proof we will look at $K$ and $L$ in various scenarios (while $J$ is easier and directly estimated)
depending on the size of $\ell$, $p$, $k_1$ and so on\dots 
We will also split them in various pieces along the argument. In most cases we are going to show that 
the contributions we obtain are bounded as in \eqref{l<p2},
while we are going to bound the $L^2$ norms as in \eqref{l<p2L2} only in Subsection \ref{Ssecl<2k_1}.

In view of \eqref{wproofclaim0} in Lemma \ref{Lemma1} and Remark \ref{rem1/Phi}, we see that
the operator defined in \eqref{l<pulI} satisfies the estimates
\begin{align}\label{lemulI1}
|\underline{I}^{p,k_1,k_2}[f,f](s,\xi)| \lesssim 2^{-p} \cdot X_{k_1,m} \cdot X_{k_2,m}, \qquad s\approx 2^m.
\end{align}



\smallskip
\subsubsection*{Estimate of \eqref{l<pibp1}}
$J$ is a boundary term and it is easy to deal with. It suffices to show
\begin{align}
\label{l<p3}
2^{\ell} \, \big| \chi_{\ell,\sqrt{3}}(\xi) \, \underline{I}^{p,k_1,k_2}[f,f](s,\xi) \big| 
  \lesssim \e_1^2 2^{-m} 2^{-3\beta'm},
\end{align}
for all $s\approx 2^m$.  From \eqref{lemulI1} we obtain the bound
\begin{align}
\label{l<p4}
2^{\ell} \, \big| \underline{I}^{p,k_1,k_2}[f,f](s,\xi) \big| 
  \lesssim 2^{\ell} \cdot 2^{-p} \cdot X_{k_1,m} \cdot X_{k_2,m} \lesssim \e_1^2 2^{(-3/2 + 2\alpha)m},
\end{align}
which is more than sufficient.

\smallskip
\subsubsection*{Estimate of \eqref{l<pibp2}}
For the other terms in \eqref{l<pibp} we need to expand $\partial_s \wt{f}$ and analyze the 
resulting quartic terms in more detail. 
We use the identity \eqref{lemdtfdec} from Lemma \ref{lemdtf} 
and write

\begin{align}
\label{l<p5.0}
K + L & = \sum_{\iota_1\iota_2\iota_3} K^{S1}_{\iota_1\iota_2\iota_3} + K^{S2}_{\iota_1\iota_2\iota_3}
  + L^{S1}_{\iota_1\iota_2\iota_3} + L^{S2}_{\iota_1\iota_2\iota_3}  + D^R,
\\
\label{l<p5}
K^{S1,2}_{\iota_1\iota_2\iota_3}(t,\xi) & := \chi_{\ell,\sqrt{3}}(\xi)\int_{0}^t  \underline{I}^{p,k_1,k_2}
  \big[ \wtF^{-1}\mathcal{C}^{S1,2}_{\iota_1\iota_2\iota_3}(f,f,f),f\big](s,\xi) \, \tau_m(s) \, ds,
\\
\label{l<p6}
L^{S1,2}_{\iota_1\iota_2\iota_3}(t,\xi) & := \chi_{\ell,\sqrt{3}}(\xi)\int_{0}^t  \underline{I}^{p,k_1,k_2}
  \big[f,\wtF^{-1}\mathcal{C}_{\iota_1\iota_2\iota_3}^{S1,2}(f,f,f) \big](s,\xi) \, \tau_m(s) \, ds,
\\
\label{l<p8}
D^R(t,\xi) & := \chi_{\ell,\sqrt{3}}(\xi) 
  \int_{0}^t  \Big( \underline{I}^{p,k_1,k_2}\big[\wtF^{-1}\mathcal{R},f\big[(s,\xi) 
  + \underline{I}^{p,k_1,k_2} \big[f,\wtF^{-1}\mathcal{R} \big] (s,\xi)\Big) \, \tau_m(s) \, ds.	
\end{align}
Notice that since we assume $k_2\leq k_1$, the expressions in \eqref{l<p5} and \eqref{l<p6} are not symmetric. 
We proceed to estimate \eqref{l<p5}-\eqref{l<p8}.

\smallskip
\subsubsection*{{\bf Step 2.1}: Estimate of $K^{S1}$ in \eqref{l<p5}}
In the formulas \eqref{formulacubiccoeff} and \eqref{defCS12} for $\mathcal{C}^{S1}$, 
observe that the signs $\lambda,\nu,\dots$ do not play any relevant role, so that we can omit them and 
write $K^{S1}_{\iota_1\iota_2\iota_3}$ as a term of the form
\begin{align}
\label{l<p10}
\begin{split}
K^{S1}_{\iota_1\iota_2\iota_3} = \int_{0}^t
  \iiiint e^{is \Psi_{\iota_1\iota_2\iota_3}(\xi,\rho,\zeta,\eta,\s) } 
  \, \frac{\varphi_p\big(\Phi(\xi,\eta,\s)\big)}{\Phi(\xi,\eta,\s)}
  \, \mathfrak{q}(\xi,\eta,\s,\rho,\zeta) \, \varphi_{k_1}(\eta)\varphi_{k_2}(\sigma) 
  \\ \times \wt{f_{\iota_1}}(\rho) \wt{f_{\iota_2}}(\zeta) \wt{f_{\iota_3}}(\eta-\rho-\zeta)
    \wt{f}(\s) \, d\eta \, d\zeta \, d\rho \, d\s \, \tau_m(s)ds,
\end{split}
\end{align}
where
\begin{align}\label{l<p10'}
\Psi_{\iota_1\iota_2\iota_3}(\xi,\rho,\zeta,\eta,\s) := 
  \jxi - \iota_1\langle\rho\rangle - \iota_2\langle\zeta\rangle - \iota_3\langle\eta-\rho-\zeta\rangle 
  - \jsig, \qquad \iota_1,\iota_2,\iota_3 \in\{+,-\},
\end{align}
and we slightly abuse notation by still denoting $\mathfrak{q}$ for the quartic symbol above,
obtained by `composing' the quadratic and cubic one.

We will sometimes denote the oscillating phase \eqref{l<p10'} just by $\Psi$ and omit the dependence on the 
signs $\iota_i$ of the profiles $f$, since these play no important role.

\begin{rem}
Note that $\Psi$ involves four input frequencies $(\rho,\zeta,\eta-\rho-\zeta,\s)$:
the first three of them 
are ``{\it correlated}'' while $\s$ is ``{\it uncorrelated}''.
In the following arguments we will always keep in mind this distinction and perform different estimates 
for the ``correlated'' frequencies and the ``uncorrelated'' ones. 
\end{rem}

We further decompose the integral over the frequencies in \eqref{l<p10} 
according to the sizes of $\rho,\zeta$ and $\eta-\rho-\zeta$ by defining
\begin{align}\label{l<p11}
\begin{split}
& 
  \underline{I}^{p,\underline{k}}(s,\xi) := \iiiint e^{is \Psi} \, \frac{\varphi_p\big(\Phi\big)}{\Phi}
  \, \mathfrak{q} \, \varphi_{\underline{k}}(\eta,\sigma,\rho,\zeta)
  \, \wt{f}(\rho)\wt{f}(\zeta) \wt{f}(\eta-\rho-\zeta)  \, \wt{f}(\s) \, d\eta \, d\zeta \, d\rho \, d\s, 
\\
& \varphi_{\underline{k}}(\eta,\sigma,\rho,\zeta) = 
  \varphi_{k_1}(\eta)\varphi_{k_2}(\sigma)\varphi_{k_3}(\rho)\varphi_{k_4}(\zeta)\varphi_{k_5}(\eta-\rho-\zeta).
\end{split}
\end{align}

Recall that we are aiming to obtain the bound \eqref{l<p2}.
Without loss of generality, we may assume that, on the support of \eqref{l<p11}, we have 
\begin{align}\label{l<porder0}
k_5 \leq k_4\leq k_3; 
\end{align}
for the moment, we also assume that $k_3 \leq -5$; see Remark \ref{remnewfreq} below for more on this.

We first dispose of all interactions with $k_5 \leq -3m$. 
In this case we can estimate all profiles $\wt{f}$ in $L^\infty$ and gain $2^{-3m}$ from integration
(recall also the notation for `$\med$' in \S\ref{secNotation}):
\begin{align}\label{-3m}
\begin{split}
& 2^{\ell} \big| \chi_{\ell,\sqrt{3}}(\xi)\underline{I}^{p,\underline{k}}(s,\xi) \big| 
  \\
  & \lesssim 2^{\ell} \, {\| \wt{f} \|}_{L^\infty}^4 
   \iiiint \frac{\varphi_p(\Phi)}{|\Phi|} 
   \varphi_{k_1}(\eta)\varphi_{k_2}(\sigma)\varphi_{k_3}(\rho)\varphi_{k_4}(\zeta)\varphi_{k_5}(\eta-\rho-\zeta)
   \, d\eta d\sigma d\rho d\zeta
\\
& \lesssim 2^{\ell} \cdot \e_1^4 \cdot 2^{-p} \cdot
  2^{k_2} 2^{\min(k_1,k_3,k_4)} 2^{\med(k_1,k_3,k_4)} 2^{k_5}
  \\ 
  & \lesssim \e_1^4 2^{-3m}, 
\end{split}
\end{align}
which is more than enough.

After treating these very small frequencies,
we are left with $O(m^3)$ choices for $k_3,k_4$ and $k_5$ in \eqref{l<p11}, 
and it suffices to show the slightly stronger bound
\begin{align}\label{l<p12}
2^{\ell} \big| \chi_{\ell,\sqrt{3}}(\xi)\underline{I}^{p,\underline{k}}(s,\xi) \big| 
\lesssim \e_1^2 2^{-2m} 2^{-5\beta'm},
\end{align}
for all $s \approx 2^m$, and for each $5$-tuple of frequencies $(k_1,k_2,k_3,k_4,k_5)$ with 
\begin{align}\label{l<porder}
|\max(k_1,k_3) - \med(k_1,k_3,k_4)|\leq 5, \quad k_5 \leq k_4\leq k_3 \leq -5.
\end{align}
The first restriction above comes from the fact that $\eta = \rho + \zeta + (\eta-\rho-\zeta)$
which forces $\max(|\eta|,|\rho|,|\zeta|,|\eta-\rho-\zeta|) \approx \max_2(|\eta|,|\rho|,|\zeta|,|\eta-\rho-\zeta|)$
(recall the notation for `$\max_2$' given towards the end of \S\ref{secNotation}),
so that, in view of \eqref{l<porder0}, i.e., $|\rho|\gtrsim |\zeta| \gtrsim |\eta-\rho-\zeta|$, 
we must have $\max(|\eta|,|\rho|) \approx \med(|\eta|,|\rho|,|\zeta|)$.

\begin{rem}\label{remnewfreq}
Concerning the restrictions \eqref{l<porder},
note that we can assume \eqref{l<porder0} 
without loss of generality, but that we are imposing the additional restriction $k_3 \leq -5$.
In particular, this means that we are not considering here 
the cases when the sizes of ``new input frequencies''
$(\rho,\zeta,\eta-\rho-\zeta)$ are (a) close to the bad frequency $\sqrt{3}$ or (b) going to infinity.
Both of these cases are actually easier to treat than the case of small frequencies
that we will concentrate on.

We will deal with the scenarios (a) and (b)
at the level of the (more complicated) quadratic and cubic interactions in Section \ref{secw'};
see in particular Subsection \ref{secQRother}, the discussion at the end of \S\ref{ssecothHF} about high frequencies,
and the estimates in \S\ref{ssecoth++l} where we deal with the bad frequencies
by relying on \eqref{Xest} to bound  the quantity $X_{k,m}(f)$.

Finally, recall that, under the assumed frequencies localization, the symbol $\mathfrak{q}$ in \eqref{l<p11}
is smooth, with $O(1)$ bounds on derivatives;
see \eqref{symest0} and \eqref{CubicS} with \eqref{formulacubiccoeff}-\eqref{defCS12}.
\end{rem}

Before proceeding with the proof of \eqref{l<p12} we discuss how to treat the 
oscillations in the ``uncorrelated'' variable $\sigma$.

\smallskip
\subsubsection*{Treatment of the uncorrelated variable $\sigma$ and a first basic bound}
Examining the definitions \eqref{l<p10}--\eqref{l<p11}, 
we see that the only oscillation involving the variable $\sigma$ is $e^{is\jsig}$.
To exploit these oscillations we integrate by parts 
in $\s$ when $k_2 \geq -m/2$ using \eqref{apriori13}, 
and otherwise estimate the profile $\varphi_{k_2}\wt{f}$ in $L^1_\s$ using \eqref{apriori12}.

More precisely, we first estimate
\begin{align}\label{l<p12.5'}
\begin{split}
2^{\ell} \big| \underline{I}^{p,\underline{k}}(s,\xi) \big| & 
  \lesssim 2^{\ell} \cdot 2^{-p}
  \cdot {\|\varphi_{k_3} \wt{f}\|}_{L^\infty} {\|\varphi_{k_4} \wt{f}\|}_{L^\infty} {\|\varphi_{k_5} \wt{f}\|}_{L^\infty}
  \\
  & 
  \times \iiint \Big( \int |\varphi_{k_2} \wt{f}(\sigma)| \,d\sigma \Big)
  \, \varphi_{k_1}(\eta) \varphi_{k_3}(\rho)\varphi_{k_4}(\zeta)\varphi_{k_5}(\eta-\rho-\zeta)
  \, d\eta d\rho d\zeta
  \\
  & \lesssim 
  {\|\varphi_{k_3} \wt{f}\|}_{L^\infty} {\|\varphi_{k_4} \wt{f}\|}_{L^\infty} {\|\varphi_{k_5} \wt{f}\|}_{L^\infty} 
  \\
& \times 2^{k_5} 2^{\min(k_1,k_3,k_4)} 2^{\med(k_1,k_3,k_4)} 
  \cdot {\|\varphi_{k_2} \wt{f}\|}_{L^1}.
\end{split}
\end{align}
Using \eqref{apriori11}-\eqref{apriori12}, and in view of 
$k_5\leq k_4 \leq k_3$ (see \eqref{l<porder}), we obtain
\begin{align}
\label{l<p13.1}
2^{\ell} \big| \underline{I}^{p,\underline{k}}(s,\xi) \big| \lesssim \e_1^3 2^{3\alpha m} 
  \cdot 2^{k_5+k_4 + \min(k_1,k_3)} \cdot 2^{(1/2)(k_3+k_4+k_5)} \cdot \| \varphi_{k_2} \widetilde{f} \|_{L^1}.
\end{align}

When $k_2 \geq -m/2$ we integrate by parts in $\sigma$ and write
\begin{align}\label{l<p12.5}
\begin{split}
\underline{I}^{p,\underline{k}} & = K_1 + K_2,
\\
K_1 & = 
  \iiiint e^{is \Psi} 
  \partial_\s \mathfrak{k}(\xi,\eta,\sigma,\rho,\zeta) \, \wt{f}(\rho)\wt{f}(\zeta) \wt{f}(\eta-\rho-\zeta)  \,
    \, \varphi_{\sim k_2}(\s)\wt{f}(\s)\, d\eta \, d\rho \, d\zeta \, d\s, 
\\
K_2 & = 
  \iiiint e^{is \Psi} 
  \mathfrak{k}(\xi,\eta,\s,\rho,\zeta) \, \wt{f}(\rho)\wt{f}(\zeta) \wt{f}(\eta-\rho-\zeta) \, 
  \partial_\s \big(\varphi_{\sim k_2}(\s)\wt{f}(\s) \big)\, d\eta \, d\rho \, d\zeta \,  d\s,
\end{split}
\end{align}
where we denoted $\varphi_{\sim k_2}(\s)=\varphi_{k_2}(\s) 2^{k_2} \jsig/\s$ 
a cutoff function with the same properties as $\varphi_{k_2}$ ($k_2 \leq 0)$,
and defined the symbol
\begin{align}\label{l<p12.6}
\begin{split}
& \mathfrak{k}(\xi,\eta,\s,\rho,\zeta) 
  := 2^{-k_2} s^{-1} \,\frac{\varphi_p\big(\Phi\big)}{\Phi} 
  \, \mathfrak{q}(\xi,\eta,\s,\rho,\zeta) \,\varphi_{k_1}(\eta)\varphi_{k_3}(\rho)\varphi_{k_4}(\zeta)\varphi_{k_5}(\eta-\rho-\zeta),
  \\
& |\mathfrak{k}(\xi,\eta,\s,\rho,\zeta)| \lesssim 2^{-m-p-k_2}
\end{split}
\end{align}  
for $s\approx 2^m$.
We then claim that $K_2$ is the main contribution in \eqref{l<p12.5}, 
while $K_1$ gives a term of the same form of $\underline{I}^{p,\underline{k}}$
but with a better symbol, that we can treat as a lower order term and disregard. 
To see this, notice that
\begin{align*}
|\partial_\s \mathfrak{k}| & = s^{-1} 2^{-k_2}
  \Big|\partial_\s \Big[ \frac{\varphi_p\big(\Phi\big)}{\Phi}
  \,\mathfrak{q}  \Big] \varphi_{k_1}(\eta)\varphi_{k_3}(\rho)\varphi_{k_4}(\zeta) \varphi_{k_5}(\eta-\rho-\zeta)\Big|
\\ & \lesssim 2^{-m} 2^{-k_2} \big[ 2^{-2p + k_2} + 2^{-p} \big]
  \lesssim 2^{-m/2 + 3\beta' m} \cdot 2^{-p},
\end{align*}
having used $p \geq -m/2 - 3\beta' m - 10$ and $k_2 \geq -m/2$.
In particular, we see that this bound is better than $O(2^{-p})$, 
which is the trivial bound for the symbol of \eqref{l<p11} used in \eqref{l<p12.5'}.

For the term $K_2$ in \eqref{l<p12.5} we can use \eqref{l<p12.6} and \eqref{apriori11}-\eqref{apriori13} to obtain
\begin{align}
\nonumber
& 2^{\ell} \big| K_2(s,\xi) \big| 
\\
\nonumber
& \lesssim 2^{\ell} 
  \cdot 2^{-p-m-k_2}
  \cdot {\|\varphi_{k_3} \wt{f}\|}_{L^\infty} {\|\varphi_{k_4}\wt{f}\|}_{L^\infty} {\|\varphi_{k_5}\wt{f}\|}_{L^\infty}
  2^{k_5+k_4 + \min(k_1,k_3)}
  \cdot {\|\partial_\s (\varphi_{\sim k_2} \wt{f}) \|}_{L^1}
\\
\label{l<p13.0}
& \lesssim \e_1^3 2^{3\alpha m} \cdot 2^{k_5+k_4 + \min(k_1,k_3)}
	\cdot 2^{(1/2)(k_3+k_4+k_5)} \cdot 2^{-m-k_2} \| \partial_\sigma(\varphi_{\sim k_2} 
	\widetilde{f} ) \|_{L^1}.
\end{align}

Putting together \eqref{l<p13.1} and \eqref{l<p13.0} we obtain the following bound:
\begin{align}
2^{\ell} \big| \underline{I}^{p,\underline{k}}(s,\xi) \big| 
\label{l<p13}
\lesssim \e_1^3 \cdot 2^{3\alpha m} 
  \cdot 2^{k_5+k_4 + \min(k_1,k_3)} \cdot 2^{(1/2)(k_3+k_4+k_5)} \cdot X_{k_2,m}.
\end{align}
With \eqref{l<p13} in hand we now proceed with the proof of \eqref{l<p12} subdividing it into two main cases. 
In what follows we fix $\delta \in (0,\alpha)$.

\smallskip
\noindent
{\it Case 1: $k_1 + k_4 \leq -m + \delta m$}.
This case corresponds to a scenario where integration by parts in the new ``correlated variables'',
that is in the directions $\partial_{\eta} + \partial_{\rho}$ and $\partial_\eta + \partial_\zeta$, is forbidden,
see also \eqref{l<p14.0}.
In this case, inequality \eqref{l<p13} suffices to get the desired bound by the right-hand side of \eqref{l<p12}.
Indeed,
using $X_{k_2,m} \lesssim \e_1 2^{-3m/4+\alpha m}$,
\eqref{l<p13} implies
\begin{equation}
\label{l<p13.5}
2^{\ell} \big| \underline{I}^{p,\underline{k}}(s,\xi) \big| 
  \lesssim \e_1^3 2^{(-3/4+4\alpha) m} \cdot 2^{k_1+3k_4},
\end{equation}
recall \eqref{l<porder}.
Since we are assuming $k_1+k_4 \leq -m+\delta m$, we must also have $k_4 \leq -m/3+\delta m$
as a consequence of the lower bound on $k_1$ in \eqref{l<ppar}.
Then $k_1+ 3k_4 \leq -3m/2$ and \eqref{l<p13.5} suffices for \eqref{l<p12}.


\smallskip
\noindent
{\it Case 2: $k_1 + k_4 \geq -m+\delta m$}.
In this case we can integrate
by parts in both the $\partial_\eta+\partial_\rho$ and $\partial_\eta+\partial_\zeta$ directions,
using that 
\begin{equation}
\label{l<p14.0}
\begin{array}{ll}
(\partial_\eta + \partial_\rho) \Psi_{\iota_1\iota_2} = -\iota_1 \dfrac{\rho}{\langle \rho \rangle},
  \qquad & |(\partial_\eta + \partial_\rho) \Psi_{\iota_1\iota_2}| \approx 2^{k_3},
\\
(\partial_\eta + \partial_\zeta) \Psi_{\iota_1\iota_2} = -\iota_2\dfrac{\zeta}{\langle \zeta \rangle}, 
  \qquad & |(\partial_\eta + \partial_\zeta) \Psi_{\iota_1\iota_2}|\approx 2^{k_4}.
\end{array}
\end{equation}
To properly implement this strategy we first need to pay attention to the cases when $k_4$ is small.


\smallskip \noindent
{\it Subcase 2.1: $k_4 \leq -m/2+\delta m$}.
In this case $k_5\leq k_4 \leq -m/2+\delta m$ and we can estimate directly using \eqref{l<p13}:
\begin{align*}
\begin{split}
2^{\ell} \big| \underline{I}^{p,\underline{k}}(s,\xi) \big| 
  & \lesssim \e_1^4 2^{(-3/4+4\alpha) m} \cdot 2^{k_4+k_5} \cdot 2^{(1/2)(k_4+k_5)}
  \lesssim \e_1^4 2^{(-3/4+4\alpha+3\delta) m} 2^{-3m/2}
\end{split}
\end{align*}
which is sufficient for \eqref{l<p12}.


\smallskip 
\noindent
{\it Subcase 2.2: $k_4 \geq -m/2 + \delta m$}. 
In this case we have $k_3,k_4 \geq -m/2+\delta m$
and, see \eqref{l<p14.0}, we can integrate by parts in both $\partial_\eta+\partial_\rho$ and $\partial_\eta+\partial_\zeta$,
using also that $k_1 +k_4 \geq -m+\delta m$.
Performing these integrations by parts, we see that
\begin{equation}
\label{l<p15}
\begin{split}
\big| \underline{I}^{p,\underline{k}}(s,\xi) \big| \lesssim 2^{-2m} 
  \sup_{s\approx 2^m} \Big| \iiiint e^{is \Psi} \, 
  (\partial_\eta + \partial_\zeta) \Big[ \frac{1}{(\partial_\eta+\partial_\zeta) \Psi} 
  (\partial_\eta + \partial_\rho) \Big( \frac{1}{(\partial_\eta + \partial_\rho) \Psi}
\\
  \times \frac{\varphi_p\big(\Phi\big)}{\Phi}  \mathfrak{q} 
  \varphi_{\underline{k}} 
  \wt{f}(\rho) \wt{f}(\zeta) \wt{f}(\eta-\rho-\zeta)\Big) \Big] 
  \, d\eta \, d\rho \, d\zeta\,  \wt{f}(\s) \, d\s \Big|.
\end{split}
\end{equation}
The expression in \eqref{l<p15} gives many different contributions, 
depending on which terms are hit by the derivatives $\partial_\eta +\partial_\rho$ and $\partial_\eta +\partial_\zeta$.
By distributing these derivatives we see that
\begin{equation}
\label{l<p15.0}
\begin{split}
\big| \underline{I}^{p,\underline{k}}(s,\xi) \big| 
  & \lesssim 2^{-2m} \sup_{s\approx 2^m} \big[ A(s,\xi) + B(s,\xi) + C(s,\xi) + D(s,\xi) \big]
\end{split}
\end{equation}
where
\begin{align}
\label{l<p15.1}
\begin{split}
A & := \Big| \iiiint e^{is \Psi} \mathfrak{a}(\xi,\eta,\s,\rho,\zeta)
	\wt{f}(\rho) \wt{f}(\zeta) \wt{f}(\eta-\rho-\zeta) \, d\eta \, d\rho \, d\zeta\,  \wt{f}(\s)  d\s \Big|
\\
& \mathfrak{a} :=   (\partial_\eta + \partial_\zeta) \Big[ \frac{1}{(\partial_\eta+\partial_\zeta) \Psi} 
  (\partial_\eta + \partial_\rho) \Big( \frac{1}{(\partial_\eta + \partial_\rho) \Psi}
 \frac{\varphi_p\big(\Phi\big)}{\Phi}  \mathfrak{q} \varphi_{\underline{k}} \Big) \Big] ,
\end{split}
\\ \nonumber
\\
\label{l<p15.2}
\begin{split}
B & := \Big| \iiiint e^{is \Psi} \mathfrak{b}(\xi,\eta,\s,\rho,\zeta)
	\big[ \partial_\rho \wt{f}(\rho) \big] \wt{f}(\zeta) \, \wt{f}(\eta-\rho-\zeta) \, d\eta \, d\rho \, d\zeta\,  \wt{f}(\s)  d\s \Big|
\\
& \mathfrak{b} :=  (\partial_\eta + \partial_\zeta) \Big[ \frac{1}{(\partial_\eta+\partial_\zeta) \Psi} 
  \frac{1}{(\partial_\eta + \partial_\rho) \Psi}
  \frac{\varphi_p\big(\Phi\big)}{\Phi}  \mathfrak{q} \varphi_{\underline{k}} \Big] ,
\end{split}
\\ \nonumber
\\
\label{l<p15.3}
\begin{split}
C & := \Big| \iiiint e^{is \Psi} \mathfrak{c}(\xi,\eta,\s,\rho,\zeta)
	\, \big[\partial_\rho \wt{f}(\rho) \big] \, \big[ \partial_\zeta \wt{f}(\zeta) \big] 
	\, \wt{f}(\eta-\rho-\zeta) \, d\eta \, d\rho \, d\zeta\,  \wt{f}(\s)  d\s \Big|
\\
& \mathfrak{c} := \frac{1}{(\partial_\eta+\partial_\zeta) \Psi} \frac{1}{(\partial_\eta + \partial_\rho) \Psi}
 \frac{\varphi_p\big(\Phi\big)}{\Phi}  \mathfrak{q} \varphi_{\underline{k}},
\end{split}
\\ \nonumber
\\
\label{l<p15.4}
\begin{split}
D & := \Big| \iiiint e^{is \Psi} \mathfrak{d}(\xi,\eta,\s,\rho,\zeta)
	 \wt{f}(\rho) \big[ \partial_\zeta \wt{f}(\zeta) \big] \, \wt{f}(\eta-\rho-\zeta) \, 
	 d\eta \, d\rho \, d\zeta\,  \wt{f}(\s)  d\s \Big|
\\
& \mathfrak{d} :=  \frac{1}{(\partial_\eta+\partial_\zeta) \Psi} 
  (\partial_\eta + \partial_\rho) \Big[ \frac{1}{(\partial_\eta + \partial_\rho) \Psi}
  \frac{\varphi_p\big(\Phi\big)}{\Phi} \mathfrak{q} \varphi_{\underline{k}} \Big].
\end{split}
\end{align}
To obtain the desired bound \eqref{l<p12} it suffices to show
\begin{align}\label{l<p15bound}
2^{\ell} \sup_{s\approx 2^m} \big( |A(s,\xi)| + |B(s,\xi)| + |C(s,\xi)| 
  + |D(s,\xi)| \big) \lesssim \e_1^4 2^{-5\beta'm}.
\end{align}

To estimate \eqref{l<p15.1} we first observe that, in view of \eqref{l<p14.0}, the symbol satisfies
\begin{align}\label{l<p15.1sym}
\begin{split}
|\mathfrak{a}| & \lesssim 2^{-k_3-k_4-p} \cdot \max\big( 2^{-k_3}, 2^{-p+k_1}, 2^{-k_1} \big) 
	\cdot \max\big( 2^{-k_4}, 2^{-p+k_1}, 2^{-k_1} \big)
	\\ 
	& \lesssim 2^{-k_3-k_4-p} \cdot \max\big( 2^{-k_3}, 2^{-k_1} \big) 
	\cdot \max\big( 2^{-k_4}, 2^{-k_1}\big)
\end{split}
\end{align}
having used that $ 2k_1 \leq p + 20$.
If we iterate the above integration by parts procedure, 
and only keep the terms where the derivatives never hit the $\widetilde{f}$, the gain at each step is
$$
2^{-2m} 2^{-k_3-k_4} \max(2^{-k_3},2^{-k_1}) \max(2^{-k_4},2^{-k_1}) \lesssim 2^{-2\delta m},
$$
since $k_4 \geq - \frac{m}{2} + \delta m$ and $k_1 + k_4 \geq -m + \delta m$. 
Thus, one obtains an arbitrarily large gain in powers of $2^m$, leading to the desired estimates.
There remain the terms where one of the $\wt{f}$ is hit, 
but they are all better behaved than $B$ and $C$, to which we now turn.

To estimate \eqref{l<p15.2} we first bound, similarly to \eqref{l<p15.1sym},
\begin{align}\label{l<p15.2sym}
|\mathfrak{b}| \lesssim 2^{-k_3-k_4-p} \cdot 2^{-\min(k_4,k_1)}.
\end{align}
Using this, integration by parts in the ``decorrelated'' variable $\s$, and the a priori bounds
placing $\partial_\rho \wt{f} \in L^2$ and the other two profiles in $L^\infty_\xi$, we get
\begin{align*}
2^{\ell} |B| & \lesssim \e_1^3 2^{3\alpha m} \cdot 2^{-k_3-k_4}
  \cdot 2^{-\min(k_4,k_1)} 
  \cdot 2^{\min(k_5,k_1) + \med(k_1,k_4,k_5)} \cdot 2^{(1/2)(k_3+k_4+k_5)} \cdot X_{k_2,m}
  \\
  & \lesssim \e_1^4 2^{-3m/4 + 4\alpha m}.
\end{align*}

Finally, \eqref{l<p15.3} can be dealt with in a similar way by using $|\mathfrak{c}| \lesssim 2^{-k_3-k_4-p}$,
the usual argument for the ``decorrelated'' variable giving a factor of $X_{k_2,m}$, 
estimating in $L^2$ the two differentiated profiles, and using the a priori bounds \eqref{aprioriw}:
\begin{align*}
2^{\ell} |C| \lesssim \e_1^3 2^{3\alpha m} \cdot 2^{-k_3-k_4}
	\cdot  2^{\min(k_5,k_1)} \cdot 2^{(1/2)(k_3+k_4+k_5)} \cdot X_{k_2,m}
	\lesssim \e_1^4 2^{-3m/4 + 4\alpha m}.
\end{align*}

Finally, we have \eqref{l<p15.4} which is similar to \eqref{l<p15.2}, with one profile differentiated
and the other derivative hitting the symbol. 
The symbol satisfies $|\mathfrak{d}| \lesssim 2^{-k_3-k_4-p} \cdot 2^{-\min(k_3,k_1)}$ and we can estimate
\begin{align*}
2^{\ell} |D| & \lesssim \e_1^3 2^{3\alpha m} \cdot 2^{-k_3-k_4}
	\cdot 2^{-\min(k_3,k_1)} 
	\cdot 2^{\min(k_5,k_1) + \med(k_1,k_3,k_5)} \cdot 2^{(1/2)(k_3+k_4+k_5)} \cdot X_{k_2,m}
	\\
	& \lesssim \e_1^4 2^{-3m/4 + 4\alpha m}.
\end{align*}

The bound \eqref{l<p15bound} is proven and \eqref{l<p12} follows, thereby completing the estimate for 
the term $K^{S1}$ in \eqref{l<p5}.


\smallskip
\subsubsection*{{\bf Step 2.2}: Estimate of $K^{S2}$ in \eqref{l<p5}}
Recall the definition of $\mathcal{C}^{S2}$ from \eqref{formulacubiccoeff} and \eqref{defCS12}.
We can see that $K^{S2}_{\iota_1\iota_2\iota_3}$ has the form
\begin{align}
\label{l<p16}
\begin{split}
K^{S2}_{\iota_1\iota_2\iota_3} = \int_{0}^t \tau_m(s) \iiiint e^{is \Gamma_{\iota_1\iota_2\iota_3}} 
  \, \frac{\varphi_p\big(\Phi\big)}{\Phi}
  \, \mathfrak{q}_2(\xi,\eta,\s,\rho,\zeta,\omega) \, \varphi_{k_1}(\eta)\varphi_{k_2}(\sigma) 
  \\ \wt{f_{\iota_1}}(\rho) \wt{f_{\iota_2}}(\zeta) \wt{f_{\iota_3}}(\eta-\rho-\zeta-\omega) \wt{f}(\s) \, d\eta \, d\rho \, d\zeta \,d\s \, 
  \pv \frac{\widehat{\phi}(\omega)}{\omega}d\omega \, ds,
\end{split}
\end{align}
with
\begin{align}\label{l<p16.1}
\Gamma_{\iota_1\iota_2\iota_3}(\xi,\rho,\eta,\omega,\s) 
  := \jxi - \iota_1\langle\rho\rangle - \iota_2\langle\zeta\rangle - \iota_3\langle\eta-\zeta-\rho-\omega\rangle 
  - \jsig, \quad \iota_1,\iota_2,\iota_3 \in \{+,-\}.
\end{align}
As before, we may assume that the symbol $\mathfrak{q}_2$ is sufficiently regular with bounded derivatives.
For lighter notation we will often omit the $\iota_i$ indexes and some of the arguments when this causes no confusion.
Recall that we aim to prove, see \eqref{l<p2} and \eqref{l<p5.0},
\begin{align}
\label{l<p17}
2^{\ell} \big| \chi_{\ell,\sqrt{3}}(\xi) \, K^{S2}(t,\xi) \big| \lesssim \e_1^2 2^{-m} 2^{-3\beta'm}.
\end{align}

We start by splitting
\begin{align}\label{l<p18}
K^{S2}(t,\xi) = \int_{0}^t  \big[ A_1(s,\xi) + A_2(s,\xi) \big]  \, \tau_m(s) ds,
\end{align}
where
\begin{align}
& A_1(s,\xi) = \int F(s,\xi,\omega) \, \pv \frac{\widehat{\phi}(\omega) \varphi_{\leq -5m}(\omega)}{\omega} d\omega,
\quad 
A_2(s,\xi) = \int F(s,\xi,\omega) \, \frac{\widehat{\phi}(\omega) \varphi_{> -5m}(\omega)}{\omega} d\omega,
\end{align}
with
\begin{align}\label{l<p18.3}
\begin{split}
F(s,\xi,\omega) := \iiiint e^{is \Gamma} 
  \, \frac{\varphi_p\big(\Phi\big)}{\Phi}
  \, \mathfrak{q}_2\, \varphi_{k_1}(\eta)\varphi_{k_2}(\sigma) 
  \wt{f}(\rho) \wt{f}(\zeta) \wt{f}(\eta-\rho-\zeta-\omega) \wt{f}(\s) \, d\eta\, d\rho \,d\zeta \, d\s.
\end{split}
\end{align}

\smallskip
\subsubsection*{Estimate of $A_1$} 
In this case, $\omega$ is very small and we need to use the principal value. 
We estimate, for $s\approx 2^m$,
\begin{align}\label{l<p18.4}
\begin{split}
 |A_1(s,\xi)| & \lesssim
  \int \Big| F(s,\xi,\omega) - F(s,\xi,0) \Big| \frac{\varphi_{\leq -5m}(\omega)}{|\omega|} \, d\omega \, 
  \\
 & \lesssim  2^{-5m} \cdot \sup_{s\approx 2^m} \sup_{|\omega|\lesssim 2^{-5m}} 
  \big|\partial_\omega F(s,\xi,\omega)\big|.
\end{split}
\end{align}
Inspecting the formula \eqref{l<p18.3} we see that $\partial_\omega F$ has three contributions corresponding to 
the derivative hitting the phase $\G$, the symbol $\mathfrak{q}_2$, 
or the profile $\wt{f}(\eta-\rho-\zeta-\omega)$.
The main term is the first one so that, up to lower order (faster decaying) terms, 
we have
\begin{align*}
\begin{split}
\partial_\omega F(s,\xi,\omega) & \approx \iiint is ( \partial_\omega \Gamma)\, e^{is \Gamma} 
  \, \frac{\varphi_p\big(\Phi\big)}{\Phi}
  \, \mathfrak{q}_2 \, \varphi_{k_1}(\eta)\varphi_{k_2}(\sigma) 
  \, \wt{f}(\rho) \wt{f}(\zeta) \wt{f}(\eta-\rho-\zeta-\omega) \wt{f}(\s) \, d\eta \, d\rho \,d\zeta \, d\s,
\end{split}
\end{align*}
from which we deduce that, for $s\approx 2^m$,
\begin{align}\label{l<p18.5}
\begin{split}
|\partial_\omega F(s,\xi,\omega)| \lesssim 2^m \cdot 2^{-p} \e_1^3.
\end{split}
\end{align}
From this and \eqref{l<p18.4} we obtain the desired bound \eqref{l<p17} for $A_1$.

\smallskip
\subsubsection*{Estimate of $A_2$} 
We decompose the support of the integral
according to the size of the input frequencies 
$\rho$,$\zeta$,$\eta-\rho-\zeta-\omega$ and $\omega$ by defining 
\begin{align}\label{l<p19}
\begin{split}
& A_{\underline{k},q}(t,\xi) := 
  \int F_{\underline{k}}(s,\xi,\omega) 
  \, \frac{\varphi_{q}(\omega)}{\omega} d\omega 
\end{split}
\end{align}
where
\begin{align}\label{l<p19.1}
\begin{split}
& F_{\underline{k}}(s,\xi,\omega) := \iiiint e^{is \Gamma} 
  \, \frac{\varphi_p\big(\Phi\big)}{\Phi} \, \mathfrak{q}_2 \, \varphi_{\underline{k}}(\eta,\sigma,\rho,\zeta,\omega)\, 
  \wt{f}(\rho) \wt{f}(\zeta) \wt{f}(\eta-\rho-\zeta-\omega) \wt{f}(\s) \, d\eta \, d\rho \, d\zeta \,d\s,
  \\
& \varphi_{\underline{k}}(\eta,\sigma,\rho,\zeta,\omega) := 
  \varphi_{k_1}(\eta)\varphi_{k_2}(\sigma)\varphi_{k_3}(\rho)\varphi_{k_4}(\zeta)\varphi_{k_5}(\eta-\rho-\zeta-\omega).  
\end{split}
\end{align}
Since we can easily dispose of the cases with $\min(k_1,\dots,k_5) \leq -5m$ 
(see for example the estimate \eqref{-3m})
or $\max(k_1,\dots,k_5) \geq m$ (using the Sobolev-type bound in \eqref{propbootfas}),
we are only left with $O(m^5)$ terms like $A_{\underline{k},q}$.
We can then reduce the proof of \eqref{l<p17} to showing the slightly stronger bound
\begin{align}\label{l<p19.5}
2^{\ell} \big| \chi_{\ell,\sqrt{3}}(\xi)A_{\underline{k},q}(t,\xi) \big| 
\lesssim \e_1^2 2^{-2m} 2^{-4\beta' m},
\end{align}
for each fixed set of frequencies with
\[|\med(k_1,k_3,k_4) - \max(k_1,k_3) |\leq 5, \qquad k_5 \leq k_4 \leq k_3\leq -5, \qquad -5m \leq q \leq -D_0,\]
where $D_0$ is a sufficiently large absolute constant,
with the main constraints \eqref{l<ppar} holding as well.
See Remark \ref{remnewfreq} for a justification of the second restriction above,
and notice that the case $q \geq -D_0$ is much easier to deal with since the $\pv$ is not singular.

Notice that on the support of \eqref{l<p19.1} we have, see \eqref{l<p16.1},
\begin{equation}
  \big|(\partial_\eta + \partial_\rho) \Gamma \big| = \big|\frac{\rho}{\jrho}\big| \approx 2^{k_3},
  \qquad 
  |(\partial_\eta + \partial_\zeta) \Gamma|  = \big|\frac{\zeta}{\langle \zeta \rangle}\big| \approx 2^{k_4},
\end{equation}
This is identical to \eqref{l<p14.0} in the case of 
the terms $K^{S1}_{\iota_1\iota_2}$ in \eqref{l<p10} treated before.
We can then proceed in the same way as we did in Step 2.1 above and estimate \eqref{l<p19.1}
for each fixed $\omega$ (note that the terms \eqref{l<p11} are basically the same, up to the 
smooth $\mathfrak{q}$ and $\mathfrak{q}_2$ symbols, as the expression \eqref{l<p19.1} evaluated at $\omega=0$).
This procedure will give a bound by the right-hand side of \eqref{l<p12} for $F_{\underline{k}}(s,\xi,\omega)$,
and integrating over $\omega$ in \eqref{l<p19}, one arrives at \eqref{l<p19.5}.


\smallskip
\subsubsection*{{\bf Step 3}: Estimate of $L^{S1,2}$ in \eqref{l<p6}}
As already pointed out after the formulas \eqref{l<p5.0},
the terms $L^{S1,2}$ are not exactly the same as the terms $K^{S1,2}$, since $k_1$ and $k_2$ do not play the same role,
and we can deduce a little less information on the smaller frequency $|\s| \approx 2^{k_2}$
from information on $|\eta| \approx 2^{k_1}$.
Nevertheless, we can apply the same arguments as in Step 2.1 and Step 2.2 above.
In particular, all the proofs based on integration by parts (see Case 2 starting on page \pageref{l<p14.0}) apply verbatim, 
just by exchanging $k_1$ and $k_2$.
The only exception is the argument in Case 1 on page \pageref{l<p13.5}
where the constraint $k_1 \geq -2m/3$ from \eqref{l<ppar} was used.
Since such a lower bound might not hold for $k_2$ we need some modification of the argument, which we give below.

First, analogously to \eqref{l<p10}, we use the formulas \eqref{CubicS}
and write $L^{S1}$ as a term of the form
\begin{align}\label{l<p21.5}
\begin{split}
& L^{S1}_{\iota_1\iota_2\iota_3} = \int_{0}^t \iiiint e^{is \Psi'_{\iota_1\iota_2\iota_3}} 
  \, \frac{\varphi_p\big(\Phi\big)}{\Phi}
  \, \mathfrak{q}' \, \varphi_{k_1}(\eta)\varphi_{k_2}(\sigma) 
  \\ & \hskip50pt \times \wt{f}(\eta) \wt{f}(\rho) \wt{f}(\zeta) \wt{f}(\s-\rho-\zeta) 
  \, d\eta \, d\s \, d\zeta \, d\rho \, \tau_m(s)  ds,
\\
& \Psi'_{\iota_1\iota_2\iota_3} := \jxi - \jeta - \iota_1\jrho - \iota_2\langle\zeta\rangle 
  - \iota_2\langle\s-\rho-\zeta\rangle, \qquad \iota_1,\iota_2,\iota_3 \in \{+,-\},
\end{split}
\end{align}
where, abusing notation, we still denote by $\mathfrak{q}$ the quartic symbol.
Introducing frequency cutoffs for the new correlated variables we can reduce matters to estimating
\begin{align}
\label{l<pL}
\begin{split}
& (\underline{I}^{p,\underline{k}})'(s,\xi) := 
  \iiiint e^{is \Psi'_{\iota_1\iota_2}} \, \frac{\varphi_p\big(\Phi\big)}{\Phi}
  \, \mathfrak{q}' \, \varphi_{\underline{k}}'(\eta,\sigma,\rho,\zeta)
  \, \wt{f}(\eta) \wt{f}(\rho) \wt{f}(\zeta) \wt{f}(\s-\rho-\zeta) \, d\eta \, d\s \, d\zeta \, d\rho, 
\\
& \varphi_{\underline{k}}'(\eta,\sigma,\rho,\zeta) := \varphi_{k_1}(\eta)\varphi_{k_2}(\sigma)
  \varphi_{k_3}(\rho)\varphi_{k_4}(\zeta)\varphi_{k_5}(\sigma-\rho-\zeta),
\end{split}
\end{align}
as follows: for all $s\approx 2^m$
\begin{align}
\label{l<p20}
& 2^{\ell} \big|\chi_{\ell,\sqrt{3}} (\underline{I}^{p,\underline{k}})'(s,\xi) \big| 
  \lesssim \e_1^3 2^{-2m} 2^{-5\beta'm},
\\
\label{l<p20'}
& |\max(k_2,k_3) - \med(k_2,k_3,k_4)|\leq 5, \qquad k_5 \leq k_4 \leq k_3 \leq -5,
\end{align} 
under the constraints \eqref{l<ppar}. Compare with \eqref{l<p12}-\eqref{l<porder}.

Applying the same exact reasoning as in pages \pageref{l<p13.1}-\pageref{l<p13} we can obtain the 
analogue of \eqref{l<p13} for this term, that is
\begin{align}\label{l<p21}
2^{\ell} \big| (\underline{I}^{p,\underline{k}})'(s,\xi) \big| 
  \lesssim \e_1^3 \cdot 2^{3\alpha m} \cdot 2^{k_5+k_4 + \min(k_2,k_3)} 
  \cdot 2^{(1/2)(k_3+k_4+k_5)} \cdot X_{k_1,m}.
\end{align}
As in Step 2.1 above we distinguish two main scenarios: in the first one (Case 1 below) integration by parts
in the new correlated variables $\partial_{\s+\zeta}$ and $\partial_{\s+\rho}$ is forbidden
and we need an argument based on \eqref{l<p21};
in the second case, integration by parts is possible and we can proceed as in Case 2 of Step 2.1.

\smallskip
\noindent
{\it Case 1: $k_2 + k_4 \leq -m + \delta m$}.
\eqref{l<p21} with \eqref{wproofnot2} yields
\begin{equation*}
2^{\ell} \big| (\underline{I}^{p,\underline{k}})'(s,\xi) \big| 
  \lesssim \e_1^4 2^{(-3/4+4\alpha) m} \cdot 2^{k_2+3k_4}.
\end{equation*}
Since we are assuming $k_2+k_4 \leq -m +\delta m$,
the above bound suffices to obtain \eqref{l<p20} if $k_4 \leq -m/7$, since in this case
\begin{equation*}
2^{(-3/4+4\alpha) m} \cdot 2^{k_2+3k_4} \lesssim 2^{(-3/4+5\alpha)m} \cdot 2^{2k_4} 
  \lesssim 2^{(-57/28+5\alpha)m}.
\end{equation*}
Notice that we indeed must have $k_4 \leq -m/7$,
for otherwise we would have $k_2 \leq -6m/7 +\delta m$, 
which implies 
\[ -3m/4 + 3k_2/2 \leq -57m/28 + 3 \delta m / 2.
\]
contradicting the constraint \eqref{l<ppar'} for $\delta,\alpha,\beta'$ small enough. 



\smallskip
\noindent
{\it Case 2: $k_2 + k_4 \geq -m + \delta m$}.
This case can be treated by integration by parts as in Case 2 on page \pageref{l<p14.0}, so we skip the details.

\smallskip
\noindent
Finally, notice that $L^{S2}_{\iota_1\iota_2\iota_3}$ can treated similarly to
$L^{S1}_{\iota_1\iota_2\iota_3}$, in the same way that 
$K^{S2}_{\iota_1\iota_2\iota_3}$ was treated similarly to $K^{S1}_{\iota_1\iota_2\iota_3}$,
after taking care of the $\pv$ as in \eqref{l<p18}-\eqref{l<p18.5};
we omit the details.

\smallskip
\subsubsection*{{\bf Step 4}: Estimate of $D^R$ in \eqref{l<p8}}
These terms are relatively easy to estimate under the current assumption $\ell \leq p+10$,
relying on the estimate \eqref{lemdtfR} for the remainder term $\mathcal{R}$.
From \eqref{l<p8} we see that
\begin{align}\label{l<p50}
2^{\ell} |D^R(t,\xi)| & \lesssim 2^\ell \cdot 2^m \sup_{s\approx 2^m}
  \Big( \big| \underline{I}^{p,k_1,k_2}\big[\wtF^{-1}\mathcal{R}, f\big](s,\xi) \big|
  + \big| \underline{I}^{p,k_1,k_2}\big[f, \wtF^{-1}\mathcal{R} \big](s,\xi) \big| \Big)
\end{align}
where $\underline{I}^{p,k_1,k_2}$ is the bilinear operator defined in \eqref{l<pulI}.
Let us look at the first of the two terms on the right-hand side of \eqref{l<p50}; the other one can be treated identically.
Using the integration by parts argument on the profile $f$ (whose frequency is uncorrelated to that of $\mathcal{R}$)
we can see that 
\begin{align*}
2^\ell \big| \underline{I}^{p,k_1,k_2}\big[\mathcal{R},f\big](s,\xi) \big| & \lesssim
  2^\ell  \cdot 2^{-p} \cdot 2^{k_1/2}{\| \mathcal{R}(s) \|}_{L^2} \cdot X_{k_2,m} 
  \\
& \lesssim \e_1^4 2^{-3m/2 + 2\alpha m} \cdot 2^{-3m/4 + \alpha m}
\end{align*}
consistently with \eqref{l<p2} and \eqref{l<p5.0}. 
This completes the proof of the bound \eqref{l<p1}-\eqref{l<ppar}.

\medskip
\subsection{Case $p \geq -m/3 - 10\beta'm$ and $p\leq \ell-10$}\label{Ssecl<2k_1}
First notice that we must have 
\begin{align}\label{p<l0}
|\ell-2k_1|\leq 10, \qquad k_1 \geq p/2 + 10 \geq -m/6 - 5\beta'm + 10.
\end{align}
The analysis in this case is similar to the one in Subsection \ref{Ssecl<p}, but we have decided to separate it 
for better clarity, and to better highlight the difficulties of the case treated in Subsection \ref{Ssecl<p}.
Since $\Phi$ has a strong lower bound, our starting point is again the integration by parts in $s$ 
giving the terms \eqref{l<pibp}-\eqref{l<pibp3} and we aim to prove the bound \eqref{l<p2L2} (or \eqref{l<p2}).

Estimating as in \eqref{l<p4} suffices to deal with the boundary term $J$, 
\begin{align*}
2^{\ell} \, \big| \underline{I}^{p,k_1,k_2}(f,f)(s,\xi) \big| 
&   \lesssim 2^{\ell} \cdot 2^{-p} \cdot X_{k_1,m} \cdot X_{k_2,m} 
  \\ & \lesssim \e_1^2 2^{-p} 2^{(-3/2 + 2\alpha)m} \lesssim \e_1^2 2^{-m} 2^{-3\beta'm},
\end{align*}
since $p \geq -m/3-10\beta' m$, and $2\alpha<\beta'$ sufficiently small.

Next, we write out the terms $K$ and $L$ in \eqref{l<pibp2}-\eqref{l<pibp3} as in \eqref{l<p5.0}--\eqref{l<p8},
and aim to show (as usual we dispense of the $\iota$'s)
\begin{align}\label{p<lbound}
2^{\ell}\big(|K^{S1}| + |K^{S2}| + |L^{S1}| + |L^{S2}|) + 2^{\ell/2}{\| D^R \|}_{L^2}
  \lesssim \e_1^3 2^{-m} 2^{-3\beta'm}
\end{align}
which will imply the main conclusion \eqref{wproofdecest}.

The terms $K^{S2}$ and $L^{S2}$ can be treated in the same way that we will treat the terms $K^{S1}$ and $L^{S1}$ 
below, in analogy to how the terms $K^{S2}$ and $L^{S2}$ were treated in Step 2.2 on page \pageref{l<p16}
in the previous case $\ell \leq p+10$.
Recalling the definitions of $K^{S1}$ and $D^R$ in \eqref{l<p5} and \eqref{l<p8},
we may then reduce the bound \eqref{p<lbound} to showing the following:
\begin{align}\label{p<lbound1}
2^\ell \sup_{s\approx2^m}|\underline{I}^{p,k_1,k_2}[\wtF^{-1}\mathcal{C}^{S1},f](s,\xi) | 
  & \lesssim \e_1^3 2^{-2m} 2^{-3\beta'm},
\\
\label{p<lbound2}
2^\ell \sup_{s\approx2^m}|\underline{I}^{p,k_1,k_2}[f, \wtF^{-1}\mathcal{C}^{S1}](s,\xi) | 
  & \lesssim \e_1^3 2^{-2m} 2^{-3\beta'm},
\end{align}
and
\begin{align}\label{p<lbound3}
2^{\ell/2} & \sup_{s\approx2^m} {\big\|\underline{I}^{p,k_1,k_2}[\wtF^{-1}\mathcal{R},f](s,\xi) \big\|}_{L^2} 
  \lesssim \e_1^3 2^{-2m} 2^{-3\beta'm},
\\
\label{p<lbound4}
2^\ell & \sup_{s\approx2^m} \big|\underline{I}^{p,k_1,k_2}[f, \wtF^{-1}\mathcal{R}](s,\xi) \big| 
  \lesssim \e_1^3 2^{-2m} 2^{-3\beta'm}.
\end{align}

\smallskip
\subsubsection{Proof of \eqref{p<lbound1}}\label{ssecp<l1}
We proceed in a similar way to Step 2.1 on page \pageref{l<p10}.
Many of the initial computations are the same so we will not repeat them.
The way that some terms are eventually estimated differs, and this we will detail.

We write out the term $\mathcal{C}^{S1}$ (with the usual notation simplifications)
and further localize the expression by considering
\begin{align}\label{p<lulI}
\begin{split}
& \underline{I}^{p,\underline{k}}(s,\xi) := \iiiint e^{is \Psi_{\iota_1\iota_2\iota_3}} \, \frac{\varphi_p\big(\Phi\big)}{\Phi}
  \, \mathfrak{q} \, \varphi_{\underline{k}}(\eta,\sigma,\rho,\zeta)
  \, \wt{f}(\rho)\wt{f}(\zeta) \wt{f}(\eta-\rho-\zeta)  \, \wt{f}(\s) \, d\eta\, d\zeta\, d\rho \, d\s,
\\
& \varphi_{\underline{k}}(\eta,\sigma,\rho,\zeta) := 
  \varphi_{k_1}(\eta)\varphi_{k_2}(\sigma)\varphi_{k_3}(\rho)\varphi_{k_4}(\zeta)\varphi_{k_5}(\eta-\rho-\zeta),
\\
& \Psi_{\iota_1\iota_2\iota_3}(\xi,\rho,\zeta,\eta,\s) := 
  \jxi - \iota_1\langle\rho\rangle - \iota_2\langle\zeta\rangle - \iota_3\langle\eta-\rho-\zeta\rangle 
  - \jsig, \qquad \iota_1,\iota_2,\iota_3 = \pm,
\\
& |\max(k_1,k_3) - \med(k_1,k_3,k_4)|\leq 5, \qquad -3m \leq k_5 \leq k_4\leq k_3\leq -5; 
\end{split}
\end{align}
compare with \eqref{l<p10'}-\eqref{l<p11}, and notice that we are using the same
notation $\underline{I}^{p,\underline{k}}$ although the terms are slightly different.
Our aim then is to obtain for this term 
a slightly stronger bound than \eqref{p<lbound1}, with an extra factor of $2^{-\beta'm}$. 

The estimates \eqref{l<p13.1} and \eqref{l<p13.0} apply here verbatim, and lead to inequality \eqref{l<p13};
the only difference being the $2^{\ell - p}$ factor which was dropped there, and must be kept here. This gives
\begin{align}\label{p<l13}
\begin{split}
2^{\ell} \big| \underline{I}^{p,\underline{k}} \big| 
\lesssim \e_1^4 \cdot 2^{3\alpha m}  \cdot 2^\ell \cdot 2^{-p} \cdot 2^{k_5+k_4 + \min(k_1,k_3)}
	\cdot 2^{(1/2)(k_3+k_4+k_5)} \cdot X_{k_2,m}.
\end{split}
\end{align}
We fix $\delta \in (0,\alpha)$ and look at three different cases.

\medskip 
\noindent
{\it Case 1: $k_1 + k_4 \leq -m + \delta m$}. 
Inequality \eqref{p<l13} and $\ell \leq 2k_1 +10$ imply
\begin{equation}
\label{p<l13.5}
2^{\ell} \big| \underline{I}^{p,\underline{k}}  \big| 
  \lesssim \e_1^4 \cdot 2^{-p} \cdot 2^{(-3/4+4\alpha) m} \cdot 2^{3(k_1+k_4)} \lesssim \e_1^4 2^{-p} 2^{-7m/2} 
\end{equation}
which is easily bounded by the right-hand side of \eqref{p<lbound1}.

\medskip \noindent
{\it Case 2: $k_4 \leq -m/2 + \delta m$}.
We can integrate by parts in the formula \eqref{p<lulI} in the direction $\partial_\eta+\partial_\rho$,
using $|(\partial_\eta+\partial_\rho)\Psi| = | \rho/\jrho| \approx 2^{k_3}$.
Up to faster decaying remainders, this gives a term of the form
\begin{align}\label{p<l15}
\begin{split}
I_1 & := \iiiint e^{is \Psi} \mathfrak{i}_1(\xi,\eta,\s,\rho,\zeta)
	\big[ \partial_\rho \wt{f}(\rho) \big] \wt{f}(\zeta) \, \wt{f}(\eta-\rho-\zeta) 
	\, d\eta \, d\rho \, d\zeta\,  \wt{f}(\s)  d\s,
\\
\mathfrak{i}_1 & :=  
  \frac{1}{s (\partial_\eta + \partial_\rho) \Psi}
  \frac{\varphi_p\big(\Phi\big)}{\Phi}  \mathfrak{q} \varphi_{\underline{k}}. 
\end{split}
\end{align}
Estimating $|\mathfrak{i}_1| \lesssim 2^{-m-k_3-p}$,
applying the usual argument to treat the uncorrelated variable $\s$, and using the a priori bounds \eqref{aprioriw}, 
we obtain
\begin{align*}
2^{\ell} \big| I_1 \big| 
&  \lesssim \e_1^4 \cdot 2^{\ell} \cdot 2^{-m-p-k_3} \cdot 2^{(-3/4+\alpha) m} 
  \cdot 2^{k_3/2} 2^{\alpha m} \cdot 2^{(3/2)(k_4+k_5)} 2^{2\alpha m}
  \\ & \lesssim \e_1^4 2^{4\alpha m} \cdot  2^{-7m/4} \cdot 2^{3k_4/2} \cdot 2^{-p};
\end{align*}
using $k_4 \leq -m/2+\delta m$ and 
$p\geq -m/3-10\beta'm$
we can comfortably bound this by the right-hand side of \eqref{p<lbound1} as desired.

\medskip \noindent
{\it Case 3: $k_4 \geq -m/2+\delta m$ and $k_1+k_4 \geq -m +\delta m$.} 
In this case we can integrate by parts both in $\partial_\eta + \partial_\rho$ and $\partial_\eta + \partial_\zeta$
using \eqref{l<p14.0}.
This case corresponds to the Subcase 2.2 on page \pageref{l<p14.0},
and the integration by parts produces the terms \eqref{l<p15.1}-\eqref{l<p15.4}.
As before, the main contribution is the one where the derivatives hit the profiles, that is,
\begin{align}\label{p<l16}
\begin{split}
I_2 & := \Big| \iiiint e^{is \Psi} \mathfrak{i}_2(\xi,\eta,\s,\rho,\zeta)
	\, \big[\partial_\rho \wt{f}(\rho) \big] \, \big[ \partial_\zeta \wt{f}(\zeta) \big] 
	\, \wt{f}(\eta-\rho-\zeta) \, d\eta \, d\rho \, d\zeta\,  \wt{f}(\s)  d\s \Big|
\\
\mathfrak{i}_2 & := \frac{1}{s(\partial_\eta+\partial_\zeta) \Psi} \frac{1}{s(\partial_\eta + \partial_\rho) \Psi}
 \frac{\varphi_p\big(\Phi\big)}{\Phi}  \mathfrak{q} \varphi_{\underline{k}}, 
 \qquad |\mathfrak{i}_2| \lesssim 2^{-2m-k_3-k_4 - p},
\end{split}
\end{align}
see \eqref{l<p15.3}.
The usual integration by parts argument in $\s$, and the a priori bounds, give
\begin{align*}
2^{\ell} \big| I_2 \big| 
 & \lesssim \e_1^4 \cdot 2^{\ell} \cdot 2^{-2m-p-k_3-k_4} \cdot  2^{(-3/4+\alpha) m} 
  \cdot 2^{k_3/2} 2^{\alpha m} \cdot  2^{k_4/2} 2^{\alpha m} \cdot  2^{k_5}
  \\ & \lesssim \e_1^4 2^{3\alpha m} \cdot  2^{-11m/4} \cdot 2^{-p}
\end{align*}
which is enough. 
This concludes the proof of \eqref{p<lbound1}.

\smallskip
\subsubsection{Proof of \eqref{p<lbound2}}\label{ssecp<l2}
The proof of this estimate is not too dissimilar from the previous one,
but we need to pay some more attention to a few additional frequency configurations. 
Again, the issue is that the expressions $\underline{I}^{p,k_1,k_2}[\wtF^{-1}\mathcal{C}^{S1},f](s,\xi)$
and $\underline{I}^{p,k_1,k_2}[f,\wtF^{-1}\mathcal{C}^{S1}](s,\xi)$
are not symmetric, and that we have fewer restrictions on $k_2$ than on $k_1$, see \eqref{p<l0}.
We detail below all the terms that need different treatment than before 
and only sketch the estimates for the others ones.

Writing out $\mathcal{C}^{S1}$, we further localize the expression and consider
\begin{align}\label{p<lulIbis}
\begin{split}
& (\underline{I}^{p,\underline{k}})' := \iiiint e^{is \Psi_{\iota_1\iota_2\iota_3}'} \, \frac{\varphi_p\big(\Phi\big)}{\Phi}
  \, \mathfrak{q}' \, \varphi_{\underline{k}}'(\eta,\sigma,\rho,\zeta)
  \, \wt{f}(\rho)\wt{f}(\zeta) \wt{f}(\sigma-\rho-\zeta)  \, \wt{f}(\eta) \, d\s \, d\zeta \, d\rho \, d\eta,
\\
& \varphi_{\underline{k}}'(\eta,\sigma,\rho,\zeta) := 
  \varphi_{k_1}(\eta)\varphi_{k_2}(\sigma)\varphi_{k_3}(\rho)\varphi_{k_4}(\zeta)\varphi_{k_5}(\s-\rho-\zeta),
\\
& \Psi_{\iota_1\iota_2\iota_3}'(\xi,\rho,\zeta,\eta,\s) := 
  \jxi - \iota_1\langle\rho\rangle - \iota_2\langle\zeta\rangle - \iota_3\langle\s-\rho-\zeta\rangle 
  - \jeta, \qquad \iota_1,\iota_2,\iota_3 = \pm,
\\
& |\max(k_2,k_3) - \med(k_2,k_3,k_4)|\leq 5, \qquad -3m \leq k_5 \leq k_4 \leq k_3 \leq -5; 
\end{split}
\end{align}
compare with \eqref{l<pL} and \eqref{l<p20'}.
For \eqref{p<lbound2} it suffices to show 
\[ 2^{\ell} \sup_{s\approx 2^m} |(\underline{I}^{p,\underline{k}})'| \lesssim 2^{-2m} 2^{-4\beta'm}. \]

Recall the inequality \eqref{l<p21} proved earlier; 
it applies here with an additional $2^{\ell - p}$ factor which was discarded there
\begin{align}\label{p<l21}
\begin{split}
2^{\ell} \big| (\underline{I}^{p,\underline{k}})' \big| 
\lesssim \e_1^3 \cdot 2^{3\alpha m} \cdot 2^\ell \cdot X_{k_1,m} \cdot 2^{-p} \cdot 2^{k_5+k_4 + \min(k_2,k_3)}
	\cdot 2^{(1/2)(k_3+k_4+k_5)}.
\end{split}
\end{align}
Note that, using $\ell \leq 2k_1  + 10$ we have $2^{\ell} \cdot X_{k_1,m} \leq 2^{-m+\alpha m}$. 
Then inequality \eqref{p<l21}, and $k_5\leq k_4$, give
\begin{align}\label{p<l21'}
2^{\ell} \big| (\underline{I}^{p,\underline{k}})' \big| 
  & \lesssim \e_1^4 \cdot 2^{4\alpha m} \cdot 2^{-m} \cdot 2^{-p} \cdot 2^{\min(k_2,k_3) + 3k_4}
\end{align}
As in the proof of \eqref{p<lbound1} we fix $\delta \in (0,\alpha)$ and look at three cases.

\smallskip 
\noindent
{\it Case 1: $k_2 + k_4 \leq -m + \delta m$}.
In this case \eqref{p<l21'} gives
\begin{align}\label{p<l25}
\begin{split}
2^{\ell} \big| (\underline{I}^{p,\underline{k}})' \big| 
  & \lesssim \e_1^4 \cdot 2^{4\alpha m} \cdot 2^{-m} \cdot 2^{-p} \cdot 2^{k_2 + 3k_4}
  \\
  & \lesssim \e_1^4 \cdot 2^{5\alpha m} \cdot 2^{-2m} \cdot 2^{-p} \cdot 2^{2k_4}.
\end{split}
\end{align}
Since $p \geq -m/3 - 10\beta'm$ we see that \eqref{p<l25} would suffices if,
for example, $2k_4 \leq -4m/9 + 2\delta m$.
To see that this condition is satisfied, assume by contradiction that 
instead $k_4 \geq -2m/9 + \delta m$. Then we must have $k_2 \leq -7m/9$
which implies $(1/2+\beta)\ell -m-k_1/2 + 3k_2/2 \leq -2m - m/6$ 
violating the constraint on the parameters \eqref{l<ppar'}.

\smallskip 
\noindent
{\it Case 2: $k_4 \leq -m/2 +\delta m$}. 
Using again \eqref{p<l21'} we see that
$2^{\ell} \big| (\underline{I}^{p,\underline{k}})' \big| 
  \lesssim \e_1^4 \cdot 2^{7\alpha m} \cdot 2^{-5m/2} \cdot 2^{-p}$,
which suffices.

\smallskip
\noindent
{\it Case 3: $k_2 +k_4 \geq -m+\delta m$ and $k_4 \geq -m/2 +\delta m$.} 
In this case, which is analogous to Subcase 2 on page \pageref{l<p15} and Case 3 on page \pageref{p<l16} above,
we have $k_2 +k_4 \geq -m+\delta m$ (and thus $k_2+k_3 \geq -m+\delta m$ as well)
and have the possibility of integrating by parts in $\partial_\s+\partial_\zeta$ and $\partial_\s + \partial_\rho$.
Once again, the main term is the one where derivatives hit the profiles, all the other contributions being of lower order.
We then want to estimate
\begin{align}\label{p<l26}
\begin{split}
H & := \Big| \iiiint e^{is \Psi'} \mathfrak{h}(\xi,\eta,\s,\rho,\zeta)
	\, \big[\partial_\rho \wt{f}(\rho) \big] \, \big[ \partial_\zeta \wt{f}(\zeta) \big] 
	\, \wt{f}(\s-\rho-\zeta) \, d\s \, d\rho \, d\zeta\,  \wt{f}(\eta)  d\eta \Big|,
\\
\mathfrak{h} & := \frac{1}{s(\partial_\eta+\partial_\zeta) \Psi} \frac{1}{s(\partial_\eta + \partial_\rho) \Psi}
 \frac{\varphi_p\big(\Phi\big)}{\Phi}  \mathfrak{q}' \varphi_{\underline{k}}', 
 \qquad |\mathfrak{h}| \lesssim 2^{-2m-k_3-k_4-p},
\end{split}
\end{align}
see the analogous term \eqref{p<l16}.
Applying the usual treatment to the uncorrelated variable $\eta$
together with $2^{\ell} X_{k_1,m} \leq 2^{-m+\alpha m}$,
and using the a priori bounds \eqref{aprioriw}, we obtain
\begin{align*}
2^{\ell} \big| H \big| 
 & \lesssim \e_1^3 \cdot 2^{-2m-p-k_3-k_4} \cdot 2^{\ell} X_{k_1,m} \cdot 
  2^{k_3/2} 2^{\alpha m} \cdot  2^{k_4/2} 2^{\alpha m} \cdot 2^{k_5}
  \\ & \lesssim \e_1^4 2^{3\alpha m} \cdot  2^{-3m} \cdot 2^{-p}
\end{align*}
which is more than enough. 
This concludes the proof of \eqref{p<lbound2}.

\medskip
\subsubsection{Proof of \eqref{p<lbound3} and \eqref{p<lbound4}}\label{ssecp<l3-4}
To estimate these terms we rely on the fast decay of $\mathcal{R}$ from \eqref{lemdtfR}. 
%
%
%
%
Arguing as in the proof of Lemma \ref{Lemma1} (without integrating by parts in $\eta$ in \eqref{S-1})
we can see that the following variant of \eqref{wproofclaim2} holds: 
\begin{align}\label{wproofclaim2add}
{\|\underline{I}^{p,k_1,k_2}[\wtF^{-1}\mathcal{R},f](s,\xi)\|}_{L^2} 
  & \lesssim 2^{-k_1/2} \cdot {\|\mathcal{R}(s)\|}_{L^2} \cdot X_{k_2,m}
\end{align}
Using \eqref{lemdtfR}, $2^{\ell/2} \lesssim 2^{k_1}$ and \eqref{wproofclaim2add},
we see that for $s\approx 2^m$
\begin{align*}
2^{\ell/2} {\|\underline{I}^{p,k_1,k_2}[\wtF^{-1}\mathcal{R},f](s,\xi)\|}_{L^2} 
  & \lesssim 2^{\ell/2} \cdot 2^{-k_1/2} \cdot {\|\mathcal{R}(s)\|}_{L^2} \cdot X_{k_2,m}
  \\
  & \lesssim \e_1^3 2^{-3m/2 + 2\alpha m} \cdot 2^{-3m/4 + \alpha m},
\end{align*}
which implies \eqref{p<lbound3}.


For \eqref{p<lbound4} we use 
another simple variant of \eqref{wproofclaim0} in Lemma \ref{Lemma1} to estimate
\begin{align*}
2^\ell \big|\underline{I}^{p,k_1,k_2}[f,\wtF^{-1}\mathcal{R}](s,\xi) \big|
  & \lesssim 2^{\ell} \cdot 2^{-p} \cdot X_{k_1,m} \cdot {\| \varphi_{k_2} \,\mathcal{R}(s)\|}_{L^1}
  \\
  & \lesssim \e_1^3 \cdot 2^{-p} \cdot 2^{-m} \cdot 2^{k_2/2} 2^{-3m/2 + 2\alpha m}.
\end{align*}
This is enough since $p \geq -m/3 - 10\beta'm$.

\medskip
We have concluded the proof of \eqref{p<lbound}
and obtained the bound \eqref{wproofdecest} in Lemma \ref{lemred}. 
This gives the proof of the main bound \eqref{prowRest} for the main interactions with \eqref{wRmain}.


\bigskip
\section{Weighted estimates part II: the main ``singular'' interaction}\label{secwL}

\subsection{Setup}
The aim of this section is to prove the weighted bound on the norm \eqref{wnorm}
of the singular cubic terms $\mathcal{C}^{S1}_{+-+}(f,f,f)$ and $\mathcal{C}^{S2}_{+-+}(f,f,f)$
defined in \eqref{CubicS}-\eqref{CubicS12},
with a restriction to interacting frequencies close to $\sqrt{3}$. 
Interactions of other frequencies and other singular cubic contributions 
(namely $\mathcal{C}^{S1,2}_{\iota_1 \iota_2 \iota_3}$, with $\{\iota_1,\iota_2,\iota_3\} \neq \{+,+,-\}$) 
will be dealt with in Section \ref{secw'}, together with the higher order terms coming from 
$\mathcal{C}^{S1,2}_{+-+}(g,g,g) - \mathcal{C}^{S1,2}_{+-+}(f,f,f)$ (see \eqref{eqdtf}).
In particular, this section contains the first and main step in the proof of the following:

\begin{prop}\label{propCSbound}
Let $W_T$ be the space defined by the norm \eqref{wnorm}, and consider $u$, solution of \eqref{KG} such that 
the a priori assumptions \eqref{propbootfas} hold for the renormalized profile $f$. 
Then
$$
\left\| \int_0^t \mathcal{C}^{S1}_{+-+}(f,f,f) \,ds \right\|_{W_T} 
  + \left\| \int_0^t \mathcal{C}^{S2}_{+-+}(f,f,f) \,ds\right\|_{W_T}  \lesssim \e_1^3.
$$
\end{prop}
The proof of Proposition \ref{propCSbound} will be completed in Subsection \ref{Cubother}.

The terms $\mathcal{C}^{S1}_{+-+}$ and $\mathcal{C}^{S2}_{+-+}$ 
are as the sum over $\lambda, \mu, \nu, \lambda',\mu',\nu',\iota_2$ 
of more elementary terms, see \eqref{formulacubiccoeff}.
In the present section, we will simply focus on one of them, 
since all the corresponding estimates are identical, up to flipping the sign of various frequencies. 
Furthermore, we discard the complex conjugation signs over $\widetilde{f}$, 
since they do not play any role in the estimates.
More precisely, we consider
\begin{align}\label{Cubic11}
\begin{split}
\mathcal{C}^{S1}_{+-+}(f,f,f)  = &\iint e^{is \Psi(\xi,\eta,\zeta)} \mathfrak{p}(\xi,\eta,\zeta)  
\widetilde{f}(\xi-\eta) \widetilde{f}(\xi-\eta-\zeta) \widetilde{f}(\xi-\zeta) \, d\eta \,d\zeta 
\\ \mbox{where} \qquad 
& \Psi(\xi,\eta,\zeta)  = \Phi_{+-+}(\xi,\xi-\eta,\xi-\eta-\zeta,\xi-\zeta) 
\\&\qquad  \qquad = \langle \xi \rangle - \langle \xi - \eta \rangle + \langle \xi - \eta -\zeta  \rangle - \langle \xi - \zeta\rangle
\end{split}
\end{align}
and
\begin{align*}
\mathcal{C}^{S2}_{+-+}(f,f,f) = & \iiint e^{is \Psi(\xi,\eta,\zeta,\theta)} \mathfrak{p}(\xi,\eta,\zeta,\theta) 
\widetilde{f}(\xi-\eta) \widetilde{f}(\xi-\eta-\zeta-\theta) 
  \widetilde{f}(\xi-\zeta) \frac{\widehat{\phi}(\theta)}{\theta} \, d\eta \,d\zeta \,d\theta 
  \\
\qquad \mbox{where} \quad & \Psi(\xi,\eta,\zeta,\theta) = \Phi_{+ - +}(\xi,\xi-\eta,\xi-\eta-\zeta-\theta,\xi-\zeta) 
\\
& \qquad \qquad = \langle \xi \rangle - \langle \xi - \eta \rangle + \langle \xi - \eta -\zeta -\theta  \rangle - \langle \xi - \zeta\rangle.
\end{align*}
We omit the $\pv$ sign for lighter notation,
and slightly abuse notation in denoting the symbols $\mathfrak{p}$ 
and the phases $\Psi$ with the same letter in the two different expressions above;
the presence of the extra variable $\theta$ should resolve any confusion.

Let us say a word about the parameterization of the frequencies which was chosen above. 
If we were dealing with the nonlinear Schr\"odinger equation, the phase resulting 
from the above parameterization would be (say, for $\mathcal{C}^{S1}_{+-+}$)
$$
\xi^2 - (\xi-\eta)^2 + (\xi - \eta -\zeta)^2 - (\xi - \zeta)^2 = 2 \eta \zeta,
$$
which does not depend on $\xi$, and is thus very favorable to deriving estimates. 
Of course, we are not dealing with the nonlinear Schr\"odinger equation, 
but close to interactions of the type $(\xi,\xi,\xi) \to \xi$, 
the above identity holds to leading order; this should be kept in mind in the estimates which follow.

Finally, we will assume in this section that the symbols $\mathfrak{p}$ satisfy
\begin{align}\label{Cubic12}
\begin{split}
& \mathfrak{p}(\xi,\eta,\zeta) = \mathfrak{p}(\xi,\eta,\zeta) \varphi_{\leq -10}(|\xi-\sqrt{3}|+|\eta|+|\zeta|),
\\
& \mathfrak{p}(\xi,\eta,\zeta,\theta) = \mathfrak{p}(\xi,\eta,\zeta,\theta)
  \varphi_{\leq -10}(|\xi-\sqrt{3}|+|\eta|+|\zeta|+|\theta|), 
\end{split}
\end{align}
i.e., the frequencies are localized to $|\xi-\sqrt{3}| \ll 1$ and $|\eta|, |\zeta|, |\theta| \ll 1$,
and that they are in $\mathcal{C}_0^\infty$ with $O(1)$ bounds on their derivatives. 
This latter assumption is justified (in the current frequency configuration) 
in view of the explicit formula \eqref{formulacubiccoeff} and
and the smoothness of the coefficients involved in it, and the estimate of Lemma \ref{californiacondor}.

\subsection{The bound for $\mathcal{C}^{S1}_{+-+}$}\label{subsecCS1}
First note that for $\tau(\xi) = \langle \xi \rangle$, 
one has $\tau'(\xi) = \frac{\xi}{ \langle \xi \rangle}$, $\tau''(\xi) = \frac{1}{\langle \xi \rangle^3}$, 
and $\tau'''(\xi) = - \frac{3\xi}{\langle \xi \rangle^5}$. 
Therefore, in the regime which interests us here ($|\xi - \sqrt{3}| \ll 1$ and $|\eta|, |\zeta| \ll 1$),
we have the expansions
\begin{align}\label{CubicPsiexp}
\begin{split}
& \Psi(\xi,\eta,\zeta) = \frac{1}{\langle \xi \rangle^3} 2 \eta \zeta + O(|\eta,\zeta|^3), 
\\
& \partial_\xi \Psi(\xi,\eta,\zeta) = - \frac{3\xi}{\langle \xi \rangle^5} \eta \zeta + O(|\eta,\zeta|^3),
\\
& \partial_\eta \Psi(\xi,\eta,\zeta) = \frac{1}{\langle \xi -\eta \rangle^3} \zeta + O(|\zeta|^2),
\\
& (\partial_\eta - \partial_\zeta) \Psi(\xi,\eta,\zeta) = \frac{1}{\langle \xi -\eta \rangle^3} (\zeta - \eta) 
  + O( |\eta-\zeta|^2), 
\\
& \partial_\eta^2 \Psi(\xi,\eta,\zeta) =\frac{3(\xi-\eta)}{\langle \xi - \eta \rangle^5} \zeta + O(|\zeta|^2),
\\
& \partial_\eta \partial_\zeta \Psi(\xi,\eta,\zeta) = \frac{1}{\langle \xi - \eta - \zeta \rangle^3}.
\end{split}
\end{align}

Applying $\partial_\xi$ to~\eqref{Cubic11} gives
$$
\mathcal{J}^1 + \mathcal{J}^2 + \mathcal{J}^3 + \{ \mbox{symmetrical or easier terms} \}
$$
where
\begin{align*}
&\mathcal{J}^1 = \iint e^{is \Psi(\xi,\eta,\zeta)} \mathfrak{p}(\xi,\eta,\zeta) i s \partial_\xi \Psi(\xi,\eta,\zeta) \widetilde{f}(\xi-\eta) \widetilde{f}(\xi-\eta-\zeta) \widetilde{f}(\xi-\zeta) \, d\eta \,d\zeta, \\
&\mathcal{J}^2 = \iint e^{is \Psi(\xi,\eta,\zeta)} \mathfrak{p}(\xi,\eta,\zeta)  \partial_\xi \widetilde{f}(\xi-\eta) \widetilde{f}(\xi-\eta-\zeta) \widetilde{f}(\xi-\zeta) \, d\eta \,d\zeta, \\
&\mathcal{J}^3 = \iint e^{is \Psi(\xi,\eta,\zeta)} \mathfrak{p}(\xi,\eta,\zeta)  \widetilde{f}(\xi-\eta) \partial_\xi \widetilde{f}(\xi-\eta-\zeta) \widetilde{f}(\xi-\zeta) \, d\eta \,d\zeta .
\end{align*}
First observe that $\mathcal{J}^1$ can be reduced to the other cases. 
Indeed, $\partial_\xi \Psi/\partial_\eta \Psi$ is a smooth function, 
and therefore, it is possible to integrate by parts in $\eta$ in $\mathcal{J}^1$ 
via the identity $\frac{1}{i s \partial_\eta \Psi} \partial_\eta e^{is \Psi} =  e^{is \Psi}$,
obtaining terms similar to $\mathcal{J}^2$ and $\mathcal{J}^3$. 
Therefore, it will be sufficient to treat $\mathcal{J}^2$ and $\mathcal{J}^3$. 

In what follows we will localize the variables 
$\xi$, $\xi-\eta$, $\xi-\eta-\zeta$, $\xi - \zeta$ and $s$, on the dyadic scales
\begin{equation} \label{defljm}
\begin{split}
& s \approx 2^m, \quad |\xi - \sqrt{3}| \approx 2^\ell, 
\\
& |\xi-\eta - \sqrt{3}| \approx 2^{j_1},
\quad |\xi-\eta-\zeta - \sqrt{3}| \approx 2^{j_2},
\quad |\xi-\zeta-\sqrt{3}| \approx 2^{j_3}.
\end{split}
\end{equation}
Consistently with \eqref{wnorm}, our aim will be to show that
under the a priori assumptions \eqref{propbootfas}, we have 
\begin{align}\label{Cubic1main}
\sup_{\ell \in \mathbb{Z} \cap [\lfloor -\gamma n \rfloor,0]} 
\left\|\tau_n(t) \varphi_{\ell}^{[-\gamma n,0]}(\xi-\sqrt{3}) 
  \int_0^t \mathcal{J}^{2,3}(s,\xi) \, ds \right\|_{L^\infty_t([0,T]) L^2_\xi(\R)} 
  \lesssim 2^{\alpha n} 2^{-\beta \ell} \e_1^3.
\end{align}

\subsubsection{Bound for $\mathcal{J}^2$}\label{sectionJ2}
We add a localization in time in the integrand, and consider
$$
\mathcal{J}_m^2 = \iint e^{is \Psi(\xi,\eta,\zeta)} \tau_m(s) \mathfrak{p}(\xi,\eta,\zeta)
  \partial_\xi \widetilde{f}(\xi-\eta) \widetilde{f}(\xi-\eta-\zeta) \widetilde{f}(\xi-\zeta) \, d\eta \,d\zeta.
$$

\medskip
\noindent 
{\bf Case 1: $|\ell + \gamma n| \leq 5$.} 
By taking the inverse Fourier transform of this expression, using Plancherel's equality,
the a priori bounds~\eqref{lembb1} and~\eqref{aprioridecay} and Lemma~\ref{lemmamultilin1},
\begin{align*}
& \left\| \int_0^t  \mathcal{J}_{m}^2(s,\xi) \, ds \right\|_{L^2}
  \lesssim \int_0^t \tau_m(s) \| \partial_\xi \widetilde{f} \|_{L^2} 
  \| e^{is\langle D \rangle} \mathcal{W}^* f \|_{L^\infty}^2 \,ds \lesssim 2^{(\alpha + \beta \gamma)m} 
   \e_1^3.
\end{align*}
Therefore,
$$
\sum_{m \leq n} \left\| \int_0^t  \mathcal{J}_{m}^2(s,\xi) \, ds \right\|_{L^2}
  \lesssim 2^{(\alpha + \beta \gamma)n}\e_1^3.
$$

\medskip
\noindent 
{\bf Case 2: $-\gamma n \leq \ell \leq -\gamma m$.}
Similarly to the previous case,
\begin{align*}
& \left\| \int_0^t  \mathcal{J}_{m}^2(s,\xi) \, ds \right\|_{L^2}
  \lesssim \int_0^t \tau_m(s) \| \partial_\xi \widetilde{f} \|_{L^2} 
  \| e^{is\langle D \rangle}  \mathcal{W}^* f \|_{L^\infty}^2 \,ds \lesssim 2^{(\alpha + \beta \gamma)m} 
   \e_1^3.
\end{align*}
This suffices since if $-\gamma n \leq \ell \leq -\gamma m$,
\begin{align*}
\sum_{\substack{m<-\ell/\gamma \\ m\leq n}} \left\|  \int_0^t \mathcal{J}_{m}^2(s,\xi) \, ds \right\|_{L^2}
  \lesssim \sum_{\substack{m \leq - {\ell}/{\gamma} \\ m \leq n}}  2^{(\alpha + \beta \gamma) m} \e_1^3 
 \lesssim 2^{\alpha n - \beta \ell} \e_1^3.
\end{align*}

\medskip
\noindent 
{\bf Case 3: $j_1 > \ell - 100$.} 
We now localize in $\xi-\eta$, by defining
$$
\mathcal{J}_{m,\ell}^{2,(3)} = \iint e^{is \Psi(\xi,\eta,\zeta)} \mathfrak{m}(\xi,\eta,\zeta)  \partial_\xi \widetilde{f}(\xi-\eta) \widetilde{f}(\xi-\eta-\zeta) \widetilde{f}(\xi-\zeta) \, d\eta \,d\zeta
$$
where
$$
\mathfrak{m}(\xi,\eta,\zeta) 
  = \mathfrak{m}_{m,\ell}^{\mathcal{J}^{2,(3)}} (\xi,\eta,\zeta) = \mathfrak{p}(\xi,\eta,\zeta) \varphi_\ell(\xi-\sqrt{3})
  \varphi_{>\ell - 100}(\xi-\eta - \sqrt{3}) \tau_m(s).
$$
Estimating as above we have
$$
\left\| \int_0^t \mathcal{J}_{m,\ell}^{2,(3)}(s,\xi) \,ds \right\|_{L^2}  \lesssim \int_0^t \tau_m(s) 
  \|\varphi_{> \ell - 100}(\cdot - \sqrt 3) \partial_\xi \widetilde{f}(\xi) \|_{L^2} 
  \| e^{is\langle D \rangle} \mathcal{W}^* f \|_{L^\infty}^2\, ds \lesssim 2^{\alpha m} 2^{-\beta \ell} \e_1^3.
$$
Therefore,
$$
\sum_{m\leq n}\left\|  \int_0^t \mathcal{J}_{m,\ell}^{2,(3)}(s,\xi) \,ds \right\|_{L^2} 
  \lesssim 2^{\alpha n - \beta \ell} \e_1^3.
$$

\medskip
\noindent 
{\bf Case 4: $\ell > -\gamma m$ and $j_1 \leq \ell - 100$.}
Let us now consider
\begin{align*}
& \mathcal{J}_{m,\ell}^{2,(4)} = \iint e^{is \Psi(\xi,\eta,\zeta)} \mathfrak{m}(\xi,\eta,\zeta)  \partial_\xi \widetilde{f}(\xi-\eta) \widetilde{f}(\xi-\eta-\zeta) \widetilde{f}(\xi-\zeta) \, d\eta \,d\zeta
\end{align*}
where
\begin{align*}
& \mathfrak{m}(\xi,\eta,\zeta) = \mathfrak{m}_{m,\ell}^{\mathcal{J}^{2,(4)}} (\xi,\eta,\zeta) =  \mathfrak{p}(\xi,\eta,\zeta)\varphi_\ell(\xi-\sqrt{3}) \varphi_{<\ell-100}(\xi-\eta - \sqrt{3})  \tau_m(s).
\end{align*}
Observe that, on the support of the integrand, $|\eta| \sim 2^{\ell}$, which implies, see \eqref{CubicPsiexp},
$$
\left| \frac{1}{\partial_\zeta \Psi} \right| \sim 2^{-\ell}, 
  \qquad \left| \frac{\partial_\zeta^2 \Psi}{(\partial_\zeta \Psi)^2} \right| \sim 2^{-\ell}.
$$
Integrating by parts in $\zeta$ we obtain (we are omitting irrelevant numerical constants)
\begin{align*}
\mathcal{J}^{2,(4)}_{m,\ell}(s,\xi) 
  & = \iint  e^{is \Psi} \partial_\zeta \left[  \frac{\mathfrak{m}}{s \partial_\zeta \Psi} \right] \partial_\xi \widetilde{f}(\xi-\eta)  \widetilde{f}(\xi-\eta-\zeta) \widetilde{f}(\xi-\zeta) \, d\eta \,d\zeta 
  \\
  & \qquad + \iint  e^{is \Psi} \frac{\mathfrak{m}}{s \partial_\zeta \Psi} \partial_\xi \widetilde{f}(\xi-\eta)  \partial_\xi \widetilde{f}(\xi-\eta-\zeta) \widetilde{f}(\xi-\zeta) \, d\eta \,d\zeta 
  \\
& \qquad + \{ \mbox{symmetrical term} \}
\\
& = \mathcal{J}^{2,(4)\flat}_{m,\ell} + \mathcal{J}^{2,(4)\sharp}_{m,\ell} +  \{ \mbox{symmetrical term} \}.
\end{align*}

We notice first that the term $\mathcal{J}^{2,(4)\flat}_{m,\ell}$ is much simpler to estimate than 
$\mathcal{J}^{2,(4)\sharp}_{m,\ell}$. Indeed, both symbols enjoy the same estimates, 
but two functions $\widetilde{f}$ are differentiated in the latter, and only one in the former.
Therefore, we only concentrate on $\mathcal{J}^{2,(4)\sharp}_{m,\ell}$, 
for which we would like to apply Lemma \ref{lemmamultilin1}, using that 
\begin{equation}
\label{hummingbird}
\left\| \mathcal{F}\left( \frac{\mathfrak{m}}{ \partial_\zeta \Psi} \right) \right\|_{L^1} \lesssim 2^{-\ell}.
\end{equation}
To see why this is true, notice that, on the support of $\mathfrak{m}$, 
the variables $\xi$ and $\eta$ enjoy the localization $|\xi-\sqrt{3}| + |\eta| \lesssim 2^\ell$ 
while for any $a,b,c$, 
\begin{align*}
\left|  \partial_\xi^a \partial_\eta^b \partial_\zeta^c \frac{1}{\partial_\zeta \Psi} \right| \lesssim 2^{-(b+1)\ell} \quad \mbox{and} \quad \left|  \partial_\xi^a \partial_\eta^b \partial_\zeta^c \mathfrak{m} \right| \lesssim 2^{-(a+b)\ell}.
\end{align*}
By Remark \ref{multilinrem} following Lemma \ref{lemmamultilin1}, we obtain \eqref{hummingbird}.

Thus, we can apply Lemma \ref{lemmamultilin1}, 
using the bounds~\eqref{lembb1}, \eqref{lembb3} and~\eqref{aprioridecay}, to obtain
\begin{equation}
\label{parrot}
\begin{split}
 \left\| \int_0^t\mathcal{J}^{2,(4)\flat}_{m,\ell} \,ds \right\|_{L^2} 
  & \lesssim \int_0^t \left\| \mathcal{F}\left( \frac{\mathfrak{m}}{ \partial_\zeta \Psi} \right) \right\|_{L^1} \| \varphi_{<\ell}(\cdot - \sqrt 3) \partial_\xi \widetilde{f} \|_{L^1} \| \partial_\xi \widetilde{f} \|_{L^2} \| e^{is\langle D \rangle} \mathcal{W}^* f \|_{L^\infty} \frac{ds}{s} \\
  & \lesssim  2^{-\ell} \| \varphi_{<\ell}(\cdot - \sqrt 3) \partial_\xi \widetilde{f} \|_{L^1} \|  \partial_\xi \widetilde{f} \|_{L^2} \| e^{is\langle D \rangle} \mathcal{W}^* f \|_{L^\infty} \\
  & \lesssim 2^{-\ell} \cdot 2^{\beta' \ell + \alpha m}\e_1
  \cdot 2^{(\alpha + \beta \gamma)m} \e_1 \cdot 2^{-m/2} \e_1
  \\
  & \lesssim  2^{-\beta \ell + \alpha m} \e_1^3,
\end{split}
\end{equation}
where we used that $\ell > -\gamma m$ and $\alpha + \beta \gamma < \frac{1}{4}$. 
We abused notations slightly by simply denoting $L^\infty$ instead of $L^\infty_{s \approx 2^m} L^\infty_x$; 
we will use this shorthand repeatedly in the following.

\subsubsection{Bound for $\mathcal{J}^3$}\label{boundJ3}

\noindent 
{\bf Cases 1,2,3: $\ell < -\gamma m$ or $j_2 > \ell - 100$.} 
These cases are identical to cases 1,2, and 3 of the estimate for $\mathcal{J}^2$, 
except that the roles of $j_1$ and $j_2$ are exchanged.

\medskip
\noindent 
{\bf Case 4: $\ell > -\gamma m$ and $j_2 < \ell - 100$.} 
Without loss of generality, we can assume that $j_1 \geq j_3$. 
In the following we will add  an index $j_4$ to track the localization of $\eta - \zeta$: 
$$|\eta - \zeta|\approx 2^{j_4}.$$
Due to the definitions of $\ell, j_1, j_2, j_3$ 
(see \eqref{defljm}), it suffices to consider (up to the symmetry between $j_1$ and $j_3$)
the following cases:

\begin{itemize}
\item Case 4.1: $2^\ell \approx 2^{j_1} \gtrsim 2^{j_3}$ and ${j_4} > {\ell - 100}$; 
\item Case 4.2: $2^\ell \approx 2^{j_1} \approx 2^{j_3}$ and ${j_4} < {\ell - 100}$;
\item Case 4.3: $2^{j_1} \approx 2^{j_3} \approx 2^{j_4}\gg 2^\ell$.
\end{itemize}

\bigskip

\noindent {\bf Case 4.1: $\ell > -\gamma m$, $j_2 < \ell - 100$, 
  $2^\ell \approx 2^{j_1} \gtrsim 2^{j_3}$ and $j_4 > \ell - 100$.} 
In other words, we are considering here
$$
\mathcal{J}_{m,\ell}^{3,(1)} = \iint e^{is \Psi(\xi,\eta,\zeta)} \mathfrak{m}(\xi,\eta,\zeta)\widetilde{f}(\xi-\eta)   \partial_\xi \widetilde{f}(\xi-\eta-\zeta) \widetilde{f}(\xi-\zeta) \, d\eta \,d\zeta 
$$
where
\begin{align*}
\mathfrak{m}(\xi,\eta,\zeta) & = \mathfrak{m}_{m,\ell}^{\mathcal{J}^{3,(1)}}(\xi,\eta,\zeta) \\
&  =  \mathfrak{p}(\xi,\eta,\zeta) \varphi_\ell(\xi-\sqrt{3}) \varphi_{\sim \ell}(\xi-\eta - \sqrt{3}) 
  \varphi_{<\ell-100}(\xi-\eta -\zeta- \sqrt{3}) 
  \\
& \qquad \qquad \qquad \varphi_{\leq\ell}(\xi-\zeta-\sqrt{3}) \varphi_{>\ell - 100}(\eta-\zeta) \tau_m(s).
\end{align*}
On the support of the symbol,
$$
| (\partial_\eta - \partial_\zeta) \Psi | \sim |\eta - \zeta| \approx 2^{\ell},
  \qquad | (\partial_\eta - \partial_\zeta)^2 \Psi | \approx 1.
$$
Integrating by parts using the identity 
$\frac{1}{is(\partial_\eta - \partial_\zeta) \Psi} (\partial_\eta - \partial_\zeta) e^{is\Psi} = e^{is\Psi}$ gives
\begin{align*}
 \mathcal{J}^{3,(1)}_{m, \ell} & = \iint e^{is \Psi(\xi,\eta,\zeta)} (\partial_\eta - \partial_\zeta) \left[ \frac{\mathfrak{m}(\xi,\eta,\zeta)}{is(\partial_\eta - \partial_\zeta) \Psi)} \right]  \widetilde{f}(\xi-\eta)   \partial_\xi \widetilde{f}(\xi-\eta-\zeta) \widetilde{f}(\xi-\zeta) \, d\eta \,d\zeta \\
&  \qquad +  \iint e^{is \Psi(\xi,\eta,\zeta)}\frac{1}{is(\partial_\eta - \partial_\zeta) \Psi} \mathfrak{m}(\xi,\eta,\zeta) \widetilde{f}(\xi-\eta)   \partial_\xi \widetilde{f}(\xi-\eta-\zeta) \partial_\xi \widetilde{f}(\xi-\zeta) \, d\eta \,d\zeta \\
& \qquad + \{ \mbox{symmetrical terms} \} \\
& = \mathcal{J}^{3,(1)\flat}_{m, \ell}+ \mathcal{J}^{3,(1)\sharp}_{m, \ell} + \{ \mbox{symmetrical terms} \}.
\end{align*}
In order to estimate $\mathcal{J}^{3,(1)\flat}_{m, \ell}$, we claim that 
$$
\left\| \mathcal{F} \left(  (\partial_\eta - \partial_\zeta) \left[ \frac{\mathfrak{m}(\xi,\eta,\zeta)}{i(\partial_\eta - \partial_\zeta) \Psi)} \right] \right) \right\|_{L^1} \lesssim 2^{- 2 \ell}.
$$
This follows from the remark after Lemma~\ref{lemmamultilin1} since on the support of $\mathfrak{m}$, 
the variables $\xi$, $\eta$ and $\zeta$ are such that $|\xi|, |\eta|, |\zeta| \lesssim 2^{\ell}$, and for any $a,b,c$, 
$$
\left| \partial_\xi^a \partial_\eta^b \partial_\zeta^c \frac{1}{(\partial_\eta - \partial_\zeta) \Psi} \right| \lesssim 2^{- \ell (1 + b + c)} \qquad \mbox{while} \qquad \left| \partial_\xi^a \partial_\eta^b \partial_\zeta^c \mathfrak{m} \right| \lesssim 2^{- \ell (a + b + c)}.
$$
Thus we can apply Lemma~\ref{lemmamultilin1} together with~\eqref{lembb2} and~\eqref{lembb3} to obtain
\begin{align*}
\left\|  \int_0^t  \mathcal{J}_{m,\ell}^{3,(1)\flat} \,ds \right\|_{L^2} & \lesssim 2^{-2 \ell}  \| \varphi_{<\ell}(\cdot - \sqrt 3)  \widetilde{f} \|_{L^1} \| \varphi_{<\ell}(\cdot - \sqrt 3) \partial_\xi \widetilde{f} \|_{L^1} \| \varphi_{<\ell}(\cdot - \sqrt 3)  \widetilde{f} \|_{L^2} \\
& \lesssim 2^{-2\ell} \cdot 2^{\ell}\e_1 \cdot 2^{\beta' \ell + \alpha m}\e_1 \cdot 2^{\ell/2} \e_1
= 2^{-\beta \ell + \alpha m} \e_1^3.
\end{align*}

Turning to $\mathcal{J}^{3,(1)\sharp}_{m, \ell}$, by the arguments given above,
$$
\left\| \mathcal{F} \left( \frac{\mathfrak{m}(\xi,\eta,\zeta) }{(\partial_\eta - \partial_\zeta) \Psi}  \right) 
  \right\|_{L^1} \lesssim 2^{-\ell}.
$$
Thus we can apply Lemma~\ref{lemmamultilin1} to obtain
$$
\left\|  \int_0^t  \mathcal{J}_{m,\ell}^{3,(1)\sharp} \,ds \right\|_{L^2} \lesssim 2^{-\ell} 
\| e^{is\langle D \rangle} \mathcal{W}^* f \|_{L^\infty}
\| \varphi_{<\ell}(\cdot - \sqrt 3) \partial_\xi \widetilde{f} \|_{L^1} \| \partial_\xi \widetilde{f}\|_{L^2} .
$$
From here, the estimate proceeds just like for~\eqref{parrot} above.

\bigskip
\noindent {\bf Case 4.2: $\ell > -\gamma m$, $j_2 < \ell - 100$, 
  $2^\ell \approx 2^{j_1} \approx 2^{j_3}$, and $j_4 < \ell - 100$.} 
In other words, we are considering here
$$
\mathcal{J}_{m,\ell}^{3,(2)} = \iint e^{is \Psi(\xi,\eta,\zeta)} 
\mathfrak{m}(\xi,\eta,\zeta)\widetilde{f}(\xi-\eta) 
  \partial_\xi \widetilde{f}(\xi-\eta-\zeta) \widetilde{f}(\xi-\zeta) \, d\eta \,d\zeta 
$$
where
\begin{align*}
& \mathfrak{m}(\xi,\eta,\zeta) = \mathfrak{m}_{m,\ell}^{\mathcal{J}^{3,(2)}}(\xi,\eta,\zeta) \\
& \qquad = \mathfrak{p}(\xi,\eta,\zeta) \varphi_\ell(\xi-\sqrt{3}) \varphi_{\sim \ell}(\xi-\eta - \sqrt{3})  \\
& \qquad \qquad \qquad \varphi_{<\ell-100}(\xi-\eta -\zeta- \sqrt{3})\varphi_{\sim \ell}(\xi-\zeta- \sqrt{3}) 
  \varphi_{<\ell - 100}(\eta-\zeta) \tau_m(s).
\end{align*}
Notice that, on the support of $\mathfrak{m}$,
$$
|\xi - \eta - \sqrt 3| \approx |\xi - \zeta - \sqrt 3| \approx |\eta| \approx |\zeta| \sim 2^\ell,
$$
so that
$$
| \Psi| \approx |\partial_\xi \Psi| \approx 2^{2 \ell} \qquad \mbox{and} 
  \qquad |\partial_\eta \Psi| \approx |\partial_\zeta \Psi| \approx 2^\ell.
$$
In the expression above giving $\mathcal{J}_{m,\ell}^{3,(2)}$, we write $\partial_\xi \widetilde{f}(\xi-\eta-\zeta) = -\partial_\eta \widetilde{f}(\xi-\eta-\zeta)$, and integrate by parts in $\eta$. This results into
\begin{align*}
\mathcal{J}_{m,\ell}^{3,(2)} & = \iint e^{is \Psi(\xi,\eta,\zeta)} \partial_\eta \mathfrak{m}(\xi,\eta,\zeta) \widetilde{f}(\xi-\eta) \widetilde{f}(\xi-\eta-\zeta) \widetilde{f}(\xi-\zeta) \, d\eta \,d\zeta \\
& \qquad -  \iint e^{is \Psi(\xi,\eta,\zeta)} \mathfrak{m}(\xi,\eta,\zeta) \partial_\xi \widetilde{f}(\xi-\eta) \widetilde{f}(\xi-\eta-\zeta) \widetilde{f}(\xi-\zeta) \, d\eta \,d\zeta \\
& \qquad +  \iint e^{is \Psi(\xi,\eta,\zeta)} is\partial_\eta \Psi \mathfrak{m}(\xi,\eta,\zeta)  \widetilde{f}(\xi-\eta) \widetilde{f}(\xi-\eta-\zeta) \widetilde{f}(\xi-\zeta) \, d\eta \,d\zeta \\
& = \mathcal{J}_{m,\ell}^{3,(2)\flat} + \mathcal{J}_{m,\ell}^{3,(2)\sharp} + \mathcal{J}_{m,\ell}^{3,(2)\natural}.
\end{align*}
We claim that the first term in the above right-hand side, namely $\mathcal{J}_{m,\ell}^{3,(2)\flat}$, 
is easier to treat than the third, $\mathcal{J}_{m,\ell}^{3,(2)\natural}$,
because $|\partial_\eta \mathfrak m | \approx 2^{-\ell} \lesssim 2^{m+\ell} \approx |s \partial_\eta \Psi|$ 
with corresponding bounds for the $L^1$ norm of their Fourier transform. 
The second term, $\mathcal{J}_{m,\ell}^{3,(2)\sharp}$, 
can be treated like Case 3 for $\mathcal{J}^2$;
thus we are left with analyzing the third one. 
In order to bound it, we integrate by parts in $s$:
\begin{align*}
\int_0^t \mathcal{J}_{m,\ell}^{3,(2)\natural} \,ds & = 
\int_0^t \iint e^{is \Psi(\xi,\eta,\zeta)} \frac{\partial_\eta \Psi}{\Psi} s \mathfrak{m}(\xi,\eta,\zeta) \partial_s \widetilde{f}(\xi-\eta) \widetilde{f}(\xi-\eta-\zeta) \widetilde{f}(\xi-\zeta) \, d\eta \,d\zeta\,ds \\
& \qquad +  \int_0^t \iint e^{is \Psi(\xi,\eta,\zeta)}  \frac{\partial_\eta \Psi}{\Psi}  
  \partial_s [s \mathfrak{m}(\xi,\eta,\zeta)]  \widetilde{f}(\xi-\eta) \widetilde{f}(\xi-\eta-\zeta) \widetilde{f}(\xi-\zeta) \, d\eta \,d\zeta \,ds 
  \\
& \qquad + \{ \mbox{similar or easier terms} \}.
\end{align*}
The ``similar or easier'' terms here include also the boundary terms coming from the integration by parts,
which can be estimated like the other two terms.

We show how to bound the first term in the right-hand side above, 
since the second one can be bounded in the same way. 
First observe that Lemma \ref{lemmamultilin1} applies since on the support of $\mathfrak{m}$, 
the variables are such that $|\xi-\sqrt 3| + |\eta| + |\zeta| \lesssim 2^\ell$ and for any $a,b,c$, 
$$
\displaystyle \left| \partial_\xi^a \partial_\eta^b \partial_\zeta^c \frac{\partial_\eta \Psi}{\Psi} \right| 
  \lesssim 2^{-\ell (1  +b +c)}.
$$
We write
\begin{align}\label{similarterms}
\begin{split}
& \int_0^t \iint e^{is \Psi(\xi,\eta,\zeta)} i \frac{\partial_\eta \Psi}{\Psi} s \mathfrak{m}(\xi,\eta,\zeta) \partial_s \widetilde{f}(\xi-\eta) \widetilde{f}(\xi-\eta-\zeta) \widetilde{f}(\xi-\zeta) \, d\eta \,d\zeta \,ds 
\\
& \qquad = \sum_{\iota_1 \iota_2 \iota_3} 
\int_0^t \iiiint e^{is \Lambda_{\iota_1 \iota_2 \iota_3}(\xi,\eta,\zeta,\sigma,\rho)} i \frac{\partial_\eta \Psi}{\Psi}
  s \mathfrak{m}'(\xi,\eta,\zeta,\rho,\sigma) \\
&  \qquad \qquad \qquad \times \widetilde{f}(\xi-\eta-\sigma-\rho) \widetilde{f}(\sigma) \widetilde{f}(\rho) 
  \widetilde{f}(\xi-\eta-\zeta) \widetilde{f}(\xi-\zeta) \, d\eta \,d\zeta \,d\sigma\,d\rho \, ds
  \\
&  + \{ \mbox{similar terms} \}.
\end{split}
\end{align}
In the above equation, we denoted $\mathfrak{m}'$ for the 5-linear symbol arising when one replaces 
$\partial_s f$ by $\mathcal{C}^{S1}$; 
we omitted various indexes and complex conjugate signs to alleviate the notations, and denoted
$$
\Lambda_{\iota_1 \iota_2 \iota_3}(\xi,\eta,\zeta,\sigma,\rho) = \langle \xi \rangle 
  - \iota_1 \langle\xi-\eta-\sigma-\rho \rangle - \iota_2 \langle \sigma \rangle - \iota_3 \langle \rho \rangle 
  - \langle \xi-\eta-\zeta \rangle + \langle \xi-\zeta \rangle.
$$
By Lemma \eqref{lemCS}, the 5-linear term satisfies H\"older estimates, and can be bounded by
$$
\| \dots \|_{L^2} \lesssim 2^{-\ell} 2^{2m} \| e^{is\langle D \rangle} \mathcal{W}^* f \|_{L^\infty}^4 
  \| \varphi_{\sim \ell}(\cdot - \sqrt 3) \widetilde f \|_{L^2}
  \lesssim 2^{-\ell/2} \e_1^5 \lesssim 2^{-\beta \ell + \alpha m} \e_1^5,
$$
where the last inequality holds since $\ell > -\gamma m$ and $\alpha > \beta' \gamma$.

The ``similar terms'' in \eqref{similarterms} are of various types:
some involve principal value operators instead of $\delta$, but these can be treated identically;
other contain the regular quadratic term which can be treated similarly
using, see \eqref{QRLinftyxi},
$$
\| \varphi_{\ell}(\cdot - \sqrt 3) \mathcal{Q}^R(f,f) \|_{L^2} \lesssim 2^{\ell/2} 2^{-m}. 
$$

\bigskip
\noindent
{\bf Case 4.3: $\ell > -\gamma m$, $j_2 < \ell - 100$, $2^{j_1} \approx 2^{j_3} \gg 2^\ell$.} 
We would like to estimate here $\sum_{j_1 \gg \ell} \mathcal{J}_{m,\ell,j_1}^{3,(3)}$,
with
$$
\mathcal{J}_{m,\ell,j_1}^{3,(3)} = \iint e^{is \Psi(\xi,\eta,\zeta)} \mathfrak{m}(\xi,\eta,\zeta)\widetilde{f}(\xi-\eta)   \partial_\xi \widetilde{f}(\xi-\eta-\zeta) \widetilde{f}(\xi-\zeta) \, d\eta \,d\zeta 
$$
where
\begin{align*}
\mathfrak{m}(\xi,\eta,\zeta) & = \mathfrak{m}_{m,\ell,j_1}^{\mathcal{J}^{3,(3)}}(\xi,\eta,\zeta) \\
&  =  \mathfrak{p}(\xi,\eta,\zeta) \varphi_{\ell}(\xi-\sqrt{3}) \varphi_{j_1}(\xi-\eta - \sqrt{3})    \varphi_{<\ell-100}(\xi-\eta -\zeta- \sqrt{3}) \\
& \qquad \qquad \qquad \varphi_{\sim j_1}(\xi-\zeta-\sqrt{3}) \varphi_{\sim j_1}(\eta-\zeta) \tau_m(s).
\end{align*}
On the support of this symbol,
$$
| (\partial_\eta - \partial_\zeta) \Psi | \approx |\eta - \zeta| \approx 2^{j_1},
  \qquad | (\partial_\eta - \partial_\zeta)^2 \Psi | \approx 1.
$$
Integrating by parts through the identity
$\frac{1}{is(\partial_\eta - \partial_\zeta) \Psi} (\partial_\eta - \partial_\zeta) e^{is\Psi} = e^{is\Psi}$ gives
\begin{align*}
 \mathcal{J}^{3,(3)}_{m, \ell,j_1} & = \iint e^{is \Psi(\xi,\eta,\zeta)} (\partial_\eta - \partial_\zeta) \left[ \frac{\mathfrak{m}(\xi,\eta,\zeta)}{is(\partial_\eta - \partial_\zeta) \Psi)} \right]  \widetilde{f}(\xi-\eta)   \partial_\xi \widetilde{f}(\xi-\eta-\zeta) \widetilde{f}(\xi-\zeta) \, d\eta \,d\zeta \\
&  \qquad +  \iint e^{is \Psi(\xi,\eta,\zeta)}\frac{1}{is(\partial_\eta - \partial_\zeta) \Psi} \mathfrak{m}(\xi,\eta,\zeta) \widetilde{f}(\xi-\eta)   \partial_\xi \widetilde{f}(\xi-\eta-\zeta) \partial_\xi \widetilde{f}(\xi-\zeta) \, d\eta \,d\zeta \\
& \qquad + \{ \mbox{symmetrical terms} \} \\
& = \mathcal{J}^{3,(3)\flat}_{m,\ell,j_1} + \mathcal{J}^{3,(3)\sharp}_{m,\ell,j_1} + \{ \mbox{symmetrical terms} \}.
\end{align*}
Both of these terms can be estimated very similarly to the corresponding terms in Case 4.1; 
for completeness 
we show how to bound $\mathcal{J}^{3,(3)\sharp}_{m,\ell,j_1}$.

We claim first that 
\begin{equation}
\label{estimatefrakm}
\left\| \mathcal{F} \left[ \frac{\mathfrak{m}(\xi,\eta,\zeta)}{(\partial_\eta - \partial_\zeta) \Psi)} \right] \right\|_{L^1} \lesssim 2^{- j_1}.
\end{equation}
Indeed, 
we can write
$
\mathfrak{m} = \mathfrak{n} \varphi_{< \ell - 100}(\xi-\eta-\zeta- \sqrt 3)
$
with
\begin{align*}
\mathfrak{n}(\xi,\eta,\zeta) & = 
  \mathfrak{p}(\xi,\eta,\zeta) \varphi_{\ell}(\xi-\sqrt{3}) \varphi_{j_1}(\xi-\eta - \sqrt{3})
  \varphi_{\sim j_1}(\xi-\zeta-\sqrt{3}) \varphi_{\sim j_1}(\eta-\zeta) \varphi_{\gtrsim j_1}(\eta-\zeta) \tau_m(s)
\end{align*}
On the support of $\mathfrak{n}$ 
the variables are constrained by $|\xi - \sqrt 3| \sim 2^\ell$, $|\eta| \lesssim 2^{j_1}$, 
and $|\zeta| \lesssim 2^{j_1}$, and 
for any $a,b,c$ we have
$$
\left| \partial_\xi^a \partial_\eta^b \partial_\zeta^c \frac{1}{(\partial_\eta - \partial_\zeta) \Psi} \right| 
  \lesssim 2^{- j_1 (1 + b + c)}, \qquad 
  \left| \partial_\xi^a \partial_\eta^b \partial_\zeta^c \mathfrak{n} \right| \lesssim 2^{- a \ell - (b + c) j_1}.
$$
From Remark \ref{multilinrem} after Lemma \ref{lemmamultilin1}
we deduce
$$
\left\| \mathcal{F} \left[ \frac{\mathfrak{n}(\xi,\eta,\zeta)}{(\partial_\eta - \partial_\zeta) \Psi)} \right] \right\|_{L^1} \lesssim 2^{- j_1}.
$$
hence
\eqref{estimatefrakm}.

Applying Lemma \ref{lemmamultilin1} we obtain
$$
\left\|  \int_0^t  \mathcal{J}_{m,\ell,j_1}^{3,(1)\sharp} \,ds \right\|_{L^2} \lesssim 2^{-j_1} 
\| e^{is\langle D \rangle} \mathcal{W}^* f \|_{L^\infty}
\| \varphi_{<\ell}(\cdot - \sqrt 3) \partial_\xi \widetilde{f} \|_{L^1} \| \partial_\xi \widetilde{f}\|_{L^2} ,
$$
from which, after summing in $j_1 \gg \ell$, one can proceed as in~\eqref{parrot}.


\subsection{The bound for $\mathcal{C}^{S2}_{+-+}$}\label{subsecCS2}
We now look at the `$\pv$' contributions of the form
\begin{align}\label{9CS2}
\begin{split}
\iiint e^{is \Psi(\xi,\eta,\zeta,\theta)} \mathfrak{p}(\xi,\eta,\zeta,\theta) &
  \wt{f}(\xi-\eta) \wt{f}(\xi-\eta-\zeta-\theta) \wt{f}(\xi-\zeta) 
  \frac{\widehat{\phi}(\theta)}{\theta} \, d\eta \,d\zeta \,d\theta 
  \\
\Psi(\xi,\eta,\zeta,\theta) & : = \Phi_{+ - +}(\xi,\xi-\eta,\xi-\eta-\zeta-\theta,\xi-\zeta) 
\\
& = \langle \xi \rangle - \langle \xi - \eta \rangle + \langle \xi - \eta -\zeta -\theta  \rangle 
  - \langle \xi - \zeta\rangle.
\end{split}
\end{align}
In the regime which interests us here ($|\xi - \sqrt{3}| \ll 1$ and $|\eta|, |\zeta|, |\theta| \ll 1$), we have
the expansions
\begin{align*}
& \Psi(\xi,\eta,\zeta,\theta) = - \frac{\xi}{\langle \xi \rangle} \theta + \frac{1}{\langle \xi \rangle^3} 
  \left(\frac{\theta^2}{2} + \eta \zeta +  \eta \theta +  \zeta \theta\right) + O(|\eta,\zeta,\theta|^3) = -\frac{\xi}{\langle \xi \rangle} \theta + O(|\eta,\zeta,\theta|^2),
  \\
& \partial_\xi \Psi (\xi,\eta,\zeta,\theta) = - \frac{1}{\langle \xi \rangle^3} \theta + O(|\eta,\zeta,\theta|^2), \\
& \partial_\eta \Psi (\xi,\eta,\zeta,\theta) = \frac{1}{\langle \xi -\eta \rangle^3} (\zeta + \theta) + O( |\zeta + \theta|^2), \\
& \partial_\eta^2 \Psi (\xi,\eta,\zeta,\theta) = \frac{3(\xi-\eta)}{\langle \xi - \eta \rangle^5} (\zeta + \theta) + O(|\zeta + \theta|^2), \\
& (\partial_\eta - \partial_\zeta) \Psi = \frac{1}{\langle \xi -\eta \rangle^3} (\zeta - \eta) + O( |\eta-\zeta|^2).
\end{align*}

\subsubsection{Commuting $\jxi \partial_\xi$ with~\eqref{9CS2}} 
In order to perform this commutation, it is convenient to adopt a new set of coordinates, namely write 
$$
\eqref{9CS2} = \iiint e^{is \Phi_{+-+}(\xi,\eta,\s,\theta)} \mathfrak{p}(\xi,\eta,\s,\theta) 
  \wt{f}(\eta) \wt{f}(\sigma) \wt{f}(\theta)  \frac{\widehat{\phi}(p)}{p} \, d\eta \,d\sigma \,d\theta, 
$$
with $p = \xi - \eta - \sigma - \theta$. 
In the expression above we have abused notation slightly
by also denoting $\mathfrak{p}$ the symbol in the new coordinates, 
and by omitting irrelevant sign changes.  
To commute $\jxi \partial_\xi$ with \eqref{9CS2}, we will rely on the following elegant identity: 
observe that
\begin{align} \label{niceidentity}
(\jxi \partial_\xi + X_{\eta,\s,\theta}) \Phi_{+-+} = p,
 \qquad X_{\eta,\s,\theta} := \jeta \partial_\eta - \jsig \partial_\sigma
  + \langle \theta \rangle \partial_\theta.
\end{align}
When applying $\jxi \partial_\xi$ to~\eqref{9CS2}, we can use this identity 
to integrate by parts in $\eta,\sigma$ and $\theta$.
Since the adjoint satisfies $X_{\eta,\s,\theta}^\ast = - X_{\eta,\s,\theta}$, 
up to terms which are easier to estimate, we see that estimating $\jxi \partial_\xi \eqref{9CS2}$ reduces to bounding
\begin{subequations}
\begin{align}
\label{elephant1}
& i t \iiint e^{it \Phi_{+-+}} \mathfrak{p}(\xi,\eta,\sigma,\theta) 
\, \wt{f}(\eta) \wt{f}(\sigma)  \wt{f}(\theta) 
  \widehat{\phi}(p) \, d\eta \, d\sigma \,d\theta
\\
\label{elephant2}
& + \iiint e^{it \Phi_{+-+}} \mathfrak{p}(\xi,\eta,\sigma,\theta) 
  X_{\eta,\s,\theta} \big(\wt{f}(\eta) \wt{f}(\sigma)  \wt{f}(\theta) \big)
  \frac{\widehat{\phi}(p)}{p} \, d\eta \, d\sigma \,d\theta
\\
\label{elephant3}
& + \iiint e^{it \Phi_{+-+}} \mathfrak{p}(\xi,\eta,\sigma,\theta) 
\, \wt{f}(\eta) \wt{f}(\sigma)  \wt{f}(\theta) 
  \, \big( \jxi \partial_\xi + X_{\eta,\s,\theta} \big) \Big[ \frac{\widehat{\phi}(p)}{p}\Big] \, d\eta \, d\sigma \,d\theta \\
  \label{elephant4}
  & +  \iiint e^{it \Phi_{+-+}} X_{\eta,\s,\theta} \, \mathfrak{p}(\xi,\eta,\sigma,\theta) 
\, \wt{f}(\eta) \wt{f}(\sigma)  \wt{f}(\theta) 
  \, \frac{\widehat{\phi}(p)}{p} \, d\eta \, d\sigma \,d\theta 
 .
\end{align}
\end{subequations}

\smallskip
\noindent
{\it Estimate of \eqref{elephant1}}.
This term does not have a singularity and can be estimated
integrating by parts in the ``uncorrelated'' variables $\eta,\sigma$ and $\theta$.
Each of the three inputs then would gives a gain of $\jt^{-3/4+\alpha}$ which is sufficient to
absorb the power of $t$ in front and integrate over time.
Similar (in fact, harder) terms have been treated in Section \ref{secwR}, so we can skip the details.

\smallskip
\noindent
{\it Estimate of \eqref{elephant3}}.
For this term we observe, see \eqref{Cubotherid}, that
\begin{align}\label{niceidentity2}
\big( \jxi \partial_\xi + X_{\eta,\s,\theta} \big) \Big[ \frac{\widehat{\phi}(p)}{p}\Big] 
  = \Phi_{+-+}(\xi,\eta,\s,\theta) \partial_p \Big[ \frac{\widehat{\phi}(p)}{p}\Big].
\end{align}
Note that this identity is formal as it is written, since $\partial_p (1/p)$ does not converge (even in the $\pv$ sense);
however, it can be made rigorous by localizing a little away from $p=0$, 
and using the $\pv$ to deal with very small values of $p$.

From \eqref{niceidentity2} we obtain, upon integration by parts in $s$, that
\begin{align}\label{rhinoceros1}
\int_0^t i \eqref{elephant3} \, ds & = 
\iiint e^{is \Phi_{+-+}} \mathfrak{p}(\xi,\eta,\sigma,\theta)
\, \wt{f}(\eta) \wt{f}(\sigma)  \wt{f}(\theta) 
  \, \partial_p \frac{\widehat{\phi}(p)}{p}\, d\eta \, d\sigma \,d\theta \, \Big|_{s=0}^{s=t}
\\
\label{rhinoceros2}
& - \int_0^t \iiint e^{is \Phi_{+-+}} \mathfrak{p}(\xi,\eta,\sigma,\theta)
\, \partial_s \Big[ \wt{f}(\eta) \wt{f}(\sigma)  \wt{f}(\theta) \Big]
  \partial_p \frac{\widehat{\phi}(p)}{p} \, d\eta \, d\sigma \,d\theta \, ds.
\end{align}
To estimate \eqref{rhinoceros1} we convert the $\partial_p$ into $\partial_\eta$ and
integrate by parts in $\eta$. The worst term is when $\partial_\eta$ hits the exponential;
this causes a loss of $t$ but an $L^2\times L^\infty\times L^\infty$ H\"older estimate using Lemma \ref{lemCS}
suffices to recover it.

The term \eqref{rhinoceros2} is similar. We may assume that $\partial_s$ hits $\wt{f}(\sigma)$.
Again we convert $\partial_p$ into $\partial_\eta$ and
integrate by parts in $\eta$.
This causes a loss of $s$ when hitting the exponential phase
which is offset by an $L^\infty\times L^2\times L^\infty$ estimate
with $\partial_s \wt{f}$ placed in $L^2$ and giving $\jt^{-1}$ decay using \eqref{dtfL2}.

\smallskip
\noindent
{\it Estimate of \eqref{elephant4}}. 
This term can be estimated directly using the trilinear estimates from Lemma \ref{lemCS}.

\bigskip
This leaves us with \eqref{elephant2}, which, 
coming back to the original coordinates, and taking into account the symmetry between the $\eta$ and $\zeta$ variables, 
reduces to the two following terms
\begin{align*}
&\mathcal{H}^2 = \iiint e^{is \Psi(\xi,\eta,\zeta,\theta)} \mathfrak{p}(\xi,\eta,\zeta,\theta)  \partial_\xi \widetilde{f}(\xi-\eta) \widetilde{f}(\xi-\eta-\zeta-\theta) \widetilde{f}(\xi-\zeta) \frac{\widehat{\phi}(\theta)}{\theta}\, d\eta \,d\zeta\,d\theta, \\
&\mathcal{H}^3 = \iiint e^{is \Psi(\xi,\eta,\zeta,\theta)} \mathfrak{p}(\xi,\eta,\zeta,\theta)  \widetilde{f}(\xi-\eta) \partial_\xi \widetilde{f}(\xi-\eta-\zeta-\theta) \widetilde{f}(\xi-\zeta) \frac{\widehat{\phi}(\theta)}{\theta}\, d\eta \,d\zeta\,d\theta. \\
\end{align*}
In order to bound these terms, we will localize dyadically  the variables in the problem as follows:
\begin{equation} \label{defljm2}
\begin{split}
& s \approx 2^m, \quad |\xi - \sqrt{3}| \approx 2^\ell, \quad |\theta| \approx 2^h,
\\
& |\xi-\eta - \sqrt{3}| \approx 2^{j_1},
\quad |\xi-\eta-\zeta -\theta- \sqrt{3}| \approx 2^{j_2},
\quad |\xi-\zeta-\sqrt{3}| \approx 2^{j_3}.
\end{split}
\end{equation}

\medskip
\subsubsection{Bound for $\mathcal{H}^2$}\label{secwLpv2} 

\noindent{\bf Cases 1,2,3: $\ell < -\gamma m$ or $j_1 > \ell - 100$.}
These cases can be dealt with as in \S\ref{sectionJ2}, 
relying on Lemma \ref{lemmamultilin2} instead of Lemma~\ref{lemmamultilin1}.

\medskip
\noindent{\bf Case 4.1: $\ell > -\gamma m$, $j_1 < \ell - 100$ and $|\theta+\xi-\sqrt 3| \geq 2^{\ell-10} $}. 
We want to bound here $\sum_{h > \ell - 10} \mathcal{H}^{2,(1)}_{m,\ell,h}$, where
\begin{align*}
\mathcal{H}^{2,(1)}_{m,\ell,h} := 
  \iiint e^{is \Psi(\xi,\eta,\zeta,\theta)} \mathfrak{m}(\xi,\eta,\zeta,\theta)
    \partial_\xi\wt{f}(\xi-\eta) 
    \wt{f}(\xi-\eta-\zeta-\theta) \wt{f}(\xi-\zeta) \frac{\widehat{\phi}(\theta)}{\theta}\, d\eta \,d\zeta \,d\theta
\end{align*}
with
\begin{align*}
 \mathfrak{m}(\xi,\eta,\zeta,\theta) & = \mathfrak{m}^{\mathcal{H}^{2,(1)}}_{m,\ell,h}(\xi,\eta,\zeta,\theta) 
 \\
& =  \mathfrak{p}(\xi,\eta,\zeta) \varphi_\ell(\xi - \sqrt{3}) \varphi_{<\ell-100}(\xi-\eta - \sqrt{3}) \varphi_{h} (\theta + \xi - \sqrt{3}) \tau_m(s).
\end{align*}
On the support of $\mathfrak{m}(\xi,\eta,\zeta)$, $|\partial_\zeta \Psi| \approx |\eta + \theta| \gtrsim |\eta|$,
and $|\xi - \sqrt{3}|, |\eta| \sim 2^\ell$ and $|\theta| \lesssim 2^h$;
moreover, for any $a$, $b$, $c$,
$$
\left| \partial_\xi^a \partial_\eta^b \partial_\zeta^c \partial_\theta^d \frac{1}{\partial_\zeta \Psi} \right| 
  \lesssim 2^{-h(1+b+d)}
$$
Therefore, Lemma \ref{lemmamultilin2} applies and we can proceed exactly as in Case 4 of 
\S\ref{sectionJ2},
since the sum over $h > \ell$ of $2^{-h}$ gives the same factor of $2^{\ell}$ there.

\medskip
\noindent {\bf Case 4.2: $\ell > -\gamma m$, $j_1 < \ell - 100$ and $|\theta+\xi-\sqrt 3| < 2^{\ell - 10}$.} 
We want to bound here
\begin{align*}
\mathcal{H}^{2,(2)}_{m,\ell} = \iiint e^{is \Psi(\xi,\eta,\zeta,\theta)} \mathfrak{m}(\xi,\eta,\zeta)\partial_\xi\widetilde{f}(\xi-\eta)  \widetilde{f}(\xi-\eta-\zeta-\theta) \widetilde{f}(\xi-\zeta) \frac{\widehat{\phi}(\theta)}{\theta}\, d\eta \,d\zeta \,d\theta
\end{align*}
where
\begin{align*}
 \mathfrak{m}(\xi,\eta,\zeta) & = \mathfrak{m}^{\mathcal{H}^{2,(2)}}_{m,\ell}(\xi,\eta,\zeta) \\
& = \mathfrak{p}(\xi,\eta,\zeta) \varphi_\ell(\xi - \sqrt{3}) \varphi_{<\ell-100}(\xi-\eta - \sqrt{3})  \varphi_{<\ell - 10} (\theta + \xi - \sqrt{3}) \tau_m(s) .
\end{align*}
On the support of the integrand, $|\theta| \approx 2^\ell$. 
Therefore, noticing that $\partial_\theta \Psi$ is smooth and bounded away from zero, 
we integrate by parts in $\theta$ to obtain
\begin{align*}
\mathcal{H}^{2,(2)}_{m,\ell} & =  \iiint  e^{is \Psi(\xi,\eta,\zeta,\theta)} \partial_\xi\widetilde{f}(\xi-\eta) \widetilde{f}(\xi-\eta-\zeta-\theta) \widetilde{f}(\xi-\zeta) \partial_\theta \left[ \frac{\mathfrak{m}}{s\partial_\theta \Psi} \frac{\widehat{\phi}(\theta)}{\theta} \right] \, d\eta \,d\zeta \,d\theta \\
& \qquad +  \iiint  e^{is \Psi(\xi,\eta,\zeta,\theta)} \frac{\mathfrak{m}}{s\partial_\theta \Psi} \partial_\xi\widetilde{f}(\xi-\eta)   \partial_\xi \widetilde{f}(\xi-\eta-\zeta-\theta) \widetilde{f}(\xi-\zeta) \frac{\widehat{\phi}(\theta)}{\theta}\, d\eta \,d\zeta \,d\theta \\
& = \mathcal{H}^{2,(2) \flat}_{m,\ell} + \mathcal{H}^{2,(2) \sharp}_{m,\ell}.
\end{align*}
Using that $|\xi|, |\eta|, |\theta| \lesssim 2^\ell$ 
and $| \partial_\xi^a \partial_\eta^b \partial_\zeta^c \partial_\theta^d \mathfrak{m}| \lesssim 2^{-\ell (a+b+d)}$, 
we get by Lemma~\ref{lemmamultilin2},~\eqref{lembb2},~\eqref{lembb3}, and~\eqref{aprioridecay},
\begin{align*}
\int_0^t \left\| \mathcal{H}^{2,(2) \flat}_{m,\ell}  \right\|_{L^2} \,ds 
  & \lesssim 2^{-\ell} \| \varphi_{<\ell}(\cdot - \sqrt{3}) \partial_\xi \widetilde f \|_{L^1} 
  \| \widetilde{f} \|_{L^2} \| e^{-it\langle D \rangle} \mathcal{W}^* f \|_{L^\infty} 
  \\
& \lesssim 2^{-\ell} 2^{\beta' \ell + \alpha m} 2^{-m/2} \e_1^3 
  \lesssim 2^{-\beta \ell + \alpha m} \e_1^3.
\end{align*}
Similarly, by Lemma~\ref{lemmamultilin2},~\eqref{lembb1},~\eqref{lembb3}, and~\eqref{aprioridecay},
\begin{align*}
\int_0^t \left\| \mathcal{H}^{2,(2) \sharp}_{m,\ell}  \right\|_{L^2} \,ds &
  \lesssim \| \varphi_{<\ell}(\cdot - \sqrt{3}) \partial_\xi \widetilde f \|_{L^1} \|  \partial_\xi \widetilde f \|_{L^2}  
  \| e^{-it\langle D \rangle} \mathcal{W}^* f \|_{L^\infty} 
  \\
& \lesssim 2^{\beta' \ell + \alpha m} 2^{(\alpha + \beta \gamma)m} 2^{-m/2} \e_1^3 
  \lesssim 2^{-\beta \ell + \alpha m}  \e_1^3.
\end{align*}

\subsubsection{Bound for $\mathcal{H}^3$} 

\noindent
{\bf Cases 1,2,3: $\ell < -\gamma m$ or $j_2 > \ell - 100$.}
With the help of Lemma \ref{lemmamultilin2} instead of Lemma \ref{lemmamultilin1}, 
these cases are dealt with exactly as for $\mathcal{J}^3$ in \S\ref{boundJ3}.
We introduce one further index to record the size of $|\eta - \zeta|$:
\begin{align*}
& |\eta - \zeta| \approx 2^{j_4}.
\end{align*}

\medskip

From now on, we can assume that $\ell > -\gamma m$ and $j_1 < \ell - 100$. The cases that remain to be distinguished are
\begin{itemize}
\item Case 4.1: $2^{j_1},2^{j_3} \lesssim 2^\ell$, $2^{j_4} \gtrsim 2^\ell$;
\item Case 4.2.1: $2^{j_1},2^{j_3} \lesssim 2^\ell$, $2^{j_4} \ll 2^\ell$, $2^{2\ell} \lesssim 2^h \lesssim 2^\ell$;
\item Case 4.2.2: $2^{j_1},2^{j_3} \lesssim 2^\ell$, $2^{j_4} \ll 2^\ell$, $2^h \ll 2^{2\ell}$;
\item Case 4.3.1: $2^{j_1} \approx 2^{j_3} \gg 2^\ell$, $2^h \ll 2^{j_1}$;
\item Case 4.3.2: $2^{j_1} \gg 2^\ell$, $2^h \gtrsim 2^{j_1}$.
\end{itemize}

\medskip
\noindent {\bf Case 4.1: $\ell > -\gamma m$, $j_2 < \ell - 100$, $2^{j_1}, 2^{j_3} \lesssim 2^\ell$ 
and $j_4 > \ell-100$.} 
This corresponds to the symbol
\begin{align*}
\mathfrak{m}_{\ell,m}^{\mathcal{H}^{3,(1)}} & = \varphi_{\ell}(\xi - \sqrt{3}) 
  \varphi_{< \ell +10}(\xi-\eta - \sqrt{3}) \varphi_{< \ell - 100}(\xi-\eta -\zeta - \theta -\sqrt{3}) \\
& \qquad \qquad \qquad \qquad \varphi_{< \ell+10}(\xi-\zeta- \sqrt{3}) \varphi_{>\ell-100}(\eta-\zeta).
\end{align*}
On the support of this symbol, the frequency variables enjoy the localizations $|\xi-\sqrt{3}|, |\eta|, |\zeta|, |\theta| \lesssim 2^\ell$, and the symbol satisfies the estimates
$$
| \partial_\xi^a \partial_\eta^b \partial_\zeta^c \partial_\theta^d \mathfrak{m} | \lesssim 2^{-(a+b+c+d) \ell}.
$$
Furthermore, $|(\partial_\eta - \partial_\zeta) \Psi| \gtrsim 2^\ell$, 
and therefore one can proceed as in Case 4.1 of \S\ref{boundJ3}. 
Indeed, $(\partial_\eta - \partial_\zeta) \Psi$ is independent of $\theta$
and, on the support of the symbol, 
$\left| \partial_\xi^a \partial_\eta^b \partial_\zeta^c \partial_\theta^d \frac{1}{(\partial_\eta - \partial_\zeta) \Psi} 
\right| \lesssim 2^{-\ell(1+b+c)}$ for any $a,b,c,d$.

\medskip
\noindent
{\bf Case 4.2.1: $\ell >  -\gamma m$, $j_2 < \ell - 100$, $2^{j_1}, 2^{j_3} \lesssim 2^\ell$, 
$j_4 \leq \ell - 100$, and $h > 2\ell -100$.} 
After the change of variables $\eta' = \eta + \theta$, let
$$
\mathcal{H}^{3,(2)}_{h,\ell,m} = \iiint e^{is \Phi(\xi,\eta',\zeta,\theta)} \mathfrak{m}(\xi,\eta',\zeta,\theta)  \widetilde{f}(\xi-\eta'+\theta) \partial_\xi\overline{\widetilde{f}(\xi-\eta'-\zeta)} \widetilde{f}(\xi-\zeta) \frac{\widehat{\phi}(\theta)}{\theta}\, d\eta' \,d\zeta\,d\theta,
$$
where
\begin{align*}
 \Phi(\xi,\eta',\zeta,\theta) = \langle \xi \rangle - \langle \xi - \eta' + \theta \rangle + \langle \xi - \eta' -\zeta \rangle - \langle \xi - \zeta\rangle
\end{align*}
and
\begin{align*}
 \mathfrak{m}(\xi,\eta',\zeta) & = \mathfrak{m}^{\mathcal{H}^{3,(2)}}_{h,\ell,m}(\xi,\eta',\zeta) \\
& = \varphi_{\ell}(\xi - \sqrt{3}) \varphi_{<\ell-100}(\xi-\eta' -\zeta - \sqrt{3}) \varphi_{< \ell + 10}(\xi-\eta' + \theta -\sqrt{3})\varphi_{< \ell + 10}(\xi-\zeta- \sqrt{3}) \\
& \qquad \qquad \qquad \varphi_{<\ell-100} (\eta'-\theta-\zeta) \varphi_{h} (\theta) \tau_m(s) \mathfrak{p}(\xi,\eta' - \theta,\zeta).
\end{align*}
On the support of this symbol, $2^h \lesssim 2^\ell$; therefore, in the following we will bound
$\sum_{ 2\ell - 100 < h < \ell + 100} \mathcal{H}^{3,(2)}_{h,\ell,m}$.
Noticing that $\partial_\theta \Phi = \frac{ \xi - \eta' - \theta }{\langle \xi - \eta' - \theta \rangle}$ 
is smooth and $\gtrsim 1$, we integrate by parts in $\theta$, to obtain
\begin{align*}
\mathcal{H}^{3,(2)}_{h,\ell,m} & = \iiint e^{is \Phi(\xi,\eta',\zeta,\theta)} \frac{\mathfrak{m}}{i s\partial_\theta \Phi} \partial_\xi \widetilde{f}(\xi-\eta'+\theta) \partial_\xi \widetilde{f}(\xi-\eta'-\zeta) \widetilde{f}(\xi-\zeta) \frac{\widehat{\phi}(\theta)}{\theta}\, d\eta' \,d\zeta\,d\theta \\
& \qquad -  \iiint e^{is \Phi(\xi,\eta',\zeta,\theta)} \widetilde{f}(\xi-\eta'+\theta) \partial_\xi \widetilde{f}(\xi-\eta'-\zeta) \widetilde{f}(\xi-\zeta) \partial_\theta \left[ \frac{\mathfrak{m}}{i s\partial_\theta \Phi}    \frac{\widehat{\phi}(\theta)}{\theta}\right]\, d\eta' \,d\zeta\,d\theta \\
& = \mathcal{H}^{3,(2)\flat}_{h,\ell,m} + \mathcal{H}^{3,(2)\sharp}_{h,\ell,m}.
\end{align*}
Estimating $ \mathcal{H}^{3,(2)\flat}_{h,\ell,m}$ is now straightforward, 
using that $|\xi-\sqrt 3|, |\eta'|, |\zeta| \lesssim 2^\ell$, $|\theta| \approx 2^h$ 
and $|\partial_\xi^a \partial_{\eta'}^b \partial_\zeta^c \partial_\theta^d \mathfrak{m}| \lesssim 2^{-(a+b+c) \ell - dh}$:
\begin{align*}
\int_0^t \sum_{ 2\ell-100 < h < \ell+100} 
  \left\|  \mathcal{H}^{3,(2)\flat}_{h,\ell,m} \right\|_{L^2}\,ds & \lesssim 
  |\ell| {\| \varphi_{\sim \ell}(\cdot - \sqrt{3}) \partial_\xi \widetilde{f} \|}_{L^2} 
  {\| \varphi_{<\ell} (\cdot - \sqrt{3}) \partial_\xi \widetilde{f} \|}_{L^1} 
  {\| e^{-it\langle D \rangle} \mathcal{W}^* f \|}_{L^\infty} 
  \\
& \lesssim |\ell|  2^{-\beta \ell + \alpha m} 2^{\beta' \ell + \alpha m} 2^{-m/2} \e_1^3 
  \lesssim 2^{-\beta \ell + \alpha m} \e_1^3.
\end{align*}
To estimate $\mathcal{H}^{3,(2)\sharp}_{h,\ell,m}$, observe that
$$
\left\| \mathcal{F} \left( \partial_\theta \left[ \frac{ \mathfrak{m}}{\partial_\theta \Phi} 
  \frac{\widehat{\phi}(\theta)}{\theta} \right] \right) \right\|_{L^1} \lesssim 2^{-h}
$$
and, therefore,
\begin{align*}
& \int_0^t \left\| \sum_{2\ell-100 < h <\ell+100} \mathcal{H}^{3,(2)\sharp}_{h,\ell,m} \right\|_{L^2}\,ds 
\\
& \qquad \lesssim \sum_{2\ell-100 < h <\ell+100}  2^{- h} 
  \| \varphi_{<\ell} (\cdot - \sqrt 3) \widetilde{f} \|_{L^2} 
  \| \varphi_{<\ell} (\cdot - \sqrt{3}) \partial_\xi \widetilde{f} \|_{L^1} 
  \| e^{-it\langle D \rangle} \mathcal{W}^* f \|_{L^\infty} 
  \\
& \qquad \lesssim \sum_{2\ell-100 < h <\ell+100} 2^{- h} 
  2^{\ell/2} 2^{\beta' \ell + \alpha m} 2^{-m/2} \e_1^3 \lesssim 2^{-\beta \ell + \alpha m} \e_1^3,
\end{align*}
where we used that $\ell > -\gamma m > -m/2$.

\medskip
\noindent {\bf Case 4.2.2: $\ell >  -\gamma m$, $j_2 < \ell - 100$, $2^\ell \sim 2^{j_1} \sim 2^{j_3}$, 
$j_4 < \ell - 100$ and $h < 2\ell-100$.} 
In this case, $|\eta| , |\zeta| \approx \left| \xi - \sqrt 3 \right| \approx 2^\ell$ 
and $|\theta| \ll 2^\ell$, so that $|\Psi| \approx 2^{2 \ell}$, $|\partial_\eta \Psi| \approx 2^\ell$, and,
for any $a,b,c,d$,
$$
\left| \partial_\xi^a \partial_\eta^b \partial_\zeta^c \partial_\theta^d \frac{\partial_\eta \Psi}{\Psi} \right| 
  \lesssim 2^{-\ell(1+b+c)-2\ell d}.
$$
Therefore, this case can be dealt with as in Case 4.2 of \S\ref{boundJ3}.

\medskip
\noindent {\bf Case 4.3.1: $\ell >  -\gamma m$, $j_2 < \ell - 100$, $2^\ell \ll 2^{j_1} \approx 2^{j_3}$, $2^h \ll 2^{j_1}$.}  
This corresponds to the symbol
\begin{align*}
\mathfrak{m}(\xi,\eta,\zeta,\theta) = \mathfrak{m}_{\ell,j_1,m}^{\mathcal{H}^{3,(3)}}(\xi,\eta,\zeta,\theta) 
  & = \varphi_{\ell}(\xi - \sqrt{3}) \varphi_{j_1}(\xi-\eta - \sqrt{3}) 
  \varphi_{<\ell-100}(\xi-\eta -\zeta - \theta -\sqrt{3}) \\
& \qquad \qquad \qquad \qquad \varphi_{\sim j_1}(\xi-\zeta- \sqrt{3}) \varphi_{\sim j_1}(\eta - \zeta) \varphi_{<j_1 - 100} (\theta)\tau_m(s).
\end{align*}
which is such that, on its support, 
$|\eta|,|\zeta|, |\eta - \zeta| \approx 2^{j_1}$ and $|\theta| \ll 2^{j_1}$.

This case can mostly be treated like Case 4.3 in \S\ref{boundJ3} 
since $(\partial_\eta - \partial_\zeta) \Psi$ is independent of $\theta$. 
The only delicate point is that Lemma \ref{lemmamultilin2}, and the remark following it, 
do not directly apply here. For this reason, we let
$
\mathfrak{m}(\xi,\eta,\zeta,\theta) = \mathfrak{n}(\xi,\eta,\zeta,\theta)
  \varphi_{<\ell-100}(\xi-\eta -\zeta - \theta -\sqrt{3}),
$
and consider the symbol 
$$
\mathfrak{m}_1(\xi,\eta,\zeta,\theta) = \frac{\mathfrak{m}(\xi,\eta,\zeta,\theta)}{(\partial_\eta - \partial_\zeta) \Psi}
  = \frac{\mathfrak{n}(\xi,\eta,\zeta,\theta)}{(\partial_\eta - \partial_\zeta) \Psi} 
  \varphi_{<\ell-100}(\xi-\eta -\zeta - \theta -\sqrt{3}),
$$ 
which appears if one follows the proof of Case 4.3 in $\mathcal{J}^3$. 
Following the same argument used to show \eqref{estimatefrakm},
one can see that
$$
\left| \mathcal{F}  \left[\frac{\mathfrak{n}}{(\partial_\eta - \partial_\zeta) \Psi} \right](x,y,z,t) \right| 
  \lesssim 2^{\ell + 2 j_1} F(2^\ell x) F(2^{j_1} (y,z,t)),
$$
where we denote $F$ for a generic rapidly decaying function of size $O(1)$
together with its derivatives.
The Fourier transform of $\mathfrak{m}_1$ is bounded by the convolution of the above right-hand side with 
$$
\left| \mathcal{F}  \varphi_{<\ell-100}(\xi-\eta -\zeta - \theta -\sqrt{3}) \right| 
= 2^{\ell} F(2^\ell x) \delta(-x = y = z = t),
$$
that is,
$$
| \widehat{\mathfrak{m}_1}(x,y,z,t) | \lesssim 2^{2\ell + j_1} F(2^\ell x) F(2^\ell y) F(2^{j_1}(y-z)) F(2^{j_1}(y-t)).
$$
Therefore, 
$$
\left| \int \widehat{\mathfrak{m}_1} (x,y,z,t) \,dt \right|
  \lesssim 2^{2\ell} F(2^\ell x) F(2^\ell y) F(2^{j_1}(y-z))
$$
and Lemma~\ref{lemmamultilin2} applies, giving that the norm of $V_{\mathfrak{m}_1}$ 
between Lebesgue spaces at the H\"older scaling is $\lesssim 2^{-j_1}$.

\medskip
\noindent {\bf Case 4.3.2: $\ell >  -\gamma m$, $j_2 < \ell - 100$, $2^\ell \ll 2^{j_1}$, $2^h \gtrsim2^{j_1}$.}  
This corresponds to the symbol
\begin{align*}
\mathfrak{m}(\xi,\eta,\zeta,\theta) = \mathfrak{m}_{\ell,j_1,m}^{\mathcal{H}^{3,(3)}}(\xi,\eta,\zeta,\theta) 
  & = \varphi_{\ell}(\xi - \sqrt{3}) \varphi_{j_1}(\xi-\eta - \sqrt{3}) 
  \varphi_{<\ell-100}(\xi-\eta -\zeta - \theta -\sqrt{3}) \varphi_h (\theta)\tau_m(s).
\end{align*}
In this case, it is possible to integrate by parts in $\theta$, and argue exactly as in Case 4.2.1. 
The only thing to check is that Lemma~\ref{lemmamultilin2} applies; 
therefore, we need to bound the Fourier transform of $\mathfrak{m}$. A computation reveals that
$$
\left| \widehat{\mathfrak{m}}(x,y,z,t) \right| \lesssim 2^{2\ell + j_1 + h} F(2^\ell z) F(2^\ell(x+z)) F(2^{j_1}(y-z)) F(2^h(t-z)),
$$
so that
$$
\left| \int \widehat{\mathfrak{m}_1} (x,y,z,t) \,dt \right| \lesssim  2^{2\ell + j_1 } F(2^\ell z) F(2^\ell(x+z)) F(2^{j_1}(y-z)).
$$
By Lemma~\ref{lemmamultilin2}, the norm of $V_{\mathfrak{m}}$ between Lebesgue spaces at the H\"older scaling is $\lesssim 1$.



\bigskip
\section{Pointwise estimates for the ``singular'' part}\label{secLinfS}

The aim of this section is to prove the following proposition:

\begin{prop}\label{propLinfS}
Under the assumptions of Theorem \ref{maintheo}, consider $u$
solution of \eqref{KG} 
and assume the a priori bounds \eqref{propbootfas} on the renormalized profile $\wt{f}$.
Let $\mathcal{C}^{S} = \mathcal{C}^{S1} + \mathcal{C}^{S2}$ be the cubic singular terms 
defined in \eqref{CubicS}-\eqref{CubicS12} with \eqref{formulacubiccoeff}.

Denote 
\begin{align}\label{notfxi}
X(\xi) = (X_+(\xi), X_-(\xi)) := (\wt{f}(\xi), \wt{f}(-\xi)), \qquad \xi >0.
\end{align}
Then

\begin{itemize}
\item There exists a real valued Hamiltonian $H = H(X_+,X_-)$ 
such that, for $\xi>0$,
\begin{align}\label{LinfSasy}
\begin{split}
\mathcal{C}^{S}(f,f,f)(t,\xi) & = - \frac{i}{t} \, \frac{d}{d\bar{X_+}} H + R_+(t,\xi),
\\
\mathcal{C}^{S}(f,f,f)(t,-\xi) & = - \frac{i}{t} \, \frac{d}{d\bar{X_-}} H + R_-(t,\xi),
\end{split}
\end{align}
for all $t \geq 1$; see \eqref{asyHam0} for the exact formula for $H$.

\medskip
\item There exists $\delta_0 >0$ such that the remainders satisfy, for all $m=0,1,\dots$,
\begin{align}\label{LinfSrem}
{\Big\| \jxi^{3/2} \int_0^t R_\eps (s,\xi)\, \tau_m(s) ds \Big\|}_{L^\infty_\xi}
  \lesssim \e_1^3 2^{-\delta_0m}, \qquad \eps\in\{+,-\}.
\end{align}


\medskip
\item 
There exists an asymptotic profile $W^\infty = (W^\infty_+, W^\infty_-) \in \big(\jxi^{-3/2}L^\infty_\xi\big)^2$ 
such that, for all $t\geq 0$
\begin{align}\label{LinfSasyf}
\begin{split}
& \jxi^{3/2} \Big| X(t,\xi)
  \\ & - S^{-1}(\xi) \exp\Big( -(5 i/12) \,(\log t) \, 
  \mathrm{diag}\big( \ell_{+ \infty}^2\big| W^{\infty}_+ \big|^2, \ell_{- \infty}^2\big| W^{\infty}_- \big|^2 \big) \Big) 
  \,W^\infty(\xi)  \Big| 
  \\ & \qquad \qquad \qquad \lesssim \e_1^3 \jt^{-\delta_0},
\end{split}
\end{align}
where $S(\xi)$ is the scattering matrix associated to the potential $V$ defined in \eqref{scatmat}.

\end{itemize}
\end{prop}

\medskip
Here are a few remarks about the statement above and its consequences:

\begin{itemize}

\medskip
\item 
From the inequality \eqref{asyZbound} appearing in the proof of 
\eqref{LinfSasyf} we deduce in particular the bound ${\| \jxi^{3/2}\wt{f} \|}_{L^\infty} \leq \e_1$;
this is one of the bounds needed for the main bootstrap Proposition \ref{propbootf}.

\medskip
\item The behavior at negative times can be obtained by using time-reversal symmetry.
In the context of our distorted Fourier space asymptotics this ends up involving 
a conjugation by the scattering matrix $S$. 
We refer the reader to the explicit calculation in \cite[Remark 1.2]{ChPu}.
When applied to \eqref{LinfSasyf} this conjugation will simplify the formula and give 
asymptotics akin to the flat ones that do not involve $S$;
see \eqref{AsyFlat}.

\medskip
\item It is of course possible to derive asymptotics in physical space from the asymptotics in Fourier space,
both in $L^\infty_x$ and $L^2_x$.
This can be done by putting together the linear asymptotics \eqref{remlinear} with \eqref{LinfSasyf},
passing from $f$ to $g$ using \eqref{Renof} (note that $T(g,g)$ is a fast decaying remainder in $L^\infty_x$)
and eventually passing from $g$ to $u$ using \eqref{mtprofile}.
We refer the reader to \cite{IoPunote} for the details of such an argument in the context of water waves
and to \cite{HN} for the Nonlinear Schr\"odinger equation.

\end{itemize}

 

\medskip
Let us show first how to obtain the final asymptotic formula \eqref{LinfSasyf}
given \eqref{LinfSasy} and \eqref{LinfSrem}.
The argument is fairly standard as it appears in similar forms 
in \cite{HN,KP,IoPu1,GPR2}. We refer the reader to these papers for more detailed presentations.

\begin{proof}[Proof of \eqref{LinfSasyf}]
Recall that the evolution of $\wt{f}$ is given by \eqref{eqdtf}. 
In Subsection \ref{SsecLinfR}, we show that all the terms on the right-hand side of \eqref{eqdtf},
with the exception of the cubic terms $\mathcal{C}^S$, 
satisfy bounds of the same type as the remainders in \eqref{LinfSrem};
see in particular Propositions \ref{proQRLinftyxi} and \ref{proCSLinftyxi}.
Then, with the notation \eqref{notfxi}, the asymptotitcs  \eqref{LinfSasy}-\eqref{LinfSrem} imply, for $\xi>0$,
\begin{align}\label{LinfSasy'}
\begin{split}
\partial_t X_+ & = - \frac{i}{t} \, \frac{\partial}{\partial\bar{X_+}} H + R_1(t,\xi),
\\
\partial_t X_- & = - \frac{i}{t} \, \frac{\partial}{\partial\bar{X_-}} H + R_2(t,\xi),
\end{split}
\end{align}
with $R_i$ satisfying bounds as in \eqref{LinfSrem}.
In the rest of the proof of \eqref{LinfSasyf} 
we will denote just by $R=R(t,\xi)$ any generic remainder terms satisfying \eqref{LinfSrem}.
This should not be confused with $R_+$ and $R_-$, the transmission coefficients.
Note that such a bound implies that $R(t)$ has a well defined anti-derivative which is uniformly
bounded in time in $\jxi^{-3/2}L^\infty_\xi$.


\eqref{LinfSasy'} with \eqref{asyHam0} can be written as
\begin{align*}
\begin{split}
\partial_t X_+ = & - \frac{5i}{12 \, t} \Big[ \ell_{+ \infty}^2 \big| (S X)_1 \big|^2 (S X)_1 \bar{T}(\xi)
	+  \ell_{- \infty}^2 \big| (S X)_2 \big|^2 (S X)_2 \bar{R_-}(\xi) \Big] + R(t,\xi),
\\
\partial_t X_- = & - \frac{5i}{12 \, t} \Big[ \ell_{+ \infty}^2 \big| (S X)_1 \big|^2 (S X)_1 \bar{R_+}(\xi)
 	 +  \ell_{- \infty}^2 \big| (S X)_2 \big|^2 (S X)_2 \bar{T}(\xi) \Big] + R(t,\xi),
\end{split}
\end{align*}
where $S$ is the scattering matrix \eqref{scatmat}.
If we denote $(Z_+(\xi),Z_-(\xi))^t := S(\xi) (X_+(\xi), X_-(\xi))^t$, and use \eqref{TR},
this simplifies to give
\begin{align}\label{asyHamevS}
\begin{split}
\partial_t Z_{\pm} = & - \frac{5i}{12\, t} \ell_{\pm \infty}^2 |Z_{\pm}|^2 Z_{\pm} + R(t,\xi).
\end{split}
\end{align}

Defining the modified profile $W=(W_+(\xi),W_-(\xi))^t$ by
\begin{align}\label{asymodprof}
\begin{split}
& W_{\pm}(t,\xi) := 
  \exp\Big( \frac{5i}{12} \ell_{\pm \infty}^2\int_0^t |Z_{\pm}(s,\xi)|^2 \, \frac{ds}{s+1}  \Big) \, Z_\pm(t,\xi),
\\
& |W_\pm(t,\xi)|=|Z_\pm(t,\xi)|,
\end{split}
\end{align}
we see that 
\begin{align}\label{asydtW}
\partial_t W_{\pm}(t,\xi) = \exp\Big( \frac{5i}{12} \ell_{\pm \infty}^2 
  \int_0^t |Z_{\pm}|^2\, \frac{ds}{s+1}  \Big) \, R(t,\xi).
\end{align}

In particular this implies that 
$\partial_t |W_{\pm}|^2 = 2\mathrm{Re}\big(R(t,\xi) \bar{Z_{\pm}}\big)$ and therefore,
since $S$ is unitary,
\begin{align}\label{asyZbound}
\begin{split}
\jxi^{3}\Big( |\wt{f}(t,\xi)|^2 - |\wt{f}(0,\xi)|^2 \Big) 
  \lesssim \Big| \mathrm{Re} \int_0^t \jxi^{3/2} R(s,\xi) \cdot \jxi^{3/2}\bar{Z_{\pm}}(s,\xi) \, ds \Big| 
  \lesssim \e_1^4.
\end{split}
\end{align}
For this last inequality we have used the bounds \eqref{LinfSrem} (which in particular imply that 
$R$ has a well defined anti-derivative), integration by parts in $s$, the equation \eqref{asyHamevS}, 
and the a priori assumption $|\jxi^{3/2}Z_\pm(t,\xi)| \lesssim \e_1$.

Similarly, from \eqref{asydtW}, using that the remainders satisfy estimates like \eqref{LinfSrem},
integrating by parts in $s$ and using that  the time derivative of the exponential factor is $O(\e_1^2 \jt^{-1})$, 
we can see that, for all $0<t_1<t_2$,
\begin{align*}
\jxi^{3/2} \big| W(t_1,\xi) - W(t_2,\xi) \big| \lesssim 
	\e_1^3 \, t_1^{-\delta_0}.
\end{align*}
By letting $W_\eps^\infty(\xi) := \lim_{t \rightarrow \infty} W_\eps(t,\xi)$ in the space $\jxi^{-3/2}L^\infty$, it follows that
\begin{align}\label{asymodprof2}
\jxi^{3/2} \Big| |Z_\pm(t,\xi)| -|W_\pm^\infty(\xi)| \Big| \lesssim \e_1^3 \, \jt^{-\delta_0}.
\end{align}
The conclusion \eqref{LinfSasyf} follows from \eqref{asymodprof} and \eqref{asymodprof2},
up to possibly redefining the asympotic profile $W_\pm^\infty$ by a constant phase.
\end{proof}

\medskip
The rest of this section is organized as follows.
In Subsection \ref{secha} we provide asymptotic formulas for oscillatory integrals like 
those defining $\mathcal{C}^{S1,2}$. These formulas are first obtained at a formal level
by applying heuristic stationary phase type estimates.
In Subsection \ref{secModScatt} we use these formulas 
to derive the leading order of \eqref{LinfSasy} with the proper Hamiltonian structure.
The precise bounds needed to rigorously justify these formulas, that is,
the error estimates \eqref{LinfSrem}, are proved in Subsection \ref{secra}.

\medskip
\subsection{Heuristic asymptotics}\label{secha}
Our first aim is to compute the asymptotics as $t \to \infty$ for the main model operators
\begin{align}
\label{Idelta}
& I_\delta(t) := \iiint e^{it \Phi_{\kappa_1 \kappa_2 \kappa_3}(\xi,\eta,\eta',\sigma')} F(\xi,\eta,\eta',\sigma') 
  \delta(p_*)\,d\eta\,d\eta' \,d\sigma', 
\\
\label{Ipv}
& I_{\operatorname{p.v.}}(t) := \iiint e^{it \Phi_{\kappa_1 \kappa_2 \kappa_3}(\xi,\eta,\eta',\sigma')} 
  F(\xi,\eta,\eta',\sigma') \frac{\widehat{\phi}(p_*)}{p_*}\,d\eta\,d\eta' \,d\sigma',
\end{align}
where
\begin{align}\label{asypstar}
p_* = \lambda_* \xi - \mu_* \eta - \mu'_* \eta' - \nu'_* \sigma',
\end{align}
and 
\begin{align}\label{asyPhi}
\Phi_{\kappa_1 \kappa_2 \kappa_3}(\xi,\eta,\eta',\lambda_* \xi - \mu_* \eta - \mu'_* \eta'-p_*) 
  = \langle \xi \rangle -\kappa_1 \langle \eta \rangle - \kappa_2 \langle \eta' \rangle 
  - \kappa_3 \langle \lambda_* \xi - \mu_* \eta - \mu'_* \eta'-p_* \rangle.
\end{align}
Note how the operators $\mathcal{C}^{S1,2}$ are of the form \eqref{Idelta}-\eqref{Ipv} above.

We begin by examining the phase, which we sometimes denote just by $\Phi = \Phi(\xi,\eta,\eta',\s')$, 
keeping the dependence on the various parameters $ \kappa_1,\kappa_2, \kappa_3, \lambda_*,\mu_*,\mu'_*,\nu'_*$ implicit.

Observing that $\nabla_{\eta,\eta'} \Phi = 0$ implies that $\eta' = \kappa_1 \kappa_2 \mu'_* \mu_* \eta$,
a small computation shows that the stationary point with respect to $\eta$ and $\eta'$ is given by
\begin{align}\label{asystatpts}
\nabla_{\eta,\eta'} \Phi = 0 \qquad \Longleftrightarrow \qquad
\left\{
\begin{array}{l}
 \eta = \eta_S = \kappa_3 \mu_*(\kappa_1 + \kappa_3 + \kappa_1 \kappa_2 \kappa_3)^{-1} (\lambda_* \xi - p_*).
\\
\\
\eta' = \eta'_S = \kappa_1 \kappa_2 \kappa_3 \mu_*'(\kappa_1 + \kappa_3 + \kappa_1 \kappa_2 \kappa_3)^{-1} 
  (\lambda_* \xi - p_*),
\\
\\
\sigma' = \sigma'_S = \nu'_*  \kappa_1 (\kappa_1 + \kappa_3 + \kappa_1 \kappa_2 \kappa_3)^{-1} (\lambda_* \xi - p_*).
\end{array}
\right.
\end{align}
Furthermore, at the stationary point,
\begin{align}\label{asyHess}
\operatorname{Hess}_{\eta,\eta'} \Phi = 
\left(
\begin{array}{ll} - \kappa_1 \tau''(\eta_S) - \kappa_3 \tau''(\sigma'_S) \quad & - \kappa_3 \mu_* \mu'_* \tau''(\sigma'_S) 
\\
\\ 
- \kappa_3 \mu_* \mu'_*  \tau''(\sigma'_S) & -\kappa_2 \tau''(\eta'_S) - \kappa_3  \tau''(\sigma'_S)
\end{array}
\right), \qquad \tau(x) := \langle x \rangle.
\end{align}

\medskip 
\subsection*{Asymptotics for $I_\delta$}
In this case $p_*=0$, and the stationary point is
\begin{align}\label{asystatptsd}
\begin{array}{l}
 \eta = \eta_{S0} = \kappa_3 \mu_*(\kappa_1 + \kappa_3 + \kappa_1 \kappa_2 \kappa_3)^{-1} \lambda_* \xi,  
 \\
\eta' = \eta'_{S0} = \kappa_1 \kappa_2 \kappa_3 \mu_*'(\kappa_1 + \kappa_3 + \kappa_1 \kappa_2 \kappa_3)^{-1} \lambda_* \xi,
\\
\sigma' = \sigma'_{S0} = \nu'_*  \kappa_1 (\kappa_1 + \kappa_3 + \kappa_1 \kappa_2 \kappa_3)^{-1} \lambda_* \xi.
\end{array}
\end{align}
We distinguish two cases:

\begin{itemize}

\smallskip
\item If $|\kappa_1 + \kappa_3 + \kappa_1 \kappa_2 \kappa_3| = 1$, then $|\eta_{S0}| = |\eta'_{S0}| = |\sigma'_{S0} | = |\xi|$ and
\begin{align}\label{asystatptsd'}
\begin{split}
& \Phi(\xi,\eta_{S0},\eta'_{S0},\sigma'_{S0}) = (1 - \kappa_1 - \kappa_2 - \kappa_3) \langle \xi \rangle,
\\
& \operatorname{det} \operatorname{Hess} \Phi(\xi,\eta_{S0},\eta'_{S0},\sigma'_{S0})
  = [(\kappa_1 + \kappa_3)(\kappa_2 + \kappa_3) - 1 ] \tau''(\xi)^2 = - \tau''(\xi)^2,
\\
& \operatorname{sign} \operatorname{Hess} \Phi(\xi,\eta_{S0},\eta'_{S0},\sigma'_{S0}) = 0,
\end{split}
\end{align}
where we denote $\operatorname{sign}M$ for the number of positive minus the number of negative eigenvalues
of a matrix $M$.

\smallskip
\item If $|\kappa_1 + \kappa_3 + \kappa_1 \kappa_2 \kappa_3| = 3$, 
then $|\eta_{S0}| = |\eta'_{S0}| = |\sigma'_{S0} | = |\xi|/3$ and
\begin{align*}
& \Phi(\xi,\eta_{S0},\eta'_{S0},\sigma'_{S0}) = \langle \xi \rangle -  
  ( \kappa_1 + \kappa_2 + \kappa_3)\langle \xi /3 \rangle, 
  \\
& \operatorname{det} \operatorname{Hess} \Phi(\xi,\eta_{S0},\eta'_{S0},\sigma'_{S0}) 
  = [(\kappa_1 + \kappa_3)(\kappa_2 + \kappa_3) - 1 ] \tau''(\xi/3)^2, 
  \\
& \operatorname{Tr} \operatorname{Hess} \Phi(\xi,\eta_{S0},\eta'_{S0},\sigma'_{S0}) 
  =- (\kappa_1 + \kappa_2 + 2 \kappa_3) \tau''(\xi/3).
\end{align*}
\end{itemize}

In both cases, by the stationary phase lemma,
$$
I_\delta(t) \overset{t \to + \infty}{\sim} \frac{2\pi}{t}
  \frac{e^{i\frac{\pi}{4} \operatorname{sign} \operatorname{Hess} \Phi}}{| \operatorname{det} \operatorname{Hess} \Phi|^{1/2}} 
  e^{it\Phi(\xi,\eta_{S0},\eta'_{S0},\sigma'_{S0})} F(\xi,\eta_{S0},\eta'_{S0},\sigma'_{S0}).
$$
If $\Phi = 0$ (hence $\{ \kappa_1,\kappa_2,\kappa_3 \} = \{+,+,-\}$), 
\begin{align}\label{asydelta}
I_\delta(t) \overset{t \to + \infty}{\sim} \frac{2\pi}{t} \frac{1}{\tau''(\xi)} F(\xi,\eta_{S0},\eta'_{S0},\sigma'_{S0}).
\end{align}

\medskip
\subsection*{Asymptotics for $I_{\pv}$}
Here $p_* \neq 0$ and, to leading order in $p_*$ small, 
$\Phi$ and  $\operatorname{det} \operatorname{Hess} \Phi(\xi,\eta_S,\eta'_S,\sigma'_S)$
agree with their value at $p_*=0$ computed above. We also compute the next order in $p_*$ of $\Phi$:
\begin{itemize}
\item 
If $|\kappa_1 + \kappa_3 + \kappa_1 \kappa_2 \kappa_3| = 1$, then
$$
\Phi(\xi,\eta_S,\eta'_S,\sigma'_S) =  \langle \xi \rangle - (\kappa_1 + \kappa_2 + \kappa_3) \left[ \langle \xi \rangle - \tau'(\lambda_* \xi) p_* \right] + O(p_*^2).
$$
\item 
If $|\kappa_1 + \kappa_3 + \kappa_1 \kappa_2 \kappa_3| = 3$, then
$$
\Phi(\xi,\eta_S,\eta'_S,\sigma'_S) = \langle \xi \rangle - (\kappa_1 + \kappa_2 + \kappa_3) \left[ \langle \xi /3\rangle - \frac{1}{3}\tau'\left( \frac{\lambda_* \xi}{3} \right) p_*\right] + O(p_*^2).
$$
\end{itemize}
In order to give asymptotics, we focus on the former case ($|\kappa_1 + \kappa_3 + \kappa_1 \kappa_2 \kappa_3| = 1$) 
since it is the most relevant one. Applying the stationary phase lemma for $p_*$ fixed,
$$
I_{\operatorname{p.v.}}(t) \overset{t \to + \infty}{\sim} \frac{2\pi}{t} \frac{1}{\tau''(\xi)} 
e^{it \langle \xi \rangle (1 - \kappa_1 - \kappa_2 - \kappa_3)} F(\xi,\eta_{S0},\eta'_{S0},\sigma'_{S0})
\int  e^{it \tau'(\lambda_* \xi) (\kappa_1 + \kappa_2 + \kappa_3) p_*} \frac{\widehat{\phi}(p_*)}{p_*}\,dp_*.
$$
Since $\widehat{\mathcal{F}} \frac{\widehat{\phi}(x)}{x} = \frac{1}{\sqrt{2\pi}} \whF(1/x) \ast \phi =
- \frac{i}{2} \sign \ast \phi$, see \eqref{Fsign},
$$
I_{\operatorname{p.v.}}(t) \overset{t \to + \infty}{\sim} i \frac{\pi \sqrt{2\pi}}{t} \frac{1}{\tau''(\xi)} 
e^{it \langle \xi \rangle (1 - \kappa_1 - \kappa_2 - \kappa_3)} F(\xi,\eta_{S0},\eta'_{S0},\sigma'_{S0}) \sign( \tau'(\lambda_* \xi) (\kappa_1 + \kappa_2 + \kappa_3)).
$$
If $\Phi=0$,
\begin{align}\label{asypv}
I_{\operatorname{p.v.}}(t) \overset{t \to + \infty}{\sim} i \frac{\pi \sqrt{2\pi}}{t} \frac{1}{\tau''(\xi)}  F(\xi,\eta_{S0},\eta'_{S0},\sigma'_{S0}) \sign(\lambda_* \xi ).
\end{align}

\medskip
\subsection*{Asymptotics for $\widetilde{f}$}
We apply the above asymptotics to the situation that interests us to derive (formally, for the moment) 
asymptotics for $\partial_t\wt{f}$, see \eqref{eqdtf}.

The only relevant terms in the expansion are \eqref{asydelta} and \eqref{asypv} for which $\Phi = 0$, 
and that correspond to $\{ \kappa_1,\kappa_2, \kappa_3 \} = \{+,+,-\}$. 
Comparing the definition of $p_\ast$ in \eqref{asypstar} with \eqref{defpstar}, 
we obtain (recall $\kappa_1=\iota_1$)
$$
\lambda_* = \iota_2 \lambda \nu, \qquad \mu_* = \kappa_1 \iota_2 \mu \nu, \qquad \mu'_* = \kappa_2 \iota_2 \lambda' \mu', 
  \qquad \nu'_* = \kappa_3 \iota_2 \lambda' \nu',
$$
and find that
\begin{align*}
& \eta_{S0} = \lambda \mu \xi, \qquad \eta'_{S0} = \lambda \nu \lambda' \mu' \xi,
  \qquad \sigma'_{S0} = \lambda \nu \lambda' \nu' \xi. 
\end{align*}
Comparing with the definition of $\Sigma_0$ in \eqref{variousdef}
we get that the value of $\Sigma_0$ at these points is
\begin{align}\label{Sigma_0}
\Sigma_0 =
\left\{ \begin{array}{ll}
2 \iota_2 \nu \lambda \xi & \,\, \mbox{if \,\, $(\kappa_1,\kappa_2,\kappa_3) = (-,+,+)$,}
\\
\\
0 & \,\, \mbox{if \,\, $(\kappa_1,\kappa_2,\kappa_3) = (+,-,+)$ or $(+,+,-)$.} 
\end{array} \right.
\end{align}

Recall the formula \eqref{formulacubiccoeff} for the leading order symbol in the cubic terms appearing in 
\eqref{CubicS}-\eqref{CubicS12}.
We can compute their asymptotics as $t\rightarrow \infty$
using \eqref{asydelta} and \eqref{asypv}, obtaining
\begin{align*} 
& \mathcal{C}^S_{\kappa_1 \kappa_2 \kappa_3}(f,f,f)(\xi)
\\
& \quad \sim - \frac{1}{32 i t} \langle \xi \rangle^3 \sum_{\substack{\epsilon,\epsilon',\iota_2 
\\ \lambda,\mu,\nu,\lambda',\mu',\nu'}}\frac{1}{ \pi^2 ({(1-\kappa_1)\jxi - \iota_2\langle\Sigma_0\rangle}) 
\langle \xi \rangle^3 \langle \Sigma_0 \rangle  } \big(A^{\epsilon,\epsilon'}_{\nu,\lambda'}(\Sigma_0)\big)_{\iota_2}
\\
&  \quad \quad \times \mathbf{a}_{\lambda,-}^\epsilon(\xi) \mathbf{a}_{\mu,\kappa_1}^\epsilon(\lambda \mu \xi) 
\mathbf{a}_{\mu',\kappa_2}^{\epsilon'}(\lambda \nu \lambda' \mu' \xi) 
\mathbf{a}_{\nu',\kappa_3}^{\epsilon'}(\lambda \nu \lambda' \nu' \xi) 
  \wt{f}_{\kappa_1}(\lambda \mu \xi)  \wt{f}_{\kappa_2}(\lambda \nu \lambda' \mu' \xi) 
  \wt{f}_{\kappa_3}(\lambda \nu \lambda' \nu' \xi) 
\\ 
& \quad \quad \times \ell_{\epsilon \infty} \ell_{\epsilon' \infty} \left[ \pi^2 (1+\eps\eps'\nu \lambda' ) \iota_2 + \pi^2(\eps'\lambda'+\eps\nu) 
  \operatorname{sign}(\iota_2 \lambda \nu \xi)\right].
\end{align*}
Changing $\lambda'$ to $\lambda \nu \lambda'$, this becomes
\begin{align}\label{CShafinal} 
\begin{split}
\mathcal{C}^S_{\kappa_1 \kappa_2 \kappa_3}(f,f,f)(\xi)
& \sim \frac{i}{32 t} \sum_{\substack{\epsilon,\epsilon',\iota_2 
\\ 
\lambda,\mu,\nu, \lambda',\mu',\nu'}}\frac{\iota_2}{ ((1-\kappa_1)\jxi - \iota_2\langle\Sigma_0\rangle) 
 \langle \Sigma_0 \rangle} \big(A^{\epsilon,\epsilon'}_{\nu,\lambda\nu\lambda'}(\Sigma_0)\big)_{\iota_2}
\\
& \qquad \times \mathbf{a}_{\lambda,-}^\epsilon(\xi) \mathbf{a}_{\mu,\kappa_1}^\epsilon(\lambda \mu \xi)  
\mathbf{a}_{\mu',\kappa_2}^{\epsilon'}( \lambda' \mu' \xi) \mathbf{a}_{\nu',\kappa_3}^{\epsilon'}(\lambda' \nu' \xi) \wt{f}_{\kappa_1}(\lambda \mu \xi)  \wt{f}_{\kappa_2}( \lambda' \mu' \xi)  \wt{f}_{\kappa_3}( \lambda' \nu' \xi)
\\
& \qquad \times \ell_{\epsilon \infty} \ell_{\epsilon' \infty} \big[1+\eps\eps'\lambda \lambda' + (\eps'\lambda'+ \eps\l) \operatorname{sign}(\xi)\big].
\end{split}
\end{align}

\def\k{\kappa}

\medskip
\subsection{Structure of modified scattering}\label{secModScatt} 
In this subsection we analyze the leading orders in the (resonant) asymptotic terms.
In view of \eqref{CShafinal}, we are interested in the structure of the term
\begin{align}\label{asy100}
\begin{split}
& N_{\k_1\k_2\k_3}(f,f,f)(t,\xi) := 
  \frac{1}{\langle \Sigma_0 \rangle} 
  \sum_{\substack{\epsilon,\epsilon',\iota_2 
\\ \lambda,\mu,\nu,\lambda',\mu',\nu'}}
  \frac{\iota_2}{(1-\kappa_1)\jxi - \iota_2 \langle\Sigma_0\rangle} 
  \big(A^{\epsilon,\epsilon'}_{\nu,\lambda\nu\lambda'}(\Sigma_0)\big)_{\iota_2}
\\
& \qquad \qquad \times \mathbf{a}_{\lambda,-}^\epsilon(\xi) \, \mathbf{a}_{\mu,\kappa_1}^\epsilon(\lambda \mu \xi) 
  \, \mathbf{a}_{\mu',\kappa_2}^{\epsilon'}( \lambda' \mu' \xi)
  \, \mathbf{a}_{\nu',\kappa_3}^{\epsilon'}(\lambda' \nu' \xi) 
  \\
  & \qquad \qquad \times \widetilde{f}_{\kappa_1}(\lambda \mu \xi) \, \widetilde{f}_{\kappa_2}( \lambda' \mu' \xi)
  \, \widetilde{f}_{\kappa_3}( \lambda' \nu' \xi) \, \ell_{\epsilon \infty} \ell_{\epsilon' \infty}
  \big[1 + \epsilon\epsilon^\prime\lambda \lambda' + [ \epsilon\lambda + \epsilon^\prime\lambda'] 
  \operatorname{sign}(\xi)\big].
\end{split}
\end{align}
For the term $\big[1 + \epsilon\epsilon^\prime\lambda \lambda' + [\epsilon\lambda + \epsilon^\prime\lambda'] 
\operatorname{sign}(\xi)\big]$ to be non zero, 
we need $\epsilon \lambda = \epsilon' \lambda'$. 
In view of the formulas \eqref{Acoeff} for  $A^{\epsilon,\epsilon'}_{\l,\l'}$,
this further imposes that $\epsilon = \epsilon'$ and $\lambda = \lambda'$, 
in which case $A^{\epsilon, \epsilon'}_{\lambda, \lambda'} = 1/2$. 
Finally, one can sum over $\nu$, which does not appear in the expression, and \eqref{asy100} becomes
\begin{align*}
N_{\k_1\k_2\k_3}(f,f,f)(t,\xi) & := 
  \frac{2}{\langle \Sigma_0 \rangle} 
  \sum_{\substack{\epsilon,\iota_2 
\\ \lambda,\mu,\mu',\nu'}}
  \frac{\iota_2}{(1-\kappa_1)\jxi - \iota_2 \langle\Sigma_0\rangle}
  \\
& \times \mathbf{a}_{\lambda,-}^\epsilon(\xi) \, \mathbf{a}_{\mu,\kappa_1}^\epsilon(\lambda \mu \xi) 
  \, \mathbf{a}_{\mu',\kappa_2}^{\epsilon}( \lambda \mu' \xi)
  \, \mathbf{a}_{\nu',\kappa_3}^{\epsilon}(\lambda \nu' \xi)
  \\
  & \times \widetilde{f}_{\kappa_1}(\lambda \mu \xi) \, \widetilde{f}_{\kappa_2}( \lambda \mu' \xi)
  \, \widetilde{f}_{\kappa_3}( \lambda \nu' \xi)
  \,\ell_{\epsilon \infty}^2 \big[1 + \epsilon\lambda  \operatorname{sign}(\xi)\big].
\end{align*}

\medskip
\noindent
{\bf The flat case}. 
For the reader's convenience, we first look at the simpler case $V=0$, 
for which $T\equiv 1$, $R_\pm \equiv 0$, so that \eqref{mucoeffexp} reads
\begin{align*}
\mathbf{a}^-_+(\xi) = \mathbf{a}^+_+(\xi) \equiv 1, \qquad \mathbf{a}^+_-(\xi) = \mathbf{a}^-_-(\xi) \equiv 0,
\end{align*}
and $\ell_{+\infty} = -\ell_{-\infty}$ which we set to $1$ without loss of generality.
Then, in \eqref{asy100} only the sum over $\lambda = \mu = \mu' = \nu'=+$ survives. 
Moreover, summing over $\epsilon$ eliminates the contribution to the summand 
from the $\epsilon \lambda \operatorname{sign}(\xi)$ factor.
Overall, this gives
\begin{align}\label{asy11}
\begin{split}
N_{\k_1\k_2\k_3}(f,f,f)(t,\xi) & = \frac{4}{\langle \Sigma_0 \rangle} 
  \sum_{\iota_2}
  \frac{\iota_2}{(1-\kappa_1)\jxi - \iota_2\langle\Sigma_0\rangle}
  \, \widetilde{f}_{\kappa_1}(\xi) \, \widetilde{f}_{\kappa_2}(\xi)
  \, \widetilde{f}_{\kappa_3}(\xi).
\end{split}
\end{align}

When $\k_1=+$, so that $\k_2\k_3=-$ and $\Sigma_0 = 0$, this is 
\begin{align*}
\begin{split}
- 8 \widetilde{f}(\xi) \, {\big| \wt{f}(\xi) \big|}^2.
\end{split}
\end{align*}
When $\kappa_1=-$, so that $\langle\Sigma_0 \rangle=\langle 2 \xi \rangle$, we get
\begin{align*}
\begin{split}
\frac{4}{\langle 2 \xi \rangle} \sum_{\iota_2} \frac{\iota_2}{2 \jxi -\iota_2\langle 2 \xi \rangle }
  \, \overline{\widetilde{f}}(\xi) \, {\big( \wt{f}(\xi) \big)}^2
  = \frac{8}{3} \, {|\widetilde{f}(\xi)|}^2 \, \wt{f}(\xi).
\end{split}
\end{align*}
Overall, we find that
$$
\sum_{\substack{(\kappa_1,\kappa_2,\kappa_3) = (+,+,-), \\ (+,-,+), (-,+,+)}} N_{\k_1\k_2\k_3}(f,f,f)(t,\xi) 
=-\frac{40}{3} |\widetilde{f}(\xi)|^2 \widetilde{f}(\xi).
$$
In particular, from \eqref{CShafinal}-\eqref{asy100}, and the fact that all the other terms 
in the equation \eqref{eqdtf} for $\partial_t \wt{f}$
are lower orders (in the sense that they satisfy estimates like \eqref{LinfSrem}), we can deduce 
\begin{align}\label{AsyFlat}
\partial_t \wt{f}(t,\xi) \approx -\frac{5i}{12t} |\widetilde{f}(\xi)|^2 \widetilde{f}(\xi). 
\end{align}
This leads to `standard' modified scattering as in \cite{HNKG}.

\medskip 
\noindent
{\bf The general case}.
If $(\k_1,\k_2,\k_3)=(+,+,-)$ (which is identical to $(\k_1,\k_2,\k_3)=(+,-,+)$), $\Sigma_0=0$ and 
\begin{align*}
\begin{split}
N_{+--}(f,f,f)(t,\xi) & =  -4\sum_{\substack{\epsilon, 
  \\ \lambda,\mu,\mu',\nu'}}
  \bar{\mathbf{a}_{\lambda}^\epsilon(\xi)} \, \mathbf{a}_{\mu}^\epsilon(\lambda \mu \xi)
  \, \mathbf{a}_{\mu'}^{\epsilon}(\lambda \mu' \xi)
  \, \bar{\mathbf{a}_{\nu'}^{\epsilon}(\lambda \nu' \xi)}
  \\ & \qquad \qquad \qquad \, \times \wt{f}(\lambda \mu \xi) \, \wt{f}(\lambda \mu' \xi)\,\bar{\wt{f}(\lambda \nu' \xi)}
  \, \ell_{\epsilon \infty}^2 [1 + \epsilon \lambda  \operatorname{sign}(\xi)].
\end{split}
\end{align*}
If $(\k_1,\k_2,\k_3)=(-,+,+)$, $\Sigma_0=2\iota_2 \nu \lambda \xi$ and we get
\begin{align*}
N_{-++}(f,f,f)(t,\xi) & = \frac{4}{3}\sum_{\substack{\epsilon, 
  \\ \lambda,\mu,\mu',\nu'}}
  \bar{\mathbf{a}_{\lambda}^\epsilon(\xi)} \, \mathbf{a}_{\mu}^\epsilon(\lambda \mu \xi)
  \, \mathbf{a}_{\mu'}^{\epsilon}(\lambda \mu' \xi)
  \, \bar{\mathbf{a}_{\nu'}^{\epsilon}(\lambda \nu' \xi)}
  \\
  & \qquad \qquad \qquad \,\times \wt{f}(\lambda \mu \xi) \, \wt{f}(\lambda \mu' \xi) \, \bar{\wt{f}(\lambda \nu' \xi)}
  \,\ell_{\epsilon \infty}^2 [1 + \epsilon \lambda  \operatorname{sign}(\xi)].
\end{align*}
Therefore,
\begin{align}\label{asy20}
\begin{split}
\sum_{\substack{(\kappa_1,\kappa_2,\kappa_3) = (+,+,-), \\ (+,-,+), (-,+,+)}} N_{\k_1\k_2\k_3}(f,f,f)(t,\xi)&
  = - \frac{20}{3} \sum_{\substack{\epsilon, 
  \\ \lambda,\mu,\mu',\nu'}}
  \bar{\mathbf{a}_{\lambda}^\epsilon(\xi)} \, \mathbf{a}_{\mu}^\epsilon(\lambda \mu \xi)
  \, \mathbf{a}_{\mu'}^{\epsilon}(\lambda \mu' \xi)
  \, \bar{\mathbf{a}_{\nu'}^{\epsilon}(\lambda \nu' \xi)}\\
  & \times \wt{f}(\lambda \mu \xi) \, \wt{f}(\lambda \mu' \xi) \, \bar{\wt{f}(\lambda \nu' \xi)}
  \, \ell_{\epsilon \infty}^2 [1 + \epsilon \lambda  \operatorname{sign}(\xi)].
\end{split}
\end{align}

\bigskip
\noindent
{\bf Hamiltonian structure}.
Recall the evolution equation \eqref{eqdtf} for $\wt{f}$.
As we show in Subsection \ref{SsecLinfR}, all the terms on the right-hand side of \eqref{eqdtf}, 
with the exception of $\mathcal{C}^{S1,2}$,
decay at an integrable-in-time rate.
Then, from \eqref{CShafinal}, \eqref{asy100} and \eqref{asy20} ,
and letting $t = \log t'$,
we see that the asymptotic evolution of $\wt{f}$ is governed to leading order by the ODE
\begin{align*}
\partial_{t'} \widetilde{f}(t',\xi) & = -\frac{5i}{24} \sum_{\substack{\epsilon, 
  \\ \lambda,\mu,\mu',\nu'}}
  \bar{\mathbf{a}_{\lambda}^\epsilon(\xi)} \, \mathbf{a}_{\mu}^\epsilon(\lambda \mu \xi)
  \, \mathbf{a}_{\mu'}^{\epsilon}(\lambda \mu' \xi)
  \, \bar{\mathbf{a}_{\nu'}^{\epsilon}(\lambda \nu' \xi)}
  \\ & \times \wt{f}(t',\lambda \mu \xi) \, \wt{f}(t',\lambda \mu' \xi) \, \bar{\wt{f}(t',\lambda \nu' \xi)}
  \, \ell_{\epsilon \infty}^2 [1 + \epsilon \lambda  \operatorname{sign}(\xi) ].
\end{align*}

We now show how to view the joint evolution of $\widetilde{f}(t',\xi)$ and $\widetilde{f}(t',-\xi)$ into the form of an Hamiltonian system.
For $\xi>0$ we let
$$
X_+(t') = \widetilde{f}(t',\xi), \qquad X_-(t') = \widetilde{f}(t',-\xi).
$$
Then the evolution is
\begin{align*}
& \frac{d}{dt'} X_+ = -\frac{5}{12} i \sum_{\substack{\epsilon,\mu,\mu',\nu'}}
  \bar{\mathbf{a}_{\epsilon}^\epsilon(\xi)} \, \mathbf{a}_{\mu}^\epsilon(\epsilon \mu \xi)
  \, \mathbf{a}_{\mu'}^{\epsilon}(\epsilon \mu' \xi)
  \, \bar{\mathbf{a}_{\nu'}^{\epsilon}(\epsilon \nu' \xi)} \, X_{\epsilon \mu} \, X_{\epsilon \mu' } 
  \, \bar{X_{\epsilon \nu'}} \, \ell_{\epsilon \infty}^2,
\\
&   \frac{d}{dt'} X_- =  -\frac{5}{12} i \sum_{\substack{\epsilon,\mu,\mu',\nu'}}
  \bar{\mathbf{a}_{-\epsilon}^\epsilon(-\xi)} \, \mathbf{a}_{\mu}^\epsilon(\epsilon \mu \xi)
  \, \mathbf{a}_{\mu'}^{\epsilon}(\epsilon \mu' \xi)
  \, \bar{\mathbf{a}_{\nu'}^{\epsilon}(\epsilon \nu' \xi)} \, X_{\epsilon \mu} \, X_{\epsilon \mu' } \, \bar{X_{\epsilon \nu'}} \,
  \, \ell_{\epsilon \infty}^2 .
\end{align*}
The main observation is that this derives from the Hamiltonian
\begin{align}\label{asyHam0}
\begin{split}
H(X)  & = \frac{5}{24} \sum_{\epsilon} 
	\left| \sum_{\mu} \mathbf{a}_{\mu}^\epsilon(\epsilon \mu \xi) X_{\epsilon \mu} \right|^4  
	\ell_{\epsilon \infty}^2
	\\ &= \frac{5}{24} \left[ \ell_{+\infty}^2 |(S(\xi)X)_1|^4 + \ell_{-\infty}^2 |(S(\xi) X)_2|^4 \right],
\end{split}
\end{align}
where we regard $(X_+,X_-)$ as conjugate variables of $(\bar{X_+},\bar{X_-})$,
and consider the standard (complex) symplectic form $J=-i$;
in \eqref{asyHam0} $S$ denotes the scattering matrix, and $X$ the vector $(X_+,X_-)$.
The evolution associated to $H$ is
\begin{align}\label{asyHamev}
\begin{split}
\frac{d}{dt'} X = - i \frac{\partial}{\partial \overline{X}} H.
\end{split}
\end{align}
Note that since $H$ is invariant under phase rotations 
the evolution \eqref{asyHamev} conserves $|X_+|^2 + |X_-|^2$.



\medskip
\subsection{Rigorous asymptotics}\label{secra}
Here we give the estimates necessary 
to justify the asymptotic formulas \eqref{asydelta} and \eqref{asypv} for the integrals \eqref{Idelta} and \eqref{Ipv},
thus obtaining a proof of the main asymptotics \eqref{LinfSasy}-\eqref{LinfSrem} in Proposition \ref{propLinfS}
We refer the reader to similar arguments in the literature, such 
as those in \cite{IoPu1,IoPu3} for fractional NLS equations, and in \cite{GPR2,ChPu} for the NLS with a potential;
see also references therein, for other works that use different approaches, as well as 
\cite{LLS20} and the more recent \cite{LSSineG}.

In \S\ref{ssecraI1} and \S\ref{ssecraI2} we look at the cases with $\{\kappa_1,\kappa_2,\kappa_3\} = \{+,+,-\}$
that give the leading order 
terms on the right-hand sides of \eqref{asydelta} and \eqref{asypv} 
and, eventually, combined with the algebraic calculations of Subsection \ref{secModScatt}, the asymptotics in \eqref{LinfSasy}.
In \S\ref{ssecraIoth} we discuss how to handle all the other non-resonant and faster-decaying terms.

\medskip
\subsubsection{Asymptotics for \eqref{Idelta} when $\{\kappa_1,\kappa_2,\kappa_3\} = \{+,+,-\}$}\label{ssecraI1}
For simplicity, and without loss of generality, we may choose a single combination 
of the signs $\lambda,\mu,\dots$, appearing in \eqref{asypstar}, and thus concentrate on the expression
\begin{align}\label{raI1}
\begin{split}
& I_1(t,\xi) := \iint e^{it \Phi_1(\xi,\eta,\sigma)} F_1(t,\xi,\eta,\s) 
  \,d\eta\,d\s,
\\
& \Phi_1(\xi,\eta,\s) := 
\jxi -\jeta + \jsig - \langle \xi - \eta + \s \rangle,
\\
&  F_1(t,\xi,\eta,\s) :=  \mathfrak{m}_1(\xi,\eta,\s) \, \wt{f}(t,\eta) \bar{\wt{f}(t,\s)} \wt{f}(t,\xi-\eta+\sigma).
\end{split}
\end{align}
From the explicit formula \eqref{formulacubiccoeff}, 
and the bounds in Lemma \ref{californiacondor} (see also the proofs of Lemmas \ref{lemTbound} and \ref{lemCS}), 
we can think that the symbol $\mathfrak{m}_1$ is smooth and satisfies
\begin{equation}\label{bluejay}
\left| \partial_\xi^a \partial_\eta^b \partial_\sigma^c \mathfrak{m}_1(\xi,\eta,\s) \right| 
	\lesssim \frac{1}{\langle \eta \rangle \langle \sigma \rangle \langle \xi - \eta + \sigma \rangle} 
	\langle \med(|\eta|,|\s|,|\xi-\eta-\sigma|) \rangle^{a+b+c}.
\end{equation}

As we calculated earlier in the section, see \eqref{asystatptsd},
the only (time-frequency) stationary point of the integral $I_1$ is at $(\eta,\sigma) = (\xi,\xi)$. 
Although one should think that the hardest case is when $|\xi| \approx \sqrt{3}$,
below we do not need to decompose in frequency space with respect to
the distance to $\sqrt{3}$, and it will suffice to use the bounds
\begin{align}\label{radk1}
{\| \varphi_{[-5,5]} \partial_\xi \wt{f} \|}_{L^2} \lesssim \e_1 \jt^{\rho}, 
\qquad {\| \jxi (1-\varphi_{[-5,5]}) \partial_\xi \wt{f} \|}_{L^2} \lesssim \e_1 \jt^\a, 
\end{align}
where $\rho:= \a +\b\g$, see \eqref{lembb1}. 

We change variables so that the stationary point is at $(0,0)$ and
look at
\begin{align}\label{ra10}
\begin{split}
& I_1(t,\xi) = \iint e^{it \Phi_1(\xi,\xi+\eta,\xi+\eta+\s)} \, F_1(t,\xi,\xi+\eta,\xi+\eta+\s) 
  \,d\eta\,d\s,
\\
& F_1(t,\xi,\xi+\eta,\xi+\eta+\s) = \mathfrak{m}_1(\xi,\xi+\eta,\xi+\eta+\s) \wt{f}(\xi+\eta) \bar{\wt{f}(\xi+\eta+\s)} \wt{f}(\xi+\s).
\end{split}
\end{align}
To verify \eqref{asydelta}, we show that
\begin{align}\label{ra1bound}
\begin{split}
\jxi^{3/2} \left| \int_0^t \Big[ I_1(s,\xi) - \frac{2\pi}{s} \, \jxi^3 \, F_1(s,\xi,\xi,\xi) \Big] \, \tau_m(s) ds \right| \lesssim \e_1^3 2^{-\delta_0 m}
\end{split}
\end{align}
where  $\delta_0>0$ small enough is to be chosen below. 
Notice that this is consistent with the estimates for the remainders in \eqref{LinfSasy}-\eqref{LinfSrem}.
Also notice that $$F_1(t,\xi,\xi,\xi) = \mathfrak{m}_1(\xi,\xi,\xi)|\wt{f}(t,\xi)|^2\wt{f}(t,\xi)$$ 
where $\mathfrak{m}_1(\xi,\xi,\xi)$ coincides with the symbol appearing in the trilinear terms of Subsection \ref{secModScatt}.


Without loss of generality, we may assume that $|\eta|\geq |\s|$.
Moreover, we claim that it suffices to deal with 
\begin{align}\label{raxi}
 \jt^{-2\alpha - 2\delta_0} \lesssim |\xi| \lesssim \jt^{(1/2)(\alpha+p_0)+\delta_0}.
\end{align}
To see this, we use the interpolation inequality
\begin{align}\label{rainter}
{\| \wt{f} \|}_{L^\infty} \lesssim {\big\|\wt{f}\big\|}_{L^2}^{1/2} {\big\|\partial_{\xi} \wt{f}\big\|}_{L^2}^{1/2},
\end{align}
which, for $k\geq5$, gives us
\begin{align}\label{Linfinter}
\begin{split}
{\| \varphi_k \jxi^{3/2}\wt{f} \|}_{L^\infty} & \lesssim 
  {\big\| \varphi_k\jxi^3 \wt{f}\big\|}_{L^2}^{1/2} \big( {\big\| \varphi_k' \wt{f}\big\|}_{L^2} + {\big\| \varphi_k \partial_{\xi} \wt{f}\big\|}_{L^2} \big)^{1/2}
  \\
  & \lesssim 2^{-k}{\big\| \varphi_k\jxi^4 \wt{f}\big\|}_{L^2}^{1/2} 
  \cdot (\e_1 \jt^{\alpha})^{1/2}
  \lesssim 2^{-k} \cdot \e_1^2 \jt^{(p_0+\alpha)/2},
\end{split}
\end{align}
having used the a priori bounds \eqref{propbootfas}.
Then, if $2^k \geq \jt^{(\alpha+p_0)/2+\delta_0}$ 
we already control uniformly in $t$ and $\xi$ the quantity $\jxi^{3/2}\wt{f}$.
For $k \leq -5$ instead, we have, see \eqref{apriori11},
\begin{align}\label{Linfinter'}
\begin{split}
{\| \varphi_k \wt{f} \|}_{L^\infty} & 
  \lesssim 2^{k/2} \jt^{\alpha}
\end{split}
\end{align}
and therefore obtain the desired control whenever $2^k \lesssim \jt^{-2\alpha-2\delta_0}$, as claimed.

We let $|\xi|\approx 2^k$, with the constraints \eqref{raxi},
let $t \approx 2^m$, $m=1,\dots$, and split
\begin{align}\label{ra13}
\begin{split}
& I_1(t,\xi) = \sum_{k_1 \geq k_2, \, k_1,k_2\in[k_0\infty)\cap \Z} I^{(1)}_{k_1,k_2}(t,\xi), \qquad k_0:= -m/2 + \delta m,
\\
& I^{(1)}_{k_1,k_2}(t,\xi) := \iint e^{it \Phi_1} \, F_1(t) 
  \,\varphi_{k_1}^{(k_0)}(\eta\jxi^{-3/2})\varphi_{k_2}^{(k_0)}(\s\jxi^{-3/2}) \, d\eta\,d\s,
\end{split}
\end{align}
where $\delta$ will be chosen small enough, 
and we are omitting the arguments $(\xi,\xi+\eta,\xi+\eta+\s)$ in $\Phi_1$ and $F_1$ for brevity.
For simplicity we also restrict our attention to the main case 
when the size of all input frequencies are comparable to $2^k$
by looking at the case $|\eta|,|\s| \ll |\xi|$, that is, $k_1,k_2\leq k - (3/2)k^+ - 10$;
to simplify our notation we omit the cut-offs induced by this restriction, that is,
$\varphi_{\sim k} (\xi+\eta)$,  $\varphi_{\sim k} (\xi+\sigma+\eta)$ and $\varphi_{\sim k} (\xi+\sigma)$ .
All other cases are simpler to handle.

\medskip
\noindent
{\it Case $k_1=k_2=k_0$}.
Note that from \eqref{radk1} under the restriction \eqref{raxi}, we can infer (for $\alpha$ small enough)
\begin{align}\label{radk2}
{\| \varphi_k \partial_\xi \wt{f} \|}_{L^2} \lesssim \e_1 2^{-10k^+} \jt^{\rho}.
\end{align}
Then, since $|\eta|,|\sigma| \ll |\xi|$, from the a priori bounds \eqref{propbootfas} and \eqref{bluejay} we have
\begin{align}\label{ra11}
\begin{split}
& |F_1(t,\xi,\xi+\eta,\xi+\eta+\s)| \lesssim \e_1^3 2^{-9k^+/2},
\\
& {\| 
  \partial_\eta F_1(t,\xi,\xi+\eta,\xi+\eta+\s) \|}_{L^2_\eta} 
  \lesssim \e_1^3 \cdot 2^{-10k^+} \cdot \jt^{\rho},
\\
& {\| 
  \partial_\s F_1(t,\xi,\xi+\eta,\xi+\eta+\s) \|}_{L^2_\s} 
  \lesssim \e_1^3 \cdot 2^{-10k^+} \cdot \jt^{\rho}, 
  \qquad \rho:=\alpha+\beta\gamma.
\end{split}
\end{align}
As a consequence
\begin{align}\label{ra12}
\begin{split}
& |F_1(t,\xi,\xi+\eta,\xi+\eta+\s) - F_1(t,\xi,\xi,\s+\xi)| 
 \\ 
 & + |F_1(t,\xi,\xi,\s+\xi) - F_1(t,\xi,\xi,\xi)| \lesssim \e_1^3 2^{-10k^+}\jt^\rho \cdot (|\eta| + |\sigma|)^{1/2} .
\end{split}
\end{align}

Using \eqref{ra12} we can see that the contribution close to the stationary points gives us the 
leading order term by arguing as follows. First, observe that
\begin{align}\label{ra14.1}
\begin{split}
2^{3k^+/2} & \Big| I^{(1)}_{k_0,k_0}(t,\xi) - F_1(t,\xi,\xi,\xi) \,
  \iint e^{it \Phi(\xi,\xi+\eta,\xi+\eta+\s)} 
  \,\varphi_{\leq k_0}(\eta\jxi^{-3/2})\varphi_{\leq k_0}(\s\jxi^{-3/2}) \, d\eta\,d\s \Big| 
  \\
  & \lesssim 2^{3k^+/2}  \cdot 
  \e_1^3 2^{k_0/2} 2^{-10k^+} 2^{\rho m} \cdot 2^{2k_0 + 3k^+} 
  = \e_1^3 \cdot 2^{(-5/4 + \rho + 5\delta/2)m};
\end{split}
\end{align}
second, Taylor expanding we have
\begin{align}\label{ra1Tay}
\Phi_1(\xi,\xi+\eta,\xi+\eta+\s) = \frac{2\eta\sigma}{\jxi^3} + O\big((|\eta|+|\s|\big)^3 \jxi^{-4})
\end{align}
and therefore
\begin{align}\label{ra14.2}
\begin{split}
2^{3k^+/2} & \Big| F_1(t,\xi,\xi,\xi) \, \iint \Big[ e^{it \Phi(\xi,\xi+\eta,\xi+\eta+\s)} - e^{2it \eta\s \jxi^{-3}} \Big] 
  \,\varphi_{\leq k_0}(\eta\jxi^{-3/2})\varphi_{\leq k_0}(\s\jxi^{-3/2}) \, d\eta\,d\s \Big|
  \\
  & \lesssim \e_1^3 2^{-3k^+} \cdot 2^m 2^{3k_0} 2^{-4k^+} \cdot 2^{2k_0}2^{3k^+} 
  \\
  & \lesssim \e_1^3 2^{-3m/2 + 5\delta m};
\end{split}
\end{align}
third, a calculation shows that
\begin{align}\label{ra14.3}
\iint e^{2it\eta\s \jxi^{-3}} \,\varphi_{\leq k_0}(\eta\jxi^{-3/2})\varphi_{\leq k_0}(\s\jxi^{-3/2}) \, d\eta\,d\s 
  =  \jxi^3 \frac{\pi}{t} + O(|t|^{-5/4}).
\end{align}

Finally, we need to ensure that we can choose $\delta>0$ 
so that \eqref{ra14.1}-\eqref{ra14.3} are consistent with the right-hand side of 
the desired bound \eqref{ra1bound}.
According to \eqref{ra14.1} and \eqref{raxi}, by making $\delta_0$ small enough, 
it suffices to pick $\delta$ such that 
\begin{align}\label{radelta}
(5/2)\delta < 1/4-\rho; 
\end{align}
this is possible since, see 
\eqref{wnormparam0},
\begin{align}\label{radelta'}
1/4-\rho 
= -\alpha +(1/2)(\b'+\g') - \b'\g' > \g'\beta.
\end{align}

\medskip
\noindent
{\it Case $k_1 > k_0$}. 
Since we are assuming $|\eta|\geq |\s|$, we may restrict, without loss of generality, to $k_1\geq k_2$;
for brevity, we will often omit to write this restriction.
In the case $k_1>k_0$, we want to exploit integration by parts in $\sigma$ through the identity 
$e^{it \Phi(\xi,\xi+\eta,\xi+\eta+\s)} = (it\partial_\sigma\Phi)^{-1} \partial_\s e^{it \Phi(\xi,\xi+\eta,\xi+\eta+\s)}$
using that, on the support of the integral $I^{(1)}_{k_1,k_2}$,
\begin{align}\label{ra21}
\begin{split}
|\partial_\sigma\Phi(\xi,\xi+\eta,\xi+\eta+\s)| = \Big| \frac{\xi+\s}{\langle \xi+\sigma \rangle} 
  - \frac{\xi+\s+\eta}{\langle \xi+\sigma+\eta \rangle} \Big|
  & \gtrsim  |\eta|\jxi^{-3} 
  \approx 2^{k_1}2^{-(3/2)k^+}. 
\end{split}
\end{align}

Note that, under our current frequencies restrictions, 
we have a bound on the norm of trilinear operators with symbol $(\partial_\sigma\Phi)^{-1}$
consistent with the (pointwise) bound  from \eqref{ra21}, that is, 
$2^{-k_1}2^{(3/2)k^+}$.
We treated similar terms multiple times in Section \ref{secwL}, see for example \eqref{hummingbird} and \eqref{estimatefrakm}.
%
%
We first use this fact to integrate by parts and estimate the $L^2$ norm of $I^{(1)}_{k_1,k_2}$.
Up to faster decaying lower orders (which include contributions from hitting the symbol $\mathfrak{m}_1$ or the cutoffs,
where one can repeat integration by parts) we have
\begin{align}\label{ra22}
\begin{split}
\big| I^{(1)}_{k_1,k_2}(t,\xi) \big|& \lesssim \frac{1}{t} \,
  \Big| \iint e^{it \Phi(\xi,\xi+\eta,\xi+\eta+\s)} 
  \, \wt{f}(\eta+\xi) \partial_\s\Big( \bar{\wt{f}(t,\xi+\eta+\s)} \wt{f}(t,\xi+\sigma) \Big)
  \\&  \qquad
  \times \frac{1}{\partial_\s\Phi}\, \mathfrak{m}_1(\xi,\xi+\eta,\xi+\eta+\s) 
  \,\varphi_{k_1}^{(k_0)}(\eta\jxi^{-3/2})\varphi_{k_2}^{(k_0)}(\s\jxi^{-3/2}) \, d\eta\,d\s \Big|,
\end{split}
\end{align}
so that
\begin{align}\label{ra23}
\begin{split}\Big\| \jxi^2 \varphi_k(\xi) \sum_{k_1>k_0} I^{(1)}_{k_1,k_2}(t) \Big\|_{L^2} 
& \lesssim 2^{2k^+} \cdot 2^{-m} 2^{-k_1} 2^{(3/2)k^+}
  {\| P_k e^{-it\jnab} \W f(t) \|}_{L^\infty}^2 {\| \varphi_k \partial_\xi \wt{f} \|}_{L^2}
  \\
& \lesssim 2^{7k^+/2} \cdot 2^{-m} \cdot 2^{-k_0} 
  \cdot (\e_1 2^{-m/2})^2 \cdot \e_1 
  2^{-10k^+} 2^{\rho m}   
\\
& \lesssim \e_1^3 2^{-3m/2} \cdot  2^{(\rho +\delta)m},
\end{split}
\end{align}
having used \eqref{radk2} 
the linear decay estimate, 
and $k_0 = -m/2+\delta m$.

Next, we estimate the $L^2$ norm of $\jxi \partial_\xi I^{(1)}_{k_1,k_2}$ and then will  interpolate with the bound \eqref{ra23}.
It is convenient to look back at the original integration variables as they appear in \eqref{raI1}, 
so as to have simpler formulas. Let us write, for $k_1>k_0$,
\begin{align}\label{raI1'}
\begin{split}
& I^{(1)}_{k_1,k_2}(t,\xi) := \iint e^{it \Phi_1(\xi,\eta,\sigma)} F_1(t,\xi,\eta,\s) 
\varphi_{k_1}((\eta-\xi) \jxi^{3/2}) \varphi_{k_2}^{(k_0)}((\s-\eta)\jxi^{3/2})  \,d\eta\,d\s.
\end{split}
\end{align}
When applying $\jxi \partial_\xi$ to $I^{(1)}_{k_1,k_2}$ we obtain one main term, that is,
\begin{align*}
& \iint t \, \jxi \partial_\xi \Phi_1(\xi,\eta,\s) \, e^{it \Phi_1(\xi,\eta,\s)}  
  \, \wt{f}(t,\eta) \bar{\wt{f}(t,\xi-\eta+\s)} \wt{f}(t,\sigma)
 \\ & \qquad  \times \mathfrak{m}_1(\xi,\eta,\s) 
  \,\varphi_{k_1}((\eta-\xi)\jxi^{-3/2})\varphi_{k_2}^{(k_0)}((\s-\eta)\jxi^{-3/2}) \, d\eta\,d\s,
\end{align*}
plus other lower order terms which are easier to estimate.
We then note that the following identity holds
\begin{align}\label{raid0}
\big( \jxi\partial_\xi + \jeta \partial_\eta  + \jsig \partial_\s \big)  \Phi_1 =  - \frac{\xi-\eta+\s}{\langle \xi-\eta + \s \rangle}\Phi_1.
\end{align}
Using \eqref{raid0} we can integrate by parts in $\eta,\s$ and $s$, and obtain
\begin{align}\label{ra24}
\begin{split}
& \Big| \sum_{k_1>k_0} \int_0^t \jxi \partial_\xi I^{(1)}_{k_1,k_2}(s,\xi) \tau_m(s) \, ds \Big| \lesssim 
\Big| \int_0^t A(s,\xi) \tau_m(s) \, ds \Big| + \Big| \int_0^t B(s,\xi) \tau_m(s) \, ds \Big| + \cdots 
\\
& A(t,\xi) :=  \sum_{k_1>k_0} \iint e^{it \Phi_1(\xi,\eta,\s)}
  \, \jeta \partial_\eta \wt{f}(t,\eta)  \bar{\wt{f}(t,\s)} \wt{f}(t,\xi-\eta+\s)
  \\
  & \qquad \qquad \qquad \times \mathfrak{m}_1(\xi,\eta,\s) 
  \,\varphi_{k_1}((\eta-\xi)\jxi^{-3/2})\varphi_{k_2}^{(k_0)}((\s-\eta)\jxi^{-3/2}) \, d\eta\,d\s,
\\
&B(t,\xi) :=  \sum_{k_1>k_0} t \iint e^{it \Phi_1(\xi,\eta,\s)} \frac{\xi-\eta+\s}{\langle \xi-\eta + \s \rangle}
  \, \partial_t \wt{f}(t,\eta)  \bar{\wt{f}(t,\s)} \wt{f}(t,\xi-\eta+\s)
  \\
  & \qquad \qquad \qquad \times \mathfrak{m}_1(\xi,\eta,\s) 
  \,\varphi_{k_1}((\eta-\xi)\jxi^{-3/2})\varphi_{k_2}^{(k_0)}((\s-\eta)\jxi^{-3/2}) \, d\eta\,d\s,
\end{split}
\end{align}
where, as usual, $``\cdots"$ denote similar terms or faster decaying remainders.
Using an $L^2\times L^\infty \times L^\infty$ H\"older estimate for  both terms in \eqref{ra24} gives
\begin{align}\label{ra25}
\begin{split}
\Big\| \jxi \partial_\xi  \sum_{k_1>k_0} \int_0^t I^{(1)}_{k_1,k_2}(s,\xi) \tau_m(s)\, ds \Big\|_{L^2} & \lesssim  
 2^m \sup_{s\approx 2^m }{\| \jxi \partial_\xi \wt{f}(s) \big\|}_{L^2} {\| P_k e^{is\jnab} \W f(s) \|}_{L^\infty}^2 
 \\ & \quad +
 2^{2m} \sup_{s\approx 2^m }{\| \partial_s \wt{f}(s) \big\|}_{L^2} {\| P_k e^{is\jnab} \W f(s) \|}_{L^\infty}^2  
\\
&  \lesssim \e_1^3 2^{m \rho},
\end{split}
\end{align}
having used \eqref{lembb1}, the a priori $L^\infty$ decay and \eqref{dtfL2}.

Interpolating \eqref{ra23} and \eqref{ra25} through \eqref{rainter},
and in view of our choice of parameters \eqref{radelta}-\eqref{radelta'} and \eqref{raxi}, we obtain
\begin{align*}
2^{(3/2)k^+} \Big| \sum_{k_1 > k_0
  } \int_0^t I^{(1)}_{k_1,k_2}(s,\xi) \tau_m(s)\,ds \Big| 
  & \lesssim \e_1^3 2^{-m/4} 2^{(\rho + \delta/2) m},
\end{align*}
which suffices in view of \eqref{radelta}.

\medskip
\subsubsection{Asymptotics for \eqref{Ipv} when $\{\kappa_1,\kappa_2,\kappa_3\} = \{+,+,-\}$}\label{ssecraI2}
Once again, without loss of generality, we may choose a single combination 
of the signs $\lambda,\mu,\dots$, appearing in \eqref{asypstar}, and concentrate on the expression
\begin{align}\label{raI2}
\begin{split}
& I_2(t,\xi) = \iiint e^{it \Phi_2(\xi,\eta,\sigma,p)} 
  F_2(t,\xi,\eta,\eta',\s) \frac{\what{\phi}(p)}{p}\,d\eta\,d\eta'\,d\sigma,
\\
& \Phi_2(\xi,\eta,\eta',\s) = \jxi -\jeta + \langle \eta' \rangle -\jsig, 
  \qquad  p = \xi-\eta+\eta'-\sigma,
\\
& F_2(t,\xi,\eta,\eta',\s) :=  \mathfrak{m}_2(\xi,\eta,\eta',\s) \, \wt{f}(t,\eta) \bar{\wt{f}(t,\eta')} \wt{f}(t,\s).
\end{split}
\end{align}
From the formula \eqref{CubicS}-\eqref{CubicS12} 
with \eqref{formulacubiccoeff} and the bounds \eqref{californiacondor1} 
we may assume that the symbol is smooth and satisfies
\begin{equation}\label{bluejay2}
\left| \partial_\xi^a \partial_\eta^b \partial_{\eta'}^c \partial_\sigma^d \mathfrak{m}_2(\xi,\eta,\eta',\sigma') \right| 
	\lesssim \frac{1}{\langle \eta \rangle \langle \eta' \rangle  \langle \sigma \rangle } 
	\langle \med(|\eta|,|\eta'|,|\s|) \rangle^{a+b+c+d}.
\end{equation}

To obtain asymptotics for \eqref{raI2} and a rigorous proof of \eqref{asypv}, 
we can use ideas similar to those used to treat \eqref{raI1} in \S\ref{ssecraI1} above;
we will then follow similar steps and concentrate on the main differences, in particular on how to treat the $\pv$,
while skipping some of the other details.

Recall from \eqref{asystatpts} that
the stationary points of the integral $I_2$ in \eqref{raI2} are $(\eta,\eta',\sigma) = (\xi-p,\xi-p,\xi-p)$.
As in the treatment of \eqref{raI1}, it is convenient to change variables by letting
$\eta \mapsto \xi+\eta-p$,  $\eta'\mapsto \xi+\eta+\s-p$ and $\s\mapsto \xi+\s-p$;
this centers the stationary points (in $(\eta,\s,p)$) at the origin, and gives the expression
\begin{align}\label{ra30}
\begin{split}
I_2(t,\xi) = \iiint & e^{it \Phi(\xi,\eta,\sigma,p)} 
  F(t,\xi,\eta,\s,p) \frac{\what{\phi}(p)}{p}\,d\eta\,d\sigma \,dp,
\\
\Phi(\xi,\eta,\s,p) & := \Phi_2(\xi,\xi+\eta-p,\xi+\eta+\s-p, \xi+\s-p) 
\\
\qquad \qquad \qquad & = \jxi - \langle \xi+\eta-p\rangle + \langle \xi+\eta+\s - p \rangle - \langle \xi+\s -p \rangle,
\\
\mathfrak{m}(\xi,\eta,\s,p) & := \mathfrak{m}_2(\xi,\xi+\eta-p,\xi+\s+\eta-p,\xi+\s-p),
\\
F(t,\xi,\eta,\s,p) & :=  \mathfrak{m}(\xi,\eta,\s,p) \wt{f}(\xi+\eta-p) \bar{\wt{f}(\xi+\eta+\s-p)} \wt{f}(\xi+\s-p).
\end{split}
\end{align}

To verify \eqref{asypv}, with a remainder estimate consistent with \eqref{LinfSrem}, we need to show that
\begin{align}\label{ra2bound}
\begin{split}
\jxi^{3/2} \left| \int_0^t \Big[ I_2(s,\xi) - \frac{i\pi\sqrt{2\pi}}{t} \, \jxi^3 \, F(s,\xi,0,0,0) \, \sign(\xi) \Big] \tau_m(s) ds \right|
  \lesssim \e_1^3 2^{-\delta_0 m}
\end{split}
\end{align}
for $\delta_0$ small enough. 
As per our usual notation, we let $|\xi| \approx 2^k$ and, in view of \eqref{raxi} 
(and the interpolation argument that follows it)
we may restrict to $-2\alpha-2\delta_0 \leq k \leq (\alpha+p_0)/2 + \delta_0$.
We again restrict to the most difficult case when all input frequencies have sizes comparable to $2^k$
by assuming $|\eta|,|\s| \ll |\xi|$;
we omit the corresponding cutoffs for lighter notation.

\medskip
\noindent
{\it Step 1}.
The first step needed to deal with the $\pv$ singularity in \eqref{ra30} is to remove a neighbourhood of $p=0$ as follows.
For very small $p$, say $|p|\lesssim 2^{-10m}$, we may substitute $F(t,\xi,\eta,\s,p)$ by $F(t,\xi,\eta,\s,0)$,
and $\Phi(\xi,\eta,\s,p)$ by $\Phi(\xi,\eta,\s,0)$ up to very fast decaying remainders. 
The resulting integral vanishes by using the $\pv$ and the fact that $\phi$ is even.

This leaves us with the expression 
\begin{align}\label{ra31}
& 
\iiint e^{it \Phi(\xi,\eta,\sigma,p)} 
	F(t,\xi,\eta,\s,p) \varphi_{\geq 0}(|p|\jt^{10}) \frac{\what{\phi}(p)}{p}\,d\eta\,d\sigma \,dp.
\end{align}
In what follows we will omit the cutoff localizing to $|p| \gtrsim \jt^{-10}$ for simplicity,
since its the presence does not cause any additional difficulty in the arguments.

Similarly to \eqref{ra13}, we define a localized version of $I_2$ by
\begin{align}\label{ra31.5}
I^{(2)}_{k_1,k_2}(t,\xi) := \iiint e^{it \Phi(\xi,\eta,\sigma,p)}
  F(t,\xi,\eta,\s,p) \, \varphi_{k_1}^{(k_0)}(\eta\jxi^{-3/2}) \varphi_{k_2}^{(k_0)}(\s\jxi^{-3/2})
  \, \frac{\what{\phi}(p)}{p}\,d\eta\,d\sigma \,dp
\end{align}
where $k_0 := -m/2 + \delta m$.

\medskip
\noindent
{\it Step 2}.
We first look at the case $k_1=k_2=k_0$ which gives the leading order term in the asymptotitcs.

\medskip
\noindent
{\it Step 2.1}.
This contribution can be analyzed similarly to how we did for the term $I^{(1)}_{k_0,k_0}$ before; 
see the definition in \eqref{ra13}.
The same exact argument used above leads to the following analogue of \eqref{ra14.1} 
(see \eqref{ra11}-\eqref{ra12}):
\begin{align}\label{ra32}
\begin{split}
& 2^{3k^+/2} \left| I^{(2)}_{k_0,k_0}(t,\xi) - \iiint e^{it \Phi(\xi,\eta,\sigma,p)} F(t,\xi,0,0,p) 
\varphi_{\leq k_0}(\eta\jxi^{-3/2}) \varphi_{\leq k_0}(\s\jxi^{-3/2}) \frac{\what{\phi}(p)}{p}\,d\eta\,d\sigma \,dp \right|
\\  
& \qquad \qquad \lesssim \e_1^3 
  \cdot 2^{(-5/4 + \rho + 5\delta/2)m} m;
\end{split}
\end{align}
Note that we have included an additional $\log t$ factor on the right-hand side above 
to take into account the integration of $\what{\phi}(p)/p$ over the region $|p|\gtrsim 2^{-10m}$. 
This still gives an acceptable bound under the conditions \eqref{radelta}.

\medskip
\noindent
{\it Step 2.2}.
To calculate the asymptotics for the integral in \eqref{ra32} we first notice that, 
if $|p| \gtrsim 2^{-3m/5}$, we can use integration by parts in $p$ since $|\partial_p \Phi |\approx 2^{k^-}$.
More precisely, when one of the profiles gets differentiated we estimate it in $L^2$, put $1/p$ in $L^2$ as well, 
and estimate the other two profiles in $L^\infty$; 
when instead $\partial_p$ hits $1/p$ (or the cutoff in $p$) we estimate this in $L^1$, 
and place all the profiles in $L^\infty_\xi$;
we then obtain a contribution bounded by 
\begin{align}\label{ra35}
C 2^{3k^+/2} \cdot 2^{-m} 2^{-k^-} \cdot \e_1 \big(2^{3m/10} 2^{\rho m} + 2^{3m/5} \big) 
	\cdot (\e_1 2^{-3k^+/2})^2\cdot 2^{2k_0}
\end{align}
where $\rho = \alpha +\beta\gamma$.
This a remainder term of the desired $O(2^{-m-\delta_0m})$ size for $\delta_0$ small enough.
From now on we assume $|p| \lesssim 2^{-3m/5}$, and will sometimes omit the cutoff for notational simplicity.

By Taylor expanding the phase,
\begin{align}\label{ra2Tay}
\Phi(\xi,\eta,\s,p) = \jxi - \langle \xi -p \rangle + \frac{2\eta\sigma}{\jxi^3} 
	+ O\big((|\eta|+|\s| + |p| \big)^3 \jxi^{-3}),
\end{align}
and using $|\eta|+|\s| + |p|\lesssim 2^{k_0}$ we obtain, see \eqref{ra1Tay} and the estimate below that,
\begin{align*}
\begin{split}
& 2^{(3/2)k^+} \Big|  \iiint \Big[ e^{it \Phi(\xi,\eta,\sigma,p)} - e^{it(\jxi - \langle \xi-p\rangle +  2\eta\sigma\jxi^{-3} )} \Big]
	\\ &\qquad \qquad  \times F(t,\xi,0,0,p) 
\varphi_{k_0}(\eta\jxi^{-3/2}) \varphi_{k_0}(\s\jxi^{-3/2}) \frac{\what{\phi}(p)}{p}\,d\eta\,d\sigma \,dp \Big|
\lesssim \e_1^3 \cdot 2^{(-3/2 + 5\delta)m} m.
\end{split}
\end{align*}

Performing the integral in $(\eta,\s)$ using \eqref{ra14.3} we see that, for $|t| \approx 2^m$,
\begin{align}\label{ra40}
\begin{split}
& 2^{(3/2)k^+}\Big| I^{(2)}_{k_0,k_0}(t,\xi) - \jxi^3\frac{\pi}{t} L(t,\xi) \Big| \lesssim \e_1^3 2^{(-1-\delta_0)m},
\\ 
& L(t,\xi) := \int e^{it(\jxi - \langle \xi-p\rangle)} F(t,\xi,0,0,p) \varphi_{\leq -3m/5} (p) \frac{\what{\phi}(p)}{p} \,dp.
\end{split}
\end{align}

\medskip
\noindent
{\it Step 2.3}. 
Recall that $F(t,\xi,0,0,p) = \mathfrak{m}_2(\xi,\xi-p,\xi-p,\xi-p) |\wt{f}(\xi-p)|^2\wt{f}(\xi-p)$.
In order to obtain \eqref{asypv}, we want to show that, for $|t| \approx 2^m$, $|\xi|\approx 2^{k^+}$, 
\begin{align}\label{ra41}
\begin{split}
L(t,\xi) & = \int e^{it(\jxi - \langle \xi-p\rangle)}F(t,\xi,0,0,p)  \varphi_{\leq -3m/5} (p) \frac{\what{\phi}(p)}{p} \,dp 
\\
& = i  \sqrt{\frac{\pi}{2}} \, F(t,\xi,0,0,0) 
\, \sign(t\xi)  + O(2^{-9k^+/2} 2^{-\delta_0m}).
\end{split}
\end{align}

From the bounds on $\mathfrak{m}_2$ and the a priori assumptions on $\widetilde{f}$, we have
\begin{align*}
\begin{split}
& |F(t,\xi,0,0,p)-F(t,\xi,0,0,0)| 
\\
& \lesssim \big| \mathfrak{m}_2(\xi,\xi-p,\xi-p,\xi-p) - \mathfrak{m}_2(\xi,\xi,\xi,\xi) \big| |\wt{f}(\xi-p)|^3
\\
& \quad + \big|\mathfrak{m}_2(\xi,\xi,\xi,\xi) \big( |\wt{f}(\xi-p)|^2\wt{f}(\xi-p) - |\wt{f}(\xi)|^2\wt{f}(\xi) \big) \big| 
\\
& \lesssim \e_1^3 \left[  2^{-(13/2) k^+} |p| + 2^{-7 k^+} |p|^{1/2} 2^{\rho m} \right].
\end{split}
\end{align*}
This allows us to replace $F(t,\xi,0,0,p)$ by $F(t,\xi,0,0,0)$ and,
after Taylor expanding $\jxi - \langle \xi-p\rangle = p\xi\jxi^{-1} + O(|p|^2)$ and using that $|p| \lesssim 2^{-3m/5}$,
we obtain
\begin{align}
L(t,\xi) 
\label{ra43}
	& = F(t,\xi,0,0,0) \, \int e^{it p \xi \jxi^{-1}} \,
	\varphi_{\leq -3m/5} (p) \frac{\what{\phi}(p)}{p}  \,dp + O(2^{-9k^+/2} 2^{-\delta_0m}).
\end{align}
In \eqref{ra43} we can further replace $\what{\phi}(p)$ by $\what{\phi}(0)$ 
and eventually dispense of the cutoff in $p$ (again via integration by parts), arriving at
\begin{align}\label{ra47}
\begin{split}
L(t,\xi) & = F(t,\xi,0,0,0) \, \widehat{\phi}(0) \, \int e^{it p \xi \jxi^{-1}} 
	\pv \frac{1}{p} \,dp + O(2^{-9k^+/2}  2^{-\delta_0m})
\\
& = F(t,\xi,0,0,0) \what{(\pv 1/x)} (-t\xi\jxi^{-1}) + O(2^{-9k^+/2} 2^{-\delta_0m}).
\end{split}
\end{align}
Using the last identity in \eqref{Fsign} gives us \eqref{ra41}.

%
%
%
%
%

\medskip
\noindent
{\it Step 3: Case $k_1>k_0$}.
To conclude the rigorous derivation of the asymptotics \eqref{asypv} 
we need to show that the remaining contributions from $I^{(2)}_{k_1,k_2}(s,\xi)$, see \eqref{ra31.5}, 
satisfy $O(2^{-\delta_0m})$ bounds when integrated over $|s|\approx 2^m$ and measured in $\jxi^{-3/2}L^\infty$.
This can be done similarly to the analogous estimate for the integral $I^{(1)}_{k_1,k_2}$; 
see the argument starting from \eqref{ra21}.

First, we observe that we  may restrict to $|p| \ll 2^k $.
Indeed, if $|p| \gtrsim 2^k$ the $\pv$ in \eqref{ra31.5} is not singular and contributes a very small loss in view of \eqref{raxi};
moreover, we can integrate by parts 
both in $p$ 
and, depending on which profile is hit by $\partial_p$ then integrate in one of the variables $\eta$ or $\s$ or $\eta-\s$.

Under our assumption that all input frequencies have size about $2^k$ 
we have the following analogue of \eqref{ra21}:
\begin{align}\label{ra21'}
\begin{split}
|\partial_\sigma\Phi(\xi,\eta,\s,p)| = \Big| \frac{\xi+\s-p}{\langle \xi+\sigma -p\rangle} 
  - \frac{\xi+\s+\eta-p}{\langle \xi+\sigma+\eta-p \rangle} \Big|
  \gtrsim 2^{k_1}2^{-(3/2)k^+}, 
\end{split}
\end{align}
see the definition of $\Phi$ in \eqref{ra30}. 
Integrating by parts in $\s$ to obtain an inequality analogous to \eqref{ra22},
and then using Lemma \ref{lemmamultilin2} to estimate similarly to \eqref{ra23}, we obtain, for $|t|\approx 2^m$
\begin{align}\label{ra23'}
\begin{split}
\Big\| \jxi^2 \sum_{k_1>k_0} I^{(2)}_{k_1,k_2}(t) \Big\|_{L^2}
\lesssim \e_1^3 2^{-3m/2} \cdot  2^{(\rho +\delta)m}.
\end{split}
\end{align}

To obtain the desired pointwise bound in $\jxi^{-3/2}L^\infty$ it is enough to interpolate \eqref{ra23'}
with the weighted $L^2$-bound
\begin{align}\label{ra25'}
\begin{split}
\Big\| \jxi \partial_\xi \int_0^t \sum_{k_1>k_0}  I^{(2)}_{k_1,k_2}(s,\xi) \, \tau_m(s) ds \Big\|_{L^2} & \lesssim  
 \e_1^3 2^{m \rho},
\end{split}
\end{align}
which we now prove.

We first need an analogue of \eqref{raid0} for the phase $\Phi_2=\Phi_2(\xi,\eta,\eta',\s)$, 
see \eqref{ra30}.
By defining $X_{a} := \langle a \rangle \partial_a$ we have
\begin{align}\label{raid2}
(X_\xi + X_\eta + X_{\eta'} + X_\s )\Phi_2 = \xi-\eta+\eta'-\s = p.
\end{align}
Note that this is essentially the same identity appearing in \eqref{Cubotherid}, and that we use
to establish a weighted $L^2$ bound for the singular cubic terms 
when the inputs are away from the degenerate frequencies $\pm\sqrt{3}$.

The identity \eqref{raid2} can be applied to the time integral of $\jxi \partial_\xi I^{(2)}_{k_1,k_2}$
- expressed in the original variables $(\eta,\eta',\s)$, see \eqref{raI2}) - to integrate by parts in $(\eta,\eta',\s)$.
This procedure gives
\begin{align}
\label{ra50}
& \Big| \int_0^t \jxi\partial_\xi I^{(2)}_{k_1,k_2}(s,\xi) \, ds \Big| \lesssim 
	\Big| \int_0^t C(s,\xi) \tau_m(s) \, ds \Big| + \Big| \int_0^t D(s,\xi) \tau_m(s) \, ds \Big| + O(\e_1^3 2^{m\rho}),
\\ \label{ra51}
& C(s,\xi) := \iiint s \, e^{is \Phi_2} F_2(t,\xi,\eta,\eta',\s) 
\varphi_{k_1}(\eta-\xi+p) \varphi_{k_2}^{(k_0)}(\s-\xi+p) \what{\phi}(p) \,d\eta\,d\eta'\,d\s,
\\ \label{ra52}
& D(s,\xi) := \iiint e^{is \Phi_2} F_2(t,\xi,\eta,\eta',\s) 
\varphi_{k_1}(\eta-\xi+p) \varphi_{k_2}^{(k_0)}(\s-\xi+p) \, 
	X \frac{\what{\phi}(p)}{p} \,d\eta\,d\eta'\,d\s,
\end{align}
where $p := \xi-\eta+\eta'-\s$ and $X := ( X_\xi + X_\eta + X_{\eta'} + X_\s )$;
the $O(\e_1^3 2^{m\rho})$ in \eqref{ra50} includes similar and lower order terms,
such as the three terms where the derivatives $X_a$ hit the profiles 
(these are similar to the first term in \eqref{ra24}) or the symbol and the various cutoffs.

The term \eqref{ra51} comes from canceling the $\pv 1/p$ with the $p$ factor in the right-hand side of \eqref{raid2};
note the factor of $s$ coming from differentiating the phase.
Integrating by parts in all the three variables $\eta,\eta'$ and $\s$ this can be estimated in 
$L^2$ by $2^{m} X_{k,m}^3 \lesssim 2^m(\e_1 2^{-(3/4-)m})^3$,
where, recall, $X_{k,m}$ is the quantity defined in \eqref{wproofnot2.0}-\eqref{wproofnot2} for $k\leq0$,
and, more generally, in \eqref{Xvar}-\eqref{Xvarest}.
Upon integration over time this is an acceptable contribution for \eqref{ra25'}.

To estimate the time integral of  \eqref{ra52} 
we note that 
\begin{align}
X \frac{\what{\phi}(p)}{p} = \Phi_2(\xi,\eta,\eta',\s) \, \partial_p \frac{\what{\phi}(p)}{p}.
\end{align}
Using the $\Phi_2$ factor we can then integrate by parts in $s$ through the usual
identity $e^{is\Phi_2} = (i\Phi_2)^{-1}\partial_s e^{is\Phi_2}$
and obtain boundary terms plus additional time integrated terms;
since these can all be treated similarly we just look at the main time-integrated term, that is
(recall the definition of $F$ in \eqref{raI2})
\begin{align}\label{ra52.1}
\begin{split}
\int_0^t \iiint e^{is \Phi_2} \big[ \partial_s \wt{f}(\eta) \big]  \wt{f}(\eta')  \wt{f}(\s) 
	\, \varphi_{k_1}
	\varphi_{k_2}^{(k_0)} \, \mathfrak{m}_2(\xi,\eta,\eta',\s) 
	\, \partial_p \frac{\what{\phi}(p)}{p} \,d\eta\,d\eta'\,d\s \, ds.
\end{split}
\end{align}
The main observation here is that we can
write $\partial_p (1/p) = - \partial_\s (1/p)$ and integrate by parts in $\s$.
The worst terms is the one where $\partial_\s$ hits the exponential factor which 
will cause an additional loss of $s\approx 2^m$.
Applying Lemma \ref{lemmamultilin2} and \eqref{dtfL2} we get
\begin{align}
\begin{split}
{\| \eqref{ra52.1} \|}_{L^2} \lesssim 2^{2m} {\| P_{\sim k} \partial_s f \|}_{L^2} 
	{\| P_{\sim k} e^{-it\jnab} \W f(t) \|}_{L^\infty}^4 \lesssim \e_1^5.
\end{split}
\end{align}
A similar argument is also used and further detailed after \eqref{Cubotherid}.
This concludes the proof of \eqref{ra25'} and of the asymptotic formula \eqref{asypv}
in the main case $\{\kappa_1,\kappa_2,\kappa_3\}=\{+,+.-\}$.

\medskip
\subsubsection{Estimates of \eqref{Idelta}-\eqref{Ipv} for $\{\kappa_1,\kappa_2,\kappa_3\}\neq\{+,+,-\}$}\label{ssecraIoth}
In the non-resonant cases $\{\kappa_1,\kappa_2,\kappa_3\}\neq\{+,+,-\}$,
we show how the integrals \eqref{Idelta} and \eqref{Ipv} 
can be absorbed into the remainder term appearing in \eqref{LinfSasy}-\eqref{LinfSrem}. 
For ease of reference we recall the formulas for these integrals:
\begin{align}
\label{Idelta'}
& I_\delta(t) := \iiint e^{it \Phi_{\kappa_1 \kappa_2 \kappa_3}(\xi,\eta,\eta',\sigma)} F(t,\xi,\eta,\eta',\sigma) 
  \delta(p)\,d\eta\,d\eta' \,dp, 
\\
\label{Ipv'}
& I_{\operatorname{p.v.}}(t) := \iiint e^{it \Phi_{\kappa_1 \kappa_2 \kappa_3}(\xi,\eta,\eta',\sigma)} 
  F(t,\xi,\eta,\eta',\sigma) \frac{\widehat{\phi}(p)}{p}\,d\eta\,d\eta' \,dp,
\end{align}
where
\begin{align}
F(t,\xi,\eta,\eta',\s) =  \mathfrak{m}(\xi,\eta,\eta',\s) \wt{f}(t,\eta)\wt{f}(t,\eta')\wt{f}(t,\s).
\end{align}
and, see  \eqref{asypstar}-\eqref{asyPhi},
\begin{align}\label{asyPhi'}
\begin{split}
& \Phi_{\kappa_1 \kappa_2 \kappa_3}(\xi,\eta,\eta', \s) 
	:= \langle \xi \rangle -\kappa_1 \langle \eta \rangle - \kappa_2 \langle \eta' \rangle - \kappa_3 \langle 
	\s \rangle,
	\qquad \s := \xi - \eta  - \eta' - p,
\end{split}
\end{align}
having chosen, without loss of generality, a fixed combination of the the signs $\l_\ast,\mu_\ast,\dots$


The main idea to estimate in $\jxi^{-3/2}L^\infty$
(the time integrals of) \eqref{Idelta'} and \eqref{Ipv'} is similar to the one used in the two previous paragraphs,
based on interpolating the $\jxi^{-2} L^2$ and $\jxi^{-1}\dot{H}^1$ norms via \eqref{rainter}.

Since the phase \eqref{asyPhi'} does not have stationary points in $(\eta,\eta')$ at which it simultaneously vanishes,
we will show below that can obtain fast decay for the $L^2$ norm:
\begin{align}\label{raothest1}
\Big\| \jxi^2 
	\int_0^t I(s) \, \tau_m(s) ds \Big\|_{L^2} & \lesssim  \e_1^3 2^{-m/2} 2^{(\rho + 2\delta) m }, \qquad I \in \{I_{\delta}, I_\pv\}.
\end{align}
Moreover, thanks to \eqref{propCubother0} in Proposition \ref{propCubother}, 
we have that
\begin{align}\label{raothest2}
\Big\| \jxi \partial_\xi \int_0^t I(s) \, \tau_m(s) ds \Big\|_{L^2} & \lesssim  \e_1^3 2^{\alpha m} 
	, \qquad I \in \{I_{\delta}, I_\pv \}.
\end{align}
Interpolating \eqref{raothest1} and \eqref{raothest2} we arrive at a bound consistent with \eqref{LinfSrem}, for $\delta$ and $\delta_0$ small enough.
We are just left with proving \eqref{raothest1}.

\medskip
\noindent
{\it Proof of \eqref{raothest1}.}
The case $(\kappa_1,\kappa_2,\kappa_3)= (-,-,-)$ is the easiest since $|\Phi_{---}| \gtrsim 1$.
The cases $\{\kappa_1,\kappa_2,\kappa_3\} = \{-,-,+\}$ and $\{+,+,+\}$ are similar, so let us just concentrate on the latter sign combination.
Moreover, it suffices to only look at the more complicated case of $I_\pv$.

We start by recalling that 
\begin{align}\label{raothstatpts}
\begin{split}
& \nabla_{\eta,\eta'}\Phi_{+++} = 0 \quad \Longleftrightarrow \quad \eta=\eta'=\s = (1/3)(\xi-p) =: \xi_0,
\\
& \Phi_{+++}(\xi,\xi_0,\xi_0,\xi_0) = \jxi - 3 \langle\xi_0\rangle,
\end{split}
\end{align}
see \eqref{asystatpts}.
We let $|t|\approx 2^m$,  and dyadically localize the frequencies into
\begin{align}\label{raothloc}
|\xi|\approx 2^k, \quad |\eta| \approx 2^{k_1}, \quad |\eta'| \approx 2^{k_2}, \quad, |\s| \approx 2^{k_3}, \quad |p| \approx 2^q,
\end{align}
denote by $\varphi_{\underline{k}}$ the standard smooth cutoff that localizes to the region where \eqref{raothloc} holds,
and define the localized version of the integral \eqref{Ipv'} by
\begin{align}\label{Ipv'loc}
& I_{\underline{k}}(t) := \iiint e^{it \Phi_{+++}(\xi,\eta,\eta',\sigma)} 
  F(t,\xi,\eta,\eta',\sigma) \frac{\widehat{\phi}(p)}{p}\, \varphi_{\underline{k}}(\xi,\eta,\eta',\s) \, d\eta\,d\eta' \,dp.
\end{align}

By the usual arguments (including dealing with very small values of $|p|$ by using the $\pv$) we may reduce to prove a slightly stronger bound 
than \eqref{raothest1} (with a factor of $\delta$ instead of $2\delta$, say) for $I_{\underline{k}}$.
We split the proof into several cases:

\medskip \noindent
{\it Case $q \geq \min(k, -k)-10$.}
Let us first look at the case $k \leq 0$ for which we have $|p| \gtrsim |\xi|$. 
In this case the $\pv$ is not singular
and \eqref{Ipv'loc} is a `regular' cubic term, up to an additional factor of $2^{-k}$ coming from $1/p$;
however, since $k \geq (-2\alpha-2\delta_0) m$, see \eqref{raxi}, this represent a very small loss.
When $k_3 \geq -m/3$, an integration by parts in $p$ suffices to obtain the desired bound.
When instead $k_3 \leq -m/3$ we can use directly an $L^\infty \times L^\infty \times L^2$ trilinear H\"older estimate (in physical space):
the linear decay estimate applied to the $L^\infty$ norms gives two factors of $\e_12^{-m/2}$,
and \eqref{lembb1} applied to the $L^2$ norm of the profile with frequency $\approx 2^{k_3}$ gives an additional factor of $2^{k_3} 2^{\alpha m}$,
yielding a stronger bound than the right-hand side of \eqref{raothest1}.


When $k\geq 0$ we have $|p| \gtrsim |\xi|^{-1}$, and we can use a similar argument. 
Notice that we do not need to worry about  the loss of a possibly large factor of $|\xi|$ from $1/|p|$ thanks to the upperbound in \eqref{raxi}.
From now on we may assume that $|p| \ll \min(|\xi| , |\xi|^{-1})$.

\medskip \noindent
{\it Case $\max(|k_1-k_2|,|k_1-k_3|,|k_2-k_3|) \geq 5$.}
Without loss of generality we may assume $k_1 \geq k_2 + 5$.
Also, we may assume $k_1 \geq -m/3$ for otherwise an $L^2\times L^\infty\times L^\infty$ 
estimate will give a bound of 
$C \cdot \e_1 2^{k_1} 2^{\alpha m} \cdot (\e_1 2^{-m/2})^2$ for the $L^2$ norm of \eqref{Ipv'loc}
which is better than \eqref{raothest1}.
Let us also assume that $k_1\leq 10$ since the complementary case is easier to treat.

Since $|\eta-\eta'| \gtrsim 2^{k_1}$, we have 
\begin{align}\label{rao11}
|\partial_\eta \Phi_{+++}| \gtrsim 2^{k_1^-} .
\end{align}
Integrating by parts in $\eta$ in \eqref{Ipv'loc} we obtain one main term when the derivative hits the profiles, that is,
\begin{align}\label{rao12}
\iiint e^{it \Phi_{\kappa_1 \kappa_2 \kappa_3}(\xi,\eta,\eta',\sigma)} 
 \frac{\mathfrak{m} \varphi_{\underline{k}}}{t \, \partial_\eta \Phi_{+++}}(\xi,\eta,\eta',\sigma) 
 \, \partial_\eta \big[ \wt{f}(\eta) \wt{f}(\xi-\eta-\eta'-p) \big] 
 \wt{f}(\eta') \, \frac{\widehat{\phi}(p)}{p}\,d\eta\,d\eta' \,dp,
\end{align}
plus other faster decaying 
terms when the derivative hits the symbol or the various cutoffs.

One can check that a bound of $C 2^{-k_1^-}$ 
holds for the norm of the trilinear operator 
associated to the (localization of the) symbol $(\partial_\eta \Phi_{+++})^{-1}$, consistently with \eqref{rao11}. 
Then, an $L^2 \times L^\infty \times L^\infty$ estimate, using \eqref{raxi} and \eqref{radk1}, gives us
\begin{align}\label{rao13}
\begin{split}
{\| \jxi^2 \eqref{rao12} \|}_{L^2} & \lesssim 2^{-m} 2^{2k} \cdot 2^{-k_1^-} 
	\max\big( {\| \partial_\eta \varphi_{k_1} \wt{f} \|}_{L^2}, {\|  \partial_\eta \varphi_{k_3} \wt{f} \|}_{L^2} \big) \cdot \big(\e_1 2^{-m/2} \big)^2
	\\
	& \lesssim \e_1^3 2^{-2m}  2^{2(\alpha+\delta_0)m} 2^{m/3} \cdot 2^{\rho m},
\end{split}
\end{align}
which suffices after time integration.

\medskip 
\noindent
{\it Case $\max(|k_1-k_2|,|k_1-k_3|,|k_2-k_3|) \leq 5$.}
In particular we must also have that $k \leq \max(k_1,k_2,k_3) + 20$.
Motivated by \eqref{raothstatpts}, we further decompose dyadically
\begin{align*}
|\xi-2\eta-\eta'-p| \approx 2^{n_1}, \qquad |\xi-\eta-2\eta'-p| \approx 2^{n_2},
\end{align*}
by inserting smooth cutoffs, that we implicitly include into $\varphi_{\underline{k}}$.

\medskip \noindent
{\it Subcase $\max(n_1,n_2) \geq -m/2 + \delta m$.}
Let us assume, without loss of generality, that $n_1\geq n_2$, so that $|\xi-2\eta-\eta'-p| \gtrsim 2^{-m/2+\delta m}$.
Note that, in view of our restriction we have $n_1 \leq k_1 + 10$.
In this case we can integrate by parts in $\eta$ in the expression \eqref{Ipv'loc} using that
\begin{align*}
|\partial_\eta \Phi_{+++}| \gtrsim 2^{-3k_1^+} 2^{n_1},
\end{align*}
and that an estimate of $2^{-n_1+3k_1}$ holds for the norm of the trilinear operator associated to $(\partial_\eta\Phi_{+++})^{-1}$.
Up to easier and faster decaying terms, this integration by parts gives us the 
same main term \eqref{rao12} above.
Estimating as in \eqref{rao13} gives 
\begin{align}\label{rao13'}
\begin{split}
{\| \jxi^2 \eqref{rao12} \|}_{L^2} \lesssim 2^{-m} \cdot 2^{-n_1} 2^{3k_1^+} \cdot 
	\max\big( {\| \partial_\eta \varphi_{k_1} \wt{f} \|}_{L^2}, {\| \partial_\eta \varphi_{k_3} \wt{f} \|}_{L^2} \big) \cdot \big(\e_1 2^{-m/2} \big)^2
	\\ \lesssim \e_1^3 2^{-2m} \cdot 2^{(1/2-\delta +\rho)m}.
\end{split}
\end{align}

\medskip \noindent
{\it Subcase $\min(n_1,n_2) \leq -m/2 + \delta m$.}
Let us denote by  $I_{\underline{k},0}$ the localization of \eqref{Ipv'loc} to this region,
and note that in this case both $\eta$ and $\eta'$ are very close to $\xi_0:=(1/3)(\xi-p)$, 
and so is $\s=\xi-\eta-\eta'-p$.
More precisely, we can see that 
\begin{align*}
|\Phi_{+++}(\xi,\eta,\eta',\s)| \gtrsim \big| \jxi - 3 \langle \xi_0 \rangle \big| + O(2^{-m/2 + \delta m}) \gtrsim \jxi^{-1},
\end{align*}
since $|p| \ll \jxi^{-1}$, and we have the upperbound \eqref{raxi} for $|\xi|$.
It is then possible to integrate by parts in $s$ in the time integral of $I_{\underline{k},0}$,
incurring in a minimal loss of $2^{k^+}$. 
Then, from the usual trilinear H\"older estimates, using the bound \eqref{dtfL2} for $\partial_t f$ 
the $H^4$-type a priori bound in \eqref{propbootfas} and the $L^\infty_x$ decay, we can see that
\begin{align*}
\Big\| \jxi^2 \int_0^t I_{\underline{k},0}(s,\xi) \, \tau_m(s) ds \Big\|_{L^2}  
  \lesssim \e_1^3 2^{-m+p_0m} + 2^{m} \cdot \e_1^5 2^{-2m+ p_0 m}.
\end{align*} 
This concludes the proof of \eqref{raothest1} and of the main Proposition \ref{propLinfS}.



\bigskip
\section{Estimates of lower order terms}\label{secw'}

This section contains estimates for all the terms that have not been treated
in Sections \ref{secBoot}-\ref{secLinfS}:

\begin{itemize}

\smallskip
\item In Subsection \ref{secQRother} we prove the weighted bound for all the quadratic interactions that 
are not covered in Section \ref{secwR}.

\smallskip
\item In Subsection \ref{SsecSob} we complete the proof of the a priori bounds on the Sobolev-type
component of our norm by estimating the regular quadratic terms $\mathcal{Q}^R(f,f)$,
which were left out from the proofs in Subsection \ref{secBoot};
see Lemma \ref{lemQRexp} and the first line of \eqref{QRexp}.

\smallskip
\item In Subsection \ref{SsecLinfR} we prove the Fourier-$L^\infty$ bound for the regular quadratic terms,
and for all other terms that are not the main ones covered in Section \ref{secLinfS}.

\smallskip
\item Finally, Subsection \ref{Cubother} contains the estimates for the weighted norms 
of all the cubic interactions $\mathcal{C}^{S1,2}_{\k_1\k_2\k_3}$, see \eqref{CubicS}, which were left out from 
the analysis of Section \ref{secwL}; these are of two types:
the 
interactions with $\{\k_1,\k_2,\k_3\}=\{++-\}$ with not all frequencies close to $\sqrt{3}$ (or $-\sqrt{3}$),
and the interactions corresponding to all the other signs combinations.

\end{itemize}

\medskip
\subsection{Other quadratic interactions}\label{secQRother}
Here we estimate the weighted norm of the regular quadratic term $\mathcal{Q}^{R}$, see \eqref{QR1},
for all the interactions that are not the main ones considered in Section \ref{secwR}.
These are:
\begin{itemize}

\item The interactions with $(\iota_1\iota_2)=(++)$ that do not satisfy \eqref{wRmain} 
or, more precisely \eqref{red1}, \eqref{red2}; 
with the notation from 
\eqref{wproofnot1++} and \eqref{wproofdecest}-\eqref{wproofdecpar}
these are interactions where we either have $\ell > -7\beta'm$ or $k_1>-10$.

\item The interactions with $(\iota_1\iota_2)\neq(++)$.

\end{itemize}

\noindent
By estimating these terms we will complete the proof of the main Proposition \ref{prowR}.
Several of the arguments that we are going to use below are along the same lines as in Section \ref{secwR},
and simpler in many cases, so we will omit some details.

\smallskip
\subsubsection{Notation and preliminary reductions}\label{ssecQR}
Let us begin by recalling some definitions. 
Recall the notation \eqref{wproofQR1}-\eqref{wproofQR2}, 
\begin{align}\label{othQR}
\begin{split}
\mathcal{Q}^R_{\iota_1\iota_2}[a,b](t,\xi) & = -\iota_1\iota_2 \iint e^{it \Phi_{\iota_1\iota_2}(\xi,\eta,\s)} 
  \, \mathfrak{q}_{\iota_1\iota_2}(\xi,\eta,\s) \, \wt{a}_{\iota_1}(\eta) \, \wt{b}_{\iota_2}(\s) \, d\eta \, d\s,
\\
\Phi_{\iota_1\iota_2}(\xi,\eta,\s) & := \jxi - \iota_1 \jeta - \iota_2 \jsig,
\end{split}
\end{align}
where we omit the irrelevant signs $\kappa_1,\kappa_2$ and 
the indicator functions according to Remark \ref{Remkappas}; see also Remark \ref{remdxi} and Lemma \ref{lemdxiQR}.
Recall 
the definition of the main localized operator from \eqref{wproofnot1}: 
\begin{align}\label{othImain}
\begin{split}
I^{p,k_1,k_2}_{\iota_1\iota_2}[a,b](t,\xi) & := \iint e^{it \Phi_{\iota_1\iota_2}(\xi,\eta,\s)} 
  \, \varphi_p^{(p_0)}\big(\Phi_{\iota_1\iota_2}(\xi,\eta,\s)\big) 
  \, \mathfrak{q}_{\iota_1\iota_2}(\xi,\eta,\s) \, 
  \\ & \qquad \times \varphi_{k_1}(\eta) \wt{a}_{\iota_1}(\eta) \, \varphi_{k_2}(\sigma) \wt{b}_{\iota_2}(\s) \, d\eta \, d\s,
  \qquad p_0 = -m + \d m,
\\
\Phi_{\iota_1\iota_2}(\xi,\eta,\s) & := \jxi - \iota_1 \jeta - \iota_2 \jsig,
\end{split}
\end{align}
where $t\approx 2^m$, $m=0,1,\dots$;
also recall from the definition of $\mathfrak{q}_{\iota_1\iota_2}$ in \eqref{QR2}, 
with $\mu_R$ as in \eqref{muR}-\eqref{muR'}, 
and Remark \ref{Remkappas}, that we may assume
for all $a,b,c \geq 0$ and arbitrarily large $N$,
\begin{align}\label{HFloss}
|\partial_\xi^a \partial_\eta^b \partial_\sigma^c  \mathfrak{q}_{\iota_1\iota_2}(\xi,\eta,\sigma) |
  \lesssim \frac{1}{\jeta\jsig} \cdot \big[ 1 + \inf_{\mu,\nu}|\xi - \mu\eta-\nu\sigma| \big]^{-N}
  \cdot R(\eta,\s)^{a+b+c+1},
\end{align}
where $R(\eta,\s) \approx \min(\jeta,\jsig)$, see \eqref{Rfactor}. 
Notice in particular that the symbol decays very fast when one of the frequencies $(\xi,\eta,\s)$ 
is much larger than the other two;
this will allow us to concentrate on ``diagonal" interactions where $\max(|\xi|,|\eta|,|\s|) \approx \med(|\xi|,|\eta|,|\s|)$.

In view of the preliminary reductions made in Section \ref{secwR}, see in particular the estimates leading 
to Lemma \ref{lemred}, it suffices to obtain the two following estimates:
\begin{align}\label{othest++l}
\begin{split}
2^m 2^k \left\| \varphi_k(\xi) \chi_{\ell,\sqrt{3}}(\xi)  \int_{0}^t I^{p,k_1,k_2}_{++}(f,f)(\xi,s) 
  \, \tau_m(s)\, ds \right\|_{L^2_\xi} \lesssim \e_1^2 \, 2^{-\beta\ell} 2^{-2\beta'm}, 
  \\ \ell > - 7\beta'm \qquad \mbox{or}
  \qquad  k_1 > - 10,
\end{split}
\end{align}
and 
\begin{align}\label{othestsigns}
\begin{split}
2^m 2^k \left\| \varphi_k(\xi) \chi_{\ell,\sqrt{3}}(\xi)
  \int_{0}^t I^{p,k_1,k_2}_{\iota_1\iota_2}(f,f)(s,\xi) \, \tau_m(s)\, ds \right\|_{L^2_\xi} \lesssim 
  \e_1^2 \, 2^{-\beta\ell} 2^{-2\beta'm},
  \\ \mbox{with} \qquad (\iota_1\iota_2) \in \{ (+-), (-+), (--) \}.
\end{split}
\end{align}
We assume without loss of generality that
$$
k_1 \geq k_2.
$$

Notice that in addition to the localizations already present in Lemma \ref{lemred},
we have included here also a localization in $|\xi|\approx 2^k$,
and a factor of $2^k$ on the left-hand sides of \eqref{othest++l}-\eqref{othestsigns},
which is consistent with the fact that $\jxi\nabla_\xi\Phi_{\iota_1\iota_2} = \xi$;
see the formula \eqref{wproof5}, and recall that this factor 
was disregarded in the estimates of Section \ref{secwR} since there we were only looking at the case $|\xi|\approx \sqrt{3}$.
For small $\xi$ this factor turns out to be helpful in the analysis of the signs combinations other than $(++)$. 

Notice also that in both \eqref{othest++l} and \eqref{othestsigns} 
we have discarded the summations over $(k,k_1,k_2)$ (and $p$) and reduced ourselves to a bound 
for fixed triples $(k,k_1,k_2)$.
To justify this reduction it suffices to show how to bound the sums over $\max(k,k_1,k_2)\geq 10m$ or $\min(k,k_1,k_2) \leq -10m$,
because then the sum over the remaining $O(m^3)$ terms ($O(m^4)$ when we include $p$) 
can be accounted for by the factor of $2^{-2\beta' m}$ (and the lack of the $2^{\alpha m}$ factor) 
as we did for the parameters $k_1,k_2$ and $p$ before (see the paragraph after Lemma \ref{lemred}).

\def\conv{2^{-20| k^+ - k_1^+|}} 

Let us briefly explain how to deal with the cases $\max(k,k_1)\geq 10m$ or $\min(k,k_2) \leq -10m$.
Observe that the pointwise bound on the symbol from \eqref{HFloss} gives us, on the support of the integral,
\begin{align}\label{sest}
|\mathfrak{q}_{\iota_1\iota_2}(\xi,\eta,\sigma) | \lesssim 2^{-k_1^+} \conv \big[ 1 + \inf_{\mu,\nu}|\xi - \mu\eta-\nu\sigma| \big]^{-5};
\end{align}
then, Young's inequality yields
\begin{align}\label{HFksum0}
\begin{split}
2^m 2^k & \left\| \varphi_k(\xi) \chi_{\ell,\sqrt{3}}(\xi)  \int_{0}^t I^{p,k_1,k_2}_{++}(f,f)(\xi,s) 
  \, \tau_m(s)\, ds \right\|_{L^2_\xi} 
  \\ & \lesssim 2^{2m} 2^k \cdot 2^{-k_1^+} \conv \cdot {\| \varphi_{k_1} \wt{f} \|}_{L^2} {\| \varphi_{k_2} \wt{f} \|}_{L^1}
  \\ & \lesssim 2^{2m} 2^k \cdot 2^{-k_1^+} \conv \cdot \min\big(2^{-4k_1},2^{k_1/2}\big) 2^{\alpha m} \e_1 
  \cdot \min \big(2^{k_2}, 2^{-7k_2/2}\big) 2^{\alpha m} \e_1,
\end{split}
\end{align}
having using the a priori bound on the $H^4$ Sobolev norm in the last inequality.
If $\max(k,k_1)\geq 10m$ and $|k_1-k| <10$ we can use the factor of $2^{-5k_1}$ in \eqref{HFksum0}
to sum over $k$ and $k_1$ (the sum over $k_2$ can be done independently) 
and obtain a stronger upper bound than the right-hand sides of \eqref{othest++l}-\eqref{othestsigns};
when instead $|k-k_1| \geq 10$ we can use the decay of the symbol away from the diagonal
which results in the extra power of $2^{-20\max(k,k_1)}$ in \eqref{HFksum0}. 

In the case $k \leq -10m$ 
the factor of $2^k$ in front of the estimate \eqref{HFksum0} already allows to sum over $k$ 
and again obtain stronger bounds than \eqref{othest++l}-\eqref{othestsigns}.
Similarly \eqref{HFksum0} suffices if $k_2 \leq -10m$.

\medskip
Before proving \eqref{othest++l} and \eqref{othestsigns} we show how to deal with relatively large input frequencies,
and, in particular, how to handle the non-standard estimate for the symbol $\mathfrak{q}$ appearing in \eqref{HFloss}.

\smallskip
\subsubsection{High frequencies}\label{ssecothHF}
From the estimates for the symbol $\mathfrak{q}$  in \eqref{HFloss}.
we see that $\mathfrak{q}$ is essentially smooth, and fast decaying in the quantity
$\inf_{\mu,\nu}|\xi - \mu\eta-\nu\sigma|$, but has the non-standard feature
that its derivatives in $\xi,\eta,\sigma$ might grow for frequencies larger than $1$.
Therefore, in each of our integration by parts arguments there is a potential 
loss of a factor of $R(\eta,\s)$ when derivatives hit the symbol. 
However, thanks to the $H^4$ control on our solution we can comfortably handle this, 
using the following lemma:

\begin{lem}[High frequencies]\label{lemHF}
Assume $k_2\leq k_1$. For all $s\approx 2^m$ we have
\begin{equation}\label{HFest}
\begin{split}
& {\| \varphi_k (\xi) I^{p,k_1,k_2}_{\iota_1\iota_2}(f,f)(s,\xi) \|}_{L^2} \lesssim 2^{-k_1^+} \conv \cdot
	 {\| \varphi_{k_1} \wt{f} \|}_{L^2} 
	 \\ &\qquad \qquad \qquad \cdot \min \Big( {\| \varphi_{k_2} \wt{f} \|}_{L^1},  
	2^{-m-k_2^-} \big( (2^{k_2^+}+2^{-k_2^-}) {\| \varphi_{k_2} \widetilde{f} \|}_{L^1} 
	+ {\| \partial_\xi(\varphi_{k_2} \wt{f}) \|}_{L^1} \big) \Big).
 \end{split}
 \end{equation}
\end{lem}

\begin{proof}
Bringing the absolute values inside the integral gives us the basic bound
\begin{align*}
|I^{p,k_1,k_2}_{\iota_1\iota_2}(f,f)(t,\xi)| 
  \lesssim \iint
  |\mathfrak{q}_{\iota_1\iota_2}(\xi,\eta,\s)|
  \cdot \varphi_{k_1}(\eta) |\wt{f}_{\iota_1}(\eta)| \cdot \varphi_{k_2}(\sigma) |\wt{f}_{\iota_2}(\s)| \, d\eta \, d\s,
\end{align*}
which, using \eqref{HFloss} (with $a=b=c=0)$,
$R(\eta,\sigma) \approx 2^{k_2^+}$, and Young's inequality, implies 
\begin{align}\label{HFestpr2}
{\| \varphi_k (\xi) I^{p,k_1,k_2}_{\iota_1\iota_2}(f,f)(s,\xi) \|}_{L^2} 
  \lesssim  2^{-k_1^+}  \conv \cdot {\| \varphi_{k_1} \wt{f} \|}_{L^2} \cdot  {\| \varphi_{k_2} \wt{f} \|}_{L^1}.
\end{align}

To prove \eqref{HFest} we need to show, for $k_2 \geq -m/2$, that
\begin{align}\label{HFestpr1}
\begin{split}
{\| \varphi_k (\cdot) I^{p,k_1,k_2}_{\iota_1\iota_2}[f,f](\cdot,s) \|}_{L^2} 
  \lesssim 2^{-k_1^+}  \conv  \cdot {\| \varphi_{k_1} \wt{f} \|}_{L^2} \cdot 2^{-m-k_2^-}
  \\ \cdot \big( {\| \partial_\s(\varphi_{k_2} \wt{f}) \|}_{L^1} + (2^{k_2^+}+2^{-k_2^-}) {\|\varphi_{k_2} \wt{f}\|}_{L^1} \big).
\end{split}
\end{align}
This is done integrating by parts in $\s$ first, and then estimating as in \eqref{HFestpr2} above.
More precisely, we look at the formula \eqref{othImain} 
and note that $|\partial_\s \Phi_{\iota_1\iota_2}| \approx 2^{k_2^-}$.
Using this, and the usual identity 
$(is\partial_\s\Phi_{\iota_1\iota_2})^{-1}\partial_\s e^{is\Phi_{\iota_1\iota_2}} =  e^{is\Phi_{\iota_1\iota_2}}$
we can integrate by parts in $\s$ gaining the factor $2^{-m-k_2^-}$.
When $\partial_\s$ hits $\varphi_p^{(p_0)}$ we use the argument that led to
Remark \ref{rem1/Phi} and repeat the integration by parts as needed.
If $\partial_\s$ hits the symbol $\mathfrak{q}$, we use \eqref{HFloss} to deduce a bound of
\begin{align*}
C \sup_{|\eta|\approx2^{k_1},|\s|\approx2^{k_2}} R(\eta,\s)^2
  \cdot 2^{-k_1^+}\conv \cdot {\| \varphi_{k_1} \wt{f} \|}_{L^2} \cdot 2^{-k_2^+} 2^{-m-k_2^-}{\| \varphi_{k_2} \wt{f} \|}_{L^1}
  \\
\lesssim 2^{-k_1^+} \conv \cdot {\| \varphi_{k_1} \wt{f} \|}_{L^2} \cdot 2^{k_2^+}
	2^{-m-k_2^-}{\| \varphi_{k_2} \wt{f} \|}_{L^1}.
\end{align*}
If instead $\partial_\s$ falls on $\varphi_{k_2} \wt{f}$, we estimate using Young's
and finally obtain \eqref{HFestpr1}.
%
\end{proof}

Let us define, just for the purpose of the estimate in this section, the following variant of $X_{k,m}$, 
see \eqref{wproofnot2.0},
which we still denote in the same way, to take into account also frequencies $k\geq 0$:
\begin{align}\label{Xvar}
X_{k,m}(c) := \min\Big( {\| \varphi_k \wt{c} \|}_{L^1}, 2^{-m -k^-} \big({\|  \partial_\xi(\varphi_k\wt{c}) \|}_{L^1} 
  + (2^{k^+} + 2^{-k^-}) {\| \varphi_{[k-5,k+5]} \wt{c} \|}_{L^1}\big) \Big). 
\end{align}
Note that this coincides with \eqref{wproofnot2.0} when $k \leq 0$.
Also note that this extended definition still satisfies the
upper bound $X_{k,m}(f(t) \tau_m(t)) \lesssim \e_1 2^{-3m/4} 2^{\alpha m}$, see \eqref{wproofnot2},
which we used many times before;
more precisely, using the a priori bounds \eqref{propbootfas} (see also \eqref{apriori11bis}-\eqref{apriori13})
we can estimate:
\begin{align}\label{Xvarest}
X_{k,m}(f(t)\tau_m(t))
  \lesssim \e_1 \min  \big(2^{-7k^+/2}, 2^{3k^-/2}, 2^{-m-k^-/2}, 2^{-m-k^+/2} \big) 2^{\alpha m}.
\end{align}

As a consequence of Lemma \ref{lemHF}, we see that
\begin{align}\label{HFest1}
\begin{split}
{\| \varphi_k (\xi) I^{p,k_1,k_2}_{\iota_1\iota_2}[f,f](s,\xi) \|}_{L^2}
 & \lesssim 2^{-k_1^+} \conv \cdot {\| \varphi_{k_1} \wt{f} \|}_{L^2} \cdot X_{k_2,m}(f) 
 \\ 
 & \lesssim 2^{-k_1^+} \conv \cdot  2^{-4k_1^+} {\| f \|}_{H^4} \cdot X_{k_2,m}(f).
\end{split}
\end{align}
Then, we see that if $\max(k_1,k_2) =k_1 \geq m/3$, by our a priori bounds we get
\begin{align}\label{HFest2}
\begin{split}
2^k {\| \varphi_k (\xi) I^{p,k_1,k_2}_{\iota_1\iota_2}[f,f](s,\xi) \|}_{L^2} 
& \lesssim 2^k 2^{-k_1^+} \conv \cdot \e_1 2^{\alpha m} 2^{-4m/3} \cdot \e_1 2^{-3m/4} 2^{\alpha m} 
 \\ & \lesssim \e_1^2 2^{-2m - 5\beta'm}.
\end{split}
\end{align}
This implies \eqref{othest++l}-\eqref{othestsigns} in the large frequencies regime $\max(k_1,k_2) \geq m/3$.

\begin{rem}[Handling the derivatives of $\mathfrak{q}$ for large frequencies]\label{remHF}
Thanks to the above argument, we are only left with proving the main bounds \eqref{othest++l}-\eqref{othestsigns}, 
for $\max(k,k_1) \leq m/3$ (and $\min(k,k_2) \geq -10m$, say).
In particular this allows us to disregard all the terms in our integration by parts arguments in frequency space
where derivatives fall on the symbol $\mathfrak{q}$,
despite the non-standard growth of its derivatives for frequencies larger than $1$;
see the factor of $R(\eta,\s)$ in \eqref{HFloss}.
Indeed, on the support of \eqref{othImain}, we have $R(\eta,\s) \approx 2^{k_2^+}$  (recall that we also assume $k_2\leq k_1$)
so that each derivative of $\mathfrak{q}$ can cost at most a factor of $2^{m/3}$,
while the gain from any integration by parts argument is always at least $2^{-m/2}$.
Therefore, a term where a derivative hits $\mathfrak{q}$ is always better behaved
than terms where derivatives hit the profiles (or other cutoffs). We will then analyze only these latter types of terms.
\end{rem}

\medskip
\subsubsection{Proof of \eqref{othest++l}}\label{ssecoth++l} 

Recall the relation between the parameters \eqref{wnormparam0},
and that by symmetry we assume $k_1\geq k_2$.
Also, recall that we are assuming that at least one of the two conditions $\ell > -7\beta'm$ or $k_1 > -10$ holds true.
As usual we divide the proof into a few cases.

\medskip 
\noindent
{\it Step 1: $k \leq -5$ or $k_1 \leq -4\beta'm-10$}.
Let us first discuss the case $k\leq -5$ where we have 
$|\xi| \ll 1$ and therefore $|\Phi_{++}| \geq 2 - \jxi \gtrsim 1$. 
In this case we can integrate by parts 
in $s$ without introducing any loss, and then analyze the resulting quartic terms 
(boundary terms are easy to handle) as done in Section \ref{secwR} starting on page \pageref{l<pulI}.
Here we should take some care since
the input frequencies $\eta$ and $\sigma$ could be close to $\pm\sqrt{3}$;
then, integration by parts in one of these frequencies could potentially introduce a loss,
since, for example, the $L^2$-norm of $\partial_\eta\wt{f}(\eta)$ degenerates for $\eta$ close to $\sqrt{3}$.
However, it suffices to observe that in this case
the quantity $X_{k_1,m}$, see \eqref{Xvar}, 
satisfies
for any $j_1 \in [-\gamma m, 0] \cap \Z$ (hence $|k_1|\leq 5$) 
\begin{align}\label{Xest}
\begin{split}
X_{k_1,m}\big(\wtF^{-1} (\chi_{j_1,\sqrt{3}} \wt{f}(t)) \tau_m(t) \big)
  & \approx 
  \min\big( {\| \varphi_{[k_1-5,k_1+5]} \chi_{j_1,\sqrt{3}} \wt{f} \|}_{L^1}, 
  2^{-m} {\|  \partial_\xi(\varphi_{k_1} \chi_{j_1,\sqrt{3}}\wt{f}) \|}_{L^1} \big)
  \\ 
 & \lesssim 2^{-m} {\|  \partial_\xi(\varphi_{k_1} \chi_{j_1,\sqrt{3}}\wt{f}) \|}_{L^1} 
  \lesssim \e_1 2^{-m} 2^{\beta'j_1 + \alpha m},
\end{split}
\end{align}
having used the consequence of the a priori bounds in \eqref{lembb3}.
The estimate \eqref{Xest} is substantially better than the general bound of 
$2^{-3m/4+\alpha m}$ used throughout Section \ref{secwR}, where we considered $k_1 < 0$.

Next, let us discuss the case $k > -5$ and $k_1 \leq -4\beta'm-10$.
In particular, since $k_1 \leq -10$, we are under the assumption that $\ell > -7\beta'm$.
Then, we see that 
$$|\Phi_{++}| \geq |2 - \jxi| - 2(\jeta-1) \approx 2^{\ell}$$ 
since $||\xi|-\sqrt{3}| \approx 2^\ell \gg 2^{2k_1}  \approx \jeta -1$. 
Therefore, we have the strong lower bound $|\Phi_{++}| \gtrsim 2^\ell \geq 2^{-7\beta'm}$ and
integration by parts in time can be used to handle this case as well. 

We can then assume
\begin{align}\label{oth++cond0}
k > -5, \qquad k_1 > -4\beta'm-10, \qquad k_2 \leq k_1, \qquad p < 2k_1-10,
\end{align}
where this last condition is just a consequence of restricting to the case $p < -8\beta'm - 30$
(the complementary case being again easier to deal with by integration by parts in time).

\medskip \noindent
{\it Step 2: Decomposition in $||\eta| -\sqrt{3}| \approx 2^{j_1}$}.
To proceed further we need to decompose the integral \eqref{othest++l} by inserting cutoffs in 
the size of $||\eta|-\sqrt{3}|\approx 2^{j_1}$. 
We first notice that if $j_1 \geq -10$, i.e. $\eta$ is away from $\pm\sqrt{3}$,
then we are in a situation which is similar to the one of Section \ref{secwR},
with the additional advantage that $|\eta|$ cannot be small
(and is in fact almost lower bounded by $1$, see \eqref{oth++cond0}).
An application of \eqref{wproofclaim2} in Lemma \ref{Lemma1}, together with the bound \eqref{sest} for the symbol,
 would then give us 
\begin{align}\label{oth++0.5}
\begin{split}
& 2^k {\big\| I^{p,k_1,k_2}_{++}\big[ \wtF^{-1}(\chi_{j_1,\sqrt{3}}\wt{f}),f \big](\cdot,s) \big\|}_{L^2}
  \\
  & \lesssim 2^k 2^{-k_1^+} \conv 
  \cdot 2^{p-k_1/2} \cdot 2^{-m-k_1} {\| \partial_\xi (\chi_{j_1,\sqrt{3}}\wt{f}) \|}_{L^2} \cdot X_{k_2,m}
  \\
  & \lesssim \e_1^2 2^{p-3k_1/2} \cdot 2^{-m+\alpha m} \cdot 2^{-3m/4+\alpha m}.
\end{split}
\end{align}
This is enough provided, for example, that $p\leq -m/4 - 10\beta' m$.
In the complementary case $p > -m/4 - 10\beta' m$ we can efficiently integrate by parts in time
and argue as in Subsections \ref{Ssecl<p} and \ref{Ssecl<2k_1}.
We may then reduce matters to the harder case $j_1 \leq -10$.

We define
\begin{align}\label{oth++1}
\begin{split}
I^{p,k,j_1,k_2}[a,b](t,\xi) := \varphi_k(\xi) \iint e^{it \Phi_{++}(\xi,\eta,\s)} 
  \, \varphi_p^{(p_0)}\big(\Phi_{\iota_1\iota_2}(\xi,\eta,\s)\big) 
  \, \mathfrak{q}^\prime(\xi,\eta,\s) \\ \times \chi_{j_1,\sqrt{3}}^{[-\gamma m,0]}(\eta) \, 
  \wt{a}_{\iota_1}(\eta) \, \varphi_{k_2}(\sigma) \wt{b}_{\iota_2}(\s) \, d\eta \, d\s,
\end{split}
\end{align}
(note that in the notation we have dispensed with the irrelevant parameter $k_1$ associated to the localization 
in $|\eta|\approx 2^{k_1}\approx 1$)
where
\begin{align}\label{s'}
& \mathfrak{q}^\prime(\xi,\eta,\s) := \mathfrak{q}^\prime_{\iota_1\iota_2;k,k_1,k_2}(\xi,\eta,\s) 
	:= \varphi_{\sim k}(\xi)\varphi_{\sim k_1}(\eta)\varphi_{\sim k_2}(\s)
	\mathfrak{q}_{\iota_1\iota_2}(\xi,\eta,\s) \cdot 2^k. 
\end{align} 
From \eqref{HFloss}, and since we are considering $|k_1| \leq 5$ and $k>-5$, we have
\begin{align}\label{s'est}
& |\partial_\eta^b \partial_\s^c \mathfrak{q}^\prime_{\iota_1\iota_2}(\xi,\eta,\s)| \lesssim 2^{(b+c)k_2^+} 2^{-20k^+}.
\end{align}
To obtain \eqref{othest++l}, it then suffices to prove the following: 


\begin{lem}\label{oth++2lem}
With the definitions \eqref{oth++1}-\eqref{s'est}, under our a priori assumptions,
we have, for all $j_1\leq -10$, 
\begin{align}\label{oth++2}
\begin{split}
2^m \left\| 
  \int_{0}^t I^{p,k,j_1,k_2}[f,f](\cdot,s) \, \tau_m(s)\, ds \right\|_{L^2_\xi} \lesssim \e_1^2 \, 2^{-3\beta'm}.
\end{split}
\end{align}
\end{lem}
Note that in \eqref{oth++2} we have discarded the factor of $2^{-\beta\ell}$ on the right-hand side,
which is of little help when $\ell$ is close to $0$.


Before proceeding with the proof of Lemma \ref{oth++2lem},
let us observe that the same argument used to prove \eqref{wproofclaim1} in Lemma \ref{Lemma1}, gives us
(we can use the $L^2_\xi$ norm instead of the $L^\infty_\xi$ by H\"older 
and \eqref{s'est})
\begin{align}\label{oth++3.0}
\begin{split}
{\| I^{p,k,j_1,k_2}[f,f](s) \|}_{L^2} 
  & \lesssim X_{k_1,m}\big(\wtF^{-1}( \chi_{j_1,\sqrt{3}}\wt{f}(s)) \tau_m(s) \big) 
  \cdot X_{k_2,m}(f)
  \\
  & \lesssim \e_1^2 \cdot 2^{-m+\alpha m} 2^{\beta' j_1} \cdot \e_1 2^{3k_2/2} 
  2^{\alpha m}.
\end{split}
\end{align}
In particular we have \eqref{oth++2} when $3k_2/2 \leq -m-2\alpha-3\beta'm$.
We may therefore assume from now on that
\begin{align}\label{oth++3.0cons1} 
k_2 \geq - 2m/3 - 10\beta'm.
\end{align}

For later use we record here also the following analogue of the estimate \eqref{wproofclaim2} applied to \eqref{oth++1}:
\begin{align}\label{oth++3}
\begin{split}
{\big\| I^{p,k,j_1,k_2}[f,f](s) \big\|}_{L^2} 
  & \lesssim 2^{p} \cdot 2^{-m} {\big\| \partial_\xi (\chi_{j_1,\sqrt{3}}^{[-\gamma m,0]} \wt{f}) \big\|}_{L^2} 
  \cdot X_{k_2,m}(f)
  \\
  & \lesssim 2^{p} \cdot 2^{-m} \e_1 2^{\gamma \beta m + \alpha m} \cdot 
  \e_1 \min(2^{3k_2/2},2^{-m + k_2/2})2^{\alpha m}.
\end{split}
\end{align}

\medskip \noindent
{\it Step 3: $p \leq -m/2 - 5\beta'm$}.
In this case, integration by parts in time is not efficient.
However, since $\gamma \beta + 2\alpha \leq 1/4$, see \eqref{wnormparam0},
we see that \eqref{oth++3} already suffices to give \eqref{oth++2}.

\medskip \noindent
{\it Step 4: $p \geq -m/2 - 5\beta'm$}.
This is the hardest case in the proof of \eqref{othest++l}.
The basic idea is to integrate by parts as in Subsection \ref{Ssecl<p},
and analyze the resulting terms, which are similar to those in \eqref{l<pibp}-\eqref{l<pibp3},
with the notation \eqref{l<pulI} and the identities \eqref{l<p5.0}-\eqref{l<p8}.
In the present case we have similar formulas, with a different localization in $\eta$
at the scale $||\eta|-\sqrt{3}|\approx 2^{j_1} \ll 1$, instead of $|\eta|\approx 2^{k_1} \ll 1$.

We define, similarly to \eqref{l<pulI},
\begin{align}
\label{oth++10}
\begin{split}
\underline{I}^{p,k,j_1,k_2}[g,h](s,\xi) := \iint e^{is \Phi_{++}} \, \frac{\varphi_p\big(\Phi\big)}{\Phi}
  \, \mathfrak{q}'(\xi,\eta,\s) \, \chi_{j_1,\sqrt{3}}^{[-\gamma m,0]}(\eta)
  \wt{g}(\eta) \, \varphi_{k_2}(\sigma) \wt{h}(\s) \, d\eta \, d\s,
\end{split} 
\end{align}
where $\mathfrak{q}'$ satisifes \eqref{s'est}.
In particular, since $k_2\leq 5$ we may think of this just as a smooth symbol which decays very fast in $\jxi$,
and with $O(1)$ bounds on its derivatives.

Disregarding the boundary terms that come from the integration by parts in time which
can be treated as before, we reduce matters to estimating
\begin{align}
\label{oth++10.1}
K^{S1,2}_{\iota_1\iota_2\iota_3}(t,\xi) & := \underline{I}^{p,k,j_1,k_2}
  \big[ \wtF^{-1} \mathcal{C}^{S1,2}_{\iota_1\iota_2\iota_3}[f,f,f],f\big](s,\xi),
\\
\label{oth++10.2}
L^{S1,2}_{\iota_1\iota_2\iota_3}(t,\xi) & := \underline{I}^{p,k,j_1,k_2}
  \big[f, \wtF^{-1} \mathcal{C}_{\iota_1\iota_2\iota_3}^{S1,2}[f,f,f] \big](s,\xi),
\end{align}
and
\begin{align}
\label{oth++10.3}
D^R(t,\xi) & := \underline{I}^{p,j_1,k_2}\big[ \wtF^{-1} \mathcal{R},f\big](s,\xi) 
  + \underline{I}^{p,j_1,k_2}\big[f, \wtF^{-1} \mathcal{R} \big](s,\xi).
\end{align}
For \eqref{oth++2} it suffices to prove an upper bound of $\e_1^2 2^{-2m - 3\beta'm}$ 
for the $L^2_\xi$-norms of \eqref{oth++10.1}-\eqref{oth++10.3}.

\medskip \noindent
{\it Estimate of \eqref{oth++10.1}-\eqref{oth++10.2}}.
As in Section \ref{secwR}, since we are assuming $k_2 \leq k_1$, the two terms \eqref{oth++10.1} and \eqref{oth++10.2}
are not symmetric; similarly to before, it turns out that the second one is slightly harder to treat,
so it will be the focus of our analysis.
We concentrate on the $\mathcal{C}^{S1}$ contribution since the one with $\mathcal{C}^{S2}$ 
will only differ slightly; see for example the arguments on page \pageref{l<p16}.

Expanding out as in \eqref{l<p21.5} 
and introducing frequency cutoffs for the new correlated variables,
we reduce to estimating quartic terms of the form
\begin{align}\label{l<pLot}
\begin{split}
& L_{\underline{k}} := \iiiint e^{it\Psi_{\iota_1\iota_2\iota_2}} \, \frac{\varphi_p\big(\Phi\big)}{\Phi}
  \, \mathfrak{q}\, \varphi_{\underline{k}}(\eta,\sigma,\rho,\zeta)
  \, \wt{f}(\eta) \wt{f}(\rho) \wt{f}(\zeta) \wt{f}(\s-\rho-\zeta) \, d\eta \, d\s \, d\zeta \, d\rho,
\\
& \Psi_{\iota_1\iota_2\iota_3} := \jxi - \jeta - \iota_1\jrho - \iota_2\langle\zeta\rangle 
  - \iota_2\langle\s-\rho-\zeta\rangle, \qquad \iota_1,\iota_2,\iota_3 \in \{+,-\},
\\
& \varphi_{\underline{k}}(\eta,\sigma,\rho,\zeta) := \chi_{j_1,\sqrt{3}}(\eta)\varphi_{k_2}(\sigma)
  \varphi_{k_3}(\rho)\varphi_{k_4}(\zeta)\varphi_{k_5}(\sigma-\rho-\zeta),
\end{split}
\end{align}
for a smooth symbol $\mathfrak{q}$.
It suffices to show that for $|t| \approx 2^m$
and $|\max(k_2,k_3) - \med(k_2,k_3,k_4)|\leq 5, k_5 \leq k_4 \leq k_3 \leq 0$, we have
\begin{align}\label{l<p20ot}
& \big| L_{\underline{k}}(t,\xi) \big| \lesssim \e_1^3 2^{-2m} 2^{-4\beta'm}.
\end{align} 

Integration by parts in the uncorrelated variable $\eta$, using that $|\eta| \approx 1$ and \eqref{Xest}, 
gives the following analogue of \eqref{l<p21}:
\begin{align}\label{l<p21ot}
\big| L_{\underline{k}}(t,\xi) \big| 
  \lesssim \e_1^4 \cdot 2^{-p} \cdot 2^{-m+4\alpha m} \cdot 2^{k_5+k_4 + \min(k_2,k_3)} 
  \cdot 2^{(1/2)(k_3+k_4+k_5)}.
\end{align}

We want to combine \eqref{l<p21ot} with exploiting the oscillations 
in the integral \eqref{l<pLot} in the directions of $\partial_\s$,
$\partial_\s + \partial_\rho$ and $\partial_\s + \partial_\zeta$ whenever this is convenient,
and proceed similarly to Case 2 on page \pageref{l<p14.0}.
Before doing this, we need to show how to deal with the cases when integration by parts in not possible
(see the analogous Case 1 of Step 3 on page \pageref{l<p21.5}).
We fix $\delta \in (0,\alpha)$.

\smallskip \noindent
{\it Case 1: $\min(k_2,k_4) + k_4 \leq -m + \delta m$}.
This is the case when integration by parts in $\partial_\s + \partial_\zeta$ is not possible,
because $|(\partial_\s + \partial_\zeta) \Psi| \approx |\zeta| \approx 2^{k_4}$
and hitting cutoffs will cost $2^{-k_2} + 2^{-k_4}$. 
From \eqref{l<p21ot} we have
\begin{equation*}
\big| L_{\underline{k}}(t,\xi) \big| \lesssim \e_1^4 2^{-p} \cdot 2^{-m + 4\alpha m} \cdot 2^{k_2+3k_4}.
\end{equation*}
Since we are assuming $k_2+k_4 \leq -m +\delta m$,
the above bound would suffice to obtain \eqref{l<p20ot} if it were the case that  $-p+2k_4 \leq-6\beta'm$.
On the other hand, if $-p+2k_4 \geq-6\beta'm$, we would have $k_4 \geq -m/4 -9\beta'm$
and therefore $k_2 \leq -3m/4 + 10\beta'm$ which is a contradiction to \eqref{oth++3.0cons1}.


\smallskip
\noindent
{\it Case 2: $\min(k_2,k_4) + k_4 \geq -m + \delta m$}.
In this case we also have $k_2 + k_3 \geq -m + \delta m$
and we can integrate by parts both in $\partial_\s + \partial_\rho$ and $\partial_\s + \partial_\zeta$.
We further distinguish between the case $\min(k_2,k_5) \geq -m/2$ and $\min(k_2,k_5) \leq -m/2$.
In the first case we can integrate by parts in $\sigma$ to obtain, up to faster decaying terms,
\begin{align*}
\big| L_{\underline{k}}(t,\xi) \big| \lesssim 2^{-p} \cdot \e_1 2^{-m + \alpha m} \cdot
  X_{k_3,m} \cdot X_{k_4,m} \cdot X_{k_5,m} 
\end{align*}
which is more than sufficient since $ X_{k,m} \lesssim \e_1 2^{-3m/4+\alpha m}$.
In the case $\min(k_2,k_5) \leq -m/2$ we estimate the profile 
$\varphi_{k_5}\wt{f}(\s-\rho-\zeta)$ in $L^\infty$ (this gives $\e_12^{k_5/2}2^{\alpha m}$)
and integrate $\varphi_{k_2}$ in $d\s$ (this gives a $2^{k_2}$ factor), thus obtaining, again up to faster decaying terms,
\begin{align*}
\big| L_{\underline{k}}(t,\xi) \big| \lesssim 2^{-p} \cdot \e_1 2^{-m + \alpha m} \cdot
  X_{k_3,m} \cdot X_{k_4,m} \cdot 2^{k_2} \cdot \e_1 2^{k_5/2} 2^{\alpha m};
\end{align*}
this suffices in view of the lower bound on $p$.

\medskip 
\noindent
{\it Estimate of \eqref{oth++10.3}}.
Finally we estimate the terms \eqref{oth++10.3}.
A bound analogous to \eqref{wproofclaim2add}  (here $|k_1|\leq 5$) directly gives us what we want:
\begin{align*}
{\big\|\underline{I}^{p,k,j_1,k_2} \big[ \wtF^{-1}\mathcal{R},f\big](s,\xi) \big\|}_{L^2} 
  & \lesssim 2^{-k_1/2} {\|\mathcal{R}(s)\|}_{L^2} \cdot X_{k_2,m} \lesssim \e_1^3 2^{-3m/2 + 2\alpha m} 
    \cdot \e_1 2^{-3m/4 + \alpha m},
\end{align*}
having used \eqref{lemdtfR}.
With a similar estimate, using also \eqref{Xest}, we can bound
\begin{align*}
{\big\|\underline{I}^{p,j_1,k_2}\big[f,\wtF^{-1}\mathcal{R} \big](s,\xi)\big\|}_{L^2} 
  \lesssim X_{k_1,m}\big(\wtF^{-1}( \chi_{j_1,\sqrt{3}}\wt{f}(s))\tau_m(s) \big)
  \cdot 2^{-k_2/2} {\|\mathcal{R}(s)\|}_{L^2} 
  \\
  \lesssim \e_1 2^{-m+\alpha m} \cdot 2^{-k_2/2} \e_1^3 2^{-3m/2 + 2\alpha m} 
\end{align*}
which, in view of \eqref{oth++3.0cons1}, suffices.


\medskip
\subsubsection{Proof of \eqref{othestsigns}}
To conclude the proof of Proposition \ref{prowR} 
we show how to treat the other sign combinations.
The main point here is that the phases 
satisfy
\[\Phi_{\iota_1\iota_2}(\xi,0,0):= \jxi - \iota_1-\iota_2\]
and, therefore, are not completely resonant since $(\iota_1\iota_2)\neq(++)$.
On the other hand, we still need to pay some attention to the case when one of the inputs is close to
the degenerate frequencies $\pm\sqrt{3}$.

\medskip \noindent
{\it Step 1: Preliminary reductions}.
First, notice that if $(\iota_1\iota_2)=(--)$ we have $|\Phi|\gtrsim 1$. 
This case is then easily handled integrating by parts in $s$. 
By symmetry we can reduce matters to the case $\iota_1=+=-\iota_2$
(but we do not assume a relation between $k_1$ and $k_2$),
and look at the integral in \eqref{othImain} with phase 
\begin{align}\label{othsigns1}
\Phi_{+-}(\xi,\eta,\s) = \jxi-\jeta+\jsig.
\end{align}
Notice that if 
$|\eta|,|\s|\leq 1$, then $|\Phi_{+-}|\gtrsim 1$,
and the bound \eqref{othestsigns} would be again easy to prove.
We may then assume $\max(k_1,k_2)\geq -5$ and, for a similar reason, $k_1\geq k_2$.


We decompose into the size of $||\eta|-\sqrt{3}| \approx 2^{j_1}$ by letting
\begin{align}\label{othsigns2}
\begin{split}
I^{p,k,k_1,k_2,j_1}(t,\xi) := \iint e^{it \Phi_{+-}(\xi,\eta,\s)} 
  \, \varphi_p^{(p_0)}\big(\Phi_{\iota_1\iota_2}(\xi,\eta,\s)\big) 
  \, \mathfrak{q}_{\iota_1\iota_2}(\xi,\eta,\s) \, \chi_{j_1,\sqrt{3}}^{[-\gamma m,0]}(\eta) \, 
  \\ \times \varphi_k(\xi) \varphi_{k_1}(\eta) \wt{a}_{\iota_1}(\eta) \, \varphi_{k_2}(\sigma) \wt{b}_{\iota_2}(\s) \, d\eta \, d\s,
\end{split}
\end{align}
(recall the definition \eqref{chil'})
and aim to prove
\begin{align}\label{othsigns3}
\begin{split}
2^m 2^{k}
  \left\| \chi_{\ell,\sqrt{3}}(\cdot)  \int_{0}^t I^{p,k,k_1,k_2,j_1}(\cdot,s) \, \tau_m(s)\, ds \right\|_{L^2_\xi} 
  \lesssim \e_1^2 \, 2^{-\beta\ell} 2^{-3\beta'm}.
\end{split}
\end{align}

First, we may only concentrate on the case $j_1 \leq -10$, for otherwise there is no 
degeneracy of $\partial_\eta \wt{f}$, Lemma \ref{Lemma1} applies verbatim, 
and the proof can proceed as in Section \ref{secwR}; see also the argument following \eqref{oth++0.5}.
We can then forget the localization in $|\eta|\approx 1$, eliminating the cutoff $\varphi_{k_1}$ 
from \eqref{othsigns2}, and rename it as $I^{p,k,j_1,k_2}$.

Moreover, when $j_1 \leq -10$ we may also assume that $\ell \geq -5$, for otherwise 
we would have $||\xi|-\sqrt{3}| \leq 2^{\ell+2}$ and therefore 
\begin{align*}
|\Phi_{+-}(\xi,\eta,\s)| = |\jxi-2 - (\jeta-2) + \jsig| \geq  \jsig - 2^{\ell+2} - 2^{j_1+2} \gtrsim 1.
\end{align*}
Similarly, we may assume $k\leq 5$, for otherwise $|\Phi_{-+}| \gtrsim 2^k$. 

We have then reduced \eqref{othsigns3} to showing
\begin{align}\label{othsigns3red}
\begin{split}
2^m 2^{k} \left\| 
  \int_{0}^t I^{p,k,j_1,k_2}(\cdot,s) \, \tau_m(s)\, ds \right\|_{L^2_\xi} 
  \lesssim \e_1^2 \, 2^{-3\beta'm}.
  \\ \mbox{with} \qquad j_1,p \leq -10, \quad k\leq 5.
\end{split}
\end{align}
For later reference we write
\begin{align}\label{othsignsphase}
\Phi_{+-}(\xi,\eta,\s) 
  = \frac{\xi^2}{1+\jxi} + \frac{\sigma^2}{1+\jsig} - \frac{\eta^2 - 3}{\jeta+2}.
\end{align}

\medskip \noindent
{\it Step 2: Preliminary bounds}.
Following a similar approach to that of Section \ref{secwR} we want to treat a few cases
by some basic bilinear estimates like those in Lemma \ref{Lemma1}.
In particular, under the parameter restrictions in \eqref{othsigns3red} we have the following 
two analogues of \eqref{wproofclaim1} and \eqref{wproofclaim2}:
first, by estimating $L^2_\xi \mapsto L^\infty_\xi$ gaining a factor of $2^{k/2}$,
and then integrating by parts in the two uncorrelated variables $\eta$ and $\s$
(as in the proof of \eqref{wproof6.0}), we have
\begin{align}\label{Lemma1var1} 
{\big\| I^{p,k,j_1,k_2}[f,f](s) \big\|}_{L^2} 
  \lesssim 2^{k/2} \cdot \e_1 2^{-m} 2^{\beta'j_1} 2^{\alpha m}
  \cdot \e_1 \min(2^{-m-k_2/2},2^{3k_2/2}) 2^{\alpha m}, 
\end{align}
having used \eqref{Xest}; 
second,
\begin{align}\label{Lemma1var2} 
{\big\| I^{p,k,j_1,k_2}[f,f](s) \big\|}_{L^2} 
  \lesssim 2^{p} \cdot \e_1 2^{-m} 2^{-\beta j_1} 2^{\alpha m} 
  \cdot \e_1 \min(2^{-m-k_2/2},2^{3k_2/2}) 2^{\alpha m},
\end{align}
arguing as in the proof of Lemma \ref{Lemma1} and using the a priori bound on 
${\|\chi_{j_1,\sqrt{3}}\partial_\eta\wt{f}\|}_{L^2}$.

\medskip \noindent
{\it Step 3: Case $k \leq -m/6 - 5\beta'm$}.
Applying \eqref{Lemma1var1} we see that
\begin{align*}
2^k {\big\| I^{p,k,j_1,k_2}[f,f](s) \big\|}_{L^2} 
  \lesssim \e_1^2 2^{3k/2} \cdot 2^{-m+\alpha m} \cdot 2^{-3m/4 +\alpha m}.
\end{align*}
This implies \eqref{othsigns3red} for $k$ in the range under consideration.

\medskip \noindent
{\it Step 4: Case $2k \leq j_1 + 10$}.
Using \eqref{Lemma1var2} and canceling the factor of $2^{-\beta j_1}$ by the factor of $2^k$ in front of the expression, 
we can bound
\begin{align*}
2^k {\big\| I^{p,k,j_1,k_2}[f,f](s) \big\|}_{L^2} 
  \lesssim \e_1^2 2^{p} \cdot 2^{-m+\alpha m} 
  \cdot 2^{-3m/4 +\alpha m}.
\end{align*}
When $p \leq -m/4-4\beta'm$ this already suffices.
For $p \geq -m/4-4\beta'm$ we instead integrate by parts; 
the loss is only about $2^{m/4}$, and is a much smaller loss than what we had in Subsection \ref{Ssecl<2k_1} for example,
so the analysis of the resulting quartic terms performed there suffices here too.

\medskip \noindent
{\it Step 5: Case $2k > j_1 + 10$}.
In this case $\xi^2 \gg |\eta^2-3|$ and from \eqref{othsignsphase} we see that
\begin{align}
|\Phi_{+-}(\xi,\eta,\s)| \gtrsim 2^{2k} 
\end{align}
In particular, integration by parts in time is very efficient, especially by noticing that
we have an extra factor of $2^k$ in front of the expression in \eqref{othsigns3red}.
The loss incurred by diving by $\Phi_{+-}$ is then bounded by $2^{-k} \lesssim 2^{m/6 + 5\beta'm}$,
and the same arguments used in Subsection \ref{Ssecl<2k_1} apply here.



\medskip
\subsection{Sobolev estimates}\label{SsecSob} 
Here we show how to bootstrap the Sobolev bound 
in Proposition \ref{propbootf} and obtain the bound on the first norm in \eqref{propbootfconc}.

\begin{prop}\label{propSob}
Under the bootstrap assumptions \eqref{propbootgas} and \eqref{propbootfas}, for all $t\in[0,T]$ we have
\begin{align}\label{propSobconc}
{\| \jxi^{4} \wt{f}(t) \|}_{L^{2}} \leq C\e_{0} + C \e_1^2\jt^{p_0}.
\end{align}
\end{prop}

\begin{proof}
From \eqref{Renodtf} and \eqref{Renof}, and \eqref{QRexp} we have
\begin{align}
\nonumber
\partial_t \wt{f} 
& = \mathcal{Q}^R(g,g) + \mathcal{C}^{S,1}(g,g,g) + \mathcal{C}^{S,2}(g,g,g)
\\
\label{propSob1}
& = \mathcal{Q}^R(f,f) + \mathcal{R}_H(f,g) 
+ \mathcal{C}^{S,1}(g,g,g) + \mathcal{C}^{S,2}(g,g,g)
\end{align}
It then suffices to show that that each term on the right-hand side of \eqref{propSob1}
is bounded in $L^2(\jxi^8 d\xi)$ by $C \e_1^2\jt^{p_0-1}$ so that \eqref{propSobconc} follows
from integration in time, using also 
the bound  \eqref{Bootinitcondf} at time $0$.

The cubic terms can be treated directly using the trilinear estimates of Lemma \ref{lemCS} 
and the a priori Sobolev and decay assumptions \eqref{propbootgas}; 
the term $\mathcal{R}_H(f,g)$ is already estimated as desired in \eqref{QRexprem}.

For the quadratic term in \eqref{propSob1} we need an additional non-trivial argument
that uses integration by parts in frequencies, the structure of the symbol, and Lemma \ref{lemQR}.
For convenience, let us rewrite here the expression for $\mathcal{Q}^R$, see \eqref{QR},
\begin{align}\label{propSobQR}
\mathcal{Q}_{\iota_1\iota_2}^R(a,b)(t,\xi)
  = \iint e^{it \Phi_{\iota_1\iota_2}(\xi,\eta,\sigma)} \, \mathfrak{q}(\xi,\eta,\sigma) 
  \, \wt{a}_{\iota_1}(t,\eta) \wt{b}_{\iota_2}(t,\sigma) \, d\eta \, d\sigma
\end{align}
and recall that the symbol $\mathfrak{q}$ is given as in \eqref{QR1}-\eqref{QR2} and \eqref{muR}-\eqref{muR'},
and that the bilinear estimates of Lemma \ref{lemQR} hold.
Without loss of generality let us assume that the support of \eqref{propSobQR} is restricted to $|\eta|\geq |\s|$.
Also, we may assume that $|\xi|\geq 10$.
We look at three different cases depending on the size of $\xi$ and $\eta$.

\medskip
\noindent
{\it Case 1: $|\xi| \geq 5|\eta|$.}
First, we treat the case of $|\xi| \leq \jt^{p_0/10}$. In this case we use the 
estimate \eqref{lemQRpq} to obtain
\begin{align*}
{\| \jxi^4 \mathcal{Q}_{\iota_1\iota_2}^R(f,f)(t) \|}_{L^2} 
  & \lesssim \jt^{p_0/2} {\| \mathcal{Q}_{\iota_1\iota_2}^R(f,f)(t) \|}_{L^2}
  \\ & \lesssim \jt^{p_0/2} {\| e^{i\iota_1t\jnab} \mathcal{W}^*f \|}_{L^{\infty-}}
  {\| e^{i\iota_2t\jnab} \mathcal{W}^*f \|}_{L^{\infty-}}
  \lesssim \e_1^2 \jt^{-1+p_0}.
\end{align*}
If instead $|\xi| \geq \jt^{p_0/10}$, we use the decay property of $\mu_R$ in \eqref{muR'}
and see that
\begin{align*}
{\big\| \jxi^4 \mathcal{Q}_{\iota_1\iota_2}^R(f,f)(t) \big\|}_{L^2}
&  \lesssim \sup_{\substack{|\xi| \geq 2(|\eta|+|\s|) \\ |\xi| > t^{p_0/10}}} \big| \jxi^6 \mu^R(\xi,\eta,\s) \big| 
   {\| \wt{f} \|}_{L^2_\eta} {\| \wt{f} \|}_{L^2_\s} \\
& \lesssim \jt^{-(N/2-6)p_0/10} \e_1^2
\end{align*}
which suffices since we can take $N$ arbitrarily large.

\medskip
\noindent
{\it Case 2: $|\xi| < 5|\eta|$ and $|\eta| \geq \jt^{1/3}$.}
In this case the first input of $\mathcal{Q}^R$ is projected to (distorted) frequencies greater than $\jt^{1/3}$,
and we denote it by $f_1$ where $\wt{f_1} := \varphi_{\geq0}(\eta\jt^{-1/3})\wt{f}$.

We begin by integrating by parts in the uncorrelated variable $\s$, and notice that the term where the symbol $\mathfrak{q}$ is differentiated is lower order. 
Then, using that 
$$
|\mathfrak{q}(\xi,\eta,\sigma)| \lesssim \frac{1}{\langle \eta \rangle} 
  (\inf_{\l,\mu} \langle \xi \pm \eta \pm \sigma \rangle )^{-N}
$$
together with Young's inequality gives
 \begin{align*}
{\| \jxi^4 \mathcal{Q}_{\iota_1\iota_2}^R\big(f,f)(t) \|}_{L^2} & \lesssim 
  {\|\jeta^{3} \wt{f_1} \|}_{L^2} \cdot \e_1 \jt^{-3/4+\alpha}
  \\
  & \lesssim \jt^{-1/3} {\| \jeta^4 \wt{f_1} \|}_{L^2} \cdot \e_1\jt^{-3/4 + \alpha}
  \lesssim \e_1^2 \jt^{-1}.
\end{align*}

\medskip
\noindent
{\it Case 3: $|\xi| < 5|\eta|$ and $|\eta| \leq \jt^{1/3}$.}
In this last case we integrate by parts in $\eta$ as well. 
We denote the first input of $\mathcal{Q}^R$ 
by $f_2 := \varphi_{\leq 0}(\eta\jt^{-1/3})\wt{f}$.
Integrating by parts in $\eta$ gains a factor of $|t|^{-1}$ and differentiates the profile $\wt{f_2}(\eta)$.
By the same argument as above,
\begin{align*}
{\| \jxi^4 \mathcal{Q}_{\iota_1\iota_2}^R\big[f,f\big](t) \|}_{L^2} & \lesssim 
  \jt^{-1} {\| \jeta^{3} \partial_\eta \wt{f_2} \|}_{L^2} \cdot \e_1 \jt^{-3/4+\alpha}
  \\
  & \lesssim \jt^{-1} \jt^{2/3}  {\| \jeta \partial_\eta \wt{f_2}\|}_{L^2} \cdot \e_1\jt^{-3/4 + \alpha}
  \\
  & \lesssim \jt^{-1/3} \cdot \e_1 \jt^\alpha \cdot \e_1 \jt^{-3/4+\alpha},
\end{align*}
which suffices, and concludes the proof of \eqref{propSobconc}.
\end{proof}


\medskip
\subsection{Pointwise estimates for the regular part and other higher order terms}\label{SsecLinfR}
In this subsection we first show that the regular part $\mathcal{Q}^R$ in \eqref{QR1}-\eqref{QR2} 
does not contribute to the pointwise asymptotic behavior of the solution, or, in other words,
that it is a remainder when measured in $\jxi^{-3/2}L^\infty_\xi$ norm.
Then we control the $\jxi^{-3/2}L^\infty_\xi$-norm of all 
the other terms that are not the singular cubic terms treated in Section \ref{secLinfS}; 
these include cubic terms that arise when passing from the original profile, $g$, to the renormalized profile, $f$,
and quartic and higher order terms.
Along the way we also establish bounds on the weighted norm of some cubic terms
that are not already accounted for in Section \ref{secBoot}.
In particular, these estimates will conclude the proof of the bound on the last norm in \eqref{propbootfconc} 
in the main bootstrap Proposition \ref{propbootf},
and give the bounds on the remainders in \eqref{LinfSasy}-\eqref{LinfSrem} in Proposition \ref{propLinfS}.

\smallskip
\subsubsection{Remainders from the quadratic regular part}
We begin by recalling 
that from Lemma \ref{lemQRexp} we have
\begin{align}\label{QRexp'}
\begin{split}
\mathcal{Q}^R(g,g) 
  & = \mathcal{Q}^R(f,f) + \mathcal{Q}^R(f,T(f,f)) + \mathcal{Q}^R(T(f,f),f) + \mathcal{R}_2(f,g)
\end{split}
\end{align}
where $\mathcal{R}_2(f,g)$ is the quartic term defined in \eqref{QRexp0} and satisfies
\begin{align}\label{QRexprem'}
\begin{split}
{\| \jxi \partial_\xi \mathcal{R}_2(f,g)(t) \|}_{L^2}& \lesssim \e_2^2 \jt^{-1+\alpha}.
\end{split}
\end{align}
We need to control in $\jxi^{-3/2}L^\infty_\xi$ all the terms on the right-hand side of \eqref{QRexp'}
(Proposition \ref{proQRLinftyxi} below)
and the weighted norms of the cubic terms $\mathcal{Q}^R(f,T(f,f)) + \mathcal{Q}^R(T(f,f),f)$
(Proposition \ref{proQRwot} below), since the weighted norm of $\mathcal{Q}^R(f,f)$ was taken care 
of in Section \ref{secwR} and Subsection \ref{secQRother}.

\begin{prop}[$L^\infty_\xi$ control for $\mathcal{Q}^R$ and remainders] 
\label{proQRLinftyxi}
Under the assumptions of Theorem \ref{maintheo},
consider $u$ solution of \eqref{KG} satisfying \eqref{apriori0}-\eqref{bootstrap}, and let 
$f$ be 
the renormalized profile defined in \eqref{Renof}.
We have
\begin{align}
\label{QRLinftyxi}
{\| \jxi^{3/2} \mathcal{Q}^R(f,f)(t) \|}_{L^\infty_\xi} 
	\lesssim \e_{1}^2 \jt^{-3/2+2\alpha}.
\end{align}
Moreover, for any $m=0,1,\dots$ 
we have
\begin{align}\label{QRLinftyxiR1}
{\Big\| \jxi^{3/2} \int_{0}^{t} \mathcal{Q}^R[f,T(f,f)](s,\xi) \, \tau_m(s) \, ds \Big\|}_{L^\infty_\xi}
  \lesssim \e_1^3 2^{-m/20}
\end{align}
Finally,
\begin{align}\label{QRLinftyxiR2}
{\big\| \jxi^{3/2} \mathcal{R}_2(f,g)(t) \|}_{L^\infty_\xi} \lesssim \e_2^2 \jt^{-1-1/20}.
\end{align}
\end{prop}

\begin{proof}
{\it Proof of \eqref{QRLinftyxi}.} 
This bound is essentially already contained in the proof of Lemma \ref{lemwproof6}
where, however, we only dealt with bounded frequencies.
Using the same argument (integration by parts in the uncorrelated variables $\eta$ and $\s$),
decomposing dyadically the input frequencies as usual,
and using the bound \eqref{HFloss} for the symbol $\mathfrak{q}$
we get 
for all $t \approx 2^m$
\begin{align*}
{\| \jxi^{3/2} \varphi_k(\xi) \mathcal{Q}^R(f,f)(t) \|}_{L^\infty_\xi} 
  \lesssim 2^{3k^+/2} \sum_{k_1,k_2} 2^{-k_1^+} \conv \cdot X_{k_1,m} \cdot X_{k_2,m},
\end{align*}
Using the estimate $X_{k,m} \lesssim \e_1 \min(2^{3k^-/2}, 2^{-m-k^-/2}, 2^{-m-k^+/2})2^{\alpha m}$, see \eqref{Xvarest},
we can perform the two sums over $k_1,k_2$ and obtain
\begin{align*}
{\| \jxi^{3/2} \varphi_k(\xi) \mathcal{Q}^R(f,f)(t) \|}_{L^\infty_\xi} \lesssim (\e_1 2^{-3m/4 + \alpha m})^2.
\end{align*}

\medskip \noindent
{\it Proof of \eqref{QRLinftyxiR1} and  \eqref{QRLinftyxiR2}} 
We first claim that a strong bound in $L^2$ holds for the cubic terms, that is, for all $t\approx 2^m$
\begin{align}\label{QRLinftyxiR1'}
{\big\| \jxi^2 \mathcal{Q}^R(f,T(f,f))(t) \|}_{L^2} + 
  {\big\| \jxi^2 \mathcal{Q}^R(T(f,f),f)(t) \|}_{L^2} \lesssim \e_1^3 2^{-6m/5}.
\end{align}
To see this it suffices to use that, for $p$ large enough,
\begin{align*}
{\big\| e^{-it\jnab} \jnab^{1+} \mathcal{W}^*f \big\|}_{L^p} \lesssim \e_1 \jt^{-1/4},
\end{align*}
which follows from interpolating the a priori decay assumptions and the $H^4$ bound,
and then apply \eqref{lemQRsob} with $p_1,p_2$ large enough and \eqref{bilboundT'rev}.

Then, using the inequality ${\| \jxi^{3/2} f \|}_{L^\infty} \lesssim \big( {\| \jxi \partial_\xi f \|}_{L^2} 
+ {\| f \|}_{L^2}\big)^{1/2} {\| \jxi^2 f \|}_{L^\infty}^{1/2}$
to interpolate between \eqref{QRLinftyxiR1'} and the weighted bound \eqref{QRwot} from Proposition \ref{proQRwot} below,
we obtain \eqref{QRLinftyxiR1}.

Using again \eqref{lemQRpq}, \eqref{bilboundT'rev}, 
with the Sobolev norm bound and the decay for the linear evolution of $g$, 
it is easy to see that the quartic term $\mathcal{R}_2(f,g)$ in \eqref{QRexp0} satisfies
\begin{align*}
{\| \jxi^{2} \mathcal{R}_2(f,g)(t) \|}_{L^2} & \lesssim \e_2^3 \jt^{-3/2}.
\end{align*}
Interpolating this and the weighted bound \eqref{QRexprem'} we obtain \eqref{QRLinftyxiR2}.
\end{proof}


\begin{prop}[Weighted estimates for other remainders] 
\label{proQRwot}
For any $m=0,1 \dots$, we have
\begin{align}\label{QRwot}
{\Big\| \jxi \partial_\xi \int_{0}^{t} \mathcal{Q}^R[T(f,f),f](s,\xi) 
  \, \tau_m(s) ds \Big\|}_{L^2_\xi} \lesssim \e_1^3 2^{\alpha m}. 
\end{align}
\end{prop}

\begin{proof}
These cubic terms are much easier to treat than the quadratic terms analyzed in Section \ref{secwR}.
For completeness we briefly discuss how to estimate them.

For simplicity we assume $|\xi|\leq 1$;
this can be done in view of the estimate \eqref{HFloss} (see also \eqref{s'est}) for the symbol $\mathfrak{q}$.
We look at the formulas for $\mathcal{Q}^R$ (see \eqref{othQR}) and $T$ (see \eqref{RenoT})
and write out the term explicitly as a trilinear operator;
after localizing dyadically in the frequencies, and making the usual reductions, 
this leads us to consider a term of the form
\begin{align}\label{prQRwot1}
\begin{split}
& L^R_{\underline{k}}(s,\xi) := 
  \iiint e^{is \Psi_{\iota_1\iota_2\iota_3}} \,
  \, \mathfrak{q}'(\xi,\eta,\s,\rho) \, \varphi_{\underline{k}}'(\xi,\eta,\sigma,\rho)
  \,  \wt{f}(\rho) \wt{f}(\rho-\eta) \wt{f}(\s)\, d\eta \, d\s \, d\rho,
\\
& \Psi_{\iota_1\iota_2\iota_3}(\xi,\eta,\s,\rho) = \jxi - \iota_1 \langle \rho \rangle - \iota_2 \langle \eta-\rho \rangle
  - \iota_3 \jsig,
\\
& \varphi_{\underline{k}}'(\xi,\eta,\sigma,\rho) 
  := \varphi_k(\xi) 
  \varphi_{k_2}(\sigma) \varphi_{k_3}(\rho)\varphi_{k_4}(\eta-\rho),
\qquad k_4\leq k_3 \leq 0. 
\end{split}
\end{align}
Here, we may assume that $\mathfrak{q}'$ is smooth, 
with uniform bounds on its derivatives,
except at the points $\rho=0$, $\rho-\eta=0$, $\s=0$, and $\eta=0$, where it can have $\sign$-type singularities.
Note that the boundedness property holds in view of \eqref{HFloss} and the estimates on the symbol of $T$
from Lemma \ref{lemTbound}, when we assume that all the frequencies involved are $\lesssim 1$;
when frequencies are large the estimate \eqref{HFloss} degenerates but, as discussed before,
this case is not harder to treat than the case of frequencies less than $1$, 
and can be analyzed using the bounds \eqref{Xvar}-\eqref{Xvarest} for the quantity $X_{k,m}$ when $k\geq0$.
Also recall that the lack of smoothness when one of the three input variables is zero is not an issue;
the singularity at $\eta$ is instead a potential issue which we will address below.

As usual we localize time $s \approx 2^m$.
Since applying $\jxi \partial_\xi$ will cost at most a factor of $s \xi \approx 2^m 2^k$, 
we can reduce matters to obtaining the estimate
\begin{align}\label{prQRwotest'2}
2^{k} {\Big\| \int_0^t L^R_{\underline{k}}(s,\xi) \, \tau_m(s) ds \Big\|}_{L^2} 
  \lesssim \e_1^3 2^{-m}. 
\end{align}
This is implied by the stronger bound
\begin{align}\label{prQRwotest}
2^{3k/2} \big| L^R_{\underline{k}}(s,\xi) \big| \lesssim \e_1^3 2^{-2m}, \qquad s\approx 2^m.
\end{align}

The arguments needed to show \eqref{prQRwotest} are similar to those used
in Subsection \ref{Ssecl<p} to estimate the term $K_{\underline{k}}$ in \eqref{l<p11}
(see also \eqref{l<p10}-\eqref{l<p10'}).
Note that $K_{\underline{k}}$ is actually a quartic term while $L^R_{\underline{k}}$ 
is only cubic, but, on the other hand, $L^R_{\underline{k}}$ has a (smooth) bounded symbol,
while the symbol of $K_{\underline{k}}$ has a large $1/\Phi$ factor,
where $\Phi$ is only assumed to be approximately lower bounded by $2^{-m/2}$. 

Examining \eqref{prQRwot1} we see that in fact all the three input frequencies are uncorrelated,
and we have the possibility of integrating by parts in each of them. 
However, we need to account for the singularity of the symbol in $\eta$.
For this, we introduce a decomposition in $|\eta|$ by inserting cutoffs $\varphi_{k_1}(\eta)$, $k_1\in\Z$.
The sum over $|k_1| \geq 10m$ is easily dealt with using the a priori bounds \eqref{propbootfas} on the $L^\infty_\xi$ and $H^4$-type norm.
It then suffices to estimate the contribution at each fixed $k_1$, with $|k_1|\leq 10m$, of the terms 
(we are changing variables $\eta \mapsto \rho-\eta'$)
\begin{align}\label{prQRwot1'}
\begin{split}
L^R_{\underline{k},k_1}(s,\xi) := 
  \iiint e^{is \Psi_{\iota_1\iota_2\iota_3}(\xi,\rho-\eta',\rho,\s)} \,
  \, \mathfrak{q}'(\xi,\rho-\eta',\s,\rho) \, \varphi_{\underline{k}}'(\xi,\rho-\eta',\sigma,\rho) \varphi_{k_1}(\rho-\eta') \,
  \\ \times  \wt{f}(\rho) \wt{f}(\eta') \wt{f}(\s)\, d\eta' \, d\s \, d\rho.
\end{split}
\end{align}
The sum over $k_1$ can be done at the expense of an $O(m)$ loss.

In \eqref{prQRwot1'} we integrate by parts in $\rho$ and/or $\eta'$ and/or $\s$ 
whenever any of these variables have size $\gtrsim 2^{-m/2}$,
and use the a priori bounds \eqref{apriori11} when instead they are $\lesssim 2^{-m/2}$;
this gives us the usual factor of $\e_1 2^{-3m/4 + \alpha m}$ for each of the three inputs.
This suffices provided we do not differentiate the symbol or, better, the cutoff $\varphi_{k_1}$,
when integrating by parts in $\rho$ or $\eta'$.

Let us then consider the case when $|\rho|\gtrsim 2^{-m/2}$ and we hit the symbol with $\partial_\rho$.
This gives the contribution
\begin{align*}
\begin{split}
\iiint e^{is \Psi_{\iota_1\iota_2\iota_3}(\xi,\rho-\eta',\rho,\s)} \, \frac{\langle\rho\rangle}{s\rho}
  \, \mathfrak{q}'(\xi,\rho-\eta',\s,\rho) \, \varphi_{\underline{k}}'(\xi,\rho-\eta',\sigma,\rho) \varphi_{\sim k_1}(\rho-\eta') 2^{-k_1}\,
  \\ \times  \wt{f}(\rho) \wt{f}(\eta') \wt{f}(\s)\, d\eta' \, d\s \, d\rho.
\end{split}
\end{align*}
Integrating by parts in $\s$ and $\eta'$ (again we assume their sizes are $\gtrsim 2^{m/2}$,
the complementary arguments being similar) 
leads to a main term of the form
\begin{align*}
\begin{split}
\iiint e^{is \Psi_{\iota_1\iota_2\iota_3}(\xi,\rho-\eta',\rho,\s)} \, \frac{\langle\rho\rangle \langle\eta'\rangle \jsig}{s^3\rho \,\eta' \s}
  \, \mathfrak{q}'(\xi,\rho-\eta',\s,\rho) \, \varphi_{\underline{k}}'(\xi,\rho-\eta',\sigma,\rho) \varphi_{\sim k_1}(\rho-\eta') 2^{-k_1}\,
  \\ \times  \wt{f}(\rho) \, \partial_{\eta'} \wt{f}(\eta') \, \partial_\s  \wt{f}(\s)\, d\eta' \, d\s \, d\rho.
\end{split}
\end{align*}
This is bounded by 
\begin{align*}
C 2^{-3m} \cdot 2^{-k_2-k_3-k_4} \cdot {\| \varphi_{k_3}\wt{f} \|}_{L^\infty}
	\cdot 2^{k_4/2}  {\| \varphi_{k_4} \partial_{\eta'} \wt{f} \|}_{L^2} \cdot 2^{k_2/2} {\| \varphi_{k_2} \partial_\s \wt{f} \|}_{L^2}
	\\
	\lesssim 2^{-3m} \cdot 2^{-(k_2+k_3+k_4)/2} \cdot 2^{3\alpha m}.
\end{align*}
Since $\min(k_2,k_3,k_4) \geq -m/2$ in our current scenario this gives us the desired \eqref{prQRwotest}.
Similar estimates hold true if $\partial_{\eta'}$ hits the symbol instead of the profile, 
or if $\min(k_2,k_4) \leq -m/2$.
\end{proof}

\smallskip
\subsubsection{Remainders from the cubic singular terms}
From Lemma \ref{lemCSexp} we know that
\begin{align}\label{CSexp'}
\begin{split}
& \mathcal{C}^{S}(g,g,g) - \mathcal{C}^{S}(f,f,f) \\ & = \mathcal{C}^{S}(T(f,f),f,f) + \mathcal{C}^{S}(f,T(f,f),f)
	+ \mathcal{C}^{S}(f,f,T(f,f)) + \mathcal{R}_3(f,g),
\end{split}
\end{align}
where $\mathcal{R}_3(f,g)$ is the quintic term defined in \eqref{CSexp0} and satisfies \eqref{CSexprem}:
\begin{align}\label{CSexprem'}
{\| \jxi \partial_\xi \mathcal{R}_3(f,g)(t) \|}_{L^2} & \lesssim \e_2^3 \jt^{-1+\alpha}.
\end{align}
We prove control of all the terms on the right-hand side of \eqref{CSexp'}.

\begin{prop}[Weighted estimates for remainder terms] 
\label{proCSLinftyxi}
Denote $\mathcal{C} = \mathcal{C}^{S}(T(f,f),f,f)$, or \\ $\mathcal{C}^{S}(f,T(f,f),f)$
or $\mathcal{C}^{S}(f,f,T(f,f))$.
Then we have
\begin{align}\label{otherweightedC}
\begin{split}
{\Big\| \jxi \partial_\xi \int_0^t \mathcal{C}(\xi,s) \, ds \Big\|}_{L^2_\xi} \lesssim \e_1^4.
\end{split}
\end{align}
Moreover, 
for $m=0,1,\dots$,
\begin{align}\label{otherweightedCinfty}
\begin{split}
{\Big\| \jxi^{3/2} \int_0^t \mathcal{C}(\cdot,s) \, \tau_m(s) ds \Big\|}_{L^\infty_\xi} \lesssim \e_1^4 2^{-m/10}.
\end{split}
\end{align}
Finally,
\begin{align}\label{otherLinftyxiC2}
{\| \jxi^{3/2} \mathcal{R}_3(f,g)(t) \|}_{L^\infty_\xi} & \lesssim \e_2^3 \jt^{-1-\alpha}.
\end{align}
\end{prop}

Note that, as in Propositions \ref{proQRLinftyxi} and \eqref{proQRwot}, 
we need to use the time integral in \eqref{otherweightedC} and \eqref{otherweightedCinfty} too. 

\begin{proof}

{\it Proof of \eqref{otherweightedC}.} 
Let us consider the term $\mathcal{C} = \mathcal{C}^{S1}[T(f,f),f,f]$
and restrict our attention to the portion of $T$ corresponding to the $\delta$ contribution of its symbol $Z$;
see the formulas \eqref{CubicS} and \eqref{formulacubiccoeff}, \eqref{Tgg} and \eqref{QZ}.
The slight modifications that are needed to deal with the other terms pf $\pv$-type will be clear to the reader;
compare also with the arguments in \S\ref{threeremoved} (where we deal with a $\pv$ contribution) 
and the algebra following \eqref{Cubotherid}.
Writing out explicitly the quartic term under consideration,
and disregarding the irrelevant signs $\eps,\eps',\l,\mu,\nu\dots$ in the symbols \eqref{formulacubiccoeff}
and \eqref{QZ}, we obtain an expression of the form
\begin{align}\label{oWC1}
\begin{split}
& \mathcal{C}(t,\xi) = \iiint e^{it \Phi_{\iota_1\iota_2\iota_3\iota_4}} 
  \mathfrak{c}
\, \wt{f}_{\iota_1}(\eta) \wt{f}_{\iota_2}(\s)  \wt{f}_{\iota_3}(\rho) \wt{f}_{\iota_4}(\xi-\eta-\s-\rho)  
  \, d\eta \, d\s \,d\rho,
\\
& \Phi_{\iota_1\iota_2\iota_3\iota_4}(\xi,\eta,\sigma,\theta) 
  = - \jxi +\iota_1\jeta +\iota_2\jsig + \iota_3 \jrho + \iota_4 \langle\xi-\eta-\sigma-\rho\rangle.
\\
\end{split} 
\end{align}
Here, as usual, we can think of $\mathfrak{c}$ as a smooth symbol,
so that its associated $4$-linear operator satisfies standard H\"older estimates.

To handle this term, the main idea is to use the following ``commutation identity'' for $\jxi \partial_\xi$ 
and $\Phi := \Phi_{\iota_1\iota_2\iota_3\iota_4}$:
let $X_a := \langle a \rangle \partial_a$, then
\begin{align}\label{owCid}
(X_\xi + \iota_1 X_\eta + \iota_2 X_\sigma + \iota_3X_\rho) \Phi 
  = -\iota_4 \frac{\xi-\eta-\sigma-\rho}{\langle \xi-\eta-\sigma-\rho\rangle} \Phi.
\end{align}
Thanks to this we can write $\jxi \partial_\xi \mathcal{C}$ as a linear combination of
terms of the following two types, up to similar or easier ones:
\begin{align}
\label{oWCa}
& \mathcal{C}_a = \iiint e^{it \Phi} 
  \mathfrak{c}
\,\big( \jeta \partial_\eta \wt{f}_{\iota_1}(\eta) \big) \, \wt{f}_{\iota_2}(\s)
  \wt{f}_{\iota_3}(\rho) \wt{f}_{\iota_4}(\xi-\eta-\s-\rho)  
  \, d\eta \, d\s \,d\rho,
\\
\label{oWCb}
& \mathcal{C}_b = \iiint e^{it \Phi} 
  \, \big( i t\, \Phi \big)
  \mathfrak{c}
\, \wt{f}_{\iota_1}(\eta) \wt{f}_{\iota_2}(\s)  \wt{f}_{\iota_3}(\rho) \wt{f}_{\iota_4}(\xi-\eta-\s-\rho)  
  \, d\eta \, d\s \,d\rho.
\end{align}
Note that the terms where the derivatives $X_a$ hit the symbol 
can be treated easily by an $L^2\times L^\infty\times L^\infty\times L^\infty$-type estimate
using the a priori $H^4$ bound and the linear decay estimate.

$\mathcal{C}_a$ is directly estimated using a $4$-linear H\"older estimate,
the a priori bound \eqref{lembb1}, and the usual linear decay estimate:
\begin{align*}
{\| \mathcal{C}_a \|}_{L^2} \lesssim 
  {\|\jxi \partial_\xi \wt{f}_{\iota_1} \|}_{L^2} \big( \jt^{-1/2} {\| u \|}_{X_T} \big)^3
  \lesssim \e_1^4 \jt^{-5/4}
\end{align*}

The contribution from \eqref{oWCb} is estimated integrating by parts in time:
\begin{align*}
\int_0^t \mathcal{C}_b \,ds = \mathcal{C}_{b_1}(t) + \int_0^t \mathcal{C}_{b_2}(s)\,ds
\end{align*}
where $C_{b_1}$ is like $\mathcal{C}_b$ without the factor of $\Phi$,
and $C_{b_2}$ is like $\mathcal{C}_{b_1}$ with one profile $\wt{f}$ replaced by $\partial_t\wt{f}$.
In particular, we can see that
\begin{align*}
{\| \mathcal{C}_{b_1}(t) \|}_{L^2} \lesssim t {\| \wt{f}  \|}_{L^2} \big( \jt^{-1/2} {\| u \|}_{X_T} \big)^3
  \lesssim \e_1^4 \jt^{-1/2}
\end{align*}
and
\begin{align*}
{\| \mathcal{C}_{b_2}(s) \|}_{L^2} \lesssim {\big\| \partial_s \big(s \wt{f} \, \big)  \big\|}_{L^2} \big( \js^{-1/2} {\| u \|}_{X_T} \big)^3
  \lesssim \e_1^4 \js^{-5/4},
\end{align*}
having used \eqref{dtfL2}. These give us \eqref{otherweightedC}.

\medskip
{\it Proof of \eqref{otherweightedCinfty}.} 
This follows by interpolating the $\jxi^{-3/2}L^\infty_\xi$ norm between the $\jxi^{-2}L^2$ and $\jxi^{-1}\dot{H}^1$,
and using that the $\jxi^{-2}L^2$ norm of the quantity we are estimating is bounded at least by $\e_1^4 2^{-m/4}$.

\medskip
{\it Proof of \eqref{otherLinftyxiC2}.} 
From its definition we see that $\mathcal{R}_3(f,g)$ is a quintic term in $(f,g)$;
see \eqref{CSexp0} and \eqref{CSexp1}.
Then, using the multilinear estimates from Lemmas \ref{lemCS} and \ref{lemT}, 
and the decay for the linear evolution of $f$ and $g$, we can see that
\begin{align}\label{CSexprem'L2}
{\| \jxi^2 \mathcal{R}_3(f,g)(t) \|}_{L^2} & \lesssim \e_2^3 \jt^{-2+\alpha}.
\end{align}
Interpolating this and \eqref{CSexprem'} we obtain the pointwise bound \eqref{otherLinftyxiC2}.
\end{proof}


\medskip
\subsection{Other singular cubic interactions}\label{Cubother}

In this subsection, we complete the analysis of the singular cubic terms 
$\mathcal{C}^{S1}_{\iota_1 \iota_2 \iota_3}$ and $\mathcal{C}^{S2}_{\iota_1 \iota_2 \iota_3}$ 
defined in \eqref{CubicS}-\eqref{CubicS12}.  
Section \ref{secwL} was dedicated to the analysis of these terms when $(\iota_1,\iota_2,\iota_3) = (+,-,+)$, 
in the fully resonant situation
when all input frequencies are $\sqrt{3}$; this also covers the case when they are all $-\sqrt{3}$. 
We now treat all the other interactions, which are, as was to be expected, relatively easier to deal with.

We will focus only on the $\mathcal{C}^{S2}_{\iota_1 \iota_2 \iota_3}$ 
terms for the sake of brevity, 
but the terms $\mathcal{C}^{S1}_{\iota_1 \iota_2 \iota_3}$ are amenable to a similar treatment.
We make a convenient choice of the parameters $\lambda, \mu, \dots$ 
(which do not matter as far as estimates are concerned), 
and drop all irrelevant indexes as well as complex conjugation signs to obtain the following formula 
for $\mathcal{C}^{S2}_{\iota_1 \iota_2 \iota_3}$:
\begin{align}\label{Cubic2ot0}
\begin{split}
& \mathcal{C}^{S2}_{\iota_1 \iota_2 \iota_3}[a,b,c](t,\xi)
= \iiint e^{it \Phi_{\iota_1\iota_2 \iota_3}(\xi,\eta,\s,\theta)} \mathfrak{c}^{S,2}(\xi,\eta,\sigma,\theta) 
\,  \wt{a}(t,\eta) \wt{b}(t,\s)  \wt{c}(t,\theta)   \frac{\widehat{\phi}(p)}{p} \, d\eta \, d\s \,d\theta,
\\
& \Phi_{\iota_1 \iota_2 \iota_3}(\xi,\eta,\s,\theta) := 
  \jxi - \iota_1 \jeta - \iota_2  \jsig - \iota_3 \langle \theta \rangle, \qquad p = \xi - \eta - \sigma -\theta.
\end{split}
\end{align}
Recall that $\mathfrak{c}^{S,2}(\xi,\eta,\sigma,\theta)$ satisfies the bound
$$
\left| \mathfrak{c}^{S,2}(\xi,\eta,\sigma,\theta) \right| \lesssim \frac{1}{\langle \eta \rangle \langle \eta' \rangle \langle \sigma' \rangle},
$$
and the trilinear operator with this symbol enjoys the boundedness properties stated in Lemma \ref{lemCS}.

We will distinguish 
different cases, depending on whether $\eta$, $\sigma$ and $\theta$ are close to, or removed from, $\pm \sqrt{3}$. 
We define cut-off functions
\begin{align}\label{Cubothercuts}
\chi_{c}(\xi) = \chi \left( \frac{\xi - \sqrt 3}{r} \right) +  \chi \left( \frac{\xi + \sqrt 3}{r} \right), 
\qquad \chi_{r}(\xi) = 1 - \chi \left( \frac{\xi - \sqrt 3}{4r} \right)  -  \chi \left( \frac{\xi + \sqrt 3}{4r} \right)
\end{align}
where $r$ is a sufficiently small positive number, and $\chi=\varphi_{\leq 0}$, see the notation in \S\ref{secNotation}.
Notice that $\chi_c$ and $\chi_r$ do not add up to one, 
since it will be convenient in the estimates to have a separation between their supports. 
Since they can be treated with straightforward adaptations, 
we skip the estimates corresponding to $1 - \chi_c - \chi_r$ for the sake of brevity.
According to \eqref{Cubothercuts} we define the frequency projections
\begin{align}\label{Pcrf}
P_\ast f := \wtF^{-1} \big( \chi_\ast \, \wt{f} \, \big), \qquad \ast \in \{c,r\}. 
\end{align}

We will prove the following main proposition:

\begin{prop}[Weighted estimates for the singular cubic interactions]\label{propCubother}
Let $\mathcal{C}^{S} \in \{\mathcal{C}^{S1},\mathcal{C}^{S2}\}$ as defined in \eqref{CubicS}.
With the a priori assumptions \eqref{propbootfas}, we have, for $t\in[0,T]$,
\begin{align}\label{propCubother0}
\Big\| \jxi\partial_\xi \int_0^t \mathcal{C}^{S}_{\iota_1 \iota_2 \iota_3}[a,b,c](s,\xi) \, ds \Big\|_{L^2} \lesssim \e_1^3 \jt^\alpha,
\end{align}
when
\begin{align*}
& \{\iota_1, \iota_2, \iota_3\} =\{+,+,-\} \quad 
  \mbox{and} \quad \{a,b,c\} \in \{P_{p_1} f, P_{p_2}f, P_r f\}, \quad p_1,p_2 \in\{c,r\}
\\
& \mbox{or when} \quad \{\iota_1, \iota_2, \iota_3\} \neq \{+,+,-\}.
\end{align*}
Moreover, if $\{\iota_1, \iota_2, \iota_3\} =\{+,+,-\}$ and 
$a,b,c = P_c f$, but their frequencies are not all equal to either $\sqrt{3}$ or -$\sqrt{3}$,
then
\begin{align}\label{propCubother0'}
\Big\| \int_0^t \mathcal{C}^{S}_{\iota_1 \iota_2 \iota_3}[a,b,c](s,\xi) \, ds \Big\|_{W_T} \lesssim \e_1^3.
\end{align}



\end{prop}

The proof of Proposition \ref{propCubother} will complete the proof of the weighted bound 
in \eqref{propbootfconc}. 

For a better organization of our exposition, we are going to prove \eqref{propCubother0}-\eqref{propCubother0'} 
by distinguishing cases relative to whether the frequencies are close or not to $\pm\sqrt{3}$, 
and subcases depending on the $\iota$'s signs combinations.

\medskip
\subsubsection{Three frequencies removed from $\pm \sqrt 3$}\label{threeremoved}
This case is similar to the cubic nonlinear Schr\"odinger equation, where the dispersion relation 
is $\Lambda = \xi^2$, see \cite{GPR2,ChPu}. In \cite{GPR2} weighted estimates are proved under 
the assumption that the potential $V$ is generic;
here we provide a more general (and simpler) argument similar to the one in \cite{ChPu} 
that also applies to the case of exceptional potentials 
and any solution $u$ such that $\wt{u}(0)=0$.

As before, we simplify our notation by dropping some of the irrelevant indexes in our formulas.
We look at the restriction of \eqref{Cubic2ot0} to inputs with frequencies away from $\pm\sqrt{3}$ 
by defining
\begin{align}\label{Cubic2ot1}
\begin{split}
\mathcal{C}_{rrr}(a,b,c)(t,\xi)
&:= \iiint e^{it \Phi_{\iota_1\iota_2 \iota_3}(\xi,\eta,\sigma,\theta)} 
  \mathfrak{c}_{rrr}(\xi,\eta,\sigma,\theta) \wt{a}(t,\eta) \wt{b}(t,\sigma)  
  \wt{c}(t,\theta) \frac{\widehat{\phi}(p)}{p} \, d\eta \, d\sigma \,d\theta,
\\
\mathfrak{c}_{rrr} (\xi,\eta,\sigma,\theta) 
  & := \mathfrak{c}^{S,2} (\xi,\eta,\sigma,\theta) \chi_r(\eta) \chi_r(\sigma) \chi_r(\theta),
\end{split}
\end{align}
and aim to show
\begin{align}\label{Cubic2otest}
{\Big\| \jxi \partial_\xi \int_0^t \mathcal{C}_{rrr}(f, f, f)(s) \, ds \Big\|}_{L^2}
  \lesssim \e_1^3 \jt^{\alpha}.
\end{align}

Observe that
\begin{align}\label{Cubotherid}
(\jxi \partial_\xi + X_{\eta,\s,\theta}) \Phi_{\iota_1, \iota_2 ,\iota_3} = p,
 \qquad X_{\eta,\s,\theta} := \iota_1\jeta \partial_\eta + \iota_2\jsig \partial_\sigma
  + \iota_3\langle \theta \rangle \partial_\theta.
\end{align}
Then, when applying $\jxi \partial_\xi$ to \eqref{Cubic2ot1}, we can use the above identity 
to integrate by parts in $\eta,\sigma$ and $\theta$.
Since the adjoint satisfies $X_{\eta,\s,\theta}^\ast = - X_{\eta,\s,\theta}$ 
we see that
\begin{subequations}\label{Cubic2ot5}
\begin{align}\label{Cubic2ot5.1}
& \jxi \partial_\xi\mathcal{C}_{rrr}[f,f,f](t,\xi)
= i t \iiint e^{it \Phi_{\iota_1\iota_2 \iota_3}} \mathfrak{c}_{rrr}(\xi,\eta,\sigma,\theta) 
\, \wt{f}(\eta) \wt{f}(\sigma)  \wt{f}(\theta) 
  \widehat{\phi}(p) \, d\eta \, d\sigma \,d\theta
\\
\label{Cubic2ot5.2}
& + \iiint e^{it \Phi_{\iota_1\iota_2 \iota_3}} \mathfrak{c}_{rrr}(\xi,\eta,\sigma,\theta) 
  X_{\eta,\s,\theta} \big(\wt{f}(\eta) \wt{f}(\sigma)  \wt{f}(\theta) \big)
  \frac{\widehat{\phi}(p)}{p} \, d\eta \, d\sigma \,d\theta
\\
\label{Cubic2ot5.3}
& + \iiint e^{it \Phi_{\iota_1\iota_2 \iota_3}} \mathfrak{c}_{rrr}(\xi,\eta,\sigma,\theta) 
\, \wt{f}(\eta) \wt{f}(\sigma)  \wt{f}(\theta) 
  \, \big( \jxi \partial_\xi + X_{\eta,\s,\theta} \big) \Big[ \frac{\widehat{\phi}(p)}{p}\Big] \, d\eta \, d\sigma \,d\theta \\
  \label{Cubic2ot5.4}
  & +  \iiint e^{it \Phi_{\iota_1\iota_2 \iota_3}} X_{\eta,\s,\theta} \, \mathfrak{c}_{rrr}(\xi,\eta,\sigma,\theta) 
\, \wt{f}(\eta) \wt{f}(\sigma)  \wt{f}(\theta) 
  \, \frac{\widehat{\phi}(p)}{p} \, d\eta \, d\sigma \,d\theta 
 .
\end{align}
\end{subequations}

\smallskip
\noindent
{\it Estimate of \eqref{Cubic2ot5.1}}.
The first term in \eqref{Cubic2ot5} does not have a singularity and can be estimated
integrating by parts in the ``uncorrelated'' variables $\eta,\sigma$ and $\theta$.
Each of the three inputs then would gives a gain of $\jt^{-3/4+\alpha}$ which is sufficient to
absorb the power of $t$ in front and integrate over time.
Similar (in fact, harder) terms have been treated in Section \ref{secwR}, so we can skip the details.

\smallskip
\noindent
{\it Estimate of \eqref{Cubic2ot5.2}}.
For this term it suffices to use the H\"older-type estimate from Lemma \ref{lemCS},
estimating in $L^2$ the profile that is hit by the derivative, and the other two in $L^\infty_x$.

\smallskip
\noindent
{\it Estimate of \eqref{Cubic2ot5.3}}.
For this term we observe, see \eqref{Cubotherid}, that
\begin{align}\label{Cubic2ot10}
\big( \jxi \partial_\xi + X_{\eta,\s,\theta} \big) \Big[ \frac{\widehat{\phi}(p)}{p}\Big] 
  = \Phi_{\iota_1,\iota_2,\iota_3}(\xi,\eta,\s,\theta) \partial_p \Big[ \frac{\widehat{\phi}(p)}{p}\Big].
\end{align}
Note that this identity is formal as it is written, since $\partial_p (1/p)$ does not converge (even in the $\pv$ sense);
however, it can be made rigorous by localizing a little away from $p=0$, 
and using the $\pv$ to deal with very small values of $p$.


From \eqref{Cubic2ot10} we obtain, upon integration by parts in $s$, that
\begin{align}\label{Cubic2ot11}
\int_0^t i \eqref{Cubic2ot5.3} \, ds & = 
\iiint e^{is \Phi_{\iota_1\iota_2 \iota_3}} \mathfrak{c}_{rrr}(\xi,\eta,\sigma,\theta)
\, \wt{f}(\eta) \wt{f}(\sigma)  \wt{f}(\theta) 
  \, \partial_p \frac{\widehat{\phi}(p)}{p}\, d\eta \, d\sigma \,d\theta \, \Big|_{s=0}^{s=t}
\\
\label{Cubic2ot12}
& - \int_0^t \iiint e^{it \Phi_{\iota_1\iota_2 \iota_3}} \mathfrak{c}_{rrr}(\xi,\eta,\sigma,\theta)
\, \partial_s \Big[ \wt{f}(\eta) \wt{f}(\sigma)  \wt{f}(\theta) \Big]
  \partial_p \frac{\widehat{\phi}(p)}{p} \, d\eta \, d\sigma \,d\theta \, ds.
\end{align}
To estimate \eqref{Cubic2ot11} we convert the $\partial_p$ into $\partial_\eta$ and
integrate by parts in $\eta$. The worst term is when $\partial_\eta$ hits the exponential;
this causes a loss of $t$ but an $L^2\times L^\infty\times L^\infty$ H\"older estimate using Lemma \ref{lemCS}
suffices to recover it.

The term \eqref{Cubic2ot12} is similar. We may assume that $\partial_s$ hits $\wt{f}(\sigma)$.
Again we convert $\partial_p$ into $\partial_\eta$ and
integrate by parts in $\eta$.
This causes a loss of $s$ when hitting the exponential phase
which is offset by an $L^\infty\times L^2\times L^\infty$ estimate
with $\partial_s \wt{f}$ placed in $L^2$ and giving $\jt^{-1}$ decay using \eqref{dtfL2}.

\smallskip
\noindent
{\it Estimate of \eqref{Cubic2ot5.4}} 
This term can be estimated directly using the trilinear 
estimates from Lemma \ref{lemCS}. 
The only difficulty is the loss of one derivatives resulting from the differentiation of the symbol,
but this is easily recovered using the $H^4$ a priori bound from \eqref{propbootfas},
and $p_0 < \alpha$, see \eqref{apriori0}.

\medskip
\subsubsection{One frequency close, two removed from $\pm \sqrt 3$}\label{tworemoved}
Let us now consider
\begin{align*}
\mathcal{C}_{crr}[a,b,c](t,\xi)
& := \iiint e^{it \Phi_{\iota_1\iota_2 \iota_3}(\xi,\eta,\sigma,\theta)} \mathfrak{c}_{crr}(\xi,\eta,\sigma,\theta) 
  \wt{a}(t,\eta) \wt{b}(t,\sigma)  \wt{c}(t,\theta) \frac{\widehat{\phi}(p)}{p} \, d\eta \, d\sigma \,d\theta,
  \\
\mathfrak{c}_{crr} (\xi,\eta,\sigma,\theta) & := \mathfrak{c}^{S,2} (\xi,\eta,\sigma,\theta) 
\chi_c(\eta) \chi_r(\sigma) \chi_r(\theta).
\end{align*}
One can proceed exactly as in the previous subsection, with the exception of the treatment of
\eqref{Cubic2ot5.2}, which must be modified due to the degeneracy of the weighted norm close to $\sqrt{3}$. 
The only problematic term is the one where the first function 
(whose frequency is close to $\pm \sqrt{3}$) is differentiated. Thus, we are looking at
$$ 
\iiint e^{it \Phi_{\iota_1\iota_2 \iota_3}} \mathfrak{c}_{crr}(\xi,\eta,\sigma,\theta) 
  \langle \eta \rangle \partial_\eta \wt{f}(\eta) \wt{f}(\sigma)  \wt{f}(\theta) \frac{\widehat{\phi}(p)}{p} 
  \, d\eta \, d\sigma \,d\theta.
$$
We need to distinguish cases depending on the $\iota$ signs.

\medskip 
\noindent
{\it The $+++$ case}. Since only the first argument $\wt{f}(\eta)$ is differentiated, 
it is natural to try to integrate by parts in $\partial_\s-\partial_\theta$. 
We thus need to look at frequencies for which
$$
(\partial_\sigma - \partial_\theta) \Phi_{+++}(\xi,\eta,\sigma,\theta) = \Phi_{+++}(\xi,\eta,\sigma,\theta) = 0.
$$
We will refer to these as `restricted (space-time) resonances'. 
The vanishing of  $(\partial_\sigma - \partial_\theta) \Phi_{+++}$ imposes that $\sigma = \theta$. 
Therefore, resonances are given by the zeros of
$$
\Phi_{+++}(\xi,\eta,\sigma,\sigma) = \langle \eta + 2\sigma + p \rangle - \langle \eta \rangle 
- 2 \langle \sigma \rangle.
$$
Squaring both sides of $ \langle \eta + 2\sigma + p \rangle = \langle \eta \rangle + 2 \langle \sigma \rangle$ 
results in $p^2 + 2p(\eta + 2\sigma) + 4\eta \sigma = 4 + 4 \langle \eta \rangle \langle \sigma \rangle$,
which has no solutions if $|p| \ll 1$ and $(\xi,\eta,\sigma,\theta) \in \operatorname{Supp} (\mathfrak{c}_{2,crr})$.
Note that we may easily restrict to $|p| \ll 1$, 
since interactions for which $|p| \gtrsim 1$ can be treated like regular cubic terms
integrating by parts in the uncorrelated variables $\s$ and $\theta$.

Without loss of generality, we assume that $|\sigma| \geq |\theta|$; 
we then distinguish between the case where $|\sigma| \lesssim 1$, and $|\sigma| \gg 1$.

\begin{itemize}

\item If $|\sigma| \lesssim 1$, 
we resort either to integration by parts in $\partial_\sigma - \partial_\theta$
or to integration by parts in $s$, 
using that either $(\partial_\sigma - \partial_\theta)\Phi_{+++}$ or $\Phi_{+++}$ can 
be bounded away from zero.

In the former case, one finds (after adding a cutoff that we omit) the expression
$$ 
\iiint e^{it \Phi_{+++}} \frac{\mathfrak{c}_{crr}(\xi,\eta,\sigma,\theta) }{
is(\partial_\theta-\partial_\eta) \Phi_{+++}}
  \langle \eta \rangle \partial_\eta \wt{f}(\eta) \partial_\sigma \wt{f}(\sigma)
  \wt{f}(\theta) \frac{\widehat{\phi}(p)}{p}  \, d\eta \, d\sigma \,d\theta + \{ \mbox{easier terms} \},
$$
whose $L^2$-norm can be bounded by
$$
C t^{-1} {\| \jeta \partial_\eta \wt{f} \|}_{L^2} {\| \partial_\s \wt{f} \|}_{L^1}
  {\| e^{-it\langle D \rangle} \mathcal{W}^* f\|}_{L^\infty} \lesssim \e_1^3 t^{-1}.
$$
In the latter case, one finds
$$
\int_0^t \iiint e^{it \Phi_{+++}} \frac{\mathfrak{c}_{crr}(\xi,\eta,\sigma,\theta) }{i \Phi_{+++}}
  \langle \eta \rangle\partial_s [ \partial_\eta  \wt{f}(\eta) \wt{f}(\sigma)  \wt{f}(\theta) ]\frac{\widehat{\phi}(p)}{p}
  \, d\eta \, d\sigma \,d\theta \,ds.
$$
The control of this expression is easy if the time derivative hits $\wt{f}(\sigma)$ or $\wt{f}(\theta)$,
by using a trilinear estimate and \eqref{dtfL2}.
If $\partial_t$ hits $ \wt{f}(\eta)$, the $\partial_\eta$ derivative might result in an additional $t$ factor.
We use \eqref{eqdtf} and look at the two main contributions on its right-hand side: 
when we substitute $\mathcal{C}^S$ to $\partial_t f$,
we obtain a $5$-linear expression in $f$, and estimating four inputs in $L^\infty$, and one in $L^2$ suffices;
when we substitute $\mathcal{Q}^R$ to $\partial_t \wt{f}$, 
we can use the bound $\| \mathcal{Q}^R(f,f) \|_{L^2} \lesssim t^{-1+}$, 
which follows from Lemma \ref{lemQR}, and estimate the two other inputs in $L^\infty$.

\item If $|\sigma| \gg 1$, we have
$$
|\Phi_{+++}(\xi,\eta,\sigma,\theta)| \gtrsim 
\left\{ \begin{array}{ll}
  1 & \mbox{ if $\sigma,\theta$ have the same sign,} 
\\
\langle \theta \rangle 
  & \mbox{ if $\sigma,\theta$ have opposite signs,}
\end{array} \right.
$$
as long as $|p| \ll 1$.
Indeed, this is obvious if $\sigma$ and $\theta$ have opposite signs; and if they do have the same sign,
\begin{align*}
- \Phi_{+++}(\xi,\eta,\sigma,\theta) & = - \langle \eta + \sigma + \theta + p \rangle +\langle \eta \rangle 
  + \langle \sigma \rangle + \langle \theta \rangle
  \\
& = - | \eta + \sigma + \theta + p| + \langle \eta \rangle + |\sigma| + \langle \theta \rangle 
  + O\left(|\sigma|^{-1} \right) \gtrsim 1.
\end{align*}
Turning to estimates on derivatives, for any $p,\eta,\sigma,\theta$ 
such that $|p|,|\eta| \lesssim 1$, $|\sigma| \gg 1$, and $\sigma$ and $\theta$ have the same sign,
$$
\left| \partial_p^a \partial_\eta^b \partial_\sigma^c \partial_\theta^d \frac{1}{\Phi_{+++}(\eta,\sigma,\theta,p)} \right| 
\lesssim |\sigma|^{-c} \langle \theta \rangle^{-d},
$$
so that Lemma \ref{lemmamultilin2} applies. In the case where $\sigma$ and $\theta$ have opposite signs, 
the above does not hold (think of the case where $\sigma + \theta =0$). 
Assuming for instance 
$\sigma >0$, $\theta<0$ ($|\sigma| \geq |\theta|$), let $\sigma' = \sigma + \theta$.
Then, the above derivative estimate holds for the variables $(p,\eta,\sigma',\theta)$.
In both cases, Lemma \ref{lemmamultilin2} applies, and an integration by parts in time suffices. 
\end{itemize}

\medskip
\noindent
{\it The $+--$ case}. 
This can be dealt with similarly to the $+++$ case.
First, we observe that there are again no restricted resonances
$$
\sigma = \theta, \qquad \mbox{and}  \qquad \langle \eta + 2\sigma + p \rangle - \langle \eta \rangle
  + 2 \langle \sigma \rangle = 0,
$$
with $|p| \ll 1$.
This is clear if $|\sigma|\gg 1$; 
if instead $|\sigma| \lesssim 1$, it suffices to treat the case $p=0$ and argue by continuity. 
In other words, it suffices to show that there are no solutions 
of $\langle \eta + 2\sigma \rangle  = \langle \eta \rangle - 2 \langle \sigma \rangle$. 
Squaring both sides leads to $\langle \eta \rangle \langle \sigma \rangle = 1 - \eta \sigma$,
whose only solution is $\eta = -\sigma$, but this is not allowed on the support of $\mathfrak{c}_{crr}$.

Therefore, as long as $|\sigma| + |\theta| \lesssim 1$, the argument used for the $+++$ case applies. 
On the other hand, when $|\sigma| + |\theta| \gg 1$, we have
$|\Phi_{+--}(\xi,\eta,\sigma,\theta)| 
\gtrsim \langle \sigma \rangle + \langle \theta \rangle$, 
so that the argument used above also applies. 

\medskip 
\noindent
{\it The $-++$ case}. 
Once again we look at possible solutions of 
$$
\langle \eta + 2\sigma + p \rangle + \langle \eta \rangle - 2 \langle \sigma \rangle = 0.
$$
for $|p| \ll 1$, on the support of the integral.
It is easy to verify that this equation does not have solutions for $|\sigma| \gg 1$;
in the complementary case it suffices to consider the case $p=0$
and notice that the only solution to $\langle \eta + 2\sigma  \rangle = - \langle \eta \rangle + 2 \langle \sigma \rangle$
is $\eta = -\sigma$, but this does not belong to $\supp(\mathfrak{c}_{crr})$.

We then distinguish different frequencies configurations:
\begin{itemize}
\item If $|\sigma| + |\theta| \lesssim 1$, the argument given in the previous cases apply.

\item If $|\sigma| \sim |\theta| \gg 1$ and $\sigma$ and $\theta$ have opposite signs, 
then $|\Phi_{-++}| \gtrsim \jeta \approx 1$. 
If they have equal signs,
\begin{align*}
\Phi_{-++}(\xi,\eta,\sigma,\theta) & = \langle p+ \eta + \sigma + \theta \rangle + \langle \eta \rangle 
  - \langle \sigma \rangle - \langle \theta \rangle 
  \\
& = |p + \eta + \sigma + \theta|+ \langle \eta \rangle - |\sigma| - |\theta| 
  + O\left( {|\sigma|}^{-1} \right) \gtrsim 1
\end{align*}
as long as $|p| \ll 1$. The estimates on the derivatives are the natural ones, 
and an integration by parts in $s$ suffices.

\item If $|\sigma| \gg |\theta| + 1 $ we need a different argument. 
Observe that $\left| (\partial_\sigma - \partial_\theta) \Phi_{+--}(\xi,\eta,\sigma,\theta) \right| 
  \gtrsim \langle \theta \rangle^{-2}$, and more precisely
$$
\left| \partial_\sigma^a \partial_\theta^b \frac{1}{(\partial_\sigma - \partial_\theta) 
  \Phi_{-++}(\xi,\eta,\sigma,\theta)} \right| \lesssim \langle \theta \rangle^{2-b} |\sigma|^{-a}.
$$
Therefore, an integration by parts in $\partial_\sigma - \partial_\theta$,
followed by an application of (a small adaptation of) Lemma \ref{lemCS} gives
(using $|\theta|\lesssim |\s|$)
\begin{align*}
{\| \mathcal{C}_{crr}(f,f,f) \|}_{L^2} & \lesssim t^{-1} 
  \cdot {\| \chi_c(\eta) \partial_\eta \wt{f} \|}_{L^{2+}}
  {\| \jnab^{0+} e^{-it\jnab} \whF^{-1} \partial_\s \wt{f} \|}_{L^{\infty-}} \cdot 
  {\| \jnab^{0+} e^{-it\jnab} \W f \|}_{L^{\infty-}} 
  \\
  & \lesssim t^{-1} {\| \chi_c(\eta) \partial_\eta \wt{f} \|}_{L^2} 
  {\| \chi_r(\s) \jsig \partial_\s \wt{f} \|}_{L^2}
  {\| \jnab^{0+} e^{-it\jnab} \W f \|}_{L^{\infty-}} 
  \\
  & \lesssim t^{-1} \cdot \e_1 \jt^{\alpha+\beta\gamma} \cdot \e_1 \jt^{\alpha}
  \cdot \e_1 \jt^{-1/2+}
  \\
  & \lesssim \jt^{-1} \e_1^3, 
\end{align*}
having used Sobolev's embedding theorem for the second inequality (recall $|\eta| \approx \sqrt{3}$, 
$|\s| \not\approx \sqrt{3}$), 
interpolation between the linear decay and the $H^4$-norm,
the a priori bounds (see in particular \eqref{lembb1}), and $\alpha + \beta \gamma < 1/4$.
\end{itemize}

\medskip 
\noindent
{\it The $-+-$ case}. 
For this case it is obvious here that there are no restricted space-time resonances, 
since, when $\sigma = -\theta$, the phase is
$\Phi_{-+-}(\xi,\eta,\sigma,\theta) = \langle \eta + p \rangle + \langle \eta \rangle$.
Once again,
\begin{itemize}

\item If $|\sigma| + |\theta| \lesssim 1$, 
one can resort to integration by parts in $\partial_s$ or $\partial_\sigma - \partial_\theta$.

\item If $|\sigma|  + |\theta| \gg 1$, 
\begin{align*}
\Phi_{-+-}(\xi,\eta,\sigma,\theta) & = \langle p+ \eta + \sigma + \theta \rangle 
  + \langle \eta \rangle - \langle \sigma \rangle + \langle \theta \rangle \gtrsim 1
\end{align*}
as long as $|p| \ll 1$.
\end{itemize}

\medskip \noindent
{\it The $---$ case}. 
This is the easiest case since $\Phi_{---}(\xi,\eta,\sigma,\theta) \gtrsim 1$ for all $\xi,\eta,\sigma,\theta$.

\medskip \noindent
{\it The $+-+$ case}.
This is the hardest case, since restricted space time resonances are present;
the phase vanishes when $\s=-\theta$ and $p=0$.
The case $|\theta| + |\sigma| \lesssim 1$ is essentially treated in Section \ref{secwL}, 
therefore we can assume that $|\sigma| \gg 1$, and $|\sigma| \geq |\theta|$.
If $|\sigma| \gg |\theta|$, or $\sigma \approx \theta$, 
then an integration by parts in $\partial_\sigma - \partial_\theta$ suffices;
therefore, we will only focus on the case where $\sigma \approx - \theta$.

It is convenient to adopt the same parameterization of the frequency variables as in Section \ref{secwL}, 
which, after replacing the second $\wt{f}$ by $\wt{f}(-\cdot)$, leads to the question of bounding
\begin{align}\label{+-+1}
\begin{split}
& \sum_{n\geq 10} \iiint e^{is\Psi(\xi,\eta,\zeta,\theta)} \mathfrak{m}_n(\xi,\eta,\sigma,\theta) 
  \partial_\xi \wt{f}(\xi-\eta) \wt{f}(\xi-\eta-\zeta-\theta) \wt{f}(\xi-\zeta) 
  \frac{\widehat{\phi}(\theta)}{\theta} \, d\eta \,d\zeta \,d\theta,
\\
& \Psi(\xi,\eta,\zeta,\theta) = \jxi - \langle \xi - \eta \rangle 
  + \langle \xi - \eta -\zeta -\theta  \rangle - \langle \xi - \zeta\rangle,
\end{split}
\end{align}
where, slightly abusing notations by letting
$\mathfrak{c}_{crr}$ be the symbol expressed both 
in the $(\eta,\sigma,\theta)$ and $(\xi-\eta,\xi-\eta-\zeta-\theta,\xi-\zeta)$ variables,
we define
$$
 \mathfrak{m}_n(\xi,\eta,\sigma,\theta) = 
 \mathfrak{c}_{crr} (\xi,\eta,\sigma,\theta) \langle \xi -\eta \rangle 
 \varphi_{n}(\xi-\eta-\zeta-\theta) \varphi_{\sim n}(\xi-\zeta).
$$

Using the a priori $H^4$ bound, Cauchy-Schwarz's inequality and Lemma \ref{lemCS}, 
we can estimate the $L^2$ norm of each element in the sum in \eqref{+-+1} by
\begin{align*}
C 2^{-(2-)n} \| \partial_\xi \wt{f} \|_{L^2} \| \varphi_n \wt{f} \|_{L^1} 
  \| \varphi_{\sim n} \wt{f} \|_{L^1} \lesssim \e_1 \js^{\alpha + \beta \gamma} 
  2^{-(9-)n} \| f \|_{H^4}^2.
\end{align*}
This bounds suffices as long as $2^n \gtrsim \js^{1/6}$.

If, on the other hand, $2^n \lesssim \js^{1/6}$, we can now follow the skeleton of the estimate 
of $\mathcal{H}^2$ in \S\ref{secwLpv2}. Cases 1, 2 and 3 are identical, 
simply relying on the easy generalization of Lemma \ref{lemCS} to the symbol $\mathfrak{m}_n$ above.

Let us then consider the analogue of Case 4.1, 
which corresponds to the localizations $|\xi - \sqrt{3}| \approx 2^\ell$, 
$|\xi - \eta - \sqrt 3| \leq 2^{\ell-100}$, $|\theta+\xi-\sqrt{3}| \approx 2^h \geq 2^{\ell-10}$. 
To these we add $|\zeta| \approx 2^n$, in correspondence with the $n$-th 
summand in \eqref{+-+1} 
Under these conditions, 
the absolute value of the $\zeta$ derivative of $\Psi$ is
$$
|\partial_\zeta \Psi(\xi,\eta,\zeta,\theta)| = |- \tau'(\xi-\eta-\zeta-\theta) 
+ \tau'(\xi -\zeta)| \approx | \tau''(\xi-\zeta) (\eta + \theta)| \gtrsim 2^{-3n} 2^h. 
$$
More precisely, we have
$$
\left| \partial_\xi^a \partial_\eta^b \partial_\zeta^c \partial_\theta^d 
  \frac{1}{\partial_\zeta \Psi(\xi,\eta,\zeta,\theta)} \right| \lesssim 2^{3n-h} 2^{-(a+c)n} 2^{-h(b+d)},
$$
so that, recalling the bounds on $\mathfrak{c}^{S,2}$, we get
$$
{\Big\| \whF
  \frac{\mathfrak{m}(\xi,\eta,\zeta,\theta)}{\partial_\zeta \Psi(\xi,\eta,\zeta,\theta) } \Big\|}_{L^1}
  \lesssim 2^{n-h}.
$$
Integrating by parts in $\zeta$ gives several terms; the leading one is given by
$$
 \iiint e^{is\Psi(\xi,\eta,\zeta,\theta)} \frac{\mathfrak{m}_n(\xi,\eta,\sigma,\theta)}{s 
 \partial_\zeta \Psi(\xi,\eta,\zeta,\theta)}
  \partial_\xi \wt{f}(\xi-\eta) \partial_\xi \wt{f}(\xi-\eta-\zeta-\theta) 
  \wt{f}(\xi-\zeta) \frac{\widehat{\phi}(\theta)}{\theta} \, d\eta \,d\zeta \,d\theta.
$$
This can be bounded in $L^2$ by
\begin{align*}
C s^{-1} \cdot 2^{n-h} 
  \| \varphi_{< \ell-100} \partial_\xi \wt{f} \|_{L^1} 
  \|\varphi_n \partial_\xi \wt{f} \|_{L^2} \| e^{-it\langle D \rangle} \mathcal{W}^* f \|_{L^\infty} 
  \\
\lesssim \e_1^3 2^{n-h} 2^{\beta'\ell} \js^{-3/2 + 2\alpha}.
\end{align*}
Summing over $2^n \lesssim \js^{1/6}$ and $h \geq \ell - 10$, 
using that $2^\ell \gtrsim \js^{-\gamma}$, with \eqref{wnormparam0}, gives the desired bound.

Finally, there remains Case 4.2 in \S\ref{secwLpv2}, which corresponds to $|\theta| \approx 2^\ell$.
Here we can integrate by parts in $\theta$ using that,
for all $a,b,c,d$ (not all equal to zero), 
$$
\left| \partial_\xi^a \partial_\eta^b \partial_\zeta^c \partial_\theta^d 
  \frac{1}{\partial_\theta \Psi(\xi,\eta,\zeta,\theta)} \right| \lesssim 2^{-n(1+a+b+c+d)}.
$$

\medskip
\subsubsection{Two frequencies close, one removed from $\pm \sqrt{3}$} 
Defining $\mathcal{C}_{ccr}$ through the symbol
$$
\mathfrak{c}_{ccr}(\xi,\eta,\sigma,\theta) 
  = \mathfrak{c}^{2,S} (\xi,\eta,\sigma,\theta) \chi_c(\eta) \chi_c(\sigma) \chi_r(\theta),
$$
we follow once again the approach of \S\ref{threeremoved}, 
and see that the only problematic term is
$$
\iiint e^{is \Phi_{\iota_1\iota_2 \iota_3}} \mathfrak{c}_{ccr}(\xi,\eta,\sigma,\theta)  
  \jeta \partial_\eta \wt{f}(\eta) \wt{f}(\sigma) \wt{f}(\theta)
  \frac{\widehat{\phi}(p)}{p} \, d\eta \, d\sigma \,d\theta
$$
(and, symmetrically, the term where the derivative hits the second function). 
On the support of $\mathfrak{c}_{2,ccr}$
$$
\left| (\partial_\sigma - \partial_\theta) \Phi_{\iota_1 \iota_2 \iota_3} \right| \gtrsim 1,
$$
so that we can integrate by parts in $\partial_\s - \partial_\theta$. 
The worst term resulting from this is
$$
\iiint \frac{1}{t(\partial_\sigma - \partial_\theta) \Phi_{\iota_1 \iota_2 \iota_3}} 
\\
e^{it \Phi_{\iota_1\iota_2 \iota_3}} \mathfrak{c}_{2,ccr}(\xi,\eta,\sigma,\theta)  \langle \eta \rangle \partial_\eta \wt{f}(\eta) \partial_\sigma \wt{f}(\sigma)  \wt{f}(\theta)  \frac{\varphi_{<0}(p)}{p} \, d\eta \, d\sigma \,d\theta.
$$
Using (a slight adaptation of) Lemma \ref{lemCS} this expression can be bounded in $L^2$ by
$$
C s^{-1} {\| \partial_\xi \wt{f} \|}_{L^2}  
  {\| \partial_\xi \wt{f} \|}_{L^1} {\| e^{-it\langle D \rangle} \mathcal{W}^* f \|}_{L^\infty} 
  \lesssim \js^{-3/2 + 2 \alpha + \beta \gamma} \e_1^3, 
$$
and since $\alpha + \beta \gamma<1/4$, we can integrate in time and close the estimate.

\medskip
\subsubsection{Three frequencies close to $\pm \sqrt{3}$} \label{3freq}
Examining the phase in \eqref{Cubic2ot0},
we see that if $\xi,\eta,\sigma,\theta$ are all close to $\pm \sqrt{3}$, 
then $|\Phi_{\iota_1 \iota_2 \iota_3}| \gtrsim 1$ 
unless $(\iota_1 ,\iota_2 ,\iota_3) = (+,-,+)$ up to a permutation. 
We can thus restrict the discussion to the case $(\iota_1 ,\iota_2 ,\iota_3) = (+,-,+)$. 
This case was already the focus of Section \ref{secwL}, 
where it was furthermore assumed that $(\eta,\theta,\sigma)$ was close to $(\sqrt{3},-\sqrt{3},\sqrt{3})$;
notice the different sign due to the particular choice of $p$ in \eqref{Cubic2ot0}.
While the interaction analyzed in Section \ref{secwL} is the worst one,
we also need to consider another partially resonant scenario
where the phase can vanish but not its gradient, 
namely, $(\eta,\theta,\sigma)$ close to $(\sqrt{3},-\sqrt{3},-\sqrt{3})$,
and prove the corresponding estimate \eqref{propCubother0'}.

Since the approach followed is very close to that introduced in Section \ref{secwL}, 
we adopt a similar parametrization of the integration variables and consider 
the trilinear expression
\begin{align}\label{toucan}
\begin{split}
\iiint e^{is \Psi(\xi,\eta,\zeta,\theta)} \mathfrak{p}(\xi,\eta,\zeta) &
\widetilde{f}(\xi-\eta) \widetilde{f}(\xi-2\sqrt{3}-\eta-\zeta-\theta) \widetilde{f}(\xi-2\sqrt{3}-\zeta) 
\frac{\widehat{\phi}(\theta)}{\theta}  \, d\eta \,d\zeta \, d\theta,
\\ 
\mbox{where} \qquad 
\Psi(\xi,\eta,\zeta,\theta) & = \Phi_{+-+}(\xi,\xi-\eta,\xi-2\sqrt{3}-\eta-\zeta-\theta,\xi-2\sqrt{3}-\zeta) 
\\
& = \jxi- \langle \xi - \eta \rangle + 
  \langle \xi -2\sqrt{3} - \eta -\zeta -\theta \rangle - \langle \xi -2\sqrt{3}- \zeta\rangle,
\end{split}
\end{align}
where it is understood that $\mathfrak{p}$ is smooth and such that, on its support, 
$\xi$ is close to $\sqrt{3}$ and $\eta,\zeta$ (and $\theta$) to zero. 
Denoting $\tau(\xi) = \langle \xi \rangle$, 
we start by recording a few estimates  on the phase function:
\begin{align*}
& \Psi(\xi,\eta,\zeta) = \tau'(\xi)\eta - \tau'(\xi - 2 \sqrt{3} - \zeta) (\eta + \theta) + O(\eta^2 + \theta^2), \\
& \partial_\xi \Psi(\xi,\eta,\zeta) =   \tau''(\xi)\eta - \tau''(\xi - 2 \sqrt{3} - \zeta) (\eta + \theta) + O(\eta^2 + \theta^2), \\
& \partial_\eta \Psi(\xi,\eta,\zeta) = \tau'(\xi - \eta) - \tau'(\xi- 2 \sqrt{3} - \eta - \zeta -\theta), \\
& (\partial_\eta - \partial_\zeta) \Psi(\xi,\eta,\zeta) = \tau'(\xi - \eta) - \tau'(\xi- 2 \sqrt{3} - \zeta), \\
& \partial_\zeta \Psi(\xi,\eta,\zeta) = \tau''(\xi - 2 \sqrt{3} - \zeta)( \eta + \theta) + O(\eta^2 + \theta^2), \\
& \partial_\zeta^2 \Psi(\xi,\eta,\zeta) = - \tau'''(\xi - 2 \sqrt{3} - \zeta)( \eta + \theta)+ O(\eta^2 + \theta^2).
\end{align*}
As a consequence,
$$
|\Psi|, |\partial_\xi \Psi|, |\partial_\zeta \Psi|, |\partial_\zeta^2 \Psi| \approx |\eta| 
\qquad \mbox{and} \qquad | \partial_\eta \Psi|, |(\partial_\eta - \partial_\zeta) \Psi| \gtrsim 1.
$$
Applying $\partial_\xi$ to \eqref{toucan} one obtains several terms, which can be reduced to
the following main ones:
\begin{subequations}
\begin{align}
\label{toucan2}
&  \iiint e^{is \Psi} \mathfrak{p}(\xi,\eta,\zeta) \partial_\xi \widetilde{f}(\xi-\eta) 
`\widetilde{f}(\xi-2\sqrt{3}-\eta-\zeta-\theta) \widetilde{f}(\xi-2\sqrt{3}-\zeta) 
\frac{\widehat{\phi}(\theta)}{\theta}  \, d\eta \,d\zeta \, d\theta, 
\\
\label{toucan3}
& \iiint e^{is \Psi} \mathfrak{p}(\xi,\eta,\zeta) \widetilde{f}(\xi-\eta)  \partial_\xi 
\widetilde{f}(\xi-2\sqrt{3}-\eta-\zeta-\theta) \widetilde{f}(\xi-2\sqrt{3}-\zeta) 
\frac{\widehat{\phi}(\theta)}{\theta}  \, d\eta \,d\zeta \, d\theta, 
\\
\label{toucan4}
&  \iiint e^{is \Psi} \mathfrak{p}(\xi,\eta,\zeta) \widetilde{f}(\xi-\eta) 
\widetilde{f}(\xi-2\sqrt{3}-\eta-\zeta-\theta)\partial_\xi \widetilde{f}(\xi-2\sqrt{3}-\zeta) 
\frac{\widehat{\phi}(\theta)}{\theta}  \, d\eta \,d\zeta \, d\theta.
\end{align}
\end{subequations}
The term \eqref{toucan3} can be estimated in a straightforward
way by integrating by parts using the vector field $\partial_\eta - \partial_\zeta$, 
since $|(\partial_\eta - \partial_\zeta) \Psi| \gtrsim 1$; the same applies to~\eqref{toucan4} 
with the vector field $\partial_\eta$. We illustrate this estimate for~\eqref{toucan3}. 
After integrating by parts, the worst term is of the form
$$
 \iiint e^{is \Psi} \frac{\mathfrak{p}(\xi,\eta,\zeta)}{(\partial_\eta-\partial_\zeta) \Psi} 
 \partial_\xi \widetilde{f}(\xi-\eta) \partial_\xi \widetilde{f}(\xi-2\sqrt{3}-\eta-\zeta-\theta) 
 \widetilde{f}(\xi-2\sqrt{3}-\zeta) \frac{\widehat{\phi}(\theta)}{\theta}  \, d\eta \,d\zeta \, d\theta.
$$
In $L^2$, this can be estimated by
$$
C s^{-1} \| \partial_\xi \widetilde{f} \|_{L^2} \| \partial_\xi \widetilde{f} \|_{L^1}
  \| e^{-it\langle D \rangle} \mathcal{W}^* f \|_{L^\infty} 
  \lesssim \e_1^3 \js^{-3/2+2\alpha + \beta \gamma},
$$
which suffices. 

This leaves us with \eqref{toucan2}, which can be treated as $\mathcal{H}_2$ in \S\ref{secwLpv2}.
Following the notation and the approach taken there, we localize dyadically 
$\xi - \sqrt{3}$, $\xi - \eta - \sqrt{3}$ and $\theta + \xi- \sqrt{3}$ 
to the scales $2^\ell$, $2^{j_1}$, and $2^h$ respectively. 
Since $|\partial_\zeta \Psi| \approx |\partial_\zeta^2 \Psi| \approx |\eta + \theta|$, 
the approach in \S\ref{secwLpv2} can be followed almost verbatim, and we can skip the details.

With this the proof of Proposition \ref{propCSbound} is completed. 
In particular, we have obtained the 
improvement on the weighted a priori bound in \eqref{propbootfconc}.
This in turn completes the proof of Proposition \ref{propbootf} and therefore of Theorem \ref{maintheo}.

\appendix 
\medskip
\section{The linearized operator for the Double Sine-Gordon
}\label{Appendix}

This short appendix is devoted to a proof that the linearized operator 
corresponding to the double-sine Gordon model \eqref{DSG} does not have internal modes or resonances,
when linearized at the kinks $K_1$ and $K_2$ described in \S\ref{ssecDSG}.
This proof is essentially contained in \cite{KowMarMunVDB}, but we chose to present it here for the readers' convenience.

\medskip
\noindent 
{\it Change of variables}.
Recall the notation from \S\ref{ssecDSG},
denote for simplicity $U = U_{DSG}$ and $K=K_1$ or $K_2$, and let
\begin{align*}
L := -\partial_x^2 + U''(K), \qquad 
L_0 := -\partial_x^2 + P, \qquad
\Lambda := Y \partial_x Y^{-1},
\end{align*}
where
$$
P = \frac{U'(K)^2}{U(K)} - U''(K), \qquad Y = K'.
$$
Then
$$
L = \Lambda^* \Lambda \qquad \mbox{and} \qquad L_0 = \Lambda \Lambda^*.
$$
Furthermore, if $L\phi = \lambda \phi$, then $L_0 \Lambda \phi = \lambda \Lambda \phi$. 
Therefore, using the decay theory for (generalized) eigenfunctions of $L$, 
the asymptotics of $K$ and its derivatives, and the formula $\Lambda \phi = \phi' - \frac{Y'}{Y} \phi$, we see that

\begin{itemize}
\item[-] If $\phi$ is an eigenfunction of $L$, then $\Lambda \phi$ is an eigenfunction of $L_0$, unless it is zero.

\item[-] If $\phi$ is a resonance of $L$, then $\Lambda \phi$ is a resonance of $L_0$.
\end{itemize}

\medskip
\noindent 
{\it The sign condition}.
We now claim that
$$
x P'(x) \leq 0 \qquad \mbox{for all $x$}.
$$
By definition of $P$, this is equivalent to
$$
x K'(x) V'(K(x)) \leq 0  \qquad \mbox{for all $x$}, \qquad  \mbox{where} \quad V(x) = -U''(x) + \frac{(U'(x))^2}{U(x)}.
$$
Since $K'(x) \geq 0$, we can dispense with this term. 
Setting $y = K(x)$, and using that $y$ and $x$ have the same sign, since $K$ is odd, 
the above becomes $y V'(y) \leq 0$, which can be checked by an explicit computation as in \cite{KowMarMunVDB}.

\medskip
\noindent 
{\it Excluding eigenvalues}.
Assuming that $\phi$ is an eigenfunction of $L_0$, 
we start from the identity $L_0 \phi = \lambda \phi$. 
Testing it against $\phi$ and $x \phi'$ respectively gives
\begin{align*}
& \int ((\phi')^2 + P \phi^2 - \lambda \phi^2)\,dx  = 0,
\\
& \int ((\phi')^2 - P \phi^2 - xP' \phi^2 + \lambda \phi^2) \,dx = 0,
\end{align*}
and adding these two identities leads to
\begin{equation}\label{KMMVDB}
2  \int (\phi')^2 \, dx = \int (xP') \phi^2 \,dx.
\end{equation}
Since $x P'(x) \leq 0$ for all $x$, this gives a contradiction if $\phi$ is a (nonzero) eigenfunction.

\medskip
\noindent 
{\it Excluding resonances.} 
The argument is parallel to the one for eigenfunctions,
but slight complications arise since a regularization becomes necessary. 
Assuming first that $\phi$ is a resonance of $L_0$, 
it has to be even or odd since $P$ is even and, without loss of generality, 
we can assume that $\phi(\infty) = 1$. 
Also, it must have energy $m^2 = P(\infty) = U''(\infty)$.
Choose a smooth, nonnegative, compactly supported function $\chi$ 
and test the equation $-\partial_x^2 \phi + P \phi = m^2 \phi$ 
against $\chi(x/R)\phi$, and $\chi(x/R) x \partial_x\phi$.
This gives
\begin{align*}
& \int \left[ (\phi')^2 + P \chi \left(\frac{x}{R} \right) \phi^2 - m^2  \chi \left(\frac{x}{R} \right) \phi^2 \right]\,dx = o(1), 
\\
& \int \left[ (\phi')^2 - xP' \phi^2 -  P \chi \left(\frac{x}{R} \right) \phi^2 - \int x P \frac{1}{R} \chi' \left(\frac{x}{R} \right)\phi^2 + m^2  \chi \left(\frac{x}{R} \right) \phi^2  + m^2 \frac{1}{R} \chi'  \left(\frac{x}{R} \right) x \phi^2 \right] \,dx = o(1),
\end{align*}
where $o(1) \to 0$ as $R \to \infty$; we used that $\phi'$ and $P'$ decay quickly. 
Adding these two identities leads to
$$
2 \int \left[ (\phi')^2 - xP'  \phi^2 \right] \,dx + \int \frac{x}{R} 
\chi' \left( \frac{x}{R} \right) \left[- P + m^2 \right] \phi^2 \,dx = o(1).
$$
Letting $R \to \infty$ gives
$$
2  \int \left[ (\phi')^2 - xP'  \phi^2 \right] \,dx \leq 0,
$$
which is the desired contradiction.

\end{document}